\documentclass{article}
\usepackage{multirow}
\usepackage{amsmath}
\usepackage{amssymb}
\usepackage{amsthm}
\usepackage{mathtools}
\usepackage{array}
\usepackage{bbm}
\usepackage{cite}
\usepackage[justification=centering]{caption}
\usepackage{stmaryrd}
\SetSymbolFont{stmry}{bold}{U}{stmry}{m}{n}
\usepackage{tikz, tikz-cd, stmaryrd}
\usepackage{bbold} 
\usepackage{euscript}
\usepackage{enumitem}
\setlist[enumerate]{label = {\rm{\rm(\alph*)}},ref = {\rm(\alph*})}
\usepackage[arrow,curve,matrix,arc,2cell]{xy}
\UseAllTwocells
\usepackage[unicode]{hyperref}
\usepackage{slashed}
\usepackage{mathrsfs}
\tikzset{
  symbol/.style={
    draw=none,
    every to/.append style={
      edge node={node [sloped, allow upside down, auto=false]{$#1$}}}
  }
}
\usepackage{array}
\newcolumntype{P}[1]{>{\centering\arraybackslash}p{#1}}
\DeclareFontFamily{U}{mathx}{}
\DeclareFontShape{U}{mathx}{m}{n}{<-> mathx10}{}
\DeclareSymbolFont{mathx}{U}{mathx}{m}{n}
\DeclareMathAccent{\widecheck}{0}{mathx}{"71}

\DeclareMathOperator\ch{ch}
\def\qs{/\kern -.3em/\kern .03em}
\DeclareMathOperator\sk{sk}
\DeclareMathOperator\Quad{Quad}
\DeclareMathOperator\Alt{Alt}
\DeclareMathOperator\Skew{Skew}

\DeclareMathOperator\vol{vol}
\def\CP{{\mathbin{\mathbb{CP}}}}
\def\RP{{\mathbin{\mathbb{RP}}}}
\def\HP{{\mathbin{\mathbb{HP}}}}
\DeclareMathOperator\Sq{Sq}
\DeclareMathOperator\id{id}
\DeclareMathOperator\Id{Id}

\DeclareMathOperator\ev{ev}
\DeclareMathOperator\colim{colim}
\DeclareMathOperator\Hom{Hom}

\DeclareMathOperator\Aut{Aut}
\DeclareMathOperator\Map{Map}
\let\Im\relax\DeclareMathOperator\Im{Im}

\DeclareMathOperator\Ker{Ker}
\DeclareMathOperator\Coker{Coker}

\begin{document}
\def\e#1\e{\begin{equation}#1\end{equation}}
\def\ea#1\ea{\begin{align}#1\end{align}}
\def\eq#1{{\rm(\ref{#1})}}
\theoremstyle{plain}
\newtheorem{thm}{Theorem}[section]
\newtheorem{lem}[thm]{Lemma}
\newtheorem{prop}[thm]{Proposition}
\newtheorem{cor}[thm]{Corollary}
\newtheorem{conj}[thm]{Conjecture}
\newtheorem{quest}[thm]{Question}
\newtheorem{prob}[thm]{Problem}
\theoremstyle{definition}
\newtheorem{dfn}[thm]{Definition}
\newtheorem{ex}[thm]{Example}
\newtheorem{rem}[thm]{Remark}
\newtheorem{ax}[thm]{Axiom}
\newtheorem{ass}[thm]{Assumption}
\newtheorem{property}[thm]{Property}
\newtheorem{cond}[thm]{Condition}
\newtheorem{nota}[thm]{Notation}
\newtheorem*{nota*}{Notation}
\numberwithin{figure}{section}
\numberwithin{table}{section}
\numberwithin{equation}{section}
\def\bG{{\mathbin{\mathbb G}}}
\def\bL{{\mathbin{\mathbb L}}}
\def\N{\mathbb{N}}
\def\Z{\mathbb{Z}}
\def\Q{\mathbb{Q}}
\def\R{\mathbb{R}}
\def\C{\mathbb{C}}
\def\F{\mathbb{F}}
\def\H{\mathbb{H}}
\def\K{\mathbb{K}}
\def\pt{\mathrm{pt}}
\def\GL{\mathrm{GL}}
\def\SL{\mathrm{SL}}
\def\O{\mathrm{O}}
\def\SO{\mathrm{SO}}
\def\Spin{\mathrm{Spin}}
\def\Spinc{\mathrm{Spin^c}}
\def\Hol{\mathop{\rm Hol}}
\def\U{\mathrm{U}}
\def\SU{\mathrm{SU}}
\def\Sp{\mathrm{Sp}}
\def\g{{\mathfrak g}}
\def\h{{\mathfrak h}}
\def\gl{{\mathfrak{gl}}}
\def\sl{{\mathfrak{sl}}}
\def\sp{{\mathfrak{sp}}}
\def\so{{\mathfrak{so}}}
\def\spin{{\mathfrak{spin}}}
\def\u{{\mathfrak u}}
\def\su{{\mathfrak{su}}}
\def\Ad{\mathrm{Ad}}
\def\ad{\mathrm{ad}}
\def\pl{{\rm pl}}
\def\ran{\mathrm{an}}
\def\rank{\mathop{\rm rank}}
\def\top{\mathrm{top}}
\def\cla{\mathrm{cla}}
\def\vdim{\mathop{\rm vdim}\nolimits}
\def\virt{{\rm virt}}
\def\fund{{\rm fund}}
\def\coh{{\rm coh}}
\def\cA{\mathcal{A}}
\def\A{{\mathbin{\cal A}}}
\def\cB{\mathcal{B}}
\def\B{{\mathbin{\cal B}}}
\def\ovB{{\mathbin{\smash{\,\overline{\!\mathcal B}}}}}
\def\cC{\mathcal{C}}
\def\cD{\mathcal{D}}
\def\cE{\mathcal{E}}
\def\cF{\mathcal{F}}
\def\cG{\mathcal{G}}
\def\G{{{\cal G}}}
\def\cH{\mathcal{H}}
\def\cI{\mathcal{I}}
\def\cJ{\mathcal{J}}
\def\cK{\mathcal{K}}
\def\cL{\mathcal{L}}
\def\M{{\mathbin{\cal M}}}
\def\cM{\mathcal{M}}
\def\bcM{{\mathbin{\bs{\cal M}}}}
\def\cN{\mathcal{N}}
\def\cO{\mathcal{O}}
\def\cP{\mathcal{P}}
\def\cQ{\mathcal{Q}}
\def\cR{\mathcal{R}}
\def\cS{\mathcal{S}}
\def\cT{\mathcal{T}}
\def\cU{\mathcal{U}}
\def\cV{\mathcal{V}}
\def\cW{\mathcal{W}}
\def\cX{\mathcal{X}}
\def\cY{\mathcal{Y}}
\def\cZ{\mathcal{Z}}
\def\al{\alpha}
\def\be{\beta}
\def\ga{\gamma}
\def\de{\delta}
\def\io{\iota}
\def\ep{\epsilon}
\def\varep{\varepsilon}
\def\la{\lambda}
\def\ka{\kappa}
\def\th{\theta}
\def\ze{\zeta}
\def\up{\upsilon}
\def\vp{\varphi}
\def\si{\sigma}
\def\om{\omega}
\def\De{\Delta}
\def\Ka{{\rm K}}
\def\La{\Lambda}
\def\Om{\Omega}
\def\Ga{\Gamma}
\def\Si{\Sigma}
\def\Th{\Theta}
\def\Up{\Upsilon}
\def\Chi{{\rm X}}
\def\Tau{{\rm T}}
\def\Nu{{\rm N}}
\def\op{\oplus}
\def\ot{\otimes}
\def\t{\times}
\def\d{{\rm d}}
\def\db{\bar\partial}
\def\na{\nabla}
\def\bu{\bullet}
\def\iy{\infty}
\def\longra{\longrightarrow}
\def\an#1{\langle #1 \rangle}
\def\ban#1{\bigl\langle #1 \bigr\rangle}
\def\lb{\llbracket}
\def\rb{\rrbracket}
\def\ul{\underline}
\def\es{\emptyset}
\def\ab{\allowbreak}
\def\bs{\boldsymbol}
\def\ge{\geqslant}
\def\le{\leqslant}
\def\Eig{\operatorname{Eig}}
\def\Hilb{\mathop{\rm Hilb}\nolimits}
\def\adj{\mathrm{adj}}
\def\ss{{\rm ss}}
\def\Vect{\mathop{\rm Vect}}
\def\vect{{\rm vect}}
\def\APS{\mathrm{APS}}
\def\em{\mathrm{em}}
\def\im{\mathrm{im}}
\def\rsm{\mathrm{sm}}
\def\Cay{{\mathrm{Cay}}}
\def\ass{{\mathrm{ass}}}
\def\asd{{\mathrm{asd}}}
\def\inc{\mathrm{inc}}
\def\rel{{\rm rel}}
\def\trans{\mathrm{trans}}
\def\PF{\mathop{\rm PF}\nolimits}
\def\DET{\mathop{\rm DET}\nolimits}
\def\SP{\mathop{\rm SP}\nolimits}
\def\ind{\mathop{\rm ind}\nolimits}
\def\skewind{\operatorname{skew-ind}}
\def\skew{\mathrm{skew}}
\def\sym{\mathrm{sym}}
\def\Stab{\mathop{\rm Stab}\nolimits}
\def\irr{{\rm irr}}
\def\red{{\rm red}}
\def\Pf{{\rm Pf}}
\def\boo{{\mathbin{\mathbb 1}}}
\def\m{{\mathfrak m}}
\def\KOtor{\Ga_\ell\qs\Ga_{\ell+1}}
\def\Bord{{\mathfrak{Bord}}{}}
\def\tBord{\widetilde{\mathfrak{B}}{\mathfrak{ord}}{}}
\def\hBord{\widehat{\mathfrak{B}}{\mathfrak{ord}}{}}
\def\cBord{\smash{\widecheck{\mathfrak{B}}}{\mathfrak{ord}}{}}
\def\rO{\mathrm{O}}
\def\rF{\mathrm{F}}
\def\sD{\slashed{D}}
\def\rF{\mathrm{F}}
\def\Cl{C\ell}
\def\Gr{\operatorname{Gr}_\mathrm{res}}
\def\Cyl{\operatorname{Cyl}}
\def\Ell{\operatorname{Ell}}
\def\Lag{\operatorname{Lag}}
\def\Pd{\operatorname{Pd}}
\def\det{\operatorname{det}}
\def\supp{\operatorname{supp}}
\def\rM{\mathrm{M}}
\def\rB{\mathrm{B}}
\def\na{\nabla}
\def\Spc{\SP}
\def\PF{\mathop{\rm PF}\nolimits}
\def\DET{\mathop{\rm DET}\nolimits}
\def\SP{\mathop{\rm SP}\nolimits}
\def\Ell{\operatorname{Ell}}
\def\Or{\mathop{\rm Or}}
\def\rst{\mathrm{st}}
\def\nOr{\check\Or}
\def\wh{\widehat}
\def\mathscr{\scr}
\def\oBord{{\ov{\mathfrak{Bord}}}{}}
\def\osF{\bar{\mathsf{F}}}
\def\sF{\mathsf{F}}
\def\sG{\mathsf{G}}
\def\sH{\mathsf{H}}
\def\sI{\mathsf{I}}
\def\sJ{\mathsf{J}}
\def\sK{\mathsf{K}}
\def\sL{\mathsf{L}}
\def\sM{\mathsf{M}}
\def\sN{\mathsf{N}}
\def\sO{\mathsf{O}}
\def\sP{\mathsf{P}}
\def\sQ{\mathsf{Q}}
\def\sT{\mathsf{T}}
\def\sZ{\mathsf{Z}}
\def\sz{\mathsf{z}}
\def\bw{\bigwedge}
\newcommand\tor[1]{\mathop{#1\text{\rm-tor}}}
\def\pd{\partial}
\def\ts{\textstyle}
\def\st{\scriptstyle}
\def\sst{\scriptscriptstyle}
\def\w{\wedge}
\def\sm{\setminus}
\def\lt{\ltimes}
\def\bu{\bullet}
\def\sh{\sharp}
\def\di{\diamond}
\def\he{\heartsuit}
\def\od{\odot}
\def\op{\oplus}
\def\ot{\otimes}
\def\bt{\boxtimes}
\def\bp{\boxplus}
\def\ov{\overline}
\def\bigop{\bigoplus}
\def\bigot{\bigotimes}
\def\iy{\infty}
\def\es{\emptyset}
\def\ra{\rightarrow}
\def\rra{\rightrightarrows}
\def\Ra{\Rightarrow}
\def\Longra{\Longrightarrow}
\def\ci{\circ}
\def\ti{\tilde}
\def\ra{\rightarrow}
\def\ha{{\ts\frac{1}{2}}}
\def\hookra{\hookrightarrow}
\def\dashra{\dashrightarrow}
\def\md#1{\vert #1 \vert}
\def\ms#1{\vert #1 \vert^2}
\def\nm#1{\Vert #1 \Vert}
\def\bmd#1{\big\vert #1 \big\vert}
\def\bms#1{\big\vert #1 \big\vert^2}
\def\an#1{\langle #1 \rangle}
\def\ban#1{\bigl\langle #1 \bigr\rangle}
\def\Ztor{{\mathop{\Z_2\text{\rm-tor}}}}
\def\ZZtor{{\mathop{\Z\text{\rm-tor}}}}
\def\Zktor{{\mathop{\Z_k\text{\rm-tor}}}}
\def\sZtor{{\mathop{\text{\rm s-}\Z_2\text{\rm-tor}}}}
\def\sign{\mathop{\rm sign}\nolimits}
\def\Ord{\mathop{\rm Ord}}
\def\Mon{\mathop{\rm Mon}}
\def\tLeq{{\,\,\,\widetilde{\!\!\!\mathfrak{Leq}\!\!\!}\,\,\,}{}}
\title{Bordism categories and \\ orientations of moduli spaces}
\author{Dominic Joyce and Markus Upmeier}

\date{}

\maketitle

\begin{abstract}
There are many situations in geometry where one forms moduli spaces $\M$ of some geometric objects, and $\M$ may be a (possibly singular, or derived) real manifold, and one wishes to define an {\it orientation\/} on $\M$. Such orientations are needed to define enumerative invariants which `count' points in $\M$, e.g.\ Donaldson invariants of 4-manifolds.

This monograph develops a new bordism-theoretic point of view on orientations of moduli spaces. Let $X$ be a compact $n$-manifold with geometric structure $\Om$, and $\M$ a moduli space of geometric objects on $X$, e.g.\ a moduli space of connections on $X$ satisfying a curvature condition, or of calibrated submanifolds in $X$. Our theory aims to answer the questions:
\begin{itemize}
\setlength{\itemsep}{0pt}
\setlength{\parsep}{0pt}
\item[(i)] Can we prove $\M$ is orientable for all such $(X,\Om),\M$?
\item[(ii)] If not, can we give computable sufficient conditions on $(X,\Om)$ that guarantee $\M$ is orientable?
\item[(iii)] If the sufficient conditions hold, can we specify extra data on $X$ which allow us to construct a canonical orientation on~$\M$?  
\end{itemize}

Our theory is written in terms of {\it bordism groups} of certain classifying spaces $T$ in Algebraic Topology, such as $T=\cL BG$ for moduli spaces of connections on principal $G$-bundles $P\ra X$. The typical answer to (ii) is that $\M$ is orientable provided certain `bad' bordism classes $\be$ in $\Om_n^{\bs\Spin}(T)$ cannot be written in the form $\be=[X,\phi]$ for our $n$-manifold $X$, and the answer to (i) is yes if there are no `bad' classes.

We define {\it bordism categories}, such as $\Bord_n^{\bs\Spin}(BG)$ with objects $(X,P)$ for $X$ a compact spin $n$-manifold and $P\ra X$ a principal $G$-bundle. Bordism categories are {\it Picard groupoids}. Orientation problems are encoded in {\it orientation functors\/} $\sO: \Bord_n^{\bs\Spin}(BG)\ra\sZtor$. Orientations on moduli spaces $\M$ on $X$ are induced by trivializations of $\sO$ on a subcategory of $\Bord_n^{\bs\Spin}(BG)$ depending on $X$. We compute spin bordism groups $\Om_n^{\bs\Spin}(T)$ for many classifying spaces $BG,MH,K(\Z,R),\ldots$ using Algebraic Topology, and use these to answer orientability questions.

We apply our theory to study orientability and canonical orientations for moduli spaces of $G_2$-instantons and of associative 3-folds in $G_2$-manifolds, and for moduli spaces of $\Spin(7)$-instantons and of Cayley 4-folds in $\Spin(7)$-manifolds, and for moduli spaces of coherent sheaves and perfect complexes on Calabi--Yau 4-folds. The latter are needed to define Donaldson--Thomas type invariants of Calabi--Yau 4-folds.
\end{abstract}

\setcounter{tocdepth}{2}
\tableofcontents

\section{Introduction}
\label{fm1}

Coherent orientations of moduli spaces play an important role in gauge theory and for enumerative geometry \cite{BoJo,DoKr,DoSe,DoTh,Joyc5,Joyc7,OhTh,Walp1}. Despite their central role, the absence of a general framework means that orientations often remain poorly understood. Indeed, the motivating question for this work, which we answer in \S\ref{fm13}, was whether one can construct canonical orientations for Donaldson--Thomas type invariants of Calabi--Yau 4-folds~\cite{BoJo,CaLe,OhTh}.

In Differential Geometry, we study moduli spaces $\M$ of geometric objects $E$ on a compact manifold $X$ which satisfy a nonlinear elliptic p.d.e. For example:
\begin{itemize}
\setlength{\itemsep}{0pt}
\setlength{\parsep}{0pt}
\item[(a)] Let $(X,\vp,g)$ be a compact 7-manifold with coclosed $G_2$-structure, $G$ a Lie group, and $P\ra X$ a principal $G$-bundle. Write $\M_P^{G_2}$ for the moduli space of $G_2$-{\it instantons\/} on $P$, that is, irreducible connections $\nabla$ on $P$ satisfying the curvature condition $\pi^2_7(F_\nabla)=0$, modulo gauge transformations.
\item[(b)] Similarly, if $(X,\Om,g)$ is a compact 8-manifold with $\Spin(7)$-structure, we consider moduli spaces $\M_P^{\Spin(7)}$ of $\Spin(7)$-{\it instantons\/} on $P\ra X$.
\item[(c)] Let $(X,\vp,g)$ be a compact 7-manifold with $G_2$-structure. Write $\M_\al^{\rm ass}$ for the moduli space of {\it associative\/ $3$-folds\/} on $X$, that is, compact $\vp$-calibrated 3-submanifolds $N\subset X$, in a fixed L-equivalence class $\al\in\La_3^{\SO}(X)$.
\item[(d)] Similarly, if $(X,\Om,g)$ is a compact 8-manifold with $\Spin(7)$-structure, we consider moduli spaces $\M_\al^{\rm Cay}$ of {\it Cayley\/ $4$-folds\/} $N\subset X$.
\end{itemize}
In each case, the moduli space $\M$ is a smooth manifold if $\vp,\Om$ are generic, and a derived manifold \cite{Joyc2,Joyc3,Joyc4,Joyc6} in general, so orientations on $\M$ make sense.

There are also orientation problems in Algebraic Geometry. In particular:
\begin{itemize}
\setlength{\itemsep}{0pt}
\setlength{\parsep}{0pt}
\item[(e)] Let $X$ be a Calabi--Yau 4-fold, and $\M_\al^\rst(\tau)$ a proper moduli scheme of Gieseker stable coherent sheaves on $X$ with Chern character $\al$. Borisov--Joyce \cite{BoJo} showed how to make $\M_\al^\rst(\tau)$ into a derived manifold. Thus, if we can choose an {\it orientation\/} on $\M_\al^\rst(\tau)$ (which has a purely algebro-geometric definition), we get a virtual class $[\M_\al^\rst(\tau)]_\virt$ in $H_*(\M_\al^\rst(\tau),\Z)$, which can be used to define Donaldson--Thomas type `DT4 invariants' of $X$. Later, Oh--Thomas \cite{OhTh} gave an algebro-geometric definition of $[\M_\al^\rst(\tau)]_\virt$, still needing a choice of orientation. DT4 invariants are now a very active area, see for instance~\cite{Bojk1,Bojk2,BoJo,Cao1,Cao2,CaKo1,CaKo2,CKM,CaLe,CMT1, CMT2,COT1,COT2,CaQu,CaTo1,CaTo2,CaTo3,GJT,Joyc7,KiPa,OhTh,Park}.
\end{itemize}

Cao--Gross--Joyce \cite[Th.~1.15]{CGJ} showed that algebro-geometric orientations on $\M_\al^\rst(\tau)$ can be induced from orientations on moduli spaces of connections $\B_P$ on $\U(m)$-bundles $P\ra X$ for $m\gg 0$, which is essentially the same orientation problem as for $\Spin(7)$-instantons in (b) above. This brings DT4 orientations within the scope of our approach to orientations on moduli spaces.

In this monograph we develop a general framework for studying orientations of moduli spaces using {\it bordism categories}. We illustrate the idea of bordism categories using the gauge-theoretic categories $\Bord_n^{\bs\Spin}(BG)$ from \S\ref{fm4}, which are used in \S\ref{fm12}--\S\ref{fm13} to solve (a),(b),(e) above. We also define {\it submanifold bordism categories\/} $\Bord_{n,k}^{\bs B}(MH)$ in \S\ref{fm5}, used in \S\ref{fm14} to solve (c),(d), {\it cohomology bordism categories\/} $\Bord^{\bs B}_n(K(R,k))$ in \S\ref{fm6}, which are used to reduce orientability and canonical orientation problems to computations using cohomology operations, such as Steenrod squares, and {\it topological bordism categories\/} $\Bord^{\bs B}_n(T)_\top$ in~\S\ref{fm7}.

Let $G$ be a Lie group, and $n\ge 0$. The bordism category $\Bord_n^{\bs\Spin}(BG)$ is a symmetric monoidal category with objects pairs $(X,P)$ for $X$ a compact spin $n$-manifold and $P\ra X$ a principal $G$-bundle, and morphisms $[Y,Q]:(X_0,P_0)\ra(X_1,P_1)$ equivalence classes $[Y,Q]$ of pairs $(Y,Q)$, where $Y$ is a compact spin $(n+1)$-manifold with boundary $\pd Y=-X_0\amalg X_1$ and $Q\ra Y$ is a principal $G$-bundle with $Q\vert_{\pd Y}=P_0\amalg P_1$. The composition of $[Y,Q]:(X_0,P_0)\ra(X_1,P_1)$ and $[Y',Q']:(X_1,P_1)\ra(X_2,P_2)$ is $[Y\amalg_{X_1}Y',Q\amalg_{P_1}Q']$. The monoidal structure is $(X_0,P_0)\ot(X_1,P_1)=(X_0\amalg X_1,P_0\amalg P_1)$ on objects.

Then $\Bord_n^{\bs\Spin}(BG)$ is a {\it Picard groupoid}, or {\it abelian\/ $2$-group}. As explained in Appendix \ref{fmA}, Picard groupoids $\cC$ are classified up to equivalence by abelian groups $\pi_0,\pi_1$ and a linear quadratic morphism $q:\pi_0\ra\pi_1$. For $\Bord_n^{\bs\Spin}(BG)$ these are the spin bordism groups $\pi_i=\Om_{n+i}^{\bs\Spin}(BG)$ for $i=0,1$, where $BG$ is the classifying space of $G$, and $q$ is multiplication by $\al_1$ in $\Om_1^{\bs\Spin}(*)=\Z_2\an{\al_1}$. Thus, if we can compute $\Om_*^{\bs\Spin}(BG)$, we understand $\Bord_n^{\bs\Spin}(BG)$ very well.

After some background from Algebraic Topology in \S\ref{fm2}, our first main results, in \S\ref{fm3}, compute the spin bordism groups $\Om_n^{\bs\Spin}(T)$, and in some cases $\Om_{n-1}^{\bs\Spin}(\cL T)$, for $n\le 9$ and $T$ in the list of spaces
\e
\begin{gathered}
M\SU(2),M\U(2),M\Spin(4),M\SO(4),B\SU(m),B\Sp(m),\\ 
BE_8,\SU(2),\SU,\Sp,K(\Z,3),K(\Z_2,3),K(\Z,4),K(\Z_2,4).
\end{gathered}
\label{fm1eq1}
\e
Here $MH$ is the Thom space of $H\ra\O(4)$, as in \S\ref{fm25}, and $BG$ is the classifying space of the Lie group $G$, and $K(R,k)$ is an Eilenberg--MacLane space, which classifies cohomology classes in $H^k(-,R)$. We also compute some maps $\hat\xi_{n-1}^{\bs\Spin}(T):\ti\Om_{n-1}^{\bs\Spin}(\cL T;T)\ra\ti\Om_n^{\bs\Spin}(T)$ arising in our orientability theory. The proofs of theorems in \S\ref{fm3} are given in~\S\ref{fm15}--\S\ref{fm18}.

As an example, we give a partial statement here of our results for $M\U(2)$ and $B\SU(m)$, see Theorems \ref{fm3thm1}--\ref{fm3thm3} and Corollary \ref{fm3cor1}. Section \ref{fm3} also gives similar results for the other spaces in~\eq{fm1eq1}.

\begin{thm}
\label{fm1thm1}
{\bf(a)} In dimensions\/ $n\le 9$ the reduced spin bordism groups\/ $\ti{\Omega}_n^{\bs\Spin}(M\U(2))$ are as follows: 
\e
\begin{gathered}
\begin{tabular}{c|ccccccccccc}
$n$ & 
{\rm 0,1,2,3,5} &  $4$  & $6$  & $8$ & $9$ \\
\hline
\parbox[top][4ex][c]{2cm}{$\!\!\!\ti\Om_n^{\bs\Spin}(M\U(2))\!\!\!$} & 
$0$ &  $\Z\an{\de}$ & $\Z\an{\varepsilon}$ & $\Z\an{\frac{\ze_1}{2},\ze_2,\ze_3}$ & $\Z_2\an{\al_1\ze_2}$
\end{tabular}
\end{gathered}
\label{fm1eq2}
\e
Here, writing elements of\/ $\ti{\Omega}_n^{\bs\Spin}(M\U(2))$ as $[X,M]$ for $X$ a compact spin\/ $n$-manifold and\/ $M\subset X$ a compact\/ $(n-4)$-submanifold with a normal\/ $\U(2)$-structure, we have
\begin{align*}
\al_1&=[\cS^1_{\rm nb}]\in\Om_1^{\bs\Spin}(*), & \de&=[\cS^3\t\cS^1_{\rm b},\{*\}], & \varepsilon&=[\cS^5\t\cS^1_{\rm b},\cS^2], \\
\frac{\ze_1}{2}&=[(K3\t\cS^3)/\Z_2\an{\al}\t\cS^1_{\rm b},K3\t\{N\}\t\{1\}],
\!\!\!\!\!\!\!\!\!\!\!\!\!\!\!\!\!\!\!\!\!\!\!\!\!\!\!\!\!\!\!\!\!\!\!\!\!\!\!\!\!\!\!\!\!\!\!\!\!\!\!\!\!\!\!\!\!\!\!\!\!\!\!\!\! \\
\ze_2&=[\HP^2,\overline{\HP}^1]-[\HP^2,\es], & \ze_3&=[\CP^3\t(\cS^1_{\rm b})^2, \CP^2\t\{1\}^2].\!\!\!\!\!\!\!\!\!\!\!\!\!\!\!\!\!\!\!\!\!\!\!\!\!\!\!\!\!\!\!
\end{align*}
For $n=4,6,8$ the isomorphisms in \eq{fm1eq2} may be written explicitly as
\begin{align*}
&\ti\Om_4^{\bs\Spin}(M\U(2))\overset\cong\longra\Z, &[X,M]&\longmapsto\# M, \\
&\ti\Om_6^{\bs\Spin}(M\U(2))\overset\cong\longra\Z,&[X,M]&\longmapsto\ts -\frac{1}{2}\int_M c_1(\nu_M),\\
&\ti\Om_8^{\bs\Spin}(M\U(2))\overset\cong\longra\Z^3, \\
&[X,M]\longmapsto\mathrlap{\left(-\ts\frac{\operatorname{sign}(M)}{8} + \ts\int_M \frac{c_1(\nu_M)^2}{8}, \ts\int_M c_2(\nu_M), \ts\int_M c_1(\nu_M)^2\right),}
\end{align*}
which map\/ $\de\mapsto 1,$ $\varep\mapsto 1,$ and\/ $\frac{\ze_1}{2}\mapsto(1,0,0),$ $\ze_2\mapsto (0,1,0),$ $\ze_3\mapsto(0,0,1)$.
\smallskip

\noindent{\bf(b)} There is a map $\hat\xi^{\bs\Spin}_{n-1}(M\U(2)):\Om^{\bs\Spin}_{n-1}(\cL M\U(2);M\U(2))\ab\ra\ti\Om^{\bs\Spin}_n(M\U(2))$ defined in \eq{fm2eq7} below. In the notation of\/ {\bf(a)\rm,} this has image
\e
\begin{gathered}
\begin{tabular}{c|ccccccccccc}
$n$ & 
{\rm 0,1,2,3,5} &  $4$  & $6$  & $8$ & $9$ \\
\hline
\parbox[top][4ex][c]{2.3cm}{$\!\!\!\Im\hat\xi^{\bs\Spin}_{n-1}(M\U(2))\!\!\!$} & 
$0$ &  $\Z\an{\de}$ & $\Z\an{\varepsilon}$ & $\Z\an{\frac{\ze_1}{2},2\ze_2,\ze_3}$ & $\Z_2\an{\al_1\ze_2}$
\end{tabular}
\end{gathered}
\label{fm1eq3}
\e

\noindent{\bf(c)} There is a\/ $10$-connected map $M\U(2)\ra B\SU$. Also the canonical map\/ $B\SU(m)\ra B\SU$ is\/ $(2m+1)$-connected. These imply that\/ $\ti\Om_n^{\bs\Spin}(M\U(2))\cong \ti\Om_n^{\bs\Spin}(B\SU)$ for $n\le 9$ and\/ $\ti\Om_n^{\bs\Spin}(B\SU(m))\cong\ti\Om_n^{\bs\Spin}(B\SU)$ for $2m\ge n$. Hence $\ti\Om_n^{\bs\Spin}(B\SU(m))$ and\/ $\Im\hat\xi^{\bs\Spin}_{n-1}(M\SU(m))$ for $n\le 9$ and\/ $2m\ge n$ are as in\/ {\rm\eq{fm1eq2}--\eq{fm1eq3}}. For $n=4,6,8$ and $m\ge 2,3$ and\/ $4$ respectively, the description of\/ $\ti\Om_n^{\bs\Spin}(B\SU(m))$ may be written explicitly as
\ea
&\ti\Om_4^{\bs\Spin}(B\SU(m))\overset\cong\longra\Z,\qquad 
[X,P]\longmapsto\ts\int_X c_2(P),
\label{fm1eq4}\\
&\ti\Om_6^{\bs\Spin}(B\SU(m))\overset\cong\longra\Z,\qquad [X,P]\longmapsto\ts\frac{1}{2}\int_X c_3(P),
\label{fm1eq5}\\
&\ti\Om_8^{\bs\Spin}(B\SU(m))\overset\cong\longra\Z^3,
\label{fm1eq6}\\
&[X,P]\longmapsto \left(\ts\int_X[\frac{c_4(P)}{6} - \frac{c_2(P)^2}{12} - \frac{p_1(TX)c_2(P)}{24}],\ts\int_X c_2(P)^2, \int_X c_4(P)\right),\nonumber
\ea
which map\/ $\de\mapsto 1,$ $\varep\mapsto 1$ and\/ $\frac{\ze_1}{2}\mapsto(1,0,0),$ $\ze_2\mapsto(0,1,0),$ $\ze_3\mapsto(0,0,1)$.
\end{thm}

Bordism theory is applicable to orientations of moduli spaces as the second author \cite{Upme2} establishes a fundamental new technique that formalizes the idea that for orientation problems based on twisted Dirac operators, orientations are functorial along bordisms; we recall this result as Theorem \ref{fm9thm2} below.

Let $X$ be a compact spin $n$-manifold and $P\ra X$ a principal $G$-bundle. Write $\B_P=\A_P/\Aut(P)$ for the moduli stack of all connections $\nabla$ on $P$, modulo gauge transformations. Suppose $n\equiv 1,7,8\mod 8$. Then we can define a principal $\Z_2$-bundle $O_P\ra\B_P$ which parametrizes orientations at $[\nabla]\in\B_P$ for the (positive) Dirac operator $\sD^{(+)}_X\ot\ad(\nabla)$ twisted by the connection $\ad(\nabla)$ on the adjoint bundle $\g\hookra \ad(P)\ra X$. (The restriction to $n\equiv 1,7,8\mod 8$ is because otherwise $\sD^{(+)}_X$ is $\C$- or $\H$-linear, and is trivially oriented.) Such orientations are relevant to instanton moduli spaces, since in (a),(b) we have $\M_P\subset\B_P$, and $O_P\vert_{\M_P}$ is the orientation bundle on $\M_P$, so that an orientation on $\B_P$ induces orientations of $G_2$- or $\Spin(7)$-instanton moduli spaces~$\M_P$.

Using Theorem \ref{fm9thm2} we define a monoidal functor $\sO:\Bord_n^{\bs\Spin}(BG)\ra\Ztor$ or $\sZtor$, where $\Ztor,\sZtor$ are the Picard groupoids of (super) $\Z_2$-torsors. This functor controls orientations of moduli spaces $\B_P$ in the following sense: let $X$ be a compact spin $n$-manifold, and write $\Bord_X^{\bs\Spin}(BG)$ for the subcategory of $\Bord_n^{\bs\Spin}(BG)$ with objects $(X,P)$ and morphisms $[X\t[0,1],Q]$. Then a natural isomorphism $\eta_X$ of functors in the diagram
\e
\begin{gathered}
\xymatrix@C=120pt@R=15pt{
*+[r]{\Bord_X(BG)_{\vphantom{(}}} \drtwocell_{}\omit^{}\omit{^{\eta_X\,\,\,}} \ar[r]_(0.35)\boo \ar@{^{(}->}[d]^{\inc} & *+[l]{\Ztor} \\
*+[r]{\Bord_n^{\bs\Spin}(BG)} \ar[r]^(0.6)\sO & *+[l]{\sZtor} \ar[u]^{\text{forget $\Z_2$-grading}} }
\end{gathered}
\label{fm1eq7}
\e
is equivalent to choices of orientations for $\B_P$ for all principal $G$-bundles $P\ra X$, invariant under isomorphisms $P\cong P'$. Using the description of $\Bord_n^{\bs\Spin}(BG)$ in terms of $\Om_{n+i}^{\bs\Spin}(BG)$ for $i=0,1$, in \S\ref{fm9} we will prove:

\begin{thm}
\label{fm1thm2}
Let\/ $X$ be a compact spin $n$-manifold for $n\equiv 1,7,8\mod 8,$ and $G$ be a Lie group. Then $\B_P$ is orientable \textup(i.e.\ the principal\/ $\Z_2$-bundle $O_P\ra\B_P$ is trivializable\/\textup) for all principal\/ $G$-bundles $P\ra X$ if and only if
\begin{equation*}
\pi_1(\sO)\bigl([X\t\cS^1,\phi]\bigr)=\ul{0}\qquad\text{in $\Z_2$ for all maps $\phi:X\t\cS^1\ra BG,$}
\end{equation*}
where $\pi_1(\sO):\Om_{n+1}^{\bs\Spin}(BG)\ra\Z_2$ is a group morphism. 

Writing $\xi_n^{\bs\Spin}(BG):\Om_n^{\bs\Spin}(\cL BG)\ab\ra\Om_{n+1}^{\bs\Spin}(BG)$ for the natural map, where $\cL BG$ is the free loop space of\/ $BG,$ this condition may also be written
\begin{equation*}
\pi_1(\sO)\ci\xi_n^{\bs\Spin}(BG)\bigl([X,\phi']\bigr)=\ul{0}\qquad\text{in $\Z_2$ for all maps $\phi':X\ra\cL BG$.}
\end{equation*}
Furthermore, $\B_P$ is orientable for all\/ $X,P$ if and only if\/ $\pi_1(\sO)\ci\xi_n^{\bs\Spin}(BG)\equiv\ul{0}$.
\end{thm}

We also prove similar results relevant to orientations of moduli spaces of special submanifolds, as in (c),(d) above.

Results like Theorem \ref{fm1thm2} mean that we can answer orientability problems by explicit computation of bordism groups $\Om_m^{\bs\Spin}(BG),\Om_m^{\bs\Spin}(\cL BG)$ and morphisms $\pi_i(\sO)$. As in Theorem \ref{fm1thm1}, we calculate spin bordism groups of many classifying spaces $BG,\cL BG,MH,\cL MH,K(R,k),\cL K(R,k)$ that we will be interested in, and the results are given in~\S\ref{fm3}.

For a given spin $n$-manifold $X$, testing whether a class $\al\in\Om_{n+1}^{\bs\Spin}(BG)$ may be written in the form $[X\t\cS^1,\phi]$ may not be easy. We develop a technique to help with this. It turns out that in many examples we care about, the orientation functor $\sO:\Bord_n^{\bs\Spin}(BG)\ra\sZtor$ factors via a monoidal functor $\sT:\Bord_n^{\bs\Spin}(BG)\ra \Bord_n^{\bs\Spin}(K(\Z,4))$, where $\Bord_n^{\bs\Spin}(K(\Z,4))$ is a {\it cohomology bordism category\/} from \S\ref{fm6}. Testing whether a class $\be\in\Om_{n+1}^{\bs\Spin}(K(\Z,4))$ may be written in the form $[X\t\cS^1,\psi]$ reduces to conditions on the cohomology of $X$, e.g.\ involving Steenrod squares, which are usually more computable.

Next we consider the question: suppose we have proved orientability of some class of moduli spaces $\M$ on $X$, such as those in (a)--(e) above, by our bordism-theoretic methods. What additional data on $X$ do we need to construct a {\it canonical\/} orientation on each moduli space $\M$? For example, we would like such canonical orientations to develop a theory of DT4 invariants in (e) above.

We explained above that a natural isomorphism $\eta_X$ in \eq{fm1eq7} is equivalent to choices of orientations for $\B_P$ for all principal $G$-bundles $P\ra X$. However, saying $\eta_X$ is the additional data we want is {\it not\/} a good answer, as this allows an arbitrary choice of orientation on $\B_P$ for each isomorphism class of $G$-bundles $P\ra X$. Instead, we hope that a much smaller amount of data on $X$ --- ideally, only a finite choice --- can be used to construct $\eta_X$ in~\eq{fm1eq7}.

Our solution, when $\sO$ factors as $\sO\cong\sH\ci\sT$ for a monoidal functor $\mathsf{T}:\Bord_n^{\bs\Spin}(BG)\ra \Bord_n^{\bs\Spin}(K(\Z,4))$, is to define a {\it flag structure\/} $F$ on $X$ to be a natural isomorphism of functors in the following diagram:
\e
\begin{gathered}
\xymatrix@C=120pt@R=15pt{
*+[r]{\Bord_X(K(\Z,4))_{\vphantom{(}}} \drtwocell_{}\omit^{}\omit{^{F\,\,\,}} \ar[r]_(0.45)\boo \ar@{^{(}->}[d]^{\inc} & *+[l]{\Ztor} \\
*+[r]{\Bord_n^{\bs\Spin}(K(\Z,4))} \ar[r]^(0.6)\sH & *+[l]{\sZtor.} \ar[u]^{\text{forget $\Z_2$-grading}} }
\end{gathered}
\label{fm1eq8}
\e
Having chosen a flag structure on $X$, we define $\eta_X$ in \eq{fm1eq7} by pullback along $\sT$.

This is related to previous work of the authors \cite{Joyc5,JoUp}. The first author \cite[\S 3]{Joyc5} defined a notion of flag structure on a compact 7-manifold $X$, and showed that if $(X,\vp,g)$ is a compact $G_2$-manifold then a flag structure on $X$ induces orientations on moduli spaces $\M_\al^{\rm ass}$ of associative 3-folds in $X$, as in (c) above. Then the authors \cite{JoUp} showed that a flag structure on $X$ induces orientations on moduli spaces $\M_P^{G_2}$ of $G_2$-instantons in $X$ when $G=\SU(m)$ or $\U(m)$, as in (a) above. We show in Theorem \ref{fm10thm1} that the notions of flag structure in \cite[\S 3]{Joyc5} and \eq{fm1eq8} are equivalent when~$n=7$. 

Equation \eq{fm1eq8} provides a way to define flag structures in 8 dimensions, when the ideas of \cite{Joyc5} do not apply. We show in Theorem \ref{fm10thm2} that a compact spin 8-manifold $X$ admits a flag structure if and only if the following condition holds:
\begin{itemize}
\setlength{\itemsep}{0pt}
\setlength{\parsep}{0pt}
\item[$(*)$] There does not exist a class $\al\in H^3(X,\Z)$ such that $\int_X\bar\al\cup\Sq^2(\bar\al)=\ul{1}$ in $\Z_2,$ where $\bar\al\in H^3(X,\Z_2)$ is the mod 2 reduction of $\al,$ and $\Sq^2(\bar\al)$ in $H^5(X,\Z_2)$ is its Steenrod square.
\end{itemize}
Then the set of flag structures $F$ on $X$ is a torsor for $\Map(H^4(X,\Z),\Z_2)$.

As defined in \eq{fm1eq8}, a flag structure $F$ includes an independent $\Z_2$ choice for each $\al\in H^4(X,\Z)$, where the choice is canonical for $\al=0$. However, by imposing extra conditions on $F$, e.g.\ by requiring $F$ to factor via $\Bord_n^{\bs\Spin}(K(\Z,4))\ra\Bord_n^{\bs\Spin}(K(\Z_2,4))$, we can reduce this to a finite choice.

To see how flag structures make choices of orientations more canonical, consider the case $G=\U(m)$. Then the $\Z_2$-choice in $F$ in \eq{fm1eq8} for $\al\in H^4(X,\Z)$ determines the $\Z_2$-choices in $\eta_X$ in \eq{fm1eq7} for all principal $\U(m)$-bundles $P\ra X$ with $c_2(P)-c_1(P)^2=\al$, independently of the other Chern classes~$c_i(P)$.

We develop our theory of bordism categories and orientations of moduli spaces in \S\ref{fm4}--\S\ref{fm11}, with some proofs deferred to \S\ref{fm19}. Sections \ref{fm12}--\ref{fm14} give applications of our theory to problems (a)--(e) above. Theorems \ref{fm1thm3}, \ref{fm1thm4}, \ref{fm1thm5} and \ref{fm1thm6} summarize some main results from these sections. See \S\ref{fm10} and \S\ref{fm12}--\S\ref{fm14} for definitions and notation.

\begin{thm}
\label{fm1thm3}
Let\/ $G$ be a compact Lie group on the list
\e
E_8,\; E_7,\; E_6,\; G_2,\; \Spin(3),\; \SU(m),\; \U(m),\; \Spin(2m), \quad \text{for $m\ge 1$.}
\label{fm1eq9}
\e

\noindent{\bf(a)} Let\/ $(X,g)$ be a compact Riemannian spin $7$-manifold. Then $\B_P$ is orientable \textup(i.e.\ the principal\/ $\Z_2$-bundle $O_P\ra\B_P$ is trivializable\/\textup) for all principal\/ $G$-bundles $P\ra X$. A choice of orientation for $\det\sD_X\cong\R$ and a flag structure on $X$ in the sense of\/ {\rm\S\ref{fm101}} determine orientations on $\B_P$ for all\/ $P\ra X$.
\smallskip

\noindent{\bf(b)} Let\/ $(X,g)$ be a compact Riemannian spin $8$-manifold which satisfies $(*)$ above. Then $\B_P$ is orientable for all principal\/ $G$-bundles $P\ra X$. A choice of orientation for $\det\sD^+_X\cong\R$ and a flag structure on $X$ in the sense of\/ {\rm\S\ref{fm102}} determine orientations on $\B_P$ for all\/~$P\ra X$.
\end{thm}

\begin{rem}
\label{fm1rem1}
{\bf(i)} Let $(X,g)$ be a compact Riemannian spin 8-manifold. When $G=E_8$, we show that condition $(*)$ holds for $X$ {\it if and only if\/} (not just only if) $\B_P$ is orientable for all principal\/ $E_8$-bundles $P\ra X$.
\smallskip

\noindent{\bf(ii)} The compact 8-manifold $X=\SU(3)$ does not satisfy condition $(*)$. We show in Example \ref{fm12ex1} that $\B_P$ is {\it not\/} orientable when $P\ra X$ is the trivial $G$-bundle for $G$ any of $\SU(m),\U(m),$ or $\Spin(2m)$ for $m\ge 3$, $G_2,E_6,E_7,$ or~$E_8$.

\smallskip

\noindent{\bf(iii)} Let $X$ be a compact spin 8-manifold, not necessarily satisfying condition $(*)$, and $P\ra X$ be a principal $G$-bundle for $G=\SU(m)$ or $\U(m)$. Cao--Gross--Joyce \cite[Th.~1.11]{CGJ} claimed to prove $\B_P$ is orientable. As in {\bf(ii)}, this is {\it false\/} for $X=\SU(3)$, and there is a mistake in the proof of \cite[Th.~1.11]{CGJ}, which is explained in Remark \ref{fm12rem3} below. The first author would like to apologize for this mistake. One of the goals of this monograph was to fix the problems with \cite{CGJ} under additional conditions on~$X$.
\smallskip

\noindent{\bf(iv)} We can also use our theory to answer orientability questions for Lie groups $G$ not in \eq{fm1eq9}. For example, we show in Example \ref{fm12ex2} that if $X$ is the compact spin 7-manifold $\Sp(2)\t_{\Sp(1)\t\Sp(1)}\Sp(1)$, so that $X\t\cS^1$ is a compact spin 8-manifold, then $\B_{X\t G}$ and $\B_{X\t\cS^1\t G}$ are not orientable for the trivial $G$-bundles $X\t G\ra X$, $X\t\cS^1\t G\ra X\t\cS^1$ for any Lie group $G$ on the list
\begin{equation*}
F_4, \quad \Sp(m+1), \quad  \Spin(2m+3), \quad \SO(2m+3), \quad \text{where $m\ge 1.$}
\end{equation*}
\end{rem}

\begin{thm}
\label{fm1thm4}
{\bf(a)} Let\/ $(X,\vp,g)$ be a compact\/ $G_2$-manifold with\/ $\d(*\vp)=0,$ and let\/ $G$ be any of the Lie groups in the list\/ \eq{fm1eq9}. Then the $G_2$-instanton moduli space $\M_P^{G_2}$ is orientable for all principal\/ $G$-bundles $P\ra X,$ and a choice of orientation for $\det\sD_X\cong\R$ and a flag structure on $X$ in the sense of\/ {\rm\S\ref{fm101}} determine orientations on $\M_P^{G_2}$ for all\/~$P\ra X$.

If\/ $(\vp,g)$ is torsion-free then $\det\sD_X$ has a canonical orientation.
\smallskip

\noindent{\bf(b)} Let\/ $(X,\vp,g)$ be a compact\/ $G_2$-manifold. Then the moduli space of associative $3$-folds $\M_\al^\ass$ is orientable for all\/ $\al\in \La_3^{\SO(4)}(X),$ and a choice of flag structure $F$ on $X$ determines orientations on $\M_\al^\ass$ for all\/~$\al$.
\end{thm}

\begin{rem}
\label{fm1rem2}
Theorem \ref{fm1thm4}(a) was already known for $G=\SU(m)$ and $\U(m)$ by the authors \cite[Cor.~1.4]{JoUp}. Walpuski \cite[\S 6.1]{Walp1} earlier proved orientability for $G=\SU(m)$. Theorem \ref{fm1thm4}(b) was already known when $\M_\al^\ass$ is unobstructed by the first author \cite[\S 3.2]{Joyc5}. In both cases we provide new proofs.
\end{rem}

\begin{rem}
\label{fm1rem3}
{\bf(On Floer gradings.)}
We can also use our theory to study generalizations of orientations. For example, if $n\equiv 7\mod 8$ then the material of \S\ref{fm92}--\S\ref{fm93} implies that principal $\Z_2$-bundle $O_P\ra\B_P$ above is the $\Z_2$-reduction of a natural principal $\Z$-bundle $O_P^\Z\ra\B_P$. This $O_P^\Z$ encodes information about counting how many eigenvalues of twisted Dirac operators $\sD_X\ot\ad(\nabla)$ cross zero as one deforms $\nabla$ in a 1-parameter family. Write $O_P^{\Z_k}\ra\B_P$ for the $\Z_k$-reduction of $O_P^\Z\ra\B_P$ for $k\ge 2$, a principal $\Z_k$-bundle, with~$O_P^{\Z_2}=O_P$.

Our theory also gives criteria for when the principal $\Z$- or $\Z_k$-bundles $O_P^\Z$ or $O_P^{\Z_k}$ are trivializable, essentially the same as Theorem \ref{fm1thm2}, but using morphisms $\pi_1(\sO)$ to $\Z$ or $\Z_k$ instead of~$\Z_2$.

Results of this kind may have applications to {\it gradings of Floer theories}. As in Donaldson \cite{Dona}, if $X$ is a compact oriented homology 3-sphere, one can define the {\it instanton Floer homology\/} $HF_*(X)$. Roughly speaking, this is the homology of a chain complex $(CF_*(X),\pd)$ in which the generators are flat connections on principal $G$-bundles $P\ra X$, usually for $G=\U(2)$ or $\SO(3)$, and the differential $\pd$ is defined by counting self-dual instantons on a principal $G$-bundle $Q\ra X\t\R$ which are asymptotic at $\pm\iy$ in $\R$ to given flat connections $\nabla_\pm$ on $X$. The issue of {\it grading\/} Floer theories is whether for $HF_k(X)$ the index $k$ should lie in $\Z$, or $\Z_k$, or $\Z_2$ (or more generally in a torsor over $\Z,\Z_k$ or~$\Z_2$).

Following Donaldson--Thomas \cite{DoTh} and Donaldson--Segal \cite{DoSe}, if $(X,\vp,g)$ is a compact $G_2$-manifold, one can imagine trying to define a Floer theory in which the Floer complex is generated by $G_2$-instantons on a compact $G_2$-manifold $(X,\vp,g)$, and the differentials obtained by counting $\Spin(7)$-instantons on $X\t\R$. There are of course serious analytic difficulties. Our theory is relevant to gradings of such a Floer theory, if it exists, since trivializing $O_P^\Z\ra\B_P$ or $O_P^{\Z_k}\ra\B_P$ would induce a grading of the Floer theory over $\Z$ or~$\Z_k$.

In fact, we have proved a {\it negative result\/}: Corollary \ref{fm11cor1} implies that for the groups $G=\SU(m),$ $\Sp(m)$ for $m\ge 2$ and $E_8$, Theorem \ref{fm1thm3}(a) does not generalize to $\Z$- or $\Z_k$-gradings for any $k>2$. That is, for any $G$ on this list, there exists a compact spin 7-manifold $X$ and a principal $G$-bundle $P\ra X$ such that $O_P^\Z\ra\B_P$ or $O_P^{\Z_k}\ra\B_P$ is not trivializable.
\end{rem}

The analogue of Theorem \ref{fm1thm4} for $\Spin(7)$-instantons is the following:

\begin{thm}
\label{fm1thm5}
Let\/ $(X,\Om,g)$ be a compact\/ $\Spin(7)$-manifold satisfying condition $(*)$ above. Then:
\smallskip

\noindent{\bf(a)} Let\/ $G$ be any of the Lie groups in the list\/ \eq{fm1eq9}. Then the $\Spin(7)$-instanton moduli space $\M_P^{\Spin(7)}$ is orientable for all principal\/ $G$-bundles $P\ra X,$ and a choice of orientation for $\det\sD^+_X\cong\R$ and a flag structure on $X$ in the sense of\/ {\rm\S\ref{fm102}} determine orientations on $\M_P^{\Spin(7)}$ for all\/~$P\ra X$.

If\/ $(\Om,g)$ is torsion-free then $\det\sD^+_X$ has a canonical orientation.
\smallskip

\noindent{\bf(b)} The moduli spaces $\M_\al^{\Cay,\Spin(4)}$ of \begin{bfseries}spin\end{bfseries} Cayley $4$-folds $N$ in $X$ are orientable for all\/ $\al\in\La_4^{\Spin(4)}(X)$. A choice of flag structure $F$ on $X$ determines orientations on $\M_\al^{\Cay,\Spin(4)}$ for all\/~$\al\in\La_4^{\Spin(4)}(X)$.
\end{thm}

\begin{rem}
\label{fm1rem4}
Cao--Gross--Joyce \cite[Cor.~1.12]{CGJ} claimed to prove orientability in (a) for $G=\SU(m)$ or $\U(m)$, without requiring $X$ to satisfy condition $(*)$. However, the proof relied on \cite[Th.~1.11]{CGJ}, which is false as in Remark~\ref{fm1rem1}(iii).
\end{rem}

\begin{thm}
\label{fm1thm6}
Let\/ $X$ be a Calabi--Yau\/ $4$-fold, and suppose that\/ $X$ satisfies condition $(*)$. Then the moduli stacks\/ $\M$ of all objects in $D^b\coh(X),$ and\/ $\M_\al\subset\M$ of objects $F^\bu$ in $D^b\coh(X)$ with $\lb F^\bu\rb=\al\in K^0_\top(X),$ are orientable in the sense of Definition\/ {\rm\ref{fm13def2}}. A choice of flag structure $F$ on $X$ in the sense of\/ {\rm\S\ref{fm102}} determines an  orientation on the moduli stacks $\M,\M_\al$. Such orientations are necessary for defining \begin{bfseries}DT4 invariants\end{bfseries} of\/ $X,$ as in Borisov--Joyce\/ {\rm\cite{BoJo}} and Oh--Thomas\/ {\rm\cite{OhTh}}.

If\/ $c_2(\al)-c_1(\al)^2=0$ in $H^4(X,\Z)$ then we can construct a canonical orientation on $\M_\al$ without choosing a flag structure.
\end{thm}

\begin{rem}
{\bf(i)} Cao--Gross--Joyce \cite[Cor.~1.17]{CGJ} claimed to prove orientability in Theorem \ref{fm1thm6} without requiring $X$ to satisfy condition $(*)$. However, the proof relied on \cite[Th.~1.11]{CGJ}, which is false as in Remark~\ref{fm1rem1}(iii).
\smallskip

\noindent{\bf(ii)} As explained in \S\ref{fm13}, the last part of Theorem \ref{fm1thm6} resolves an apparent paradox in the literature on DT4 invariants. In the papers \cite{Bojk2,Cao1,Cao2,CaKo1,CaKo2,CKM,CaLe,CMT1,CMT2,COT1,COT2,CaQu,CaTo1,CaTo2,CaTo3} one can find many (often conjectural) relations of the form
\e
\text{Conventional invariants of $X$}\simeq \text{DT4 invariants of $X$,}
\label{fm1eq10}
\e
where by `conventional invariants' of $X$ we mean things like the Euler characteristic and Gromov--Witten invariants, and the relation `$\simeq$' may involve change of variables in a generating function, etc. 

Now the left hand side of \eq{fm1eq10} does not involve orientations, but the right hand side needs a choice of orientations on the moduli spaces $\M_\al$ to determine the signs of the DT4 invariants. All relations \eq{fm1eq10} in \cite{Bojk2,Cao1,Cao2,CaKo1,CaKo2,CKM,CaLe,CMT1,CMT2,COT1,COT2,CaQu,CaTo1,CaTo2,CaTo3} involve only moduli spaces $\M_\al$ on $X$ with $c_1(\al)=c_2(\al)=0$ in $H^*(X,\Z)$. These have canonical orientations without arbitrary choices, resolving the paradox.
\smallskip

\noindent{\bf(iii)} The authors do not know of any examples of either compact 8-manifolds with holonomy $\Spin(7)$, or Calabi--Yau 4-folds, for which condition $(*)$ does not hold. It would be interesting to know if such examples exist.
\end{rem}

\noindent{\it Acknowledgements.} This research was partly funded by a Simons Collaboration Grant on `Special Holonomy in Geometry, Analysis and Physics'. The second author thanks Mark Grant for several useful discussions.

For the purpose of open access, the authors have applied a CC BY public copyright licence to any Author Accepted Manuscript (AAM) version arising from this submission.

\section{Background from Algebraic Topology}
\label{fm2}

\subsection{Bordism theory}
\label{fm21}

\subsubsection{Tangential structures}
\label{fm211}

To define bordism groups with different flavours, we first define {\it tangential structures}. Our treatment is based on Lashof \cite{Lash} and Stong~\cite{Ston1}.

\begin{dfn}
\label{fm2def1}
Let $B\O=\colim_{n\ra\iy}B\O(n)$ be the classifying space of the stable orthogonal group, the direct limit of the classifying spaces $B\O(n)$ of the orthogonal groups $\O(n)$ under the inclusions $\O(n)\hookra\O(n+1)$. There are natural continuous maps $\io_{B\O(n)}:B\O(n)\ra B\O$ coming from the direct limit.

The inclusions $\O(m)\t\O(n)\ra\O(m+n)$ induce a binary operation $\mu_{B\O}:B\O\t B\O\ra B\O$ which is up to homotopy associative, unital, and commutative. Hence $B\O$ is a commutative H-space.

If $X$ is a smooth $n$-manifold (possibly with boundary or corners) then choosing a Riemannian metric $g$ on $X$ gives $TX$ an $\O(n)$-structure, so we have a classifying map $\phi_{TX}:X\ra B\O(n)$, which we compose with $\io_{B\O(n)}:B\O(n)\ra B\O$ to get a map $\phi^\rst_{TX}:X\ra B\O$ classifying the {\it stable tangent bundle} of $X$. Up to contractible choice this is unique and independent of the Riemannian metric $g$ on~$X,$ which permits us to fix the choice of $\phi^\rst_{TX}$ below. We have a homotopy commutative diagram
\begin{equation*}
\xymatrix@C=60pt@R=15pt{ X \ar[d]^{\phi_{TX}} \ar[dr]^(0.7){\phi_{TX\op\ul{\R}}} \ar@/^.8pc/[drr]^(0.7){\phi^\rst_{TX}} \\
B\O(n) \ar[r] & B\O(n+1) \ar[r]^{\io_{B\O(n+1)}} & B\O, }	
\end{equation*}
where $\ul{\R}\ra X$ is the trivial line bundle. The vector bundle isomorphism $\id_{TX}\op -\id_{\ul\R}:TX\op\ul{\R}\ra TX\op\ul{\R}$ induces a homotopy $-1_{\phi_{TX\op\ul{\R}}}:\phi_{TX\op\ul{\R}}\Ra\phi_{TX\op\ul{\R}}$ whose square is homotopic to the constant homotopy $\Id_{\phi_{TX\op\ul{\R}}}$. We define $-1_{\phi^\rst_{TX}}:\phi^\rst_{TX}\Ra\phi^\rst_{TX}$ to be the horizontal composition of this with~$\Id_{\io_{B\O(n)}}$.	

If $X$ has boundary or corners then $TX\vert_{\pd X}\cong T(\pd X)\op\ul{\R}$, where $\ul{\R}\ra X$ is the trivial line bundle. Thus we have a homotopy commutative diagram
\begin{equation*}
\xymatrix@C=120pt@R=15pt{
*+[r]{\pd X} \drtwocell_{}\omit^{}\omit{^{}} \ar[r]_(0.35){i_{\pd X}} \ar[d]^{\phi_{T\pd X}}  & *+[l]{X} \ar[d]_{\phi_{TX}} \\
*+[r]{B\O(n-1)} \ar[r] & *+[l]{B\O(n).\!} }
\end{equation*}
Composing with $B\O(n)\ra B\O$ shows that
\e
\phi^\rst_{TX}\vert_{\pd X}\simeq \phi^\rst_{T(\pd X)}.
\label{fm2eq1}
\e

A {\it tangential structure\/} $\bs B=(B,\be)$ is a topological space $B$ and a continuous map $\be:B\ra B\O$. We say that $\bs B$ {\it has products} if we are given a continuous map $\mu_{\bs B}:B\t B\ra B$, which is homotopy commutative and associative, in a homotopy commutative diagram
\begin{equation*}
\xymatrix@C=120pt@R=15pt{
*+[r]{B\t B} \ar[r]_(0.35){\mu_{\bs B}} \ar[d]^{\be\t\be} \drtwocell_{}\omit^{}\omit{^{\eta_{\bs B}\,\,\,\,}} & *+[l]{B} \ar[d]_{\be} \\
*+[r]{B\O\t B\O} \ar[r]^(0.65){\mu_{B\O}} & *+[l]{B\O.\!} }
\end{equation*}

A $\bs B$-{\it structure\/} $\bs\ga_X=(\ga_X,\eta_X)$ on a smooth manifold $X$ (possibly with boundary or corners) is a homotopy commutative diagram of continuous maps
\e
\begin{gathered}
\xymatrix@C=70pt@R=13pt{
\drrtwocell_{}\omit^{}\omit{_{\,\,\,\,\eta_X}} & B \ar[dr]^\be  \\
X \ar[ur]^{\ga_X} \ar[rr]^(0.35){\phi_X^\rst} && B\O.}
\end{gathered}
\label{fm2eq2}
\e

An {\it isomorphism} of tangential structures $\bs\ga_X=(\ga_X,\eta_X)$ and $\bs\ga_X'=(\ga_X',\eta_X')$ is represented by a homotopy $\eta:\ga_X\Ra\ga_X'$ such that the diagram
\begin{equation*}
\begin{tikzcd}
 \be\circ\ga_X\arrow[rd,Rightarrow,"\eta_X"]\arrow[rr,Rightarrow,"\Id_\be\circ\eta"] && \be\circ\ga_X'\arrow[ld,Rightarrow,"\eta_X'"]\\
 &\phi_X^\rst	
\end{tikzcd}
\end{equation*}
commutes up to homotopy (of homotopies). Here we only care about $\eta$ up to homotopy and often we will only care about isomorphism classes of $\bs B$-structures.

The {\it opposite $\bs B$-structure\/} $-\bs\ga_X$ is obtained by composing homotopies across the diagram
\begin{equation*}
\xymatrix@C=70pt@R=13pt{
\drrtwocell_{}\omit^{}\omit{_{\,\,\,\,\eta_X}} & B \ar[dr]^\be  \\
X \drrtwocell_{}\omit^{}\omit{_{\qquad -1_{\phi^\rst_{TX}}}} \ar@/_2pc/[rr]_(0.15){\phi_X^\rst} \ar[ur]^{\ga_X} \ar[rr]_(0.3){\phi_X^\rst} && B\O. \\
&& }
\end{equation*}

Often we just write a manifold with $\bs B$-structure $(X,\bs\ga_X)$ as $X$, omitting $\bs\ga_X$ from the notation. In this case we write $-X$ as a shorthand for $(X,-\bs\ga_X)$, that is, $X$ with the opposite $\bs B$-structure.

From \eq{fm2eq1} we see that if $X$ has boundary or corners then composing \eq{fm2eq2} with $i_X:\pd X\ra X$ gives a restriction $\bs\ga_X\vert_{\pd X}$ which is a $\bs B$-structure on~$\pd X$.

If $\bs B=(B,\be)$, $\bs B'=(B',\be')$ are tangential structures, we say that $\bs B$ {\it factors through\/} $\bs B'$ if there is a homotopy commutative diagram
\e
\begin{gathered}
\xymatrix@C=70pt@R=13pt{
\drrtwocell_{}\omit^{}\omit{_{}} & B' \ar[dr]^{\be'}  \\
B \ar[ur] \ar[rr]^(0.35){\be} && B\O.}
\end{gathered}
\label{fm2eq3}
\e
Composing with this diagram, a $\bs B$-structure on $X$ induces a $\bs B'$-structure.
\end{dfn}

Here are some examples, including well known geometric structures such as orientations and spin structures. 

\begin{ex}
\label{fm2ex1}
{\bf(a)} The {\it orthogonal tangential structure\/} is $\bs{\O}=(B\O,\id_{B\O})$. Every manifold $X$ has an $\bs{\O}$-structure unique up to homotopy.
\smallskip

\noindent{\bf(b)} The {\it special orthogonal tangential structure\/} is $\bs{\SO}=(B\SO,\be_\SO)$, where $B\SO=\colim_{n\ra\iy}B\SO(n)$ and $\be_\SO:B\SO\ra B\O$ is induced by the inclusions $\SO(n)\hookra\O(n)$. A $\bs{\SO}$-structure on $X$ is equivalent to an orientation on $X$. The opposite $\bs{\SO}$-structure is equivalent to the opposite orientation.
\smallskip

\noindent{\bf(c)} The {\it spin tangential structure\/} is $\bs{\Spin}\!=\!(B\Spin,\be_\Spin)$, where $B\Spin\!=\!\colim_{n\ra\iy}\ab B\Spin(n)$ and $\be_\Spin:B\Spin\ra B\O$ is induced by $\Spin(n)\ra\O(n).$ A $\bs{\Spin}$-structure on $X$ is equivalent to an orientation and a spin structure.
\smallskip

\noindent{\bf(d)} The {\it spin$^c$ tangential structure\/} is $\bs{\Spinc}=(B\Spinc,\be_{\Spinc})$, for $B\Spinc=\colim_{n\ra\iy}\ab B\Spinc(n)$ and $\be_{\Spinc}:B\Spinc\ra B\O$ induced by $\Spinc(n)\ra\O(n)$. A $\bs{\Spinc}$-structure on $X$ amounts to an orientation and a spin$^{\rm c}$ structure.
\smallskip

\noindent{\bf(e)} The {\it unitary tangential structure\/} is $\bs{\U}=(B\U,\be_\U)$, where $B\U=\colim_{m\ra\iy}\ab B\U(m)$ and $\be_\U:B\U\ra B\O$ is induced by the commutative diagram
\begin{equation*}
\xymatrix@C=20pt@R=15pt{
\cdots \ar[r] & \U(m) \ar[r] \ar[d] & \U(m) \ar[r] \ar[d] & \U(m+1) \ar[r] \ar[d] & \U(m+1) \ar[r] \ar[d] & \cdots \\ 
\cdots \ar[r] & \O(2m) \ar[r]  & \O(2m+1) \ar[r] & \O(2m+2) \ar[r] & \O(2m+3) \ar[r] & \cdots.\! }	
\end{equation*}
A $\bs{\U}$-structure on $X$ is equivalent to a {\it stable almost complex structure\/} on $X.$
\smallskip

\noindent{\bf(f)} The {\it special unitary tangential structure\/} is $\bs{\SU}=(B\SU,\be_\SU)$, where $B\SU=\colim_{m\ra\iy}\ab B\SU(m)$ and $\be_\SU:B\SU\ra B\O$ is defined as in {\bf(e)}.
\smallskip

\noindent{\bf(g)} The {\it quaternionic tangential structure\/} is $\bs{\Sp}=(B\Sp,\be_\Sp)$, where $B\Sp=\colim_{m\ra\iy}\ab B\Sp(m)$ and $\be_\Sp:B\Sp\ra B\O$ is defined in a similar way to~{\bf(e)}.
\smallskip

All of the tangential structures in {\bf(a)}--{\bf(g)} have products. Also {\bf(b)}--{\bf(g)} factor through $\bs{\SO}$, but {\bf(a)} does not.
\end{ex}

\subsubsection{Bordism, as a generalized homology theory}
\label{fm212}

Let $\bs B$ be a tangential structure. Then $\bs B$-{\it bordism\/} $\Om_*^{\bs B}(-)$ is a generalized homology theory of topological spaces $T$, in which the `$n$-chains' are continuous maps $f:X\ra T$ for $X$ a compact $n$-manifold with a $\bs B$-structure. The subject began with the work of Thom \cite{Thom}. Bordism was introduced by Atiyah \cite{Atiy}, and good references are Conner \cite[\S I]{Conn} and Stong \cite{Ston1}.

\begin{dfn}
\label{fm2def2}
Let $\bs B$ be a tangential structure, $T$ be a topological space, and $n\in\N$. Consider triples $(X,\bs\ga_X,f)$, where $X$ is a compact manifold with $\dim X=n$, $\bs\ga_X$ is a $\bs B$-structure on $X$, and $f:X\ra T$ is a continuous map. Given two such triples, a {\it bordism\/} from $(X_0,\bs\ga_{X_0},f_0)$ to $(X_1,\bs\ga_{X_1},f_1)$ is a triple $(W,\bs\ga_W,e)$, where:
\begin{itemize}
\setlength{\itemsep}{0pt}
\setlength{\parsep}{0pt}
\item[(i)] $W$ is a compact $(n+1)$-manifold with boundary, with a given identification $\pd W\cong X_0\amalg X_1$.
\item[(ii)] $\bs\ga_W$ is a $\bs B$-structure on $W$, with a given isomorphism of $\bs B$-structures $\bs\ga_W\vert_{\pd W}\ab\cong -\bs\ga_{X_0}\amalg\bs\ga_{X_1}$ on $\pd W\cong X_0\amalg X_1.$
\item[(iii)] $e:W\ra T$ is a continuous map such that $e\vert_{\pd W}\cong f_0\amalg f_1$ under the identification $\pd W\cong X_0\amalg X_1.$
\end{itemize}

Write $(X_0,\bs\ga_{X_0},f_0)\sim(X_1,\bs\ga_{X_1},f_1)$ if there exists such a bordism $(W,\bs\ga_W,e)$. Then `$\sim$' is an equivalence relation, called $\bs B$-{\it bordism}, and the equivalence class $[X,\bs\ga_X,f]$ is called a $\bs B$-{\it bordism class}. The $n^{\it th}$ $\bs B$-{\it bordism group\/} $\Om^{\bs B}_n(T)$ is the set of $\bs B$-bordism classes $[X,\bs\ga_X,f]$ with $\dim X=n,$ where the group structure has zero element $0_T=[\es,\es,\es],$ addition $[X,\bs\ga_X,f]+[X',\bs\ga_{X'},f']=[X\amalg X',\bs\ga_X\amalg\bs\ga_{X'},f\amalg f'],$ and inverse $-[X,\bs\ga_X,f]=[X,-\bs\ga_X,f].$

When $T$ is a point we may omit the necessarily constant map $f$ from the notation, and write elements of $\Om^{\bs B}_n(*)$ as~$[X,\bs\ga_X]$.

If $T$ is a smooth manifold then, as smooth maps are dense in continuous maps, we can take $f:X\ra T$ and $e:W\ra F$ above to be smooth.

Now suppose that $\bs B'$ is another tangential structure and that $\bs B$ factors through $\bs B'$ as in \eq{fm2eq3} of Definition \ref{fm2def1}. Then a $\bs B$-structure $\bs\ga_X$ on a manifold $X$ induces a $\bs B'$-structure $\Pi_{\bs B}^{\bs B'}(\bs\ga_X)$ on $X$. This defines a group morphism
\begin{equation*}
\Om^{\bs B}_n(T)\overset{\Pi_{\bs B}^{\bs B'}}{\longra}\Om_n^{\bs B'}(T),\qquad[X,\bs\ga_X,f]\longmapsto\bigl[X,\Pi_{\bs B}^{\bs B'}(\bs\ga_X),f\bigr].
\end{equation*}

If $\io: U\hookrightarrow T$ is a subspace we can define {\it relative bordism groups\/} $\Om^{\bs B}_n(T;U)$, whose elements $[X,\bs\ga_X,f]$ are bordism classes of triples $(X,\bs\ga_X,f)$ with $X$ a compact $n$-manifold with boundary and $\bs B$-structure $\bs\ga_X$, and $f:X\ra T$ a continuous map with $f(\pd X)\subseteq U\subseteq T$. These fit into a long exact sequence
\begin{equation*}
\xymatrix@C=21pt{ \cdots \ar[r] & \Om^{\bs B}_n(U) \ar[r]^{\io_*} & \Om^{\bs B}_n(T) \ar[r]^(0.45){\pi_*} & \Om^{\bs B}_n(T;U) \ar[r]^(0.53)\pd & \Om_{n-1}^{\bs B}(U) \ar[r] & \cdots. }
\end{equation*}

If $T$ is path-connected we define the {\it reduced bordism groups\/} to be $\ti\Om^{\bs B}_n(T)=\Om^{\bs B}_n(T;\{t_0\})$, for $t_0\in T$ any base point. As the inclusion of $U=\{t_0\}$ has a left inverse, the long exact sequence reduces to short exact sequences
\begin{equation*}
\xymatrix@C=30pt{ 0 \ar[r] & \Om^{\bs B}_n(*) \ar[r]^{\io_*} & \Om^{\bs B}_n(T) \ar[r]^{\pi_*} & \ti\Om^{\bs B}_n(T) \ar[r] & 0. }
\end{equation*}
\end{dfn}

\begin{table}[htb]
\centerline{\begin{tabular}{|l|l|l|l|l|l|l|l|l|l|l|}
\hline
 & $n\!=\!0$ & $n\!=\!1$ & $n\!=\!2$ & $n\!=\!3$ & $n\!=\!4$ & $n\!=\!5$ & $n\!=\!6$ & $n\!=\!7$ & $n\!=\!8$ & $n\!=\!9$ \\
\hline
\vphantom{$\bigl(^(_($}$\Om_n^{\bs\SO}(*)$ & $\Z$ & 0 & 0 & 0 & $\Z$ & $\Z_2$ & 0 & 0 & $\Z^2$ & \\
\hline
\vphantom{$\bigl(^(_($}$\Om_n^{\bs\O}(*)$ & $\Z_2$ & 0 & $\Z_2$ & 0 & $\Z_2^2$ & $\Z_2$ & $\Z_2^3$ & $\Z_2$ & $\Z_2^5$ & \\
\hline
\vphantom{$\bigl(^(_($}$\Om_n^{\bs\Spin}(*)$ & $\Z$ & $\Z_2$ & $\Z_2$ & 0 & $\Z$ & 0 & 0 & 0 & $\Z^2$ & $\Z_2^2$ \\
\hline
\vphantom{$\bigl(^(_($}$\Om_n^{\bs\Spinc}(*)$ & $\Z$ & 0 & $\Z$ & 0 & $\Z^2$ & 0 & $\Z^2$ & 0 & $\Z^4$ & 0 \\
\hline
\vphantom{$\bigl(^(_($}$\Om_n^{\bs\U}(*)$ & $\Z$ & 0 & $\Z$ & 0 & $\Z^2$ & 0 & $\Z^3$ & 0 & $\Z^5$ & \\
\hline
\vphantom{$\bigl(^(_($}$\Om_n^{\bs\SU}(*)$ & $\Z$ & $\Z_2$ & $\Z_2$ & 0 & $\Z$ & 0 & $\Z$ & 0 & $\Z^2$ & \\
\hline
\end{tabular}}
\caption{$\bs B$-bordism groups of the point}
\label{fm2tab1}
\end{table}

As $\bs B$-bordism $\Om_*^{\bs B}(-)$ is a generalized homology theory, as in \S\ref{fm22} there is an {\it Atiyah--Hirzebruch spectral sequence\/} $H_p(T,\Om^{\bs B}_q(*))\Ra\Om^{\bs B}_{p+q}(T)$, where $*$ is the point. If $T$ is path-connected then as the splitting $\Om^{\bs B}_*(T)=\Om^{\bs B}_*(*)\op\ti\Om^{\bs B}_*(T)$ is functorial, this induces a spectral sequence $\ti H_p(T,\Om^{\bs B}_q(*))\Ra\ti\Om^{\bs B}_{p+q}(T)$. Thus, a lot of important behaviour of bordism depends on the bordism groups $\Om^{\bs B}_n(*)$ of the point, so much effort has gone into calculating these. Table \ref{fm2tab1} gives values of $\Om^{\bs B}_n(*)$ for $\bs B=\bs{\SO},\bs{\O},\bs{\Spin},\bs{\Spinc},\bs{\U},\bs{\SU}$ and $n\le 9$, which are taken from Stong \cite{Ston1}, Anderson, Brown and Peterson \cite{ABP}, and Gilkey \cite[pp.~330, 333]{Gilk}. Note that we omit the letter $\bs{\rm B}$ from the notation for the bordism group for the classical tangential structures defined in Example \ref{fm2ex1}.

\begin{table}[htb]	
\centerline{\begin{tabular}{|l|l|l|l|l|l|l|}
\hline
$n$ & $0$ & $1$ & $2$ & $4$  & $8$ & $9$  \\
\hline
\parbox[top][4ex][c]{1.3cm}{$\Om^{\bs\Spin}_n(*)$\!\!\!} & $\Z\an{1}$ & $\Z_2\an{\al_1}$ & $\Z_2\an{\al_1^2}$ & $\Z\an{\al_4}$ & $\Z\an{\al_8,\al_8'}$ & $\Z_2\an{\al_1\al_8,\al_1\al_8'}$ \\
\hline
\end{tabular}}
\smallskip

\centerline{$\al_1=[\cS^1_{\rm nb}],$ $\al_4=[K3]$, $\al_8=[\HP^2]$, $\al_8'=[(K3\t K3)/\Z_2^2]$}
\caption{Explicit presentation of $\Om^{\bs\Spin}_n(*)$, $n\le 9$}
\label{fm2tab2}
\end{table}

A presentation of $\Om^{\bs\Spin}_n(*)$, $n\le 9$ is given in Table \ref{fm2tab2}, where 
$\cS^1_{\rm nb}$ is $\cS^1$ with the {\it non-bounding\/} spin structure, i.e.\ the unique spin structure that does not extend from $\cS^1=\pd D^2$ to~$D^2$.

\subsubsection{Free loop spaces, based loop spaces, and their bordism}
\label{fm213}

\begin{dfn}
\label{fm2def3}
Let $T$ be a topological space, which we suppose is path-connected, with a basepoint $t_0\in T$. Write $\cS^1=\R/\Z$, with basepoint $\ul 0=0+\Z$. The {\it free loop space\/} of $T$ is $\cL T=\Map_{C_0}(\cS^1,T)$, with the compact-open topology. Points of $\cL T$ are continuous maps $\ga:\cS^1\ra T$, that is, loops in $T$. The {\it based loop space\/} of $T$ is $\Om T=\Map_{C_0}((\cS^1,\ul{0}),(T,t_0))$. Points of $\Om T$ are continuous maps $\ga:\cS^1\ra T$ with $\ga(\ul{0})=t_0$, that is, based loops in $T$. 

Mapping $\ga\mapsto\ga(\ul{0})$ defines an evaluation map $\ev_{\ul{0}}:\cL T\ra T$, with $\ev_{\ul{0}}^{-1}(0)=\Om T$. It is a homotopy fibration, with fibre $\Om T$. Let $\bs B$ be a tangential structure. As $\bs B$-bordism $\Om_*^{\bs B}(-)$ is a generalized homology theory, the homotopy fibration gives an Atiyah--Hirzebruch spectral sequence
\e
H_p\bigl(T,\Om_q^{\bs B}(\Om T)\bigr)\Longra \Om_{p+q}^{\bs B}(\cL T).
\label{fm2eq4}
\e

The fibration has a section $s_T:T\ra\cL T$ mapping $t\in T$ to the constant loop $\cS^1\ra\{t\}\subset T$, with $\ev_{\ul{0}}\ci s_T=\id_T$. The morphisms $(\ev_{\ul{0}})_*:\Om^{\bs B}_n(\cL T)\ra \Om^{\bs B}_n(T)$ and $(s_T)_*:\Om^{\bs B}_n(T)\ra \Om^{\bs B}_n(\cL T)$ induce a splitting
\e
\begin{split}
&\Om^{\bs B}_n(\cL T)=\Om^{\bs B}_n(T)\op \Om^{\bs B}_n(\cL T;T),\qquad\text{where}\\
&\Om^{\bs B}_n(\cL T;T)=\Ker\bigl((\ev_{\ul{0}})_*:\Om^{\bs B}_n(\cL T)\longra \Om^{\bs B}_n(T)\bigr).
\end{split}
\label{fm2eq5}
\e
Write $\Pi^{\bs B}_n(T):\Om^{\bs B}_n(\cL T)\ra\Om^{\bs B}_n(\cL T;T)$ for the projection in this direct sum. We regard $\Om^{\bs B}_n(\cL T;T)$ as a {\it relative $\bs B$-bordism group}.

The decomposition \eq{fm2eq5} of $\Om^{\bs B}_n(\cL T)$ corresponds in the spectral sequence \eq{fm2eq4} to the decomposition $\Om_q^{\bs B}(\Om T)=\Om_q^{\bs B}(*)\op\ti\Om_q^{\bs B}(\Om T)$. Therefore \eq{fm2eq4} splits as the direct sum of two spectral sequences, with the first $H_p\bigl(T,\Om_q^{\bs B}(*)\bigr)\Ra \Om_{p+q}^{\bs B}(T)$ the usual spectral sequence for computing $\Om_*^{\bs B}(T)$, and the second being~$H_p\bigl(T,\ti\Om_q^{\bs B}(\Om T)\bigr)\Ra \Om_{p+q}^{\bs B}(\cL T;T)$.

Let $f:S\ra T$ be a continuous map of connected topological spaces, and write $\cL f:\cL S\ra\cL T$ for the induced map of free loop spaces. Then $\cL f_*:\Om^{\bs B}_n(\cL S)\ra\Om^{\bs B}_n(\cL T)$ is compatible with the splittings \eq{fm2eq5}. Write
\begin{equation*}
f_\rel^{\bs B}:\Om^{\bs B}_n(\cL S;S)\longra\Om^{\bs B}_n(\cL T;T)
\end{equation*}
for the restriction of $\cL f_*$ to relative $\bs B$-bordism.

Now a continuous map $\phi:X\ra\cL T$ is equivalent to a continuous map $\phi':X\t\cS^1\ra T$, by the tautological definition $\phi'(x,y)=\phi(x)(y)$ for $x\in X$ and $y\in\cS^1$. Define a morphism
\e
\begin{split}
\xi^{\bs B}_n(T)&:\Om^{\bs B}_n(\cL T)\longra\ti\Om_{n+1}^{\bs B}(T)\qquad\text{by}\\
\xi^{\bs B}_n(T)&:[X,\bs\ga_X,\phi]\longmapsto [X\t\cS^1,\bs\ga_X\t\bs\ga_{\cS^1},\phi'],
\end{split}
\label{fm2eq6}
\e
where $\phi'$ is as above and $\bs\ga_{\cS^1}$ is the $\bs B$-structure on $\cS^1$ induced from the standard $\bs B$-structure on the closed unit disc $D^2\subset\R^2$ by identifying $\cS^1=\pd D^2$. Here $\xi^{\bs B}_n(T)$ maps to $\ti\Om_{n+1}^{\bs B}(T)=\Ker\bigl(\Om_{n+1}^{\bs B}(T)\ra\Om_{n+1}^{\bs B}(*)\bigr)\subset \Om_{n+1}^{\bs B}(T)$ as the projection of \eq{fm2eq6} to $\Om_{n+1}^{\bs B}(*)$ is $[X\t\cS^1]=[X]\cdot[\cS^1]=0$ since~$\cS^1=\pd D^2$.

Equation \eq{fm2eq6} is compatible with equivalences $(X_0,\bs\ga_{X_0},\phi_0)\sim\ab(X_1,\ab\bs\ga_{X_1},\ab\phi_1)$, and so is well defined. Note that $[\cS^1,\bs\ga_{\cS^1}]=0$ in $\Om_1^{\bs B}(*)$. When $\bs B=\Spin$ there is also a second $\Spin$-structure $\bs\ga'_{\cS^1}$ on $\cS^1$ with $[\cS^1,\bs\ga'_{\cS^1}]\ne 0$ in $\Om_1^{\bs\Spin}(*)$; it is important that we use $\bs\ga_{\cS^1}$ rather than $\bs\ga'_{\cS^1}$ in~\eq{fm2eq6}.

Consider the diagram
\e
\begin{gathered}
\xymatrix@C=35pt@R=15pt{ 
0 \ar[r] & \Om^{\bs B}_n(T) \ar@<-1ex>[dr]_0 \ar[r]^{(s_T)_*} & \Om^{\bs B}_n(\cL T) \ar[r]^{\Pi^{\bs B}_n(T)} \ar[d]^{\xi^{\bs B}_n(T)} & \Om^{\bs B}_n(\cL T;T) \ar@<1ex>@{..>}[dl]^{\hat\xi^{\bs B}_n(T)} \ar[r] & 0 \\
&& \ti\Om_{n+1}^{\bs B}(T).\! }	
\end{gathered}
\label{fm2eq7}
\e
The top row is exact. If $[X,\bs\ga_X,\psi]\in \Om^{\bs B}_n(T)$ then
\begin{align*}
&\xi^{\bs B}_n(T)\ci (s_T)_*\bigl([X,\bs\ga_X,\psi]\bigr)=\xi^{\bs B}_n(T)\bigl([X,\bs\ga_X,s_T\ci\psi]\bigr)
\\
&\;\> =[X\t\cS^1,\bs\ga_X\t\bs\ga_{\cS^1},\psi\ci\pi_X]=[X,\bs\ga_X,\psi]*[\cS^1,\bs\ga_{\cS^1}]=[X,\bs\ga_X,\psi]*0=0,
\end{align*}
where $*:\Om^{\bs B}_n(T)\t\Om_q^{\bs B}(*)\ra\Om_{n+1}^{\bs B}(T)$ is the natural product. Thus the left hand triangle of \eq{fm2eq7} commutes, so there exists a unique morphism $\hat\xi^{\bs B}_n(T):\Om^{\bs B}_n(\cL T;T)\ra\ti\Om_{n+1}^{\bs B}(T)$ making the right hand triangle commute.
\end{dfn}

\subsection{The Atiyah--Hirzebruch spectral sequence}
\label{fm22}

As spin bordism is a generalized homology theory, for each topological space with basepoint $t_0\in T$ there is an {\it Atiyah--Hirzebruch spectral sequence}
\e
 \ti H_p\bigl(T,\Om^{\bs\Spin}_q(*)\bigr)\Longra\ti\Om^{\bs\Spin}_{p+q}(T).
\label{fm2eq8}
\e
Evaluation of the spectral sequence requires knowledge of the graded coefficient ring $\bigoplus_{n\in\N}\Om^{\bs\Spin}_n(*)$, which is shown in Table \ref{fm2tab2} for $n\le 9$.

We explain the meaning of \eq{fm2eq8}, and refer to McCleary \cite{McCl} for a comprehensive introduction to spectral sequences. A spectral sequence consists of a sequence of bigraded abelian groups $\{(E_{p,q}^r,d_{p,q}^r)\}_{p,q\in\Z}$ for each $r=2,3,\ldots$, the {\it pages} of the spectral sequence, and differentials $d^r_{p,q}: E^r_{p,q}\ra E^r_{p-r,q+r-1}$. The groups $E_{p,q}^{r+1}$ of the $E^{r+1}$-page are the homology groups of $(E_{p,q}^r,d_{p,q}^r)$ of the $E^r$-page, beginning with $E_{p,q}^2=\ti H_p(T,\Om^{\bs\Spin}_q(*))$. If the spectral sequence is {\it convergent}, the groups $E_{p,q}^r$ stabilize for large values of $r$ to groups $E_{p,q}^\iy$. Moreover, for each $n$ there is a filtration
\begin{equation*}
0=F_{-1,n}\subset F_{0,n}\subset \cdots\subset F_{p,n}\subset F_{p+1,n}\subset\cdots\subset F_{n,n}=\ti \Om^{\bs\Spin}_n(T).
\end{equation*}
If $T$ is a CW complex, then $F_{p,n}$ is defined to be the image of the morphism $\ti\Om_n^{\bs\Spin}(T_p)\ra\ti\Om_n^{\bs\Spin}(T)$ induced by the inclusion of the $p$-skeleton $T_p$. For all $p,n$ in $\N$ there are natural isomorphisms
\e
\label{fm2eq9}
 F_{p,n}/F_{p-1,n}\cong E_{p,n-p}^\iy.
\e

Fix $p$. The direct sum $\bigoplus_q E_{p,q}^r$ over all groups in the $p$-th column of the $E^r$-page of the spectral sequence forms a graded module over the graded ring $\bigoplus_{n\in\N}\Om^{\bs\Spin}_n(*)$, and the differentials $d^r_{p,q}$ are linear for this module structure. Moreover, the filtration is compatible with the graded module structure and for each $\al\in\Om_d^{\bs\Spin}(*)$ there is a commutative diagram
\begin{equation*}
\begin{tikzcd}
 0\rar & F_{p-1,n}\rar{\subset}\dar{\cdot\al} & F_{p,n}\dar{\cdot\al}\rar & E_{p,n-p}^\iy\rar\dar{\cdot\al} & 0\\
 0\rar & F_{p-1,n+d}\rar{\subset} & F_{p,n+d}\rar & E_{p,n+d-p}^\iy\rar & 0.\!
\end{tikzcd}
\end{equation*}

To emphasize the dependence of the spectral sequence and of the filtration on the space $T$, we sometimes write $E^r_{p,q}(T)$, $d^r_{p,q}(T)$, and $F_{p,n}(T)$.

The suspension isomorphism $\si:\ti\Om_n^{\bs\Spin}(T)\ra\ti\Om_{n+1}^{\bs\Spin}(\Si T)$ maps the filtration $F_{p,n}$ to $F_{p+1,n+1}$ and so induces a morphism of spectral sequences
\e
\label{fm2eq10}
 E^r_{p,q}(T)\longra E^r_{p+1,q}(\Si T).
\e
On the $E^2$-page the morphism \eq{fm2eq10} coincides with the suspension isomorphism $\si:\ti H_p(T,\Om_q^{\bs\Spin}(*))\ra\ti H_{p+1}(T,\Om_q^{\bs\Spin}(*))$ in ordinary homology. Moreover, we have commutative diagrams
\begin{equation*}
\begin{tikzcd}[column sep=4ex]
	0\rar & F_{p-1,n}(T)\rar{\subset}\arrow[d,"\si","\cong"'] & F_{p,n}(T)\arrow[d,"\si","\cong"']\rar & E^\iy_{p,n-p}(T)\rar\arrow[d,"\si","\cong"'] & 0\\
	0\rar & F_{p,n+1}(\Si T)\rar{\subset} & F_{p+1,n+1}(\Si T)\rar & E^\iy_{p+1,n-p}(\Si T)\rar & 0.\!
\end{tikzcd}
\end{equation*}

In general, the differentials in a spectral sequence are difficult to compute. For the spectral sequence \eq{fm2eq8} it is possible to compute $d^2_{p,0}$, $d^2_{p,1}$ by the following result of Maunder \cite{Maun}, see also \cite[Lem.~5.6]{ABP}.

\begin{prop}
\label{fm2prop1}
For the Atiyah--Hirzebruch spectral sequence of spin bordism, the differentials\/ $d^2_{p,1}:\ti H_p(T,\Z_2)\cong E_{p,1}^2\ra E_{p-2,2}^2\cong \ti H_{p-2}(T,\Z_2)$ on the\/ $E^2$-page are dual to the Steenrod operation\/ $\Sq^2$. Moreover,\/ $d^2_{p,0}=d^2_{p,1}\circ \rho_2$ where\/ $\rho_2:H_p(T,\Z)\ra H_p(T,\Z_2)$ is reduction modulo two.
\end{prop}

By taking the push-forward of the {\it fundamental class} of a compact spin $n$-manifold, we obtain a morphism
\e
 \Psi_n(T):\ti\Om^{\bs\Spin}_n(T)\longra \ti H_n(T,\Z),\quad [f:X\ra T]\longmapsto f_*([X]).
 \label{fm2eq11}
\e

\begin{prop}
\label{fm2prop2}
\e
\begin{aligned}
\text{$\Psi_n(T)$ injective}&\iff \text{$E^\iy_{n-1,1}=\cdots=E^\iy_{0,n}=0$,}\\
\text{$\Psi_n(T)$ surjective}&\iff \text{$d^r_{n,0}:E^r_{n,0}\ra E^r_{n-r,r-1}$ vanish for all\/ $r\geqslant 2$.}
\end{aligned}
\label{fm2eq12}
\e
\end{prop}

\begin{proof}
Since \eq{fm2eq11} is a transformation of cohomology theories, it induces a morphism of Atiyah--Hirzebruch spectral sequences. The $E^2$-page of the spectral sequence $\ti H_p(T;H_q(*,\Z))\Ra\ti H_{p+q}(T,\Z)$ consists of a single row, so the spectral sequence collapses and we have a trivial filtration with $E^2_{n,0}=\ti H_n(T,\Z)$. Moreover, we see that $\Ker\Psi_n(T)=F_{n-1,n}$. We can thus factor $\Psi_n(T)$ as
\e
 \ti\Om^{\bs\Spin}_n(T)=F_{n,n} \mathbin{\relbar\joinrel\twoheadrightarrow} F_{n,n}/F_{n-1,n}\cong E^\iy_{n,0} \mathbin{\lhook\joinrel\longra}
 E^2_{n,0}=\ti H_n(T,\Z),
 \label{fm2eq13}
\e
where the first map is surjective and the last map is injective. Hence \eq{fm2eq13} is injective if and only if $F_{n-1,n}=0$. Moreover, \eq{fm2eq13} is surjective if and only if $E^2_{n,0}=E^\iy_{n,0}$. These conditions are equivalent to those stated in \eq{fm2eq12}.
\end{proof}

If $E^\iy_{n,0}\subset E^2_{n,0}=\ti H_n(T,\Z)$ is freely generated by homology classes $\al_1,\ldots,\al_k$ of orders $m_1,\ldots,m_k\ge 0$, we can interpret $\Psi_n(T)$ as a morphism $\ti\Om^{\bs\Spin}_n(T)\ra\Z_{m_1}\op\cdots\op\Z_{m_k}$. Suppose that $a_j\in H^n(T,\Z_{m_j})$ are cohomology classes satisfying $\an{a_i,\al_j\bmod{m_i}}=\de_{i,j}$. Then we can write this morphism as
\begin{equation*}
 \Psi_n(T)([X,f])=\bigl(\ts\int_X f^*(a_1),\ldots,\ts\int_X f^*(a_k)\bigr).
\end{equation*}
If \eq{fm2eq12} holds, these give explicit isomorphisms $\ti\Om^{\bs\Spin}_n(T)\cong\Z_{m_1}\op\cdots\op\Z_{m_k}$.

\subsection{Eilenberg--MacLane spaces and Steenrod squares}
\label{fm23}

The differentials in the spectral sequence \eq{fm2eq8} are determined by Steenrod squares, which are linked to Eilenberg--MacLane spaces. As both of these concepts play an important role in what follows, we briefly review them here.

\begin{dfn}
Let $A$ be an abelian group and $n\ge 1$. The {\it Eilenberg--MacLane space} $K(A,n)$ is a connected topological space, well-defined up to homotopy equivalence, such that $\pi_k(K(A,n))=0$ for all $n\neq k$ and $A\cong\pi_n(K(A,n))$.
\end{dfn}

There is a Hurewicz isomorphism
\e
 A\cong\pi_n(K(A,n))\longra H_n(K(A,n),\Z)
\label{fm2eq14}
\e
and a {\it primary class} $e_n \in H^n(K(A,n),A)$ that corresponds under the isomorphism $H^n(K(A,n),A)\cong \Hom(H_n(K(A,n),\Z),A)$ to the inverse of \eq{fm2eq14}.

A useful, alternative definition of $K(A,n)$ arises in obstruction theory and states that $K(A,n)$ is characterized by the property that cohomology classes $u\in H^n(X,A)$ correspond bijectively to homotopy classes of maps $f:X\ra K(A,n)$ satisfying $f^*(e_n)=u$.

\begin{dfn}
For any topological space $X$ and $i\ge 0$ the {\it Steenrod operation} is a group morphism
\begin{equation*}
 \Sq^i: H^*(X,\Z_2)\longra H^{*+i}(X,\Z_2),
\end{equation*}
characterized uniquely by the following axioms:
\begin{itemize}
\setlength{\itemsep}{0pt}
\setlength{\parsep}{0pt}
\item[(a)] {\it Naturality.} For any continuous map $f:X\ra Y$,
\begin{equation*}
 f^*\circ\Sq^i=\Sq^i\circ f^*.
\end{equation*}
\item[(b)] {\it Normalization.} If $x\in H^n(X,\Z_2)$ has degree $n$, then
\begin{align*}
\Sq^0(x)&=x, &\Sq^n(x)&=x\cup x, &\Sq^m(x)&=0\quad \text{for all $m>n$.}
\end{align*}
\item[(c)] {\it Cartan formula.} For all $x,y\in H^*(X,\Z_2)$ we have
\begin{equation*}
 \Sq^n(x\cup y)=\sum_{i+j=n}\Sq^i(x)\cup\Sq^j(y).
\end{equation*}
\end{itemize}
\end{dfn}

It follows from these axioms that $\Sq^i$ is a {\it stable} cohomology operation, meaning that it commutes with suspension isomorphisms. Another important consequence of the axioms is the {\it Adem relation} for all  $i,j>0$, $i<2j$,
\e
\label{fm2eq15}
 \Sq^i\Sq^j = \sum_{k=0}^{\lfloor i/2 \rfloor} \binom{j-k-1}{i-2k} \Sq^{i+j-k} \Sq^k.
\e
The first Steenrod operation $\Sq^1$ is closely connected to the Bockstein homomorphism $\be_2:H^*(X,\Z_2)\ra H^{*+1}(X,\Z)$ associated to the short exact sequence
\begin{equation*}
\begin{tikzcd}[row sep=0]
 0\rar & \Z\rar{\mu_2} & \Z\rar{\rho_2} & \Z_2\rar & 0,
\end{tikzcd}
\end{equation*}
where $\mu_2$ is multiplication by $2$ and $\rho_2$ is reduction modulo $2$. Specifically,
\e
\label{fm2eq16}
 \Sq^1=\rho_2\circ\be_2
\e
Moreover, for all odd $i=2k+1$ there are {\it integral Steenrod operations}
\e
\label{fm2eq17}
 \Sq^{2k+1}_\Z=\be_2\circ\Sq^{2k}:H^*(X,\Z_2)\longra H^{*+2k+1}(X,\Z)
\e
into integral cohomology. The Adem relation $\Sq^1\circ\Sq^{2k}=\Sq^{2k+1}$ implies
\begin{equation*}
 \rho_2\circ\Sq^{2k+1}_\Z=(\rho_2\circ\be_2)\circ\Sq^{2k}=\Sq^1\circ\Sq^{2k}=\Sq^{2k+1}.
\end{equation*}

\begin{ex}
The $\Z_2$-cohomology of the classifying space $B\O(n)$ is a polynomial $\Z_2$-algebra generated by the Stiefel--Whitney classes $w_1,\ldots, w_n$. The {\it Wu formula} calculates the Steenrod squares of these for $i<j$ and states
\begin{equation*}
\Sq^i(w_j)=\sum_{k=0}^i\binom{j-k-1}{i-k}w_{i+j-k}w_t.
\end{equation*}
\end{ex}

A sequence $I=(i_1,i_2,\cdots)$ of natural numbers is {\it admissible} if $i_1\geqslant 2i_2$, $i_2\geqslant2i_3$, $\cdots$. The {\it excess} of an admissible sequence is $e(I)=(i_1-2i_2)+(i_2-2i_3)+\cdots.$ Write $\bar\al\in H^n(X,\Z_2)$ for the reduction modulo two of a cohomology class $\al\in H^n(X,\Z)$. From Serre {\cite[Th.~3]{Serr}} we recall the following theorem.

\begin{thm}
\label{fm2thm1}
Let\/ $n\ge 2$. The cohomology\/ $H^*(K(\Z,n),\Z_2)$ is the polynomial\/ $\Z_2$-algebra on generators\/ $\Sq^I(\bar e_n)$ for all admissible sequences\/ $I=(i_1,i_2,\cdots)$ with excess\/ $e(I)<n$ and all entries\/ $i_k\neq1.$
\end{thm}

We shall also need the following cohomology operation. Following Browder--Thomas \cite{BrTh}, the Pontrjagin square is characterized by the following axioms.

\begin{dfn}
\label{fm2def4}
For any topological space $X$ the {\it Pontrjagin square} is a (non-additive) map
\begin{equation*}
 \cP:H^{2*}(X,\Z_2)\longra H^{4*}(X,\Z_4)
\end{equation*}
characterized uniquely by the following axioms:
\begin{itemize}
\setlength{\itemsep}{0pt}
\setlength{\parsep}{0pt}
\item[(a)] {\it Naturality.}
$\cP$ is natural in $X$;
\item[(b)] {\it Addition formula.} For all $\bar x, \bar y\in H^{2n}(X,\Z_2)$ we have
\begin{equation*}
\cP(\bar x+\bar y)=\cP(\bar x)+\cP(\bar y)+\ti\mu_2(\bar x\cup \bar y).
\end{equation*}
\item[(c)] {\it Normalization.} The mod $2$ reduction of $\cP$ is the cup square:
\begin{equation*}
 \bar\rho_2\circ\cP(\bar x)=\bar x^2.
\end{equation*}
If $\bar x=\bar\rho_2(\ti x)$ can be lifted to $\ti x\in H^{2n}(X,\Z_4)$, then
\begin{equation*}
 \cP(\bar\rho_2(\ti x))=\ti x^2.
\end{equation*}
\end{itemize}

Here, $\ti\mu_2$ is the map of coefficients $\Z_2\ra\Z_4$ induced by multiplication by $2$, and $\bar\rho_2$ is the map of coefficients\/ $\Z_4\ra\Z_2$ induced by projection modulo $2$.
\end{dfn}

\subsection{The spectral sequence for the free loop space}
\label{fm24}

Let $T$ be a path-connected topological space with basepoint $t_0\in T$.

There is a more general version of the Atiyah--Hirzebruch spectral sequence for Serre fibrations, which we shall use for the loop space fibration defined as follows. The {\it free loop space} is the set $\cL T$ of all continuous maps $\ga:\cS^1\ra T$, with the compact-open topology. The {\it based loop space} is the subspace of all $\ga\in\cL T$ satisfying $\ga(1)=t_0$. Evaluation at $1\in\cS^1$ defines a Serre fibration
\begin{equation*}
 \pi_T:\cL T\longra T,\quad \ga\longmapsto\ga(1),
\end{equation*}
with fibre $\Om T$, and there is an Atiyah--Hirzebruch spectral sequence
\e
 H_p\bigl(T,\Om_q^{\bs\Spin}(\Om T)\bigr)\Longra \Om_{p+q}^{\bs\Spin}(\cL T).
\label{fm2eq18}
\e

Inclusion of the constant loops defines a section $s_T:T\ra\cL T$ of the fibration, with $\pi_T\circ s_T=\id_T$ and the morphisms $(\pi_T)_*:\Om_n^{\bs\Spin}(\cL T)\ra \Om_n^{\bs\Spin}(T)$ and $(s_T)_*:\Om_n^{\bs\Spin}(T)\ra \Om_n^{\bs\Spin}(\cL T)$ induce a splitting
\e
 \Om_n^{\bs\Spin}(\cL T)=\Om_n^{\bs\Spin}(T)\op \Om_n^{\bs\Spin}(\cL T;T),
\label{fm2eq19}
\e
where $\Om_n^{\bs\Spin}(\cL T;T)$ is the relative spin bordism group. The decomposition \eq{fm2eq19} of $\Om_n^{\bs\Spin}(\cL T)$ corresponds in the spectral sequence \eq{fm2eq18} to the decomposition $\Om_q^{\bs\Spin}(\Om T)=\Om_q^{\bs\Spin}(*)\op\ti\Om_q^{\bs\Spin}(\Om T)$. Therefore \eq{fm2eq18} splits as the direct sum of two spectral sequences, the first being $H_p\bigl(T,\Om_q^{\bs\Spin}(*)\bigr)\Ra \Om_{p+q}^{\bs\Spin}(T)$ and the second being
\e
H_p\bigl(T,\ti\Om_q^{\bs\Spin}(\Om T)\bigr)\Longra \Om_{p+q}^{\bs\Spin}(\cL T;T).
\label{fm2eq20}
\e
If $T$ is a CW complex with $p$-skeleta $T_p\subset T$ for $p\ge 0$, then the filtration of $\Om_n^{\bs\Spin}(\cL T;T)$ is
\e
\label{fm2eq21}
 F_{p,n}=\Im\left(\Om_n^{\bs\Spin}(\pi_T^{-1}(T_p);T_p)\ra\Om_n^{\bs\Spin}(\cL T;T)\right).
\e
In particular, if $T$ is connected, we can take $T_0=\{*\}$ and then $\pi_T^{-1}(T_0)=\Om T$, so the inclusion $j:\Om T\ra\cL T$ of the based loop space into the free loop space induces a map 
\begin{equation*}
 j_*:\ti\Om_n^{\bs\Spin}(\Om T)=\Om_n^{\bs\Spin}(\Om T;*)\longra\Om_n^{\bs\Spin}(\cL T;T)
\end{equation*}
whose image is the first term $F_{0,n}$ of the filtration \eq{fm2eq21}. The spectral sequence \eq{fm2eq20} converges, meaning that $F_{p,n}/F_{p-1,n}\cong E^\iy_{p,n-p}$ for all $p,n$.

\begin{prop}
	\ea
\text{$j_*$ injective in dim.~$n$}&\iff \text{$d^r_{r,n-r+1} : E^r_{r,n-r+1} \ra E^r_{0,n}$ vanish for all\/ $r\geqslant 2$,}\nonumber\\
\text{$j_*$ surjective in dim.~$n$}&\iff \text{$E_{p,n-p}^\iy=0$ for all\/ $p\geqslant 1$.}
\label{fm2eq22}
\ea
\end{prop}

\begin{proof}
We can factor $j_*$ as
\begin{equation*}
 \ti\Om_n^{\bs\Spin}(\Om T) \mathbin{\relbar\joinrel\twoheadrightarrow} E_{0,n}^\iy \mathbin{\lhook\joinrel\longra} \Om_n^{\bs\Spin}(\cL T;T),
\end{equation*}
where the first map projects from $E_{0,n}^2=H_0(T,\ti\Om_n^{\bs\Spin}(\Om T))\cong \ti\Om_n^{\bs\Spin}(\Om T)$ onto $E_{0,n}^\iy$ modulo the images of all differentials into position $(0,n)$ and the second map is the inclusion $F_{0,n}\subset \Om_n^{\bs\Spin}(\cL T;T).$ The first condition in \eq{fm2eq22} is therefore equivalent to the injectivity of $j_*$. The second condition in \eq{fm2eq22} is equivalent to the surjectivity of $j_*$ since $\Om_n^{\bs\Spin}(\cL T;T)$ is obtained from $F_{0,n}$ by a sequence of extension problems with quotients $E_{1,n-1}^\iy,E_{2,n-2}^\iy,\ldots$, so $F_{0,n}=\Om_n^{\bs\Spin}(\cL T;T)$ if and only if all of these groups vanish.
\end{proof}

\subsection{L-equivalence and Thom spaces}
\label{fm25}

Thom \cite{Thom} introduced the spaces $M\SO(k)$ to study which homology classes $\al\in H_{n-k}(X,\Z)$ in an oriented $n$-manifold $X$ can be represented by a compact oriented submanifold $M\subset X$ of codimension $k$. Extending this idea, requiring the structure group of the normal bundle to factor through a fixed representation $\rho:H\ra\SO(k)$ leads, more generally, to the Thom space~$MH$.

\begin{dfn}
\label{fm2def5}
Let $H$ be a Lie group, and $\rho:H\ra\O(k)$ be a morphism of Lie groups for $k\ge 0$. Usually $H$ will be a subgroup of $\O(k)$, such as $\O(k),\SO(k)$ or $\U(k/2)$, and $\rho$ will be the inclusion. But we will also be interested in the case when $H=\Spin(k)$ and $\rho$ is the composition of the double cover $\Spin(k)\twoheadrightarrow\SO(k)$ with the inclusion $\SO(k)\hookra\O(k)$. We will only consider $\rho:H\ra\O(k)$ which factorize via $\SO(k)\hookra\O(k)$, which is needed for the Thom isomorphisms in \eq{fm2eq23}--\eq{fm2eq24} below to work over general~$R$.

Let $X$ be a manifold or paracompact topological space, and $E\ra X$ be a (smooth or topological) rank $k$ vector bundle with a metric $g_E$ on its fibres. An $H$-{\it structure\/} on $E$ means a pair $(Q,q)$ with $Q\ra X$ a (smooth or topological) principal $H$-bundle, so that $(Q\t\O(k))/H\ra X$ is a principal $\O(k)$-bundle, where $H$ acts on $Q$ by the principal bundle action and on $\O(k)$ by $h:A\mapsto A\rho(h)^{-1}$ for $h\in H$ and $A\in\O(k)$, and $q:(Q\t\O(k))/H\ra F_E$ is an isomorphism of $\O(k)$-bundles on $X$, where $F_E\ra X$ is the frame bundle of~$E$.

If the rank $k$ vector bundle $E\ra X$ does not come with a metric $g_E$ on its fibres, then when we want an $H$-structure on $E$, we implicitly choose some $g_E$. As the space of choices of $g_E$ is contractible, and we only care about $H$-structures up to isotopy, this choice is unimportant, and we will usually not mention it.

We will define the {\it Thom space\/} $MH$, a topological space we think of as natural up to homotopy equivalence. It depends on $\rho$ as well as $H$, although we omit this from the notation. Let $BH$ be the classifying space of $H$. There is a tautological rank $k$ vector bundle $E_k\ra BH$, with an $H$-structure. Write $S_k\subset D_k\subset E_k$ for the subbundles with fibres the unit sphere $\bigl\{(x_1,\ldots,x_k)\in\R^k:x_1^2+\cdots+x_k^2=1\bigr\}$ in $\R^k$, and the unit disc $\bigl\{(x_1,\ldots,x_k)\in\R^k:x_1^2+\cdots+x_k^2\le 1\bigr\}$ in $\R^k$, respectively. Then define $MH$ to be the quotient topological space $D_k/S_k$ obtained by collapsing the closed subspace $S_k$ in $D_k$ to a point, written~$\{\iy\}$. 

Observe that a composition $\rho_1=\rho_2\circ \io:H_1\,{\buildrel\io\over\longra}\,H_2\,{\buildrel\rho_2\over\longra}\,\O(k)$ induces a map
\begin{equation*}
 \io_*:MH_1\longra MH_2.
\end{equation*}
If $\rho$ factorizes via $\SO(k)\hookra\O(k)$, then for (co)homology with coefficients in a commutative ring $R$ there are {\it Thom isomorphisms}, see \cite[\S II.2]{Thom},
\ea
H^{n-k}(BH,R)&\overset{\cong}{\longra} \ti H^n(MH,R), & a&\longmapsto \pi^*(a)\cup t,
\label{fm2eq23}\\
\ti H_n(MH,R)&\overset{\cong}{\longra} H_{n-k}(BH,R), & \be &\longmapsto \pi_*(t\cap\be),
\label{fm2eq24}
\ea
defined in terms of relative cup and cap products with the {\it Thom class} $t$ in $\ti H^k(MH,R)\cong H^k(D_\rho, S_\rho,R)$. In particular, $\ti H^n(MH,R)$ and $\ti H_n(MH,R)$ both vanish for $n<k$. If $\rho$ does not factorize via $\SO(k)\hookra\O(k)$, then \eq{fm2eq23}--\eq{fm2eq24} hold for commutative rings $R$ of characteristic 2. For the inverse of the homology Thom isomorphism \eq{fm2eq24} we introduce the notation
\begin{equation*}
 H_n(BH,R)\,{\buildrel\cong\over\longra}\,\ti H_{n+k}(MH,R),\quad \al\longmapsto\al^T.
\end{equation*}
Write $\tau=1^T\in H_k(MH,R)$ for the dual of the Thom class.

Let $s:BH\ra MH$ be composition of the zero section of $D_\rho$ with the projection onto $MH$. The {\it Euler class} of $E_\rho$ is defined by
\begin{equation*}
 e=s^*(t)\in H^k(BH,\Z),
\end{equation*}
and it determines the ring structure on $H^*(MH,\Z)$ according to the rule
\e
t\cup t=\pi^*(e)\cup t.
\label{fm2eq25}
\e

When $X$ is a compact manifold, the {\it Pontrjagin--Thom construction} is a correspondence between homotopy classes of maps $X\ra MH$ and bordism classes of submanifolds $M\subset X$ with a normal $H$-structure. Under the excision isomorphism, we can view the Thom class of the normal bundle of $M$ as an element of $H^k(D\nu_M;S\nu_M,R)\cong H^k(X;X\sm M,R)$ whose image in $H^k(X)$ is Poincar\'e dual to the fundamental class $[M]\in H_{n-k}(X,R)$.
\end{dfn}

Thom spaces were introduced by Ren\'e Thom \cite{Thom}, who studied the question of when a homology class $\al\in H_{n-k}(X,\Z)$ in an $n$-manifold $X$ can be represented by a compact $(n-k)$-submanifold $M\subset X$. In \cite[\S IV.3]{Thom} (see also Novikov \cite[\S 1.1]{Novi}) he defines {\it L-equivalence\/} of submanifolds.

\begin{dfn}
\label{fm2def6}
Let $X$ be an $n$-manifold, possibly with boundary or corners. Recall from Hirsch \cite[Ch.~1, \S 4]{Hirs} that an $(n-k)$-submanifold $M\subset X$ is {\it neat\/} if $M$ intersects $\pd X$ transversely and $\pd M\cong M\cap\pd X.$ For such submanifolds there exist collars of $\pd X$ that restrict to collars of $\pd M$. Moreover, $\pd M\subset\pd X$ is again a neat $(n-k-1)$-submanifold such that $\nu_{\pd M}\cong\nu_M\vert_{\pd M}$ for the normal bundles.

Let $0\le k\le n$ and $\rho:H\ra\O(k)$ be a Lie group morphism. Let $X$ be a compact $n$-manifold without boundary, and consider pairs $(M,\ga_M)$ of a compact, embedded $(n-k)$-submanifold $M\subset X$ without boundary, with normal bundle $\nu_M\ra M$ of rank $k$, and an $H$-structure $\ga_M$ on~$\nu_M$. 

We say that two such pairs $(M_0,\ga_{M_0}),(M_1,\ga_{M_1})$ are {\it L-equivalent\/} if there exists a compact, embedded, neat $n-k+1$-submanifold $N\subset X\t[0,1]$ whose normal bundle $\nu_N\ra N$ is equipped with an $H$-structure $\ga_N$, such that $\pd N=(M_0\t\{0\})\amalg(M_1\t\{1\})$, and $\ga_N\vert_{M_i\t\{i\}}=\ga_{M_i}$ for $i=0,1$. This is an equivalence relation on such pairs $(M,\ga_M)$. Write $\La_k^H(X)$ for the set of L-equivalence classes $[M,\ga_M]$ of pairs~$(M,\ga_M)$.
\end{dfn}

The next theorem is proved by Thom \cite[Th.~IV.6]{Thom} (see also \cite[Th.~II.1]{Thom}) when $H$ is $\SO(k)$ or $\O(k)$, and extended by Novikov \cite[Lem.~2.1]{Novi} to general $H$. The point of the theorem is that we can hope to compute $[X,MH]$ using homotopy theory, and this then tells us about submanifolds of $X$ up to L-equivalence.

\begin{thm}
\label{fm2thm2}
{\bf(Thom \cite{Thom}.)} In Definition\/ {\rm\ref{fm2def6},} there is a natural bijection
\e
\La_k^H(X)\cong [X,MH],
\label{fm2eq26}
\e
with\/ $[X,MH]$ the set of isotopy classes $[f]$ of continuous~$f:X\ra MH$.
\end{thm}

\begin{proof}[Sketch proof] Let $f:X\ra MH$ be continuous. Write $U=X\sm f^{-1}(\iy)$, so that $U$ is open in $X$. On $MH\sm\{\iy\}$ there is a tautological rank $k$ vector bundle $E_k\ra MH\sm\{\iy\}$ with $H$-structure, and a section $s_k:MH\sm\{\iy\}\ra E_k$, such that $\lim_{m\ra\iy}\nm{s_k}(m)=1$, that is, the length $\nm{s_k}$ approaches 1 at $\iy$ in~$MH$.

Thus, on $U\subset X$ we have a {\it topological\/} vector bundle $f^*(E_k)\ra U$ with $\O(k)$-structure and $H$-structure, and a {\it continuous\/} section $f^*(s_k)$ of $f^*(E_k)$, whose length $\nm{f^*(s_k)}$ approaches 1 at $\bar U\sm U$ in $X$.

Choose a smooth vector bundle structure on $f^*(E_k)\ra U$ compatible with its $H$-structure. Choose a smooth, generic section $\ti s_k$ of $f^*(E_k)$ which is $C^0$-close to $f^*(s_k)$. Then $\ti s_k$ is transverse, so $M=\ti s_k^{-1}(0)$ is a smooth, embedded $(n-k)$-submanifold of $U$. As $\nm{\ti s_k}$ is close to 1 near $\bar U\sm U$, $M$ cannot approach $\bar U\sm U$, so $M$ is compact, as $X$ is. There is a natural isomorphism $\d\ti s_k\vert_M:\nu_M\ra f^*(E_k)\vert_M$. Thus, $\nu_M$ has an $H$-structure $\ga_M$, induced from that on $f^*(E_k)$. So $[M,\ga_M]\in [X,MH]$. Thom shows that $[M,\ga_M]$ depends only on the isotopy class $[f]$, and the map $[f]\mapsto[M,\ga_M]$ gives the bijection~\eq{fm2eq26}.
\end{proof}

\section{Spin bordism groups of classifying spaces}
\label{fm3}

\subsection{Spin bordism groups of some Thom spaces}
\label{fm31}

For low values of $n$, the structure of $\Om_n^{\bs\Spin}(*)$ was determined by Milnor \cite{Miln} and is given in Table \ref{fm2tab2}. The group $\Om_1^{\bs\Spin}(*)\cong\Z_2$ is generated by $\al_1=[\cS^1_{\rm nb}]$, $\Om_2^{\bs\Spin}(*)\cong\Z_2$ is generated by $\al_1^2=[\cS^1_{\rm nb}\t\cS^1_{\rm nb}]$, $\Om_4^{\bs\Spin}(*)\cong\Z$ is generated by $\al_4=[K3]$, and $\Om_n^{\bs\Spin}(*)=0$ for all other values of $n\le 7$. For each pair of spaces $(T,S)$ with $S\subset T$, the spin bordism groups $\Om_*^{\bs\Spin}(T;S)=\bigoplus_n\Om_n^{\bs\Spin}(T;S)$ are a graded module over the coefficient ring $\bigoplus_n\Om_n^{\bs\Spin}(*)$. In particular, write $\tilde\Om_n^{\bs\Spin}(T)=\Om_n^{\bs\Spin}(T;*)$ for the reduced spin bordism groups.

\subsubsection{Description of generators}
\label{fm311}

Before stating our first theorem, Theorem \ref{fm3thm1}, which calculates the spin bordism groups of Thom spaces, we describe elements of the spin bordism groups that will serve as explicit generators, and describe a convention that allows us to regard these as elements in different spin bordism groups.

Let $H$ be a Lie group with a fixed representation on $\R^4$, and let $MH$ be its Thom space. Each morphism $H_1\ra H_2$ that is compatible with the representations induces a map $MH_1\ra MH_2$. In particular, we apply this construction to the maps $\{1\}\subset\SU(2)$, $\SU(2)\subset \U(2)$, $\SU(2)\ra\Spin(4)$, $\U(2)\subset\SO(4)$, and $\Spin(4)\ra\SO(4)$. Additionally, let $M\SO(4)\rightarrow K(\Z,4)$ be the classifying map of the Thom class in $H^4(M\SO(4),\Z)$ and let $K(\Z,4)\ra K(\Z_2,4)$ be the reduction mod $2$. We summarize the situation in the following diagram:
\begin{equation}
\begin{tikzcd}[column sep=small]
&& M\U(2)\arrow[rd] &&\\
M\{1\}\rar & M\SU(2)=M\Sp(1)\!\!\!\!\!\!\!\! \arrow[ru]\arrow[rd] && M\SO(4)\arrow[r] & K(\Z,4)\dar\\
&& M\Spin(4)\arrow[ru]	&& K(\Z_2,4).\!
\end{tikzcd}
\label{fm3eq1}
\end{equation}
Each map induces a morphism of spin bordism groups, so we can regard elements in $\ti\Om_n^{\bs\Spin}(MH_1)$ as elements in $\ti\Om_n^{\bs\Spin}(MH_2)$, in $\ti\Om_n^{\bs\Spin}(K(\Z,4))$, or in $\ti\Om_n^{\bs\Spin}(K(\Z_2,4))$. {\it We use the same letter to denote these elements.}

In order to have explicit generators that can be used in both Theorem \ref{fm3thm1} and Theorem \ref{fm3thm2} below, we will present them in a form that makes it clear whether they can be lifted along $\hat\xi^{\bs\Spin}_{n-1}(MH)$. For $T=MH$, an element in the image of $\hat\xi^{\bs\Spin}_{n-1}(MH)$ has the form $[X\t\cS^1_{\rm b}, M\t\{1\}]$, where $X$ is a compact spin $(n-1)$-manifold and the normal bundle of $M\t\{1\}$ has an $H$-structure.

\paragraph{Generator $\de$ in Dimension 4.}

The normal bundle of a point $*$ in $\cS^3\t \cS^1_{\rm b}$ is trivial, so
\begin{equation*}
\de=[\cS^3\t\cS^1_{\rm b},\{*\}] \quad\text{in}\quad\ti\Om_4^{\bs\Spin}(M\{1\}).
\end{equation*}

\paragraph{Generator $\varep$ in Dimension 6.}

Recall that both $\cS^6\cong G_2/\SU(3)$ and $\cS^2\cong\CP^1$ are almost complex manifolds. In particular, the normal bundle of $\cS^2\subset\cS^6$ has structure group $\U(2)$. The normal bundle of $\cS^2\subset\cS^5\t\cS^1_{\rm b}$ is isomorphic to the normal bundle of $\cS^2\subset \cS^6$, so also has structure group $\U(2)$. Define
\begin{equation*}
\varepsilon=[\cS^5\t\cS^1_{\rm b},\cS^2] \quad\text{in}\quad\ti\Om_6^{\bs\Spin}(M\U(2)).
\end{equation*}

\paragraph{Generators $\ze_1,\frac{\ze_1}{2},\frac{\ze_1}{4},\ze_2,\ze_2',\ze_3$ in Dimension 8.}

The $K3$ surface as a submanifold of $K3\t\cS^3\t\cS^1_{\rm b}$ has trivial normal bundle, so
\e
\ze_1=[K3\t\cS^3\t\cS^1_{\rm b},\,K3\t\{N\}\t\{1\}]\quad\text{in}\quad \ti\Om_8^{\bs\Spin}(M\{1\}),
\label{fm3eq2}
\e
where $N=(1,0,0,0)\in\cS^3$. We next construct a natural divisor $\frac{\ze_1}{2}$ of the image of $\ze_1$ in $\ti\Om_8^{\bs\Spin}(M\U(2))$, and a natural divisor $\frac{\ze_1}{4}$ of the image of $\frac{\ze_1}{2}$ in $\ti\Om_8^{\bs\Spin}(M\SO(4))$. Recall that the $K3$ surface has involutions $\al, \be$ such that if $I,J,K$ are hyperk\"ahler complex structures on $K3$, then $\al$ maps $I\mapsto I,$ $J\mapsto -J,$ $K\mapsto -K$ and $\be$ maps $I\mapsto -I,$ $J\mapsto J,$ $K\mapsto -K.$ Then $\Z_2\an{\al,\be}$ acts freely on $K3$. Let $\Z_2\an{\al,\be}$ act on $\cS^3$ by $\al(x_0,x_1,x_2,x_3)=(x_0,x_1,-x_2,-x_3)$ and $\be(x_0,x_1,x_2,x_3)=(x_0,-x_1,x_2,-x_3)$. 

Now $N$ is a fixed point of this action and one of $\{\al,\be\}$ can be made to commute with the obvious complex structure on the (trivial) normal bundle. If, say, $\al$ commutes with the complex structure, then $K3$ has a normal $\U(2)$-structure when viewed as a submanifold of the compact spin $8$-manifold $(K3\t\cS^3)/\Z_2\an{\al}\t\cS^1_{\rm b}$ and a normal $\SO(4)$-structure as a submanifold of the compact spin $8$-manifold $(K3\t\cS^3)/\Z_2\an{\al,\be}\t\cS^1_{\rm b}$. This defines
\e
\begin{aligned}
\frac{\ze_1}{2}&=[(K3\t\cS^3)/\Z_2\an{\al}\t\cS^1_{\rm b},K3\t\{N\}\t\{1\}] &&\text{in}\;\>\ti\Om_8^{\bs\Spin}(M\U(2)),\\
\frac{\ze_1}{4}&=[(K3\t\cS^3)/\Z_2\an{\al,\be}\!\t\!\cS^1_{\rm b},K3\!\t\!\{N\}\!\t\!\{1\}] &&\text{in}\;\>\ti\Om_8^{\bs\Spin}(M\SO(4)).
\end{aligned}
\label{fm3eq3}
\e
We will show that indeed $\ti\Om_8^{\bs\Spin}(M\{1\})\ra\ti\Om_8^{\bs\Spin}(M\U(2))$ maps $\ze_1\mapsto 2\frac{\ze_1}{2}$ and $\ti\Om_8^{\bs\Spin}(M\U(2))\ra\ti\Om_8^{\bs\Spin}(M\SO(4))$ maps $\frac{\ze_1}{2}\mapsto 2\frac{\ze_1}{4}$, thus justifying the notation.

\smallskip

Recall that a {\it spin structure\/} on an oriented Euclidean vector bundle $W\ra M$ of rank $4$ is given by a pair of quaternionic line bundles $\Si^\pm_W\ra M$ and an isomorphism $W\cong \Hom_\H(\Si^-_W,\Si^+_W)$. In particular, a quaternionic line bundle structure on $W$ induces a spin structure on $W$ with $\Si^+_W=W$ and $\Si^-_W$ trivial.

The normal bundle of the embedding $\HP^1\hookrightarrow\HP^2$, $[q_0:q_1]\mapsto[q_0:q_1:0]$ is $\nu\cong\Hom_\H(\cL,\H)$, where $\cL$ is the tautological quaternionic line bundle over $\HP^1$. In particular, $\nu$ has a $\Spin(4)$-structure with spinor bundles $\Si_\nu^-=\cL$ and $\Si_\nu^+$ trivial. Let $\ov{\HP}^1$ be $\HP^1$ with the reversed orientation. There is then similarly an embedding $\ov{\HP}^1\hookrightarrow\HP^2$, $[q_0:q_1]\mapsto[\bar q_0:\bar q_1:0]$ whose normal bundle $\nu\cong\cL$ is a quaternionic line bundle.
Define
\e
\begin{aligned}
\ze_2'&=[\HP^2,\HP^1]-[\HP^2,\es]&&\text{in}\quad\ti\Om_8^{\bs\Spin}(M\Spin(4)),\\
\ze_2&=[\HP^2,\overline{\HP}^1]-[\HP^2,\es] &&\text{in}\quad\ti\Om_8^{\bs\Spin}(M\SU(2)).
\end{aligned}
\label{fm3eq4}
\e
Here we subtract $[\HP^2,\es]$ to get elements of $\ti\Om_8^{\bs\Spin}(T)\subset\Om_8^{\bs\Spin}(T)$.

Finally, $\CP^2\subset\CP^3$ has a complex normal bundle, and we define
\e
\ze_3=[\CP^3\t(\cS^1_{\rm b})^2, \CP^2\t\{1\}^2]\quad\text{in}\quad\ti\Om_8^{\bs\Spin}(M\U(2)).
\label{fm3eq5}
\e

\paragraph{Generator $\eta$ in Dimension 9.}

The {\it Wu manifold\/} $\SU(3)/\SO(3)$ is a compact $5$-manifold that is oriented, but not spin; it is a submanifold of the $3$-sphere bundle $(\SU(3)\t\cS^3)/\SO(3)$, where $\SO(3)$ acts on $\SU(3)$ by translation and on $\cS^3$ by rotation, by embedding it as $(\SU(3)\t\{N\})/\SO(3)$. Define
\e
\eta=[(\SU(3)\t\cS^3)/\SO(3)\t\cS^1_{\rm b},\SU(3)/\SO(3)\t\{1\}]\quad\text{in}
\quad\ti\Om_9^{\bs\Spin}(M\SO(4)).
\label{fm3eq6}
\e

\subsubsection{Statement of theorem}
\label{fm312}

For an abelian group $A$, the notation $A=\Z_m\an{a_1,\ldots,a_k}$ indicates that $A$ is freely generated by the elements $a_1,\ldots,a_k\in A$ of order $m$.

The following theorem is proved in~\S\ref{fm15}.

\begin{thm}
\label{fm3thm1}
{\bf(a)} In dimensions\/ $n\le 9$ the spin bordism groups\/ $\ti{\Omega}_n^{\bs\Spin}(T)$ of\/ $T=M\SU(2),$ $M\U(2),$ $M\Spin(4),$ $M\SO(4),$ $K(\Z,4),$ and\/ $K(\Z_2,4)$ are as shown in Table {\rm\ref{fm3tab1},} with generators $\de,\ep,\ldots$ as in\/ {\rm\S\ref{fm311}} and\/~$\al_1=[\cS^1_{\rm nb}]$.
\smallskip

\begin{table}[!ht]
\renewcommand{\arraystretch}{1.5}
\centering

\begin{tabular}{P{0.3cm}|P{2.75cm}|P{2.75cm}|P{2.75cm}}
$n$ & $\ti\Om_n^{\bs\Spin}(M\SU(2))$ & $\ti\Om_n^{\bs\Spin}(M\U(2))$ & $\ti\Om_n^{\bs\Spin}(M\Spin(4))$\\ \hline
$4$   &  $\Z\an{\de}$   &   $\Z\an{\de}$  &   $\Z\an{\de}$\\
$5$   &  $\Z_2\an{\al_1\de}$  &   &  $\Z_2\an{\al_1\de}$\\
$6$   &  $\Z_2\an{\al_1^2\de}$  &  $\Z\an{\varepsilon}$  & $\Z_2\an{\al_1^2\de}$\\
$8$   &  $\Z\an{\ze_1,\ze_2}$  &  $\Z\an{\frac{\ze_1}{2},\ze_2,\ze_3}$           &  $\Z\an{\ze_1,\ze_2,\ze_2'}$\\
$9$   &  $\Z_2\an{\al_1\ze_2}$   & $\Z_2\an{\al_1\ze_2}$  &  $\Z_2\an{\al_1\ze_2,\al_1\ze_2'}$\\
\end{tabular}
\medskip

\centering
\begin{tabular}{P{0.3cm}|P{2.75cm}|P{2.75cm}|P{2.75cm}}
$n$ & $\ti\Om_n^{\bs\Spin}(M\SO(4))$ & $\ti\Om_n^{\bs\Spin}(K(\Z,4))$ & $\ti\Om_n^{\bs\Spin}(K(\Z_2,4))$ \\ \hline
$4$   &  $\Z\an{\de}$  &  $\Z\an{\de}$ & $\Z_2\an{\de}$\\
$8$   &  $\Z\an{\frac{\ze_1}{4},\ze_2,\ze_3}$  & $\Z\an{\ze_2,\ze_3}$ & $\Z_4\an{\ze_2}$\\
$9$   &  $\Z_2\an{\al_1\frac{\ze_1}{4},\al_1\ze_2,\eta}$   &  $\Z_2\an{\al_1\ze_2}$  &  $\Z_2\an{\al_1\ze_2}$\\
\end{tabular}

\caption{The spin bordism groups in all dimensions $n\le 9$ of various Thom spaces and of\/ $K(\Z,4),K(\Z_2,4)$. A missing entry stands for the trivial group. If\/ $n\le 9$ is omitted, all groups $\ti\Om_n^{\bs\Spin}(T)$ are trivial for these $T$. }
\label{fm3tab1}
\end{table}

\noindent{\bf(b)} The morphisms between the groups in Table {\rm\ref{fm3tab1}} induced by the maps in \eq{fm3eq1} are as indicated by the notation \textup(or obviously trivial\textup) and the rules
\e
\label{fm3eq7}
\begin{aligned}
 \ti\Om_*^{\bs\Spin}(M\SO(4))&\longra\ti\Om_*^{\bs\Spin}(K(\Z,4)), &\ts\frac{\ze_1}{4}&\longmapsto 0,\enskip \eta\longmapsto\al_1\ze_2,\\
\ti\Om_*^{\bs\Spin}(M\Spin(4))&\longra\ti\Om_*^{\bs\Spin}(M\SO(4)), &\ze_2'&\longmapsto \ts\frac{\ze_1}{4}+\ze_2+4\ze_3,\\
 \ti\Om_*^{\bs\Spin}(M\U(2))&\longra\ti\Om_*^{\bs\Spin}(K(\Z_2,4)),& \ze_3&\longmapsto 0.
\end{aligned}
\e

\noindent{\bf(c)} The action of\/ $\al_1\in\Om_1^{\bs\Spin}(*)$ is as indicated by the notation and the rules
\e
\al_1\ze_1=0,\qquad \al_1\frac{\ze_1}{2}=0,\qquad \al_1\ze_3=0.
\label{fm3eq8}
\e

\noindent{\bf(d)} Representing elements of\/ $\ti\Om_n^{\bs\Spin}(MH)$ as pairs\/ $[X,M],$ where\/ $X$ is a compact spin\/ $n$-manifold and\/ $M\subset X$ is a submanifold with an\/ $H$-structure on its normal bundle\/ $\nu_M,$ we have the following explicit isomorphisms.

For\/ $M\SU(2)$ we have explicit isomorphisms
\ea
 &\ti\Om_4^{\bs\Spin}(M\SU(2))\overset\cong\longra\Z, &[X,M]&\longmapsto\# M, &
\label{fm3eq9}\\
 &\ti\Om_8^{\bs\Spin}(M\SU(2))\overset\cong\longra\Z^2, &[X,M]&\longmapsto\left(-\ts\frac{\operatorname{sign}(M)}{16},\ts\int_M c_2(\nu_M)\right), &
\label{fm3eq10}
\ea
which map\/ $\de\mapsto 1$ and\/ $\ze_1\mapsto(1,0),$ $\ze_2\mapsto(0,1)$.

For\/ $M\U(2)$ we have explicit isomorphisms
\ea
 &\ti\Om_4^{\bs\Spin}(M\U(2))\overset\cong\longra\Z, &[X,M]&\longmapsto\# M,&
\label{fm3eq11}\\
 &\ti\Om_6^{\bs\Spin}(M\U(2))\overset\cong\longra\Z,&[X,M]&\longmapsto\ts -\frac{1}{2}\int_M c_1(\nu_M),&
\label{fm3eq12}\\
  &\ti\Om_8^{\bs\Spin}(M\U(2))\overset\cong\longra\Z^3,& &
\label{fm3eq13}\\
 &[X,M]\longmapsto\mathrlap{\left(-\ts\frac{\operatorname{sign}(M)}{8} + \ts\int_M \frac{c_1(\nu_M)^2}{8}, \ts\int_M c_2(\nu_M), \ts\int_M c_1(\nu_M)^2\right),}
\nonumber
\ea
which map\/ $\de\mapsto 1,$ $\varep\mapsto 1,$ and\/ $\frac{\ze_1}{2}\mapsto(1,0,0),$ $\ze_2\mapsto (0,1,0),$ $\ze_3\mapsto(0,0,1)$.

For\/ $M\Spin(4)$ we have explicit isomorphisms
\ea
 &\ti\Om_4^{\bs\Spin}(M\Spin(4))\overset\cong\longra\Z, &[X,M]&\longmapsto\# M, &
\label{fm3eq14}\\
 &\ti\Om_8^{\bs\Spin}(M\Spin(4))\overset\cong\longra\Z^3,& &
\label{fm3eq15}\\
 &[X,M]\longmapsto\mathrlap{\left(\ts-\frac{\operatorname{sign}(M)}{16},\ts\int_M c_2(\Si^+_{\nu_M}),\ts\int_M c_2(\Si^-_{\nu_M})\right),}
\nonumber&
\ea
which map\/ $\de\mapsto 1$ and\/ $\ze_1\mapsto(1,0,0),$ $\ze_2\mapsto(0,1,0),$ $\ze_2'\mapsto(0,0,-1)$.

For\/ $M\SO(4)$ we have explicit isomorphisms
\ea
 &\ti\Om_4^{\bs\Spin}(M\SO(4))\overset\cong\longra\Z, &[X,M]&\longmapsto\# M,
\label{fm3eq16} &\\
 &\ti\Om_8^{\bs\Spin}(M\SO(4))\overset\cong\longra\Z^3, &[X,M]&\longmapsto &
\label{fm3eq17}\\
 &\mathrlap{\left(\ts-\frac{\operatorname{sign}(M)}{4}+\ts\int_M[\frac{e(\nu_M)}{2}+\frac{p_1(\nu_M)}{4}], \ts\int_M e(\nu_M), \ts\int_M (2e(\nu_M)+p_1(\nu_M))\right),}
 \nonumber&
\ea
which map\/ $\de\mapsto 1$ and \/ $\frac{\ze_1}{4}\mapsto(1,0,0),$ $\ze_2\mapsto(0,1,0),$ $\ze_3\mapsto(0,0,1)$.

For\/ $K(\Z,4)$ we have explicit isomorphisms
\ea
 &\ti\Om_4^{\bs\Spin}(K(\Z,4))\overset\cong\longra\Z, &[X,\al]&\longmapsto\ts\int_X\al,&
\label{fm3eq18}\\
 &\ti\Om^{\bs\Spin}_8(K(\Z,4))\overset\cong\longra\Z^2,
 &[X,\al]&\longmapsto\left(\ts\int_X\al^2,\ts\int_X\al\cup\frac{p_1(TX)+2\al}{4}\right),
\label{fm3eq19}
\ea
which map\/ $\de\mapsto 1$ and\/ $\ze_2\mapsto(1,0),$ $\ze_3\mapsto (0,1)$.

For\/ $K(\Z_2,4)$ we have explicit isomorphisms
\ea
 &\ti\Om_4^{\bs\Spin}(K(\Z,4))\overset\cong\longra\Z_2, &[X,\bar\al]&\longmapsto\ts\int_X\bar\al,
\label{fm3eq20}\\
 &\ti\Om^{\bs\Spin}_8(K(\Z_2,4))\overset\cong\longra\Z_4, &[X,\bar\al]&\longmapsto\ts\int_X\cP(\bar\al),&
\label{fm3eq21}
\ea
which map\/ $\de\mapsto 1$ and\/ $\ze_2\mapsto 1$. Here\/ $\cP$ denotes the \begin{bfseries}Pontrjagin square\end{bfseries} from Definition~{\rm\ref{fm2def4}}.
\end{thm}

\subsection{Spin bordism groups of some free loop spaces}
\label{fm32}

In \S\ref{fm91} we will show that the orientability of a large class of moduli spaces depends only on the images of the maps $\hat\xi^{\bs\Spin}_{n-1}(T)$ from \eq{fm2eq7}. The next theorem computes the spin bordism groups of free loop spaces of various Thom spaces $MH$ and the image of $\hat\xi^{\bs\Spin}_{n-1}(MH):\Om_{n-1}^{\bs\Spin}(\cL MH;MH)\ra\ab\ti\Om_n^{\bs\Spin}(MH)$, which we describe explicitly using the elements $\de,\ab\varep,\ab\ze_1,\ab\frac{\ze_1}{2},\ab\frac{\ze_1}{4},\ab\ze_2,\ab\ze_2',\ab\ze_3,\eta$ constructed in \S\ref{fm31}. The following theorem is proved in~\S\ref{fm18}.

\begin{thm}
\label{fm3thm2}
{\bf(a)} In dimensions\/ $n\le 8$ the relative spin bordism groups of the free loop space of\/ $MH$ are as given in Table~{\rm\ref{fm3tab2}}.

\begin{table}[!ht]
\renewcommand{\arraystretch}{1.5}
\begin{tabular}{p{4.7cm}|ccccccc}
$n$ & \centering $0,1,2$ & $3$ & $4$ & $5$ & $6$ & $7$ & $8$\\
\hline
$\Om_n^{\bs\Spin}(\cL M\SU(2);M\SU(2))$ & $0$ & $\Z$ & $\Z_2$ & $\Z_2$ & $0$ & $\Z^2$ & $\Z_2$\\
$\Om^{\bs\Spin}_n(\cL M\U(2);M\U(2))$ & $0$ & $\Z$ & $0$ & $\Z$ & $0$ & $\Z^3$ & $\Z$\\
$\ti\Om^{\bs\Spin}_n(\cL M\Spin(4);M\Spin(4))$ & $0$ & $\Z$ & $\Z_2$ & $\Z_2$ & $0$ & $\Z^3$ & $\Z_2^2$\\
$\Om^{\bs\Spin}_n(\cL K(\Z,4);K(\Z,4))$ & $0$ & $\Z$ & $0$ & $0$ & $0$ & $\Z^2$ & $\Z_2$\\
$\Om^{\bs\Spin}_n(\cL K(\Z_2,4);K(\Z_2,4))$ & $0$ & $\Z_2$ & $0$ & $0$ & $0$ & ? & $\Z_2^2$\\
\end{tabular}
\caption{The relative spin bordism groups of free loop spaces in dimensions $n\le 8$. The group\/ $\Om^{\bs\Spin}_7(\cL K(\Z_2,4);K(\Z_2,4))$ is not important for our purposes; it is either\/ $\Z_8,$ $\Z_4\op\Z_2,$ $\Z_4,$ or\/ $\Z_2^2$.}
\label{fm3tab2}
\end{table}

\noindent{\bf(b)} In dimensions\/ $n\le 9$ the images of the morphism\/ $\hat\xi^{\bs\Spin}_{n-1}(T):\Om^{\bs\Spin}_{n-1}(\cL T;T)\ab\ra\ti\Om^{\bs\Spin}_n(T)$ from \eq{fm2eq7} are as shown in Table~{\rm\ref{fm3tab3}}.
\smallskip

\begin{table}[ht!]
\renewcommand{\arraystretch}{1.5}
\centering
\begin{tabular}{c|c|c|c}
$n$ & $\Im\hat\xi^{\bs\Spin}_{n-1}(M\SU(2))$ & $\Im\hat\xi^{\bs\Spin}_{n-1}(M\U(2))$ & $\Im\hat\xi^{\bs\Spin}_{n-1}(M\Spin(4))$\\ \hline
$0,1,2,3,7$ & & \\
$4$   &  $\Z\an{\de}$   &   $\Z\an{\de}$  &   $\Z\an{\de}$\\
$5$   &  $\Z_2\an{\al_1\de}$  &   &  $\Z_2\an{\al_1\de}$\\
$6$   &  $\Z_2\an{\al_1^2\de}$  &  $\Z\an{\varepsilon}$  & $\Z_2\an{\al_1^2\de}$\\
$8$   &  $\Z\an{\ze_1,2\ze_2}$  &  $\Z\an{\frac{\ze_1}{2},2\ze_2,\ze_3}$           &  $\Z\an{\ze_1,\ze_2-\ze_2',\ze_2+\ze_2'}$\\
$9$   &     & $\Z_2\an{\al_1\ze_2}$  &  $\Z_2\an{\al_1(\ze_2+\ze_2')}$\\
\end{tabular}
\bigskip

\centering
\begin{tabular}{c|c|c|c}
$n$ & \!$\Im\hat\xi^{\bs\Spin}_{n-1}(M\SO(4))$\!\! & \!$\Im\hat\xi^{\bs\Spin}_{n-1}(K(\Z,4))$\!\! & \!$\Im\hat\xi^{\bs\Spin}_{n-1}(K(\Z_2,4))$\!\! \\ \hline
$0,1,2,3,5,6,7$ & & \\
$4$   &  $\Z\an{\de}$  &  $\Z\an{\de}$ &  $\Z_2\an{\de}$ \\
$8$   &  $\Z\an{\frac{\ze_1}{4},2\ze_2,\ze_3}$  & $\Z\an{2\ze_2,\ze_3}$ & $\Z_2\an{2\ze_2}$ \\
$9$   &  $\Z_2\an{\al_1\frac{\ze_1}{4},\al_1\ze_2,\eta}$   &  $\Z_2\an{\al_1\ze_2}$ &  $\Z_2\an{\al_1\ze_2}$ \\
\end{tabular}

\caption{The image $\Im\hat\xi^{\bs\Spin}_{n-1}(T)\subset\ti\Om_n^{\bs\Spin}(T)$ of the morphism \eq{fm2eq7}. \\ A missing entry stands for the trivial group.}
\label{fm3tab3}
\end{table}

\noindent{\bf(c)} Classes in $\Om^{\bs\Spin}_n(\cL K(\Z_2,4);K(\Z_2,4))$ are represented by pairs\/ $[X,\bar\al]$ of a compact spin\/ $8$-manifold\/ $X$ and a class\/ $\bar\al\in H^4(X\t\cS^1,\Z_2)$. Under the identification $\ti\Om^{\bs\Spin}_9(K(\Z_2,4))\cong\Z_2$ the map\/ $\hat\xi^{\bs\Spin}_8(K(\Z_2,4)):\Om^{\bs\Spin}_8(\cL K(\Z_2,4),\ab K(\Z_2,4))\ra\ti\Om^{\bs\Spin}_9(K(\Z_2,4))$ is
\e
\label{fm3eq22}
 \hat\xi^{\bs\Spin}_8(K(\Z_2,4))([X,\bar\al])=\ts\int_X\bar\be\cup\Sq^2(\bar\be),
\e
where we decompose\/ $\bar\al=\bar\be\bt[\cS^1]+\bar\ga\bt 1$ under the K\"unneth iso\-morph\-ism. Similarly, $\hat\xi^{\bs\Spin}_8(K(\Z,4)):\Om^{\bs\Spin}_8(\cL K(\Z,4);K(\Z,4))\ra\ti\Om^{\bs\Spin}_9(K(\Z,4))\cong\Z_2$ maps\/ $[X,\al]\mapsto\int_X\bar\be\cup\Sq^2(\bar\be),$ where we instead decompose the reduction of\/ $\al$ modulo two.
\end{thm}

\subsection{Relation between certain Thom and classifying spaces}
\label{fm33}

We are interested in orientation problems both for moduli spaces of calibrated submanifolds and for moduli spaces of instantons. These turn out to be closely connected. Indeed, the next theorem shows that, in low dimensions, maps from a compact $n$-manifold $X$ into a Thom space $MH$ (which describe submanifolds $M\subset X$) are equivalent to maps from $X$ into a classifying space $BG$ (which describe principal $G$-bundles $P\ra X$). Referring to \S\ref{fm25} for details, recall that the Thom class of an embedded $(n-k)$-submanifold $M\subset X$ can be regarded as a cohomology class $t\in H^k(X;X\sm M,R)$, and that it induces the Thom isomorphism $\pi^*(-)\cup t: H^*(M)\ra H^{*+k}(X;X\sm M,R)$, which we will use in equations \eq{fm3eq24}, \eq{fm3eq26}, and \eq{fm3eq28} below. The principal bundles appearing in these formulas are framed outside of $M$, so their characteristic classes are naturally elements of the relative cohomology $H^*(X;X\sm M,R)$.

Recall that a principal $\Sp(m)$-bundle $P$ has {\it symplectic Pontrjagin classes} $q_i(P)$ defined in terms of the Chern classes of the associated principal $\U(2m)$-bundle $Q\ra X$ by $q_i(P)=(-1)^i c_{2i}(Q)$. As $H^4(BE_8,\Z)\cong\Z$, a principal $E_8$-bundle has a characteristic class $a_1(P)\in H^4(X,\Z)$.

Recall that a map $f:X\ra Y$ of topological spaces is {\it $n$-connected} if the induced map $\pi_k(f):\pi_k(X)\ra\pi_k(Y)$ of homotopy groups is an isomorphism for all $k<n$ and surjective for $k=n$. If $n=\iy$, we call $f$ a {\it weak homotopy equivalence}. The following theorem is proved in~\S\ref{fm16}.

\begin{thm}
\label{fm3thm3}
{\bf(a)} There is a weak homotopy equivalence
\e
\label{fm3eq23}
M\SU(2)=M\Sp(1)\longra B\SU(2).
\e
If\/ $P\ra X$ is the principal\/ $\SU(2)$-bundle corresponding under \eq{fm3eq23} to the submanifold\/ $M\subset X$ with normal\/ $\Sp(1)$-structure, then
\e
\label{fm3eq24}
c_2(P)^k=\pi^*\bigl(c_2(\nu_M)^{k-1}\bigr)\cup t\qquad \text{for all\/ $k\ge 1$.}
\e

\noindent{\bf(b)} There is a\/ $10$-connected map
\e
\label{fm3eq25}
 M\U(2)\longra B\SU.
\e
Moreover, the canonical map\/ $B\SU(m)\ra B\SU$ is\/ $(2m+1)$-connected. If\/ $P\ra X$ is the principal\/ $\SU(m)$-bundle corresponding under \eq{fm3eq25} to the submanifold\/ $M\subset X$ with normal\/ $\U(2)$-structure, then
\e
\label{fm3eq26}
\begin{split}
c_1(P)&=0,\\
c_k(P)&=(-1)^k \pi^*\bigl(c_1(\nu_M)^{k-2}\bigr)\cup t\qquad \text{for all\/ $k\ge 2,$}\\
c_2(P)^2&=\pi^*\bigl(c_2(\nu_M)\bigr)t,\\
c_2(P)c_3(P)&=-\pi^*\bigl(c_1(\nu_M)c_2(\nu_M)\bigr)\cup t.
\end{split}
\e

\noindent{\bf(c)} There is a\/ $12$-connected map
\e
\label{fm3eq27}
M\Spin(4)\longra B\Sp.
\e
Moreover, the canonical map\/ $B\Sp(m)\ra B\Sp$ is\/ $(4m+3)$-connected. If\/ $P\ra X$ is the principal\/ $\Sp(m)$-bundle corresponding under \eq{fm3eq27} to the submanifold\/ $M\subset X$ with normal spinor bundles\/ $\Si_{\nu_M}^\pm\ra M,$ then
\e
\label{fm3eq28}
\begin{split}
q_k(P)&=-\pi^*\bigl(c_2(\Si^-_{\nu_M})^{k-1}\bigr)\cup t \qquad\text{for all\/ $k\ge 1,$}\\
q_1(P)^2&=\pi^*\bigl(c_2(\Si^+_{\nu_M})-c_2(\Si^-_{\nu_M})\bigr)\cup t,\\
q_1(P)^3&=-\pi^*\bigl(c_2(\Si^+_{\nu_M})^2-2c_2(\Si^+_{\nu_M})c_2(\Si^-_{\nu_M})+c_2(\Si^-_{\nu_M})^2\bigr)\cup t,\\
q_1(P)q_2(P)&=\pi^*\bigl(c_2(\Si^+_{\nu_M})c_2(\Si^-_{\nu_M})-c_2(\Si^-_{\nu_M})^2\bigr)\cup t.
\end{split}
\e

\noindent{\bf(d)} There is a\/ $16$-connected map
\e
\label{fm3eq29}
 BE_8\longra K(\Z,4).
\e
If\/ $P\ra X$ is the principal\/ $E_8$-bundle corresponding under \eq{fm3eq29} to a cohomology class\/ $\al\in H^4(X,\Z),$ then
\e
\label{fm3eq30}
 a_1(P)=\al.
\e
\end{thm}

Theorem \ref{fm3thm3} implies that Theorem \ref{fm3thm1} determines  the spin bordism groups $\ti\Om_n^{\bs\Spin}(BG)$ for $n\le 9$ of various classifying spaces. The generators in Table \ref{fm3tab1} correspond under \eq{fm3eq23}, \eq{fm3eq25}, \eq{fm3eq27} and the morphisms $B\SU(m)\ra B\SU$, $B\Sp(m)\ra B\Sp$ to principal $G$-bundles, which we denote by the same letter.

\begin{cor}
\label{fm3cor1}
{\bf(a)} In dimensions\/ $n\le 9$ the spin bordism groups\/ $\ti\Om_n^{\bs\Spin}(BG)$ for\/ $G=\SU(2),$ $\SU(m)$ for\/ $2m\ge n,$ $\Sp(m)$ for $4m+2\ge n,$ and\/ $E_8$ are as given in Table~{\rm\ref{fm3tab4}}.
\smallskip

\begin{table}[ht!]
\renewcommand{\arraystretch}{1.5}
\centering

\begin{tabular}{c|c|c|c}
$n$ & $\!\!\ti\Om_n^{\bs\Spin}(B\SU(2))\!\!$ & $\!\!\ti\Om_n^{\bs\Spin}(B\SU(m)),\!$ $2m\!\ge\! n$\! & $\!\!\ti\Om_n^{\bs\Spin}(B\Sp(m)),\!$ $4m\!+\!2\!\ge\! n$\!\! \\ \hline
{\rm 0--3,7} & & \\
$4$   &  $\Z\an{\de}$   &   $\Z\an{\de}$  &   $\Z\an{\de}$\\
$5$   &  $\Z_2\an{\al_1\de}$  &   &  $\Z_2\an{\al_1\de}$\\
$6$   &  $\Z_2\an{\al_1^2\de}$  &  $\Z\an{\varepsilon}$  & $\Z_2\an{\al_1^2\de}$\\
$8$   &  $\Z\an{\ze_1,\ze_2}$  &  $\Z\an{\frac{\ze_1}{2},\ze_2,\ze_3}$           &  $\Z\an{\ze_1,\ze_2,\ze_2'}$\\
$9$   &  $\Z_2\an{\al_1\ze_2}$   & $\Z_2\an{\al_1\ze_2}$  &  $\Z_2\an{\al_1\ze_2,\al_1\ze_2'}$
\end{tabular}
\bigskip

\centering
\begin{tabular}{c|c}
$n$ &  $\ti\Om_n^{\bs\Spin}(BE_8)$ \\ \hline
$0,1,2,3,5,6,7$ & \\
$4$   &   $\Z\an{\de}$\\
$8$   &  $\Z\an{\ze_2,\ze_3}$\\
$9$   &   $\Z_2\an{\al_1\ze_2}$
\end{tabular}

\caption{Spin bordism groups of classifying spaces $BG,$ \\ where a blank entry stands for the trivial group.}
\label{fm3tab4}
\end{table}

\noindent{\bf(b)}  Representing elements of\/ $\ti\Om_n^{\bs\Spin}(BG)$ as pairs\/ $[X,P],$ where\/ $X$ a compact spin $n$-manifold and\/ $P\ra X$ is a principal\/ $G$-bundle, the composition of the isomorphisms\/ $\ti\Om_n^{\bs\Spin}(BG)\cong\ti\Om_n^{\bs\Spin}(MH)$ from {\rm\eq{fm3eq23}, \eq{fm3eq25}, \eq{fm3eq27}, \eq{fm3eq29}} with the isomorphisms \eq{fm3eq11}--\eq{fm3eq19} have the following expressions.

For $B\SU(2)$ we have explicit isomorphisms
\ea
&\ti\Om_4^{\bs\Spin}(B\SU(2))\overset\cong\longra\Z,\qquad 
[X,P]\longmapsto\ts\int_X c_2(P),
\label{fm3eq31}\\
&\ti\Om_8^{\bs\Spin}(B\SU(2))\overset\cong\longra\Z^2,
\label{fm3eq32}\\
&[X,P]\longmapsto\left(\ts\int_X[-\frac{c_2(P)^2}{24}-\frac{p_1(TX)c_2(P)}{48}],\int_X c_2(P)^2\right),
\nonumber
\ea
which map\/ $\de\mapsto 1$ and\/ $\ze_1\mapsto(1,0),$ $\ze_2\mapsto(0,1)$.

For $B\SU(m),$ $m\ge 2,3$ and\/ $4$ respectively, we have explicit isomorphisms
\ea
&\ti\Om_4^{\bs\Spin}(B\SU(m))\overset\cong\longra\Z,\qquad 
[X,P]\longmapsto\ts\int_X c_2(P),
\label{fm3eq33}\\
&\ti\Om_6^{\bs\Spin}(B\SU(m))\overset\cong\longra\Z,\qquad [X,P]\longmapsto\ts\frac{1}{2}\int_X c_3(P),
\label{fm3eq34}\\
&\ti\Om_8^{\bs\Spin}(B\SU(m))\overset\cong\longra\Z^3,
\label{fm3eq35}\\
&[X,P]\longmapsto \left(\ts\int_X[\frac{c_4(P)}{6} - \frac{c_2(P)^2}{12} - \frac{p_1(TX)c_2(P)}{24}],\ts\int_X c_2(P)^2, \int_X c_4(P)\right),\nonumber
\ea
which map\/ $\de\mapsto 1,$ $\varep\mapsto 1$ and\/ $\frac{\ze_1}{2}\mapsto(1,0,0),$ $\ze_2\mapsto(0,1,0),$ $\ze_3\mapsto(0,0,1)$.

For $B\Sp(m),$ $m\ge 2,$ we have explicit isomorphisms
\ea
&\ti\Om_4^{\bs\Spin}(B\Sp(m))\overset\cong\longra\Z,\qquad 
[X,P]\longmapsto\ts-\int_X q_1(P),
\label{fm3eq36}\\
&\ti\Om_8^{\bs\Spin}(B\Sp(m))\overset\cong\longra\Z^3,\quad [X,P]\longmapsto
\label{fm3eq37}\\
 &\left(\ts\int_X[\frac{p_1(TX)q_1(P)}{48}-\frac{q_1(P)^2}{24}+\frac{q_2(P)}{12}],\ts\int_X [q_1(P)^2-q_2(P)],-\int_X q_2(P)\right),\nonumber
\ea
which map\/ $\de\mapsto 1$ and\/ $\ze_1\mapsto(1,0,0),$ $\ze_2\mapsto(0,1,0),$ $\ze_2'\mapsto(0,0,-1)$.

The isomorphism\/ $\ti\Om_8^{\bs\Spin}(BE_8)\cong\Z^2$ is just \eq{fm3eq19} for\/ $\al=a_1(P)$ the characteristic class of the principal\/ $E_8$-bundle\/ $P\ra X$.
\smallskip

\noindent{\bf(c)} Theorem {\rm\ref{fm3thm2}} determines the bordism groups\/ $\Om_n^{\bs\Spin}(\cL BG;BG),$ $n\le 8,$ for each of the groups\/ $G=\SU(2),$ $\SU(m)$ for $2m-1\ge n,$ $\Sp(m)$ for $4m+1\ge n,$ and\/ $E_8$. We do not write these out explicitly.
\smallskip

\noindent{\bf(d)} Theorem {\rm\ref{fm3thm2}} also determines the images $\Im\hat\xi^{\bs\Spin}_{n-1}(BG),$ $n\le 9,$ for each of the groups\/ $G=\SU(2),$ $\SU(m)$ for\/ $2m\ge n,$ $\Sp(m)$ for $4m+2\ge n,$ and\/ $E_8$. We give these in Table\/~{\rm\ref{fm3tab5}}.
 
\begin{table}[ht!]
\renewcommand{\arraystretch}{1.5}
\centering
\begin{tabular}{c|c|c}
$n$ & $\Im\hat\xi^{\bs\Spin}_{n-1}(B\SU(2))$  & $\Im\hat\xi^{\bs\Spin}_{n-1}(B\Sp(m)),$ $4m+2\ge n$ \\ \hline
$0,1,2,3,7$ & & \\
$4$   &  $\Z\an{\de}$   &   $\Z\an{\de}$\\
$5$   &  $\Z_2\an{\al_1\de}$   &  $\Z_2\an{\al_1\de}$\\
$6$   &  $\Z_2\an{\al_1^2\de}$   & $\Z_2\an{\al_1^2\de}$\\
$8$   &  $\Z\an{\ze_1,2\ze_2}$  &  $\Z\an{\ze_1,\ze_2-\ze_2',\ze_2+\ze_2'}$\\
$9$   &      &  $\Z_2\an{\al_1(\ze_2+\ze_2')}$
\end{tabular}
\medskip

\centering
\begin{tabular}{c|c|c}
$n$ & $\Im\hat\xi^{\bs\Spin}_{n-1}(B\SU(m)),$ $2m\ge n$ & $\Im\hat\xi^{\bs\Spin}_{n-1}(BE_8)$ \\ \hline
$0,1,2,3,5,7$ & & \\
$4$ & $\Z\an{\de}$  &   $\Z\an{\de}$\\
$6$ & $\Z\an{\varepsilon}$ &      \\
$8$ & $\Z\an{\frac{\ze_1}{2},2\ze_2,\ze_3}$ &   $\Z\an{2\ze_2,\ze_3}$\\
$9$ &  $\Z_2\an{\al_1\ze_2}$ &    $\Z_2\an{\al_1\ze_2}$
\end{tabular}

\caption{The image $\Im\hat\xi^{\bs\Spin}_{n-1}(BG)\subset\ti\Om_n^{\bs\Spin}(BG)$ of the morphism \eq{fm2eq7}.}
\label{fm3tab5}
\end{table}

\end{cor}

\begin{proof}
As the coefficient groups $\Om^{\bs\Spin}_n(*)$ of spin bordism are concentrated in non-negative degrees $n\ge 0$, the Atiyah--Hirzebruch spectral sequence (see \S\ref{fm22}) implies that a $d$-connected map $f:X\ra Y$ induces an isomorphism $f_*:\ti\Om_n^{\bs\Spin}(X)\ra \ti\Om_n^{\bs\Spin}(Y)$ in all dimensions $n<d$. Hence Theorem \ref{fm3thm3} implies Theorem \ref{fm3thm1} in dimension $n\le 9$ in each of the cases. This proves (a).

The isomorphisms \eq{fm3eq9}--\eq{fm3eq15} are rewritten using \eq{fm3eq24}, \eq{fm3eq26} and \eq{fm3eq28}, which leads to \eq{fm3eq32}--\eq{fm3eq37}. As an example, we explain in detail how \eq{fm3eq15} implies \eq{fm3eq37}. By the Hirzebruch signature formula, $\frac{-\operatorname{sign}(M)}{16}=-\frac{1}{48}\int_M p_1(TM)$. Moreover, the splitting $TX|_M=TM\op\nu_M$ gives $p_1(TX)|_M=p_1(TM)+p_1(\nu_M)$ and $p_1(\nu_M)=-c_2(\nu_M\ot\C)$. But $\nu_M\ot\C\cong\Hom_\C(\Si_{\nu_M}^-,\Si_{\nu_M}^+)\ab\cong\Si_{\nu_M}^-\ot_\C\Si_{\nu_M}^+$, as the quaternionic structure makes $\Si_{\nu_M}^-$ self-conjugate. Hence $p_1(\nu_M)=-c_2(\Si_{\nu_M}^-\ot_\C\Si_{\nu_M}^+)=-2c_2(\Si_{\nu_M}^-)-2c_2(\Si_{\nu_M}^+)$. Substituting these expression into the calculation of $\frac{-\operatorname{sign}(M)}{16}$ gives
\begin{equation*}
 \frac{-\operatorname{sign}(M)}{16}=-\frac{1}{48}\int_M p_1(TX)|_M-\frac{1}{24}\int_M c_2(\Si_{\nu_M}^-)-\frac{1}{24}\int_M c_2(\Si_{\nu_M}^+).
\end{equation*}
According to \eq{fm3eq28}, $-q_1(P)$ is Poincar\'e dual to $[M]$, $\int_X q_2(P)=-\int_M c_2(\Si_{\nu_M}^-)$ and $\int_X[q_1(P)^2-q_2(P)]=\int_M c_2(\Si_{\nu_M}^+)$ which, in particular, gives the second and third component of \eq{fm3eq37}. We can now rewrite the integrals over $M$ in the previous equation as integrals over $X$,
\begin{equation*}
 \frac{-\operatorname{sign}(M)}{16}=\frac{1}{48}\int_X p_1(TX)q_1(P)+\frac{1}{24}\int_X q_2(P)-\frac{1}{24}\int_X[q_1(P)^2-q_2(P)],
\end{equation*}
giving the first component of~\eq{fm3eq37}.

Equations \eq{fm3eq31}--\eq{fm3eq36} have similar proofs, and are left to the reader. Parts (c),(d) are immediate from Theorems \ref{fm3thm2} and~\ref{fm3thm3}.
\end{proof}

\subsection{Spin bordism groups of some Lie groups}
\label{fm34}

Our next result, Theorem \ref{fm3thm5}, is a prerequisite for the proof of Theorem \ref{fm3thm2} and computes the spin bordism groups of various Lie groups and of the topological groups $K(\Z,3)$ and $K(\Z_2,3)$.

\subsubsection{Description of generators}
\label{fm341}

Again, we first construct elements in $\ti\Om_n^{\bs\Spin}(G)$ that will serve as generators, and describe a convention that allows us to regard these in different spin bordism groups. Classes in $\ti\Om_n^{\bs\Spin}(G)$ are represented by pairs $[X,\phi]$ of a compact spin $n$-manifold $X$ and a continuous map $\phi:X\ra G$. A continuous group morphism $G_1\ra G_2$ induces a morphism of spin bordism groups, so we can view elements in $\ti\Om_n^{\bs\Spin}(G_1)$ as elements in $\ti\Om_n^{\bs\Spin}(G_2)$. {\it We use the same letter to denote these elements.} We summarize the situation in the following diagram.
\begin{equation}
\begin{tikzcd}[column sep=small]
			& \SU\arrow[dr]\\
	\SU(2)=\Sp(1) \arrow[ru]\arrow[rd]	&		& K(\Z,3)\rar & K(\Z_2,3). \\
			& \Sp\arrow[ur]
\end{tikzcd}
\label{fm3eq38}
\end{equation}

Recall from \cite[Chs.~3--4]{BCM} that $H^*(\SU,\Z)$ is an exterior algebra with generators $b_i$ of degree $|b_i|=2i-1$ for all $i\ge 2$. The transgression of $b_i$ is the Chern class $c_i\in H^{2i}(B\SU,\Z)$. For the symplectic group, $H^*(\Sp,\Z)$ is an exterior algebra with generators $a_j$ of degree $|a_j|=4j-1$ for all $j\ge 1$. The transgression of $a_j$ is the symplectic Pontrjagin class $q_j\in H^{4j}(B\Sp,\Z)$. The map in cohomology induced by $\Sp\ra\SU$ maps $b_{2i}\mapsto(-1)^ia_i$, $b_{2i+1}\mapsto0$ and, similarly, $B\Sp\ra B\SU$ maps $c_{2i}\mapsto(-1)^i q_i$, $c_{2i+1}\mapsto0$. Our convention in \eq{fm3eq38} is that the map $\SU\ra K(\Z,3)$ classifies $b_2\in H^3(\SU,\Z)$ and $\Sp\ra K(\Z,3)$ classifies $-a_1\in H^3(\Sp,\Z)$, so that \eq{fm3eq38} commutes up to homotopy.

Observe that the map $\Sp\ra\SU$ is not compatible with the zig-zags $B\SU\overset{\enskip\simeq_{10}}{\longleftarrow} M\U(2)\ra M\SO(4)$ and $B\Sp\overset{\enskip\simeq_{12}}{\longleftarrow} M\Spin(4)\ra M\SO(4)$ constructed using \eq{fm3eq25} and \eq{fm3eq27} (indeed, the maps \eq{fm3eq25} and \eq{fm3eq27} are constructed in a very different way). To avoid confusion, we will therefore {\it not} apply the convention above for group morphisms $G_1\ra G_2$ in the case $\Sp\ra\SU$.

\paragraph{Generator $\rho$ in Dimension 3.}

Let $\phi:\cS^3\ra\SU(2)$ represent the generator of $\pi_3(\SU(2))\cong\Z$ and define
\begin{equation*}
\rho=[\cS^3,\phi]\quad\text{in}\quad\ti\Om_3^{\bs\Spin}(\SU(2)).
\end{equation*}

\paragraph{Generator $\varsigma$ in Dimension 5.}

Let $\phi:\cS^5\ra\SU(3)$ represent the generator of $\pi_5(\SU(3))\cong\Z$ and define
\begin{equation*}
\varsigma=[\cS^5,\phi] \quad\text{in}\quad\ti\Om_5^{\bs\Spin}(\SU).
\end{equation*}

\paragraph{Generators $\vartheta_1,\frac{\vartheta_1}{2},\vartheta_2,\vartheta_3$ in Dimension 7.}

The spin $7$-manifold $K3\t\cS^3$ and the projection $\phi: K3\t\cS^3\ra\cS^3\cong\SU(2)$ define
\begin{equation*}
\vartheta_1=[K3\t\cS^3,\phi] \quad\text{in}\quad\ti\Om_7^{\bs\Spin}(\SU(2)).
\end{equation*}
The image of $\vartheta_1$ in $\ti\Om_7^{\bs\Spin}(\SU)$ is naturally divisible by two, by \eq{fm3eq45} below; an explicit construction is more complicated: recall from above that $\frac{\ze_1}{2}$ is represented by the submanifold $(K3\t\cS^3)/\Z_2\an{\al}\t\{1\}$ of $(K3\t\cS^3)/\Z_2\an{\al}\t\cS^1_{\rm b}$. As the submanifold does not meet $(K3\t\cS^3)/\Z_2\an{\al}\t\{-1\}$, the Pontrjagin--Thom collapse determines a map $(K3\t\cS^3)/\Z_2\an{\al}\ra \Om M\U(2)$ into the {\it based} loop space, which we can map via \eq{fm3eq25} to $\Om B\SU\simeq \SU$. This gives $\phi:(K3\t\cS^3)/\Z_2\an{\al}\ra\SU$ and defines
\begin{equation*}
\frac{\vartheta_1}{2}=[(K3\t\cS^3)/\Z_2\an{\al},\phi] \quad\text{in}\quad\ti\Om_7^{\bs\Spin}(\SU).
\end{equation*}

Next, let $(r,s)\in\Sp(1)\t\Sp(1)$ act on the right of $(A,q)\in\Sp(2)\t\Sp(1)$ by
\begin{equation*}
 \left(A\begin{pmatrix}r & 0\\ 0 & s\end{pmatrix}, s^{-1}q s\right).
\end{equation*}
The quotient manifold $X=(\Sp(2)\t\Sp(1))/(\Sp(1)\t\Sp(1))$ is a compact spin $7$-manifold. We can construct a continuous map
\begin{equation*}
 \phi: X\longra\Sp(2),\quad \left(\begin{pmatrix}a & b\\ c& d\end{pmatrix},q\right)\longmapsto
 \begin{pmatrix}
 |a|^2 + bq\bar{b} & a\bar{c} + bq\bar{d}\\
 c\bar{a}+dq\bar{b} & |c|^2 + dq\bar{d}
 \end{pmatrix}
\end{equation*}
which we use to define
\begin{align*}
 \vartheta_2&=[(\Sp(2)\t\Sp(1))/(\Sp(1)\t\Sp(1)),\phi] \quad\text{in}\quad\ti\Om_7^{\bs\Spin}(\Sp).
\end{align*}

Finally, the spin $7$-manifold $\CP^3\t\cS^1_{\rm b}$ has a map $\phi_0:\CP^3\t\cS^1\ra\U(4)$ that sends a $1$-dimensional subspace $L\subset\C^4$ and $\la\in\cS^1$ to the transformation $\la\id_L\op\id_{L^\perp}$ of $\C^4$. Define $\phi:\CP^3\t\cS^1\ra\SU(5)$ by $\phi(L,\la)=(\phi_0(L),\la^{-1})$ and
\begin{equation*}
\vartheta_3=[\CP^3\t\cS^1_{\rm b},\phi] \quad\text{in}\quad\ti\Om_7^{\bs\Spin}(\SU).
\end{equation*}

\paragraph{Generator $\upsilon$ in Dimension 8.}

The compact spin $8$-manifold $\SU(3)$ and the natural map $\phi:\SU(3)\ra\SU$ define
\begin{equation*}
\upsilon=[\SU(3),\phi] \quad\text{in}\quad\ti\Om_8^{\bs\Spin}(\SU).
\end{equation*}

\subsubsection{Statement of theorem}
\label{fm342}

The following theorem is proved in \S\ref{fm17}.

\begin{thm}
\label{fm3thm5}
{\bf(a)} In dimension\/ $n\le 8$ the spin bordism groups of the spaces\/ $\SU(2),$ $\SU,$ $\Sp,$ $K(\Z,3),$ and\/ $K(\Z_2,3)$ are as shown in Table~{\rm\ref{fm3tab6}}.
\smallskip

\begin{table}[!ht]
\renewcommand{\arraystretch}{1.5}
\centering

\begin{tabular}{c|c|c|c|c|c}
$n$ & \!\!$\ti\Om_n^{\bs\Spin}(\SU(2))$\!\! & \!\!$\ti\Om_n^{\bs\Spin}(\SU)$\!\! & \!\!$\ti\Om_n^{\bs\Spin}(\Sp)$\!\! & \!\!$\ti\Om_n^{\bs\Spin}(K(\Z,3))$\!\! & \!\!$\ti\Om_n^{\bs\Spin}(K(\Z_2,3))$\!\! \\ \hline
$3$ &$\Z\an{\rho}$			&$\Z\an{\rho}$	&$\Z\an{\rho}$	&$\Z\an{\rho}$	&$\Z_2\an{\rho}$\\
$4$ &$\Z_2\an{\al_1\rho}$	&		&$\Z_2\an{\al_1\rho}$	&		&\\
$5$ &$\Z_2\an{\al_1^2\rho}$	&$\Z\an{\varsigma}$	&$\Z_2\an{\al_1^2\rho}$	&		&\\
$7$ &$\Z\an{\vartheta_1}$		&$\Z\an{\frac{\vartheta_1}{2},\vartheta_3}$	&$\Z\an{\vartheta_1,\vartheta_2}$	&$\Z\an{\vartheta_2}$	& $\Z_2^2$ or $\Z_4$\\
$8$ &						&$\Z\an{\upsilon}$	&$\Z_2\an{\al_1\vartheta_2}$	&$\Z_2\an{\upsilon}$	&$\Z_2\an{\upsilon}$\\
\end{tabular}

\caption{The spin bordism groups $\ti\Om_n^{\bs\Spin}(G)$ for all $n\le 8$ and\/ $G=K(\Z,3),\ab K(\Z_2,3),\ab \SU(2),\ab \SU,\ab \Sp$. A missing entry stands for the trivial group. The group $\ti\Om_7^{\bs\Spin}(K(\Z_2,3))$ is either $\Z_2^2$ or $\Z_4$; it will not be needed.}
\label{fm3tab6}
\end{table}

\noindent{\bf(b)}  The morphisms between the groups in Table {\rm\ref{fm3tab6}} induced by the maps in \eq{fm3eq38} are as indicated by the notation and the rules
\e
\begin{aligned}
\ti\Om^{\bs\Spin}_*(\SU)&\longra\ti\Om^{\bs\Spin}_*(K(\Z,3)),
&\frac{\vartheta_1}{2}&\longmapsto -6\vartheta_2, &\vartheta_3&\longmapsto\vartheta_2,\\
\ti\Om^{\bs\Spin}_*(\Sp)&\longra\ti\Om^{\bs\Spin}_*(K(\Z,3)),
&\vartheta_1&\longmapsto -12\vartheta_2, &\al_1\ze_2&\longmapsto 0.
\end{aligned}
\label{fm3eq39}
\e

\noindent{\bf(c)}  The action of\/ $\al_1\in\Om_1^{\bs\Spin}(*)$ is as indicated by the notation and the rules
\ea
\al_1\frac{\vartheta_1}{2}&=0,
&\al_1\vartheta_1&=0,
&\al_1\vartheta_3&=0,
&\al_1\vartheta_2&=\upsilon.
\label{fm3eq40}
\ea

\noindent{\bf(d)}  For\/ $\SU(2)$ we have explicit isomorphisms
\ea
 \ti\Om^{\bs\Spin}_3(\SU(2))&\longra\Z,		& [X,\phi]&\longmapsto\ts\int_X\phi^*(b_2),
\label{fm3eq41}\\
 \ti\Om^{\bs\Spin}_7(\SU(2))&\longra\Z,		& [X,\phi]&\longmapsto -\ts\int_X\frac{p_1(TX)\phi^*(b_2)}{48},
\label{fm3eq42}
\ea
which map\/ $\rho\mapsto 1,$ $\vartheta_1\mapsto 1$.

For\/ $\SU$ we have explicit isomorphisms
\ea
\ti\Om^{\bs\Spin}_3(\SU)&\longra\Z, & [X,\phi]&\longmapsto\ts\int_X\phi^*(b_2),
\label{fm3eq43}\\
\ti\Om^{\bs\Spin}_5(\SU)&\longra\Z, & [X,\phi]&\longmapsto\ts\frac12\int_X\phi^*(b_3),
\label{fm3eq44}\\
\ti\Om^{\bs\Spin}_7(\SU)&\longra\Z^2, & 
\label{fm3eq45}\\
[X,\phi]&\mathrlap{\longmapsto\left(\ts\int_X[\frac{\phi^*(b_4)}{6}-\frac{p_1(TX)\phi^*(b_2)}{24}],\ts\int_X\phi^*(b_4)\right),}\nonumber\\
\ti\Om^{\bs\Spin}_8(\SU)&\longra\Z, & [X,\phi]&\longmapsto\ts\int_X\phi^*(b_2\cup b_3),
\label{fm3eq46}
\ea
which map $\rho\mapsto 1,$ $\varsigma\mapsto 1,$ $\frac{\vartheta_1}{2}\mapsto(1,0),$ $\vartheta_3\mapsto (0,1),$ $\upsilon\mapsto 1$.

For\/ $\Sp$ we have explicit isomorphisms
\ea
 \ti\Om^{\bs\Spin}_3(\Sp)&\longra\Z,		& [X,\phi]&\longmapsto -\ts\int_X\phi^*(a_1),
\label{fm3eq47}\\
 \ti\Om^{\bs\Spin}_7(\Sp)&\longra\Z^2,		
\label{fm3eq48}\\
 [X,\phi]&\longmapsto
 \mathrlap{\left(\ts\int_X \frac{p_1(TX)\phi^*(a_1)}{48}+\frac{\phi^*(a_2)}{12},-\ts\int_X \phi^*(a_2)\right),}\nonumber
\ea
which map\/ $\rho\mapsto 1,$ $\vartheta_1\mapsto(1,0),$ $\vartheta_2\mapsto(0,-1)$.

For\/ $K(\Z,3)$ we have explicit isomorphisms
\ea
&\ti\Om^{\bs\Spin}_3(K(\Z,3))\longra\Z, & &[X,\al]\longmapsto \ts\int_X\al,
\label{fm3eq49}\\
&\ti\Om^{\bs\Spin}_7(K(\Z,3))\longra\Z, & &[X,\al]\longmapsto \ts\frac{1}{4}\int_X p_1(TX)\cup\al,
\label{fm3eq50}\\
&\ti\Om^{\bs\Spin}_8(K(\Z,3))\longra\Z_2, & &[X,\al]\longmapsto \ts\int_X\bar\al\cup\Sq^2(\bar\al),
\label{fm3eq51}
\ea
which map\/ $\rho\mapsto 1,$ $\vartheta_2\mapsto -1,$ and\/ $\upsilon\mapsto 1$.

For\/ $K(\Z_2,3)$ we have explicit isomorphisms
\ea
 \ti\Om^{\bs\Spin}_3(K(\Z_2,3))&\longra\Z_2,		& [X,\bar\al]&\longmapsto \ts\int_X \bar\al,
 \label{fm3eq52}\\
 \ti\Om^{\bs\Spin}_8(K(\Z_2,3))&\longra\Z_2,		& [X,\bar\al]&\longmapsto \ts\int_X \bar\al\cup\Sq^2(\bar\al),
 \label{fm3eq53}
\ea
which map\/ $\rho\mapsto 1$ and\/ $\upsilon\mapsto 1$.
\end{thm}

When possible, we wrote the generators of $\ti\Om_n^{\bs\Spin}(MH)$ as images under $\hat\xi^{\bs\Spin}_{n-1}(MH)$ of elements in $\Om_{n-1}^{\bs\Spin}(\cL MH;MH)$. If we use Theorem \ref{fm3thm3} to identify $\Om_{n-1}^{\bs\Spin}(\cL MH;MH)\cong\Om_{n-1}^{\bs\Spin}(\cL BG;BG)$, it may be possible to lift classes further:

Let $G$ be a topological group. By clutching a pair of trivial bundles over the two cones in $\Si G$, we obtain a principal $G$-bundle over the suspension, classified by a map
\e
\label{fm3eq54}
 \chi:\Si G\longra BG.
\e
Its adjoint is a homotopy equivalence $G\ra \Om BG$. Consider the morphism
\e
\label{fm3eq55}
 \ti\Om_{n-1}^{\bs\Spin}(G)\longra \Om_{n-1}^{\bs\Spin}(\cL BG; BG)
\e
induced by $(G,*)\simeq(\Om BG,*)\ra (\cL BG,BG)$. There are topological abelian group structures on $G=K(\Z,3)$ and\/ $G=K(\Z_2,3)$ for which $BG$ is $K(\Z,4)$ and\/ $K(\Z_2,4)$, so we can apply the above constructions in these cases also.

The following proposition is proved in \S\ref{fm176}.

\begin{prop}
\label{fm3prop6}
There is a commutative diagram
\begin{equation}
\begin{tikzcd}[column sep=0ex]
& \Om_{n-1}^{\bs\Spin}(\cL BG;BG)\arrow[rd,"\hat\xi^{\bs\Spin}_{n-1}(BG)"]\\
\ti\Om_{n-1}^{\bs\Spin}(G)\arrow[ru,"\eq{fm3eq55}"]\arrow[rd,"\cong"] && \ti\Om_n^{\bs\Spin}(BG).\\
& \ti\Om_n^{\bs\Spin}(\Si G)\arrow[ru,"\chi_*"]
\end{tikzcd}
\label{fm3eq56}
\end{equation}
Here, we use the suspension isomorphism of the spin bordism generalized homology theory. Explicitly, the composition of the diagram takes\/ $\phi:X\ra G$ to the mapping torus principal\/ $G$-bundle $P_\phi\ra X\t \cS^1_{\rm b}$.

On generators, the composition
\e
\ti\Om_{n-1}^{\bs\Spin}(G)\longra\ti\Om_n^{\bs\Spin}(BG)
\label{fm3eq57}
\e
either way around\/ \eq{fm3eq56} has the following effect:
\e
\begin{aligned}
\rho&\mapsto\de, &&\text{any\/ $G$}, \\
\varsigma&\mapsto\varep, && \text{$G=\SU$},\\
\vartheta_1&\mapsto\ze_1,&&\text{$G=\SU(2)$},
&\ts\frac{\vartheta_1}{2}&\mapsto\ts\frac{\ze_1}{2},&&\text{$G=\SU$},\\
\vartheta_2&\mapsto\ze_2-\ze_2', && \text{$G=\Sp$}, &\vartheta_3&\mapsto\ze_3, && \text{$G=\SU$},\\
\upsilon&\mapsto\al_1\ze_2, && \text{$G=\SU$.}
\end{aligned}
\label{fm3eq58}
\e
\end{prop}

\section{Gauge-theoretic bordism categories}
\label{fm4}

\subsection{\texorpdfstring{Bordism categories $\Bord^{\bs B}_n(BG)$}{Bordism categories Bordₙᴮ(BG)}}
\label{fm41}

\begin{dfn}
\label{fm4def1}
Fix a dimension $n\ge 0$, a tangential structure $\bs B$ in the sense of \S\ref{fm211}, and a Lie group $G$. We will define a symmetric monoidal category $\Bord^{\bs B}_n(BG)$ that we call a {\it bordism category}.
\begin{itemize}
\setlength{\itemsep}{0pt}
\setlength{\parsep}{0pt}
\item[(a)] {\it Objects\/} of $\Bord^{\bs B}_n(BG)$ are pairs $(X,P)$, where $X$ is a compact $n$-manifold without boundary with a $\bs B$-structure $\bs\ga_X$, which we generally omit from the notation, and $P\ra X$ is a principal $G$-bundle.
\item[(b)] {\it Morphisms\/} $[Y,Q]:(X_0,P_0)\ra(X_1,P_1)$ in $\Bord^{\bs B}_n(BG)$ are equivalence classes of pairs $(Y,Q),$ see (c), where $Y$ is a compact $(n+1)$-manifold with $\bs B$-structure $\bs\ga_Y$, there is a chosen isomorphism $\pd Y\cong -X_0\amalg X_1$ of the boundary preserving $\bs B$-structures (where $-X_0$ indicates that $X_0$ has the opposite $\bs B$-structure $-\bs\ga_{X_0}$), and $Q\ra Y$ is a principal $G$-bundle with a chosen isomorphism $Q\vert_{\pd Y}\cong P_0\amalg P_1$. We suppress the isomorphisms from the notation.
\item[(c)] In the situation of (b), let $(Y_0,Q_0)$ and $(Y_1,Q_1)$ be two choices for $(Y,Q)$. We say that $(Y_0,Q_0)\sim(Y_1,Q_1)$ if there exists a pair $(Z,R)$, where $Z$ is a compact $(n+2)$-manifold with corners and $\bs B$-structure $\bs\ga_Z$, with a chosen isomorphism of boundaries identifying $\bs B$-structures
\e
\pd Z\cong (-X_0\t[0,1])\amalg (X_1\t[0,1]) \amalg -Y_0\amalg Y_1
\label{fm4eq1}
\e
such that along $\pd^2Z$ we identify $\pd Y_i$ with $(-X_0\amalg X_1)\t\{i\}$ for $i=0,1$ in the obvious way, and $R\ra Z$ is a principal $G$-bundle such that under \eq{fm4eq1} we have
\begin{equation*}
R\vert_{\pd Z}\cong (P_0\t[0,1])\amalg (P_1\t[0,1])\amalg Q_0\amalg Q_1,
\end{equation*}
with the obvious compatibility with the chosen isomorphisms $Q_i\vert_{\pd Y_i}\cong P_0\amalg P_1$ over $(X_0\amalg X_1)\t\{i\}$. It is easy to see that `$\sim$' is an equivalence relation, so the equivalence classes $[Y,Q]$ are well defined.
\item[(d)] If $[Y,Q]:(X_0,P_0)\ra(X_1,P_1)$ and $[Y',Q']:(X_1,P_1)\ra(X_2,P_2)$ are morphisms, the {\it composition\/} is
\begin{equation*}
[Y',Q']\ci [Y,Q]=[Y'\amalg_{X_1}Y,Q'\amalg_{P_1}Q]:(X_0,P_0)\longra(X_2,P_2).
\end{equation*}
That is, we glue $Y,Y'$ along their common boundary component $X_1$ to make a manifold $Y'\amalg_{X_1}Y$ with $\bs B$-structure and boundary $\pd(Y'\amalg_{X_1}Y)=-X_0\amalg X_2$. To define the smooth structure on $Y'\amalg_{X_1}Y$ we should choose `collars' $X_1\t(-\ep,0]\subset Y$, $X_1\t[0,\ep)\subset Y'$ of $X_1$ in $Y,Y'$, and similarly for $Q,Q'$, but the choices do not change the equivalence class $[Y'\amalg_{X_1}Y,Q'\amalg_{P_1}Q]$. Composition is associative.
\item[(e)] If $(X,P)$ is an object in $\Bord^{\bs B}_n(BG),$ the {\it identity morphism\/} is
\begin{equation*}
\id_{(X,P)}=\bigl[X\t[0,1],P\t[0,1]\bigr]:(X,P)\longra(X,P).
\end{equation*}
\item[(f)] If $[Y,Q]:(X_0,P_0)\ra(X_1,P_1)$ is a morphism, we can prove that it has an inverse morphism
\begin{align*}
[Y,Q]^{-1}&=\bigl[-Y\amalg (Y\amalg_{X_0\amalg X_1}-Y),Q\amalg(Q\amalg_{P_0\amalg P_1}-Q)\bigr]:\\
&\qquad (X_1,P_1)\longra(X_0,P_0),
\end{align*}
noting that $\pd(-Y)=-(-X_0\amalg X_1)=-X_1\amalg X_0$. Thus the category $\Bord^{\bs B}_n(BG)$ is a {\it groupoid}, that is, all morphisms are isomorphisms.
\item[(g)] Define a {\it monoidal structure\/} $\ot$ on $\Bord^{\bs B}_n(BG)$ by, on objects
\begin{equation*}
(X,P)\ot (X',P')=(X\amalg X',P\amalg P'),
\end{equation*}
and if $[Y,Q]:(X_0,P_0)\ra(X_1,P_1)$, $[Y',Q']:(X_0',P_0')\ra(X_1',P_1')$ are morphisms, then
\begin{align*}
&[Y,Q]\ot [Y',Q']=[Y\amalg Y',Q\amalg Q']:\\
&(X_0,P_0)\ot(X_0',P_0')\longra(X_1,P_1)\ot (X_1',P_1').
\end{align*}
This is compatible with `$\sim$', and with compositions and identities.
\item[(h)] The {\it identity\/} in $\Bord^{\bs B}_n(BG)$ is $\boo=(\es,\es)$.
\item[(i)] If $(X,P)\in\Bord^{\bs B}_n(BG)$ we write $-(X,P)=(-X,P)$, that is, we give $X$ the opposite $\bs B$-structure $-\bs\ga_X$. Observe that we have an isomorphism
\begin{equation*}
\bigl[X\t[0,1],P\t[0,1]\bigr]:(-X,P)\ot(X,P)\longra\boo.
\end{equation*}
Thus $-(X,P)$ is an inverse for $(X,P)$ under `$\ot$'.
\item[(j)] The {\it symmetry isomorphism} $\si_{(X,P),(X',P')}=[Y,Q]\colon(X,P)\ot(X',P')\ra(X',P')\ot(X,P)$
has $(Y,Q)=((X\amalg X')\t[0,1],(P\amalg P')\t[0,1])$ with the obvious identification of $\pd Y$ with the disjoint union of $-(X\amalg X')$ and~$X'\amalg X.$
\end{itemize}
Hence $\Bord^{\bs B}_n(BG)$ is a {\it Picard groupoid}, as in~Appendix \ref{fmA}.

In the case $G=\{1\}$ we will write $\Bord^{\bs B}_n(*)$ instead of $\Bord^{\bs B}_n(B\{1\})$. By definition, objects of $\Bord^{\bs B}_n(*)$ are pairs $(X,P)$, where $P\ra X$ is a principal $\{1\}$-bundle. But as principal $\{1\}$-bundles are trivial (we may take $P\ra X$ to be $\id_X:X\ra X$) we may omit $P$, and write objects of $\Bord^{\bs B}_n(*)$ as $X$, morphisms as $[Y]:X_0\ra X_1$, and so on.

If $\ga:G_1\ra G_2$ is a morphism of Lie groups, there is an obvious functor
\e
F_\ga:\Bord^{\bs B}_n(BG_1)\longra\Bord^{\bs B}_n(BG_2)
\label{fm4eq2}
\e
mapping $(X,P)\mapsto (X,(P\t G_2)/G_1)$ on objects and $[Y,Q]\mapsto[Y,(Q\t G_2)/G_1]$ on morphisms, where $G_1$ acts on $P\t G_2$ by the principal bundle action on $P$, and by $g_1:g_2\mapsto g_2\cdot\ga(g_1)^{-1}$ on $G_2$. In particular, the morphisms $\{1\}\hookra G$, $G\twoheadrightarrow\{1\}$ induce functors $\Bord^{\bs B}_n(*)\ra\Bord^{\bs B}_n(BG)$ and~$\Bord^{\bs B}_n(BG)\ra\Bord^{\bs B}_n(*)$.

Similarly, a morphism of tangential structures induces a functor.
\end{dfn}

The next proposition motivates the name bordism category, and the choice of notation `$BG$' in $\Bord^{\bs B}_n(BG)$. It shows the $\Bord^{\bs B}_n(BG)$ can be understood explicitly using homotopy-theoretic methods. As in \S\ref{fm33}, the groups $\Om^{\bs B}_n(BG)$ are often explicitly computable.

\begin{prop}
\label{fm4prop1}
{\bf(a)} $\Bord^{\bs B}_n(BG)$ is a Picard groupoid. Its invariants in Theorem\/ {\rm\ref{fmAthm2}(a)} are the $\bs B$-bordism groups
\ea
\pi_0\bigl(\Bord^{\bs B}_n(BG) \bigr)&\cong\Om^{\bs B}_n(BG),
\label{fm4eq3}\\
\pi_1\bigl(\Bord^{\bs B}_n(BG) \bigr)&\cong\Om_{n+1}^{\bs B}(BG),
\label{fm4eq4}
\ea
and\/ $q:\Om^{\bs B}_n(BG)\ra\Om_{n+1}^{\bs B}(BG)$ mapping\/ $[X,P]\mapsto[X\t\cS^1,P\t\cS^1]$. 

Here in $X\t\cS^1,$ the $\cS^1$ has a $\U(1)$-equivariant\/ $\bs B$-structure with the usual orientation, so when $\bs B=\bs\Spin,$ it has the \begin{bfseries}non-bounding\end{bfseries} spin structure $\cS^1_{\rm nb}$. Note that for $\bs B=\bs\Spin,$ this means that\/ $q$ is multiplication by $\al_1=[\cS^1_{\rm nb}]$ in $\Om_1^{\bs\Spin}(*)$ in Table\/ {\rm\ref{fm2tab2}} under the natural action of\/ $\Om_*^{\bs\Spin}(*)$ on\/~$\Om_*^{\bs\Spin}(BG)$.
\smallskip

\noindent{\bf(b)} The isomorphisms \eq{fm4eq3} and \eq{fm4eq4} are compatible with change of group functors, in particular with\/ $\{1\}\ra G.$
\smallskip

\noindent{\bf(c)} As in every Picard groupoid, a morphism\/ $\la:(X_0,P_0)\ab\ra(X_1,P_1)$ in\/ $\Bord^{\bs B}_n(BG)$ determines a bijection
\e
\Om_{n+1}^{\bs B}(BG)\longra\Hom_{\Bord^{\bs B}_n(BG)}\bigl((X_0,P_0),(X_1,P_1)\bigr)
\label{fm4eq5}
\e
given by composition in the diagram of bijections
\e
\begin{gathered}
\xymatrix@C=144pt@R=15pt{ 
*+[r]{\Om_{n+1}^{\bs B}(BG)} \ar[d]^{\eq{fm4eq4}} \ar[r]_(0.22)\cong &  *+[l]{\Hom_{\Bord^{\bs B}_n(BG)}\bigl((X_0,P_0),(X_1,P_1)\bigr)} \ar@{=}[d] \\
*+[r]{\Hom_{\Bord^{\bs B}_n(BG)}(\boo,\boo)} \ar[r]^(0.33){\raisebox{5pt}{$\scriptstyle\ot\la$}}  & *+[l]{\Hom_{\Bord^{\bs B}_n(BG)}\bigl(\boo\!\ot\!(X_0,P_0),\boo\!\ot\!(X_1,P_1)\bigr).}
}\!\!\!\!
\end{gathered}
\label{fm4eq6}
\e

\noindent{\bf(d)} For the category $\Bord^{\bs B}_n(*),$ the analogues of\/ {\rm\eq{fm4eq3}--\eq{fm4eq4}} are
\e
\pi_0(\Bord^{\bs B}_n(*))\cong \Om^{\bs B}_n(*),\quad \pi_1(\Bord^{\bs B}_n(*))\cong \Om_{n+1}^{\bs B}(*).
\label{fm4eq7}
\e
Under the identifications {\rm\eq{fm4eq7}, \eq{fm4eq3}, \eq{fm4eq4},} the functor $F_\inc$ from $\inc:\{1\}\hookra G$ induces the morphisms $\Om_m^{\bs B}(*)\ra\Om_m^{\bs B}(BG)$ induced by $*\ra BG,$ $*\mapsto\iy$ for~$m=n,n+1$.
\end{prop}

\begin{proof} 
For (a), let $(X,P)\in\Bord^{\bs B}_n(BG).$ Write $-X$ for $X$ with the opposite $\bs B$-structure, as in Definition \ref{fm2def1}. There is an isomorphism $\bigl[X\t[0,1],P\t[0,1]\bigr]:(-X,P)\ot(X,P)\ra\boo$ and hence $(-X,P)$ is an inverse for $(X,P)$ under the monoidal structure. If $[Y,Q]:\ab(X_0,P_0)\ra(X_1,P_1)$ is a morphism, we can prove that it has the inverse morphism
\begin{equation*}
[Y,Q]^{-1}=\bigl[-Y\amalg (Y\amalg_{X_0\amalg X_1}-Y),Q\amalg(Q\amalg_{P_0\amalg P_1}Q)\bigr],
\end{equation*}
where we note that $\pd(-Y)=-(-X_0\amalg X_1)=-X_1\amalg X_0.$ Thus, all morphisms in $\Bord^{\bs B}_n(BG)$ are isomorphisms. 

To prove \eq{fm4eq3}, let $(X,P)\in\Bord^{\bs B}_n(BG)$. The principal $G$-bundle $P\ra X$ is classified by a continuous map $f_P:X\ra BG,$ unique up to homotopy, with $f_P^*(EG)\cong P$ for $EG\ra BG$ the universal principal $G$-bundle. Hence $[X,f_P]\in\Om^{\bs B}_n(BG).$ If $[W,Q]:(X_0,P_0)\to(X_1,P_1)$ is a morphism in $\Bord^{\bs B}_n(BG),$ then $\pd W=-X_0\amalg X_1$ and the classifying map $f_Q: W\ra BG$ of the principal $G$-bundle $Q\ra W$ can be chosen to extend $f_{P_0}\amalg f_{P_1}.$ Hence $[X,f_P]$ depends only on the isomorphism class $[X,P]$, so mapping $[X,P]\mapsto[X,f_P]$ gives the map $\pi_0(\Bord^{\bs B}_n(BG))\ra\Om^{\bs B}_n(BG)$ in \eq{fm4eq3}. The inverse map takes $[X,f]$ to $[X,f^*(EG)]$. Here $f^*(EG)\ra X$ is initially a topological principal $G$-bundle, but it can be made into a smooth $G$-bundle uniquely up to isomorphism.

For \eq{fm4eq4}, morphisms $0\ra 0$ are equivalence classes $[W,Q]$ where $\pd W=-\es\amalg\es=\es$, so $W$ is without boundary. We then map $[W,Q]\mapsto[W,f]$ as for~\eq{fm4eq3}. 

We can show from the definitions that $q:\Om^{\bs B}_n(BG)\ra\Om_{n+1}^{\bs B}(BG)$ maps $P\ra X$ to the mapping torus of the $\Z_2$-action on $P\amalg P\ra X\amalg X$ that exchanges the two copies of $X$, so that $q\bigl([X,P]\bigr)=\bigl[((X\amalg X)\t[0,1])/\mathbin{\sim},\ab((P\amalg P)\t[0,1])/\mathbin{\sim}\bigr]$. But $((X\amalg X)\t[0,1])/\mathbin{\sim}\cong X\t\cS^1$ and $((P\amalg P)\t[0,1])/\mathbin{\sim}\cong P\t\cS^1$.

Part (b) is easy. Part (c) is immediate from the theory of Picard groupoids, and (d) follows from (a) with~$B\{1\}\simeq *$.	
\end{proof}

\begin{ex}
\label{fm4ex1}
From Proposition \ref{fm4prop1}(a) and Tables \ref{fm2tab1} and \ref{fm3tab4}, we see that there are equivalences of Picard groupoids
\e
\begin{aligned}
\Bord_7^{\bs\Spin}(B\SU(2))&\cong (0\op 0)\qs (\Z^2\op\Z^2), \\
\Bord_7^{\bs\Spin}(B\SU(m))&\cong (0\op 0)\qs (\Z^2\op\Z^3), \quad m\ge 4,\\
\Bord_7^{\bs\Spin}(B\Sp(m))&\cong (0\op 0)\qs(\Z^2\op\Z^3), \quad m\ge 2, \\
\Bord_7^{\bs\Spin}(B E_8)&\cong (0\op 0)\qs(\Z^2\op\Z_2^2), \\
\Bord_8^{\bs\Spin}(B\SU(2))&\cong (\Z^2\op \Z^2)\qs (\Z_2^2\op\Z_2), \\
\Bord_8^{\bs\Spin}(B\SU(m))&\cong (\Z^2\op \Z^3)\qs (\Z_2^2\op\Z_2), \quad m\ge 5,\\
\Bord_8^{\bs\Spin}(B\Sp(m))&\cong (\Z^2\op\Z^3)\qs(\Z_2^2\op\Z_2^2), \quad m\ge 2, \\
\Bord_8^{\bs\Spin}(B E_8)&\cong (\Z^2\op\Z^2)\qs(\Z_2^2\op\Z_2).
\end{aligned}
\label{fm4eq8}
\e
Here the right hand sides are Picard groupoids of the form $\pi_0\qs\pi_1$ as in Theorem \ref{fmAthm2} for abelian groups $\pi_0,\pi_1$. The decomposition of the $\pi_i$ as $A\op B$ in \eq{fm4eq8} corresponds to the splitting $\Om_n^{\bs\Spin}(BG)=\Om_n^{\bs\Spin}(*)\op\ti\Om_n^{\bs\Spin}(BG),$ where $\Om_n^{\bs\Spin}(*)$ is given in Table \ref{fm2tab1} and $\ti\Om_n^{\bs\Spin}(BG)$ in Table \ref{fm3tab4}. Note that $\pi_0\qs\pi_1$ also depends on a linear quadratic map $q:\pi_0\ra\pi_1$, which can be computed from Theorem \ref{fm3thm1} and Corollary~\ref{fm3cor1}.
\end{ex}

\subsection{\texorpdfstring{Loop bordism categories $\Bord^{\bs B}_n(\cL BG)$}{Loop bordism categories Bordₙᴮ(ℒBG)}}
\label{fm42}

\begin{dfn}
\label{fm4def2}
Fix a dimension $n\ge -1$, a tangential structure $\bs B$ in the sense of \S\ref{fm211}, and a Lie group $G$. We will define another Picard groupoid $\Bord^{\bs B}_n(\cL BG)$ that we call a {\it loop bordism category}. It is a simple modification of Definition \ref{fm4def1}: we replace the principal $G$-bundles $P\ra X$, $Q\ra Y$, $R\ra Z$ by principal $G$-bundles $P\ra X\t\cS^1$, $Q\ra Y\t\cS^1$, $R\ra Z\t\cS^1$ throughout. So, for example, objects of $\Bord^{\bs B}_n(\cL BG)$ are pairs $(X,P)$, where $X$ is a compact $n$-manifold without boundary with a $\bs B$-structure $\bs\ga_X$, which we generally omit from the notation, and $P\ra X\t\cS^1$ is a principal $G$-bundle.

In the case $G=\{1\}$ we will write $\Bord^{\bs B}_n(*)$ instead of $\Bord^{\bs B}_n(\cL B\{1\})$. Then the data $P\ra X\t\cS^1$, $Q\ra Y\t\cS^1$ is trivial, so we may write objects of $\Bord^{\bs B}_n(*)$ as $X$, morphisms as $[Y]:X_0\ra X_1$, and so on. This is equivalent to $\Bord^{\bs B}_n(*)$ in Definition~\ref{fm4def1}.

If $\ga:G_1\ra G_2$ is a morphism of Lie groups, as in \eq{fm4eq2} there is a functor
\e
F_\ga:\Bord^{\bs B}_n(\cL BG_1)\longra\Bord^{\bs B}_n(\cL BG_2)
\label{fm4eq9}
\e
mapping $(X,P)\mapsto (X,(P\t G_2)/G_1)$ on objects and $[Y,Q]\mapsto[Y,(Q\t G_2)/G_1]$ on morphisms. In particular, the morphisms $\{1\}\hookra G$, $G\twoheadrightarrow\{1\}$ induce functors $\Bord^{\bs B}_n(*)\ra\Bord^{\bs B}_n(\cL BG)$ and~$\Bord^{\bs B}_n(\cL BG)\ra\Bord^{\bs B}_n(*)$.

Similarly, a morphism of tangential structures induces a functor.
\end{dfn}

We relate the categories of Definitions \ref{fm4def1} and \ref{fm4def2}.

\begin{dfn}
\label{fm4def3}
Let $n\ge 0$ and $\bs B,G$ be as above. Define a functor
\e
I_n^{\bs B,G}:\Bord_{n-1}^{\bs B}(\cL BG)\longra \Bord^{\bs B}_n(BG)
\label{fm4eq10}
\e
to act on objects by $I_n^{\bs B,G}:(X,P)\mapsto (X\t\cS^1,P)$, and on morphisms by $I_n^{\bs B,G}:[Y,Q]\mapsto [Y\t\cS^1,Q]$. Here given the $\bs B$-structures on $X,Y$, to define the $\bs B$-structures on $X\t\cS^1,Y\t\cS^1$ we use the standard $\bs B$-structure on $\cS^1=\R/\Z$, which is invariant under the action of $\R/\Z\cong\U(1)$. So, for example, when $\bs B=\Spin$, we use the $\Spin$-structure on $\cS^1$ whose principal $\Spin(1)$-bundle is the trivial bundle $(\R/\Z)\t\Spin(1)\ra\R/\Z$. It is easy to check that $I_n^{\bs B,G}$ is a well-defined symmetric monoidal functor. Also $I^{\bs B}_n(G_1),I^{\bs B}_n(G_2)$ commute with the change-of-group functors $F_\ga$ in \eq{fm4eq2} and \eq{fm4eq9} in the obvious way.
\end{dfn}

\begin{prop}
\label{fm4prop2}
{\bf(a)} $\Bord^{\bs B}_n(\cL BG)$ is a Picard groupoid. Its invariants in Theorem\/ {\rm\ref{fmAthm2}(a)} are the $\bs B$-bordism groups
\ea
\pi_0\bigl(\Bord^{\bs B}_n(\cL BG) \bigr)&\cong\Om^{\bs B}_n(\cL BG),
\label{fm4eq11}\\
\pi_1\bigl(\Bord^{\bs B}_n(\cL BG) \bigr)&\cong\Om_{n+1}^{\bs B}(\cL BG),
\label{fm4eq12}
\ea
where\/ $\Om_m^{\bs B}(\cL BG)$ is the bordism group of the free loop space\/ $\cL BG=\Map_{C^0}(\cS^1,\ab BG)$ of the topological classifying space\/ $BG$ of\/ $G,$ and the linear quadratic map\/ $q:\Om^{\bs B}_n(\cL BG)\ra\Om_{n+1}^{\bs B}(\cL BG)$ mapping\/ $[X,P]\mapsto[X\t\cS^1,P\t\cS^1]$. 

Here in $X\t\cS^1,$ the $\cS^1$ has a $\U(1)$-equivariant\/ $\bs B$-structure with the usual orientation, so when $\bs B=\bs\Spin,$ it has the \begin{bfseries}non-bounding\end{bfseries} spin structure $\cS^1_{\rm nb}$. Note that for $\bs B=\bs\Spin,$ this means that\/ $q$ is multiplication by $\al_1=[\cS^1_{\rm nb}]$ in $\Om_1^{\bs\Spin}(*)$ in Table\/ {\rm\ref{fm2tab2}} under the natural action of\/ $\Om_*^{\bs\Spin}(*)$ on\/~$\Om_*^{\bs\Spin}(\cL BG)$.
\smallskip

\noindent{\bf(b)} The isomorphisms \eq{fm4eq11} and \eq{fm4eq12} are compatible with change of group functors, in particular with\/ $\{1\}\ra G.$
\smallskip

\noindent{\bf(c)} As in every Picard groupoid, a morphism\/ $\la:(X_0,P_0)\ab\ra(X_1,P_1)$ in\/ $\Bord^{\bs B}_n(\cL BG)$ determines a bijection
\begin{equation*}
\Om_{n+1}^{\bs B}(\cL BG)\longra\Hom_{\Bord^{\bs B}_n(\cL BG)}\bigl((X_0,P_0),(X_1,P_1)\bigr)
\end{equation*}
given by composition in the diagram of bijections
\begin{equation*}
\xymatrix@C=150pt@R=15pt{ 
*+[r]{\Om_{n+1}^{\bs B}(\cL BG)} \ar[d]^{\eq{fm4eq12}} \ar[r]_(0.22)\cong &  *+[l]{\Hom_{\Bord^{\bs B}_n(\cL BG)}\bigl((X_0,P_0),(X_1,P_1)\bigr)} \ar@{=}[d] \\
*+[r]{\Hom_{\Bord^{\bs B}_n(\cL BG)}(\boo,\boo)} \ar[r]^(0.34){\ot\la}  & *+[l]{\Hom_{\Bord^{\bs B}_n(\cL BG)}\bigl(\boo\ot(X_0,P_0),\boo\ot(X_1,P_1)\bigr).}}
\end{equation*}
\item[{\bf(d)}] There is a commutative diagram
\e
\begin{gathered}
\xymatrix@C=155pt@R=15pt{ 
*+[r]{\Om^{\bs B}_n(\cL BG)} \ar[d]^{\eq{fm4eq12}}_\cong \ar[r]^{\xi^{\bs B}_n(BG)}_{\eq{fm2eq6}} &  *+[l]{\ti\Om_{n+1}^{\bs B}(BG)} \ar[d]_{\eq{fm4eq4}} \\
*+[r]{\Aut_{\Bord_{n-1}^{\bs B}(\cL BG)}(\boo)} \ar[r]^{I_n^{\bs B,G}}_{\eq{fm4eq10}}  & *+[l]{\Aut_{\Bord^{\bs B}_n(BG)}(\boo).\!}
}
\end{gathered}
\label{fm4eq13}
\e
\end{prop}

\begin{proof}
For (a), let $(X,P)\in\Bord^{\bs B}_n(\cL BG)$. The principal $G$-bundle $P\ra X\t\cS^1$ is classified by a continuous map $f_P:X\t\cS^1\ra BG$, unique up to homotopy. Then $f_P$ is equivalent to a continuous map $\bar f_P:X\ra\Map_{C^0}(\cS^1,BG)=\cL BG$, unique up to homotopy, by $\bar f_P(x):e^{i\th}\mapsto f_P(x,e^{i\th})$. Hence $[X,\bar f_P]\in\Om^{\bs B}_n(\cL BG)$.

If $[W,Q]:(X_0,P_0)\to(X_1,P_1)$ is a morphism in $\Bord^{\bs B}_n(\cL BG),$ then $\pd W=-X_0\amalg X_1$ and the classifying map $f_Q: W\t\cS^1\ra BG$ of the principal $G$-bundle $Q\ra W$ can be chosen to extend $f_{P_0}\amalg f_{P_1}$, so that $\bar f_Q:W\ra\cL BG$ extends $\bar f_{P_0}\amalg\bar f_{P_1}$. Thus $[X_0,\bar f_{P_0}]=[X_1,\bar f_{P_1}]$ in $\Om^{\bs B}_n(\cL BG)$, so $[X,\bar f_P]$ depends only on the isomorphism class $[X,P]$ of $(X,P)$, and mapping $[X,P]\mapsto [X,\bar f_P]$ gives the map $\pi_0(\Bord^{\bs B}_n(\cL BG))\ra\Om^{\bs B}_n(\cL BG)$ in \eq{fm4eq11}. 

For the inverse map, if $[X,\bar f]\in \Om^{\bs B}_n(\cL BG)$ then $\bar f:X\ra\cL BG$ is continuous. Define $f:X\t\cS^1\ra BG$ by $f(x,e^{i\th})=\bar f(x)(e^{i\th})$. Then $P=f^*(EG)$ is initially a topological principal $G$-bundle $P\ra X\t\cS^1$, but it can be made into a smooth $G$-bundle uniquely up to isomorphism. Thus $(X,P)$ is an object in $\Bord^{\bs B}_n(\cL BG)$, and $[X,P]\in\pi_0(\Bord^{\bs B}_n(\cL BG))$, and the inverse map takes $[X,\bar f]\mapsto[X,P]$. This gives the bijection \eq{fm4eq11}. It is easy to see from the definitions that it is a group isomorphism.

To prove \eq{fm4eq12}, morphisms $0\ra 0$ are equivalence classes $[Y,Q]$ where $\pd Y=-\es\amalg\es=\es$, so $Y$ is without boundary. We then map the morphism $[Y,Q]$ to the bordism class $[Y,\bar f_Q]$ as for the proof of \eq{fm4eq11}, but increasing dimensions from $n$ to $n+1$. Part (b) is easy, (c) is immediate from properties of monoidal groupoids, and (d) is obvious.	
\end{proof}

\subsection{\texorpdfstring{Bordism categories $\Bord_X(BG)$}{Bordism categories Bordₓ(BG)}}
\label{fm43}

The next definition is a variation of Definition \ref{fm4def1}, in which we fix the $n$-manifold $X$, and take $Y=X\t[0,1]$ and $Z=X\t[0,1]^2$.

\begin{dfn}
\label{fm4def4}
Let $X$ be a compact $n$-manifold and $G$ a Lie group. Define $\Bord_X(BG)$ to be the category with objects $P$ for $P\ra X$ a principal $G$-bundle, and morphisms $[Q]:P_0\ra P_1$ be $\sim$-equivalence classes $[Q]$ of principal $G$-bundles $Q\ra X\t[0,1]$ with chosen isomorphisms $Q\vert_{X\t\{i\}}\cong P_i$ for $i=0,1$. If $Q,Q'$ are alternative choices for $Q$, we write $Q\sim Q'$ if there exists a principal $G$-bundle $R\ra X\t[0,1]^2$ with chosen isomorphisms
\begin{align*}
R\vert_{X\t\{0\}\t[0,1]}&\cong P_0\t[0,1],  & R\vert_{X\t\{1\}\t[0,1]}&\cong P_1\t[0,1], \\ R\vert_{X\t[0,1]\t\{0\}}&\cong Q, & R\vert_{X\t[0,1]\t\{1\}}&\cong Q',
\end{align*}
which are compatible over $X\t\{0,1\}^2$ with the given isomorphisms $Q\vert_{X\t\{i\}}\cong P_i\cong Q'\vert_{X\t\{i\}}$. To define composition of morphisms $[Q]:P_0\ra P_1$ and $[Q']:P_1\ra P_2$ we set $[Q']\ci[Q]=[Q'']$, where $Q''\ra X\t[0,1]$ is given by $Q''\vert_{X\t\{t\}}=Q\vert_{X\t\{2t\}}$ for $t\in[0,\ha]$, and $Q''\vert_{X\t\{t\}}=Q'\vert_{X\t\{2t-1\}}$ for $t\in[\ha,1]$, and when $t=\ha$ we identify $Q''\vert_{X\t\{\frac{1}{2}\}}=Q\vert_{X\t\{1\}}=Q'\vert_{X\t\{0\}}$ via the given isomorphisms $Q\vert_{X\t\{1\}}\cong P_1\cong Q'\vert_{X\t\{0\}}$. To define the smooth structure on $Q''$ near $X\t\{\ha\}$ we use collars as in Definition \ref{fm4def1}(d). 

It is then easy to show that composition is associative, so that $\Bord_X(BG)$ is a category, where identity morphisms are $\id_P=[P\t[0,1]]:P\ra P$. Every morphism in $\Bord_X(BG)$ is invertible, where the inverse of $[Q]:P_0\ra P_1$ is $[Q]^{-1}=[Q']:P_1\ra P_0$, with $Q'\vert_{X\t\{t\}}=Q\vert_{X\t\{1-t\}}$ for~$t\in[0,1]$.

Now suppose that $\bs B$ is a tangential structure, and $X$ has a $\bs B$-structure $\bs\ga_X$. Since the stable tangent bundles of $X\t[0,1]$ and $X\t[0,1]^2$ are the pullbacks of the stable tangent bundle of $X$, pullback of $\bs\ga_X$ along the projections $X\t[0,1]\ra X$, $X\t[0,1]^2\ra X$ induces $\bs B$-structures on $X\t[0,1]$ and $X\t[0,1]^2$. Define a functor
\e
\Pi_X^{\bs B}:\Bord_X(BG)\longra\Bord^{\bs B}_n(BG)
\label{fm4eq14}
\e
to map $P\mapsto(X,P)$ on objects and $[Q]\mapsto\bigl[X\t[0,1],Q\bigr]$ on morphisms, using the $\bs B$-structures on $X,X\t[0,1]$. This is well defined as writing $Y=X\t[0,1]$ and $Z=X\t[0,1]^2$, the definitions above of the equivalence $\sim$ on $Q$ and $(X\t[0,1],Q)$, and of compositions of morphisms, and so on, map to those in Definition \ref{fm4def1}.

If $P\ra X$ is a principal $G$-bundle, we write $\Bord_X(BG)_P\subset\Bord_X(BG)$ to be the full subcategory with one object $P$ in $\Bord_X(BG)$. Write $\Pi_{X,P}^{\bs B}$ for the restriction of $\Pi_X^{\bs B}$ to~$\Bord_X(BG)_P\subset\Bord_X(BG)$.

If $\ga:G_1\ra G_2$ is a morphism of Lie groups as for \eq{fm4eq2} there is a functor
\begin{equation*}
F_{X,\ga}:\Bord_X(BG_1)\longra\Bord_X(BG_2)
\end{equation*}
mapping $(X,P)\mapsto (X,(P\t G_2)/G_1)$ on objects and $[Q]\mapsto[(Q\t G_2)/G_1]$ on morphisms.
\end{dfn}

In a similar way to Propositions \ref{fm4prop1} and \ref{fm4prop2}, we can use homotopy theory to give a partial description of the categories $\Bord_X(BG)$ and functors~$\Pi_X^{\bs B}$.

\begin{prop}
\label{fm4prop3}
Suppose\/ $\bs B$ is a tangential structure,\/ $X$ a compact\/ $n$-manifold with\/ $\bs B$-structure\/ $\bs\ga_X,$ $G$ a Lie group, and\/ $P\ra X$ a principal\/ $G$-bundle. Then\/ $P$ is an object in\/ $\Bord_X(BG),$ and\/ $(X,P)$ an object in\/ $\Bord^{\bs B}_n(BG),$ and\/ $\Pi_X^{\bs B}:P\mapsto(X,P)$. We have a commutative diagram
\e
\begin{gathered}
\xymatrix@C=150pt@R=18pt{ 
*+[r]{\Aut_{\Bord_X(BG)}(P)} \ar[r]_(0.45){\Pi_X^{\bs B}} \ar[d]^{\chi_P^{\bs B}} & *+[l]{\Aut_{\Bord^{\bs B}_n(BG)}(X,P)} \\
*+[r]{\Om^{\bs B}_n(\cL BG)} \ar[r]^(0.45){\xi^{\bs B}_n(BG)} & *+[l]{\ti\Om_{n+1}^{\bs B}(BG),} \ar[u]^{\eq{fm4eq5}} }
\end{gathered}
\label{fm4eq15}
\e
where\/ $\xi^{\bs B}_n(BG)$ is in Definition\/ {\rm\ref{fm2def3},} the right hand column is\/ {\rm\eq{fm4eq5}} restricted to $\ti\Om_{n+1}^{\bs B}(BG)\subset\Om_{n+1}^{\bs B}(BG),$ and\/ $\chi_P^{\bs B}$ is defined as follows: let\/ $\phi_P:X\ra BG$ be a classifying map for\/ $P$. Then for\/ $[Q]:P\ra P$ in\/ $\Aut_{\Bord_X(BG)}(P),$ as\/ $Q\ra X\t[0,1]$ is a principal\/ $G$-bundle with chosen isomorphisms\/ $Q\vert_{X\t\{0\}}\cong P\cong Q\vert_{X\t\{1\}},$ we can choose a classifying map\/ $\phi_Q:X\t[0,1]\ra BG$ for\/ $Q$ such that\/ $\phi_Q\vert_{X\t\{0\}}=\phi_Q\vert_{X\t\{1\}}=\phi_P$. Writing\/ $\cS^1=\R/\Z=[0,1]/(0\sim 1)$ with projection\/ $\pi:[0,1]\ra\cS^1,$ define\/ $\bar\phi_Q:X\t\cS^1\ra BG$ by\/ $\bar\phi_Q\ci(\id_X\t\pi)=\phi_Q$. Let\/ $\ti\phi_Q:X\ra \cL BG=\Map_{C^0}(\cS^1,BG)$ be the induced map. Then define
\e
\chi_P^{\bs B}([Q])=[X,\ti\phi_Q].
\label{fm4eq16}
\e
\end{prop}

\begin{proof}
Let $[Q]\in\Aut_{\Bord_X(BG)}(P)$. Then $Q\ra X\t[0,1]$ is a principal $G$-bundle with $Q\vert_{X\t\{0\}}=Q\vert_{X\t\{1\}}=P$. Let $\phi_P:X\ra BG$, $\phi_Q:X\t[0,1]\ra BG$, $\bar\phi_Q:X\t\cS^1\ra BG$ and $\ti\phi_Q:X\ra \cL BG$ be as in the proposition, so that
\begin{equation*}
\chi_P^{\bs B}([Q])=[X,\ti\phi_Q]
\end{equation*}
by \eq{fm4eq16}. Note that $\phi_P$ is natural up to homotopy, and $\phi_Q$ natural up to homotopies relative to the homotopies of $\phi_P$, so $\bar\phi_Q$ and hence $\ti\phi_Q$ are natural up to homotopy. The bordism class $[X,\ti\phi_Q]$ is independent of these homotopies, and also depends only on $Q$ up to isomorphisms relative to the chosen isomorphisms $Q\vert_{X\t\{0\}}=Q\vert_{X\t\{1\}}=P$. Hence $[X,\ti\phi_Q]$ depends only on $[Q]$, and $\chi_P^{\bs B}([Q])$ is well defined.

The construction $\bar\phi_Q\mapsto\ti\phi_Q$ above is the inverse of the construction $\phi\mapsto\phi'$ in the definition of $\xi^{\bs B}_n(BG)$ in Definition \ref{fm2def3}. Thus we see from \eq{fm2eq6} that
\begin{equation*}
\xi^{\bs B}_n(BG)\ci\chi_P^{\bs B}([Q])=[X\t\cS^1,\bs\ga_X\t\bs\ga_{\cS^1},\bar\phi_Q].
\end{equation*}

Define $\bar Q\ra X\t\cS^1$ to be the principal $G$-bundle obtained by identifying $X\t\{0\}\cong X\t\{1\}$ and 
$Q\vert_{X\t\{0\}}\cong P\cong Q\vert_{X\t\{1\}}$, so that $Q\ra X\t[0,1]$ is the pullback of $\bar Q\ra X\t\cS^1$ by $\id_X\t\pi$, where $\pi:[0,1]\ra\cS^1=\R/\Z=[0,1]/(0\!\sim\! 1)$ is the projection. Then $\bar\phi_Q:X\t\cS^1\ra BG$ is a classifying map for $\bar Q$. Thus we see from the definition of the isomorphism \eq{fm4eq4} that \eq{fm4eq4} identifies
\begin{equation*}
\xymatrix@C=22pt{ 
\xi^{\bs B}_n(BG)\!\ci\!\chi_P^{\bs B}([Q])\!\in\!\ti\Om_{n+1}^{\bs B}(BG)
\ar@{<->}[r]^(0.44){\eq{fm4eq4}} &  [X\!\t\!\cS^1,\bs\ga_X\!\t\!\bs\ga_{\cS^1},\bar Q]\!\in\!\Aut_{\Bord^{\bs B}_n(BG)}(\boo). }
\end{equation*}
Hence by \eq{fm4eq5}--\eq{fm4eq6} we see that
\ea
&\eq{fm4eq5}\ci\xi^{\bs B}_n(BG)\ci\chi_P^{\bs B}([Q])
\label{fm4eq17}\\
&=\bigl[(X\t[0,1])\amalg (X\t\cS^1),(\bs\ga_X\t\bs\ga_{[0,1]})\amalg(\bs\ga_X\t\bs\ga_{\cS^1}),(P\t[0,1])\amalg\bar Q\bigr]
\nonumber
\ea
in $\Aut_{\Bord^{\bs B}_n(BG)}(X,P)$. By Definition \ref{fm4def4} we have
\e
\Pi_X^{\bs B}([Q])=\bigl[X\t[0,1],Q\bigr].
\label{fm4eq18}
\e

Thus, to prove \eq{fm4eq15} we must show that \eq{fm4eq17} and \eq{fm4eq18} agree. By Definition \ref{fm4def1} this is equivalent to
\ea
&(X\t[0,1],\bs\ga_X\t\bs\ga_{[0,1]},Q)
\label{fm4eq19}\\
&\sim (X\t\cS^1,\bs\ga_X\t\bs\ga_{\cS^1},\bar Q)\amalg (X\t[0,1],\bs\ga_X\t\bs\ga_{[0,1]},P\t[0,1]),	
\nonumber
\ea
where $\sim$ is the equivalence relation in Definition \ref{fm4def1}(c).

\begin{figure}[htb]
\centerline{$\splinetolerance{.8pt}
\begin{xy}
0;<.7cm,0cm>:
,(0,1.7)*+{Y_1}
,(0,-1.5)*+{Y_2}
,(-2,2)*+{\bu}
,(2,2)*+{\bu}
,(-2,-2)*+{\bu}
,(2,-2)*+{\bu}
,(-1.9,2);(1.9,2)**\crv{}
,(-1.8,2);(1.8,2)**\crv{}
,(-1.9,-2);(1.9,-2)**\crv{}
,(-1.8,-2);(1.8,-2)**\crv{}
,(2,-1.9);(2,1.9)**\crv{}
,(2,-1.8);(2,1.8)**\crv{}
,(-2,-1.9);(-2,1.9)**\crv{}
,(-2,-1.8);(-2,1.8)**\crv{}
,(0,1);(2,-2)**\crv{~**\dir{--}(1.2,3)}
,(0,1);(-2,-2)**\crv{~**\dir{--}(-1.2,3)}
,(0,0)*{\ellipse<.7cm>{-}}
\end{xy}$}
\caption{2-manifold with corners $Y$ in proof of Proposition \ref{fm4prop3}}
\label{fm4fig1}
\end{figure}
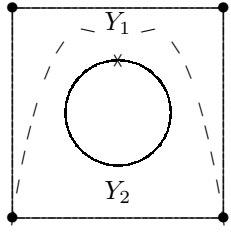

Define a 2-manifold with corners $Y\subset\R^2$ as in Figure \ref{fm4fig1}, by taking the square $[0,1]^2$ and deleting an open disc in its interior, with boundary $\cS^1$. Divide $Y$ into two regions $Y_1,Y_2$ by the dotted curves shown, where $Y_1$ is the region above and $Y_2$ the region below the dotted lines. In $Y_2$, separating the point where the dotted lines meet into two points $(0,1),(1,1)$, we may identify $Y_2\cong [0,1]^2$, such that $[0,1]\t\{0\}$ is the bottom side of the square $Y$, and $[0,1]\t\{1\}$ is the circle inside the square, with $(0,1)\sim(1,1)$. 

Then in Definition \ref{fm4def1}(c) we take $V=X\t Y$, and we define a principal $G$-bundle $R\ra V$ by taking $R$ to be $\Pi_X^*(P)$ on the product of $X\t Y_1$, and to be $Q\t[0,1]$ on $X\t Y_2\cong X\t([0,1]^2)=(X\t[0,1])\t[0,1]$. The appropriate $\bs B$-structure $\bs\ga_Y$ on $Y$ is that restricted from the standard $\bs B$-structure on $\R^2$. In particular, this implies that the boundary $\bs B$-structure on $\cS^1$ is that induced from the standard $\bs B$-structure on the closed unit disc $D^2\subset\R^2$ by identifying $S^1=\pd D^2$, as in Definition \ref{fm2def3}. This proves \eq{fm4eq19}, so \eq{fm4eq15} commutes.
\end{proof}

\section{Submanifold bordism categories}
\label{fm5}

\subsection{\texorpdfstring{Bordism categories $\Bord_{n,k}^{\bs B}(MH)$}{Bordism categories Bordₙₖᴮ(MH)}}
\label{fm51}

We define {\it submanifold bordism categories\/} $\Bord_{n,k}^{\bs B}(MH)$.

\begin{dfn}
\label{fm5def1}
Fix $0\le k\le n,$ a tangential structure $\bs B,$ a Lie group $H$, and a morphism of Lie groups $\rho:H\ra\O(k)$. We define the {\it submanifold bordism category\/} $\Bord_{n,k}^{\bs B}(MH)$ as follows.
\begin{itemize}
\setlength{\itemsep}{0pt}
\setlength{\parsep}{0pt}
\item[(a)] The {\it objects} of $\Bord_{n,k}^{\bs B}(MH)$ are quadruples $(X,\be,M,\ga),$ where $X$ is a compact $n$-manifold without boundary with $\bs B$-structure $\be$ and $M\subset X$ is a compact, embedded $(n-k)$-submanifold of $X$ without boundary, with an $H$-structure $\ga$ on the normal bundle $\nu_M\ra M$ of $M$ in $X$. Here $H$-structures are as in Definition~\ref{fm2def5}.
\item[(b)] {\it Morphisms} $[Y,\be',N,\ga']:(X_0,\be_0,M_0,\ga_0)\ra(X_1,\be_1,M_1,\ga_1)$ are equivalence classes of quadruples $(Y,\be',N,\ga'),$ see (c), where $Y$ is a compact $(n+1)$-manifold with boundary with $\bs B$-structure $\be'$, $N\subset Y$ is a compact, embedded neat $(n-k+1)$-submanifold with a normal $H$-structure $\ga'$, and a $\bs B$-structure-preserving diffeomorphism
\e
\pd Y\cong X_0\amalg X_1
\label{fm5eq1}
\e
identifying $\bs B$-structures $\be'\vert_{\pd Y}\cong -\be_0\amalg\be_1$ up to chosen isotopies (here $-\be_0$ is the opposite $\bs B$-structure to $\be_0$), which identifies $\pd N$ with $M_0\amalg M_1,$ and identifies the normal $H$-structures $\ga'\vert_{\pd N}\cong \ga_0\amalg\ga_1$ up to chosen isotopies.
\item[(c)] In the situation of (b), two choices $(Y_0,\be_0',N_0,\ga_0')$ and $(Y_1,\be_1',N_1,\ga_1')$ are {\it equivalent} if there exists a quadruple $(Z,\be'',L,\ga''),$ where $Z$ is a compact $(n+2)$-manifold with corners with a $\bs B$-structure $\be''$, and a diffeomorphism
\e
\pd Z\cong (X_0\t[0,1])\amalg (X_1\t[0,1]) \amalg (-Y_0\t\{0\})\amalg (Y_1\t\{1\})
\label{fm5eq2}
\e
that identifies $\bs B$-structures $\be''\vert_{\pd Z}\cong (-\be_0\t[0,1])\amalg(\be_1\t[0,1])\amalg (-\be'_0)\t\{0\}\amalg \be'_1\t\{1\}$ up to chosen isotopies, and along $\pd^2Z$ identifies $\pd(Y_a\t\{a\})$ with $(-X_0\t\{0\}\t\{a\})\amalg (X_1\t\{1\}\t\{a\})$ via \eq{fm5eq1} for $a=0,1$, compatibly with $\bs B$-structures. 

Moreover, $L\subset Z$ is a compact, embedded, neat $(n-k+2)$-submanifold with boundary, with a normal $H$-structure $\ga''$, and we require \eq{fm5eq2} to restrict to a diffeomorphism
\begin{equation*}
\pd L\cong (M_0\t[0,1])\amalg (M_1\t[0,1])\amalg N_0\amalg N_1,
\end{equation*}
identifying the normal $H$-structures up to chosen isotopies.
\item[(d)] If $[Y,\be',N,\ga']:(X_0,\be_0,M_0,\ga_0)\ra(X_1,\be_1,M_1,\ga_1)$ and $[\hat Y,\hat\be',\hat N,\hat\ga']:(X_1,\be_1,M_1,\ga_1)\ra(X_2,\be_2,M_2,\ga_2)$ are morphisms, the {\it composition} is 
\begin{equation*}
[\hat Y,\hat\be',\hat N,\hat\ga']\ci[Y,\be',N,\ga']=\bigl[\hat Y\amalg_{X_1}Y,\hat\be'\amalg_{\be_1}\be',\hat N\amalg_{M_1}N,\hat\ga'\amalg_{\ga_1}\ga'\bigr].
\end{equation*}
 That is, we glue $Y,\hat Y$ along their common boundary component $X_1$ to make a manifold $Y'\amalg_{X_1}Y$, and we glue the $\bs B$-structures $\be',\hat\be'$ along $X_1$ using the chosen isotopies $\be'\vert_{X_1}\simeq \be_1$, $\hat\be'\vert_{X_1}\simeq -\be_1$, and we glue the submanifolds $N,\hat N$ along their common boundary component $M_1\subset X_1$ to make a neat submanifold $\hat N\amalg_{M_1}N\subset \hat Y\amalg_{X_1}Y$, and we glue the normal $H$-structures $\ga',\hat\ga'$ along $M_1$ using $\ga'\vert_{X_1}\simeq \ga_1$, $\hat\ga'\vert_{X_1}\simeq \ga_1$. More precisely, we should choose collars $X_1\t(-\ep,0]\subset Y,$ $X_1\t[0,\ep)\subset\hat Y$ which restrict to collars $M_1\t(-\ep,0]\subset N,$ $M_1\t[0,\ep)\subset\hat Y$ such that the normal $H$-structures are also of product form, but the choices do not change the equivalence class $[\hat Y\amalg_{X_1}Y,\hat N\amalg_{M_1}N].$ Composition is associative.
\item[(e)] {\it Identities\/} are $\id_{(X,\be,M,\ga)}=\bigl[X\t[0,1],\be\t[0,1],M\t[0,1],\ga\t[0,1]\bigr]$.
\item[(f)] The {\it monoidal structure\/} on $\Bord_{n,k}^{\bs B}(MH)$ is defined as
\begin{equation*}
(X,\be,M,\ga)\ot (\hat X,\hat\be,\hat M,\hat\ga)=(X\amalg\hat X,\be\amalg\hat\be,M\amalg\hat M,\ga\amalg\hat\ga)
\end{equation*}
on objects, and similarly for morphisms.
\item[(g)] The {\it unit object} in $\Bord_{n,k}^{\bs B}(MH)$ is $\boo=(\es,\es,\es,\es).$
\item[(h)] For a pair of objects, the {\it symmetry isomorphism}
\begin{align*}
&\si_{(X_0,\be_0,M_0,\ga_0),(X_1,\be_1,M_1,\ga_1)}=[Y,\be',N,\ga']:\\
&(X_0,\be_0,M_0,\ga_0)\ot (X_1,\be_1,M_1,\ga_1)\longra (X_1,\be_1,M_1,\ga_1)\ot (X_0,\be_0,M_0,\ga_0)
\end{align*}
is the cylinders $Y=(X_0\amalg X_1)\t[0,1],$ $\be'=(\be_0\amalg\be_1)\t[0,1],$ $N=(M_0\amalg M_1)\t[0,1]$, $\ga'=(\ga_0\amalg\ga_1)\t[0,1],$ where the boundary diffeomorphisms are the obvious identifications $(X_0\amalg X_1)\t\{0\}\cong X_0\amalg X_1$ and $(X_0\amalg X_1)\t\{1\}\cong X_1\amalg X_0$, swapping round factors at $1\in\pd[0,1]$, and similarly for boundary identifications of $\be',N,\ga'$.
\end{itemize}

To simplify notation, from now on we usually omit $\bs B$-structures $\be,\be',\be''$ and $H$-structures $\ga,\ga',\ga''$, leaving them implicit, so we write objects as pairs $(X,M)$, where $X$ is a compact $n$-manifold with $\bs B$-structure and $M\subset X$ is a compact $(n-k)$-submanifold with normal $H$-structure, and morphisms are $[Y,N]:(X_0,M_0)\ra(X_1,M_1)$, and so on.

With these definitions, it is easy to check that $\Bord_{n,k}^{\bs B}(MH)$ is indeed a {\it symmetric monoidal category}, in the sense of Appendix~\ref{fmA}.

Define a symmetric monoidal category $\Bord^{\bs B}_n(*)$ as above but omitting all submanifolds $M,N,\ldots$ and normal $H$-structures $\ga,\ga',\ldots,$ so that objects of $\Bord^{\bs B}_n(*)$ are compact $n$-manifolds $X$ with $\bs B$-structure $\be$, and so on. There is a natural monoidal inclusion functor
\begin{equation*}
I_{n,k}^{\bs B,H}:\Bord^{\bs B}_n(*)\longra\Bord_{n,k}^{\bs B}(MH) 
\end{equation*}
acting by $X\mapsto(X,\es)$ on objects and $[Y]\mapsto[Y,\es]$ on morphisms.

A composition of Lie groups $H_1\,{\buildrel\io\over\longra}\, H_2\,{\buildrel\rho_2\over\longra}\,\O(k)$ induces a functor
\e
F_\io:\Bord_{n,k}^{\bs B}(MH_1)\longra\Bord_{n,k}^{\bs B}(MH_2),
\label{fm5eq3}
\e
by converting normal $H_1$-structures $\ga,\ga',\ldots$ on $M\subset X,N\subset Y,\ldots$ to normal $H_2$-structures.
\end{dfn}

The next proposition uses the theory of Picard groupoids in Appendix~\ref{fmA}.

\begin{prop}
\label{fm5prop1}
{\bf(a)} $\Bord_{n,k}^{\bs B}(MH)$ is a Picard groupoid. Its invariants in Theorem\/ {\rm\ref{fmAthm2}(a)} are the $\bs B$-bordism groups
\ea
\pi_0\bigl(\Bord_{n,k}^{\bs B}(MH) \bigr)&\cong\Om^{\bs B}_n(MH),
\label{fm5eq4}\\
\pi_1\bigl(\Bord_{n,k}^{\bs B}(MH) \bigr)&\cong\Om_{n+1}^{\bs B}(MH),
\label{fm5eq5}
\ea
and\/ $q:\Om^{\bs B}_n(MH)\ra\Om_{n+1}^{\bs B}(MH)$ mapping\/ $[X,M]\mapsto[X\t\cS^1,M\t\cS^1]$. 

Here in $X\t\cS^1,$ the $\cS^1$ has a $\U(1)$-equivariant\/ $\bs B$-structure with the usual orientation, so when $\bs B=\bs\Spin,$ it has the \begin{bfseries}non-bounding\end{bfseries} spin structure $\cS^1_{\rm nb}$. Note that for $\bs B=\bs\Spin,$ this means that\/ $q$ is multiplication by $\al_1=[\cS^1_{\rm nb}]$ in $\Om_1^{\bs\Spin}(*)$ in Table\/ {\rm\ref{fm2tab2}} under the natural action of\/ $\Om_*^{\bs\Spin}(*)$ on\/~$\Om_*^{\bs\Spin}(MH)$.
\smallskip

\noindent{\bf(b)} As in every Picard groupoid, a morphism\/ $\la:(X_0,M_0)\ab\ra(X_1,M_1)$ in\/ $\Bord_{n,k}^{\bs B}(MH)$ determines a bijection
\e
\Om_{n+1}^{\bs B}(MH)\longra\Hom_{\Bord_{n,k}^{\bs B}(MH)}\bigl((X_0,M_0),(X_1,M_1)\bigr)
\label{fm5eq6}
\e
given by composition in the diagram of bijections
\e
\begin{gathered}
\xymatrix@C=144pt@R=15pt{ 
*+[r]{\Om_{n+1}^{\bs B}(MH)} \ar[d]^{\eq{fm5eq5}} \ar[r]_(0.22)\cong &  *+[l]{\Hom_{\Bord_{n,k}^{\bs B}(MH)}\bigl((X_0,M_0),(X_1,M_1)\bigr)} \ar@{=}[d] \\
*+[r]{\Hom_{\Bord_{n,k}^{\bs B}(MH)}(\boo,\boo)} \ar[r]^(0.33){\raisebox{5pt}{$\scriptstyle\ot\la$}}  & *+[l]{\Hom_{\Bord_{n,k}^{\bs B}(MH)}\bigl(\boo\!\ot\!(X_0,M_0),\boo\!\ot\!(X_1,M_1)\bigr).}
}\!\!\!\!
\end{gathered}
\label{fm5eq7}
\e

\noindent{\bf(c)} For the category $\Bord^{\bs B}_n(*),$ the analogues of\/ {\rm\eq{fm5eq4}--\eq{fm5eq5}} are
\e
\pi_0(\Bord^{\bs B}_n(*))\cong \Om^{\bs B}_n(*),\quad \pi_1(\Bord^{\bs B}_n(*))\cong \Om_{n+1}^{\bs B}(*).
\label{fm5eq8}
\e
Under the identifications {\rm\eq{fm5eq8}, \eq{fm5eq4}, \eq{fm5eq5},} the functor $I_{n,k}^{\bs B,H}$ induces the morphisms $\Om_m^{\bs B}(*)\!\ra\!\Om_m^{\bs B}(MH)$ induced by $*\!\ra\! MH,$ $*\!\mapsto\!\iy$ for $m=n,n+1$.
\end{prop}

\begin{proof}
For (a), let $(X,M)\in\Bord_{n,k}^{\bs B}(MH).$ Write $-X$ for $X$ with the opposite $\bs B$-structure, as in Definition \ref{fm2def1}. There is an isomorphism $\bigl[X\t[0,1],M\t[0,1]\bigr]:(-X,M)\ot(X,M)\ra\boo$ and hence $(-X,M)$ is an inverse for $(X,M)$ under the monoidal structure. If $[Y,N]:\ab(X_0,M_0)\ra(X_1,M_1)$ is a morphism, we can prove that it has the inverse morphism
\begin{equation*}
[Y,N]^{-1}=\bigl[-Y\amalg (Y\amalg_{X_0\amalg X_1}-Y),N\amalg(N\amalg_{M_0\amalg M_1}N)\bigr],
\end{equation*}
where we note that $\pd(-Y)=-(-X_0\amalg X_1)=-X_1\amalg X_0.$ Thus, all morphisms in $\Bord_{n,k}^{\bs B}(MH)$ are isomorphisms. Equations \eq{fm5eq4}--\eq{fm5eq8} follow from Theorem \ref{fm2thm2} and the definition of $\Om^{\bs B}_n(-)$. 

We can show from the definitions that $q:\Om^{\bs B}_n(MH)\ra\Om_{n+1}^{\bs B}(MH)$ maps $M\subset X$ to the mapping torus of the $\Z_2$-action on $M\amalg M\subset X\amalg X$ that exchanges the two copies of $X$, so that $q\bigl([X,M]\bigr)=\bigl[((X\amalg X)\t[0,1])/\mathbin{\sim},\ab((M\amalg M)\t[0,1])/\mathbin{\sim}\bigr]$. But $((X\amalg X)\t[0,1])/\mathbin{\sim}\cong X\t\cS^1$ and $((M\amalg M)\t[0,1])/\mathbin{\sim}\cong M\t\cS^1$. Part (b) is immediate from the theory of Picard groupoids, and (c) is straightforward.
\end{proof}

\begin{ex}
\label{fm5ex1}
From Proposition \ref{fm5prop1}(a) and Tables \ref{fm2tab1} and \ref{fm3tab1} we see that there are equivalences of Picard groupoids
\e
\begin{aligned}
\Bord_{7,4}^{\bs\Spin}(M\SO(4))&\cong (0\op 0)\qs (\Z^2\op\Z^3), \\
\Bord_{8,4}^{\bs\Spin}(M\SO(4))&\cong (\Z^2\op \Z^3)\qs(\Z_2^2\op\Z_2^3),\\
\Bord_{8,4}^{\bs\Spin}(M\Spin(4))&\cong (\Z^2\op\Z^3)\qs(\Z_2^2\op\Z_2^2).
\end{aligned}
\label{fm5eq9}
\e
Here the right hand sides are Picard groupoids of the form $\pi_0\qs\pi_1$ as in Theorem \ref{fmAthm2} for abelian groups $\pi_0,\pi_1$. The decomposition of the $\pi_i$ as $A\op B$ in \eq{fm5eq9} corresponds to the splitting $\Om_n^{\bs\Spin}(MH)=\Om_n^{\bs\Spin}(*)\op\ti\Om_n^{\bs\Spin}(MH),$ where $\Om_n^{\bs\Spin}(*)$ is given in Table\ \ref{fm2tab1} and $\ti\Om_n^{\bs\Spin}(MH)$ in Table \ref{fm3tab1}. Note that $\pi_0\qs\pi_1$ also depends on a linear quadratic map $q:\pi_0\ra\pi_1$, which can be computed from Theorem~\ref{fm3thm1}.
\end{ex}

\subsection{\texorpdfstring{Loop bordism categories $\Bord_{n-1,k}^{\bs B}(\cL MH)$}{Loop bordism categories Bordₙ₋₁ₖᴮ(ℒMH)}}
\label{fm52}

\begin{dfn}
\label{fm5def2}
Fix $k,n$ with $1\le k\le n$, a tangential structure $\bs B$, and a Lie group morphism $\rho:H\ra\O(k)$. We define a Picard groupoid, the {\it submanifold loop bordism category\/} $\Bord_{n-1,k}^{\bs B}(\cL MH)$ as for $\Bord_{n,k}^{\bs B}(MH)$ as in Definition \ref{fm5def1}, but modified as follows: we reduce the dimensions of $X,Y,Z$ by 1, so $\dim X=n-1$, $\dim Y=n$, $\dim Z=n+1$, and we take $M,N,L$ to be submanifolds of $X\t\cS^1,Y\t\cS^1,Z\t\cS^1$ instead of $X,Y,Z$. So, for example, objects of $\Bord_{n-1,k}^{\bs B}(\cL MH)$ are pairs $(X,M)$ with $X$ a compact $(n-1)$-manifold with $\bs B$-structure $\be$, and $M\subset X\t\cS^1$ a compact embedded $(n-k)$-submanifold with an $H$-structure $\ga$ on the normal bundle $\nu_M\ra M$ of $M$ in~$X\t\cS^1$. 
\end{dfn}

We relate the categories of Definitions \ref{fm5def1} and \ref{fm5def2}.

\begin{dfn}
\label{fm5def3}
Let $k,n,\bs B,H$ be as above. Define a functor
\e
I_{n,k}^{\bs B,H} :\Bord_{n-1,k}^{\bs B}(\cL MH) \longra \Bord_{n,k}^{\bs B}(MH) 
\label{fm5eq10}
\e
to act on objects by $I_{n,k}^{\bs B,H} :(X,M)\mapsto (X\t\cS^1,M)$, and on morphisms by $I_{n,k}^{\bs B,H}:[Y,N]\mapsto [Y\t\cS^1,N]$. Here given the $\bs B$-structures on $X,Y$, to define the $\bs B$-structures on $X\t\cS^1,Y\t\cS^1$ we use the standard $\bs B$-structure on $\cS^1=\R/\Z$, which is invariant under the action of $\R/\Z\cong\U(1)$. It is easy to check that $I_{n,k}^{\bs B,H}$ is a well-defined symmetric monoidal functor.
\end{dfn}

Here is the analogue of Proposition \ref{fm5prop1}.

\begin{prop}
\label{fm5prop2}
{\bf(a)} $\Bord_{n-1,k}^{\bs B}(\cL MH)$ is a Picard groupoid. Its invariants in Theorem\/ {\rm\ref{fmAthm2}(a)} are the $\bs B$-bordism groups
\ea
\pi_0\bigl(\Bord_{n-1,k}^{\bs B}(\cL MH)\bigr)&\cong\Om_{n-1}^{\bs B}(\cL MH),
\label{fm5eq11}\\
\pi_1\bigr(\Bord_{n-1,k}^{\bs B}(\cL MH)\bigr)&\cong\Om^{\bs B}_n(\cL MH),
\label{fm5eq12}
\ea
and\/ $q:\Om_{n-1}^{\bs B}(\cL MH)\ra\Om^{\bs B}_n(\cL MH)$ which maps\/ $M\subset X\t\cS^1$ to $M\t\cS^1\subset X\t\cS^1\t\cS^1$.
\smallskip

\noindent{\bf(b)} There is a commutative diagram
\e
\begin{gathered}
\xymatrix@C=155pt@R=15pt{ 
*+[r]{\Om^{\bs B}_n(\cL MH)} \ar[d]^{\eq{fm5eq12}}_\cong \ar[r]^{\xi^{\bs B}_n(MH)}_{\eq{fm2eq6}} &  *+[l]{\ti\Om_{n+1}^{\bs B}(MH)} \ar[d]_{\eq{fm5eq5}} \\
*+[r]{\Aut_{\Bord_{n-1,k}^{\bs B}(\cL MH)}(\boo)} \ar[r]^(0.52){I_{n,k}^{\bs B,H}}_(0.52){\eq{fm5eq10}}  & *+[l]{\Aut_{\Bord_{n,k}^{\bs B}(MH)}(\boo).\!} }
\end{gathered}
\label{fm5eq13}
\e
\end{prop}

\begin{proof}
The proof of (a) is very similar to the proof of Propositions \ref{fm5prop1}(a), with the difference that objects $(X,M)$ in $\Bord_{n-1,k}^{\bs B}(\cL MH)$ correspond to submanifolds $M\subset X\t\cS^1$, and hence to elements of $[X\t\cS^1,MH]$, modifying equation \eq{fm2eq26}. But a continuous map $X\t\cS^1\ra MH$ is equivalent to a continuous map $X\ra\cL MH$, which is why $\Om_*^{\bs B}(\cL MH)$ appear in \eq{fm5eq11}--\eq{fm5eq12}. Part (b) is obvious from the definitions.
\end{proof}

\subsection{\texorpdfstring{L-equivalence categories $\Bord_X^k(MH)$}{L-equivalence categories Bordₓᵏ(MCₙ₋ₖ)}}
\label{fm53}

The next definition is a variation of Definition \ref{fm5def1}, in which we fix the $n$-manifold $X$, and take $Y=X\t[0,1]$ and~$Z=X\t[0,1]^2$.

\begin{dfn}
\label{fm5def4}
Let $0\le k\le n$, and $X$ be a compact $n$-manifold without boundary, and $\rho:H\ra\O(k)$ be a Lie group morphism. We define the {\it L-equivalence category\/} $\Bord_X^k(MH)$ as follows:
\begin{itemize}
\setlength{\itemsep}{0pt}
\setlength{\parsep}{0pt}
\item[(a)] The {\it objects} $(M,\ga)$ are compact, embedded $(n-k)$-submanifolds $M\subset X$ with an $H$-structure $\ga$ on the normal bundle $\nu_M\ra M$ of $M$ in $X$.
\item[(b)] A {\it morphism} $[N,\ga']: (M_0,\ga_0)\ra(M_1,\ga_1)$ is represented (modulo the equivalence relation described in (c)) by a compact, neat, embedded $(n-k+1)$-submanifold $N\subset X\t[0,1]$ with an $H$-structure $\ga'$ on the normal bundle $\nu_N\ra N$, such that
\begin{equation*}
\pd N=(M_0\t\{0\})\amalg (M_1\t\{1\}),\qquad \ga'\vert_{\pd N}\cong\ga_0\amalg\ga_1.
\end{equation*}
\item[(c)] In (b), two representatives $(N_0,\ga'_0)$ and $(N_1,\ga'_1)$ are {\it equivalent} if there exists a compact, neat, embedded $(n-k+2)$-submanifold $L\subset X\t[0,1]^2$ with an $H$-structure $\ga''$ on the normal bundle $\nu_L\ra L$ such that 
\begin{align*}
\pd L&=(M_0\!\t\!\{0\}\!\t\![0,1])\!\amalg\! (M_1\!\t\!\{1\}\!\t\![0,1])\!\amalg\! (N_0\!\t\!\{0\})\!\amalg\! (N_1\!\t\!\{1\}),\\
\ga''\vert_{\pd L}&\cong (\ga_0\!\t\!\{0\}\!\t\![0,1])\!\amalg\!(\ga_1\!\t\!\{1\}\!\t\![0,1])\!\amalg\! (\ga'_0\!\t\!\{0\})\!\amalg\! (\ga'_1\!\t\!\{1\}).
\end{align*}
\item[(d)] If $[N,\ga']:(M_0,\ga_0)\ra(M_1,\ga_1)$ and $[\hat N,\hat\ga']: (M_1,\ga_1)\ra (M_2,\ga_2)$ are morphisms, the {\it composition} is the morphism $(M_0,\ga_0)\ra(M_2,\ga_2)$ obtained by mapping $N\hookra X\t[0,\ha]$ by $(x,t)\mapsto (x,\ha t)$ and mapping $\hat N\hookra X\t[\ha,1]$ by $(x,t)\mapsto (x,\ha(t+1))$, and gluing the submanifolds $N\subset X\t[0,\ha]$ and $\hat N\subset X\t[\ha,1]$ along their common boundary component $M_1\subset X\t\{\ha\}$ to make a submanifold $N\amalg_{M_1}\hat N$ in $X\t[0,1]$. To make this smooth at $M_1$ we should choose $(N,\ga')$, $(\hat N,\hat\ga')$ in their equivalence classes so that they are of product form $M_1\t(1-\ep,1]$, $M_1\t[0,\ep)$ near the $M_1$ boundary. We glue the $H$-structures in a similar way, and set 
\begin{equation*}
[\hat N,\hat\ga']\ci [N,\ga]=[N\amalg_{M_1}\hat N,\ga'\amalg_{\ga_1}\hat\ga'].
\end{equation*}
Composition is associative.
\item[(e)] The {\it identity morphism} at $(M,\ga)$ is $\id_{(M,\ga)}=\bigl[M\t[0,1],\ga\t[0,1]\bigr].$
\end{itemize}

As in Definition \ref{fm5def1}, to simplify notation, from now on we usually omit the $H$-structures $\ga,\ga',\ga''$, leaving them implicit, so we write objects as $M$, and morphisms as $[N]:M_0\ra M_1$, and so on.

Observe that $\Bord_X^k(MH)$ is a specialization of $\Bord_{n,k}^{\bs B}(MH)$ in Definition \ref{fm5def1}, in which the varying ambient manifolds $X,Y,Z$ in Definition \ref{fm5def1} are replaced by $X,X\t[0,1],X\t[0,1]^2$ respectively, for $X$ fixed. Let $\bs B$ be a tangential structure, and suppose $X$ has a $\bs B$-structure $\be$. This induces $\bs B$-structures $\be',\be''$ on $X\!\t\![0,1],X\!\t\![0,1]^2$. Define a functor
\e
\Pi_X^{\bs B}:\Bord_X^k(MH) \longra\Bord_{n,k}^{\bs B}(MH) 
\label{fm5eq14}
\e
to map $M\mapsto(X,M)$ on objects and $[N]\mapsto[X\t[0,1],N]$ on morphisms, using the $\bs B$-structures $\be,\be'$ on $X,X\t[0,1]$. This is well defined as writing $Y=X\t[0,1]$ and $Z=X\t[0,1]^2$, the definitions above of the equivalence on $N$ and $(X\t[0,1],N)$, and of compositions of morphisms, and so on, map to those in Definition \ref{fm5def1}.
\end{dfn}

The next proposition justifies the name `L-equivalence category'.

\begin{prop}
\label{fm5prop3}
In Definition\/ {\rm\ref{fm5def4},} $\Bord_X^k(MH)$ is a groupoid (that is, all morphisms are isomorphisms), and there is a natural bijection
\e
\pi_0\bigl(\Bord_X^k(MH) \bigr)\cong \La_k^H(X),
\label{fm5eq15}
\e
where $\La_k^H(X)$ is the set of L-equivalence classes in Definition\/ {\rm\ref{fm2def6},} which is described using homotopy theory in Theorem\/~{\rm\ref{fm2thm2}}.
\end{prop}

\begin{proof}
If $[N]:M_0\ra M_1$ is a morphism in $\Bord_X^k(MH)$, so that $N\subset X\t[0,1]$ is a compact embedded submanifold with normal $H$-structure $\ga'$, the inverse morphism is $[N]^{-1}=[\hat N]:M_1\ra M_0$, where 
\begin{equation*}
\hat N=\bigl\{(x,1-t):(x,t)\in N\bigr\}\subset X\t[0,1].
\end{equation*}
The obvious identification $\hat N\cong N$ identifies the normal bundles $\nu_{N'}\cong\nu_N$, and we give $\hat N$ the normal $H$-structure $\hat\ga'$ corresponding to $\ga'$ under $\nu_{N'}\cong\nu_N$. Hence $\Bord_X^k(MH)$ is a groupoid. The isomorphism \eq{fm5eq15} follows by comparing Definitions \ref{fm2def6} and~\ref{fm5def4}.
\end{proof}

In a similar way to Propositions \ref{fm5prop1} and \ref{fm5prop2}, we can use homotopy theory to give a partial description of the categories $\Bord_X^k(MH)$ and functors $\Pi_X^{\bs B}$. The proof of the next proposition is essentially the same as that of Proposition~\ref{fm4prop3}.

\begin{prop}
\label{fm5prop4}
Fix\/ $0\le k\le n,$ a tangential structure $\bs B,$ a Lie group $H$, and a morphism of Lie groups $\rho:H\ra\O(k)$. Suppose $X$ is a compact\/ $n$-manifold with\/ $\bs B$-structure, and\/ $M\subset X$ is a compact, embedded\/ $(n-k)$-submanifold with normal\/ $H$-structure. Then\/ $M$ is an object in\/ $\Bord_X^k(MH),$ and\/ $(X,M)$ an object in\/ $\Bord_{n,k}^{\bs B}(MH),$ and\/ $\Pi_X^{\bs B}:M\mapsto(X,M)$. We have a commutative diagram
\e
\begin{gathered}
\xymatrix@C=150pt@R=18pt{ 
*+[r]{\Aut_{\Bord_X^k(MH)}(M)} \ar[r]_(0.45){\Pi_X^{\bs B}} \ar[d]^{\chi_M^{\bs B}} & *+[l]{\Aut_{\Bord_{n,k}^{\bs B}(MH)}(X,M)} \\
*+[r]{\Om^{\bs B}_n(\cL MH)} \ar[r]^(0.45){\xi^{\bs B}_n(MH)} & *+[l]{\ti\Om_{n+1}^{\bs B}(MH),} \ar[u]^{\eq{fm5eq6}} }
\end{gathered}
\label{fm5eq16}
\e
where\/ $\xi^{\bs B}_n(MH)$ is in Definition\/ {\rm\ref{fm2def3},} the right hand column is\/ {\rm\eq{fm5eq6}} restricted to $\ti\Om_{n+1}^{\bs B}(MH)\subset\Om_{n+1}^{\bs B}(MH),$ and\/ $\chi_M^{\bs B}$ is defined as follows: let\/ $\phi_M:X\ra MH$ be a classifying map for\/ $M$. Then for\/ $[N]:M\ra M$ in\/ $\Aut_{\Bord_X^k(MH)}(M),$ as\/ $N\subset X\t[0,1]$ is a submanifold with normal $H$-structure with\/ $\pd N=M\t\{0,1\},$ we can choose a classifying map\/ $\phi_N:X\t[0,1]\ra MH$ for\/ $N$ such that\/ $\phi_N\vert_{X\t\{0\}}=\phi_N\vert_{X\t\{1\}}=\phi_M$. Writing\/ $\cS^1=\R/\Z=[0,1]/(0\sim 1)$ with projection\/ $\pi:[0,1]\ra\cS^1,$ define\/ $\bar\phi_N:X\t\cS^1\ra MH$ by\/ $\bar\phi_N\ci(\id_X\t\pi)=\phi_N$. Let\/ $\ti\phi_N:X\ra \cL MH=\Map_{C^0}(\cS^1,MH)$ be the induced map. Then define
\e
\chi_M^{\bs B}([N])=[X,\ti\phi_N].
\label{fm5eq17}
\e
\end{prop}

\section{Cohomology bordism categories}
\label{fm6}

\subsection{\texorpdfstring{Bordism categories $\Bord^{\bs B}_n(K(R,k))$}{Bordism categories Bordₙᴮ(K(R,k)))}}
\label{fm61}

\begin{dfn}
\label{fm6def1}
Let $R$ be a commutative ring. Write $C^\bu(X,R)$ for the cochain complex of some cohomology theory which is defined for compact smooth $n$-manifolds with corners $X$ and computes the cohomology $H^*(X,R)$. The exact choice of cohomology theory does not matter much, provided there are functorial pullbacks $f^*:C^\bu(Y,R)\ra C^\bu(X,R)$ along smooth maps $f:X\ra Y$. To be definite, we could take $C^\bu(X,R)$ to be topological singular cochains.

Fix $0\le k\le n$ and a tangential structure $\bs B$. We define the {\it cohomology bordism category\/} $\Bord^{\bs B}_n(K(R,k))$ as follows.
\begin{itemize}
\setlength{\itemsep}{0pt}
\setlength{\parsep}{0pt}
\item[(a)] The {\it objects} of $\Bord^{\bs B}_n(K(R,k))$ are triples $(X,\be,C),$ where $X$ is a compact $n$-manifold without boundary with $\bs B$-structure $\be$ and $C\in C^k(X,R)$ is a $k$-cocycle on $X$. (Here we call $C$ a {\it cocycle\/} when $\d C=0$ in $C^{k+1}(X,R)$.)
\item[(b)] {\it Morphisms} $[Y,\be',D]:(X_0,\be_0,C_0)\ra(X_1,\be_1,C_1)$ are equivalence classes of triples $(Y,\be',D),$ see (c), where $Y$ is a compact $(n+1)$-manifold with boundary with $\bs B$-structure $\be'$, and $D\in C^k(Y,R)$ is a $k$-cocycle on $Y$, with a $\bs B$-structure-preserving diffeomorphism
\e
\pd Y\cong X_0\amalg X_1
\label{fm6eq1}
\e
identifying $\bs B$-structures $\be'\vert_{\pd Y}\cong -\be_0\amalg\be_1$ up to chosen isotopies, and which identifies $D\vert_{\pd Y}$ with $C_0\amalg C_1$ in $C^k(X_0\amalg X_1,R)$.
\item[(c)] In the situation of (b), two choices $(Y_0,\be_0',D_0)$ and $(Y_1,\be_1',D_1)$ are {\it equivalent} if there exists a triple $(Z,\be'',E),$ where $Z$ is a compact $(n+2)$-manifold with corners with a $\bs B$-structure $\be''$, and a diffeomorphism
\e
\pd Z\cong (X_0\t[0,1])\amalg (X_1\t[0,1]) \amalg (Y_0\t\{0\})\amalg (Y_1\t\{1\})
\label{fm6eq2}
\e
that identifies $\bs B$-structures $\be''\vert_{\pd Z}\cong (-\be_0\t[0,1])\amalg(\be_1\t[0,1])\amalg (-\be'_0)\t\{0\}\amalg \be'_1\t\{1\}$ up to chosen isotopies, and along $\pd^2Z$ identifies $\pd(Y_a\t\{a\})$ with $(-X_0\t\{0\}\t\{a\})\amalg (X_1\t\{1\}\t\{a\})$ via \eq{fm6eq1} for $a=0,1$, compatibly with $\bs B$-structures. 

Moreover, $E\in C^k(Z,R)$ is a $k$-cocycle such that under \eq{fm6eq2} we have
\begin{align*}
E\vert_{X_0\t[0,1]}&=\Pi_{X_0}^*(C_0),& E\vert_{X_1\t[0,1]}&=\Pi_{X_1}^*(C_1),\\
E\vert_{Y_0\t\{0\}}&=D_0,& E\vert_{Y_1\t\{1\}}&=D_1.
\end{align*}
\item[(d)] If $[Y,\be',D]:(X_0,\be_0,C_0)\ra(X_1,\be_1,C_1)$ and $[\hat Y,\hat\be',\hat D]:(X_1,\be_1,C_1)\ra(X_2,\be_2,C_2)$ are morphisms, the {\it composition} is 
\begin{equation*}
[\hat Y,\hat\be',\hat D]\ci[Y,\be',D]=\bigl[\hat Y\amalg_{X_1}Y,\hat\be'\amalg_{\be_1}\be',\hat D\amalg_{C_1}D\bigr].
\end{equation*}
That is, we glue $Y,\hat Y$ along their common boundary component $X_1$ to make a manifold $Y'\amalg_{X_1}Y$, and we glue the $\bs B$-structures $\be',\hat\be'$ along $X_1$ using the chosen isotopies $\be'\vert_{X_1}\simeq \be_1$, $\hat\be'\vert_{X_1}\simeq -\be_1$, and we glue the $k$-cocycles $C,\hat D$ along their common restrictions $C_1$ on $X_1$ to make a $k$-cocycle $\hat D\amalg_{C_1}D$ on $\hat Y\amalg_{X_1}Y$. Doing these gluings requires choices, but these do not change the equivalence class. Composition is associative.
\item[(e)] {\it Identities\/} are $\id_{(X,\be,C)}=\bigl[X\t[0,1],\be\t[0,1],\Pi_X^*(C)\bigr]$.
\item[(f)] The {\it monoidal structure\/} on $\Bord^{\bs B}_n(K(R,k))$ is defined as
\begin{equation*}
(X,\be,C)\ot (\hat X,\hat\be,\hat C)=(X\amalg\hat X,\be\amalg\hat\be,C\amalg\hat C)
\end{equation*}
on objects, and similarly for morphisms.
\item[(g)] The {\it unit object} in $\Bord_{n,k}^{\bs B}(K(R,k))$ is $\boo=(\es,\es,0).$
\item[(h)] For a pair of objects, the {\it symmetry isomorphism}
\begin{align*}
&\si_{(X_0,\be_0,C_0),(X_1,\be_1,C_1)}=[Y,\be',D]:\\
&(X_0,\be_0,C_0)\ot (X_1,\be_1,C_1)\longra (X_1,\be_1,C_1)\ot (X_0,\be_0,C_0)
\end{align*}
is the cylinders $Y=(X_0\amalg X_1)\t[0,1],$ $\be'=(\be_0\amalg\be_1)\t[0,1],$ $D=\Pi_{X_0}^*(C_0)\amalg \Pi_{X_1}^*(C_1)$, where the boundary diffeomorphisms are the obvious identifications $(X_0\amalg X_1)\t\{0\}\cong X_0\amalg X_1$ and $(X_0\amalg X_1)\t\{1\}\cong X_1\amalg X_0$, swapping round factors at $1\in\pd[0,1]$.
\end{itemize}

To simplify notation, from now on we usually omit $\bs B$-structures $\be,\be',\be''$, leaving them implicit, so we write objects as pairs $(X,C)$, where $X$ is a compact $n$-manifold with $\bs B$-structure and $C\in C^k(X,R)$, and morphisms are $[Y,D]:(X_0,C_0)\ra(X_1,C_1)$, and so on. With these definitions, it is easy to check that $\Bord^{\bs B}_n(K(R,k))$ is indeed a {\it symmetric monoidal category}, in the sense of Appendix~\ref{fmA}.

The isomorphism class in $\Bord^{\bs B}_n(K(R,k))$ of an object $(X,C)$ depends only on $X$ and the cohomology class $[C]\in H^k(X,R)$. By an abuse of notation, we will sometimes write objects as $(X,\al)$ for $\al\in H^k(X,R)$, by which we mean $(X,C)$ for some choice of $k$-cocycle $C$ representing~$\al$.

A morphism of commutative rings $\rho:R_1\ra R_2$ induces a functor
\e
F_\rho:\Bord^{\bs B}_n(K(R_1,k))\longra\Bord^{\bs B}_n(K(R_2,k)),
\label{fm6eq3}
\e
by converting $R_1$-cochains $C,D,E,\ldots$ on $X,Y,Z,\ldots$ to $R_2$-cochains.
\end{dfn}

The next proposition is proved in a very similar way to Propositions \ref{fm5prop1} and \ref{fm4prop1}. It justifies the notation~$\Bord^{\bs B}_n(K(R,k))$.

\begin{prop}
\label{fm6prop1}
{\bf(a)} $\Bord^{\bs B}_n(K(R,k))$ is a Picard groupoid. Its invariants in Theorem\/ {\rm\ref{fmAthm2}(a)} are the $\bs B$-bordism groups
\ea
\pi_0\bigl(\Bord^{\bs B}_n(K(R,k))\bigr)&\cong\Om^{\bs B}_n(K(R,k)),
\label{fm6eq4}\\
\pi_1\bigl(\Bord^{\bs B}_n(K(R,k))\bigr)&\cong\Om_{n+1}^{\bs B}(K(R,k)),
\label{fm6eq5}
\ea
and\/ $q:\Om^{\bs B}_n(K(R,k))\ra\Om_{n+1}^{\bs B}(K(R,k))$ mapping\/ $[X,C]\mapsto[X\t\cS^1,\Pi_X^*(C)]$. Here $K(R,k)$ is the Eilenberg--MacLane space classifying cohomology $H^k(-,R)$ over $R,$ and in $X\t\cS^1,$ the $\cS^1$ has a $\U(1)$-equivariant\/ $\bs B$-structure with the usual orientation, so when $\bs B=\bs\Spin,$ it has the \begin{bfseries}non-bounding\end{bfseries} spin structure $\cS^1_{\rm nb}$.
\smallskip

\noindent{\bf(b)} As in every Picard groupoid, a morphism\/ $\la:(X_0,C_0)\ab\ra(X_1,C_1)$ in\/ $\Bord^{\bs B}_n(K(R,k))$ determines a bijection
\e
\Om_{n+1}^{\bs B}(K(R,k))\longra\Hom_{\Bord^{\bs B}_n(K(R,k))}\bigl((X_0,C_0),(X_1,C_1)\bigr)
\label{fm6eq6}
\e
given by composition in the diagram of bijections
\begin{equation*}
\xymatrix@C=160pt@R=15pt{ 
*+[r]{\Om_{n+1}^{\bs B}(K(R,k))} \ar[d]^{\eq{fm6eq5}} \ar[r]_(0.22)\cong &  *+[l]{\Hom_{\Bord^{\bs B}_n(K(R,k))}\bigl((X_0,C_0),(X_1,C_1)\bigr)} \ar@{=}[d] \\
*+[r]{\Hom_{\Bord^{\bs B}_n(K(R,k))}(\boo,\boo)} \ar[r]^(0.35){
\ot\la}  & *+[l]{\Hom_{\Bord^{\bs B}_n(K(R,k))}\bigl(\boo\!\ot\!(X_0,C_0),\boo\!\ot\!(X_1,C_1)\bigr).} }
\end{equation*}
\end{prop}

\begin{ex}
\label{fm6ex1}
From Proposition \ref{fm6prop1}(a) and Tables \ref{fm2tab1} and \ref{fm3tab1} we see that there are equivalences of Picard groupoids
\e
\begin{aligned}
\Bord_7^{\bs\Spin}(K(\Z,4))&\cong (0\op 0)\qs(\Z^2\op\Z_2^2), \\
\Bord_8^{\bs\Spin}(K(\Z,4))&\cong (\Z^2\op\Z^2)\qs(\Z_2^2\op\Z_2),\\
\Bord_7^{\bs\Spin}(K(\Z_2,4))&\cong (0\op 0)\qs(\Z^2\op\Z_4), \\
\Bord_8^{\bs\Spin}(K(\Z_2,4))&\cong (\Z^2\op\Z_4)\qs(\Z_2^2\op\Z_2).
\end{aligned}
\label{fm6eq7}
\e
Here the right hand sides are Picard groupoids of the form $\pi_0\qs\pi_1$ as in Theorem \ref{fmAthm2} for abelian groups $\pi_0,\pi_1$. The decomposition of the $\pi_i$ as $A\op B$ in \eq{fm6eq7} corresponds to the splitting $\Om_n^{\bs\Spin}(K(R,4))=\Om_n^{\bs\Spin}(*)\op\ti\Om_n^{\bs\Spin}(K(R,4)),$ where $\Om_n^{\bs\Spin}(*)$ is given in Table \ref{fm2tab1} and $\ti\Om_n^{\bs\Spin}(K(R,4))$ in Table \ref{fm3tab1}. Note that $\pi_0\qs\pi_1$ also depends on a linear quadratic map $q:\pi_0\ra\pi_1$, which can be computed from Theorem~\ref{fm3thm1}.
\end{ex}

\subsection{\texorpdfstring{Loop bordism categories $\Bord^{\bs B}_n(\cL K(R,k))$}{Loop bordism categories Bordₙᴮ(ℒK(R,k))}}
\label{fm62}

\begin{dfn}
\label{fm6def2}
Let $R$ be a commutative ring. Fix $n\ge -1$, $0\le k\le n+1$, and a tangential structure $\bs B$ in the sense of \S\ref{fm211}. We will define another Picard groupoid $\Bord^{\bs B}_n(\cL K(R,k))$ that we call a {\it loop bordism category}. It is a simple modification of Definition \ref{fm6def1}: we replace the $k$-cocycles $C\in C^k(X,R)$, $D\in C^k(Y,R)$, $E\in C^k(Z,R)$ by $k$-cocycles $C\in C^k(X\t\cS^1,R)$, $D\in C^k(Y\t\cS^1,R)$, $E\in C^k(Z\t\cS^1,R)$ throughout. So, for example, objects of $\Bord^{\bs B}_n(\cL K(R,k))$ are pairs $(X,C)$, where $X $ is a compact $n$-manifold without boundary with a $\bs B$-structure $\bs\ga_X$, which we generally omit from the notation, and $C\in C^k(X\t\cS^1,R)$ is a $k$-cocycle.
\end{dfn}

We relate the categories of Definitions \ref{fm6def1} and \ref{fm6def2}.

\begin{dfn}
\label{fm6def3}
Let $0\le k\le n$ and $\bs B$ be as above. Define a functor
\e
I_n^{\bs B,K(R,k)}:\Bord_{n-1}^{\bs B}(\cL K(R,k))\longra \Bord^{\bs B}_n(K(R,k))
\label{fm6eq8}
\e
to act on objects by $I_n^{\bs B,K(R,k)}:(X,C)\mapsto (X\t\cS^1,C)$, and on morphisms by $I_n^{\bs B,K(R,k)}:[Y,D]\mapsto [Y\t\cS^1,D]$. Here given the $\bs B$-structures on $X,Y$, to define the $\bs B$-structures on $X\t\cS^1,Y\t\cS^1$ we use the standard $\bs B$-structure on $\cS^1=\R/\Z$, which is invariant under the action of $\R/\Z\cong\U(1)$. So, for example, when $\bs B=\Spin$, we use the $\Spin$-structure on $\cS^1$ whose principal $\Spin(1)$-bundle is the trivial bundle $(\R/\Z)\t\Spin(1)\ra\R/\Z$. It is easy to check that $I_n^{\bs B,K(R,k)}$ is a well-defined symmetric monoidal functor.
\end{dfn}

Here is the analogue of Propositions \ref{fm5prop2} and \ref{fm4prop2}, proved in the same way.

\begin{prop}
\label{fm6prop2}
{\bf(a)} $\Bord^{\bs B}_n(\cL K(R,k))$ is a Picard groupoid. Its invariants in Theorem\/ {\rm\ref{fmAthm2}(a)} are the $\bs B$-bordism groups
\ea
\pi_0(\Bord^{\bs B}_n(\cL K(R,k)))&\cong \Om^{\bs B}_n(\cL K(R,k)),
\label{fm6eq9}\\
\pi_1(\Bord^{\bs B}_n(\cL K(R,k)))=\Aut_{\Bord^{\bs B}_n(\cL K(R,k))}(\boo)&\cong\Om_{n+1}^{\bs B}(\cL K(R,k)),
\label{fm6eq10}
\ea
where\/ $\Om_m^{\bs B}(\cL K(R,k))$ is the bordism group of the free loop space\/ $\cL K(R,k)=\Map_{C^0}(\cS^1,\ab K(R,k))$ of\/ $K(R,k),$ and the linear quadratic map\/ $q\!:\!\Om^{\bs B}_n(\cL K(R,k))\ab\ra\Om_{n+1}^{\bs B}(\cL K(R,k))$ maps\/ $[X,C]\mapsto[\Pi_X^*(C)]$. 

Here in $X\t\cS^1,$ the $\cS^1$ has a $\U(1)$-equivariant\/ $\bs B$-structure with the usual orientation, so when $\bs B=\bs\Spin,$ it has the \begin{bfseries}non-bounding\end{bfseries} spin structure $\cS^1_{\rm nb}$. Note that for $\bs B=\bs\Spin,$ this means that\/ $q$ is multiplication by $\al_1=[\cS^1_{\rm nb}]$ in $\Om_1^{\bs\Spin}(*)$ in Table\/ {\rm\ref{fm2tab2}} under the natural action of\/ $\Om_*^{\bs\Spin}(*)$ on\/~$\Om_*^{\bs\Spin}(\cL BK(R,k))$.
\smallskip

\noindent{\bf(b)} As in every Picard groupoid, a morphism\/ $\la:(X_0,C_0)\ab\to(X_1,C_1)$ in\/ $\Bord^{\bs B}_n(\cL K(R,k))$ determines a bijection
\begin{equation*}
\Om_{n+1}^{\bs B}(\cL K(R,k))\longra\Hom_{\Bord^{\bs B}_n(\cL K(R,k))}\bigl((X_0,C_0),(X_1,C_1)\bigr)
\end{equation*}
given by composition in the diagram of bijections
\begin{equation*}
\xymatrix@C=160pt@R=15pt{ 
*+[r]{\Om_{n+1}^{\bs B}(\cL K(R,k))} \ar[d]^{\eq{fm6eq10}} \ar[r]_(0.22)\cong &  *+[l]{\Hom_{\Bord^{\bs B}_n(\cL K(R,k))}\bigl((X_0,C_0),(X_1,C_1)\bigr)} \ar@{=}[d] \\
*+[r]{\Hom_{\Bord^{\bs B}_n(\cL K(R,k))}(\boo,\boo)} \ar[r]^(0.35){\ot\la}  & *+[l]{\Hom_{\Bord^{\bs B}_n(\cL K(R,k))}\bigl(\boo\!\ot\!(X_0,C_0),\boo\!\ot\!(X_1,C_1)\bigr).}
}
\end{equation*}
\item[{\bf(c)}] There is a commutative diagram
\e
\begin{gathered}
\xymatrix@C=155pt@R=15pt{ 
*+[r]{\Om^{\bs B}_n(\cL K(R,k))} \ar[d]^{\eq{fm6eq10}}_\cong \ar[r]^{\xi^{\bs B}_n(K(R,k))}_{\eq{fm2eq6}} &  *+[l]{\ti\Om_{n+1}^{\bs B}(K(R,k))} \ar[d]_{\eq{fm6eq5}} \\
*+[r]{\Aut_{\Bord_{n-1}^{\bs B}(\cL K(R,k))}(\boo)} \ar[r]^{I_n^{\bs B,K(R,k)}}_{\eq{fm6eq8}}  & *+[l]{\Aut_{\Bord^{\bs B}_n(K(R,k))}(\boo).\!}
}
\end{gathered}
\label{fm6eq11}
\e
\end{prop}

\subsection{\texorpdfstring{Bordism categories $\Bord_X(K(R,k))$}{Bordism categories Bordₓ(K(R,k))}}
\label{fm63}

The next definition is a variation on Definition \ref{fm6def1}, in which we fix the $n$-manifold $X$, and take $Y=X\t[0,1]$ and $Z=X\t[0,1]^2$.

\begin{dfn}
\label{fm6def4}
Let $R$ be a commutative ring. Let $0\le k\le n$ and $X$ be a compact $n$-manifold. Define $\Bord_X(K(R,k))$ to be the category with objects $C$ for $C\in C^k(X,R)$ a $k$-cocycle, and morphisms $[D]:C_0\ra C_1$ be $\sim$-equivalence classes $[D]$ of $k$-cocycles $D\in C^k(X\t[0,1],R)$ with $D\vert_{X\t\{i\}}=C_i$ for $i=0,1$. If $D,D'$ are alternative choices for $D$, we write $D\sim D'$ if there exists $E\in C^k(X\t[0,1]^2,R)$ with
\begin{align*}
E\vert_{X\t\{0\}\t[0,1]}&=\Pi_X^*(C_0),  & E\vert_{X\t\{1\}\t[0,1]}&=\Pi_X^*(C_1), \\ E\vert_{X\t[0,1]\t\{0\}}&=D, & E\vert_{X\t[0,1]\t\{1\}}&=D'.
\end{align*}
To define composition of morphisms $[D]:P_0\ra P_1$ and $[D']:P_1\ra P_2$ we set $[D']\ci[D]=[D'']$, where $D''\in C^k(X\t[0,1],R)$ is given by $D''\vert_{X\t [0,\ha]}=(\id_X\t 2t)^*(D)$ and $D''\vert_{X\t [\ha,1]}=(\id_X\t (2t-1))^*(D')$, using the maps $2t:[0,\ha]\ra[0,1]$ and $2t-1:[\ha,1]\ra[0,1]$.

It is then easy to show composition is associative, so that $\Bord_X(K(R,k))$ is a category, where identity morphisms are $\id_C=\Pi_X^*(C):C\ra C$. Every morphism in $\Bord_X(K(R,k))$ is invertible, where the inverse of $[D]:C_0\ra C_1$ is $[D]^{-1}=[D']:P_1\ra P_0$ with~$D'=(\id_X\t(1-t))^{-1}(D)$.

Now suppose that $\bs B$ is a tangential structure, and $X$ has a $\bs B$-structure $\bs\ga_X$. Since the stable tangent bundles of $X\t[0,1]$ and $X\t[0,1]^2$ are the pullbacks of the stable tangent bundle of $X$, pullback of $\bs\ga_X$ along the projections $X\t[0,1]\ra X$, $X\t[0,1]^2\ra X$ induces $\bs B$-structures on $X\t[0,1]$ and $X\t[0,1]^2$. Define a functor
\e
\Pi_X^{\bs B}:\Bord_X(K(R,k))\longra\Bord^{\bs B}_n(K(R,k))
\label{fm6eq12}
\e
to map $C\mapsto(X,C)$ on objects and $[D]\mapsto\bigl[X\t[0,1],D\bigr]$ on morphisms, using the $\bs B$-structures on $X,X\t[0,1]$. This is well defined as writing $Y=X\t[0,1]$ and $Z=X\t[0,1]^2$, the definitions above of the equivalence $\sim$ on $D$ and $(X\t[0,1],D)$, and of compositions of morphisms, and so on, map to those in Definition~\ref{fm6def1}.
\end{dfn}

In a similar way to Propositions \ref{fm6prop1} and \ref{fm6prop2}, we can use homotopy theory to give a partial description of the categories $\Bord_X(K(R,k))$ and functors~$\Pi_X^{\bs B}$.

\begin{prop}
\label{fm6prop3}
Suppose\/ $\bs B$ is a tangential structure,\/ $X$ a compact\/ $n$-manifold with\/ $\bs B$-structure, and\/ $C\in C^k(X,R)$ for $0\le k\le n$. Then\/ $C$ is an object in\/ $\Bord_X(K(R,k)),$ and\/ $(X,C)$ an object in\/ $\Bord^{\bs B}_n(K(R,k)),$ and\/ $\Pi_X^{\bs B}:C\mapsto(X,C)$. We have a commutative diagram
\begin{equation*}
\xymatrix@C=150pt@R=18pt{ 
*+[r]{\Aut_{\Bord_X(K(R,k))}(C)} \ar[r]_(0.45){\Pi_X^{\bs B}} \ar[d]^{\chi_C^{\bs B}} & *+[l]{\Aut_{\Bord^{\bs B}_n(K(R,k))}(X,C)} \\
*+[r]{\Om^{\bs B}_n(\cL K(R,k))} \ar[r]^(0.45){\xi^{\bs B}_n(K(R,k))} & *+[l]{\ti\Om_{n+1}^{\bs B}(K(R,k)),} \ar[u]^{\eq{fm6eq6}} }
\end{equation*}
where\/ $\xi^{\bs B}_n(K(R,k))$ is in Definition\/ {\rm\ref{fm2def3},} the right hand column is\/ \eq{fm6eq6} restricted to $\ti\Om_{n+1}^{\bs B}(K(R,k))\subset\Om_{n+1}^{\bs B}(K(R,k)),$ and\/ $\chi_C^{\bs B}$ is defined as follows: let\/ $\phi_C:X\ra K(R,k)$ be a classifying map for\/ $C$. Then for\/ $[D]:C\ra C$ in\/ $\Aut_{\Bord_X(K(R,k))}(C),$ as\/ $D\in C^k(X\t[0,1],R)$ with\/ $D\vert_{X\t\{0\}}=C=D\vert_{X\t\{1\}},$ we can choose a classifying map\/ $\phi_D:X\t[0,1]\ra K(R,k)$ for\/ $D$ such that\/ $\phi_D\vert_{X\t\{0\}}=\phi_D\vert_{X\t\{1\}}=\phi_C$. Writing\/ $\cS^1=\R/\Z=[0,1]/(0\sim 1)$ with projection\/ $\pi:[0,1]\ra\cS^1,$ define\/ $\bar\phi_D:X\t\cS^1\ra K(R,k)$ by\/ $\bar\phi_D\ci(\id_X\t\pi)=\phi_D$. Let\/ $\ti\phi_D:X\ra \cL K(R,k)=\Map_{C^0}(\cS^1,K(R,k))$ be the induced map. Then define
\begin{equation*}
\chi_P^{\bs B}([D])=[X,\ti\phi_Q].
\end{equation*}
\end{prop}

\begin{rem}
\label{fm6rem1}
{\bf(a)} One can show that the categories $\Bord_X(K(R,k))$ have an alternate, explicit, very simple description:
\begin{itemize}
\setlength{\itemsep}{0pt}
\setlength{\parsep}{0pt}
\item[(i)] Objects of $\Bord_X(K(R,k))$ are $k$-cocycles $C$ in $C_k(X,R)$ with $\d C=0$.
\item[(ii)] Morphisms $D:C_0\ra C_1$ in $\Bord_X(K(R,k))$ are cohomology classes $[D]$ of $(k-1)$-cochains $D$ in $C_{k-1}(X,R)$ with $\d D=C_1-C_0$.
\end{itemize}
Although $\Bord_X(K(R,k))$ itself is simple, the orientation functors we define on $\Bord_X(K(R,k))$ in \S\ref{fm9} will not be easy to understand.
\smallskip

\noindent{\bf(b)} There is a natural symmetric monoidal structure on $\Bord_X(K(R,k))$, from adding $k$-cocycles and $(k-1)$-cochains on $X$, making it into a Picard groupoid. Note that this is {\it unrelated\/} to the Picard groupoid structure on $\Bord^{\bs B}_n(K(R,k))$ in \S\ref{fm61}, and the functor $\Pi_X^{\bs B}:\Bord_X(K(R,k))\ra\Bord^{\bs B}_n(K(R,k))$ in \eq{fm6eq12} is {\it not monoidal}. The monoidal structure on $\Bord_X(K(R,k))$ is connected to the notion of {\it additive\/} flag structures on 7-manifolds in \S\ref{fm101}, but otherwise we will make no use of it.
\end{rem}

\section{Topological bordism categories}
\label{fm7}

The bordism categories $\Bord^{\bs B}_n(BG)$, $\Bord_{n,k}^{\bs B}(MH)$, $\Bord^{\bs B}_n(K(R,k))$ of \S\ref{fm41}, \S\ref{fm51}, \S\ref{fm61} are all equivalent as Picard groupoids to examples of a single construction.

\begin{dfn}
\label{fm7def1}
Fix $n\ge 0$, a tangential structure $\bs B$, and a topological space $T$. We will define a symmetric monoidal category $\Bord^{\bs B}_n(T)_\top$ that we call a {\it topological bordism category}.
\begin{itemize}
\setlength{\itemsep}{0pt}
\setlength{\parsep}{0pt}
\item[(a)] {\it Objects\/} of $\Bord^{\bs B}_n(T)_\top$ are pairs $(X,f)$, where $X$ is a compact $n$-manifold without boundary with a $\bs B$-structure $\bs\ga_X$, which we generally omit from the notation, and $f:X\ra T$ is a continuous map.
\item[(b)] {\it Morphisms\/} $[Y,g]:(X_0,f_0)\ra(X_1,f_1)$ in $\Bord^{\bs B}_n(T)_\top$ are equivalence classes of pairs $(Y,g),$ see (c), where $Y$ is a compact $(n+1)$-manifold with $\bs B$-structure $\bs\ga_Y$, there is a chosen isomorphism $\pd Y\cong -X_0\amalg X_1$ of the boundary preserving $\bs B$-structures, and $g:Y\ra T$ is a continuous map with $g\vert_{\pd Y}=f_0\amalg f_1$.
\item[(c)] In the situation of (b), let $(Y_0,g_0)$ and $(Y_1,g_1)$ be two choices for $(Y,g)$. We say that $(Y_0,g_0)\sim(Y_1,g_1)$ if there exists a pair $(Z,h)$, where $Z$ is a compact $(n+2)$-manifold with corners and $\bs B$-structure $\bs\ga_Z$, with a chosen isomorphism of boundaries identifying $\bs B$-structures
\e
\pd Z\cong (-X_0\t[0,1])\amalg (X_1\t[0,1]) \amalg -Y_0\amalg Y_1
\label{fm7eq1}
\e
such that along $\pd^2Z$ we identify $\pd Y_i$ with $(-X_0\amalg X_1)\t\{i\}$ for $i=0,1$, and $h:Z\ra T$ is a continuous map such that under \eq{fm7eq1} we have
\begin{equation*}
h\vert_{\pd Z}=(f_0\ci\Pi_{X_0})\amalg (f_1\ci\Pi_{X_1})\amalg g_0\amalg g_1.
\end{equation*}
It is easy to see that `$\sim$' is an equivalence relation.
\item[(d)] If $[Y,g]:(X_0,f_0)\ra(X_1,f_1)$ and $[Y',g']:(X_1,f_1)\ra(X_2,f_2)$ are morphisms, the {\it composition\/} is
\begin{equation*}
[Y',g']\ci [Y,g]=[Y'\amalg_{X_1}Y,g'\amalg_{f_1}g]:(X_0,f_0)\longra(X_2,f_2).
\end{equation*}
To define the smooth structure on $Y'\amalg_{X_1}Y$ we should choose `collars' $X_1\t(-\ep,0]\subset Y$, $X_1\t[0,\ep)\subset Y'$ of $X_1$ in $Y,Y'$, but the choices do not change the equivalence class $[Y'\amalg_{X_1}Y,g'\amalg_{f_1}g]$. Composition is associative.
\item[(e)] If $(X,f)$ is an object in $\Bord^{\bs B}_n(T)_\top,$ the {\it identity morphism\/} is
\begin{equation*}
\id_{(X,f)}=\bigl[X\t[0,1],f\ci\Pi_X\bigr]:(X,f)\longra(X,f).
\end{equation*}
\item[(f)] If $[Y,g]:(X_0,f_0)\ra(X_1,f_1)$ is a morphism, it has an inverse morphism
\begin{align*}
[Y,g]^{-1}&=\bigl[-Y\amalg (Y\amalg_{X_0\amalg X_1}-Y),g\amalg(g\amalg_{f_0\amalg f_1}g)\bigr]:\\
&\qquad (X_1,f_1)\longra(X_0,f_0),
\end{align*}
noting that $\pd(-Y)=-(-X_0\amalg X_1)=-X_1\amalg X_0$. Thus the category $\Bord^{\bs B}_n(T)_\top$ is a {\it groupoid}.
\item[(g)] Define a {\it monoidal structure\/} $\ot$ on $\Bord^{\bs B}_n(T)_\top$ by, on objects
\begin{equation*}
(X,f)\ot (X',f')=(X\amalg X',f\amalg f'),
\end{equation*}
and if $[Y,g]:(X_0,f_0)\ra(X_1,f_1)$, $[Y',g']:(X_0',f_0')\ra(X_1',f_1')$ are morphisms, then
\begin{align*}
&[Y,g]\ot [Y',g']=[Y\amalg Y',g\amalg g']:\\
&(X_0,f_0)\ot(X_0',f_0')\longra(X_1,f_1)\ot (X_1',f_1').
\end{align*}
\item[(h)] The {\it identity\/} in $\Bord^{\bs B}_n(T)_\top$ is $\boo=(\es,\es)$.
\item[(i)] If $(X,f)\in\Bord^{\bs B}_n(T)_\top$ we write $-(X,f)=(-X,f)$, that is, we give $X$ the opposite $\bs B$-structure $-\bs\ga_X$. Observe that we have an isomorphism
\begin{equation*}
\bigl[X\t[0,1],f\ci\Pi_X\bigr]:(-X,f)\ot(X,f)\longra\boo.
\end{equation*}
Thus $-(X,f)$ is an inverse for $(X,f)$ under `$\ot$'.
\item[(j)] The {\it symmetry isomorphism} $\si_{(X,f),(X',f')}=[Y,g]\colon(X,f)\ot(X',f')\ra(X',f')\ot(X,f)$ has $(Y,g)=((X\amalg X')\t[0,1],(f\ci\Pi_X)\amalg(f'\ci\Pi_{X'})$.
\end{itemize}
Hence $\Bord^{\bs B}_n(T)_\top$ is a {\it Picard groupoid}, as in~Appendix \ref{fmA}.
\end{dfn}

The following analogue of Propositions \ref{fm4prop1}, \ref{fm5prop1} and \ref{fm6prop1} is straightforward.

\begin{prop}
\label{fm7prop1}
The invariants of the Picard groupoid\/ $\Bord^{\bs B}_n(T)_\top$  in Theorem\/ {\rm\ref{fmAthm2}(a)} are the $\bs B$-bordism groups
\ea
\pi_0\bigl(\Bord^{\bs B}_n(T) \bigr)&\cong\Om^{\bs B}_n(T),
\label{fm7eq2}\\
\pi_1\bigl(\Bord^{\bs B}_n(T) \bigr)&\cong\Om_{n+1}^{\bs B}(T),
\label{fm7eq3}
\ea
and\/ $q:\Om^{\bs B}_n(T)\ra\Om_{n+1}^{\bs B}(T)$ mapping\/ $[X,f]\mapsto[X\t\cS^1,f\ci\Pi_X]$. 
\end{prop}

Here is the analogue of the bordism categories $\Bord_X(BG)$, $\Bord_X^k(MH)$ and $\Bord_X(K(R,k))$ of \S\ref{fm43}, \S\ref{fm53} and~\S\ref{fm63}.

\begin{dfn}
\label{fm7def2}
Let $X$ be a compact $n$-manifold and $T$ a topological space. Define $\Bord_X(T)_\top$ to be the category with objects $f$ for $f:X\ra T$ a continuous map, and morphisms $[g]:f_0\ra f_1$ be $\sim$-equivalence classes $[g]$ of continuous maps $g:X\t[0,1]\ra T$ with $g\vert_{X\t\{i\}}=f_i$ for $i=0,1$. If $g,g'$ are alternative choices for $g$, we write $g\sim g'$ if there exists a continuous map $h:X\t[0,1]^2\ra T$ with
\begin{align*}
h\vert_{X\t\{0\}\t[0,1]}&=f_0\ci\Pi_X,  & h\vert_{X\t\{1\}\t[0,1]}&=f_1\ci\Pi_X, \\ h\vert_{X\t[0,1]\t\{0\}}&=g, & h\vert_{X\t[0,1]\t\{1\}}&=g'.
\end{align*}
To define composition of morphisms $[g]:f_0\ra f_1$ and $[g']:f_1\ra f_2$ we set $[g']\ci[g]=[g'']$, where $g'':X\t[0,1]\ra T$ is given by $g''(x,t)=g(x,2t)$ for $t\in[0,\ha]$ and $g''(x,t)=g'(x,2t-1)$ for $t\in[\ha,1]$.

It is then easy to show that composition is associative, so that $\Bord_X(T)_\top$ is a category, where identity morphisms are $\id_f=[f\ci\Pi_X]:f\ra f$. Every morphism in $\Bord_X(T)_\top$ is invertible, where the inverse of $[g]:f_0\ra f_1$ is $[g]^{-1}=[g']:f_1\ra f_0$ with~$g'(x,t)=g(x,1-t)$.

Now suppose that $\bs B$ is a tangential structure, and $X$ has a $\bs B$-structure $\bs\ga_X$. Define a functor
\e
\Pi_X^{\bs B}:\Bord_X(T)_\top\longra\Bord^{\bs B}_n(T)_\top
\label{fm7eq4}
\e
to map $f\mapsto(X,f)$ on objects and $[g]\mapsto\bigl[X\t[0,1],g\bigr]$ on morphisms, using the $\bs B$-structures on~$X,X\t[0,1]$.
\end{dfn}

\begin{rem}
\label{fm7rem1}
{\bf(a)} Here is how to relate Definition \ref{fm7def1} to the categories of \S\ref{fm4}--\S\ref{fm6}. For $\Bord^{\bs B}_n(BG)$ as in \S\ref{fm41}, we can construct a functor $\Up:\Bord^{\bs B}_n(BG)\ra\Bord^{\bs B}_n(BG)_\top$ for $T=BG$ using the Axiom of Choice, as follows: for each object $(X,P)$ in $\Bord^{\bs B}_n(BG)$, we choose a classifying map $f_P:X\ra BG$ for $P$, as already used in the proof of Proposition \ref{fm4prop1}, and set $\Up:(X,P)\mapsto(X,f_P)$ on objects. For morphisms $[Y,Q]:(X_0,P_0)\ra(X_1,P_1)$ with $Q\vert_{\pd Y}\cong P_0\amalg P_1$ we choose a classifying map $f_Q:Y\ra BG$ for $Q$ with $f_Q\vert_{\pd Y}=f_{P_0}\amalg f_{P_1}$ and set $\Up:[Y,Q]\mapsto[Y,f_Q]$. Then $[Y,f_Q]$ is independent of the choice of $f_Q$. Comparing Propositions \ref{fm4prop1} and \ref{fm7prop1}, we see that $\Up$ is an equivalence of categories.

Similarly, $\Bord_{n,k}^{\bs B}(MH)$, $\Bord^{\bs B}_n(K(R,k))$ are equivalent to $\Bord^{\bs B}_n(T)_\top$ with $T=MH$ and $T=K(R,k)$.
\smallskip

\noindent{\bf(b)} The loop categories $\Bord^{\bs B}_n(\cL BG)$, $\Bord_{n,k}^{\bs B}(\cL MH)$, $\Bord^{\bs B}_n(\cL K(R,k))$ of \S\ref{fm42}, \S\ref{fm52}, \S\ref{fm62} are equivalent to $\Bord^{\bs B}_n(T)_\top$ for $T=\cL BG$, $\cL MH$, $\cL K(R,k)$.

\smallskip

\noindent{\bf(c)} In a similar way to {\bf(a)}, the categories $\Bord_X(BG)$, $\Bord_X^k(MH)$ and $\Bord_X(K(R,k))$ of \S\ref{fm43}, \S\ref{fm53} and \S\ref{fm63} are equivalent to $\Bord_X(T)_\top$ in Definition \ref{fm7def2} for $T=BG$, $MH$, $K(R,k)$, and these equivalences and those in {\bf(a)} identify the functors $\Pi_X^{\bs B}$ in \eq{fm4eq14}, \eq{fm5eq14}, \eq{fm6eq12} with their analogues in \eq{fm7eq4} up to natural isomorphism.
\smallskip

\noindent{\bf(d)} In Remark \ref{fm6rem1}(b) we noted that $\Bord_X(K(R,k))$ is a Picard groupoid, which is unrelated to the Picard groupoid structure on $\Bord^{\bs B}_n(K(R,k))$ in \S\ref{fm61}.

In a similar way, if the topological space $T$ is an {\it $E_1$-space\/} (a strong kind of H-space, whose multiplication is homotopy commutative and homotopy associative in coherent ways) then we can give $\Bord_X(T)_\top$ a symmetric monoidal structure, and if $T$ is a {\it grouplike $E_1$-space\/} this makes $\Bord_X(T)_\top$ into a Picard groupoid. But $\Pi_X^{\bs B}:\Bord_X(T)_\top\ra\Bord^{\bs B}_n(T)_\top$ is not monoidal in general.
\end{rem}

\section{Transfer functors between bordism categories}
\label{fm8}

\subsection{Transfer functors}
\label{fm81}

The next definition sets up the situation we want to discuss.

\begin{dfn}
\label{fm8def1}
{\bf(a)} Let $\cG,\cG'$ be examples of bordism categories $\Bord^{\bs B}_n(BG)$, $\Bord_{n,k}^{\bs B}(MH)$, $\Bord^{\bs B}_n(K(R,k))$, $\Bord^{\bs B}_n(T)_\top$, $\Bord^{\bs B}_n(*)$ from \S\ref{fm4}--\S\ref{fm7}. Then $\cG,\cG'$ are Picard groupoids, as in Appendix \ref{fmA}. Use the notation $\pi_i=\pi_i(\cG)$, $\pi_i'=\pi_i(\cG')$ for $i=0,1$, and $q:\pi_0\ra\pi_1$, $q':\pi'_0\ra\pi'_1$ for the linear quadratic invariants classifying $\cG,\cG'$ as Picard groupoids in Theorem \ref{fmAthm2}. Write $T,T'$ for the topological classifying spaces $MH,BG,K(R,k),T,*$ corresponding to $\cG,\cG'$ in the notation $\Bord_{n,*}^{\bs B}(T)_*$. Write $n,n'$ and $\bs B,\bs B'$ for the $n,\bs B$ in $\Bord_{n,k}^{\bs B}(MH),\ab\ldots,\ab\Bord^{\bs B}_n(K(R,k))$ in $\cG,\cG'$. Then Propositions \ref{fm4prop1}, \ref{fm5prop1}, \ref{fm6prop1} and \ref{fm7prop1} imply that $\pi_i=\Om_{n+i}^{\bs B}(T)$, $i=0,1$, and $q$ is multiplication by $\al_1=[\cS^1]\in\Om_1^{\bs B}(*)$ under the action of $\Om_*^{\bs B}(*)$ on $\Om_*^{\bs B}(T)$, where $\cS^1$ has the $\U(1)$-invariant $\bs B$-structure, and similarly for~$\pi_0',\pi_1',q'$.

In this section we will study symmetric monoidal functors $F:\cG\ra\cG'$, as in Appendix \ref{fmA}, usually with $n=n'$ and $\bs B=\bs B'$. By Theorem \ref{fmAthm2}(b),(c), any such $F$ induces group morphisms $f_0:\pi_0\ra\pi_0'$ and $f_1:\pi_1\ra\pi_1'$ with $q'\ci f_0=f_1\ci q$, and given such $f_0,f_1$, the set of $F$ up to monoidal natural isomorphism is a torsor over $H^2_\sym(\pi_0,\pi_1')$. We will call such $F$ {\it transfer functors}.

We can roughly divide transfer functors we will study into two kinds, {\bf(i)} {\it topological}, and {\bf(ii)} {\it geometric}, where:
\begin{itemize}
\setlength{\itemsep}{0pt}
\setlength{\parsep}{0pt}
\item[{\bf(i)}] To define a {\it topological\/} transfer functor $F:\cG\ra\cG'$, we require $n=n'$ and $\bs B=\bs B'$, and we choose a continuous map $\phi:T\ra T'$ up to homotopy. 

If $\cG,\cG'$ are topological bordism categories $\Bord^{\bs B}_n(T)_\top$, $\Bord^{\bs B}_n(T')_\top$, then we define $F$ explicitly by $F:(X,f)\mapsto(X,\phi\ci f)$ on objects and $F:[Y,g]\mapsto[Y,\phi\ci g]$ on morphisms. Otherwise, we combine this functor with the equivalences of $\cG,\cG'$ with $\Bord^{\bs B}_n(T)_\top$, $\Bord^{\bs B}_n(T')_\top$ in Remark \ref{fm7rem1}(a) to define $F$ uniquely up to monoidal natural isomorphism. 

In the classification of symmetric monoidal functors $F:\cG\ra\cG'$ in Theorem \ref{fmAthm2}(b),(c), the morphisms $f_i$ for $i=0,1$ are $\phi_*:\Om_{n+i}^{\bs B}(T)\ra \Om_{n+i}^{\bs B}(T')$. The condition $q'\ci f_0=f_1\ci q$ is automatic, as $\phi_*:\Om_*^{\bs B}(T)\ra \Om_*^{\bs B}(T')$ is $\Om_*^{\bs B}(*)$-linear. Thus, Theorem \ref{fmAthm2}(b),(c) tell us that $F:\cG\ra\cG'$ exists with these invariants, and lies in a $H^2_\sym(\pi_0,\pi_1')$-torsor up to natural isomorphism. The topological construction above determines $F$ uniquely up to monoidal natural isomorphism, not just modulo~$H^2_\sym(\pi_0,\pi_1')$.

Observe that if $\phi:T\ra T'$ is $k$-{\it connected\/} for $k\ge n+1$ then $\phi_*:\Om_m^{\bs B}(T)\ra \Om_m^{\bs B}(T')$ is an isomorphism for $m\le k$, so $f_0,f_1$ are isomorphisms (taking $m=n,n+1$), and $F:\cG\ra\cG'$ is an {\it equivalence of categories}. 

Actually, the construction above is not sufficiently general for some purposes. Suppose instead that we are given another topological space $T''$ and continuous maps $\phi:T\ra T''$, $\phi':T'\ra T''$ with $\phi'$ $k$-connected for $k\ge n+1$. Then as $F_{\phi'}$ is an equivalence, there is $F:\cG\ra\cG'$ unique up to monoidal natural isomorphism, such that the following diagram commutes up to monoidal natural isomorphism:
\e
\begin{gathered}
\xymatrix@C=70pt@R=17pt{
*+[r]{\cG} \drtwocell_{}\omit_{}\omit{} \ar@{..>}[d]^{F}  \ar@/^.5pc/[drr]^(0.6){F_\phi} \\
*+[r]{\cG'} \ar[rr]^(0.3){F_{\phi'}}_(0.3)\simeq && *+[l]{\Bord^{\bs B}_n(T'')_\top.} }
\end{gathered}
\label{fm8eq1}
\e
The data $f_0,f_1$ in Theorem \ref{fmAthm2}(b),(c) for $F$ is $f_i=(\phi_*')^{-1}\ci\phi_*$.
\item[{\bf(ii)}] To define a {\it geometric\/} transfer functor $F:\cG\ra\cG'$, we write down $F$ on objects $(X,M),(X,P),(X,C)$ and morphisms $[Y,N],\ab[Y,Q],\ab[Y,D]$ of $\cG,\cG'$ by some explicit geometric construction.

Sometimes defining $F$ explicitly in this way is too much to ask, unless we are willing to use the Axiom of Choice. Instead, by a looser construction, we can define $F$ via a diagram of Picard groupoids and symmetric monoidal functors commuting up to monoidal natural isomorphism:
\e
\begin{gathered}
\xymatrix@C=50pt@R=17pt{
*+[r]{\check{\mathcal G}} \ar[d]^\Pi \drtwocell_{}\omit^{}\omit{} \ar@/^.5pc/[drr]^(0.6){\check F} \\
*+[r]{\cG} \ar@{..>}[rr]^F && *+[l]{\cG'.} }
\end{gathered}
\label{fm8eq2}
\e
Here we take $\check{\mathcal G}$ to be a modification of $\cG$ in which the objects are objects $(X,M),(X,P),(X,C)$ of $\cG$ together with some choice of extra geometric data $\cE$ on $X$ (for example, a Riemannian metric on $X$, a section of a vector bundle on $X,\ldots$). The functor $\Pi$ is the `forgetful functor' which forgets the extra data $\cE$. We require that $\Pi$ should be an equivalence of categories (this happens if the set of choices of extra data $\cE$ for $X$ is contractible modulo bordisms $Y:X_0\ra X_1$). Also $\check F:\check{\mathcal G}\ra\cG'$ should be given by an explicit geometric construction involving the extra data~$\cE$.

Given such $\check{\mathcal G},\Pi,\check F$, since $\Pi$ is an equivalence, there exists $F$ unique up to monoidal natural isomorphism making \eq{fm8eq2} 2-commute, though we may need the Axiom of Choice to actually construct such $F$.
\end{itemize}

\noindent{\bf(b)} The division of transfer functors into topological and geometric in {\bf(a)} is not absolute; for example, given an topological functor, we may be able to find a geometric construction making it into a geometric functor.

\smallskip

\noindent{\bf(c)} For $\Bord^{\bs B}_n(BG)$, $\Bord_{n,k}^{\bs B}(MH)$, $\Bord^{\bs B}_n(K(R,k))$ as in \S\ref{fm4}--\S\ref{fm6}, there are obvious functors to and from $\Bord^{\bs B}_n(*)$:
\e
\begin{gathered}
\xymatrix@C=100pt@R=15pt{
\Bord^{\bs B}_n(*) \ar@<.5ex>[r]^{X\longmapsto (X,X\t G)} & \Bord^{\bs B}_n(BG), \ar@<.5ex>[l]^{(X,P)\longmapsto X} \\ 
\Bord^{\bs B}_n(*) \ar@<.5ex>[r]^{X\longmapsto (X,\es)} & \Bord_{n,k}^{\bs B}(MH), \ar@<.5ex>[l]^{(X,M)\longmapsto X} \\ 
\Bord^{\bs B}_n(*) \ar@<.5ex>[r]^{X\longmapsto (X,0)} & \Bord^{\bs B}_n(K(R,k)). \ar@<.5ex>[l]^{(X,C)\longmapsto X} }
\end{gathered}
\label{fm8eq3}
\e
When $n=n'$ and $\bs B=\bs B'$, we will always choose our functors $F:\cG\ra\cG'$ to commute with these up to monoidal natural isomorphism. Therefore the maps $f_i:\Om_{n+i}^{\bs B}(T)\ra\Om_{n+i}^{\bs B}(T')$ preserve the splittings $\Om_{n+i}^{\bs B}(T$ or $T')=\Om_{n+i}^{\bs B}(*)\op\ti\Om_{n+i}^{\bs B}(T$ or $T')$ and act as the identity on $\Om_{n+i}^{\bs B}(*)$. (This is already automatic in the topological case.) Hence, rather than writing down $f_0,f_1$, it is sufficient to specify $\ti f_i=f_i\vert_{\ti\Om_{n+i}^{\bs B}(T)}:\ti\Om_{n+i}^{\bs B}(T)\ra\ti\Om_{n+i}^{\bs B}(T')$, and to check that~$q'\ci\ti f_0=\ti f_1\ci q$.

There is an analogous condition when $n=n'$ and $\bs B$ factors through~$\bs B'$.
\smallskip

\noindent{\bf(d)} Here is how we will use these functors $F:\cG\ra\cG'$ later. In \S\ref{fm9} we will introduce {\it orientation functors\/} $\sO:\cG\ra 0\qs\Z_2$ or $\Z_2\qs\Z_2$, which control many orientation problems for moduli spaces of instantons in gauge theory, calibrated submanifolds, \ldots. Suppose we have a 2-commutative diagram of symmetric monoidal functors:
\begin{equation*}
\xymatrix@C=70pt@R=17pt{
*+[r]{\cG} \ar[d]^F \drtwocell_{}\omit^{}\omit{} \ar@/^.5pc/[drr]^(0.5){\sO} \\
*+[r]{\cG'} \ar[rr]^(0.4){\sO'} && *+[l]{\text{$0\qs\Z_2$ or $\Z_2\qs\Z_2$.}} }
\end{equation*}

Then orientability, or a choice of orientations, for $\sO'$, implies orientability, or a choice of orientations, for $\sO$, and conversely, non-orientability for $\sO$ implies non-orientability for $\sO'$. If $F$ is an equivalence of categories, then orientability and orientations for $\sO,\sO'$ are equivalent. In this way, having understood orientability/non-orientability, and the data needed to choose orientations, in one problem, we can deduce corresponding results for other problems. 
\end{dfn}

\begin{rem}
\label{fm8rem1}
{\bf(a)} Defining topological transfer functors in Definition \ref{fm8def1}(a)(i) with target $\Bord^{\bs B}_n(K(R,k))$ is easy, as continuous maps $\phi:T\ra K(R,k)$ up to homotopy are equivalent to cohomology classes $\ga\in H^k(T,R)$. When $T=BG$, classes $\ga\in H^k(BG,R)$ are called {\it characteristic classes\/} (for example Chern classes $c_i\in H^{2i}(B\U(m),\Z)$), and are well studied. Having chosen $\ga\in H^k(T,R)$, it still requires work to compute the maps~$\phi_*:\Om^{\bs B}_n(T)\ra\Om^{\bs B}_n(K(R,k))$.
\smallskip

\noindent{\bf(b)} In nearly all the examples we study we have $H^2_\sym(\pi_0,\pi_1')=0$, so there is no issue of specifying $F$ within the $H^2_\sym(\pi_0,\pi_1')$-torsor in Theorem~\ref{fmAthm2}(c).
\end{rem}

\subsection{Examples of topological transfer functors}
\label{fm82}

Theorem \ref{fm3thm3} gives examples of classifying space maps $\phi:T\ra T'$ as in Definition \ref{fm8def1}(a)(i). This yields a large family of topological transfer functors.

\begin{thm}
\label{fm8thm1}
{\bf(a)} For all\/ $n\ge 3,$ $\bs B$ there are topological transfer functors $F_{M\SU(2)}^{B\SU(2)}:\Bord_{n,4}^{\bs B}(M\SU(2))\ab\ra\Bord^{\bs B}_n(B\SU(2))$ which are equivalences of categories, induced by the homotopy equivalence $\phi:M\SU(2)\ra B\SU(2)$ in Theorem\/~{\rm\ref{fm3thm3}(a)}.
\smallskip

\noindent{\bf(b)} For all\/ $n\ge 3,$ $\bs B$ and\/ $m\ge 2$ with\/ $2m\ge n,$ there are topological transfer functors $F_{M\U(2)}^{B\SU(m)}:\Bord_{n,4}^{\bs B}(M\U(2))\ra \Bord^{\bs B}_n(B\SU(m))$ defined as in \eq{fm8eq1} using the diagrams
\begin{equation*}
\xymatrix@C=25pt@R=17pt{
*+[r]{M\U(2)}  \ar[drr]^(0.6){\psi} \\
*+[r]{B\SU(m)} \ar[rr]^{\psi'} && *+[l]{B\SU,} }
\;\>
\xymatrix@C=40pt@R=17pt{
*+[r]{\Bord_{n,4}^{\bs B}(M\U(2))} \drtwocell_{}\omit_{}\omit{} \ar@{..>}[d]^{F_{M\U(2)}^{B\SU(m)}}  \ar@/^.5pc/[drr]^(0.6){F_\psi} \\
*+[r]{\Bord^{\bs B}_n(B\SU(m))} \ar[rr]^{F_{\psi'}}_\simeq && *+[l]{\Bord^{\bs B}_n(B\SU)_\top.} }
\end{equation*}
Here $\psi$ is the $10$-connected map and\/ $\psi'$ the $(2m+1)$-connected map in Theorem\/ {\rm\ref{fm3thm3}(b),} so that\/ $\psi'_*$ is an isomorphism for $m=n,n+1$ as $2m\ge n$. Also $F_{M\U(2)}^{B\SU(m)}$ is an equivalence of categories if\/ $n\le 9,$ as $\psi$ is $10$-connected.

\smallskip

\noindent{\bf(c)} For all\/ $n\ge 3,$ $\bs B$ and\/ $m\ge 1$ with\/ $4m+2\ge n,$ there are topological transfer functors $F_{M\Spin(4)}^{B\Sp(m)}:\Bord_{n,4}^{\bs B}(M\Spin(4))\ra \Bord^{\bs B}_n(B\Sp(m))$ defined as in \eq{fm8eq1} using the diagrams
\begin{equation*}
\xymatrix@C=25pt@R=17pt{
*+[r]{M\Spin(4)}  \ar[drr]^(0.6){\chi} \\
*+[r]{B\Sp(m)} \ar[rr]^{\chi'} && *+[l]{B\Sp,} }
\;\>
\xymatrix@C=40pt@R=17pt{
*+[r]{\Bord_{n,4}^{\bs B}(M\Spin(4))} \drtwocell_{}\omit_{}\omit{} \ar@{..>}[d]^{F_{M\Spin(4)}^{B\Sp(m)}}  \ar@/^.5pc/[drr]^(0.6){F_\chi} \\
*+[r]{\Bord^{\bs B}_n(B\Sp(m))} \ar[rr]^{F_{\chi'}}_\simeq && *+[l]{\Bord^{\bs B}_n(B\Sp)_\top.} }
\end{equation*}
Here $\chi$ is the $12$-connected map and\/ $\chi'$ the $(4m+3)$-connected map in Theorem\/ {\rm\ref{fm3thm3}(c),} so that\/ $\chi'_*$ is an isomorphism for $m=n,n+1$ as $4m+2\ge n$. Also $F_{M\Spin(4)}^{B\SU(m)}$ is an equivalence of categories if\/ $n\le 11,$ as $\chi$ is $12$-connected.
\smallskip

\noindent{\bf(d)} For all\/ $n,\bs B$ there are topological transfer functors $F_{BE_8}^{K(\Z,4)}:\Bord^{\bs B}_n(BE_8)\ab\ra\Bord^{\bs B}_n(K(\Z,4)),$ induced by the $16$-connected map $\om:BE_8\ra K(\Z,4)$ in Theorem\/ {\rm\ref{fm3thm3}(d)}. They are equivalences of categories for $n\le 15$.
\end{thm}

\subsection{Examples of geometric transfer functors}
\label{fm83}

\begin{ex}
\label{fm8ex1}
Here are some simple geometric transfer functors: 
\begin{itemize}
\setlength{\itemsep}{0pt}
\setlength{\parsep}{0pt}
\item[(i)] $F_\ga:\Bord^{\bs B}_n(BG_1)\ra\Bord^{\bs B}_n(BG_2)$ in \eq{fm4eq2} induced by a morphism of Lie groups $\ga:G_1\ra G_2$.
\item[(ii)] $F_\io:\Bord_{n,k}^{\bs B}(MH_1)\ra\Bord_{n,k}^{\bs B}(MH_2)$ in \eq{fm5eq3} induced by a composition of Lie groups $H_1\,{\buildrel\io\over\longra}\, H_2\,{\buildrel\rho_2\over\longra}\,\O(k)$.
\item[(iii)] $F_\rho:\Bord^{\bs B}_n(K(R_1,k))\ra\Bord^{\bs B}_n(K(R_2,k))$ in \eq{fm6eq3} induced by a morphism of commutative rings $\rho:R_1\ra R_2$.
\item[(iv)] If $\bs B,\bs B'$ are tangential structures and $\bs B$ factors through $\bs B'$ as in Definition \ref{fm2def1}, there are obvious functors $\Bord_{n,k}^{\bs B}(MH)\ra\Bord_{n,k}^{\bs B'}(MH)$, $\Bord^{\bs B}_n(BG)\ra\Bord_n^{\bs B'}(BG)$, $\Bord^{\bs B}_n(K(R,k))\ra\Bord_n^{\bs B'}(K(R,k))$ by converting $\bs B$-structures into $\bs B'$-structures.
\end{itemize}
\end{ex}

Here is an example of a geometric transfer functor defined as in the second part of Definition~\ref{fm8def1}(a)(ii).

\begin{dfn}
\label{fm8def8}
Let $\bs B$ be a tangential structure and $n\ge 3$. We will define a Picard groupoid $\cBord^{\bs B}_n(B\SU(2))$ which is a modification of $\Bord^{\bs B}_n(B\SU(2))$ in Definition \ref{fm4def1}. The difference is that to each principal $\SU(2)$-bundle $P\ra X,Q\ra Y,\ldots$ in the definition of $\Bord^{\bs B}_n(B\SU(2))$, we also associate smooth, transverse sections $s,t,\ldots$ of the associated $\C^2$-bundles $(P\t\C^2)/\SU(2)\ra X,$ $(Q\t\C^2)/\SU(2)\ra Y,\ldots.$ 

In more detail, define {\it objects\/} of $\cBord^{\bs B}_n(B\SU(2))$ to be triples $(X,P,s)$, where $X$ is a compact $n$-manifold with $\bs B$-structure $\be_X$, and $P\ra X$ is a principal $\SU(2)$-bundle, and $s\in\Ga^\iy((P\t\C^2)/\SU(2))$ is a smooth, transverse section of the $\C^2$-bundle $(P\t\C^2)/\SU(2)\ra X$ associated to $P\ra X$ and the standard representation of $\SU(2)$ on~$\C^2$.

Also, {\it morphisms\/} $[Y,Q,t]:(X_0,P_0,s_0)\ra(X_1,P_1,s_1)$ in $\cBord^{\bs B}_n(B\SU(2))$ are equivalence classes of triples $(Y,Q,t),$ where $Y$ is a compact $(n+1)$-manifold with $\bs B$-structure $\be_Y$, there is a chosen isomorphism $\pd Y\cong -X_0\amalg X_1$ of the boundary preserving $\bs B$-structures, and $Q\ra Y$ is a principal $\SU(2)$-bundle with a chosen isomorphism $Q\vert_{\pd Y}\cong P_0\amalg P_1$, and $t\in\Ga^\iy((Q\t\C^2)/\SU(2))$ is a smooth, transverse section of the $\C^2$-bundle $(Q\t\C^2)/\SU(2)\ra Y$ with $t\vert_{\pd Y}\cong s_0\amalg s_1$. Equivalences $(Y_0,Q_0,t_0)\ra(Y_1,Q_1,t_1)$ are defined in the obvious way.

We make $\cBord^{\bs B}_n(B\SU(2))$ into a symmetric monoidal category in the usual way. Then $\cBord^{\bs B}_n(B\SU(2))$ is a Picard groupoid.

Define a forgetful symmetric monoidal functor
\begin{equation*}
\Pi_{B\SU(2)}^{M\SU(2)}:\cBord^{\bs B}_n(B\SU(2))\longra\Bord^{\bs B}_n(B\SU(2))
\end{equation*}
to forget all transverse sections $s,t,\ldots,$ so that $\Pi_{B\SU(2)}^{M\SU(2)}$ maps $(X,P,s)\mapsto(X,P)$ on objects, for example.

For an object $(X,P,s)$ as above, observe that as $s$ is transverse, $M=s^{-1}(0)$ is an embedded $(n-4)$-submanifold of $X$. The derivative $\d s\vert_M$ induces an isomorphism $\nu_M\ra ((P\t\C^2)/\SU(2))\vert_M$ of vector bundles on $M$, where $\nu_M$ is the normal bundle of $M$ in $X$. Since $((P\t\C^2)/\SU(2))\vert_M$ has an $\SU(2)$-structure, this induces an $\SU(2)$-structure $\ga_M$ on~$\nu_M$.

For $\Bord_{n,4}^{\bs B}(M\SU(2))$ as in \S\ref{fm51}, define a symmetric monoidal functor
\begin{equation*}
\check F_{B\SU(2)}^{M\SU(2)}:\cBord^{\bs B}_n(B\SU(2))\longra\Bord_{n,4}^{\bs B}(M\SU(2))
\end{equation*}
to act by $(X,P,s)\mapsto(X,M)$ on objects, where $M=s^{-1}(0)\subset X$ with normal $\SU(2)$-structure $\ga_M$ as above, and to act by $[Y,Q,t]\mapsto[Y,N]$ on morphisms, where $N=t^{-1}(0)\subset Y$ with normal $\SU(2)$-structure $\ga_N$. Note that our definition of $t$ being transverse includes that $t$ is transverse on each boundary or corner stratum of $N$, which implies that $N=t^{-1}(0)$ is a {\it neat\/} submanifold.
\end{dfn}

\begin{thm}
\label{fm8thm2}
$\Pi_{B\SU(2)}^{M\SU(2)},\check F_{B\SU(2)}^{M\SU(2)}$ are equivalences of Picard groupoids. Hence there exists an equivalence $F_{B\SU(2)}^{M\SU(2)}$ in a $2$-commutative diagram
\e
\begin{gathered}
\xymatrix@C=170pt@R=15pt{
*+[r]{\cBord^{\bs B}_n(B\SU(2))} \ar[d]^{\Pi_{B\SU(2)}^{M\SU(2)}} \ar@/^1pc/[dr]^(0.6){\check F_{B\SU(2)}^{M\SU(2)}}
\drtwocell_{}\omit^{}\omit{} \\
*+[r]{\Bord^{\bs B}_n(B\SU(2))} \ar[r]^(0.39){F_{B\SU(2)}^{M\SU(2)}} & *+[l]{\Bord_{n,4}^{\bs B}(M\SU(2)).} }
\end{gathered}
\label{fm8eq4}
\e
This $F_{B\SU(2)}^{M\SU(2)}$ is a \begin{bfseries}geometric transfer functor\end{bfseries}. It is one of the topological transfer functors $F_{B\SU(2)}^{M\SU(2)}$ in Theorem\/ {\rm\ref{fm8thm1}(a)}.
\end{thm}

\begin{proof}
To see that $\Pi_{B\SU(2)}^{M\SU(2)}$ is an equivalence of categories, note that all the vector bundles $(P\t\C^2)/\SU(2)\ra X$, $(Q\t\C^2)/\SU(2)\ra Y,\ldots$ admit transverse sections, and choices of transverse sections on a boundary $\pd Y,\pd Z,\ldots$ (with the obvious compatibility conditions at codimension 2 corners $\pd^2Z$) can always be extended to~$Y,Z,\ldots.$

To prove $\check F_{B\SU(2)}^{M\SU(2)}$ is an equivalence is more complicated. First, let $(X,M)$ be an object in $\Bord_{n,4}^{\bs B}(M\SU(2))$, so that $X$ is a compact $n$-manifold with $\bs B$-structure $\be_X$ and $M\subset X$ is a compact embedded $(n-4)$-submanifold with an $\SU(2)$-structure $\ga_M$ on its normal bundle $\nu_M\ra M$. We will construct $(X,P,s)$ in $\cBord^{\bs B}_n(B\SU(2))$ with~$\check F_{B\SU(2)}^{M\SU(2)}(X,P,s)=(X,M)$.

Choose a tubular neighbourhood for $M$ in $X$. That is, we choose an open neighbourhood $U$ of the zero section $0(M)$ in $\nu_M$, an open neighbourhood $V$ of $M$ in $X$, and a diffeomorphism $\Phi:U\ra V$ such that $\Phi\ci 0=\inc:M\hookra V$, and the derivative of $\Phi$ normal to $0(M)$ induces the identity map $\id:\nu_M\ra\nu_M$. Write $\pi:U\ra M$ for the restriction of the projection $\nu_M\ra M$ to $U$. Then $\pi^*(\nu_N)\ra U$ is a real rank 4 vector bundle, which has a tautological section $\dot s\in\Ga^\iy(\pi^*(\nu_N))$ with $\dot s(x,e)=e$ for $(x,e)\in U\subset\nu_M$, so that $x\in M$ and $e\in\nu_M\vert_x$. Note that $\dot s$ is transverse with $\dot s^{-1}(0)=0(M)$. Also the $\SU(2)$-structure $\ga_M$ on $\nu_M$ pulls back to an $\SU(2)$-structure $\pi^*(\ga_M)$ on $\pi^*(\nu_N)\ra U$. That is, $\pi^*(\nu_N)$ has the structure of a complex rank 2 vector bundle with a Hermitian metric $g_\nu$ and a complex volume form $\th_\nu$ on the fibres.

Choose a partition of unity $(\eta_1,\eta_2)$ on $X$ subordinate to the open cover $(V,X\sm M)$. Define another smooth section $\ddot s\in\Ga^\iy(\pi^*(\nu_N))$ by
\begin{equation*}
\ddot s=(\Phi^*(\eta_1)+\Phi^*(\eta_2)\md{\dot s}^2)^{-1/2}\cdot\dot s,
\end{equation*}
where $\md{\dot s}$ is defined using $g_\nu$. As $(\Phi^*(\eta_1)+\Phi^*(\eta_2)\md{\dot s}^2)^{-1/2}$ is positive, $\ddot s$ is also transverse with $\ddot s^{-1}(0)=0(M)$, but it has the extra property that $\md{\ddot s}=1$ outside $\supp\Phi^*(\eta_1)$ in~$U$.

Since $\SU(2)$ acts freely and transitively on the unit sphere in $\C^2$, a unit length section $s$ of an $\SU(2)$-vector bundle $(P\t\C^2)/\SU(2)\ra X$ induces a trivialization of the principal $\SU(2)$-bundle $P$ identifying $s$ with the constant section with value $(1,0)\in\C^2$. Thus, $\ddot s$ induces a trivialization of the $\SU(2)$-structure of $\pi^*(\nu_N)\ra U$ on $U\sm\supp\Phi^*(\eta_1)$. Define a principal $\SU(2)$-bundle $P'\ra X$ and a smooth section $s'\in \Ga^\iy((P'\t\C^2)/\SU(2))$ by
\begin{itemize}
\setlength{\itemsep}{0pt}
\setlength{\parsep}{0pt}
\item[(a)] On $X\sm\supp\eta_1\subset X$, take $P'$ to be the trivial $\SU(2)$-bundle, and $s'$ to be the the constant section with value $(1,0)\in\C^2$.
\item[(b)] On $V\subset X$ take $P'$ to be the principal $\SU(2)$-bundle associated to the $\SU(2)$-vector bundle $\Phi_*(\pi^*(\nu_N))$ and $s'$ to be $\Phi_*(\ddot s)$.
\item[(c)] On the overlap $V\sm\supp\eta_1$ of (a) and (b), we identify the two using the trivialization of $\Phi_*(\pi^*(\nu_N))$ identifying $\Phi_*(\ddot s)\cong(1,0)$, noting that $\bmd{\Phi_*(\ddot s)}=1$ on~$V\sm\supp\eta_1$.
\end{itemize}
Then $s'$ is transverse with $s^{\prime -1}(0)=M$ by construction, and the $\SU(2)$-structure on $\nu_M$ induced by $\d s'\vert_M$ is $\ga_M$. Hence $\check F_{B\SU(2)}^{M\SU(2)}(X,P',s')=(X,M)$.

This construction $(X,M)\rightsquigarrow(X,P',s')$ (which depends on arbitrary choices $U,V,\Phi,\eta_1,\eta_2$) is close to being an inverse to $\check F_{B\SU(2)}^{M\SU(2)}:(X,P,s)\mapsto(X,M)$ (which involves no arbitrary choices). If we start with $(X,P,s)$, set $(X,M)=\check F_{B\SU(2)}^{M\SU(2)}(X,P,s)$, and then construct $(X,P',s')$ as above, we can show that there is a unique isomorphism $P'\cong P$ which identifies $s'$ with $f\cdot s$, where $f:X\ra(0,\iy)$ is continuous on $X$ and smooth on $X\sm M$, with $f\vert_M\equiv 1$, and $f(x)=\bmd{s'\vert_x}/\bmd{s\vert_x}$ for $x\in X\sm M$. In effect, the only data forgotten by $\check F_{B\SU(2)}^{M\SU(2)}:(X,P,s)\mapsto(X,M)$ is the function $\ms{s}:X\ra[0,\iy)$. This data lies in a contractible set, giving objects which are all isomorphic in~$\cBord^{\bs B}_n(B\SU(2))$.

The construction $(X,M)\rightsquigarrow(X,P',s')$ also works for $(Y,N)\rightsquigarrow(Y,Q',t')$ in morphisms $[Y,N],[Y,Q',t']$, and for the equivalence relations defining morphisms, and can be made compatible with previous choices on boundaries. Using this we can show that $\check F_{B\SU(2)}^{M\SU(2)}$ is an equivalence of categories. The existence of another equivalence $F_{B\SU(2)}^{M\SU(2)}$ in a diagram \eq{fm8eq4} then follows by category theory general nonsense.
\end{proof}

\begin{rem}
\label{fm8rem2}
We can also generalize Definition \ref{fm8def8} and Theorem \ref{fm8thm2} to the transfer functors in Theorem \ref{fm8thm1}(b),(c). We define functors of Picard groupoids
\e
\begin{gathered}
\xymatrix@C=100pt@R=15pt{
*+[r]{\cBord^{\bs B}_n(B\SU(m))} \ar[d]^{\Pi_{B\SU(m)}^{M\U(2)}} \ar@/^.3pc/[dr]^{\check F_{B\SU(m)}^{M\U(2)}} \\
*+[r]{\Bord^{\bs B}_n(B\SU(m))}  & *+[l]{\Bord_{n,4}^{\bs B}(M\U(2)).} }
\end{gathered}
\label{fm8eq5}
\e
Here objects of $\cBord^{\bs B}_n(B\SU(m))$ are $(X,P,s_1,\ldots,s_{m-1})$, where $X$ is a compact $n$-manifold with $\bs B$-structure $\be_X$, and $P\ra X$ is a principal $\SU(m)$-bundle, and $s_1,\ldots,s_{m-1}\in\Ga^\iy((P\t\C^m)/\SU(m))$ are smooth sections of the $\C^m$-bundle $(P\t\C^m)/\SU(m)\ra X$, such that $\an{s_1\vert_x,\ldots,s_{m-1}\vert_x}_\C$ has $\C$-dimension $m-2$ or $m-1$ at each $x\in X$, and $s_1,\ldots,s_{m-1}$ are generic with this condition. 

Then $\Pi_{B\SU(m)}^{M\U(2)}$ maps $(X,P,s_1,\ldots,s_{m-1})\mapsto(X,P)$, and $\check F_{B\SU(m)}^{M\U(2)}$ maps $(X,P,s_1,\ldots,s_{m-1})\mapsto(X,M)$, where $M$ is the subset of $x\in X$ such that $\an{s_1\vert_x,\ldots,s_{m-1}\vert_x}_\C$ has $\C$-dimension $m-2$. It turns out that $N$ is an embedded submanifold of $X$ of codimension 4, and we can define a $\U(2)$-structure on its normal bundle $\nu_M$ in $X$. Also $\Pi_{B\SU(m)}^{M\U(2)}$ is an equivalence if $n\le 9$, and $\check F_{B\SU(m)}^{M\U(2)}$ is an equivalence if $2m\ge n$, so if $2m\ge n$ we can complete \eq{fm8eq5} with a transfer functor $F_{M\U(2)}^{B\SU(m)}$ as in Theorem~\ref{fm8thm1}(b).

Similarly, we define functors of Picard groupoids
\e
\begin{gathered}
\xymatrix@C=100pt@R=15pt{
*+[r]{\cBord^{\bs B}_n(B\Sp(m))} \ar[d]^{\Pi_{B\Sp(m)}^{M\Spin(4)}} \ar@/^.3pc/[dr]^{\check F_{B\Sp(m)}^{M\Spin(4)}} \\
*+[r]{\Bord^{\bs B}_n(B\Sp(m))}  & *+[l]{\Bord_{n,4}^{\bs B}(M\Spin(4)).} }
\end{gathered}
\label{fm8eq6}
\e
Here objects of $\cBord^{\bs B}_n(B\Sp(m))$ are $(X,P,s_1,\ldots,s_m)$, where $X$ is a compact $n$-manifold with $\bs B$-structure $\be_X$, and $P\ra X$ is a principal $\Sp(m)$-bundle, and $s_1,\ldots,s_m\in\Ga^\iy((P\t\H^m)/\Sp(m))$ are smooth sections of the $\H^m$-bundle $(P\t\H^m)/\Sp(m)\ra X$, such that $\an{s_1\vert_x,\ldots,s_m\vert_x}_\H$ has $\H$-dimension $m-1$ or $m$ at each $x\in X$, and $s_1,\ldots,s_m$ are generic with this condition. 

Then $\Pi_{B\Sp(m)}^{M\Spin(4)}$ maps $(X,P,s_1,\ldots,s_m)\mapsto(X,P)$, and $\check F_{B\Sp(m)}^{M\Spin(4)}$ maps $(X,P,s_1,\ldots,s_m)\mapsto(X,M)$, where $M$ is the subset of $x\in X$ such that $\an{s_1\vert_x,\ldots,s_m\vert_x}_\H$ has $\H$-dimension $m-1$. It turns out that $N$ is an embedded submanifold of $X$ of codimension 4, and we can define a $\Spin(4)$-structure on its normal bundle $\nu_M$ in $X$. Also $\Pi_{B\Sp(m)}^{M\Spin(4)}$ is an equivalence if $n\le 11$, and $\check F_{B\SU(m)}^{M\U(2)}$ is an equivalence if $4m+2\ge n$, so if $4m+2\ge n$ we can complete \eq{fm8eq6} with a transfer functor $F_{M\Spin(4)}^{B\Sp(m)}$ as in Theorem~\ref{fm8thm1}(c).
\end{rem}

We can make many other transfer functors by composing functors from Example \ref{fm8ex1} and Theorems \ref{fm8thm1} and \ref{fm8thm2}, or their quasi-inverses in the case they are equivalences.

\section{Orientation functors}
\label{fm9}

\subsection{Orientation functors, orientability, and orientations}
\label{fm91}

\begin{dfn}
\label{fm9def1}
By an {\it orientation functor\/} we will mean a symmetric monoidal functor $\sO:\cC\ra A\qs B$, where:
\begin{itemize}
\setlength{\itemsep}{0pt}
\setlength{\parsep}{0pt}
\item[(a)] $\cC$ is one of $\Bord^{\bs B}_n(BG),$ $\Bord_{n,k}^{\bs B}(MH),$ $\Bord^{\bs B}_n(K(R,k)),$ $\Bord^{\bs B}_n(T)_\top$ from \S\ref{fm4}--\S\ref{fm7}, and
\item[(b)] $A\qs B$ is a Picard groupoid as in Appendix \ref{fmA}, where $A,B$ are abelian groups from the list $0,\Z$ or $\Z_k$ for $k=2,\ldots.$

Our most frequent choices for $A\qs B$ will be $A=0$, $B=\Z_2$ giving $0\qs\Z_2=\Ztor$, the Picard groupoid of $\Z_2$-{\it torsors}, or $A=B=\Z_2$ giving $\Z_2\qs\Z_2=\sZtor$, the Picard groupoid of $\Z_2$-{\it graded\/ $\Z_2$-torsors}, or {\it super\/ $\Z_2$-torsors}, since these have applications to orientations of moduli spaces. Examples with $B=\Z_k$ or $\Z$ are relevant to gradings of Floer homology theories.

As in Theorem \ref{fmAthm2}(a), the Picard groupoid $A\qs B$ depends up to equivalence on $A,B$ and a choice of linear quadratic form $q:A\ra B$.
\end{itemize} 

See \S\ref{fm93}--\S\ref{fm95} for examples of orientation functors. We will mainly discuss orientation functors of two kinds:
\begin{itemize}
\setlength{\itemsep}{0pt}
\setlength{\parsep}{0pt}
\item[(i)] {\it Abstract orientation functors}, defined using choices of group morphisms $f_0:\pi_0(\cC)\ra A$, $f_1:\pi_1(\cC)\ra B$ using Theorem \ref{fmAthm2}(b); and
\item[(ii)] {\it Analytic orientation functors}, in which $\sO(X,M),\ab\ldots$ are defined using some linear elliptic operator $L$ on $X,M,\ldots,$ and may involve $\ind(L),\ab\Ker L,\ab\Coker L$, or the spectrum of~$L$.
\end{itemize}
One could also consider orientation functors defined using techniques from topology or differential geometry, but we will not focus on these.
\end{dfn}

\begin{dfn}
\label{fm9def2}
Let $\sO:\cC\ra A\qs B$ be an orientation functor. For clarity take $\cC=\Bord^{\bs B}_n(BG)$ from \S\ref{fm41}. Let $X$ be a compact $n$-manifold with a $\bs B$-structure, so that Definition \ref{fm4def4} gives a functor
\begin{equation*}
\Pi_X^{\bs B}:\Bord_X(BG)\longra\Bord^{\bs B}_n(BG).
\end{equation*}
An {\it orientation of\/ $\sO$ for\/} $X$ is a natural isomorphism $\eta_X$ in the 2-commutative diagram of functors:
\e
\begin{gathered}
\xymatrix@C=120pt@R=15pt{
*+[r]{\Bord_X(BG)} \drtwocell_{}\omit^{}\omit{^{\eta_X\,\,\,}} \ar[r]_(0.35)\boo \ar[d]^{\Pi_X^{\bs B}} & *+[l]{0\qs B=B\text{-tor}} \\
*+[r]{\Bord^{\bs B}_n(BG)} \ar[r]^(0.6)\sO & *+[l]{A\qs B.\!} \ar[u]^{F_{A\qs B}^{0\qs B}} }
\end{gathered}
\label{fm9eq1}
\e
Here $A\qs B$ is the category of $A$-graded $B$-torsors, and $F_{A\qs B}^{0\qs B}$ is the forgetful functor which forgets the $A$-grading, and $\boo:\Bord_X(BG)\ra B$-tor is the trivial functor taking every object to $B$ and every morphism to $\id_B$.

We say that $\sO$ is {\it orientable for\/} $X$ if an orientation for $X$ exists.

We make the analogous definitions for the other classes of bordism categories
$\Bord_{n,k}^{\bs B}(MH),\Bord^{\bs B}_n(K(R,k))$ and~$\Bord^{\bs B}_n(T)_\top$.
\end{dfn}

The next example gives some motivation for these definitions:

\begin{ex}
\label{fm9ex1}
Suppose $X$ is a compact 8-manifold with a $\Spin(7)$-structure $(\Om,g)$ in the sense of \cite[\S 10]{Joyc1}, which need not have $\d\Om=0$. Then there is a natural splitting $\La^2T^*X=\La^2_7T^*X\op\La^2_{21}T^*X$ into vector subbundles of ranks 7 and 21. Suppose $G$ is a Lie group and $P\ra X$ a principal $G$-bundle. A $\Spin(7)$-{\it instanton\/} on $P$ is a connection $\nabla_P$ on $P$ with $\pi^2_7(F^{\nabla_P})=0$ in $\Ga^\iy(\Ad(P)\ot\La^2_7T^*X)$. Write $\M_P^{\Spin(7)}$ for the moduli space of irreducible $\Spin(7)$-instantons on $P$. Then $\M_P^{\Spin(7)}$ is a derived manifold in the sense of \cite{Joyc2,Joyc3,Joyc4,Joyc6}, and an ordinary manifold if $\Om$ is generic. Examples of $\Spin(7)$-instantons were given by Lewis \cite{Lewi}, Tanaka \cite{Tana}, and Walpuski \cite{Walp5}. Donaldson and Thomas \cite{DoTh} proposed defining enumerative invariants of $(X,\Om,g)$ by `counting' $\Spin(7)$-instantons. In \S\ref{fm93}, using material from \cite{Upme2}, we will define an analytic orientation functor
\begin{equation*}
\sO:\Bord_8^{\bs\Spin}(BG)\longra\sZtor.
\end{equation*}
As in \S\ref{fm123}, it turns out that an orientation of $\sO$ for $X$ in the sense of Definition \ref{fm9def2} induces orientations on $\M_P^{\Spin(7)}$ for all principal $G$-bundles $P\ra X$. These are needed for the Donaldson--Thomas programme~\cite{DoTh}. 
\end{ex}

As we will see later, orientation functors can also be used to orient many other moduli spaces in gauge theory, and moduli spaces of calibrated submanifolds, and moduli spaces of coherent sheaves on Calabi--Yau 4-folds.

Our goal in this section will be to answer the following questions:

\begin{quest}
\label{fm9quest1}
Let\/ $\sO:\cC\ra A\qs B$ be an orientation functor, where $\cC$ is one of $\Bord^{\bs B}_n(BG),\Bord_{n,k}^{\bs B}(MH)$ or $\Bord^{\bs B}_n(K(R,k))$. Then we can ask:
\begin{itemize}
\setlength{\itemsep}{0pt}
\setlength{\parsep}{0pt}
\item[{\bf(a)}] Is\/ $\sO$ orientable for all compact\/ $n$-manifolds $X$ with\/ $\bs B$-structure $\bs\ga_X$?
\item[{\bf(b)}] If the answer to {\bf(a)} is no, can we give computable necessary and sufficient, or just sufficient, conditions for\/ $\sO$ to be orientable for given\/~$X,\bs\ga_X$?
\item[{\bf(c)}] If\/ $\sO$ is orientable for\/ $X,\bs\ga_X,$ can we specify additional data on\/ $X$ which can be used to determine a canonical choice of orientation for\/ $\sO$ on\/~$X$? 
\end{itemize}

Here in {\bf(c)\rm,} the orientations for\/ $\sO$ on\/ $X$ are a torsor for\/ $\Map\bigl(\Om^{\bs B}_n(X),B\bigr)$ for\/ $T=BG,MH$ or\/ $K(R,k),$ where\/ $\Map$ means just maps of sets, not group morphisms, made into a group using the group structure on\/ $B$. So the set of orientations is usually uncountably infinite. We will be most happy with the answer to\/ {\bf(c)} if the possible choices for the additional data on\/ $X$ can be made as small as possible (e.g.\ if there is only a finite choice), and in particular, if the set of such choices is much smaller than\/~$\Map\bigl(\Om^{\bs B}_n(X),B\bigr)$.
\end{quest}

\begin{rem}
\label{fm9rem1}
A very useful technique for answering Question \ref{fm9quest1} is when {\it two orientation functors factor via a transfer functor}. We will illustrate this for the case of a topological transfer functor $F:\Bord^{\bs B}_n(BG)\ra\Bord^{\bs B}_n(K(R,k))$, as in \S\ref{fm8}, though the idea works for other transfer functors as well. Suppose we are given a 2-commutative diagram of symmetric monoidal functors
\e
\begin{gathered}
\xymatrix@C=70pt@R=17pt{
*+[r]{\Bord^{\bs B}_n(BG)} \drtwocell_{}\omit^{}\omit{^{\th\,\,\,\,}} \ar[d]^F \ar@/^.8pc/[drr]^(0.6)\sO \\
*+[r]{\Bord^{\bs B}_n(K(R,k))}  \ar[rr]^(0.5){\sO'} && *+[l]{A\qs B,} }
\end{gathered}
\label{fm9eq2}
\e
where $F$ is a transfer functor and $\sO,\sO'$ are orientation functors. Let $X$ be a compact $n$-manifold with a $\bs B$-structure, and consider the diagram
\begin{equation*}
\xymatrix@C=70pt@R=13pt{
*+[r]{\Bord_X(BG)} \drrtwocell_{}\omit^{}\omit{^{\id\,\,\,}} \ar[dr]_{F_X}  \ar[rr]_(0.35)\boo \ar[ddd]^{\Pi_X^{\bs B}} &  & *+[l]{0\qs B=B\text{-tor}} \\
\drtwocell_{}\omit^{}\omit{^{\id\,\,\,}} & \Bord_X(K(R,k)) \ar[d]^{\Pi_X^{\bs B}} \ar[ur]_\boo & \\
& \Bord^{\bs B}_n(K(R,k))  \ar[dr]^{\sO'} \urtwocell_{}\omit^{}\omit{^{\eta'_X\,\,\,}} & \dlltwocell_{}\omit^{}\omit{^{\,\,\,\,\th}} \\
*+[r]{\Bord^{\bs B}_n(BG)} \ar[ur]^F \ar[rr]^(0.6)\sO && *+[l]{A\qs B.\!} \ar[uuu]^{F_{A\qs B}^{0\qs B}} }
\end{equation*}
Here $\eta_X'$ is some orientation for $\sO'$ on $X$. Composing natural isomorphisms across the diagram gives a natural isomorphism $\eta_X$ as in \eq{fm9eq1}, which is an orientation for $\sO$ on $X$. Thus, given a diagram~\eq{fm9eq2}:
\begin{itemize}
\setlength{\itemsep}{0pt}
\setlength{\parsep}{0pt}
\item[(i)] An orientation $\eta_X'$ for $\sO'$ on $X$ determines an orientation $\eta_X$ for $\sO$ on $X$.
\item[(ii)] If $\sO'$ is orientable for $X$ then $\sO$ is orientable for $X$.
\item[(iii)] If $\sO$ is not orientable for $X$ then $\sO'$ is not orientable for~$X$.
\end{itemize}
We can use this when the target categories $\Bord_X(K(R,k)),\Bord^{\bs B}_n(K(R,k))$ are smaller and simpler than the domain categories $\Bord_X(BG),\Bord^{\bs B}_n(BG)$. Then orientability for $\sO'$ is a sufficient condition for orientability for $\sO$, as in Question \ref{fm9quest1}(b), and an orientation $\eta_X'$ for $\sO'$ on $X$ is additional data which determines an orientation $\eta_X$ on $X$, as in Question~\ref{fm9quest1}(c).
\end{rem}

Here is our main theorem on orientability, proved in~\S\ref{fm191}. 

\begin{thm}
\label{fm9thm1}
Suppose $n\ge 0,$ and\/ $\bs B$ is a tangential structure, and\/ $\sO:\Bord^{\bs B}_n(BG)\ra A\qs B$ is an orientation functor, as in Definition\/ {\rm\ref{fm9def1}}. Then:
\begin{itemize}
\setlength{\itemsep}{0pt}
\setlength{\parsep}{0pt}
\item[{\bf(a)}] Consider the commutative diagram
\end{itemize}
\e
\begin{gathered}
\xymatrix@!0@C=90pt@R=35pt{
& \Aut_{\Bord_{n-1}^{\bs B}(\cL BG)}(\boo)  \ar[rr]^(0.37){I_n^{\bs B,G}}_(0.37){\eq{fm4eq10}} && *+[l]{\Aut_{\Bord^{\bs B}_n(BG)}(\boo)} \ar[dd]_(0.4){\sO(\boo)} \\
*+[r]{\Om^{\bs B}_n(\cL BG)} \ar[ur]^{\eq{fm4eq12}}_\cong \ar[d]^{\Pi^{\bs B}_n(BG)} \ar[rr]^(0.4){\xi^{\bs B}_n(BG)}_(0.35){\eq{fm2eq6}}  && \ti\Om_{n+1}^{\bs B}(BG)  \ar[ur]^(0.4){\eq{fm4eq4}} \\
*+[r]{\Om^{\bs B}_n(\cL BG;BG)} \ar@{..>}[rrr]^(0.55){\Xi_{n,\sO}^{\bs B,G}} \ar[urr]^(0.55){\eq{fm2eq7}}_(0.6){\hat\xi^{\bs B}_n(BG)} &&& *+[l]{B=\Aut_{A\qs B}(\boo),\!} }\!\!\!
\end{gathered}
\label{fm9eq3}
\e
\begin{itemize}
\setlength{\itemsep}{0pt}
\setlength{\parsep}{0pt}
\item[] where\/ $\xi^{\bs B}_n(BG),$ $\hat\xi^{\bs B}_n(BG),$ $\Pi^{\bs B}_n(BG)$ are as in Definition\/ {\rm\ref{fm2def3}}. The top parallelogram commutes by \eq{fm4eq13}. The bottom left triangle commutes by \eq{fm2eq7}. Define\/ $\Xi_{n,\sO}^{\bs B,G}$ to be the unique morphism making the bottom right quadrilateral commute.

Then\/ $\sO$ is orientable for every compact\/ $n$-manifold $X$ with $\bs B$-structure in the sense of Definition\/ {\rm\ref{fm2def3}} if and only if\/ $\Xi_{n,\sO}^{\bs B,G}\equiv\ul{0}$.
\item[{\bf(b)}] Now let\/ $X$ be a compact\/ $n$-manifold with\/ $\bs B$-structure. Then $\sO$ is orientable for $X$ if and only if there does not exist a principal\/ $G$-bundle $Q\ra X\t\cS^1$ such that\/ $[X,Q]$ represents an element of\/ $\Om^{\bs B}_n(\cL BG;BG)\sm\Ker \Xi_{n,\sO}^{\bs B,G}$. This last condition is equivalent to $[X\t\cS^1,Q]$ representing an element of\/ $\Om_{n+1}^{\bs B}(BG)\sm \Ker\pi_1(\sO),$ using the isomorphism \eq{fm4eq4}. Here $\cS^1$ has the $\bs B$-structure induced from the standard\/ $\bs B$-structure on the closed unit disc $D^2\subset\R^2$ by identifying $S^1=\pd D^2,$ so for example the bounding spin structure when $\bs B$ is $\bs\Spin$. 
\end{itemize}

Analogues of\/ {\bf(a)\rm,\bf(b)} hold for orientation functors on the other bordism categories $\Bord_{n,k}^{\bs B}(MH),\Bord^{\bs B}_n(K(R,k)),\Bord^{\bs B}_n(T)_\top,$ as follows:
\begin{itemize}
\setlength{\itemsep}{0pt}
\setlength{\parsep}{0pt}
\item[{\bf(i)}] For $\Bord_{n,k}^{\bs B}(MH),$ the analogue of\/ \eq{fm9eq3} is
\end{itemize}
\begin{equation*}
\xymatrix@!0@C=100pt@R=35pt{
& \Aut_{\Bord_{n-1,k}^{\bs B}(\cL MH)}(\boo)  \ar[rr]^(0.42){I_{n,k}^{\bs B,H}}_(0.39){\eq{fm5eq10}} && *+[l]{\Aut_{\Bord_{n,k}^{\bs B}(MH)}(\boo)} \ar[dd]_(0.4){\sO(\boo)} \\
*+[r]{\Om^{\bs B}_n(\cL MH)} \ar[ur]^{\eq{fm5eq12}}_\cong \ar[d]^{\Pi^{\bs B}_n(MH)} \ar[rr]^(0.45){\xi^{\bs B}_n(MH)}_(0.4){\eq{fm2eq6}}  && \ti\Om_{n+1}^{\bs B}(MH)  \ar[ur]^(0.4){\eq{fm5eq5}} \\
*+[r]{\Om^{\bs B}_n(\cL MH;MH)} \ar@{..>}[rrr]^(0.6){\Xi_{n,k,\sO}^{\bs B,H}} \ar[urr]^(0.55){\eq{fm2eq7}}_(0.6){\hat\xi^{\bs B}_n(MH)} &&& *+[l]{B=\Aut_{A\qs B}(\boo).\!} }\!\!\!
\end{equation*}
\begin{itemize}
\setlength{\itemsep}{0pt}
\setlength{\parsep}{0pt}
\item[] The top parallelogram commutes by \eq{fm5eq13}. The bottom left triangle commutes by \eq{fm2eq7}.

Then in {\bf(a)\rm,} $\sO:\Bord_{n,k}^{\bs B}(MH)\ra A\qs B$ is orientable for every compact\/ $n$-manifold $X$ with $\bs B$-structure if and only if\/ $\Xi_{n,k,\sO}^{\bs B,H}\equiv\ul{0},$ and in {\bf(b)\rm,} $\sO$ is orientable for $X$ if and only if there does not exist a compact\/ $(n+1-k)$-submanifold\/ $M\subset X\t\cS^1$ with normal\/ $H$-structure such that\/ $[X,M]$ represents an element of\/ $\Om^{\bs B}_n(\cL MH;MH)\sm\Ker\Xi_{n,k,\sO}^{\bs B,H},$ or equivalently $[X\t\cS^1,M]$ represents an element of\/~$\Om_{n+1}^{\bs B}(MH)\sm \Ker\pi_1(\sO)$.
\item[{\bf(ii)}] For $\Bord^{\bs B}_n(K(R,k)),$ the analogue of\/ \eq{fm9eq3} is
\end{itemize}
\begin{equation*}
\xymatrix@!0@C=95pt@R=35pt{
& \Aut_{\Bord_{n-1}^{\bs B}(\cL K(R,k))}(\boo)  \ar[rr]^(0.41){I_n^{\bs B,K(R,k)}}_(0.44){\eq{fm6eq8}} && *+[l]{\Aut_{\Bord^{\bs B}_n(K(R,k))}(\boo)} \ar[dd]_(0.4){\sO(\boo)} \\
*+[r]{\Om^{\bs B}_n(\cL K(R,k))} \ar[ur]^{\eq{fm6eq10}}_\cong \ar[d]^{\Pi^{\bs B}_n(K(R,k))} \ar[rr]^(0.5){\xi^{\bs B}_n(K(R,k))}_(0.35){\eq{fm2eq6}}  && \ti\Om_{n+1}^{\bs B}(K(R,k))  \ar[ur]^(0.4){\eq{fm6eq5}} \\
*+[r]{\Om^{\bs B}_n(\cL K(R,k);K(R,k))} \ar@{..>}[rrr]^(0.65){\Xi_{n,\sO}^{\bs B,K(R,k)}} \ar[urr]^(0.55){\eq{fm2eq7}}_(0.6){\hat\xi^{\bs B}_n(K(R,k))} &&& *+[l]{B=\Aut_{A\qs B}(\boo).\!} }\!\!\!
\end{equation*}
\begin{itemize}
\setlength{\itemsep}{0pt}
\setlength{\parsep}{0pt}
\item[] The top parallelogram commutes by \eq{fm6eq11}. The bottom left triangle commutes by \eq{fm2eq7}.

Then in {\bf(a)\rm,} $\sO:\Bord^{\bs B}_n(K(R,k))\ra A\qs B$ is orientable for every compact\/ $n$-manifold $X$ with $\bs B$-structure if and only if\/ $\Xi_{n,\sO}^{\bs B,K(R,k)}\equiv\ul{0},$ and in {\bf(b)\rm,} $\sO$ is orientable for $X$ if and only if there does not exist a cohomology class $\ga\in H^k(X\t\cS^1,R)$ such that\/ $[X,\ga]$ represents an element of\/ $\Om^{\bs B}_n(\cL K(R,k);K(R,k))\sm\Ker\Xi_{n,\sO}^{\bs B,K(R,k)},$ or equivalently $[X\t\cS^1,\ga]$ represents an element of\/~$\Om_{n+1}^{\bs B}(K(R,k))\sm \Ker\pi_1(\sO)$.
\item[{\bf(iii)}] For $\Bord^{\bs B}_n(T)_\top,$ the analogue of\/ \eq{fm9eq3} is
\end{itemize}
\begin{equation*}
\xymatrix@!0@C=100pt@R=35pt{
& \Aut_{\Bord_{n-1}^{\bs B}(\cL T)_\top}(\boo)  \ar[rr]_(0.42){I_n^{\bs B,T}} && *+[l]{\Aut_{\Bord^{\bs B}_n(T)_\top}(\boo)} \ar[dd]_(0.4){\sO(\boo)} \\
*+[r]{\Om^{\bs B}_n(\cL T)} \ar[ur]^(0.4){\eq{fm7eq3}}_\cong \ar[d]^{\Pi^{\bs B}_n(T)} \ar[rr]^(0.45){\xi^{\bs B}_n(T)}_(0.4){\eq{fm2eq6}}  && \ti\Om_{n+1}^{\bs B}(T)  \ar[ur]^(0.4){\eq{fm7eq3}} \\
*+[r]{\Om^{\bs B}_n(\cL T;T)} \ar@{..>}[rrr]^(0.6){\Xi_{n,\sO}^{\bs B,T}} \ar[urr]^(0.55){\eq{fm2eq7}}_(0.6){\hat\xi^{\bs B}_n(T)} &&& *+[l]{B=\Aut_{A\qs B}(\boo).\!} }\!\!\!
\end{equation*}
\begin{itemize}
\setlength{\itemsep}{0pt}
\setlength{\parsep}{0pt}
\item[] Then in {\bf(a)\rm,} $\sO:\Bord^{\bs B}_n(T)_\top\ra A\qs B$ is orientable for every compact\/ $n$-manifold $X$ with $\bs B$-structure if and only if\/ $\Xi_{n,\sO}^{\bs B,T}\equiv\ul{0},$ and in {\bf(b)\rm,} $\sO$ is orientable for $X$ if and only if there does not exist a map $\phi:X\ra\cL T$ such that\/ $[X,\phi]$ represents an element of\/ $\Om^{\bs B}_n(\cL T;T)\sm\Ker\Xi_{n,\sO}^{\bs B,T},$ or equivalently there does not exist a map $\phi':X\t\cS^1\ra T$ such that\/ $[X\t\cS^1,\phi']$ represents an element of\/~$\Om_{n+1}^{\bs B}(T)\sm \Ker\pi_1(\sO)$.
\end{itemize}
\end{thm}

\begin{rem}
\label{fm9rem2}
Observe that Theorem \ref{fm9thm1}(a) and its analogues in (i)--(iii) gives an answer to Question \ref{fm9quest1}(a), and part (b) and its analogues in (i)--(iii) gives a necessary and sufficient answer to Question \ref{fm9quest1}(b). These answers to Question \ref{fm9quest1}(a) are extremely helpful: in any given problem we just have to compute the morphisms $\Xi_{n,\sO}^{\bs B,G},\Xi_{n,k,\sO}^{\bs B,H},\Xi_{n,\sO}^{\bs B,K(R,k)}$, which can often be done with enough work.

For $\Bord^{\bs B}_n(BG),\Bord_{n,k}^{\bs B}(MH)$, the answers to Question \ref{fm9quest1}(b) are not always useful: it is not easy to show the nonexistence of a principal $G$-bundle $Q\ra X\t\cS^1$ or an $(n+1-k)$-submanifold $M\subset X\t\cS^1$ satisfying given conditions. However, for $\Bord^{\bs B}_n(K(R,k))$, checking whether there exists a cohomology class $\ga\in H^k(X\t\cS^1,R)$ satisfying given conditions is much more feasible. So our favourite strategy for answering Question \ref{fm9quest1}(b) for $\Bord^{\bs B}_n(BG),\Bord_{n,k}^{\bs B}(MH)$ will be to reduce it to Question \ref{fm9quest1}(b) for $\Bord^{\bs B}_n(K(R,k))$ by factoring via a transfer functor as in Remark~\ref{fm9rem1}.
\end{rem}

\subsection{Elliptic operator bordism}
\label{fm92}

We briefly review the construction of the elliptic bordism category $\Bord_m^{\Ell_\ell}$ and the main result from the second author~\cite{Upme2}.

\begin{dfn}[{see \cite[Def.~2.1]{Upme2}}]
\label{fm9def3}
Let $\ell\in\N$ and write `$\equiv$' for equivalence modulo $8.$ A first order elliptic differential operator $D:\Ga^\iy(E_0)\ra\Ga^\iy(E_1)$ is {\it $\ell$-adapted} if the vector bundles $E_0, E_1$ have metric $\K_\ell$-linear structures, where the (skew) field $\K_\ell$ is defined according to Table \ref{fm9tab1}, and the following conditions hold.
\begin{itemize}
\setlength{\itemsep}{0pt}
\setlength{\parsep}{0pt}
\item
If $\ell\equiv 1,$ then $E_0=E_1$ and $D$ is $\R$-linear formally skew-adjoint, $D^*=-D.$
\item
If $\ell\equiv 2,$ then $E_0=\ov E_1$ and $D$ is $\C$-linear formally skew-adjoint, $D^*=-\ov{D}.$
\item
If $\ell\equiv 3,7,$ then $E_0=E_1$ and $D$ is $\K_\ell$-linear formally self-adjoint, $D^*=D.$
\item
If $\ell\equiv 5,$ then $E_0=E_1^\diamond$ and $D$ is $\H$-linear formally self-adjoint, $D^*=D^\diamond.$
\item
If $\ell\equiv 6,$ then $E_0=\ov E_1$ and $D$ is $\C$-linear formally self-adjoint, $D^*=\ov{D}.$
\item
If $\ell\equiv 0,4,$ then $D$ is $\K_\ell$-linear with no further conditions imposed.
\end{itemize}
\end{dfn}

\begin{table}[htb]
\centerline{\begin{tabular}{|l|c|c|c|c|c|c|c|c|c|c|c|}
\hline
$\bs\ell\equiv$ & $\bs 0$ & $\bs 1$ & $\bs 2$ & $\bs 3$ & $\bs 4$ & $\bs 5$ & $\bs 6$ & $\bs 7$\\ \hline
$\K_\ell$ & $\R$ & $\R$ & $\C$ & $\H$ & $\H$ & $\H$ & $\C$ & $\R$\\ \hline
$\Ga_\ell$ & $\Z$ & $\Z_2$ & $\Z_2$ & 0 & $\Z$ & 0 & 0 & 0\\ \hline
\end{tabular}}
\caption{\parbox[t]{270pt}{$\K_\ell$ is the natural base (skew) field of the real Clifford algebra $\Cl_{\ell-1}$ and $\Ga_\ell$ is the coefficient group $KO_\ell(\mathrm{pt})$ of  K-theory.}}
\label{fm9tab1}
\end{table}

For example, the {\it real Dirac operator} (meaning the positive Dirac operator $\sD_M^+$ if $\dim M\equiv 0,4$ and the skew-adjoint Dirac operator $\sD_M^\mathrm{skew}$ if $\dim M\equiv 1,2$) on a Riemannian spin manifold $M$  is $\ell$-adapted by \cite[\S 2.1]{Upme2} with~$\ell=\dim M.$

Recall from \cite[Def.~2.6]{Upme2} that for all $k,\ell\in\N$ we can define an {\it elliptic bordism category} $\Bord_k^{\Ell_\ell}.$ The objects are pairs $(M,D_M),$ where $M$ is a compact $k$-manifold without boundary and $D_M$ is an $\ell$-adapted elliptic differential operator on $M.$ A morphism $[N,D_N]$ from $M_0$ to $M_1$ is an equivalence class of pairs $(N,D_N)$ of a bordism $N$ with $\pd N=-M_0\amalg M_1$, equipped with an $(\ell+1)$-adapted elliptic differential operator $D_N$ on $N$ that restricts to a the {\it cylindrical form} $\Cyl(D_{M_0})\vert_{M_0\t[0,\varepsilon_0)}\amalg\Cyl(D_{M_1})\vert_{M_1\t(\varepsilon_1,1]}$ from \cite[Def.~2.5]{Upme2} over a collar neighbourhood $(M_0\t[0,\varepsilon_0))\amalg(M_1\t(\varepsilon_1,1])$ of the boundary of $N.$ The pair $(N,D_N),$ modulo higher bordism, is called an elliptic operator bordism.

Recall from \cite{Upme2} that the {\it orientation bundle} $\sO_\ell(\cD_M)$ of a family $\cD_M$ of $\ell$-adapted elliptic differential operators is the principal $\Z_2$-bundle $O(\DET\cD_M)$ of orientations of the determinant line bundle of $\cD_M$ if $\ell\equiv 0,$ the principal $\Z_2$-bundle $O(\PF\cD_M)$ of orientations of the Pfaffian line bundle of $\cD_M$ if $\ell\equiv 1,$ the principal $\Z$-bundle $\SP(\cD_M)$ of spectral enumerations if $\ell\equiv 3,7,$ and trivial otherwise. For a single  $\ell$-adapted elliptic differential operator $D_M,$ the orientation bundle over a point is simply called the {\it orientation torsor} $\sO_\ell(D_M)$ and is a $\Ga_\ell$-graded $\Ga_{\ell+1}$-torsor placed in the degree of the real index $\ind_\ell D_M\in\Ga_\ell.$ Note that the word `orientation' is used loosely here and includes also orientations of the Pfaffian line bundle and Floer gradings.

\begin{thm}[{see \cite[Th.~3.1]{Upme2}}]
\label{fm9thm2}
Every\/ $(\ell+1)$-adapted elliptic operator bordism\/ $[N,D_N]: (M_0,D_{M_0})\ra(M_1,D_{M_1})$ of\/ $\ell$-adapted elliptic differential operators on compact\/ $n$-manifolds\/ $M_0,$ $M_1$ without boundary induces an isomorphism of graded orientation torsors,\/ $\sO_\ell[N,D_N]: \sO_\ell(D_{M_0})\ra\sO_\ell(D_{M_1}),$ which depends on\/ $(N,D_N)$ only up to bordism and is continuous in families. This construction is functorial and compatible with disjoint unions, so determines a symmetric monoidal functor
\begin{equation*}
\sO_\ell:\Bord_n^{\Ell_\ell}\longra\KOtor
\end{equation*}
into the category of\/ $\Ga_\ell$-graded\/ $\Ga_{\ell+1}$-torsors.

The grading of\/ $\sO_\ell(M,D_M)$ is given by the real index\/ $\ind_\ell(D_M)\in\Ga_\ell.$ If\/ $[N,D_N]$ is an elliptic operator bordism between empty manifolds, then\/ $\sO_\ell(\es,\es)=\Ga_{\ell+1}$ and the induced automorphism corresponds to the real index\/ $\ind_{\ell+1}(D_N)\in\Ga_{\ell+1}$ under the natural isomorphism\/ $\Aut(\Ga_{\ell+1})\cong\Ga_{\ell+1}.$
\end{thm}

\subsection{Analytic orientation functors in gauge theory}
\label{fm93}

We recall material from the second author~\cite{Upme2}.

\begin{dfn}
\label{fm9def4}
Let $n\ge 0$, and $\bs B$ be a tangential structure factoring through $\bs\Spin$, and $G$ be a Lie group. We will define an orientation functor
\begin{equation*}
\sO_n^{\bs B,G}:\Bord^{\bs B}_n(BG)\longra\Ga_n\qs\Ga_{n+1},
\end{equation*}
following the second author \cite[\S 3.2]{Upme2}. Here $\Bord^{\bs B}_n(BG)$ and $\Ga_n\qs\Ga_{n+1}$ are the Picard groupoids from \S\ref{fm41} and \S\ref{fm92}, with $\Ga_n,\Ga_{n+1}$ as in Table~\ref{fm9tab1}.

Let $(X,P)$ be an object in $\Bord^{\bs B}_n(BG)$. Write $\A_{X,P}$ for the infinite-dim\-en\-sion\-al moduli space of pairs $(g_X,\nabla_P)$, where $g_X$ is a Riemannian metric on $X$ and $\nabla_P$ is a connection on $P\ra X$. Note that $\A_{X,P}$ is contractible, as the moduli space of metrics $g_X$ is an infinite-dimensional open convex cone in an affine space $\Ga^\iy(S^2T^*X)$, and the moduli space of connections $\nabla_P$ is an infinite-dimensional affine space modelled on~$\Ga^\iy(T^*X\ot\Ad(P))$.

For $(g_X,\nabla_P)\in\A_{X,P}$, as $\bs B$ factors through $\bs\Spin$, the metric $g_X$ induces a real Dirac operator  $\sD_X:\Ga^\iy(E_0)\ra\Ga^\iy(E_1)$ on the compact spin Riemannian $n$-manifold $(X,g_X)$ as in Definition \ref{fm9def3} (note that the definition depends on $n\mod 8$). Write $\sD_X^{\nabla_P}:\Ga^\iy(E_0\ot\Ad(P))\ra\Ga^\iy(E_1\ot\Ad(P))$ for $\sD_X$ twisted by the real vector bundle $\Ad(P)=(P\t\g)/G$ with connection induced by $\nabla_P$. Then $\sD_X^{\nabla_P}$ is an $n$-adapted elliptic operator on $X$, as in Definition \ref{fm9def3}. Hence the orientation torsor $\sO_n(\sD_X^{\nabla_P})$ is a $\Ga_n$-graded $\Ga_{n+1}$-torsor.

Now $\sO_n(\sD_X^{\nabla_P})$ depends continuously on $(g_X,\nabla_P)\in\A_{X,P}$. As $\A_{X,P}$ is contractible and $\Ga_n,\Ga_{n+1}$ are discrete, this means that the grading in $\Ga_n$ is constant on $\A_{X,P}$, and the $\Ga_{n+1}$-torsors form a principal $\Ga_{n+1}$-bundle $R_{X,P}\ra\A_{X,P}$ on $\A_{X,P}$. Define a $\Ga_n$-graded $\Ga_{n+1}$-torsor $\sO_n^{\bs B,G}(X,P)$ to have the $\Ga_n$-grading of $\sO_n(\sD_X^{\nabla_P})$ for any $(g_X,\nabla_P)\in\A_{X,P}$, which is independent of $(g_X,\nabla_P)$, and to have $\Ga_{n+1}$-torsor the set of constant sections of $R_{X,P}\ra\A_{X,P}$, which is a $\Ga_{n+1}$-torsor as $R_{X,P}$ is a principal $\Ga_{n+1}$-bundle and $\A_{X,P}$ is contractible. This defines the functor $\sO_n^{\bs B,G}$ on objects $(X,P)$ in~$\Bord^{\bs B}_n(BG)$.

Next let $[Y,Q]:(X_0,P_0)\ra(X_1,P_1)$ be a morphism in $\Bord^{\bs B}_n(BG)$. For a representative $(Y,Q)$ for $[Y,Q]$, write $\A'_{Y,Q}$ for the contractible infinite-dim\-en\-sion\-al moduli space of pairs $(g_Y,\nabla_Q)$, where $g_Y$ is a Riemannian metric on $Y$ and $\nabla_Q$ is a connection on $Q\ra X$. There are boundary restriction morphisms $\A'_{Y,Q}\ra\A_{X_0,P_0}$ and $\A'_{Y,Q}\ra\A_{X_1,P_1}$ mapping $(g_Y,\nabla_Q)\mapsto(g_{X_0},\nabla_{P_0})=(g_Y,\nabla_Q)\vert_{X_0}$ and $(g_Y,\nabla_Q)\mapsto(g_{X_1},\nabla_{P_1})=(g_Y,\nabla_Q)\vert_{X_1}$. For $(g_Y,\nabla_Q)\in\A'_{Y,Q}$, write $\sD_Y:\Ga^\iy(E'_0)\ra\Ga^\iy(E'_1)$ for the real Dirac operator on the compact spin Riemannian $(n+1)$-manifold $(Y,g_Y)$ and $\sD_Y^{\nabla_Q}:\Ga^\iy(E'_0\ot\Ad(Q))\ra\Ga^\iy(E'_1\ot\Ad(Q))$ for $\sD_Y$ twisted by $\nabla_Q$. Then 
\begin{equation*}
[Y,\sD_Y^{\nabla_Q}]:(X_0,\sD_{X_0}^{\nabla_{P_0}})\longra(X_1,\sD_{X_1}^{\nabla_{P_1}})
\end{equation*}
is an $(n+1)$-adapted elliptic operator bordism as in \S\ref{fm93}, and so by Theorem \ref{fm9thm2} induces an isomorphism $\sO_n[Y,\sD_Y^{\nabla_Q}]:\sO_n(X_0,\sD_{X_0}^{\nabla_{P_0}})\ra\sO_n(X_1,\sD_{X_1}^{\nabla_{P_1}})$ of $\Ga_n$-graded $\Ga_{n+1}$-torsors. As this depends continuously on $(g_Y,\nabla_Q)$ in the contractible space $\A'_{Y,Q}$, we may pass to global constant sections on $\A_{X_0,P_0},\ab\A_{X_0,P_0},\ab\A'_{Y,Q}$ to obtain an isomorphism $\sO_n^{\bs B,G}([Y,Q]):\sO_n^{\bs B,G}(X_0,P_0)\ra\sO_n^{\bs B,G}(X_0,P_0)$ of $\Ga_n$-graded $\Ga_{n+1}$-torsors. Theorem \ref{fm9thm2} implies that this depends on $(Y,Q)$ only up to spin bordism, and hence only on the morphism $[Y,Q]$ in $\Bord^{\bs B}_n(BG)$. This defines $\sO_n^{\bs B,G}$ on morphisms $[Y,Q]$ in $\Bord^{\bs B}_n(BG)$.

It follows from $\sO_n$ a symmetric monoidal functor in Theorem \ref{fm9thm2} that $\sO_n^{\bs B,G}$ is a symmetric monoidal functor. It is an {\it analytic orientation functor}.

Following \cite[Def.~1.2]{JTU}, we also define the {\it normalized orientation functor\/}
\begin{equation*}
\sN_n^{\bs B,G}:\Bord^{\bs B}_n(BG)\longra\Ga_n\qs\Ga_{n+1},
\end{equation*}
by $\sN_n^{\bs B,G}(X,P)\!=\!\sO_n^{\bs B,G}(X,P)\ot\sO_n^{\bs B,G}(X,X\t G)^{-1}$ on objects, and $\sN_n^{\bs B,G}([Y,Q])\ab=\sO_n^{\bs B,G}([Y,Q])\ot \sO_n^{\bs B,G}([Y,Y\t G])^{-1}$ on morphisms $[Y,Q]:(X_0,P_0)\ra (X_1,P_1)$, where $X\t G\ra X$ and $Y\t G\ra Y$ are the trivial principal $G$-bundles, and we use the Picard groupoid structure on $\Ga_n\qs\Ga_{n+1}$. Normalized orientations (or n-orientations for short) are more convenient for some purposes.

Note in particular from Table \ref{fm9tab1} that when $n=7$, $\sO_7^{\bs B,G},\sN_7^{\bs B,G}$ map to $0\qs\Z=\ZZtor$, the Picard groupoid of $\Z$-torsors. We write $\sO_{7,\Z_k}^{\bs B,G},\sN_{7,\Z_k}^{\bs B,G}$ for the compositions of $\sO_7^{\bs B,G},\sN_7^{\bs B,G}$ with the natural symmetric monoidal functor $\ZZtor\ra\Z_k$-tor of reduction mod $k$ for $k\ge 2$. Then $\sO_{7,\Z_2}^{\bs B,G},\sN_{7,\Z_2}^{\bs B,G}$ are important for orientations of moduli spaces of $G_2$-instantons on $G_2$-manifolds, and $\sO_{7,\Z_k}^{\bs B,G},\sN_{7,\Z_k}^{\bs B,G}$ are important for the grading mod $k$ of a conjectural Floer theory based on $G_2$-instantons, in the spirit of~\cite{DoTh,DoSe}.

When $n=8$, $\sO_8^{\bs B,G},\sN_8^{\bs B,G}$ map to $\Z\qs\Z_2$, the Picard groupoid of $\Z$-graded $\Z_2$-torsors. For $k\ge 1$ we write $\sO_{8,\Z_{2k}}^{\bs B,G},\sN_{8,\Z_{2k}}^{\bs B,G}:\Bord^{\bs B}_n(BG)\ra\Z_{2k}\qs\Z_2$ for the composition of $\sO_8^{\bs B,G},\sN_8^{\bs B,G}$ with the projection $\Z\qs\Z_2\ra\Z_{2k}\qs\Z_2$ which reduces gradings mod $2k$. Then $\sO_{8,\Z_2}^{\bs B,G},\sN_{8,\Z_2}^{\bs B,G}$ are important for orientations of moduli spaces of $\Spin(7)$-instantons on $\Spin(7)$-manifolds, and of coherent sheaves on Calabi--Yau 4-folds. 

\end{dfn}

\subsection{Analytic orientation functors in submanifold bordism}
\label{fm94}

Next we define a class of analytic orientation functors for the submanifold bordism categories of \S\ref{fm51}. They are chosen for their relevance to orientations of moduli spaces of associative 3-folds in $G_2$-manifolds when $n=7$, and to moduli spaces of Cayley 4-folds in $\Spin(7)$-manifolds when $n=8$, as in \S\ref{fm14}, although the definition makes sense for all~$n\ge 4$.

\begin{dfn}
\label{fm9def5}
Let $(X,g_X)$ be an oriented, spin Riemannian $n$-manifold for $n\ge 4$, possibly with corners, and let $M\subset X$ be a compact, oriented, neat submanifold of codimension $4.$ Locally, we can choose a spin structure on $M$ and construct the $(n-4)$-adapted real Dirac operator $\Ga^\iy(E_0)\ra \Ga^\iy(E_1)$ on the spinor bundles of $M,$ linear over $\mathbb{L}=\K_{n-4}$ in Table \ref{fm9tab1}. The local spin structure on $M$ induces also a 2-out-of-3 spin structure on the normal bundle of $M$ and thus a pair $\Si^\pm_\nu\ra M$ of spinor quaternionic line bundles with natural Levi-Civita connections. Thus, as in \cite[Def.~3.4]{Upme2} we can twist the real Dirac operators over $\mathbb{L}$ to form the {\it Fueter operator\/} of~$M\subset X,$
\e
F_{g_X,M}^\pm\colon\Ga^\iy(E_0\ot_{\mathbb{L}}\Si^\pm_\nu)\longra\Ga^\iy(E_1\ot_{\mathbb{L}}\Si^\pm_\nu).
\label{fm9eq4}
\e
It is an $n$-adapted elliptic operator on $M$. This generalizes the Fueter operators in Donaldson--Segal \cite[\S 6]{DoSe} for $n=7,8$, reviewed in \S\ref{fm14} below, to all~$n\ge 4$.
\end{dfn}

The next result is a simple case-by-case verification.

\begin{lem}
\label{fm9lem1}
For every oriented, neat\/ $(n-4)$-submanifold\/ $M\subset X$ of a Riemannian spin\/ $n$-manifold\/ $(X,g_X)$ the Fueter operators\/ $F_M^\pm$ are\/ $n$-adapted elliptic differential operators, independent of the choice of local spin structure on\/ $M$. Moreover, if\/ $(X,g_X)=(\pd Y,g_Y\vert_{\pd Y})$ and\/ $M=\pd N$ for a neat\/ $(n-3)$-submanifold\/ $N\subset Y$ of a Riemannian spin\/ $(n+1)$-manifold\/ $(Y,g_Y),$ then\/ $F_N^\pm$ is isomorphic to\/ $\Cyl F_M^\pm$ on a collar neighbourhood of the boundary, as in\/~{\rm\S\ref{fm92}}.
\end{lem}

\begin{dfn}
\label{fm9def6}
Let $n\ge 4,$ $\bs B$ be a tangential structure factoring through $\bs\Spin,$ and $\rho:H\ra\SO(4)\subset\O(4)$ be a Lie group morphism. We will define orientation functors
\begin{equation*}
\sO_{n,4}^{\bs B,H,+},\sO_{n,4}^{\bs B,H,-}:\Bord_{n,4}^{\bs B}(MH)\longra\Ga_n\qs\Ga_{n+1}.
\end{equation*}
Here $\Bord_{n,4}^{\bs B}(MH)$ and $\Ga_n\qs\Ga_{n+1}$ are the Picard groupoids from \S\ref{fm51} and \S\ref{fm92}, with $\Ga_n,\Ga_{n+1}$ as in Table~\ref{fm9tab1}.

Let $(X,M)$ be an object in $\Bord_{n,4}^{\bs B}(MH)$. Write $\A_X$ for the contractible moduli space of Riemannian metrics $g_X$ on $X$. For $g_X\in\A_X$, the metric $g_X$ induces $n$-adapted Fueter operators $F_{g_X,M}^\pm$ on $M$ as in \eq{fm9eq4}. Hence the orientation torsor $\sO_n(F_{g_X,M}^\pm)$ from Definition \ref{fm9def3} is a $\Ga_n$-graded $\Ga_{n+1}$-torsor.

Now $\sO_n(F_{g_X,M}^\pm)$ depends continuously on $g_X\in\A_X$. As $\A_X$ is contractible and $\Ga_n,\Ga_{n+1}$ are discrete, this means that the grading in $\Ga_n$ is constant on $\A_X$, and the $\Ga_{n+1}$-torsors form a principal $\Ga_{n+1}$-bundle $R^\pm_{X,M}\ra\A_X$ on $\A_X$. Define a $\Ga_n$-graded $\Ga_{n+1}$-torsor $\sO_{n,4}^{\bs B,H,\pm}(X,M)$ to have the $\Ga_n$-grading of $\sO_n(F_{g_X,M}^\pm)$ for any $g_X\in\A_X$, which is independent of $g_X$, and to have $\Ga_{n+1}$-torsor the set of constant sections of $R^\pm_{X,M}\ra\A_X$, which is a $\Ga_{n+1}$-torsor as $R^\pm_{X,M}$ is a principal $\Ga_{n+1}$-bundle and $\A_X$ is contractible. This defines the functors $\sO_{n,4}^{\bs B,H,\pm}$ on objects $(X,M)$ in~$\Bord_{n,4}^{\bs B}(MH)$.

Next let $[Y,N]:(X_0,M_0)\ra(X_1,M_1)$ be a morphism in $\Bord_{n,4}^{\bs B}(MH)$. For a representative $(Y,N)$ for $[Y,N]$, write $\A'_Y$ for the contractible moduli space of Riemannian metrics $g_Y$ on $Y$. There are boundary restriction morphisms $\A'_Y\ra\A_{X_0}$ and $\A'_Y\ra\A_{X_1}$ mapping $g_Y\mapsto g_{X_0}=g_Y\vert_{X_0}$ and $g_Y\mapsto g_{X_1}=g_Y\vert_{X_1}$. For $g_Y\in\A'_Y$, write $F_{g_Y,N}^\pm$ for the Fueter operators on $N$. Then 
\begin{equation*}
[N,F_{g_Y,N}^\pm]:(M_0,F_{g_{X_0},M_0}^\pm)\longra(M_1,F_{g_{X_1},M_1}^\pm)
\end{equation*}
is an $(n+1)$-adapted elliptic operator bordism as in \S\ref{fm92}, and so by Theorem \ref{fm9thm2} induces an isomorphism $\sO_n[N,F_{g_Y,N}^\pm]:\sO_n(M_0,F_{g_{X_0},M_0}^\pm)\ra\sO_n(M_1,F_{g_{X_1},M_1}^\pm)$ of $\Ga_n$-graded $\Ga_{n+1}$-torsors. As this depends continuously on $g_Y$ in the contractible space $\A'_Y$, we may pass to global constant sections on $\A_{X_0},\ab\A_{X_0},\ab\A'_Y$ to obtain an isomorphism $\sO_{n,4}^{\bs B,H,\pm}([Y,N]):\sO_{n,4}^{\bs B,H,\pm}(X_0,M_0)\ra\sO_{n,4}^{\bs B,H,\pm}(X_0,M_0)$ of $\Ga_n$-graded $\Ga_{n+1}$-torsors. Theorem \ref{fm9thm2} implies that this depends on $(Y,N)$ only up to spin bordism, and hence only on the morphism $[Y,N]$ in $\Bord_{n,4}^{\bs B}(MH)$. This defines $\sO_{n,4}^{\bs B,H,\pm}$ on morphisms $[Y,Q]$ in $\Bord_{n,4}^{\bs B}(MH)$.

It follows from $\sO_n$ a symmetric monoidal functor in Theorem \ref{fm9thm2} that $\sO_{n,4}^{\bs B,H,\pm}$ is a symmetric monoidal functor. It is an {\it analytic orientation functor}.

We also form the {\it normalized orientation functor\/} $\sO_{n,4}^{\bs B,H,0}:\Bord_{n,4}^{\bs B}(MH) \ab\ra\Ga_n\qs\Ga_{n+1}$ by mapping the object $(X,M)$ to the $\Ga_{n+1}$-torsor
\begin{equation*}
\sO_{n,4}^{\bs B,H,0}(X,M)=\Hom_{\Ga_{n+1}}\bigl(\sO_{n,4}^{\bs B,H,+}(X,M),\sO_{n,4}^{\bs B,H,-}(X,M)\bigr),
\end{equation*}
placed in degree $\deg\sO_{n,4}^{\bs B,H,-}(X,M)-\deg\sO_{n,4}^{\bs B,H,+}(X,M)\in \Ga_n$, and similarly on morphisms. That is, $\sO_{n,4}^{\bs B,H,0}=(\sO_{n,4}^{\bs B,H,+})^{-1}\ot\sO_{n,4}^{\bs B,H,-}$, using the Picard groupoid structure on $\Ga_n\qs\Ga_{n+1}$.

Note in particular from Table \ref{fm9tab1} that when $n=7$, $\sO_{7,4}^{\bs B,H,*}$ maps to $0\qs\Z=\ZZtor$, the Picard groupoid of $\Z$-torsors. We write $\sO_{7,4,\Z_k}^{\bs B,H,*}$ for the composition of $\sO_{7,4}^{\bs B,H,*}$ with the natural symmetric monoidal functor $\ZZtor\ra\Z_k$-tor of reduction mod $k$ for $k\ge 2$. Then $\sO_{7,4,\Z_2}^{\bs B,H,*}$ is important for orientations of associative 3-folds in $G_2$-manifolds, as in~\S\ref{fm142}.

Similarly, when $n=8$, $\sO_{8,4}^{\bs B,H,*}$ maps to $\Z\qs\Z_2$. For $k\ge 1$ we write $\sO_{8,4,\Z_{2k}}^{\bs B,H,*}$ for the composition of $\sO_{8,4}^{\bs B,H,*}$ with the projection $\Z\qs\Z_2\ra\Z_{2k}\qs\Z_2$ reducing gradings mod $2k$. Then $\sO_{8,4,\Z_2}^{\bs B,H,*}$ is important for orientations of Cayley 4-folds in $\Spin(7)$-manifolds, as in~\S\ref{fm143}.
\end{dfn}

\subsection{Abstract orientation functors in cohomology bordism}
\label{fm95}

We define orientation functors $\sH_7^\Z,\sH_8^\Z,\sH_7^{\Z_2},\sH_8^{\Z_2}$ which will be very important for flag structures in \S\ref{fm10}, and our study of orientations and orientability in~\S\ref{fm11}.

\begin{dfn}
\label{fm9def7}
We will define abstract orientation functors
\e
\begin{split}
\sH_7^\Z&:\Bord^{\bs\Spin}_7(K(\Z,4))\longra 0\qs\Z_2,\\
\sH_8^\Z&:\Bord^{\bs\Spin}_8(K(\Z,4))\longra \Z_2\qs\Z_2,\\
\sH_7^{\Z_2}&:\Bord^{\bs\Spin}_7(K(\Z_2,4))\longra 0\qs\Z_2,\\
\sH_8^{\Z_2}&:\Bord^{\bs\Spin}_8(K(\Z_2,4))\longra \Z_2\qs\Z_2,
\end{split}
\label{fm9eq5}
\e
by specifying the data that classifies them in Theorem \ref{fmAthm2}. Here for the Picard groupoid $\Z_2\qs\Z_2$, the linear quadratic form $q:\Z_2\ra\Z_2$ in Theorem \ref{fmAthm2}(a) is $\id:\Z_2\ra\Z_2$. By Proposition \ref{fm6prop1}, for $n=7,8$ and $i=0,1$ we have
\begin{align*}
\pi_i(\Bord^{\bs\Spin}_n(K(\Z,4)))&\cong \Om_{n+i}^{\bs\Spin}(K(\Z,4))\cong\Om_{n+i}^{\bs\Spin}(*)\op\ti\Om_{n+i}^{\bs\Spin}(K(\Z,4)),\\
\pi_i(\Bord^{\bs\Spin}_n(K(\Z_2,4)))&\cong \Om_{n+i}^{\bs\Spin}(K(\Z_2,4))\cong\Om_{n+i}^{\bs\Spin}(*)\op\ti\Om_{n+i}^{\bs\Spin}(K(\Z_2,4)),
\end{align*}
where $\Om_{n+i}^{\bs\Spin}(*)$ is given in Tables \ref{fm2tab1}--\ref{fm2tab2} and $\ti\Om_{n+i}^{\bs\Spin}(K(\Z,4)),\ti\Om_{n+i}^{\bs\Spin}(K(\Z_2,4))$ in Table \ref{fm3tab1}. We define $\pi_i(\sH_n^\Z)$ and $\pi_i(\sH_n^{\Z_2})$ for $n=7,8$ and $i=0,1$ by
\e
\begin{gathered}
\pi_i(\sH_n^\Z)\vert_{\Om_{n+i}^{\bs\Spin}(*)}\equiv \pi_i(\sH_n^{\Z_2})\vert_{\Om_{n+i}^{\bs\Spin}(*)}\equiv 0, \quad \pi_0(\sH_7^\Z)\equiv\pi_0(\sH_7^{\Z_2})\equiv 0, \\
\pi_1(\sH_7^\Z)\vert_{\ti\Om_8^{\bs\Spin}(K(\Z,4))}=\pi_0(\sH_8^\Z)\vert_{\ti\Om_8^{\bs\Spin}(K(\Z,4))}\quad\text{maps}\quad  \ze_2\mapsto \ul{1},\quad \ze_3\mapsto \ul{0},\\
\pi_1(\sH_8^\Z)\vert_{\ti\Om_8^{\bs\Spin}(K(\Z,4))}\quad\text{maps}\quad \al_1\ze_2\mapsto \ul{1}, \\
\pi_1(\sH_7^{\Z_2})\vert_{\ti\Om_8^{\bs\Spin}(K(\Z_2,4))}=\pi_0(\sH_8^{\Z_2})\vert_{\ti\Om_8^{\bs\Spin}(K(\Z_2,4))}\quad\text{maps}\quad  \ze_2\mapsto \ul{1},\\
\pi_1(\sH_8^{\Z_2})\vert_{\ti\Om_8^{\bs\Spin}(K(\Z_2,4))}\quad\text{maps}\quad \al_1\ze_2\mapsto \ul{1}.
\end{gathered}
\label{fm9eq6}
\e

These satisfy the condition $q'\ci f_0=f_1\ci q$ in Theorem \ref{fmAthm2}(b), trivially for $\sH_7^\Z,\sH_7^{\Z_2}$, and as $q:\ze_2\mapsto\al_1\ze_2$, $q:\ze_3\mapsto 0$ for $\sH_8^\Z,\sH_8^{\Z_2}$. Hence there exist symmetric monoidal functors \eq{fm9eq5} with these invariants $\pi_i(\sH_n^\Z),\pi_i(\sH_n^{\Z_2})$ by Theorem \ref{fmAthm2}(b). Theorem \ref{fmAthm2}(c) and Example \ref{fmAex6} say that the sets of such functors $\sH_7^\Z,\sH_8^\Z,\sH_7^{\Z_2},\sH_8^{\Z_2}$ modulo monoidal natural isomorphism are torsors over $H^2_\sym(0,\Z_2)=0$ and $H^2_\sym(\Z^4,\Z_2)=0$ and $H^2_\sym(0,\Z_2)=0$ and $H^2_\sym(\Z^2\op\Z_4,\Z_2)=\Z_2$ respectively. Hence $\sH_7^\Z,\sH_8^\Z,\sH_7^{\Z_2}$ are unique up to monoidal natural isomorphism. 

There are two choices for $\sH_8^{\Z_2}$ up to monoidal natural isomorphism, and we choose one arbitrarily. By Lemma \ref{fmAlem1}(b), $\sH_8^{\Z_2}$ is unique up to {\it non-monoidal\/} natural isomorphism. This will be all we care about, as we will only be interested $\sH_8^{\Z_2}$ in relation to orientations for $\sH_8^{\Z_2}$ on a compact spin 8-manifold $X$, and these do not use the monoidal structure.

By the uniqueness of $\sH_7^\Z,\sH_8^\Z$, there exist monoidal natural isomorphisms $\eta_7,\eta_8$ in the diagrams
\e
\begin{gathered}
\xymatrix@C=30pt@R=17pt{
*+[r]{\Bord^{\bs\Spin}_7(K(\Z,4))} \drtwocell_{}\omit_{}\omit{\,\,\eta_7} \ar[d]^{F_{K(\Z,4)}^{K(\Z_2,4)}}  \ar@/^1.5pc/[drr]^(0.7){\sH_7^\Z} \\
*+[r]{\Bord^{\bs\Spin}_7(K(\Z_2,4))} \ar[rr]^(0.7){\sH_7^{\Z_2}} && *+[l]{0\qs\Z_2,} }
\;\>
\xymatrix@C=30pt@R=17pt{
*+[r]{\Bord^{\bs\Spin}_8(K(\Z,4))} \drtwocell_{}\omit_{}\omit{\,\,\eta_8} \ar[d]^{F_{K(\Z,4)}^{K(\Z_2,4)}}  \ar@/^1.5pc/[drr]^(0.7){\sH_8^\Z} \\
*+[r]{\Bord^{\bs\Spin}_8(K(\Z_2,4))} \ar[rr]^(0.7){\sH_8^{\Z_2}} && *+[l]{\Z_2\qs\Z_2,} }\!\!\!\!
\end{gathered}
\label{fm9eq7}
\e
where $F_{K(\Z,4)}^{K(\Z_2,4)}$ are the topological transfer functors induced by the obvious map $K(\Z,4)\ra K(\Z_2,4)$. Here by Theorem \ref{fmAthm2}(d), $\eta_7$ is unique, and $\eta_8$ lies in a torsor over $\Hom(\Z^4,\Z_2)=\Z_2^4$. We choose $\eta_8$ arbitrarily in this torsor.

There is a natural functor $\inc:\Bord^{\bs\Spin}_n(*)\ra\Bord^{\bs\Spin}_n(K(\Z,4))$ for $n=7,8$ mapping $X\mapsto (X,0)$ on objects and $[Y]\mapsto[Y,0]$ on morphisms. In a similar way to \eq{fm9eq7}, since $\pi_i(\sH_n^\Z)\vert_{\Om_{n+i}^{\bs\Spin}(*)}\equiv 0$, there exist unique monoidal natural isomorphisms $\ze_7,\ze_8$ in the diagrams
\e
\begin{gathered}
\xymatrix@C=30pt@R=17pt{
*+[r]{\Bord^{\bs\Spin}_7(*)} \drtwocell_{}\omit_{}\omit{\,\,\ze_7} \ar[d]^{\inc}  \ar@/^1.5pc/[drr]^(0.7){\boo} \\
*+[r]{\Bord^{\bs\Spin}_7(K(\Z,4))} \ar[rr]^(0.7){\sH_7^\Z} && *+[l]{0\qs\Z_2,} }
\;\>
\xymatrix@C=30pt@R=17pt{
*+[r]{\Bord^{\bs\Spin}_8(*)} \drtwocell_{}\omit_{}\omit{\,\,\ze_8} \ar[d]^{\inc}  \ar@/^1.5pc/[drr]^(0.7){\boo} \\
*+[r]{\Bord^{\bs\Spin}_8(K(\Z,4))} \ar[rr]^(0.7){\sH_8^\Z} && *+[l]{0\qs\Z_2.} }\!\!\!\!
\end{gathered}
\label{fm9eq8}
\e
\end{dfn}

\section{Flag structures}
\label{fm10}

{\it Flag structures\/} in 7 dimensions were introduced in the first author \cite[\S 3.1]{Joyc5} and used to construct orientations on moduli spaces of associative 3-folds in $G_2$-manifolds in \cite[\S 3.2]{Joyc5}, and also to construct orientations on moduli spaces of $G_2$-instantons on $G_2$-manifolds in the authors~\cite{JoUp}. 

\subsection{Flag structures in 7 dimensions}
\label{fm101}

We recall the following from~\cite[\S 3.1]{Joyc5}.

\begin{dfn}
\label{fm10def1}
Let $X$ be an oriented 7-manifold, and consider pairs $(N,s)$ of a compact, oriented $3$-submanifold $N\subset X$, and a non-vanishing section $s$ of the normal bundle $\nu_N$ of $N$ in $X$.
We call $(N,s)$ a {\it flagged submanifold\/} in~$X$.

For non-vanishing sections $s,s'$ of $\nu_N$ define
\e
d(s,s') \coloneqq N \bullet \bigl\{ t\cdot s(y) + (1-t)\cdot s'(y) \enskip\big|\enskip t\in [0,1],\enskip y\in N\bigr\} \in \Z,
\label{fm10eq1}
\e
using the intersection product `$\bu$' between a $3$-cycle and a $4$-chain
whose boundary does not meet the cycle, where we identify $N$ with the zero section in $\nu_N$. For all non-vanishing sections $s,s',s''$ of $\nu_N$, this satisfies
\e
d(s,s'')=d(s,s')+d(s',s'').
\label{fm10eq2}
\e

Let $(N_0,s_0),(N_1,s_1)$ be disjoint flagged submanifolds with $[N_0]=[N_1]$ in $H_3(X,\Z)$. Choose an integral 4-chain $C$ with $\pd C=N_1-N_0$. Let $N_0',N_1'$ be small perturbations of $N_0,N_1$ in the normal directions $s_0,s_1$. Then $N_0'\cap N_0=N_1'\cap N_1=\es$ as $s_0,s_1$ are non-vanishing, and $N_0'\cap N_1=N_1'\cap N_0=\es$ as $N_0,N_1$ are disjoint and $N_0',N_1'$ are close to $N_0,N_1$. Define $D((N_0,s_0),(N_1,s_1))$ to be the intersection number $(N_1'-N_0')\bu C$ in homology over $\Z$. Here we regard
\begin{equation*}
[C] \in H_4(X,N_0\cup N_1,\Z),\qquad
[N_0'], [N_1'] \in H_3(X\sm (N_0\cup N_1),\Z).
\end{equation*}
Note that since $N_0', N_1'$ are small perturbations and $N_0, N_1$ are disjoint we have $(N_0 \cup N_1) \cap (N_0' \cup N_1') = \emptyset$.
This is independent of the choices of $C$ and~$N_0',N_1'$.
\end{dfn}

In \cite[Prop.s~3.3 \& 3.4]{Joyc5} we show that if $(N_0,s_0),(N_1,s_1),\ab(N_2,s_2)$ are disjoint flagged submanifolds with $[N_0]=[N_1]=[N_2]$ in $H_3(X,\Z)$  then
\e
\label{fm10eq3}
\begin{split}
D((N_0,s_0),(N_2,s_2))&\equiv D((N_0,s_0),(N_1,s_1))\\
&\qquad +D((N_1,s_1),(N_2,s_2))\mod 2,
\end{split}
\e
and if $(N',s')$ is any small deformation of $(N,s)$ with $N,N'$ disjoint then
\e
D((N,s),(N',s'))\equiv 0\mod 2.
\label{fm10eq4}
\e

\begin{dfn}
\label{fm10def2}
A {\it flag structure\/} on $X$ is a map
\e
F:\bigl\{\text{flagged submanifolds $(N,s)$ in $X$}\bigr\}\longra\{\pm 1\},
\label{fm10eq5}
\e
satisfying:
\begin{itemize}
\setlength{\itemsep}{0pt}
\setlength{\parsep}{0pt}
\item[(i)] $F(N,s) = F(N,s')\cdot (-1)^{d(s,s')}$.
\item[(ii)] If $(N_0,s_0),(N_1,s_1)$ are disjoint flagged submanifolds in $X$ with $[N_0]=[N_1]$ in $H_3(X,\Z)$ then
\begin{equation*}
F(N_1,s_1)=F(N_0,s_0)\cdot (-1)^{D((N_0,s_0),(N_1,s_1))}.
\end{equation*}
This is a well behaved condition by \eq{fm10eq3}--\eq{fm10eq4}.
\end{itemize}
We call $F$ an {\it additive flag structure\/} if it also satisfies
\begin{itemize}
\setlength{\itemsep}{0pt}
\setlength{\parsep}{0pt}
\item[(iii)] If $(N_0,s_0),(N_1,s_1)$ are disjoint flagged submanifolds then
\end{itemize}
\begin{equation*}
F(N_0\amalg N_1,s_0\amalg s_1)=F(N_0,s_0)\cdot F(N_1,s_1).
\end{equation*}
In \cite{Joyc5}, additive flag structures are just called flag structures.
\end{dfn}

Here is \cite[Prop.~3.6]{Joyc5}:

\begin{prop}
\label{fm10prop1}
Let\/ $X$ be an oriented\/ $7$-manifold. Then:
\begin{itemize}
\setlength{\itemsep}{0pt}
\setlength{\parsep}{0pt}
\item[{\bf(a)}] There exists an additive flag structure $F$ on $X$.
\item[{\bf(b)}] If\/ $F,F'$ are additive flag structures on $X$ then there exists a unique group morphism $H_3(X,\Z)\ra\{\pm 1\},$ denoted $F'/F,$ such that
\e
F'(N,s)=F(N,s)\cdot (F'/F)[N] \qquad\text{for all\/ $(N,s)$.}
\label{fm10eq6}
\e
\item[{\bf(c)}] Let\/ $F$ be an additive flag structure on $X$ and\/ $\ep:H_3(X,\Z)\ra\{\pm 1\}$ a morphism, and define $F'$ by \eq{fm10eq6} with\/ $F'/F=\varepsilon$. Then $F'$ is an additive flag structure on~$X$.
\end{itemize}
Thus the set of additive flag structures on $X$ is a torsor over $\Hom(H_3(X,\Z),\Z_2)$.
\end{prop}

The next theorem will be proved in \S\ref{fm192}. The proof involves defining a modified bordism category $\tBord^{\bs\Spin}_7(K(\Z,4))$ equivalent to $\Bord^{\bs\Spin}_7(K(\Z,4))$, but with objects $(X,N)$ for $X$ a spin 7-manifold and $N\subset X$ a 3-submanifold, and an orientation functor $\ti\sH_7^\Z:\tBord^{\bs\Spin}_7(K(\Z,4))\ab\ra 0\qs\Z_2$ equivalent to $\sH_7^\Z$, which maps $(X,N)$ to a $\Z_2$-torsor of maps $\{s:(N,s)$ flagged$\}\ra\{\pm 1\}$.

\begin{thm}
\label{fm10thm1}
Let\/ $X$ be a compact spin\/ $7$-manifold. Then a flag structure on $X$ in the sense of Definition\/ {\rm\ref{fm10def2}} is equivalent to an orientation on $X$ for the orientation functor $\sH_7^\Z$ of Definition\/~{\rm\ref{fm9def7}}.
\end{thm}

\begin{rem}
\label{fm10rem1}
{\bf(a)} By Definition \ref{fm9def2} and Theorem \ref{fm10thm1}, if $X$ is a compact spin 7-manifold then a flag structure on $X$ is equivalent to a natural isomorphism $F$ in the 2-commutative diagram
\e
\begin{gathered}
\xymatrix@C=160pt@R=15pt{
*+[r]{\Bord_X(K(\Z,4))} \drtwocell_{}\omit^{}\omit{^{F\,\,\,\,}} \ar[r]_(0.62){\boo} \ar[d]^{\Pi_X^{\bs\Spin}} & *+[l]{0\qs\Z_2} \\
*+[r]{\Bord_7^{\bs\Spin}(K(\Z,4)))} \ar[r]^(0.62){\sH_7^\Z} & *+[l]{\Z_2\qs\Z_2.} \ar[u]^{F_{\Z_2\qs\Z_2}^{0\qs\Z_2}} }
\end{gathered}
\label{fm10eq7}
\e
From now on, we will identify flag structures with such natural isomorphisms.

The set of flag structures is a torsor over $\Map(H^4(X,\Z),\Z_2)$, where $\Map$ means arbitrary maps $H^4(X,\Z)\ra\Z_2$, not just group morphisms.

Note that $\Map(H^4(X,\Z),\Z_2)$ is usually very large. So it is desirable to impose extra conditions on flag structures to cut down the choices.
\smallskip

\noindent{\bf(b)} {\it Additive\/} flag structures correspond to natural isomorphisms $F$ in \eq{fm10eq7} which are {\it monoidal\/} with respect to the monoidal structures on $\Bord_X(K(\Z,4))$ discussed in Remark \ref{fm6rem1}(b), and on $0\qs\Z_2$. Note that the monoidal structure on $\Bord_X(K(\Z,4))$ (from adding cohomology classes on a fixed $X$) is {\it unrelated\/} to the monoidal structure on $\Bord_7^{\bs\Spin}(K(\Z,4))$ (from taking disjoint unions $X_1\amalg X_2$). The set of additive flag structures is a torsor for $\Hom(H^4(X,\Z),\Z_2)$, which is finite (this agrees with $\Hom(H_3(X,\Z),\Z_2)$ in Proposition \ref{fm10prop1} by Poincar\'e duality).
\smallskip

\noindent{\bf(c)} We say that a flag structure $F$ {\it factors via\/} $\Z_2$ if we can factor \eq{fm10eq7} into a diagram of natural isomorphisms for some $F'$, where $\eta_7$ is as in \eq{fm9eq7}
\e
\begin{gathered}
\xymatrix@C=63pt@R=15pt{
*+[r]{\Bord_X(K(\Z,4))} \drtwocell_{}\omit^{}\omit{^{\id\,\,\,\,}} \ar@/^1.2pc/[rr]^{\boo}_(0.52){\Downarrow\,\,\id} \ar[r]_(0.65){F_{K(\Z,4)}^{K(\Z_2,4)}} \ar[d]^{\Pi_X^{\bs\Spin}} & \Bord_X(K(\Z_2,4)) \ar[d]^{\Pi_X^{\bs\Spin}} \ar[r]_(0.62){\boo} \drtwocell_{}\omit^{}\omit{^{F'\,\,\,\,}} & *+[l]{0\qs\Z_2} \\
*+[r]{\Bord_7^{\bs\Spin}(K(\Z,4)))} \ar[r]^(0.65){F_{K(\Z,4)}^{K(\Z_2,4)}} \ar@/_1.2pc/[rr]_{\sH_7^\Z}^(0.52){\Uparrow\,\,\eta_7} & \Bord_7^{\bs\Spin}(K(\Z_2,4))) \ar[r]^(0.62){\sH_7^{\Z_2}} & *+[l]{\Z_2\qs\Z_2.} \ar[u]^{F_{\Z_2\qs\Z_2}^{0\qs\Z_2}} }\!\!\!\!
\end{gathered}
\label{fm10eq8}
\e
Here the possible choices for $F'$ are a torsor for $\Map(H^4(X,\Z_2),\Z_2)$. But $F$ only involves $F'$ on objects in the image of $\Bord_X(K(\Z,4))\ra \Bord_X(K(\Z_2,4))$. Thus flag structures $F$ which factor via $\Z_2$ are a torsor for the finite group
\begin{equation*}
\Map\bigl(\Im(H^4(X,\Z)\ra H^4(X,\Z_2)),\Z_2\bigr).
\end{equation*}

\noindent{\bf(d)} We say that a flag structure $F$ {\it is natural at zero\/} if the composition of natural isomorphisms across the following diagram 
\e
\begin{gathered}
\xymatrix@C=63pt@R=15pt{
*+[r]{\Bord_X(*)} \drtwocell_{}\omit^{}\omit{^{\id\,\,\,\,}} \ar@/^1.2pc/[rr]^{\boo}_(0.52){\Downarrow\,\,\id} \ar[r]_(0.55){\inc} \ar[d]^{\Pi_X^{\bs\Spin}} & \Bord_X(K(\Z,4)) \ar[d]^{\Pi_X^{\bs\Spin}} \ar[r]_(0.62){\boo} \drtwocell_{}\omit^{}\omit{^{F\,\,\,\,}} & *+[l]{0\qs\Z_2} \\
*+[r]{\Bord_7^{\bs\Spin}(*)} \ar[r]^(0.55){\inc} \ar@/_1.3pc/[rr]_{\boo}^(0.52){\Uparrow\,\,\ze_7} & \Bord_7^{\bs\Spin}(K(\Z,4))) \ar[r]^(0.62){\sH_7^\Z} & *+[l]{\Z_2\qs\Z_2} \ar[u]^{F_{\Z_2\qs\Z_2}^{0\qs\Z_2}} }\!\!\!\!
\end{gathered}
\label{fm10eq9}
\e
is the identity natural isomorphism $\id:\boo\Ra\boo$, where $\ze_7$ is as in \eq{fm9eq8}, and the inclusion functors $\inc$ are as in Definition \ref{fm9def7}. This prescribes natural values for $F$ at any exact~$C\in C^4(X,\Z)$. Flag structures $F$ natural at zero always exist, and form a torsor over $\Map(H^4(X,\Z)\sm\{0\},\Z_2)$, since \eq{fm10eq9} prescribes $F$ over $0\in H^4(X,\Z)$. If we require $F$ to factor via $\Z_2$, then they form a torsor for~$\Map\bigl(\Im(H^4(X,\Z)\ra H^4(X,\Z_2))\sm\{0\},\Z_2\bigr).$

In Definition \ref{fm10def2}, requiring $F$ in \eq{fm10eq5} to be natural at zero means requiring $F(\es,\es)=1$. Additive flag structures are automatically natural at zero.
\smallskip

\noindent{\bf(e)} Combining $\Im\hat\xi^{\bs\Spin}_7(K(\Z,4))\!=\!\Z\an{2\ze_2,\ze_3}$ and $\Im\hat\xi^{\bs\Spin}_7(K(\Z_2,4))\!=\!\Z_2\an{2\ze_2}$ in Table \ref{fm3tab3}, and \eq{fm9eq6} which implies that 
\begin{align*}
\Im\hat\xi^{\bs\Spin}_7(K(\Z,4))&\subseteq\Ker\bigl(\pi_1(\sH_7^\Z)\vert_{\ti\Om_8^{\bs\Spin}(K(\Z,4))}\bigr),\\
\Im\hat\xi^{\bs\Spin}_7(K(\Z_2,4))&\subseteq\Ker\bigl(\pi_1(\sH_7^{\Z_2})\vert_{\ti\Om_8^{\bs\Spin}(K(\Z_2,4))}\bigr),
\end{align*}
and so Theorem \ref{fm9thm1}(ii) shows that $\sH_7^\Z$ and $\sH_7^{\Z_2}$ are orientable for any compact spin 7-manifold $X$. This provides an alternative proof to Proposition \ref{fm10prop1}(a) that flag structures, and flag structures factoring via $\Z_2$, exist for any compact spin 7-manifold $X$.
\end{rem}

\subsection{Flag structures in 8 dimensions}
\label{fm102}

For our applications in \S\ref{fm12}--\S\ref{fm14}, we would like a notion of `flag structure' in 8 dimensions that plays the same r\^ole for orienting moduli spaces on $\Spin(7)$-manifolds and Calabi--Yau 4-folds, that 7-dimensional flag structures do for moduli spaces on $G_2$-manifolds in \cite{Joyc5,JoUp}.

We have chosen to define flag structures in 8 dimensions by replacing $\sH_7^\Z$ in Theorem \ref{fm10thm1} by $\sH_8^\Z$. It is natural to ask whether there is an equivalent explicit geometric definition of 8-dimensional flag structures similar to Definition \ref{fm10def2}. The authors do have such a definition, but it is so complicated that we have decided not to explain it.

\begin{dfn}
\label{fm10def3}
Let $X$ be a compact spin 8-manifold. A {\it flag structure\/} $F$ on $X$ is an orientation on $X$ for the orientation functor $\sH_8^\Z$ of Definition \ref{fm9def7}. So by Definition \ref{fm9def2}, $F$ is a natural isomorphism in the 2-commutative diagram
\e
\begin{gathered}
\xymatrix@C=160pt@R=15pt{
*+[r]{\Bord_X(K(\Z,4))} \drtwocell_{}\omit^{}\omit{^{F\,\,\,\,}} \ar[r]_(0.62){\boo_{\Z_2}} \ar[d]^{\Pi_X^{\bs\Spin}} & *+[l]{0\qs\Z_2} \\
*+[r]{\Bord_8^{\bs\Spin}(K(\Z,4)))} \ar[r]^(0.62){\sH_8^\Z} & *+[l]{\Z_2\qs\Z_2.} \ar[u]^{F_{\Z_2\qs\Z_2}^{0\qs\Z_2}} }
\end{gathered}
\label{fm10eq10}
\e

We say that a flag structure $F$ {\it factors via\/} $\Z_2$ if as in \eq{fm10eq8} we can factor \eq{fm10eq10} into a diagram of natural isomorphisms for some $F'$, where $\eta_8$ is as in~\eq{fm9eq7}
\e
\begin{gathered}
\xymatrix@C=63pt@R=15pt{
*+[r]{\Bord_X(K(\Z,4))} \drtwocell_{}\omit^{}\omit{^{\id\,\,\,\,}} \ar@/^1.2pc/[rr]^{\boo}_(0.52){\Downarrow\,\,\id} \ar[r]_(0.65){F_{K(\Z,4)}^{K(\Z_2,4)}} \ar[d]^{\Pi_X^{\bs\Spin}} & \Bord_X(K(\Z_2,4)) \ar[d]^{\Pi_X^{\bs\Spin}} \ar[r]_(0.62){\boo} \drtwocell_{}\omit^{}\omit{^{F'\,\,\,\,}} & *+[l]{0\qs\Z_2} \\
*+[r]{\Bord_8^{\bs\Spin}(K(\Z,4)))} \ar[r]^(0.65){F_{K(\Z,4)}^{K(\Z_2,4)}} \ar@/_1.2pc/[rr]_{\sH_8^\Z}^(0.52){\Uparrow\,\,\eta_8} & \Bord_7^{\bs\Spin}(K(\Z_2,4))) \ar[r]^(0.62){\sH_8^{\Z_2}} & *+[l]{\Z_2\qs\Z_2.} \ar[u]^{F_{\Z_2\qs\Z_2}^{0\qs\Z_2}} }\!\!\!\!
\end{gathered}
\label{fm10eq11}
\e

We say that a flag structure $F$ {\it is natural at zero\/} if as in \eq{fm10eq9} the composition of natural isomorphisms across the following diagram 
\begin{equation*}
\xymatrix@C=63pt@R=15pt{
*+[r]{\Bord_X(*)} \drtwocell_{}\omit^{}\omit{^{\id\,\,\,\,}} \ar@/^1.2pc/[rr]^{\boo}_(0.52){\Downarrow\,\,\id} \ar[r]_(0.55){\inc} \ar[d]^{\Pi_X^{\bs\Spin}} & \Bord_X(K(\Z,4)) \ar[d]^{\Pi_X^{\bs\Spin}} \ar[r]_(0.62){\boo} \drtwocell_{}\omit^{}\omit{^{F\,\,\,\,}} & *+[l]{0\qs\Z_2} \\
*+[r]{\Bord_8^{\bs\Spin}(*)} \ar[r]^(0.55){\inc} \ar@/_1.3pc/[rr]_{\boo}^(0.52){\Uparrow\,\,\ze_8} & \Bord_8^{\bs\Spin}(K(\Z,4))) \ar[r]^(0.62){\sH_8^\Z} & *+[l]{\Z_2\qs\Z_2} \ar[u]^{F_{\Z_2\qs\Z_2}^{0\qs\Z_2}} }\!\!\!\!
\end{equation*}
is the identity natural isomorphism $\id:\boo\Ra\boo$, where $\ze_8$ is as in \eq{fm9eq8}.\end{dfn}

\begin{thm}
\label{fm10thm2}
Let\/ $X$ be a compact spin\/ $8$-manifold. Then
\smallskip

\noindent{\bf(a)} $X$ admits a flag structure if and only if the following condition holds:
\begin{itemize}
\setlength{\itemsep}{0pt}
\setlength{\parsep}{0pt}
\item[$\bf(*)$] There does not exist a class\/ $\al\in H^3(X,\Z)$ such that\/ $\int_X\bar\al\cup\Sq^2(\bar\al)=\ul{1}$ in $\Z_2,$ where $\bar\al\in H^3(X,\Z_2)$ is the mod\/ $2$ reduction of\/ $\al,$ and\/ $\Sq^2(\bar\al)\in H^5(X,\Z_2)$ is its Steenrod square.
\end{itemize}
Then the set of flag structures $F$ on $X$ is a torsor for $\Map(H^4(X,\Z),\Z_2)$.

We may also require $F$ to be natural at zero, and such flag structures form a torsor for $\Map(H^4(X,\Z)\sm\{0\},\Z_2)$.
\smallskip

\noindent{\bf(b)} $X$ admits a flag structure factoring via $\Z_2$ if and only if the following holds:
\begin{itemize}
\setlength{\itemsep}{0pt}
\setlength{\parsep}{0pt}
\item[$\bf(\dag)$] There does not exist\/ $\bar\al\in H^3(X,\Z_2)$ such that\/ $\int_X\bar\al\cup\Sq^2(\bar\al)=\ul{1}$ in $\Z_2$.
\end{itemize}
Then the set of flag structures $F$ which factor via $\Z_2$ are a torsor for the finite group
$\Map\bigl(\Im(H^4(X,\Z)\ra H^4(X,\Z_2)),\Z_2\bigr)$.

We may also require $F$ to be natural at zero, and such flag structures form a torsor for $\Map(\Im(H^4(X,\Z)\ra H^4(X,\Z_2))\sm\{0\},\Z_2)$.
\end{thm}

\begin{proof}
For (a), since $\pi_1(\sH_8^\Z)$ is nonzero on $\ti\Om_9^{\bs\Spin}(K(\Z,4))=\Z_2\an{\al_1\ze_2}$ by \eq{fm9eq6}, it follows from Theorem \ref{fm9thm1}(ii) and the explicit description of $\hat\xi^{\bs\Spin}_8(K(\Z,4))$ in Theorem \ref{fm3thm2}(c) that $\sH_8^\Z$ is orientable for $X$ (that is, flag structures exist) if and only if condition $(*)$ holds, as this is the condition for there to exist $[X,f]$ in $\Om^{\bs\Spin}_8(\cL K(\Z,4);K(\Z,4))$ with $\hat\xi^{\bs\Spin}_8(K(\Z,4))([X,f])\ne 0$ in \eq{fm3eq22}. The set of flag structures is then a torsor for $\Map(\pi_0(\Bord_X(K(\Z,4))),\Z_2)$, where $\pi_0(\Bord_X(K(\Z,4)))=H^4(X,\Z)$. The last part holds as in Remark~\ref{fm10rem1}(d).

For (b), we see from \eq{fm10eq11} that $X$ admits a flag structure factoring via $\Z_2$ if and only if $\sH_8^{\Z_2}$ is orientable for $X$. Following the argument above, we find that this is true if and only if $(\dag)$ holds. The last parts hold as in Remark~\ref{fm10rem1}(c),(d).
\end{proof}

\begin{ex}
\label{fm10ex1}
Consider the compact spin 8-manifold $\SU(3)$. The projection $\SU(3)\ra\SU(3)/\SU(2)=\cS^5$ is a fibration with fibre $\SU(2)=\cS^3$, so from the Serre spectral sequence we see that $H^*(X,R)\cong H^*(\cS^3\t\cS^5,R)$. Thus $H^3(\SU(3),\Z)=\Z\an{\al}$, $H^5(\SU(3),\Z)=\Z\an{\be}$, and $H^3(\SU(3),\Z_2)=\Z\an{\bar\al}$, $H^5(\SU(3),\Z_2)=\Z_2\an{\bar\be}$. One can show that $\Sq^2(\bar\al)=\bar\be$, so that $\int_{\SU(3)}\bar\al\cup\Sq^2(\bar\al)=\ul 1$ in $\Z_2$. Therefore Theorem \ref{fm10thm2}$(*)$ does not hold for~$\SU(3)$. 

By a theorem of Samelson, $\SU(3)$ has a left-invariant integrable complex structure, which has trivial canonical bundle. Thus we can regard $\SU(3)$ as a `non-K\"ahler Calabi--Yau 4-fold'. Note that in \S\ref{fm13}, we will want to know whether Calabi--Yau 4-folds satisfy Theorem~\ref{fm10thm2}$(*)$.
\end{ex}

\begin{rem}
\label{fm10rem2}
It is natural to ask whether there is a good notion of {\it additive\/} flag structure $F$ in 8 dimensions, parallel to the 7-dimensional version in Definition \ref{fm10def2}. This should be a compatibility between $F$ and the symmetric monoidal structure on $\Bord_X(K(\Z,4))$ discussed in Remark \ref{fm6rem1}(b). 

This does not work very well. The mod 2 intersection form on $H^4(X,\Z)$ gives obstructions to making the functor $F_{\Z_2\qs\Z_2}^{0\qs\Z_2}\ci\sH_8^\Z\ci\Pi_X^{\bs\Spin}$ in \eq{fm10eq10} symmetric monoidal. Even when it does work, it turns out not to be very useful, for reasons explained in Remark \ref{fm13rem4}. So we have not developed the idea.
\end{rem}

\section{Factorizations of orientation functors}
\label{fm11}

\subsection{Explicit computation of orientation functors}
\label{fm111}

We compute some orientation functors, which will be applied later to study orientations on moduli spaces. The next theorem will be proved in~\S\ref{fm193}.

\begin{thm}
\label{fm11thm1}
{\bf(a)} In Definition\/ {\rm\ref{fm9def4}} with\/ $n=7$ and\/ $\bs B=\bs\Spin,$ consider the normalized orientation functors 
\begin{equation*}
\sN_7^{\bs\Spin,G}:\Bord_7^{\bs\Spin}(BG)\longra 0\qs\Z=\ZZtor
\end{equation*}
for $G=\SU(2),$ or\/ $\SU(m)$ for\/ $m\ge 4,$ or\/ $\Sp(m)$ for\/ $m\ge 2,$ or\/ $E_8$. As in Theorem\/ {\rm\ref{fmAthm2},} using {\rm\eq{fm5eq4}--\eq{fm5eq5}} these are classified up to monoidal natural isomorphism by two morphisms
\begin{align*}
\pi_0(\sN_7^{\bs\Spin,G})&:\Om_7^{\bs\Spin}(BG)\longra 0,\\ 
\pi_1(\sN_7^{\bs\Spin,G})&:\Om_8^{\bs\Spin}(BG)\longra\Z, 
\end{align*}
and an element of an\/ $H_\sym^2(\Om_7^{\bs\Spin}(BG),\Z)$-torsor. As $\Om_7^{\bs\Spin}(BG)\ab=0$ by Corollary\/ {\rm\ref{fm3cor1},} this torsor is trivial, so the only nontrivial invariant is $\pi_1(\sN_7^{\bs\Spin,G})$. In the splitting
\e
\Om_n^{\bs\Spin}(BG)=\Om_n^{\bs\Spin}(*)\op\ti\Om_n^{\bs\Spin}(BG),
\label{fm11eq1}
\e
$\pi_1(\sN_7^{\bs\Spin,G})$ is\/ $0$ on\/ $\Om_8^{\bs\Spin}(*),$ by definition of normalized orientations. It acts on the generators of\/ $\ti\Om_8^{\bs\Spin}(BG)$ in Table\/ {\rm\ref{fm3tab4}} as in Table\/~{\rm\ref{fm11tab1}}.
\smallskip

\begin{table}[htb]
\centerline{
\begin{tabular}{|l|l|l|}
\hline
 & $\ze_1$ & $\ze_2$  \\
\hline
$\pi_1\bigl(\sN_7^{\bs\Spin,\SU(2)}\bigr)^{\vphantom{|^l}}_{\vphantom{|_l}}$  & $-8$ & $1$  \\
\hline
\end{tabular}
\quad
\begin{tabular}{|l|l|l|l|}
\hline
 & $\frac{\ze_1}{2}$ & $\ze_2$ &  $\ze_3$ \\
\hline
$\pi_1\bigl(\sN_7^{\bs\Spin,\SU(m)}\bigr)^{\vphantom{|^l}}_{\vphantom{|_l}}$, $m\ge 4$  & $-2m$ & $1$ & $0$ \\
\hline
\end{tabular}
}
\smallskip

\centerline{
\begin{tabular}{|l|l|l|l|}
\hline
 & $\ze_1$ & $\ze_2$ &  $\ze_2'$ \\
\hline
$\pi_1\bigl(\sN_7^{\bs\Spin,\Sp(m)}\bigr)^{\vphantom{|^l}}_{\vphantom{|_l}}$, $m\ge 2$  & $-4(m+1)$ & $1$ & $0$ \\
\hline
\end{tabular}
\quad
\begin{tabular}{|l|l|l|}
\hline
 &  $\ze_2$ &  $\ze_3$ \\
\hline
$\pi_1\bigl(\sN_7^{\bs\Spin,E_8}\bigr)^{\vphantom{|^l}}_{\vphantom{|_l}}$  & 1 & 0 \\
\hline
\end{tabular}
}
\caption{The morphisms $\pi_1\bigl(\sN_7^{\bs\Spin,G}\bigr)$}
\label{fm11tab1}
\end{table}

\noindent{\bf(b)} In Definition\/ {\rm\ref{fm9def4}} with\/ $n=8$ and\/ $\bs B=\bs\Spin,$ consider the normalized orientation functors 
\begin{equation*}
\sN_8^{\bs\Spin,G}:\Bord_8^{\bs\Spin}(BG)\longra \Z\qs\Z_2
\end{equation*}
for $G=\SU(2),$ or\/ $\SU(m)$ for\/ $m\ge 5,$ or\/ $\Sp(m)$ for\/ $m\ge 2,$ or\/ $E_8$. As in Theorem\/ {\rm\ref{fmAthm2},} using {\rm\eq{fm5eq4}--\eq{fm5eq5}} these are classified up to monoidal natural isomorphism by two morphisms
\begin{align*}
\pi_0(\sN_8^{\bs\Spin,G})&:\Om_8^{\bs\Spin}(BG)\longra \Z,\\
\pi_1(\sN_8^{\bs\Spin,G})&:\Om_9^{\bs\Spin}(BG)\longra\Z_2=\{\ul{0},\ul{1}\},
\end{align*}
and an element of a torsor over $H_\sym^2(\Om_8^{\bs\Spin}(BG),\Z_2)$. As $\Om_8^{\bs\Spin}(BG)\ab\cong\Z^k$ by Corollary\/ {\rm\ref{fm3cor1},} this torsor is trivial, since calculation gives $H^2(\Z^k,\Z_2)\cong\Alt(\Z^k,\Z_2)\cong \Z_2^{k(k-1)/2},$ so $H_\sym^2(\Z^k,\Z_2)=0$ by \eq{fmAeq9}. We have
\begin{equation*}
\pi_0(\sN_8^{\bs\Spin,G})=\pi_1(\sN_7^{\bs\Spin,G}),
\end{equation*}
where the right hand side is given in\/ {\bf(a)}. In\/ {\rm\eq{fm11eq1},} $\pi_1(\sN_8^{\bs\Spin,G})$ is\/ $0$ on\/ $\Om_9^{\bs\Spin}(*),$ and acts on the generators of\/ $\ti\Om_9^{\bs\Spin}(BG)$ in Table\/ {\rm\ref{fm3tab4}} as in Table\/~{\rm\ref{fm11tab2}}.
\smallskip

\begin{table}[htb]
\centerline{
\begin{tabular}{|l|l|}
\hline
 &  $\al_1\ze_2$  \\
\hline
$\pi_1\bigl(\sN_8^{\bs\Spin,\SU(2)}\bigr)^{\vphantom{|^l}}_{\vphantom{|_l}}$  & \ul{1}  \\
\hline
\end{tabular}
\quad
\begin{tabular}{|l|l|}
\hline
 & $\al_1\ze_2$  \\
\hline
$\pi_1\bigl(\sN_8^{\bs\Spin,\SU(m)}\bigr)^{\vphantom{|^l}}_{\vphantom{|_l}}$, $m\ge 5$ & \ul{1} \\
\hline
\end{tabular}
}
\smallskip

\centerline{
\begin{tabular}{|l|l|l|}
\hline
 &  $\al_1\ze_2$ &  $\al_1\ze_2'$ \\
\hline
$\pi_1\bigl(\sN_8^{\bs\Spin,\Sp(m)}\bigr)^{\vphantom{|^l}}_{\vphantom{|_l}}$, $m\ge 2$  & \ul{1} & \ul{0} \\
\hline
\end{tabular}
\quad
\begin{tabular}{|l|l|}
\hline
 &  $\al_1\ze_2$ \\
\hline
$\pi_1\bigl(\sN_8^{\bs\Spin,E_8}\bigr)^{\vphantom{|^l}}_{\vphantom{|_l}}$  & \ul{1} \\
\hline
\end{tabular}
}
\caption{The morphisms $\pi_1\bigl(\sN_8^{\bs\Spin,G}\bigr)$}
\label{fm11tab2}
\end{table}

\noindent{\bf(c)} The analogues of\/ {\bf(a)\rm,\bf(b)} for the orientation functors
\begin{align*}
\sN_{7,\Z_k}^{\bs\Spin,G}&:\Bord_7^{\bs\Spin}(BG)\longra 0\qs\Z_k=\text{\rm$\Z_k$-tor},\\
\sN_{8,\Z_{2k}}^{\bs\Spin,G}&:\Bord_8^{\bs\Spin}(BG)\longra \Z_{2k}\qs\Z_2
\end{align*}
in Definition\/ {\rm\ref{fm9def4}} are obtained by reducing mod\/ $k$ or\/ $2k$ in Table\/~{\rm\ref{fm11tab1}}.
\end{thm}

Now Theorem \ref{fm9thm1}(a) gives a necessary and sufficient condition for the orientation functors $\sO=\sN_7^{\bs\Spin,G},\sN_{7,\Z_k}^{\bs\Spin,G},\sN_8^{\bs\Spin,G}$ to be orientable for all compact spin 7- or 8-manifolds $X$: we must have $\Xi_{n,\sO}^{\bs\Spin,G}\equiv\ul{0}$, which by \eq{fm9eq3} is equivalent to $\Im\hat\xi^{\bs\Spin}_n(BG)\subseteq\Ker\pi_1(\sO)$. The images $\Im\hat\xi^{\bs\Spin}_n(BG)$ are given in Table \ref{fm3tab5}, and the kernels $\Ker\pi_1(\sO)$ are determined by the actions of $\pi_1(\sO)$ in Tables \ref{fm11tab1}--\ref{fm11tab2}. Thus we deduce:

\begin{cor}
\label{fm11cor1}
Of the orientation functors\/ $\sN_7^{\bs\Spin,G},\sN_{7,\Z_k}^{\bs\Spin,G},\sN_8^{\bs\Spin,G}$ and\/ $\sN_{8,\Z_{2k}}^{\bs\Spin,G}$ described in Theorem\/ {\rm\ref{fm11thm1}} for\/ $k\ge 2$ and\/ $G=\SU(2),$ or\/ $\SU(m)$ for\/ $m\ge 4$ $(n=7)$ or\/ $m\ge 5$ $(n=8),$ or\/ $\Sp(m)$ for\/ $m\ge 2,$ or\/ $E_8,$ the following, and only the following, are orientable for all compact spin\/ $7$- or\/ $8$-manifolds\/~{\rm$X$:}
\begin{equation*}
\sN_{7,\Z_2}^{\bs\Spin,\SU(2)},\;\>
\sN_{7,\Z_2}^{\bs\Spin,\SU(m)},\;\>
\sN_{7,\Z_2}^{\bs\Spin,E_8},\;\>
\sN_8^{\bs\Spin,\SU(2)},\;\>
\sN_{8,\Z_{2k}}^{\bs\Spin,\SU(2)}.
\end{equation*}
\end{cor}

The next theorem will be proved in~\S\ref{fm194}.

\begin{thm}
\label{fm11thm2}
{\bf(a)} In Definition\/ {\rm\ref{fm9def6}} with\/ $n=7,$ $\bs B=\bs\Spin,$ and\/ $H=\SO(4),$ consider the three orientation functors for $*=+,-,0$
\e
\sO_{7,4}^{\bs\Spin,\SO(4),*}:\Bord_{7,4}^{\bs\Spin}(M\SO(4))\longra 0\qs\Z=\ZZtor.
\label{fm11eq2}
\e
As in Theorem\/ {\rm\ref{fmAthm2},} using {\rm\eq{fm5eq4}--\eq{fm5eq5}} these are classified up to monoidal natural isomorphism by two morphisms
\begin{align*}
\pi_0(\sO_{7,4}^{\bs\Spin,\SO(4),*})&:\Om_7^{\bs\Spin}(M\SO(4))\longra 0,\\ 
\pi_1(\sO_{7,4}^{\bs\Spin,\SO(4),*})&:\Om_8^{\bs\Spin}(M\SO(4))\longra\Z, 
\end{align*}
and an element of an\/ $H_\sym^2(\Om_7^{\bs\Spin}(M\SO(4)),\Z)$-torsor. As $\Om_7^{\bs\Spin}(M\SO(4))\ab=0$ by Theorem\/ {\rm\ref{fm3thm1}(a),} this torsor is trivial, so the only nontrivial invariant is $\pi_1(\sO_{7,4}^{\bs\Spin,\SO(4),*})$. In the splitting
\e
\Om_n^{\bs\Spin}(M\SO(4))=\Om_n^{\bs\Spin}(*)\op\ti\Om_n^{\bs\Spin}(M\SO(4)),
\label{fm11eq3}
\e
$\pi_1(\sO_{7,4}^{\bs\Spin,\SO(4),*})$ is\/ $0$ on\/ $\Om_8^{\bs\Spin}(*),$ and acts as in Table\/ {\rm\ref{fm11tab3}} on the generators of $\ti\Om_8^{\bs\Spin}(M\SO(4))$ in Table\/~{\rm\ref{fm3tab1}}.
\smallskip

\begin{table}[htb]
\centerline{\begin{tabular}{|l|l|l|l|}
\hline
 & $\frac{\ze_1}{4}$ & $\ze_2$ &  $\ze_3$ \\
\hline
$\pi_1\bigl(\sO_{7,4}^{\bs\Spin,\SO(4),+}\bigr)^{\vphantom{|^l}}_{\vphantom{|_l}}$  & $-1$ & 0 & 0 \\
\hline
$\pi_1\bigl(\sO_{7,4}^{\bs\Spin,\SO(4),-}\bigr)^{\vphantom{|^l}}_{\vphantom{|_l}}$   & $-1$ & 1 & 0 \\
\hline
$\pi_1\bigl(\sO_{7,4}^{\bs\Spin,\SO(4),0}\bigr)^{\vphantom{|^l}}_{\vphantom{|_l}}$   & 0 & 1 & 0 \\
\hline
\end{tabular}}
\caption{The morphisms $\pi_1\bigl(\sO_{7,4}^{\bs\Spin,\SO(4),*}\bigr)$}
\label{fm11tab3}
\end{table}

\noindent{\bf(b)} In Definition\/ {\rm\ref{fm9def6}} with\/ $n=8,$ $\bs B=\bs\Spin,$ and $H=\SO(4),$ consider the three orientation functors for $*=+,-,0$
\e
\sO_{8,4}^{\bs\Spin,\SO(4),*}:\Bord_{8,4}^{\bs\Spin}(M\SO(4))\longra \Z\qs\Z_2.
\label{fm11eq4}
\e
As in Theorem\/ {\rm\ref{fmAthm2},} using {\rm\eq{fm5eq4}--\eq{fm5eq5}} these are classified up to monoidal natural isomorphism by two morphisms
\begin{align*}
\pi_0(\sO_{8,4}^{\bs\Spin,\SO(4),*})&:\Om_8^{\bs\Spin}(M\SO(4))\longra \Z,\\
\pi_1(\sO_{8,4}^{\bs\Spin,\SO(4),*})&:\Om_9^{\bs\Spin}(M\SO(4))\longra\Z_2=\{\ul{0},\ul{1}\},
\end{align*}
and an element of a torsor over $H_\sym^2(\Om_8^{\bs\Spin}(M\SO(4)),\Z_2)$. As $\Om_8^{\bs\Spin}(M\SO(4))\ab\cong\Z^5$ by Theorem\/ {\rm\ref{fm3thm1}(a),} this torsor is trivial, since calculation gives $H^2(\Z^k,\Z_2)\ab\cong\Alt(\Z^k,\Z_2)\cong \Z_2^{k(k-1)/2},$ so $H_\sym^2(\Z^k,\Z_2)=0$ by \eq{fmAeq9}. We have
\e
\pi_0(\sO_{8,4}^{\bs\Spin,\SO(4),*})=\pi_1(\sO_{7,4}^{\bs\Spin,\SO(4),*}),
\label{fm11eq5}
\e
where the right hand sides are given in {\bf(a)}. In\/ {\rm\eq{fm11eq3},} $\pi_1(\sO_{8,4}^{\bs\Spin,\SO(4),*})$ is $0$ on $\Om_9^{\bs\Spin}(*),$ and acts on the generators of\/ $\ti\Om_9^{\bs\Spin}(M\SO(4))$ in Table\/ {\rm\ref{fm3tab1}} as in Table\/~{\rm\ref{fm11tab4}}.

\begin{table}[htb]
\centerline{\begin{tabular}{|l|l|l|l|}
\hline
 & $\al_1\frac{\ze_1}{4}$ & $\al_1\ze_2$  & $\eta$ \\
\hline
$\pi_1\bigl(\sO_{8,4}^{\bs\Spin,\SO(4),+}\bigr)^{\vphantom{|^l}}_{\vphantom{|_l}}$ & \ul{1} & \ul{0} & ? \\
\hline
$\pi_1\bigl(\sO_{8,4}^{\bs\Spin,\SO(4),-}\bigr)^{\vphantom{|^l}}_{\vphantom{|_l}}$ & \ul{1} & \ul{1} & ? \\
\hline
$\pi_1\bigl(\sO_{8,4}^{\bs\Spin,\SO(4),0}\bigr)^{\vphantom{|^l}}_{\vphantom{|_l}}$ & \ul{0} & \ul{1} & \ul{0} \\
\hline
\end{tabular}}
\caption{The morphisms $\pi_1(\sO_{8,4}^{\bs\Spin,\SO(4),*})$}
\label{fm11tab4}
\end{table}

\noindent{\bf(c)} If\/ $\rho:H\ra\SO(4)\subset\O(4)$ is a morphism of Lie groups, then the functors $\sO_{7,4}^{\bs\Spin,H,*},\sO_{8,4}^{\bs\Spin,H,*}$ are determined from $\sO_{7,4}^{\bs\Spin,\SO(4),*},\sO_{8,4}^{\bs\Spin,\SO(4),*}$ by Proposition\/ {\rm\ref{fm11prop3}} with\/ $H_1=H$ and\/ $H_2=\SO(4)$. For\/ $H=\SU(2),\U(2),\Spin(4),$ Theorem\/ {\rm\ref{fm3thm1}(a),(b)} determine the groups\/ $\Om_n^{\bs\Spin}(MH)$ and morphisms\/ $\Om_n^{\bs\Spin}(MH)\ab\ra\Om_n^{\bs\Spin}(M\SO(4))$ for\/ $n=7,8,9$. Combining this with\/ {\bf(a)\rm,\bf(b)} gives the analogue of\/ {\bf(a)\rm,\bf(b)} for $H=\SU(2),\U(2),\Spin(4),$ with\/ $\pi_i(\sO_{n,4}^{\bs\Spin,H,*})$ given in Tables\/ {\rm\ref{fm11tab5}} and\/~{\rm\ref{fm11tab6}}.

\begin{table}[htb]
\centerline{
\begin{tabular}{|l|l|l|}
\hline
 & $\ze_1$ & $\ze_2$  \\
\hline
$\pi_1\bigl(\sO_{7,4}^{\bs\Spin,\SU(2),+}\bigr)^{\vphantom{|^l}}_{\vphantom{|_l}}$  & $-4$ & 0 \\
\hline
$\pi_1\bigl(\sO_{7,4}^{\bs\Spin,\SU(2),-}\bigr)^{\vphantom{|^l}}_{\vphantom{|_l}}$   & $-4$ & 1 \\
\hline
$\pi_1\bigl(\sO_{7,4}^{\bs\Spin,\SU(2),0}\bigr)^{\vphantom{|^l}}_{\vphantom{|_l}}$   & 0 & 1 \\
\hline
\end{tabular}
\quad
\begin{tabular}{|l|l|l|l|}
\hline
 & $\frac{\ze_1}{2}$ & $\ze_2$ &  $\ze_3$ \\
\hline
$\pi_1\bigl(\sO_{7,4}^{\bs\Spin,\U(2),+}\bigr)^{\vphantom{|^l}}_{\vphantom{|_l}}$  & $-2$ & 0 & 0 \\
\hline
$\pi_1\bigl(\sO_{7,4}^{\bs\Spin,\U(2),-}\bigr)^{\vphantom{|^l}}_{\vphantom{|_l}}$   & $-2$ & 1 & 0 \\
\hline
$\pi_1\bigl(\sO_{7,4}^{\bs\Spin,\U(2),0}\bigr)^{\vphantom{|^l}}_{\vphantom{|_l}}$   & 0 & 1 & 0 \\
\hline
\end{tabular}} 
\smallskip

\centerline{
\begin{tabular}{|l|l|l|l|}
\hline
 & $\ze_1$ & $\ze_2$ &  $\ze_2'$ \\
\hline
$\pi_1\bigl(\sO_{7,4}^{\bs\Spin,\Spin(4),+}\bigr)^{\vphantom{|^l}}_{\vphantom{|_l}}$  & $-4$ & 0 & 0 \\
\hline
$\pi_1\bigl(\sO_{7,4}^{\bs\Spin,\Spin(4),-}\bigr)^{\vphantom{|^l}}_{\vphantom{|_l}}$   & $-4$ & 1 & $-1$ \\
\hline
$\pi_1\bigl(\sO_{7,4}^{\bs\Spin,\Spin(4),0}\bigr)^{\vphantom{|^l}}_{\vphantom{|_l}}$   & 0 & 1 & $-1$ \\
\hline
\end{tabular}
}
\caption{The morphisms $\pi_1\bigl(\sO_{7,4}^{\bs\Spin,H,*}\bigr)$}
\label{fm11tab5}
\end{table}

\begin{table}[htb]
\centerline{\begin{tabular}{|l|l|}
\hline
 & $\al_1\ze_2$  \\
\hline
$\pi_1\bigl(\sO_{8,4}^{\bs\Spin,\SU(2),+}\bigr)^{\vphantom{|^l}}_{\vphantom{|_l}}$  & \ul{0}  \\
\hline
$\pi_1\bigl(\sO_{8,4}^{\bs\Spin,\SU(2),-}\bigr)^{\vphantom{|^l}}_{\vphantom{|_l}}$ & \ul{1}  \\
\hline
$\pi_1\bigl(\sO_{8,4}^{\bs\Spin,\SU(2),0}\bigr)^{\vphantom{|^l}}_{\vphantom{|_l}}$  & \ul{1}  \\
\hline
\end{tabular}
\quad
\begin{tabular}{|l|l|}
\hline
 & $\al_1\ze_2$  \\
\hline
$\pi_1\bigl(\sO_{8,4}^{\bs\Spin,\U(2),+}\bigr)^{\vphantom{|^l}}_{\vphantom{|_l}}$  & \ul{0}  \\
\hline
$\pi_1\bigl(\sO_{8,4}^{\bs\Spin,\U(2),-}\bigr)^{\vphantom{|^l}}_{\vphantom{|_l}}$ & \ul{1}  \\
\hline
$\pi_1\bigl(\sO_{8,4}^{\bs\Spin,\U(2),0}\bigr)^{\vphantom{|^l}}_{\vphantom{|_l}}$  & \ul{1}  \\
\hline
\end{tabular}}
\smallskip

\centerline{\begin{tabular}{|l|l|l|}
\hline
 &  $\al_1\ze_2$  & $\al_1\ze_2'$ \\
\hline
$\pi_1\bigl(\sO_{8,4}^{\bs\Spin,\Spin(4),+}\bigr)^{\vphantom{|^l}}_{\vphantom{|_l}}$  & \ul{0} & \ul{0} \\
\hline
$\pi_1\bigl(\sO_{8,4}^{\bs\Spin,\Spin(4),-}\bigr)^{\vphantom{|^l}}_{\vphantom{|_l}}$  & \ul{1} & \ul{1} \\
\hline
$\pi_1\bigl(\sO_{8,4}^{\bs\Spin,\Spin(4),0}\bigr)^{\vphantom{|^l}}_{\vphantom{|_l}}$  & \ul{1} & \ul{1} \\
\hline
\end{tabular}}
\caption{The morphisms $\pi_1(\sO_{8,4}^{\bs\Spin,H,*})$}
\label{fm11tab6}
\end{table}

\noindent{\bf(d)} The analogues of\/ {\bf(a)\rm--\bf(c)} for the orientation functors
\begin{align*}
\sO_{7,4,\Z_k}^{\bs\Spin,H,*}&:\Bord_{7,4}^{\bs\Spin}(MH)\longra 0\qs\Z_k=\text{\rm$\Z_k$-tor},\\
\sO_{8,4,\Z_{2k}}^{\bs\Spin,H,*}&:\Bord_{8,4}^{\bs\Spin}(MH)\longra \Z_{2k}\qs\Z_2
\end{align*}
in Definition\/ {\rm\ref{fm9def6}} for\/ $H=\SO(4),\SU(2),\U(2),\Spin(4)$ are obtained by reducing mod\/ $k$ or\/ $2k$ in Tables\/ {\rm\ref{fm11tab3}} and\/~{\rm\ref{fm11tab5}}.
\end{thm}

The next corollary is proved as for Corollary \ref{fm11cor1}, but using Tables \ref{fm3tab3} and~\ref{fm11tab3}--\ref{fm11tab6}.

\begin{cor}
\label{fm11cor2}
Of the orientation functors\/ $\sO_{7,4}^{\bs\Spin,H,*},$ $\sO_{7,4,\Z_k}^{\bs\Spin,H,*},$ $\sO_{8,4}^{\bs\Spin,H,*},$ $\sO_{8,4,\Z_{2k}}^{\bs\Spin,H,*}$ described in Theorem\/ {\rm\ref{fm11thm2}} for\/ $H=\SO(4),\SU(2),\U(2),\Spin(4),$ $*=+,-,0,$ and\/ $k\ge 2,$ the following, and only the following, are orientable for all compact spin\/ $7$- or\/ $8$-manifolds\/~$X$:
\begin{gather*}
\sO_{7,4,\Z_2}^{\bs\Spin,\SO(4),0},\sO_{7,4,\Z_2}^{\bs\Spin,\SU(2),*},\sO_{7,4,\Z_4}^{\bs\Spin,\SU(2),+},
\sO_{7,4,\Z_2}^{\bs\Spin,\U(2),*},\sO_{7,4,\Z_2}^{\bs\Spin,\Spin(4),*},\\
\sO_{7,4,\Z_4}^{\bs\Spin,\Spin(4),+},\sO_{8,4}^{\bs\Spin,\SU(2),*},\sO_{8,4,\Z_{2k}}^{\bs\Spin,\SU(2),*},
\sO_{8,4}^{\bs\Spin,\U(2),+},\sO_{8,4,\Z_{2k}}^{\bs\Spin,\U(2),+},\\ 
\sO_{8,4}^{\bs\Spin,\Spin(4),*},\sO_{8,4,\Z_{2k}}^{\bs\Spin,\Spin(4),*}.
\end{gather*}
\end{cor}

\subsection{Factorizing orientation functors via transfer functors}
\label{fm112}

We now consider examples of diagrams of the form
\e
\begin{gathered}
\xymatrix@C=60pt@R=15pt{ *+[r]{\cC} \drtwocell_{}\omit^{}\omit{_{\,\,\,\,\,\,\,\la}} \ar[d]^F \ar@/^.6pc/[drr]^(0.5){\sO}  \\
*+[r]{\cC'} \ar[rr]^{\sO'} && *+[l]{A\qs B,} }
\end{gathered}
\label{fm11eq6}
\e
where $\cC,\cC'$ are bordism categories, as in \S\ref{fm4}--\S\ref{fm6}, and $F$ is a transfer functor, as in \S\ref{fm8}, and $\sO,\sO'$ are orientation functors, as in \S\ref{fm9}, and $\la$ is a monoidal natural isomorphism. Then we say that {\it the orientation functor\/ $\sO$ factors via the orientation functor\/}~$\sO'$.

We will apply this as follows: suppose $X$ is a compact $n$-manifold with $\bs B$-structure, and $\cC_X,\cC_X'$ are the bordism categories $\Bord_X(\cdots)$ associated to $\cC,\cC'$ as in \S\ref{fm43}, \S\ref{fm53}, \S\ref{fm63}, and $\eta'_X$ is an orientation for $\sO'$ on $X$, then composing natural isomorphisms across the diagram
\e
\begin{gathered}
\xymatrix@C=60pt@R=13pt{ *+[r]{\cC_X} \ar@<-1ex>@{}[dr]^= \ar[d]^{F_X} \ar[r]_(0.3){\Pi_X} & \cC \drtwocell_{}\omit^{}\omit{_{\,\,\,\,\,\,\,\la}} \ar[d]^F \ar@/^.6pc/[drr]^(0.5){\sO}  \\
*+[r]{\cC'_X} \ar@/_.6pc/[drrr]_(0.3)\boo \ar[r]^(0.7){\Pi'_X} & \cC' \drrtwocell_{}\omit^{}\omit{_{\,\,\,\,\,\,\,\eta'_X}} \ar[rr]^{\sO'} && *+[l]{A\qs B} \ar[d] \\
&&& *+[l]{0\qs B} }
\end{gathered}
\label{fm11eq7}
\e
gives an orientation $\eta_X$ for $\sO$ on $X$. Hence, if $\sO'$ is orientable for $X$ then so is $\sO$. Conversely, if $\sO$ is not orientable for $X$, then neither is~$\sO'$.

We will be particularly interested in the case in which $\cC,\sO$ control orientations on moduli spaces in some geometric problem we care about in 7 or 8 dimensions, as in \S\ref{fm12}--\S\ref{fm14}, and $\sO'$ is $\sH_7^\Z$ or $\sH_8^\Z$ from Definition \ref{fm9def7}. Then an orientation for $\sO'$ is a {\it flag structure\/} in the sense of~\S\ref{fm10}.

\subsubsection{Factorizing gauge theory orientation functors}
\label{fm1121}

The next definition will be useful for comparing gauge theory orientation problems for different Lie groups.

\begin{dfn}
\label{fm11def1}
Let $\io:G\ra H$ be a morphism of Lie groups, with induced Lie algebra morphism $\io_*:\g\ra\h$. We say that $\io:G\ra H$ is {\it of complex type\/} if $\io_*:\g\ra\h$ is injective, and the quotient $G$-representation $\m=\h/\io_*(\g)$ is of complex type, that is, the real vector space $\m$ may be made into a complex vector space such that the action of $G$ on $\m$ is complex linear.
\end{dfn}

As in \cite[Prop.~3.9]{Upme2}, the next proposition follows easily from \cite[\S 2.2]{JTU}. The reason for the reduction to $\Z_2$ in \eq{fm11eq8} is that $\sO_n^{\bs B,G}$ and $\sO_n^{\bs B,H}\ci F_\io$ involve indices of {\it real} elliptic operators $\sD_X^{\nabla_P},\sD_X^{\nabla_Q}$ whose kernels and cokernels differ by {\it complex\/} vector spaces, so in particular, $\ind_\R(\sD_X^{\nabla_P})\equiv\ind_\R(\sD_X^{\nabla_Q})\mod 2$, and orientations on the (co)kernels of $\sD_X^{\nabla_P}$ and $\sD_X^{\nabla_Q}$ can be identified.

\begin{prop}
\label{fm11prop1}
Work in the situation of Definitions\/ {\rm\ref{fm9def4}} and\/ {\rm\ref{fm11def1}}.
\smallskip

\noindent{\bf(a)} Let\/ $\io:G\ra H$ be a morphism of Lie groups of complex type. Then for\/ $F_\io$ as in \eq{fm4eq9} there exists a canonical monoidal natural isomorphism\/ $\ep_{n,G}^{\bs B,H}$ making the following diagram commute:
\e
\begin{gathered}
\xymatrix@!0@C=60pt@R=30pt{
*+[r]{\Bord^{\bs B}_n(BG)} \ar[rrrr]^(0.48){F_\io} \ar[d]^{\sO_n^{\bs B,G}} & \drrtwocell_{}\omit^{}\omit{_{\,\,\,\,\,\,\,\,\,\,\,\,\ep_{n,G}^{\bs B,H}}} &&& *+[l]{\Bord^{\bs B}_n(BH)} \ar[d]_{\sO_n^{\bs B,H}}  \\
*+[r]{\Ga_n\qs\Ga_{n+1}} \ar[rr] && \Z_2\qs\Z_2 && *+[l]{\Ga_n\qs\Ga_{n+1}.} \ar[ll] }
\end{gathered}
\label{fm11eq8}
\e
Here the functors $\Ga_n\qs\Ga_{n+1}\ra\Z_2\qs\Z_2$ are induced by the obvious morphisms $\Ga_k\ra\Z_2,$ which are nontrivial if\/ $k\equiv 0,1,2,4\mod 8$. As in Remark\/ {\rm\ref{fm9rem1},} we can use \eq{fm11eq8} to compare gauge theory orientation problems with groups $G,H$.

\smallskip

\noindent{\bf(b)} Let\/ $G_1,G_2$ be Lie groups. Then there exists a canonical natural isomorphism\/ $\ze_{n,G_1,G_2}^{\bs B}$ making the following diagram commute:
\begin{equation*}
\xymatrix@C=170pt@R=15pt{
*+[r]{\Bord^{\bs B}_n(B(G_1\t G_2))} \ar[r]_(0.7){\sO_n^{\bs B,G_1\t G_2}} \ar[d]^{(F_{\Pi_{G_1}},F_{\Pi_{G_2}})} \drtwocell_{}\omit^{}\omit{^{\ze_{n,G_1,G_2}^{\bs B}\,\,\,\,\,\,\,\,\,\,\,\,\,\,\,\,\,}} & *+[l]{\Ga_n\qs\Ga_{n+1}}  \\
*+[r]{\Bord^{\bs B}_n(BG_1)\t\Bord^{\bs B}_n(BG_2)} \ar[r]^(0.62){\sO_n^{\bs B,G_1}\t\sO_n^{\bs B,G_2}} & *+[l]{\Ga_n\qs\Ga_{n+1}\!\t\!\Ga_n\qs\Ga_{n+1}.} \ar[u]^\ot }
\end{equation*}

\noindent{\bf(c)} The analogues of\/ {\bf(a)\rm,\bf(b)} hold with\/ $\sN_n^{\bs B,G}$ in place of\/ $\sO_n^{\bs B,G}$ throughout.
\smallskip

\noindent{\bf(d)} When $G$ is an abelian Lie group there is a canonical monoidal natural isomorphism from $\sN_n^{\bs B,G}:\Bord^{\bs B}_n(BG)\ra\Ga_n\qs\Ga_{n+1}$ to the constant functor\/~$\boo$.
\end{prop}

The next proposition follows easily from Proposition \ref{fm11prop1} and Theorem \ref{fm9thm1}.
 
\begin{prop}
\label{fm11prop2}
In Theorem\/ {\rm\ref{fm9thm1},} take the orientation functor\/ $\sO$ to be\/ $\sO_n^{\bs B,G}$ from Definition\/ {\rm\ref{fm9def4}}. Then the morphisms\/ $\Xi_{n,\sO_n^{\bs B,G}}^{\bs B,G}$ in Theorem\/ {\rm\ref{fm9thm1}} satisfy:
\begin{itemize}
\setlength{\itemsep}{0pt}
\setlength{\parsep}{0pt}
\item[{\bf(a)}] Let\/ $\io:G\ra H$ be a morphism of Lie groups of complex type, in the sense of Definition\/ {\rm\ref{fm11def1}}. Then the following diagram commutes:
\begin{equation*}
\xymatrix@C=70pt@R=15pt{
*+[r]{\Om^{\bs B}_n(\cL BG;BG)} \ar[rr]_{B\io_\rel^{\bs B}} \ar[d]^{\Xi_{n,\sO_n^{\bs B,G}}^{\bs B,G}} && *+[l]{\Om^{\bs B}_n(\cL BH;BH)} \ar[d]_{\Xi_{n,\sO_n^{\bs B,H}}^{\bs B,H}} \\
*+[r]{\Ga_{n+1}} \ar[r] & \Z_2 & *+[l]{\Ga_{n+1}.\!} \ar[l] }	
\end{equation*}
\item[{\bf(b)}] Let\/ $G_1,G_2$ be Lie groups. Then the following diagram commutes:
\begin{equation*}
\xymatrix@C=186pt@R=20pt{
*+[r]{\Om^{\bs B}_n(\cL B(G_1\t G_2);B(G_1\t G_2))} \ar[r]_(0.65){\Xi_{n,\sO_n^{\bs B,G_1\t G_2}}^{\bs B,G_1\t G_2}} \ar[d]^{((B\Pi_{G_1})^{\bs B}_\rel,(B\Pi_{G_2})^{\bs B}_\rel)} & *+[l]{\Ga_{n+1}}  \\
*+[r]{\Om^{\bs B}_n(\cL BG_1;BG_1)\t\Om^{\bs B}_n(\cL BG_2;BG_2)} \ar[r]^(0.67){\Xi_{n,\sO_n^{\bs B,G_1}}^{\bs B,G_1}\t\Xi_{n,\sO_n^{\bs B,G_2}}^{\bs B,G_2}} & *+[l]{\Ga_{n+1}\t\Ga_{n+1}.\!} \ar[u]^+ }	
\end{equation*}
\end{itemize}	
\end{prop}

To apply Propositions \ref{fm11prop1} and \ref{fm11prop2} it will be helpful to have a list of morphisms $\io:G\ra H$ of complex type. The next theorem will be proved in~\S\ref{fm195}.

\begin{thm}
\label{fm11thm3}
Here is a list of Lie group morphisms\/ $\io:G\ra H$ of complex type, as in Definition\/ {\rm\ref{fm11def1},} for all\/~{\rm$m\ge 1$:}
\ea
\begin{aligned}
E_7\!\t\!\U(1)&\longra E_8, & \!\!\!\!\!E_6\!\t\!\U(1)^2&\longra E_8, & \!\!\!\!\!\Spin(14)\!\t\!\U(1)&\longra E_8, \\
\!\!\SU(8)\!\t\!\U(1)&\longra E_8,  &  \!\!\!\!\!\!\!\!\!\!\Sp(3)\!\t\!\U(1)&\longra F_4, & \!\!\!\!\!\Spin(7)\!\t\!\U(1)&\longra F_4, \\
G_2&\longra\Spin(8),\!\!\!\!\! & \U(m)&\longra\SU(m\!+\!1),\!\!\!\!\!\!\!\!\!\! & \Spin(m)&\longra\SO(m),\!\!\!\!\!\!\!\!\!\!\!\!\!
\end{aligned}
\label{fm11eq9}\\
\begin{aligned}
\SU(m)\t\U(1)&\longra\SU(m+1),  & \Sp(m)\t\U(1)&\longra\Sp(m+1), \\
\SO(m)\t\U(1)&\longra\SO(m+2), & \Spin(m)\t\U(1)&\longra\Spin(m+2).
\end{aligned}
\nonumber 
\ea
\end{thm}

Here we do not specify the actual morphisms $\io$, although these are implicit in the proof, as we will not need them later. To prove Theorem \ref{fm11thm3}, we show:
\begin{itemize}
\setlength{\itemsep}{0pt}
\setlength{\parsep}{0pt}
\item[(i)] Suppose a Lie group $H$ has a torus subgroup $T\subseteq H$, and write $G=Z(T)$ for the centralizer of $T$. Then $\inc:G\hookra H$ is of complex type.
\item[(ii)] Let $\io:G\ra H$ be a morphism of connected Lie groups which is a covering map, e.g. $\Spin(n)\,{\buildrel 2:1\over\longra}\,\SO(n)$. Then $\io$ is of complex type.
\item[(iii)] Compositions of complex type morphisms are of complex type.
\end{itemize}
Using these we can easily construct many examples of complex type morphisms.

\subsubsection{Factorizing submanifold orientation functors}
\label{fm1122}

\begin{prop}
\label{fm11prop3}
Work in the situation of Definition\/ {\rm\ref{fm9def6}}. Suppose we are given a composition of Lie groups $H_1\,{\buildrel\io\over\longra}\, H_2\,{\buildrel\rho_2\over\longra}\,\SO(4)\subset\O(4)$. Then for\/ $F_\io$ as in \eq{fm5eq3} there exist canonical monoidal natural isomorphisms\/ $\ep_{n,n-4,H_1}^{\bs B,H_2,*}$ for $*=+,-,0$ making the following diagram commute:
\e
\begin{gathered}
\xymatrix@!0@C=60pt@R=30pt{
*+[r]{\Bord_{n,4}^{\bs B}(MH_1)} \ar[rrrr]^(0.48){F_\io} \ar@/_1pc/[drrrr]_(0.35){\sO_{n,4}^{\bs B,H_1,*}} & \drrtwocell_{}\omit^{}\omit{_{\,\,\,\,\,\,\,\,\,\,\,\,\,\,\,\,\,\,\,\,\,\,\,\,\ep_{n,n-4,H_1}^{\bs B,H_2,*}}} &&& *+[l]{\Bord_{n,4}^{\bs B}(MH_2)} \ar[d]_{\sO_{n,4}^{\bs B,H_2,*}}  \\
&&&& *+[l]{\Ga_n\qs\Ga_{n+1}.}
}
\end{gathered}
\label{fm11eq10}
\e
As in Remark\/ {\rm\ref{fm9rem1},} we can use \eq{fm11eq8} to compare submanifold orientation problems with normal orientations $H_1,H_2$.
\end{prop}

\begin{proof}
The definition of $\sO_{n,4}^{\bs B,H,*}$ in Definition \ref{fm9def6} did not use the normal $H$-structures except to require that $\rho:H\ra\O(4)$ factors via $\SO(4)$, and use the induced orientations of $\nu_M\ra M$, $\nu_N\ra N$. These are unchanged by $F_\io$. So the proposition holds with $\ep_{n,n-4,H_1}^{\bs B,H_2,*}$ the identity natural isomorphism.
\end{proof}

\subsubsection{Factorizations mixing classes of orientation functors}
\label{fm1123}

To prove the next theorem, we observe using Definition \ref{fm9def4} and Theorems \ref{fm3thm1}, \ref{fm11thm1}, and \ref{fm11thm2} that for each triangle in \eq{fm11eq11}--\eq{fm11eq18}, the two routes round the triangle agree on the generators of $\pi_i(\cC)$ for $i=0,1$, where $\cC$ is the Picard groupoid in the top left hand corner of the triangle, and then we apply Theorem \ref{fmAthm2} to deduce the existence of the monoidal natural isomorphisms $\la_*^*$, and Table \ref{fm2tab1} and Theorem \ref{fm3thm1} to describe the torsor in which $\la_*^*$ lives.

\begin{thm}
\label{fm11thm4}
{\bf(a)} As in equation\/ {\rm\eq{fm11eq6},} there exist monoidal natural isomorphisms $\la_1,\ldots,\la_8$ in the following diagrams of Picard groupoids:
\begin{gather}
\begin{gathered}
\text{\begin{small}$\displaystyle
\!\!\!\!\!\!\!\!\xymatrix@!0@C=70pt@R=30pt{ *+[r]{\Bord_{7,4}^{\bs\Spin}(M\SU(2))} \drrtwocell_{}\omit^{}\omit{^{\la_1\,\,\,}} \ar@<.8ex>[d]^{F_{M\SU(2)}^{B\SU(2)}}_\simeq \ar@/^1.6pc/[drr]^(0.75){\sO_{7,4}^{\bs\Spin,\SU(2),0}}  \\
*+[r]{\Bord_7^{\bs\Spin}(B\SU(2))} \ar[rr]^(0.7){\sN_{7}^{\bs\Spin,\SU(2)}} && *+[l]{0\qs\Z,} }
\;
\xymatrix@!0@C=70pt@R=30pt{ *+[r]{\Bord_{8,4}^{\bs\Spin}(M\SU(2))} \drrtwocell_{}\omit^{}\omit{^{\la_2\,\,\,}} \ar@<.8ex>[d]^{F_{M\SU(2)}^{B\SU(2)}}_\simeq \ar@/^1.6pc/[drr]^(0.75){\sO_{8,4}^{\bs\Spin,\SU(2),0}}  \\
*+[r]{\Bord_8^{\bs\Spin}(B\SU(2))} \ar[rr]^(0.7){\sN_{8}^{\bs\Spin,\SU(2)}} && *+[l]{\Z\qs\Z_2,} }\!\!\!\!\!\!\!\!\!\!\!\!$\end{small}}
\end{gathered}
\label{fm11eq11}
\allowdisplaybreaks\\
\begin{gathered}
\text{\begin{small}$\displaystyle
\!\!\!\!\!\!\!\!\xymatrix@!0@C=70pt@R=30pt{ *+[r]{\Bord_{7,4}^{\bs\Spin}(M\U(2))} \drrtwocell_{}\omit^{}\omit{^{\la_3\,\,\,}} \ar@<.8ex>[d]^{F^{B\SU(m)}_{M\U(2)}}_\simeq \ar@/^1.6pc/[drr]^(0.75){\sO_{7,4,\Z_2}^{\bs\Spin,\U(2),0}}  \\
*+[r]{\Bord_7^{\bs\Spin}(B\SU(m))} \ar[rr]^(0.7){\sN_{7,\Z_2}^{\bs\Spin,\SU(m)}} && *+[l]{0\qs\Z_2,} }
\;
\xymatrix@!0@C=70pt@R=30pt{ *+[r]{\Bord_{8,4}^{\bs\Spin}(M\U(2))} \drrtwocell_{}\omit^{}\omit{^{\la_4\,\,\,}} \ar@<.8ex>[d]^{F^{B\SU(m)}_{M\U(2)}}_\simeq \ar@/^1.6pc/[drr]^(0.75){\sO_{8,4,\Z_2}^{\bs\Spin,\U(2),0}}  \\
*+[r]{\Bord_8^{\bs\Spin}(B\SU(m))} \ar[rr]^(0.7){\sN_{8,\Z_2}^{\bs\Spin,\SU(m)}} && *+[l]{\Z_2\qs\Z_2,} }\!\!\!\!\!\!\!\!\!\!\!\!$\end{small}}
\end{gathered}
\label{fm11eq12}\\
\begin{gathered}
\text{\begin{small}$\displaystyle
\!\!\!\!\!\!\!\!\xymatrix@!0@C=70pt@R=30pt{ *+[r]{\Bord_{7,4}^{\bs\Spin}(M\Spin(4))} \drrtwocell_{}\omit^{}\omit{^{\la_5\,\,\,}} \ar@<.8ex>[d]^{F^{B\Sp(m)}_{M\Spin(4)}}_\simeq \ar@/^1.6pc/[drr]^(0.75){\sO_{7,4,\Z_2}^{\bs\Spin,\Spin(4),0}}  \\
*+[r]{\Bord_7^{\bs\Spin}(B\Sp(m))} \ar[rr]^(0.7){\sN_{7,\Z_2}^{\bs\Spin,\Sp(m)}} && *+[l]{0\qs\Z_2,} }
\;
\xymatrix@!0@C=70pt@R=30pt{ *+[r]{\Bord_{8,4}^{\bs\Spin}(M\Spin(4))} \drrtwocell_{}\omit^{}\omit{^{\la_6\,\,\,}} \ar@<.8ex>[d]^{F^{B\Sp(m)}_{M\Spin(4)}}_\simeq \ar@/^1.6pc/[drr]^(0.75){\sO_{8,4,\Z_2}^{\bs\Spin,\Spin(4),0}}  \\
*+[r]{\Bord_8^{\bs\Spin}(B\Sp(m))} \ar[rr]^(0.7){\sN_{8,\Z_2}^{\bs\Spin,\Sp(m)}} && *+[l]{\Z_2\qs\Z_2,} }\!\!\!\!\!\!\!\!\!\!\!\!$\end{small}}
\end{gathered}
\label{fm11eq13}\\
\begin{gathered}
\text{\begin{small}$\displaystyle
\!\!\!\!\!\!\!\!\xymatrix@!0@C=70pt@R=30pt{ *+[r]{\Bord_7^{\bs\Spin}(BE_8)} \drrtwocell_{}\omit^{}\omit{^{\la_7\,\,\,}} \ar@<.8ex>[d]^{F_{BE_8}^{K(\Z,4)}}_\simeq \ar@/^1.4pc/[drr]^(0.75){\sN_{7,\Z_2}^{\bs\Spin,E_8}}  \\
*+[r]{\Bord_7^{\bs\Spin}(K(\Z,4))} \ar[rr]^(0.7){\sH_7^\Z} && *+[l]{0\qs\Z_2,} }
\;
\xymatrix@!0@C=70pt@R=30pt{ *+[r]{\Bord_8^{\bs\Spin}(BE_8)} \drrtwocell_{}\omit^{}\omit{^{\la_8\,\,\,}} \ar@<.8ex>[d]^{F_{BE_8}^{K(\Z,4)}}_\simeq \ar@/^1.4pc/[drr]^(0.75){\sN_{8,\Z_2}^{\bs\Spin,E_8}}  \\
*+[r]{\Bord_8^{\bs\Spin}(K(\Z,4))} \ar[rr]^(0.7){\sH_8^\Z} && *+[l]{\Z_2\qs\Z_2.} }\!\!\!\!\!\!\!\!\!\!\!\!$\end{small}}
\end{gathered}
\label{fm11eq14}
\end{gather}
Here in \eq{fm11eq12} we have $m\ge 4$ or $m\ge 5$, and in \eq{fm11eq13} we have $m\ge 2$. In each diagram, the left hand column is a transfer functor from Theorem\/ {\rm\ref{fm8thm1},} which is an equivalence of categories, and the rightwards morphisms are orientation functors defined in\/ {\rm\S\ref{fm93}--\S\ref{fm95}}. 

The monoidal natural isomorphisms $\la_1,\la_3,\la_5,\la_7$ are unique. Also\/ $\la_2,\ab\la_4,\ab\la_6,\ab\la_8$ lie in torsors over $\Z_2^4,\Z_2^5,\Z_2^5,\Z_2^4$ respectively, although if we require them to commute with the functors \eq{fm8eq3} to and from $\Bord^{\bs\Spin}_8(*),$ the choices are $\Z_2^2,\ab\Z_2^3,\ab\Z_2^3,\ab\Z_2^2$. These choices of\/ $\la_i$ are a kind of orientation convention.
\smallskip

\noindent{\bf(b)} For $*=-,0,$ but \begin{bfseries}not\end{bfseries} for $*=+,$ there exist monoidal natural isomorphisms $\la_9^*,\ldots,\la_{14}^*$ in the following diagrams of Picard groupoids:
\ea
\begin{gathered}
\text{\begin{small}$\displaystyle
\!\!\!\!\!\!\!\!\xymatrix@!0@C=70pt@R=30pt{ *+[r]{\Bord_7^{\bs\Spin}(M\SU(2))} \drrtwocell_{}\omit^{}\omit{^{\la_9^*\,\,\,}} \ar@<.8ex>[d]^{F_{M\SU(2)}^{K(\Z,4)}} \ar@/^1.6pc/[drr]^(0.75){\sO_{7,4,\Z_2}^{\bs\Spin,\SU(2),*}}  \\
*+[r]{\Bord_7^{\bs\Spin}(K(\Z,4))} \ar[rr]^(0.7){\sH_7^\Z} && *+[l]{0\qs\Z_2,} }
\;
\xymatrix@!0@C=70pt@R=30pt{ *+[r]{\Bord_8^{\bs\Spin}(M\SU(2))} \drrtwocell_{}\omit^{}\omit{^{\la_{10}^*\,\,\,\,\,}} \ar@<.8ex>[d]^{F_{M\SU(2)}^{K(\Z,4)}} \ar@/^1.6pc/[drr]^(0.75){\sO_{8,4,\Z_2}^{\bs\Spin,\SU(2),*}}  \\
*+[r]{\Bord_8^{\bs\Spin}(K(\Z,4))} \ar[rr]^(0.7){\sH_8^\Z} && *+[l]{\Z_2\qs\Z_2,} }\!\!\!\!\!\!\!\!\!\!\!\!$\end{small}}
\end{gathered}
\label{fm11eq15}\\
\begin{gathered}
\text{\begin{small}$\displaystyle
\!\!\!\!\!\!\!\!\xymatrix@!0@C=70pt@R=30pt{ *+[r]{\Bord_7^{\bs\Spin}(M\U(2))} \drrtwocell_{}\omit^{}\omit{^{\la_{11}^*\,\,\,\,\,}} \ar@<.8ex>[d]^{F_{M\U(2)}^{K(\Z,4)}} \ar@/^1.6pc/[drr]^(0.75){\sO_{7,4,\Z_2}^{\bs\Spin,\U(2),*}}  \\
*+[r]{\Bord_7^{\bs\Spin}(K(\Z,4))} \ar[rr]^(0.7){\sH_7^\Z} && *+[l]{0\qs\Z_2,} }
\;
\xymatrix@!0@C=70pt@R=30pt{ *+[r]{\Bord_8^{\bs\Spin}(M\U(2))} \drrtwocell_{}\omit^{}\omit{^{\la_{12}^*\,\,\,\,\,}} \ar@<.8ex>[d]^{F_{M\U(2)}^{K(\Z,4)}} \ar@/^1.6pc/[drr]^(0.75){\sO_{8,4,\Z_2}^{\bs\Spin,\U(2),*}}  \\
*+[r]{\Bord_8^{\bs\Spin}(K(\Z,4))} \ar[rr]^(0.7){\sH_8^\Z} && *+[l]{\Z_2\qs\Z_2,} }\!\!\!\!\!\!\!\!\!\!\!\!$\end{small}}
\end{gathered}
\label{fm11eq16}\\
\begin{gathered}
\text{\begin{small}$\displaystyle
\!\!\!\!\!\!\!\!\xymatrix@!0@C=70pt@R=30pt{ *+[r]{\Bord_7^{\bs\Spin}(M\Spin(4))} \drrtwocell_{}\omit^{}\omit{^{\la_{13}^*\,\,\,\,\,}} \ar@<.8ex>[d]^{F_{M\Spin(4)}^{K(\Z,4)}} \ar@/^1.6pc/[drr]^(0.75){\sO_{7,4,\Z_2}^{\bs\Spin,\Spin(4),*}}  \\
*+[r]{\Bord_7^{\bs\Spin}(K(\Z,4))} \ar[rr]^(0.7){\sH_7^\Z} && *+[l]{0\qs\Z_2,} }
\;
\xymatrix@!0@C=70pt@R=30pt{ *+[r]{\Bord_8^{\bs\Spin}(M\Spin(4))} \drrtwocell_{}\omit^{}\omit{^{\la_{14}^*\,\,\,\,\,}} \ar@<.8ex>[d]^{F_{M\Spin(4)}^{K(\Z,4)}} \ar@/^1.6pc/[drr]^(0.75){\sO_{8,4,\Z_2}^{\bs\Spin,\Spin(4),*}}  \\
*+[r]{\Bord_8^{\bs\Spin}(K(\Z,4))} \ar[rr]^(0.7){\sH_8^\Z} && *+[l]{\Z_2\qs\Z_2.} }\!\!\!\!\!\!\!\!\!\!\!\!$\end{small}}
\end{gathered}
\label{fm11eq17}
\ea

For $*=0,$ but \begin{bfseries}not\end{bfseries} for\/  $*=+,-,$ there exists a monoidal natural isomorphism $\la_{15}^0$ in the following diagram of Picard groupoids:
\e
\begin{gathered}
\xymatrix@C=100pt@R=15pt{ *+[r]{\Bord_7^{\bs\Spin}(M\SO(4))} \drrtwocell_{}\omit^{}\omit{^{\la_{15}^0\,\,\,\,\,}} \ar@<.8ex>[d]^{F_{M\SO(4)}^{K(\Z,4)}} \ar@/^1.6pc/[drr]^(0.75){\sO_{7,4,\Z_2}^{\bs\Spin,\SO(4),0}}  \\
*+[r]{\Bord_7^{\bs\Spin}(K(\Z,4))} \ar[rr]^(0.7){\sH_7^\Z} && *+[l]{0\qs\Z_2.} }
\end{gathered}
\label{fm11eq18}
\e
The analogue of\/ \eq{fm11eq18} for $\sO_{8,4,\Z_2}^{\bs\Spin,\SO(4),*}$ does not commute for any\/~$*=+,-,0$.

The monoidal natural isomorphisms $\la_9^*,\la_{11}^*,\la_{13}^*,\la_{15}^0$ are unique. Also\/ $\la_{10}^*,\ab\la_{12}^*,\ab\la_{14}^*$ lie in torsors over $\Z_2^4,\Z_2^5,\Z_2^5$ respectively, although if we require them to commute with the functors \eq{fm8eq3} the choices are\/~$\Z_2^2,\Z_2^3,\Z_2^3$. 
\end{thm}

\begin{thm}
\label{fm11thm5}
{\bf(a)} The following orientation functors from {\rm\S\ref{fm93}--\S\ref{fm95}} factor via the orientation functor $\sH_7^\Z:\Bord_7^{\bs\Spin}(K(\Z,4))\ra 0\qs\Z_2$ from Definition\/ {\rm\ref{fm9def4},} which is used to define flag structures on $7$-manifolds in\/~{\rm\S\ref{fm101}:}
\begin{itemize}
\setlength{\itemsep}{0pt}
\setlength{\parsep}{0pt}
\item[{\bf(i)}] $\sN_{7,\Z_2}^{\bs\Spin,G}:\Bord_7^{\bs\Spin}(BG)\ra 0\qs\Z_2$ for\/ $G$ any of the following compact, connected Lie groups, where $E_6,E_7$ are the simply-connected versions:
\e
E_8,\; E_7,\; E_6,\; G_2,\; \SU(m),\; \U(m),\; \Spin(2m), \quad \text{for $m\ge 1$.}
\label{fm11eq19}
\e
\item[{\bf(ii)}] $\sO_{7,4,\Z_2}^{\bs\Spin,H,0}:\Bord_7^{\bs\Spin}(MH)\ra 0\qs\Z_2$ for any $\rho:H\ra\O(4)$ which factors via $\SO(4)\hookra\O(4)$.
\item[{\bf(iii)}] $\sO_{7,4,\Z_2}^{\bs\Spin,H,-}:\Bord_7^{\bs\Spin}(MH)\ra 0\qs\Z_2$ for any $\rho:H\ra\O(4)$ which factors via\/ $\U(2)\hookra\O(4)$ or\/ $\Spin(4)\ra\O(4)$.
\end{itemize}

Hence, the orientation functors in {\bf(a)(i)\rm--\bf(iii)} are orientable for every compact spin $7$-manifold\/ $X,$ and after choosing a natural isomorphism $\la_*^*$ as in\/ {\rm\eq{fm11eq11}--\eq{fm11eq18},} if\/ $X$ is a compact spin\/ $7$-manifold, then a flag structure on $X$ induces an orientation for any one of these orientation functors on\/~$X$.
\smallskip

\noindent{\bf(b)} The following orientation functors from {\rm\S\ref{fm93}--\S\ref{fm95}} factor via the orientation functor $\sH_8^\Z:\Bord_8^{\bs\Spin}(K(\Z,4))\ra\Z_2\qs\Z_2$ from Definition\/ {\rm\ref{fm9def4},} which is used to define flag structures on $8$-manifolds in\/~{\rm\S\ref{fm102}:}
\begin{itemize}
\setlength{\itemsep}{0pt}
\setlength{\parsep}{0pt}
\item[{\bf(i)}] $\sN_{8,\Z_2}^{\bs\Spin,G}:\Bord_8^{\bs\Spin}(BG)\ra \Z_2\qs\Z_2$ for\/ $G$ any of Lie groups in\/~\eq{fm11eq19}.
\item[{\bf(ii)}] $\sO_{8,4,\Z_2}^{\bs\Spin,H,-}$ and\/ $\sO_{8,4,\Z_2}^{\bs\Spin,H,0}:\Bord_8^{\bs\Spin}(MH)\ra \Z_2\qs\Z_2$ for any $\rho:H\ra\O(4)$ which factors via\/ $\U(2)\hookra\O(4)$ or\/ $\Spin(4)\ra\O(4)$.
\end{itemize}
Hence the orientation functors in {\bf(b)(i)\rm--\bf(ii)} are orientable for every compact spin $8$-manifold\/ $X$ satisfying the condition Theorem\/ {\rm\ref{fm10thm2}$(*)$,} and after choosing a natural isomorphism $\la_*^*$ as in\/ {\rm\eq{fm11eq11}--\eq{fm11eq17},} if\/ $X$ is a compact spin\/ $8$-manifold, then a flag structure on $X$ induces an orientation for any one of these orientation functors on\/~$X$.
\end{thm}

\begin{proof}
For (a)(i), the case $G=E_8$ follows from \eq{fm11eq14}. Proposition \ref{fm11prop1}(a),(c) and Theorem \ref{fm11thm3} imply that $\sN_{7,\Z_2}^{\bs\Spin,G}$ factors via $\sN_{7,\Z_2}^{\bs\Spin,E_8}$ for $G$ in the list
\begin{equation*}
E_7\t\U(1),\quad E_6\t\U(1)^2,\quad\SU(8)\t\U(1),\quad \Spin(14)\t\U(1).
\end{equation*}
Hence $\sN_{7,\Z_2}^{\bs\Spin,G}$ for these $G$ also factor via $\sH_7^\Z$. From Proposition \ref{fm11prop1}(b)--(d) we deduce that $\sN_{7,\Z_2}^{\bs\Spin,G\t\U(1)^k}$ factors via $\sH_7^\Z$ if and only if $\sN_{7,\Z_2}^{\bs\Spin,G}$ does. Thus (i) follows for $G=E_7,E_6,\SU(8),\Spin(14)$.

From the complex type morphisms $\SU(m)\t\U(1)\ra\SU(m+1)$ and $\Spin(m)\ab\t\ab\U(1)\ra\Spin(m+2)$ in Theorem \ref{fm11thm3}, and (i) for $\SU(8),\Spin(14)$, we deduce (i) for $G=\SU(m)$ with $1\le m\le 8$ and $G=\Spin(2m)$ with $1\le m\le 7$ by the arguments above. Then we deduce (i) for $G=G_2$ from the complex type morphism $G_2\ra\Spin(8)$ and (i) for $\Spin(8)$, and we deduce (i) for $G=\U(m)$ with $1\le m\le 7$ from the complex type morphism $\U(m)\ra\SU(m+1)$ and (i) for $\SU(m)$ for $2\le m\le 8$.

The argument above using the complex type morphism $\SU(m)\t\U(1)\hookra\SU(m+1)$ tells us that if (i) holds for $G=\SU(m+1)$, then (i) holds for $G=\SU(m)$. Now the functor $F_\io:\Bord_7^{\bs\Spin}(B\SU(m))\ra \Bord_7^{\bs\Spin}(B\SU(m+1))$ induced by the inclusion $\io:\SU(m)\hookra\SU(m+1)$ is an equivalence of categories provided that $m\ge 5$. Because of this, if $m\ge 5$ then (i) holds for $G=\SU(m)$ if and only if (i) holds for $G=\SU(m+1)$. So as (i) holds for $G=\SU(8)$, by induction on $m$ it holds for $\SU(m)$ for all $m\ge 8$. A similar proof using $\Spin(2m)\hookra\Spin(2m+2)$ shows that (i) holds for $G=\Spin(2m)$ for all $m\ge 7$. This completes~(a)(i).

Part (a)(ii) follows from equations \eq{fm11eq18} and \eq{fm11eq10} for $\rho:H\ra\SO(4)$. Part (a)(iii) follows from \eq{fm11eq16}--\eq{fm11eq17} for $*=-$ and \eq{fm11eq10} in the same way. The last part of (a) is immediate from equation \eq{fm11eq7} and \S\ref{fm101}. The proof of (b) is very similar to that of (a), using Theorem \ref{fm10thm2}(a).
\end{proof}

\section[Applications to moduli spaces in gauge theory]{Applications to moduli spaces in \\ gauge theory}
\label{fm12}

\subsection{\texorpdfstring{Connection moduli spaces $\A_P,\B_P$ and orientations}{Connection moduli spaces 𝒜ₚ,ℬₚ and orientations}}
\label{fm121}

The following definitions are taken from Joyce, Tanaka and Upmeier~\cite[\S 1--\S 2]{JTU}.

\begin{dfn}
\label{fm12def1}
Suppose we are given the following data:
\begin{itemize}
\setlength{\itemsep}{0pt}
\setlength{\parsep}{0pt}
\item[(a)] A compact, connected manifold $X$ of dimension $n>0$.
\item[(b)] A Lie group $G$, with $\dim G>0$, and centre $Z(G)\subseteq G$, and Lie algebra $\g$.
\item[(c)] A principal $G$-bundle $\pi:P\ra X$. We write $\Ad(P)\ra X$ for the vector bundle with fibre $\g$ defined by $\Ad(P)=(P\t\g)/G$, where $G$ acts on $P$ by the principal bundle action, and on $\g$ by the adjoint action.
\end{itemize}

Write $\A_P$ for the set of connections $\nabla_P$ on the principal bundle $P\ra X$. This is a real affine space modelled on the infinite-dimensional vector space $\Ga^\iy(\Ad(P)\ot T^*X)$, and we make $\A_P$ into a topological space using the $C^\iy$ topology on $\Ga^\iy(\Ad(P)\ot T^*X)$. Here if $E\ra X$ is a vector bundle then $\Ga^\iy(E)$ denotes the vector space of smooth sections of $E$. Note that $\A_P$ is contractible.
 
Write $\G_P=\Aut(P)$ for the infinite-dimensional Lie group of $G$-equivariant diffeomorphisms $\ga:P\ra P$ with $\pi\ci\ga=\pi$. Then $\G_P$ acts on $\A_P$ by gauge transformations, and the action is continuous for the topology on~$\A_P$. 

There is an inclusion $Z(G)\hookra\G_P$ mapping $z\in Z(G)$ to the principal bundle action of $z$ on $P$. This maps $Z(G)$ into the centre $Z(\G_P)$ of $\G_P$, so we may take the quotient group $\G_P/Z(G)$. The action of $Z(G)\subset\G_P$ on $\A_P$ is trivial, so the $\G_P$-action on $\A_P$ descends to a $\G_P/Z(G)$-action. 

Each $\nabla_P\in\A_P$ has a (finite-dimensional) {\it stabilizer group\/} $\Stab_{\G_P}(\nabla_P)\subset\G_P$ under the $\G_P$-action on $\A_P$, with $Z(G)\subseteq\Stab_{\G_P}(\nabla_P)$. As $X$ is connected, $\Stab_{\G_P}(\nabla_P)$ is isomorphic to a closed Lie subgroup $H$ of $G$ with $Z(G)\subseteq H$. As in \cite[p.~133]{DoKr} we call $\nabla_P$ {\it irreducible\/} if $\Stab_{\G_P}(\nabla_P)=Z(G)$, and {\it reducible\/} otherwise. Write $\A_P^\irr,\A_P^\red$ for the subsets of irreducible and reducible connections in $\A_P$. Then $\A_P^\irr$ is open and dense in $\A_P$, and $\A_P^\red$ is closed and of infinite codimension in the infinite-dimensional affine space $\A_P$.

We write $\B_P=[\A_P/(\G_P/Z(G))]$ for the moduli space of gauge equivalence classes of connections on $P$, considered as a {\it topological stack\/} in the sense of Metzler \cite{Metz} and Noohi \cite{Nooh1,Nooh2}. Write $\B_P^\irr=[\A_P^\irr/(\G_P/Z(G))]$ for the substack $\B_P^\irr\subseteq\B_P$ of irreducible connections. As $\G_P/Z(G)$ acts freely on $\A_P^\irr$, we may consider $\B_P^\irr$ as a topological space (which is an example of a topological stack).
\end{dfn}

We define (n-)orientation bundles $O^{E_\bu}_P,N^{E_\bu}_P$ on the moduli spaces~$\B_P$:

\begin{dfn} 
\label{fm12def2}	
Work in the situation of Definition \ref{fm12def1}, with the same notation. Suppose we are given real vector bundles $E_0,E_1\ra X$, of the same rank $r$, and a linear elliptic partial differential operator $D:\Ga^\iy(E_0)\ra\Ga^\iy(E_1)$, of degree $d$. As a shorthand we write $E_\bu=(E_0,E_1,D)$. With respect to connections $\nabla_{E_0}$ on $E_0\ot\bigot^iT^*X$ for $0\le i<d$, when $e\in\Ga^\iy(E_0)$ we may write
\e
D(e)=\ts\sum_{i=0}^d a_i\cdot \nabla_{E_0}^ie,
\label{fm12eq1}
\e
where $a_i\in \Ga^\iy(E_0^*\ot E_1\ot S^iTX)$ for $i=0,\ldots,d$. The condition that $D$ is {\it elliptic\/} is that $a_d\vert_x\cdot\ot^d\xi:E_0\vert_x\ra E_1\vert_x$ is an isomorphism for all $x\in X$ and $0\ne\xi\in T_x^*X$, and the {\it symbol\/} $\si(D)$ of $D$ is defined using~$a_d$.

Let $\nabla_P\in\A_P$. Then $\nabla_P$ induces a connection $\nabla_{\Ad(P)}$ on the vector bundle $\Ad(P)\ra X$. Thus we may form the twisted elliptic operator
\e
\begin{split}
D^{\nabla_{\Ad(P)}}&:\Ga^\iy(\Ad(P)\ot E_0)\longra\Ga^\iy(\Ad(P)\ot E_1),\\
D^{\nabla_{\Ad(P)}}&:e\longmapsto \ts\sum_{i=0}^d (\id_{\Ad(P)}\ot a_i)\cdot \nabla_{\Ad(P)\ot E_0}^ie,
\end{split}
\label{fm12eq2}
\e
using the connections $\nabla_{\Ad(P)\ot E_0}$ on $\Ad(P)\ot E_0\ot\bigot^iT^*X$ for $0\le i<d$ induced by $\nabla_{\Ad(P)}$ and~$\nabla_{E_0}$.

Since $D^{\nabla_{\Ad(P)}}$ is a linear elliptic operator on a compact manifold $X$, it has finite-dimensional kernel $\Ker(D^{\nabla_{\Ad(P)}})$ and cokernel $\Coker(D^{\nabla_{\Ad(P)}})$. The {\it determinant\/} $\det(D^{\nabla_{\Ad(P)}})$ is the 1-dimensional real vector space
\begin{equation*}
\det(D^{\nabla_{\Ad(P)}})=\det\Ker(D^{\nabla_{\Ad(P)}})\ot\bigl(\det\Coker(D^{\nabla_{\Ad(P)}})\bigr)^*,
\end{equation*}
where if $V$ is a finite-dimensional real vector space then $\det V=\La^{\dim V}V$. Recall that the index is $\ind_P^{E_\bu}=\dim\Ker(D^{\nabla_{\Ad(P)}})-\dim\Coker(D^{\nabla_{\Ad(P)}})\in\Z.$

These operators $D^{\nabla_{\Ad(P)}}$ vary continuously with $\nabla_P\in\A_P$, so they form a family of elliptic operators over the base topological space $\A_P$. Thus as in Atiyah and Singer \cite{AtSi}, there is a natural real line bundle $\hat L{}^{E_\bu}_P\ra\A_P$ with fibre $\hat L{}^{E_\bu}_P\vert_{\nabla_P}=\det(D^{\nabla_{\Ad(P)}})$ at each $\nabla_P\in\A_P$. It is equivariant under the action of $\G_P/Z(G)$ on $\A_P$, and so pushes down to real line bundles $L^{E_\bu}_P\ra\B_P$ on the topological stacks $\B_P$. We call $L^{E_\bu}_P$ the {\it determinant line bundle\/} of $\B_P$. The restriction $L^{E_\bu}_P\vert_{\B_P^\irr}$ is a topological real line bundle in the usual sense on the topological space~$\B_P^\irr$.

For a real line bundle $L\ra T$ we write $O(L)=(L\sm 0(T))/(0,\iy)$ for the principal $\Z_2$-bundle of (fibrewise) orientations on $L$. That is, we take the complement of the zero section of $L$ and quotient by $(0,\iy)$ acting on the fibres by scalar multiplication.

Define the {\it orientation bundles\/} $\hat O^{E_\bu}_P=\hat O(L^{E_\bu}_P)\ra\A_P$ and $O^{E_\bu}_P=O(L^{E_\bu}_P)\ra\B_P$. The fibres of $O^{E_\bu}_P\ra\B_P$ are orientations on the real line fibres of $L^{E_\bu}_P\ra\B_P$. The restriction $O^{E_\bu}_P\vert_{\B^\irr_P}$ is a principal $\Z_2$-bundle on the topological space $\B^\irr_P$, in the usual sense.

We say that $\B_P$ is {\it orientable\/} if $O^{E_\bu}_P$ is isomorphic to the trivial principal $\Z_2$-bundle $\B_P\t\Z_2\ra\B_P$. An {\it orientation\/} $\om$ on $\B_P$ is an isomorphism $\om:O^{E_\bu}_P\,{\buildrel\cong\over\longra}\,\B_P\t\Z_2$ of principal $\Z_2$-bundles. As $\B_P$ is connected, if $\B_P$ is orientable it has exactly two orientations.

We also define the {\it normalized orientation bundle}, or {\it n-orientation bundle\/} a principal $\Z_2$-bundle $N_P^{E_\bu}\ra\cB_P$, by $N_P^{E_\bu}=O_P^{E_\bu}\ot_{\Z_2}O_{X\t G}^{E_\bu}\vert_{[\nabla^0]}$. That is, we tensor $O_P^{E_\bu}$ with the orientation torsor $O_{X\t G}^{E_\bu}\vert_{[\nabla^0]}$ of the trivial principal $G$-bundle $X\t G\ra X$ at the trivial connection $\nabla^0$. An {\it n-orientation\/} of $\cB_P$ is an isomorphism $\nu:N^{E_\bu}_P\,{\buildrel\cong\over\longra}\,\cB_P\t\Z_2$. Note that $\cB_P$ has an orientation if and only if it has an n-orientation.
\end{dfn}

\begin{rem}
\label{fm12rem1}
{\bf(i)} Up to continuous isotopy, and hence up to isomorphism, $L^{E_\bu}_P,O^{E_\bu}_P$ in Definition \ref{fm12def2} depend on the elliptic operator $D:\Ga^\iy(E_0)\ra\Ga^\iy(E_1)$ up to continuous deformation amongst elliptic operators, and thus only on the {\it symbol\/} $\si(D)$ of $D$ (essentially, the highest order coefficients $a_d$ in \eq{fm12eq1}), up to deformation.
\smallskip

\noindent{\bf(ii)} For orienting moduli spaces of `instantons' in gauge theory, as in \S\ref{fm122}--\S\ref{fm123}, we usually start not with an elliptic operator on $X$, but with an {\it elliptic complex\/}
\e
\smash{\xymatrix@C=28pt{ 0 \ar[r] & \Ga^\iy(E_0) \ar[r]^{D_0} & \Ga^\iy(E_1) \ar[r]^(0.55){D_1} & \cdots \ar[r]^(0.4){D_{k-1}} & \Ga^\iy(E_k) \ar[r] & 0. }}
\label{fm12eq3}
\e
If $k>1$ and $\nabla_P$ is an arbitrary connection on a principal $G$-bundle $P\ra X$ then twisting \eq{fm12eq3} by $(\Ad(P),\nabla_{\Ad(P)})$ as in \eq{fm12eq2} may not yield a complex (that is, we may have $D^{\nabla_{\Ad(P)}}_{i+1}\ci D^{\nabla_{\Ad(P)}}_i\ne 0$), so the definition of $\det(D_\bu^{\nabla_{\Ad(P)}})$ does not work, though it does work if $\nabla_P$ satisfies the appropriate instanton-type curvature condition. To get round this, we choose metrics on $X$ and the $E_i$, so that we can take adjoints $D_i^*$, and replace \eq{fm12eq3} by the elliptic operator
\e
\smash{\xymatrix@C=90pt{ \Ga^\iy\bigl(\bigop_{0\le i\le k/2}E_{2i}\bigr) \ar[r]^(0.48){\sum_i(D_{2i}+D_{2i-1}^*)} & \Ga^\iy\bigl(\bigop_{0\le i< k/2}E_{2i+1}\bigr), }}
\label{fm12eq4}
\e
and then Definitions \ref{fm12def2} works with \eq{fm12eq4} in place of~$E_\bu$.
\end{rem}

\begin{prop}
\label{fm12prop1}
{\bf(a)} In the situation of Definitions\/ {\rm\ref{fm12def1}--\ref{fm12def2},} suppose\/ $(X,g)$ is a compact spin Riemannian $7$-manifold and\/ $E_\bu$ is a first order elliptic operator whose symbol is isomorphic to that of the Dirac operator $\sD_X$ of\/ $X$. Then:
\begin{itemize}
\setlength{\itemsep}{0pt}
\setlength{\parsep}{0pt}
\item[{\bf(i)}] $\B_P$ is orientable for every principal\/ $G$-bundle $P\ra X$ if and only if\/ $\sO_{7,\Z_2}^{\bs\Spin,G}$ (or equivalently, $\sN_{7,\Z_2}^{\bs\Spin,G}$) in Definition\/ {\rm\ref{fm9def4}} is orientable for\/~$X$. 
\item[{\bf(ii)}] An orientation of\/ $\sO_{7,\Z_2}^{\bs\Spin,G}$ for\/ $X$ is equivalent to an orientation of\/ $\B_P$ for all principal\/ $G$-bundles $P\ra X,$ depending on $P$ only up to isomorphism.
\item[{\bf(iii)}] An n-orientation of\/ $\sN_{7,\Z_2}^{\bs\Spin,G}$ for\/ $X$ is equivalent to an n-orientation of\/ $\B_P$ for all principal\/ $G$-bundles $P\ra X,$ up to isomorphisms of\/~$P$.
\end{itemize}

\noindent{\bf(b)} Suppose instead that\/ $(X,g)$ is a compact spin Riemannian $8$-manifold and\/ $E_\bu$ is a first order elliptic operator whose symbol is isomorphic to that of the positive Dirac operator $\sD^+_X$ of\/ $X$. Then the analogues of\/ {\bf(i)\rm--\bf(iii)} hold for $\sO_8^{\bs\Spin,G},\sN_8^{\bs\Spin,G}$ in Definition\/~{\rm\ref{fm9def4}}.
\end{prop}

\begin{proof}
For (a), since orientations of $O^{E_\bu}_P,N^{E_\bu}_P$ are unchanged by continuous deformations of $E_\bu$, we can assume that $E_\bu$ is the Dirac operator $\sD_X$. Now $\sO_{7,\Z_2}^{\bs\Spin,G}(X,P)$ in Definition \ref{fm9def4} is defined to be the $\Z_2$-reduction of the $\Z$-torsor of constant sections of the principal $\Z$-bundle over $\A_{X,P}$ with fibre $\sO_7(D_X^{\nabla_P})$ at $(g_X,\nabla_P)$. So fixing $g_X=g$, $\sO_{7,\Z_2}^{\bs\Spin,G}(X,P)$ is the $\Z_2$-torsor of constant sections of the principal $\Z_2$-bundle over $\A_P$ in Definition \ref{fm12def1} with fibre $\sO_7(D_X^{\nabla_P})\ot_\Z\Z_2$ at $\nabla_P$. But $\sO_7(D_X^{\nabla_P})\ot_\Z\Z_2$ is the $\Z_2$-torsor of orientations of $D_X^{\nabla_P}$. Thus $\sO_{7,\Z_2}^{\bs\Spin,G}(X,P)$ is canonically isomorphic to the $\Z_2$-torsor of trivializations of $\hat O^{E_\bu}_P\ra\A_P$ in Definition~\ref{fm12def2}.

Now $O^{E_\bu}_P\ra\B_P$ is the quotient of $\hat O^{E_\bu}_P\ra\A_P$ by $\G_P/Z(G)$. There is a surjective group morphism from $\G_P=\Aut(P)$ to $\Hom_{\Bord_X(BG)}(P,P)$ taking $\ga\in\G_P$ to the isomorphism class $[(P\t[0,1])_\ga]:P\ra P$ of its mapping cone $(P\t[0,1])_\ga$, that is, to the principal $G$-bundle $P\t[0,1]\ra X\t[0,1]$ where the chosen boundary isomorphism $P\vert_{X\t\{0\}}\ra P$ is the identity $\id_P$, but $P\vert_{X\t\{1\}}\ra P$ is $\ga$. Under the functor $\sO_{7,\Z_2}^{\bs\Spin,G}$, $\Hom_{\Bord_X(BG)}(P,P)$ acts on the $\Z_2$-torsor $\sO_{7,\Z_2}^{\bs\Spin,G}(X,P)$, so composing with $\G_P\ra\Hom_{\Bord_X(BG)}(P,P)$ gives an action of $\G_P$ on $\sO_{7,\Z_2}^{\bs\Spin,G}(X,P)$. Then $Z(G)\subset\G_P$ acts trivially on $\sO_{7,\Z_2}^{\bs\Spin,G}(X,P)$, so the action descends to $\G_P/Z(G)$, and the $\G_P/Z(G)$-action on $\sO_{7,\Z_2}^{\bs\Spin,G}(X,P)$ induced by $\sO_{7,\Z_2}^{\bs\Spin,G}$ agrees with the $\G_P/Z(G)$-action on the $\Z_2$-torsor of trivializations of $\hat O^{E_\bu}_P\ra\A_P$ under the identification of this with~$\sO_{7,\Z_2}^{\bs\Spin,G}(X,P)$.

Therefore an orientation of $\B_P$, that is, a trivialization of $O^{E_\bu}_P\ra\B_P$, is canonically equivalent to a natural isomorphism $\eta_X^P$ in the diagram
\e
\begin{gathered}
\xymatrix@C=80pt@R=15pt{
*+[r]{\Bord_X(BG)_P} \drtwocell_{}\omit^{}\omit{^{\eta_X^P\,\,\,}} \ar@/^.9pc/[drr]^(0.6)\boo \ar[d]^{\Pi_X^{\bs B}} \\
*+[r]{\Bord_n^{\bs\Spin}(BG)} \ar[rr]^(0.4){\sO_{7,\Z_2}^{\bs\Spin,G}} && *+[l]{0\qs \Z_2=\Ztor,\!} }
\end{gathered}
\label{fm12eq5}
\e
where $\Bord_X(BG)_P\subset\Bord_X(BG)$ is the full subcategory with one object $P$.

An orientation of $\sO_{7,\Z_2}^{\bs\Spin,G}$ for $X$ is a natural isomorphism $\eta_X$ in \eq{fm9eq1} for $\sO_{7,\Z_2}^{\bs\Spin,G}$, with $A=0$, $B=\Z_2$. Restricting \eq{fm9eq1} to $\Bord_X(BG)_P\subset\Bord_X(BG)$ gives \eq{fm12eq5}. Thus, an orientation of $\sO_{7,\Z_2}^{\bs\Spin,G}$ for $X$ induces orientations of $\B_P$ for all principal $G$-bundles $P\ra X$. As isomorphisms $\ga:P\ra P'$ lift to morphisms $[(P\t[0,1])_\ga]:P\ra P'$ in $\Bord_X(BG)$ by the mapping cone construction, these orientations on $\B_P$ depend on $P$ only up to isomorphism. Conversely, as the morphism $\G_P\ra\Hom_{\Bord_X(BG)}(P,P)$ is surjective, choices of orientation of $\B_P$ for all principal $G$-bundles $P\ra X,$ depending on $P$ only up to isomorphism, determine a unique orientation of $\sO_{7,\Z_2}^{\bs\Spin,G}$ for $X$. This proves~(a)(ii).

The proof of (a)(iii) is the same, but with $N^{E_\bu}_P,\sN_{7,\Z_2}^{\bs\Spin,G}$ in place of $O^{E_\bu}_P,\ab\sO_{7,\Z_2}^{\bs\Spin,G}$. Noting that $\B_P$ is orientable if and only if it is n-orientable, as $N^{E_\bu}_P,\sN_{7,\Z_2}^{\bs\Spin,G}$ differ from $O^{E_\bu}_P,\sO_{7,\Z_2}^{\bs\Spin,G}$ by the fixed $\Z_2$-torsor $O_{X\t G}^{E_\bu}\vert_{[\nabla^0]}$, then (i) follows from (ii),(iii), since if $\B_P$ is (n-)orientable for every principal $G$-bundle $P\ra X$ then we can choose an (n-)orientation for each isomorphism class of principal $G$-bundles $P\ra X$, and so construct an orientation on $\sO_{7,\Z_2}^{\bs\Spin,G}$ for~$X$.

Part (b) is proved as for (a).
\end{proof}

\begin{rem}
\label{fm12rem2}
We restrict to $n=7$ or 8 in Proposition \ref{fm11prop1} as that is what we need for our applications. The analogue holds when $n\equiv 7,8\mod 8$, and also when $n\equiv 1\mod 8$ if we use {\it Pfaffian orientations\/} of skew-adjoint operators, as in Freed \cite[\S 3]{Free}. When $n\equiv 2,3,4,5,6\mod 8$ the real Dirac operator $\sD_X$ is $\C$- or $\H$-linear, so moduli spaces $\B_P$ with orientation bundles $O^{E_\bu}_P$ defined using Dirac operators have canonical orientations for essentially trivial reasons.
\end{rem}

Combining Theorem \ref{fm11thm5}, Proposition \ref{fm12prop1}, and the material on flag structures in \S\ref{fm10}, yields:

\begin{thm}
\label{fm12thm1}
Let\/ $G$ be any of the Lie groups in the list\/ \eq{fm11eq19}. Then:
\begin{itemize}
\setlength{\itemsep}{0pt}
\setlength{\parsep}{0pt}
\item[{\bf(a)}] In Definitions\/ {\rm\ref{fm12def1}--\ref{fm12def2},} suppose\/ $(X,g)$ is a compact spin Riemannian $7$-manifold and\/ $E_\bu$ is a first order elliptic operator whose symbol is isomorphic to that of the Dirac operator $\sD_X$ of\/ $X$. Then $\B_P$ is orientable (equivalently, n-orientable) for every principal\/ $G$-bundle $P\ra X$. A choice of flag structure $F$ on $X$ determines n-orientations on $\B_P$ for all principal\/ $G$-bundles $P\ra X$. If we also choose an orientation for\/ $\det E_\bu\cong\R$ then a choice of flag structure on $X$ determines orientations on $\B_P$ for all\/~$P$.
\item[{\bf(b)}] In Definitions\/ {\rm\ref{fm12def1}--\ref{fm12def2},} suppose\/ $(X,g)$ is a compact spin Riemannian $8$-manifold and\/ $E_\bu$ is a first order elliptic operator whose symbol is isomorphic to that of the positive Dirac operator $\sD^+_X$ of\/ $X$. Suppose also that\/ $X$ satisfies condition Theorem\/ {\rm\ref{fm10thm2}$(*)$}. Then $\B_P$ is orientable (equivalently, n-orientable) for every principal\/ $G$-bundle $P\ra X$. A choice of flag structure $F$ on $X$ determines n-orientations on $\B_P$ for all principal\/ $G$-bundles $P\ra X$. If we also choose an orientation for\/ $\det E_\bu\cong\R$ then a choice of flag structure on $X$ determines orientations on $\B_P$ for all\/~$P$.

When $G=E_8,$ as $F_{BE_8}^{K(\Z,4)}$ in \eq{fm11eq14} is an equivalence, $\B_P$ is orientable for every\/ $E_8$-bundle $P\ra X$ \begin{bfseries}if and only if\end{bfseries} Theorem\/ {\rm\ref{fm10thm2}$(*)$} holds.
\end{itemize}
\end{thm}

\begin{rem}
\label{fm12rem3}
Cao--Gross--Joyce \cite[Th.~1.11]{CGJ} claimed to prove orientability of $\B_P$ in Theorem \ref{fm12thm1}(b) for $G=\U(m)$ or $\SU(m)$, without assuming Theorem \ref{fm10thm2}$(*)$. {\bf Unfortunately, there is a mistake in the proof of \cite[Th.~1.11]{CGJ}}, as the proof in \cite[\S 2.4]{CGJ} relies on the claim that the natural map $\pi_5(\SU(4))\cong\Z\ra H_5(\SU(4),\Z)\cong\Z$ is an isomorphism, whereas it is $24\cdot-:\Z\ra\Z$. Example \ref{fm12ex1} below gives a counterexample to \cite[Th.~1.11]{CGJ}. One of the goals of this monograph is to fix the problems with \cite{CGJ} under additional conditions on~$X$.
\end{rem}

\begin{ex}
\label{fm12ex1}
As in Example \ref{fm10ex1}, the compact spin 8-manifold $\SU(3)$ does not satisfy Theorem \ref{fm10thm2}$(*)$. Writing $\cS^1=\R/\Z$, define a principal $\SU(3)$-bundle $Q\ra\SU(3)\t\cS^1$ by
\begin{equation*}
\xymatrix@C=80pt{ Q=\frac{\ts \SU(3)\t\R\t\SU(3)}{\ts \Z,\;\> n:(\ga,x,\de)\,\mapsto(\ga,x+n,\de\ga^n)}\ar[r]^(0.63){(\ga,x,\de)\Z\mapsto (\ga,x+\Z)} & {\begin{subarray}{l}\ts \SU(3)\t\R/\Z \\ \ts =\SU(3)\t\cS^1.\end{subarray}} }
\end{equation*}
Then $c_2(Q)=\al\bt\Pd[*]$ for $H^3(\SU(3),\Z)=\Z\an{\al}$ and $H^1(\cS^1,\Z)=\Z\an{\Pd([*])}$. Thus, writing $\phi_Q:\SU(3)\t\cS^1\ra B\SU(3)$ for the classifying map of $Q$, so that $[\SU(3)\t\cS^1,\phi_Q]\in\Om_9^{\bs\Spin}(B\SU(3))$, then composing with $c_2:B\SU(3)\ra K(\Z,4)$ gives $[\SU(3)\t\cS^1,\al\bt\Pd[*]]\in\Om_9^{\bs\Spin}(K(\Z,4))$.

As in Example \ref{fm10ex1} we have $\int_{\SU(3)}\bar\al\cup\Sq^2(\bar\al)=\ul{1}\in\Z_2$. Hence by Theorem \ref{fm3thm2}(c), $[\SU(3)\t\cS^1,\al\bt\Pd[*]]$ is the nonzero element $\al_1\ze_2$ in $\Om_9^{\bs\Spin}(K(\Z,4))$, so $\pi_1(\sH_8^\Z)\bigl([\SU(3)\t\cS^1,\al\bt\Pd[*]]\bigr)=\ul{1}$ in $\Z_2$ by \eq{fm9eq6}. As $\sN_{8,\Z_2}^{\bs\Spin,\SU(3)}$ factors via $\sH_8^\Z$ by Theorem \ref{fm11thm5}(b), we see that $\pi_1(\sN_{8,\Z_2}^{\bs\Spin,\SU(3)})\bigl([\SU(3)\t\cS^1,\phi_Q]\bigr)=\ul{1}$.

Write $P$ for the trivial $\SU(3)$-bundle $P=\SU(3)\t\SU(3)\ra\SU(3)$. The bundle $Q\ra\SU(3)\t\cS^1$ induces a map $\cS^1\ra\B_P$, and the calculation above implies that the monodromy of the orientation bundle $O_P\ra\B_P$ around this loop is $\ul{1}$ in $\Z_2$. Therefore $\B_P$ is not orientable. This contradicts~\cite[Th.~1.11]{CGJ}.

If $\io:\SU(3)\hookra H$ is a complex type Lie group morphism, it follows from Proposition \ref{fm11prop1} that $\B_R$ is not orientable for $R$ the trivial $H$-bundle over $\SU(3)$. This holds when $H$ is $\SU(m),\U(m)$ or $\Spin(2m)$ for $m\ge 3$, $G_2,E_6,E_7,$ or~$E_8$.
\end{ex}

\begin{ex}
\label{fm12ex2}
Theorem \ref{fm12thm1} tells us nothing about orientability for the Lie groups $G=\Sp(m)$ for $m\ge 2$, which do not appear in \eq{fm11eq19}. 
\begin{itemize}
\setlength{\itemsep}{0pt}
\setlength{\parsep}{0pt}
\item[(i)] It follows from \cite[Ex.~2.23]{JoUp} that for $X$ is the compact spin 7-manifold $\cS^7$, $\B_P$ is (n-)orientable for any principal $\Sp(m)$-bundle $P\ra\cS^7$, all $m\ge 2$.
\item[(ii)] By \cite[Ex.~2.24]{JoUp}, for $X$ the compact spin 7-manifold $\Sp(2)\t_{\Sp(1)\t\Sp(1)}\Sp(1)$, $\B_{X\t\Sp(2)}$ is not orientable for the trivial $\Sp(2)$-bundle $X\t\Sp(2)\ra X$. Using Proposition \ref{fm11prop1}(a), Theorem \ref{fm11thm3} and the isomorphism $\Spin(5)\ab\cong\Sp(2)$, we deduce that $\B_{X\t G}$ is not orientable for the trivial $G$-bundle $X\t G\ra X$ for any Lie group $G$ on the list
\e
F_4, \quad \Sp(m+1), \quad  \Spin(2m+3), \quad \SO(2m+3), \quad \text{where $m\ge 1.$}
\label{fm12eq6}
\e
\item[(iii)] As in \cite[Ex.~1.14]{CGJ} for $G=\Sp(m)$, it follows from (ii) that for $X$ the compact spin 8-manifold $\Sp(2)\t_{\Sp(1)\t\Sp(1)}\Sp(1)\t\cS^1$, $\B_{X\t G}$ is not orientable for the trivial $G$-bundle $X\t G\ra X$ for any $G$ on~\eq{fm12eq6}.
\end{itemize}
\end{ex}

Here are some cases when we can make the orientation on $\B_P$ independent of the choice of flag structure.

\begin{prop}
\label{fm12prop2}
{\bf(a)} In Theorem\/ {\rm\ref{fm12thm1}(a),} suppose $G=\U(m),$ and\/ $P\ra X$ is a principal\/ $\U(m)$-bundle with\/ $c_2(P)-c_1(P)^2=0$ in $H^4(X,\Z_2)$. By requiring the flag structure $F$ to factor via $\Z_2$ and be natural at zero, the n-orientation on $\B_P$ is independent of the choice of\/ $F,$ and so is canonical.
\smallskip

\noindent{\bf(b)} In Theorem\/ {\rm\ref{fm12thm1}(b),} suppose $G=\U(m),$ and\/ $P\ra X$ is a principal\/ $\U(m)$-bundle with\/ $c_2(P)-c_1(P)^2=0$ in $H^4(X,\Z)$. By requiring the flag structure $F$ to be natural at zero, the n-orientation on $\B_P$ is independent of the choice of\/ $F,$ and so is canonical. 

If\/ $X$ satisfies Theorem\/ {\rm\ref{fm10thm2}$(\dag)$} then we can also require $F$ to factor via $\Z_2,$ and the above holds if instead\/ $c_2(P)-c_1(P)^2=0$ in $H^4(X,\Z_2)$.
\end{prop}

\begin{proof}
Theorem \ref{fm12thm1} works by pulling back orientations along the sequence of bordism categories $\Bord_X(B\U(m))\!\ra\! \Bord_X(B\SU(m+1))\!\ra\!\Bord_X(K(\Z,4))$, followed by $\Bord_X(K(\Z,4))\ra\Bord_X(K(\Z_2,4))$ for flag structures factoring via $\Z_2$. A $\U(m)$-bundle $P$ in $\Bord_X(B\U(m))$ is mapped to an $\SU(m+1)$-bundle $Q$ in $\Bord_X(B\SU(m+1))$ with $c_2(Q)=c_2(P)-c_1(P)^2$, as the $\C^m$-bundle $E$ corresponding to $P$ is mapped to the $\C^{m+1}$-bundle $E\op\La^mE^*$, and $c_2(E\op\La^mE^*)=c_2(E)-c_1(E)^2$. This is mapped to a 4-cocycle $C$ in $\Bord_X(K(\Z,4))$ with $[C]=c_2(Q)=c_2(P)-c_1(P)^2$ in $H^4(X,\Z)$, and further mapping to $c_2(P)-c_1(P)^2$ in $H^4(X,\Z_2)$ for flag structures factoring via~$\Z_2$.

A flag structure $F$ natural at zero is determined independent of the choice of $F$ at $C\in\Bord_X(K(\Z,4))$ with $[C]=0$ in $H^4(X,\Z)$. Thus, $F$ determines an orientation on $\B_P$ independent of the choice of $F$ if $c_2(P)-c_1(P)^2=0$ in $H^4(X,\Z)$. If also $F$ factors via $\Z_2$ then the orientation on $\B_P$ is independent of $F$ if $c_2(P)-c_1(P)^2=0$ in $H^4(X,\Z_2)$. The proposition now follows from the existence of such flag structures  in Remark \ref{fm10rem1} and Theorem~\ref{fm10thm2}.
\end{proof}

\subsection{\texorpdfstring{$G_2$-instantons on $G_2$-manifolds}{G₂-instantons on G₂-manifolds}}
\label{fm122}

We discuss the exceptional holonomy group $G_2$ in 7 dimensions, and $G_2$-in\-stan\-tons. See the first author \cite[\S 10]{Joyc1} for background on this.

\begin{dfn}
\label{fm12def3}	
Let $\R^7$ have coordinates $(x_1,\dots,x_7)$. Write $\d{\bf x}_{ijk}$ for the 3-form $\d x_i\w\d x_j\w\d x_k$ on $\R^7$. Define a 3-form $\vp_0$ on $\R^7$ by
\begin{equation*}
\vp_0=\d{\bf x}_{123}+\d{\bf x}_{145}+\d{\bf x}_{167}+\d{\bf x}_{246}-\d{\bf x}_{257}-\d{\bf x}_{347}-\d{\bf x}_{356}.
\end{equation*}
The subgroup of $\GL(7,\R)$ preserving $\vp_0$ is the holonomy group $G_2$. It also preserves the orientation and the Euclidean metric $g_0=\d x_1^2+\cdots+\d x_7^2$ on~$\R^7$. 

Let $X$ be a 7-manifold. A $G_2$-{\it structure\/} $(\vp,g)$ on $X$ is a 3-form $\vp$ and Riemannian metric $g$ on $X$, such that for all $x\in X$ there exist isomorphisms $T_xX\cong\R^7$ identifying $\vp\vert_x\cong\vp_0$ and $g\vert_x\cong g_0$. We call $(\vp,g)$ {\it torsion-free\/} if $\d\vp=\d(*\vp)=0$. This implies that $\Hol(g)\subseteq G_2$.  A $G_2$-structure induces an orientation and spin structure on $X$.  A $G_2$-{\it manifold\/} $(X,\vp,g)$ is a 7-manifold $X$ with a $G_2$-structure $(\vp,g)$. Examples of compact, torsion-free $G_2$-manifolds with holonomy $G_2$ were constructed by the first author~\cite[\S 11--\S 12]{Joyc1}. 

Suppose $(X,\vp,g)$ is a compact $G_2$-manifold with $\d(*\vp)=0$. Let $G$ be a Lie group, and $P\ra X$ a principal $G$-bundle. A $G_2$-{\it instanton\/} on $P$ is a connection $\nabla_P$ on $P$ whose curvature satisfies  $F^{\nabla_P}\w*\vp=0$ in $\Ga^\iy(\Ad(P)\ot\La^6T^*X)$. Write $\M_P^{G_2}$ for the moduli space of irreducible $G_2$-instantons on $P$, as a subspace of $\B_P^\irr$ in Definition \ref{fm12def1}. As $\d(*\vp)=0$, the deformation theory of $\M_P^{G_2}$ is controlled by an elliptic complex, so one can show that $\M_P^{G_2}$ is a derived manifold of virtual dimension 0 in the sense of \cite{Joyc2,Joyc3,Joyc4,Joyc6}. If $\vp$ is generic in its cohomology class, then $\M_P^{G_2}$ is an ordinary 0-manifold. Examples and constructions of $G_2$-instantons are given in \cite{MNS,SaEa,SaWa,Walp2,Walp3,Walp4}. 

As in \cite[\S 4.1]{JTU}, the orientation bundle of $\M_P^{G_2}$ is the restriction to $\M_P^{G_2}$ of $O^{E_\bu}_P\ra\B_P$, for $E_\bu$ the Dirac operator of the spin structure on $X$ induced by $(\vp,g)$, so we may orient $\M_P^{G_2}$ by restricting orientations on $O^{E_\bu}_P\ra\B_P$. We also define {\it normalized orientations\/} on $\M_P^{G_2}$ using~$N^{E_\bu}_P\ra\B_P$.
\end{dfn}

The next theorem follows from Theorem \ref{fm12thm1}(a) and Definition \ref{fm12def3}. It was already known for $G=\SU(m)$ and $\U(m)$ by \cite[Cor.~1.4]{JoUp}. Walpuski \cite[\S 6.1]{Walp1} earlier proved orientability for $G=\SU(m)$.

\begin{thm}
\label{fm12thm2}
Let\/ $(X,\vp,g)$ be a compact\/ $G_2$-manifold with\/ $\d(*\vp)=0,$ and let\/ $G$ be any of the Lie groups in the list\/ \eq{fm11eq19}. Then a choice of flag structure $F$ on $X$ in the sense of\/ {\rm\S\ref{fm101}} determines normalized orientations on $G_2$-instanton moduli spaces $\M_P^{G_2}$ for all principal\/ $G$-bundles $P\ra X$.

If we also choose an orientation for\/ $\det\sD_X\cong\R$ then normalized orientations are equivalent to orientations on\/ $\M_P^{G_2}$.
\end{thm}

\begin{rem}
\label{fm12rem4}
{\bf(a)} Donaldson and Segal \cite{DoSe} propose defining enumerative invariants of $(X,\vp,g)$ by counting $\M_P^{G_2}$, {\it with signs}, and adding correction terms from associative 3-folds in $X$, as in \S\ref{fm142}. The signs come from an orientation on $\M_P^{G_2}$. Thus, Theorem \ref{fm12thm2} contributes to the Donaldson--Segal programme.

\smallskip

\noindent{\bf(b)} One can imagine trying to define a {\it Floer theory\/} using $G_2$-instantons on principal $G$-bundles, similar to instanton Floer homology for compact oriented 3-manifolds \cite{Dona}, in which the Floer complex is generated by $G_2$-instantons on a compact $G_2$-manifold $(X,\vp,g)$, and the differentials obtained by counting $\Spin(7)$-instantons on $X\t\R$. There are of course serious analytic difficulties.

Our theory is relevant to {\it gradings\/} of such a Floer theory, if it exists: an orientation for $X$ of the orientation functor $\sN_7^{\bs\Spin,G}$ or $\sN_{7,\Z_k}^{\bs\Spin,G}$ from \S\ref{fm93} would induce a grading of the Floer theory over $\Z$ or~$\Z_k$.

In fact, we have proved a {\it negative result\/}: out of the groups $G=\SU(m),$ $\Sp(m)$ for $m\ge 2$ and $E_8$, Corollary \ref{fm11cor1} shows that $\sN_{7,\Z_2}^{\bs\Spin,\SU(m)}$ and $\sN_{7,\Z_2}^{\bs\Spin,E_8}$ are orientable for all compact spin 7-manifolds $X$, but no other $\sN_7^{\bs\Spin,G}$ or $\sN_{7,\Z_k}^{\bs\Spin,G}$ have this property. Thus, even for $G=\SU(2)$, there are obstructions to grading the Floer theory over $\Z$ or $\Z_k$ for any $k>2$ (at least, using the methods of this monograph), which are nontrivial for some compact spin 7-manifold~$X$.
\end{rem}

\subsection{\texorpdfstring{$\Spin(7)$-instantons on $\Spin(7)$-manifolds}{Spin(7)-instantons on Spin(7)-manifolds}}
\label{fm123}

Next we discuss the exceptional holonomy group $\Spin(7)$ in 8 dimensions, and $\Spin(7)$-instantons. See the first author \cite[\S 10]{Joyc1} for background on this.

\begin{dfn}
\label{fm12def4}	
Let $\R^8$ have coordinates $(x_1,\dots,x_8)$. Write $\d{\bf x}_{ijkl}$ for the 4-form $\d x_i\w\d x_j\w\d x_k\w\d x_l$ on $\R^8$. Define a 4-form $\Om_0$ on $\R^8$ by
\begin{align*}
\Om_0=\phantom{-}&\d{\bf x}_{1234}+\d{\bf x}_{1256} +\d{\bf
x}_{1278}+\d{\bf
x}_{1357}-\d{\bf x}_{1368}-\d{\bf x}_{1458}-\d{\bf x}_{1467}\\
-\,&\d{\bf x}_{2358}-\d{\bf x}_{2367}-\d{\bf x}_{2457}+\d{\bf
x}_{2468}+\d{\bf x}_{3456}+\d{\bf x}_{3478}+\d{\bf x}_{5678}.
\end{align*}
The subgroup of $\GL(8,\R)$ preserving $\Om_0$ is the holonomy
group $\Spin(7)$. It is a compact, connected, simply-connected,
semisimple, 21-dimensional Lie group, which is isomorphic to the double cover of $\SO(7)$. This group also preserves the orientation on $\R^8$ and the Euclidean metric $g_0=\d x_1^2+\cdots+\d x_8^2$ on~$\R^8$. 

Let $X$ be an 8-manifold. A $\Spin(7)$-{\it structure\/} $(\Om,g)$ on $X$ is a 4-form $\Om$ and Riemannian metric $g$ on $X$, such that for all $x\in X$ there exist isomorphisms $T_xX\cong\R^8$ identifying $\Om\vert_x\cong\Om_0$ and $g\vert_x\cong g_0$. We call $(\Om,g)$ {\it torsion-free\/} if $\d\Om=0$. This implies that $\Hol(g)\subseteq\Spin(7)$. 

A $\Spin(7)$-{\it manifold\/} $(X,\Om,g)$ is an 8-manifold $X$ with a $\Spin(7)$-structure $(\Om,g)$. Examples of compact torsion-free $\Spin(7)$-manifolds with holonomy $\Spin(7)$ were constructed by the first author \cite[\S 13--\S 15]{Joyc1}. Calabi--Yau 4-folds, and hyperk\"ahler 8-manifolds, are also torsion-free $\Spin(7)$-manifolds.

Let $(X,\Om,g)$ be a compact $\Spin(7)$-manifold. Then $(\Om,g)$ induces a splitting $\La^2T^*X=\La^2_7T^*X \op\La^2_{21}T^*X$ into vector subbundles of ranks $7,21$, the eigenspaces of $\al\mapsto *(\al\w\Om)$. Suppose $G$ is a Lie group and $P\ra X$ a principal $G$-bundle. A $\Spin(7)$-{\it instanton\/} on $P$ is a connection $\nabla_P$ on $P$ whose curvature satisfies $\pi^2_7(F^{\nabla_P})=0$ in $\Ga^\iy(\Ad(P)\ot\La^2_7T^*X)$. 

Write $\M_P^{\Spin(7)}$ for the moduli space of irreducible $\Spin(7)$-instantons on $P$, as a subspace of $\B_P^\irr$ in Definition \ref{fm12def1}. The deformation theory of $\M_P^{\Spin(7)}$ is controlled by an elliptic complex, so one can show that $\M_P^{\Spin(7)}$ is a derived manifold in the sense of \cite{Joyc2,Joyc3,Joyc4,Joyc6}. If $\Om$ is generic amongst $\Spin(7)$ 4-forms, then $\M_P^{\Spin(7)}$ is an ordinary manifold. Examples of $\Spin(7)$-instantons were given by Lewis \cite{Lewi}, Tanaka \cite{Tana}, and Walpuski~\cite{Walp5}. 

As in \cite[\S 4.1]{JTU}, the orientation bundle of $\M_P^{\Spin(7)}$ is the restriction to $\M_P^{\Spin(7)}$ of $O^{E_\bu}_P\ra\B_P$, for $E_\bu$ the positive Dirac operator of the spin structure on $X$ induced by $(\Om,g)$, so we may orient $\M_P^{\Spin(7)}$ by restricting orientations on $O^{E_\bu}_P$. We also define {\it normalized orientations\/} on $\M_P^{\Spin(7)}$ using~$N^{E_\bu}_P\ra\B_P$.
\end{dfn}

The next theorem follows from Theorem \ref{fm12thm1}(b) and Definition~\ref{fm12def4}.

\begin{thm}
\label{fm12thm3}
Let\/ $(X,\Om,g)$ be a compact\/ $\Spin(7)$-manifold, and suppose that\/ $X$ satisfies condition Theorem\/ {\rm\ref{fm10thm2}$(*)$}. Let\/ $G$ be any of the Lie groups in the list\/ \eq{fm11eq19}. Then a choice of flag structure $F$ on $X$ as in\/ {\rm\S\ref{fm102}} determines normalized orientations on $\Spin(7)$-instanton moduli spaces $\M_P^{\Spin(7)}$ for all principal\/ $G$-bundles $P\ra X$. If we also choose an orientation for\/ $\det\sD^+_X\cong\R$ then normalized orientations are equivalent to orientations on\/~$\M_P^{\Spin(7)}$.
\end{thm}

\begin{rem}
\label{fm12rem5}
{\bf(a)} Cao--Gross--Joyce \cite[Cor.~1.12]{CGJ} claim the result of Theorem \ref{fm12thm3} for $G=\U(m)$ or $\SU(m)$, without assuming Theorem \ref{fm10thm2}$(*)$. {\bf Unfortunately, as in Remark \ref{fm12rem3}, there is a mistake in the proof.}
\smallskip

\noindent{\bf(b)} As in Donaldson--Thomas \cite{DoTh} and Donaldson--Segal \cite{DoSe}, one might hope to define enumerative invariants of compact $\Spin(7)$-manifolds $(X,\Om,g)$, similar to Donaldson invariants of compact oriented 4-manifolds, by `counting' moduli spaces $\M_P^{\Spin(7)}$. The orientations on $\M_P^{\Spin(7)}$ in Theorem \ref{fm12thm3} would be necessary for this.
\smallskip

\noindent{\bf(c)} The proof of Theorem \ref{fm12thm3} implicitly involved a choice of natural isomorphism $\la_8$ in \eq{fm11eq14} in Theorem \ref{fm11thm4}(a), where there are 4 possibilities for $\la$ respecting trivializations on $\Om_8^{\bs\Spin}(*)$. A choice of $\la$ amounts to a bordism-invariant {\it orientation convention}.
\smallskip

\noindent{\bf(d)} There are usually many choices for the flag structure $F$ in Theorem \ref{fm12thm3}. To reduce the choice, if $X$ satisfies Theorem \ref{fm10thm2}$(\dag)$ we can require $F$ to factor via $\Z_2$, and then there are only finitely many possibilities for~$F$.
\end{rem}

\section[Applications to enumerative invariants of Calabi--Yau 4-folds]{Applications to enumerative invariants \\ of Calabi--Yau 4-folds}
\label{fm13}

\subsection{Calabi--Yau 4-folds and DT4 invariants}
\label{fm131}

\begin{dfn}
\label{fm13def1}	
A {\it Calabi--Yau\/ $m$-fold\/} $X$ is a connected smooth projective $\C$-scheme of complex dimension $m$ with trivial canonical bundle $K_X\cong\cO_X$. By the Calabi Conjecture, $X$ admits Ricci flat K\"ahler metrics $g$, which have holonomy $\Hol(g)\subseteq\SU(m)$. Often one includes the condition $\Hol(g)=\SU(m)$ in the definition of Calabi--Yau $m$-fold. See the first author \cite[\S 6]{Joyc1} for background on Calabi--Yau geometry.
\end{dfn}

We can consider coherent sheaves on $X$, including (algebraic) vector bundles (i.e.\ locally free coherent sheaves). Write $\coh(X)$ for the abelian category of coherent sheaves on $X$. See Hartshorne \cite[\S II.5]{Hart} and Huybrechts and Lehn \cite{HuLe}. We can also consider the bounded derived category $D^b\coh(X)$ of complexes of coherent sheaves. For triangulated categories and derived categories see Gelfand and Manin \cite{GeMa}, and for properties of $D^b\coh(X)$ see Huybrechts \cite{Huyb}. We will be interested in moduli spaces $\M$ of objects in $\coh(X)$ or $D^b\coh(X)$. 

We summarize some ideas from Derived Algebraic Geometry \cite{PTVV,Toen1,Toen2,ToVa,ToVe1,ToVe2} and Donaldson--Thomas type invariants of Calabi--Yau 4-folds~\cite{BoJo,CaLe,OhTh}:
\begin{itemize}
\setlength{\itemsep}{0pt}
\setlength{\parsep}{0pt}
\item[(a)] Let $X$ be a smooth projective $\C$-scheme. Then To\"en and Vaqui\'e \cite{ToVa} construct a {\it derived moduli stack\/} $\bs\M$ of objects in $\coh(X)$ or in $D^b\coh(X)$, as a locally finitely presented derived $\C$-stack in the sense of To\"en and Vezzosi \cite{Toen1,Toen2,ToVe1,ToVe2}. It has a {\it virtual dimension\/} $\vdim_\C\bs\M$, a locally constant map $\bs\M\ra\Z$. The classical truncation $\M=t_0(\bs\M)$ is the usual moduli stack, as an Artin $\C$-stack or higher $\C$-stack.
\item[(b)] Pantev, To\"en, Vaqui\'e and Vezzosi \cite{PTVV} introduced a theory of shifted symplectic Derived Algebraic Geometry, defining $k$-{\it shifted symplectic structures\/} $\om$ on a derived stack $\bs\cS$ for $k\in\Z$. If $X$ is a Calabi--Yau $m$-fold and $\bs\M$ is a derived moduli stack of objects in $\coh(X)$ or $D^b\coh(X)$ then $\bs\M$ has a $(2-m)$-shifted symplectic structure,~\cite[Cor.~2.13]{PTVV}.

Also $\bL_i:i^*(\bL_{\bcM})\ra\bL_\M$ is an $m$-{\it Calabi--Yau obstruction theory\/} on $\M$, a classical truncation of the shifted symplectic structure on~$\bcM$.
\item[(c)] If $(\bs\cS,\om)$ is a $k$-shifted symplectic derived stack for $k$ even, Borisov--Joyce \cite[\S 2.4]{BoJo} define a notion of {\it orientation\/} on $(\bs\cS,\om)$, or on the classical truncation $\cS=t_0(\bs\cS)$. See Definition \ref{fm13def2} below.
\item[(d)] Let $(\bs\cS,\om)$ be a proper, oriented $-2$-shifted symplectic derived scheme with $\cS=t_0(\bs\cS)$. Then Borisov--Joyce \cite[Cor.~1.2]{BoJo} construct a {\it virtual class\/} $[\bs\cS]_\virt$ in $H_*(\cS,\Z)$ using Derived Differential Geometry \cite{Joyc2,Joyc3,Joyc4,Joyc6}, of real dimension $\vdim_\C\bs\cS=\ha\vdim_\R\bs\cS$. Note that this is {\it half the expected dimension}. Oh--Thomas \cite{OhTh} provide an alternative definition of $[\bs\cS]_\virt$ in the style of Behrend--Fantechi~\cite{BeFa}.

Oh--Thomas \cite{OhTh} define their virtual class $[\M]_\virt$ only when $\M$ is a projective moduli scheme of Gieseker stable sheaves on a Calabi--Yau 4-fold $X$. However, Kiem--Park \cite[\S 8]{KiPa} provide an alternative definition which works for $\M$ a proper Deligne--Mumford stack with a 4-Calabi--Yau obstruction theory satisfying an `isotropic cone' condition.
\item[(e)] Let $X$ be a Calabi--Yau 4-fold, and write $K(\coh(X))$ for the numerical Grothendieck group of $\coh(X)$, which is the image of the Chern character map $\ch:K_0(\coh(X))\ra H^{\rm even}(X,\Q)$. Write $C(\coh(X))\subset K(\coh(X))$ for the set of classes $\lb E\rb\in K(\coh(X))$ of nonzero objects $E\in\coh(X)$. Let $\tau$ be a Gieseker stability condition on $\coh(X)$. 

Then for each $\al\in C(\coh(X))$ we have have moduli schemes $\M_\al^\rst(\tau)\subseteq\M_\al^\ss(\tau)$ of $\tau$-(semi)stable coherent sheaves in class $\al$. Here $\M_\al^\rst(\tau)$ is a fine moduli scheme which is the classical truncation $t_0(\bs\M_\al^\rst(\tau))$ of a $-2$-shifted symplectic derived moduli scheme $\bs\M_\al^\rst(\tau)$, and $\M_\al^\ss(\tau)$ is a coarse moduli scheme, which is proper.

Suppose that $\M_\al^\rst(\tau)=\M_\al^\ss(\tau)$, that is, there are no strictly semistable sheaves in class $\al$. Then $\M_\al^\ss(\tau)=t_0(\bs\M_\al^\ss(\tau))$ is the classical truncation of a proper $-2$-shifted symplectic derived moduli scheme $\bs\M_\al^\ss(\tau)$. 

Suppose $\bs\M_\al^\ss(\tau)$ is orientable, and choose an orientation. Then by (d) we get a virtual class $[\bs\M_\al^\ss(\tau)]_\virt$ in $H_*(\M_\al^\ss(\tau),\Z)$. Borisov--Joyce \cite{BoJo} propose to define Donaldson--Thomas type `DT4 invariants' of $X$ using these virtual classes. Cao--Leung \cite{CaLe} make a similar proposal using gauge theory rather than Derived Algebraic Geometry.
\item[(f)] The study of DT4 invariants is now a thriving field. See \cite{Bojk1,Bojk2,BoJo,Cao1,Cao2,CaKo1,CaKo2,CKM,CaLe,CMT1,CMT2,COT1,COT2,CaQu,CaTo1,CaTo2,CaTo3,GJT,Joyc7,KiPa,OhTh,Park} for some papers in this area.
\end{itemize}

We define orientations in (c), following Borisov--Joyce~\cite[\S 2.4]{BoJo}.

\begin{dfn}
\label{fm13def2}
Let $\M$ be an Artin or higher $\C$-stack with a 4-Calabi--Yau obstruction theory $\la:\cL^\bu\ra\bL_\M$, $\om:(\cL^\bu)^\vee\,{\buildrel\sim\over\longra}\,\cL^\bu[-2]$. Then we have a determinant line bundle $\det\cL^\bu\ra \M$, and $\om$ induces an isomorphism $\det\om:(\det\cL^\bu)^*\ra\det\cL^\bu$. A (4-{\it Calabi--Yau\/}) {\it orientation\/} for $(\M,\la,\om)$ is a choice of isomorphism $\mu:\cO_\M\ra\det\cL^\bu$ with~$\mu\ci\mu^*=\det\om$. 

Here $\mu$ is basically a square root of $\det\om$. Locally on $\M$ in the \'etale topology there are two choices for $\mu$, and there is a principal $\Z_2$-bundle $O_\M\ra \M$ parametrizing choices of $\mu$. We say that $(\M,\la,\om)$ is {\it orientable\/} if $O_\M$ is trivializable, and an {\it orientation\/} is a trivialization~$O_\M\cong \M\t\Z_2$.
\end{dfn}

Answering the following question is very important for developing a theory of DT4 invariants, as without orientations, DT4 invariants cannot be defined:

\begin{quest}
\label{fm13quest1}
Let\/ $X$ be a Calabi--Yau $4$-fold, and\/ $\M$ the moduli stack of objects in $\coh(X)$ or $D^b\coh(X),$ with its natural\/ $4$-Calabi--Yau obstruction theory. Is $\M$ orientable in the sense of Definition\/ {\rm\ref{fm13def2}?} If so, what extra data on $X$ is needed to construct an orientation on {\rm$\M$?}
\end{quest}

\begin{rem}
\label{fm13rem1}
Cao--Gross--Joyce \cite[Cor.~1.17]{CGJ} claimed to prove moduli stacks $\M$ of objects in $\coh(X)$ or $D^b\coh(X)$ are orientable for any (compact) Calabi--Yau 4-fold $X$. This was extended to compactly-supported coherent sheaves on noncompact Calabi--Yau 4-folds by Bojko \cite{Bojk1}. {\bf Unfortunately, as in Remark \ref{fm12rem3} and \ref{fm12rem5}(a), there is a mistake in the proof of \cite[Th.~1.11]{CGJ}, which invalidates \cite[Cor.~1.17]{CGJ}, and also the main result in \cite{Bojk1}.} The first author would like to apologize for this. Theorem \ref{fm13thm2} below corrects the mistake, and answers Question \ref{fm13quest1}, under the extra condition 
Theorem~\ref{fm10thm2}$(*)$.
\end{rem}

The next theorem summarizes parts of Cao--Gross--Joyce \cite[Th.~1.15]{CGJ}, which is not affected by the mistake in \cite[Th.~1.11]{CGJ}, plus background material from Joyce--Tanaka--Upmeier \cite[\S 2]{JTU}.

\begin{thm}
\label{fm13thm1}
Let\/ $X$ be a projective Calabi--Yau\/ $4$-fold.
\smallskip

\noindent{\bf(a)} Write\/ $\M$ for the moduli stack of objects\/ $G^\bu$ in\/ $D^b\coh(X),$ a higher stack. It has a decomposition\/ $\M=\coprod_{\al\in K^0_\top(X)}\M_\al,$ where\/ $\M_\al$ is the substack of complexes\/ $G^\bu$ with class\/ $\lb G^\bu\rb=\al$ in the topological K-theory of the underlying\/ $8$-manifold of\/ $X$. There is a natural\/ $4$-Calabi--Yau obstruction theory\/ $\phi:\cF^\bu\ra\bL_\M$, $\th:\cF^\bu\,{\buildrel\sim\over\longra}\,(\cF^\bu)^\vee[2]$ on\/ $\M,$ and hence a principal\/ $\Z_2$-bundle\/ $O^{\cF^\bu}\ra\M$ of orientations on\/ $\M$ as in Definition\/ {\rm\ref{fm13def2},} restricting to\/ $O^{\cF^\bu}_\al\ra\M_\al$.

Write\/ $\M^\top$ for the \begin{bfseries}topological realization\end{bfseries} of\/ $\M,$ a topological space natural up to homotopy equivalence, as in Simpson {\rm\cite{Simp},} Blanc\/ {\rm\cite[\S 3.1]{Blan},} and\/ {\rm\cite[\S 2.5]{CGJ}}. Then\/ $O_{\cF^\bu}$ lifts to a principal\/ $\Z_2$-bundle\/ $O^{\cF^\bu,\top}\ra\M^\top,$ restricting to\/ $O^{\cF^\bu,\top}_\al\ra\M^\top_\al,$ such that trivializations of\/ $O^{\cF^\bu}_\al$ and\/ $O^{\cF^\bu,\top}_\al$ are naturally in\/ {\rm 1-1} correspondence.
\smallskip

\noindent{\bf(b)} Write\/ $\cC=\Map_{C^0}(X, B\U\t\Z),$ where\/ $B\U=\varinjlim_{n\ra\iy}B\U(n)$ is the unitary classifying space. It has a natural decomposition\/ $\cC=\coprod_{\al\in K^0_\top(X)}\cC_\al,$ where\/ $\cC_\al$ is connected. Taking the elliptic operator\/ $E_\bu\ra X$ to be the positive Dirac operator\/ $\slashed{D}_+$ of the spin structure on\/ $X$ induced by the Calabi--Yau\/ $4$-fold structure, which for a Calabi--Yau\/ $4$-fold\/ $X$ may be written
\begin{equation*}
\slashed{D}_+=\db+\db^*:\Ga^\iy(\La^{0,{\rm even}}T^*X)\longra \Ga^\iy(\La^{0,{\rm odd}}T^*X),
\end{equation*}
in {\rm\cite[\S 2.4]{JTU}} we construct a principal\/ $\Z_2$-bundle\/ $O_\cC\ra\cC,$ restricting to\/ $O_{\cC_\al}\ra\cC_\al$. It is thought of as a bundle of orientations on\/ $\cC,$ and is obtained from the bundles\/ $O_P^{E_\bu}\ra\B_P$ in {\rm\S\ref{fm121}} for\/ $\U(m)$-bundles\/ $P\ra X$ in a limiting process as~$m\ra\iy$.

From the definition of\/ $O_{\cC_\al},$ if\/ $k\in\N$ then and\/ $\Xi_{\al,k}:\cC_\al\ra\cC_{\al+k\lb\cO_X\rb}$ is the homotopy equivalence induced by direct sum with the trivial vector bundle\/ $\bigop^{k}\cO_X\ra X,$ then there is a canonical isomorphism\/ $O_{\cC_\al}\cong \Xi_{\al,k}^*(O_{\cC_{\al+k\lb\cO_X\rb}})$. 

\textup(Actually, for a general spin\/ $8$-manifold,\/ $O_{\cC_\al}$ and\/ $\Xi_{\al,k}^*(O_{\cC_{\al+k\lb\cO_X\rb}})$ differ by the\/ $\Z_2$-torsor\/ $\Or(\det\slashed{D}_+)^{\ot^k},$ so in general we should restrict to\/ $k$ even. But as\/ $X$ is a Calabi--Yau\/ $4$-fold there is a canonical isomorphism\/ $\Or(\det\slashed{D}_+)\cong\Z_2$.\textup)
\smallskip

\noindent{\bf(c)} We relate {\bf(a)\rm,\bf(b)} as follows: using the classifying morphism of the universal complex\/ $\cU^\bu\ra X\t\M,$ as in {\rm\cite[Th.~1.15]{CGJ}} we can define a continuous map\/ $\Phi:\M^\top\ra\cC,$ natural up to homotopy, restricting to\/ $\Phi_\al:\M_\al^\top\ra\cC_\al$ for\/ $\al\in K^0_\top(X)$. Then there are natural isomorphisms\/ $O^{\cF^\bu,\top}_\al\cong \Phi^*(O_{\cC_\al})$ of principal\/ $\Z_2$-bundles on\/ $\M^\top_\al$. Hence, a trivialization of\/ $O_{\cC_\al}$ induces trivializations of\/ $O^{\cF^\bu,\top}_\al$ and\/~$O^{\cF^\bu}_\al$.
\smallskip

\noindent{\bf(d)} Let\/ $P\ra X$ be a principal\/ $\U(m)$-bundle, and\/ $O_P^{E_\bu}\ra\B_P$ be as in {\rm\S\ref{fm121}} for\/ $E_\bu\ra X$ the positive Dirac operator\/ $\slashed{D}_+$ of the spin structure on\/ $X$ induced by the Calabi--Yau\/ $4$-fold structure. Write\/ $\be=\lb P\rb\in K^0_\top(X)$.

Write\/ $\B_P^\top$ for the topological realization of the topological stack\/ $\B_P,$ a topological space natural up to homotopy equivalence. Then\/ $O_P^{E_\bu}$ lifts to a principal\/ $\Z_2$-bundle\/ $O_P^{E_\bu,\top}\ra\B_P^\top,$ such that trivializations of\/ $O_P^{E_\bu}$ and\/ $O_P^{E_\bu,\top}$ are naturally in\/ {\rm 1-1} correspondence.
\smallskip

\noindent{\bf(e)} We relate {\bf(b)\rm,\bf(d)} as follows: using the universal principal\/ $\U(m)$-bundle\/ $U_P\ra X\t\B_P$ we can define a continuous map\/ $\Psi_\be:\B_P^\top\ra\cC_\be,$ natural up to homotopy. Then the construction of\/ $O_{\cC_\be}$ implies that there is a natural isomorphism\/ $O_P^{E_\bu,\top}\cong \Psi_\be^*(O_{\cC_\be})$ of principal\/ $\Z_2$-bundles on\/ $\B_P^\top$. Hence, a trivialization of\/ $O_{\cC_\be}$ induces trivializations of\/ $O_P^{E_\bu}$ and\/ $O_P^{E_\bu,\top}$.
\smallskip

\noindent{\bf(f)} In {\bf(d)\rm,\bf(e)\rm,} suppose\/ $m\ge 5$. Then\/ $\Psi_\be:\B_P^\top\ra\cC_\be$ induces isomorphisms\/ $\pi_i(\B_P^\top)\ra\pi_i(\cC_\be)$ for\/ $i=0,1$. Therefore {\bf(e)} induces a {\rm 1-1} correspondence between trivializations of\/ $O_{\cC_\al},$\/ $O_P^{E_\bu},$ and\/ $O_P^{E_\bu,\top},$ so in particular, a trivialization of\/ $O_P^{E_\bu}$ induces a trivialization of\/ $O_{\cC_\be}$.
\smallskip

\noindent{\bf(g)} Let\/ $\al\in K^0_\top(X)$  and set\/ $k=\max(5-\rank\al,0),$ $m=\min(5,\rank\al),$ and\/ $\be=\al+k\lb\cO_X\rb$. Then there exists a principal\/ $\U(m)$-bundle\/ $P\ra X,$ unique up to isomorphism, with\/ $\lb P\rb=\be$ in\/ $K^0_\top(X)$. By {\bf(a)\rm--\bf(f)\rm,} we now see that a trivialization of\/ $O_P^{E_\bu}$ induces trivializations of\/ $O_P^{E_\bu,\top},O_{\cC_\be},O_{\cC_\al},O^{\cF^\bu,\top}_\al,$ and\/ $O^{\cF^\bu}_\al$. That is, an orientation on\/ $\B_P$ induces an orientation on\/~$\M_\al$.
\end{thm}

\begin{rem}
\label{fm13rem2}
We offer some explanation of Theorem \ref{fm13thm1}. For simplicity, let us start with moduli spaces $\M_\al^{\vect,\ss}(\tau)$ of Gieseker stable vector bundles $E\ra X$ in class $\al\in K^0_\top(X)$ with $c_1(\al)=0$, where $\rank\al=r\ge 4$. 

By the Hitchin--Kobayashi correspondence, every such $E\ra X$ admits a natural Hermitian--Einstein connection $\nabla_E$, and then $(E,\nabla_E)$ is a $\Spin(7)$-instanton. Every $\Spin(7)$-instanton connection on the complex vector bundle $E\ra X$ comes from an algebraic vector bundle structure on $E$ in this way. As $r\ge 4$, every complex vector bundle $E'\ra X$ with $\lb E'\rb=\al$ has $E'\cong E$. 

This induces an isomorphism from $\M_\al^{\vect,\ss}(\tau)$ to the moduli space $\M_P^{\Spin(7)}$ of irreducible $\Spin(7)$-instantons on the principal $\U(r)$-bundle $P\ra X$ associated to $E$, and hence an inclusion $\M_\al^{\vect,\ss}(\tau)\hookra\B_P$. Since DT4 orientations on $\M_\al^{\vect,\ss}(\tau)$ are basically orientations of $\Spin(7)$ instanton moduli spaces, as in \S\ref{fm123}, an orientation on $\B_P$ pulls back to a DT4 orientation of $\M_\al^{\vect,\ss}(\tau)$.

Now $\M_\al^{\vect,\ss}(\tau)$ is a finite-dimensional $\C$-scheme, whereas $\B_P$ is an infinite-dimensional topological stack. One might think that $\M_\al^{\vect,\ss}(\tau)$ is a simpler object, but in fact orientations on $\B_P$ are much easier to understand. In examples it is difficult to describe $\M_\al^{\vect,\ss}(\tau)$ explicitly. It could have $N\gg 0$ connected components, so that $\M_\al^{\vect,\ss}(\tau)$ would have $2^N$ orientations, but $\B_P$ is connected and so has only 2 orientations. Thus pulling back orientations from $\B_P$ to $\M_\al^{\vect,\ss}(\tau)$ gives orientations with fewer arbitrary choices.

Theorem \ref{fm13thm1} gives orientations not just on moduli spaces of vector bundles $\Vect(X)$, but also of coherent sheaves $\coh(X)$, and complexes in $D^b\coh(X)$. The rough analogue in Differential Geometry of passing from $\Vect(X)$ to $D^b\coh(X)$ is taking the limit $r\ra\iy$, for $r=\rank E$. More precisely, the analogue in Topology is passing from $\coprod_{r\ge 0}\Map_{C^0}(X,B\U(r))$ to $\Map_{C^0}(X,B\U\t\Z)$, where $B\U=\varinjlim_{n\ra\iy}B\U(n)$, and the $\Z$ factor keeps track of the rank~$r$. 

In the notation of \S\ref{fm7}, we can understand orientations in Theorem \ref{fm13thm1} as natural trivializations of an orientation functor
\begin{equation*}
F_X:\Bord_X(B\U\t\Z)_\top\longra\Z\qs\Z_2.
\end{equation*}
\end{rem}

\subsection{Orientability and canonical orientations for Calabi--Yau 4-fold moduli spaces}
\label{fm132}

We can now prove one of our main results:

\begin{thm}
\label{fm13thm2}
Let\/ $X$ be a Calabi--Yau\/ $4$-fold, and suppose that\/ $X$ satisfies condition Theorem\/ {\rm\ref{fm10thm2}$(*)$}. Then the moduli stacks\/ $\M$ of all objects in $D^b\coh(X),$ and\/ $\M_\al\subset\M$ of objects $F^\bu$ in $D^b\coh(X)$ with $\lb F^\bu\rb=\al\in K^0_\top(X),$ are orientable in the sense of Definition\/ {\rm\ref{fm13def2}}. A choice of flag structure $F$ on $X$ in the sense of\/ {\rm\S\ref{fm102}} determines an  orientation on the moduli stacks $\M,\M_\al$. Such orientations are necessary for defining \begin{bfseries}DT4 invariants\end{bfseries} of\/ $X,$ as in Borisov--Joyce\/ {\rm\cite{BoJo}} and Oh--Thomas\/ {\rm\cite{OhTh}}.

If\/ $c_2(\al)-c_1(\al)^2=0$ in $H^4(X,\Z)$ then we can construct a canonical orientation on $\M_\al$ without choosing a flag structure.

If\/ $X$ satisfies Theorem\/ {\rm\ref{fm10thm2}$(\dag)$} and\/ $c_2(\al)-c_1(\al)^2=0$ in $H^4(X,\Z_2)$ then we can construct a canonical orientation on $\M_\al$ without choosing a flag structure.
\end{thm}

\begin{proof}
Theorem \ref{fm12thm1}(b) shows that a flag structure $F$ on $X$ determines normalized orientations on $\check O^{E_\bu}_P\ra\cB_P$ for all principal $\U(m)$-bundles $P\ra X$. Since a Calabi--Yau 4-fold $X$ has a canonical trivialization of $\det\sD^+_X$, normalized orientations on $\check O^{E_\bu}_P\ra\cB_P$ are equivalent to orientations $O^{E_\bu}_P\ra\cB_P$. The first part of the theorem then follows from Theorem \ref{fm13thm1}(g). The last two parts follow from Proposition~\ref{fm12prop2}(b).
\end{proof}

\begin{ex}
\label{fm13ex1}
If $X$ is a smooth sextic in $\CP^5$ then $X$ is a Calabi--Yau 4-fold, and the Lefschetz Hyperplane Theorem implies that $H^3(X,\Z)=0$, so Theorem \ref{fm10thm2}$(*)$,$(\dag)$ hold trivially, and Theorem \ref{fm13thm2} applies. The same applies to Calabi--Yau 4-folds defined as complete intersections of ample hypersurfaces in smooth toric varieties, a large class.
\end{ex}

\begin{rem}
\label{fm13rem3}
{\bf(a)} The higher $\C$-stack $\M$ in Theorems \ref{fm13thm1} and \ref{fm13thm2} contains as open Artin $\C$-substacks the moduli stacks $\M^{\rm coh},\M^{\rm coh,ss},\M^{\rm vect}$ of coherent sheaves, and semistable coherent sheaves, and algebraic vector bundles on $X$, respectively. The principal $\Z_2$-bundle $O^{\cF^\bu}\ra\M$, and orientations on $\M$, may be restricted to $\M^{\rm coh},\ldots,\M^{\rm vect}$. Thus, Theorem \ref{fm13thm2} is still interesting if we only care about $\M^{\rm coh},\ldots,\M^{\rm vect}$ rather than~$\M$.
\smallskip

\noindent{\bf(b)} What Theorem \ref{fm13thm2} really means is that we have an {\it algorithm\/} for constructing orientations on $\M_\al$, which depends only on $X,\al$ and the flag structure $F$.

Note that other algorithms are possible, which would yield orientations on $\M_\al$ differing from those in Theorem \ref{fm13thm2} by a sign depending on natural invariants in the problem such as $\rank\al$, $\chi(X)$, $\int_Xc_1(\al)^4$ and $\int_Xc_2(\al)c_2(X)$. We have no way to say which of these algorithms is `best', if this even makes sense.

\smallskip

\noindent{\bf(c)} The authors do not know an example of a Calabi--Yau 4-fold $X$ for which Theorem \ref{fm10thm2}$(*)$ does not hold. But note from Example \ref{fm10ex1} that Theorem \ref{fm10thm2}$(*)$ fails for $\SU(3)$, which is a compact complex 4-manifold with trivial canonical bundle, and so a `non-K\"ahler Calabi--Yau 4-fold'.
\end{rem}

One frequent theme in the literature on DT4 invariants, which appears in \cite{Bojk2,Cao1,Cao2,CaKo1,CaKo2,CKM,CaLe,CMT1,CMT2,COT1,COT2,CaQu,CaTo1,CaTo2,CaTo3} and we summarize in Conjecture \ref{fm13conj1}, are relations of the form
\e
\text{Conventional invariants of $X$}\simeq \text{DT4 invariants of $X$,}
\label{fm13eq1}
\e
where by `conventional invariants' of $X$ we mean things like the Euler characteristic and Gromov--Witten invariants, and the relation `$\simeq$' may involve change of variables in a generating function, etc. 

\begin{conj}
\label{fm13conj1}
Let\/ $X$ be a projective Calabi--Yau\/ $4$-fold. Then:
\smallskip

\noindent{\bf(a)} Cao--Kool\/ {\rm\cite[Conj~1.1]{CaKo1}} propose an explicit generating function for invariants\/ $\int_{\Hilb^n(X)}c_n(L^{[n]})$ for\/ $L\ra X$ a line bundle. See also {\rm\cite{Bojk2,CaQu}}.
\smallskip

\noindent{\bf(b)} Bojko {\rm\cite{Bojk2}} proposes formulae for integrals of Segre classes, Verlinde classes and Nekrasov genera over\/ $\Hilb^n(X)$. 
\smallskip

\noindent{\bf(c)} Cao--Maulik--Toda {\rm\cite[Conj.~1.3]{CMT1}} relate genus\/ $0$ Gromov--Witten invariants of\/ $X$ and\/ $1$-dimensional DT4 invariants. Cao--Toda {\rm\cite[Conj.~1.2]{CaTo3}} make a related conjecture. See also {\rm\cite{Cao1,Cao2}}.
\smallskip

\noindent{\bf(d)} Cao--Maulik--Toda {\rm\cite[Conj.s 1.5 \& 1.6]{CMT2}} relate genus\/ $0,1$ Gromov--Witten invariants of\/ $X$ and Pandharipande--Thomas style DT4 invariants. Cao--Toda {\rm\cite[Conj.~1.6]{CaTo2}} make a related conjecture. See also {\rm\cite{CKM,CaTo1}}.
\smallskip

\noindent{\bf(e)} Cao--Kool\/ {\rm\cite[Conj.~1.1]{CaKo2}} relate genus\/ $0,1$ Gromov--Witten invariants of\/ $X$ and rank\/ $1$ DT4 invariants. See also {\rm\cite{CaLe}}.
\smallskip

\noindent{\bf(f)} For holomorphic symplectic $4$-folds $X,$ Cao--Oberdieck--Toda {\rm\cite[Conj.~2.2]{COT2}} relate reduced genus\/ $0,1,2$ Gromov--Witten invariants of\/ $X$ and reduced DT4 invariants counting $1$-dimensional sheaves, and also {\rm\cite[Conj.~1.10]{COT1}} to reduced Pandharipande--Thomas style DT4 invariants.\end{conj}

Although we state this as a conjecture, we emphasize that the cited papers also contain many theorems. All parts of Conjecture \ref{fm13conj1} involve only moduli spaces $\M_\al$ on $X$ with $c_1(\al)=c_2(\al)=0$ in $H^*(X,\Z)$. Thus the second paragraph of Theorem \ref{fm12thm2} implies:

\begin{cor}
\label{fm13cor1}
As in Conjecture\/ {\rm\ref{fm13conj1},} for a Calabi--Yau\/ $4$-fold\/ $X$ there are conjectures in {\rm\cite{Bojk2,Cao1,Cao2,CaKo1,CaKo2,CKM,CaLe,CMT1,CMT2,COT1,COT2,CaQu,CaTo1,CaTo2,CaTo3}} of the form \eq{fm13eq1} relating conventional invariants of\/ $X$ \textup(which require no choice of orientation\textup) and DT4 invariants of\/ $X$ \textup(which do require a choice of orientation\textup), an apparent paradox. If\/ $X$ satisfies Theorem\/ {\rm\ref{fm10thm2}$(*)$} then Theorem\/ {\rm\ref{fm13thm2}} provides canonical orientations for all the moduli spaces\/ $\M_\al$ occurring in Conjecture\/ {\rm\ref{fm13conj1},} resolving this paradox.
\end{cor}

\begin{rem}
\label{fm13rem4}
In the situation of Theorem \ref{fm13thm2}, supposing Theorem \ref{fm10thm2}$(*)$ holds, if $\al,\be\in K^0_\top(X)$ then given orientations $O_\al,O_\be$ on $\M_\al,\M_\be$, as in \cite[\S 3.3]{CGJ} we can construct an orientation $O_\al\star O_\be$ on $\M_{\al+\be}$, by relating orientations at $[E_\al^\bu]\in\M_\al$ and $[E_\be^\bu]\in\M_\be$ to that at $[E_\al^\bu\op E_\be^\bu]\in\M_{\al+\be}$. This satisfies $O_\be\star O_\al=(-1)^{\chi(\al,\al)\chi(\be,\be)+\chi(\al,\al)}O_\al\star O_\be$ and $(O_\al\star O_\be)\star O_\ga=O_\al\star(O_\be\star O_\ga)$. Thus, if we have constructed orientations $O_\al$ on $\M_\al$ for all $\al\in K^0_\top(X)$ then there are signs $\ep_{\al,\be}\in\{\pm 1\}$ with $O_{\al+\be}=\ep_{\al,\be}\,O_\al\star O_\be$ for all $\al,\be\in K^0_\top(X)$, which satisfy $\ep_{\al,\be}\ep_{\be,\al}=(-1)^{\chi(\al,\al)\chi(\be,\be)+\chi(\al,\al)}$ and~$\ep_{\al,\be+\ga}\ep_{\be,\ga}=\ep_{\al,\be}\ep_{\al+\be,\ga}$.

To develop a nice theory of DT4 invariants, it would be helpful if we could compute the signs $\ep_{\al,\be}$ for the orientations $O_\al$ constructed in Theorem \ref{fm13thm2} from a flag structure $F$. In general we cannot arrange that $\ep_{\al,\be}\equiv 1$ because of the identity~$\ep_{\al,\be}\ep_{\be,\al}=(-1)^{\chi(\al,\al)\chi(\be,\be)+\chi(\al,\al)}$. 

As in Remark \ref{fm13rem2}, we can interpret orientations on $\M$ in terms of an orientation functor $F_X:\Bord_X(B\U\t\Z)_\top\ra\Z\qs\Z_2$. Now $B\U\t\Z$ is a group-like H-space, and therefore $\Bord_X(B\U\t\Z)_\top$ is a Picard groupoid as in Remark \ref{fm7rem1}(d). The signs $\ep_{\al,\be}$ are related to obstructions to making $F_X$ and its trivialization monoidal. One might expect that all this is related to the question of defining additive flag structures in 8 dimensions discussed in Remark~\ref{fm10rem2}.

Unfortunately, our construction of orientations $O_\al$ on $\M_\al$ in Theorem \ref{fm13thm2} is not well adapted to computing the signs $\ep_{\al,\be}$. This is because the construction involves pulling orientations back along a functor 
\e
\Pi_{B\U\t\Z}^{K(\Z,4)}:\Bord_X(B\U\t\Z)_\top\longra \Bord_X(K(\Z,4)).
\label{fm13eq2}
\e
However, although both categories in \eq{fm13eq2} are symmetric monoidal, the functor $\Pi_{B\U\t\Z}^{K(\Z,4)}$ is {\it not\/} monoidal. To see this, note that on isomorphism classes of objects \eq{fm13eq2} maps $\al$ in $K^0_\top(X)$ to $c_2(\al)-c_1(\al)^2$ in $H^4(X,\Z)$, which is  not additive in $\al$. This non-additivity arises because we define orientations on $\B_P$ for $\U(m)$-vector bundles $E\ra X$ in \S\ref{fm12} by turning them into $\SU(m+1)$-vector bundles $E\op\La^mE^*\ra X$, which is not compatible with direct sums.
\end{rem}

\subsection{Extensions of the theory}
\label{fm133}

We now discuss applications of our orientability results to certain generalizations of DT4 invariants which either are under development, or may be developed in future. Much of what we outline has not been studied at the time of writing, and we claim no results in this section. The first author would like to thank Chenjing Bu for helpful discussions on these ideas.
\smallskip

\noindent{\bf(a)} Let $X$ be a projective Calabi--Yau 4-fold, and $G^{\sst\C}$ be a reductive complex algebraic $\C$-group. Consider the derived moduli stack $\bs\M$ of all algebraic principal $G^{\sst\C}$-bundles $P\ra X$, and open substacks $\bs\M^\ss_\al\subset\bs\M$ of bundles $P$ with fixed topological invariants $\al$, and satisfying some (semi)stability condition.

By Pantev--To\"en--Vaqui\'e--Vezzosi \cite[\S 2]{PTVV}, as $G^{\sst\C}$ is reductive, the quotient stack $[*/G^{\sst\C}]$ has a 2-shifted symplectic structure. Now $\bs\M$ is the derived mapping stack $\Map(X,[*/G^{\sst\C}])$, where $X$ is a Calabi--Yau 4-fold and $[*/G^{\sst\C}]$ is 2-shifted symplectic, so by \cite[\S 2]{PTVV}, $\bs\M$ and hence $\bs\M^\ss_\al\subset\bs\M$ have $-2$-shifted symplectic structures. The classical truncations $\M,\M^\ss_\al$ have 4-Calabi--Yau obstruction theories, as in \S\ref{fm131}. Thus as in Definition \ref{fm13def2} we have a notion of {\it orientation\/} on $\M$ and~$\M^\ss_\al$.

When $G^{\sst\C}=\GL(r,\C)$, algebraic principal $G$-bundles $P^{\sst\C}\ra X$ are equivalent to rank $r$ algebraic vector bundles $E\ra X$, which are examples of (torsion-free) coherent sheaves. The $-2$-shifted symplectic structures on $\bs\M,\bs\M^\ss_\al$ are the restrictions of those discussed in \S\ref{fm131} to the open substacks of vector bundles.

\smallskip

\noindent{\bf(b)} Let $G^{\sst\R}$ be the maximal compact subgroup of $G^{\sst\C}$, a real Lie group. Let $Q\ra X$ be a $C^\iy$ principal $G^{\sst\R}$-bundle, in the sense of Differential Geometry. Then we can consider moduli spaces $\B_Q$ of connections on $Q$, and orientations on $\B_Q$, as in \S\ref{fm121}.

As the inclusion $G^{\sst\R}\hookra G^{\sst\C}$ is a homotopy equivalence, for any $C^\iy$ principal $G^{\sst\C}$-bundle $P\ra X$ there exists a $C^\iy$ principal $G^{\sst\C}$-bundle $Q\ra X$, unique up to isomorphism, such that $P\cong Q\t_{G^{\sst\R}}G^{\sst\C}$. Given a $C^\iy$ principal $G^{\sst\R}$-bundle $Q\ra X$, write $\bs\M_Q\subset\bs\M$ for the open and closed substack of algebraic principal $G^{\sst\C}$-bundles $P\ra X$ whose underlying $C^\iy$ principal $G^{\sst\C}$-bundle $P_{C^\iy}\ra X$ has~$P_{C^\iy}\cong Q\t_{G^{\sst\R}}G^{\sst\C}$, and $\M_Q$ for the classical truncation of~$\bs\M_Q$.

The authors expect that there should be a generalization of Theorem \ref{fm13thm1} to principal $G^{\sst\C}$- and $G^{\sst\R}$-bundles, such that $\B_Q$ orientable implies that $\M_Q$ is orientable, and an orientation of $\B_Q$ lifts to an orientation of~$\M_Q$.

If this is true, then our orientability theory, and results such as Theorem \ref{fm12thm1}(b), could be used to prove orientability of, and construct canonical orientations for, moduli spaces $\M$ and $\M^\ss_\al$ for suitable groups~$G^{\sst\C}$.
\smallskip

\noindent{\bf(c)} It is natural to hope that one could build enumerative invariants `counting' (semistable) principal $G^{\sst\C}$-bundles on $X$, generalizing DT4 invariants in \S\ref{fm131}--\S\ref{fm132}. However, there is a problem with this. To form a CY4 virtual class, one needs {\it proper\/} moduli schemes. 

It is well known in Algebraic Geometry that if $X$ is smooth projective $\C$-scheme with $\dim X>1$, then moduli schemes of semistable rank $r$ vector bundles (i.e.\ principal $\GL(r,\C)$-bundles) are typically not proper (i.e.\ non-compact). To compactify the moduli spaces, we need to enlarge them to moduli schemes of semistable rank $r$ {\it torsion-free sheaves}.

So, we want a generalization of principal $G^{\sst\C}$-bundles, analogous to torsion-free sheaves, for which we expect semistable moduli schemes to be proper. There is a literature on this, some important papers are G\'omez--Sols \cite{GoSo} and Fernandez Herrero--G\'omez--Zamora \cite{FGZ}. Given a representation $\rho:G^{\sst\C}\hookra\GL(r,\C)$, they define a {\it principal\/ $\rho$-sheaf\/} $(\cE,P,\psi)$ to be a rank $r$ torsion-free sheaf $\cE$ on $X$ and an isomorphism $\psi:\cE\vert_U\ra P\t_{G^{\sst\C},\rho}\C^r$, for $U\subseteq X$ the dense open subset where $\cE$ is locally free, and $P\ra U$ a principal $G^{\sst\C}$-bundle. G\'omez--Sols \cite{GoSo} show that principal $\rho$-sheaves have projective (hence proper) coarse moduli schemes.

Unfortunately, the authors do not expect the $-2$-shifted symplectic structure on moduli stacks of principal $G^{\sst\C}$-bundles to extend to moduli stacks of principal $\rho$-sheaves. So it may not be possible to use \cite{FGZ,GoSo} to define DT4 invariants `counting' principal $G^{\sst\C}$-bundles.
\smallskip

\noindent{\bf(d)} When $G^{\sst\C}$ is $\O(r,\C),\SO(r,\C),\Spin(r,\C)$ or $\Sp(r/2,\C)$, there is an alternative method for compactifying moduli spaces of principal $G^{\sst\C}$-bundles that the authors expect will have the good properties we want.

Observe that a principal $\O(r,\C)$-bundle $P\ra X$ is equivalent to a rank $r$ vector bundle $\cE\ra X$ with an isomorphism $\om:\cE\ra\cE^\vee$ with $\om=\om^\vee$. Similarly, for $r$ even, a principal $\Sp(r/2,\C)$-bundle $P\ra X$ is equivalent to a rank $r$ vector bundle $\cE\ra X$ with an isomorphism $\om:\cE\ra\cE^\vee$ with $\om=-\om^\vee$. We propose that moduli spaces of semistable principal $\O(r,\C)$- or $\Sp(r/2,\C)$-bundles should be compactified by enlarging them to moduli spaces of rank $r$ perfect complexes $\cE^\bu$, semistable under a suitable stability condition, with an isomorphism $\om:\cE^\bu\ra(\cE^\bu)^\vee$ in $D^b\coh(X)$ with $\om=\pm\om^\vee$.

Derived moduli stacks of such $(\cE^\bu,\om)$ are the fixed points $\bs\M^{\Z_2}$ of a $\Z_2$-action on the derived moduli stack $\bs\M$ of perfect complexes $\cE^\bu$. Since $\bs\M$ is $-2$-shifted symplectic by \cite[Cor.~2.13]{PTVV}, it follows that $\bs\M^{\Z_2}$ is $-2$-shifted symplectic. To generalize $G^{\sst\C}=\O(r,\C)$ to $\SO(r,\C)$ or $\Spin(r,\C)$, we add an orientation, or a spin structure, to $(\cE^\bu,\om)$. This makes sense, and is well behaved.

The first author's PhD student Chenjing Bu \cite{Bu1,Bu2} is developing the foundations of an exciting theory of enumerative invariants `counting' $\O(r,\C)$- or $\Sp(r/2,\C)$-bundles compactified in this way, and their wall crossing formulae. There is not a Calabi--Yau 4-fold version of the theory available at the time of writing, but it seems clear that this should be possible.
\smallskip

\noindent{\bf(e)} Write $\bs\M^\O,\bs\M^\SO,\bs\M^\Spin,\bs\M^\Sp,$ for the $-2$-shifted symplectic derived moduli stacks of pairs $(\cE^\bu,\om)$ on a Calabi--Yau 4-fold $X$ in {\bf(d)}, with orientations or spin structures for $\bs\M^\SO,\bs\M^\Spin$. We decompose $\bs\M^\O=\bs\M^\O_{\rm ev}\amalg \bs\M^\O_{\rm od}$ into substacks with $\rank\cE^\bu$ even or odd, and similarly for $\bs\M^\SO,\bs\M^\Spin$. 

The authors expect that there should be a generalization of Theorem \ref{fm13thm1} which says (roughly) that orientations of $\bs\M^\O_{\rm ev}$ (or $\bs\M^\O_{\rm od}$) can be pulled back from orientations of $\B_Q$ in \S\ref{fm121} for $Q\ra X$ a principal $\O(2m)$-bundle (or a principal $\O(2m+1)$-bundle, respectively) for $m\gg 0$, where we stabilize using the inclusions $\O(2m)\hookra \O(2m+2)$ or $\O(2m+1)\hookra \O(2m+3)$ (note that these are of {\it complex type}, whereas $\O(2m)\hookra \O(2m+1)$ is not), and similarly for $\bs\M^\SO_{\rm ev},\bs\M^\SO_{\rm od},\bs\M^\Spin_{\rm ev},\bs\M^\Spin_{\rm od}$ and $\bs\M^\Sp$. Since $\Spin(2m)$ lies on the list \eq{fm11eq19}, from Theorem \ref{fm12thm1}(b) we make the prediction that if $X$ satisfies condition Theorem \ref{fm10thm2}$(*)$ then $\bs\M^\Spin_{\rm ev}$ should be orientable, and a choice of flag structure on $X$ should determine an orientation on $\bs\M^\Spin_{\rm ev}$.

\section[Applications to moduli spaces of submanifolds]{Applications to moduli spaces of \\ submanifolds}
\label{fm14}

In  \S\ref{fm142}--\S\ref{fm143} we will study orientations on various moduli spaces $\M_\al$ from calibrated geometry. In good cases $\M_\al$ is a manifold. Our approach is to write down an embedding $\M_\al\hookra\cB_\al$ of $\M_\al$ into an {\it infinite-dimensional\/} moduli space $\cB_\al$ defined in \S\ref{fm141}. If $\M_\al$ is a moduli space of submanifolds $M\subset X$ satisfying a p.d.e., then $\cB_\al$ is the moduli space of all submanifolds in the same L-equivalence class~$\al\in\La_k^H(X)$.

We will construct principal $\Z_2$-bundles $O_{\cB_\al}\ra\cB_\al$ such that the orientation bundle $O_{\M_\al}\ra\M_\al$ is $O_{\cB_\al}\vert_{\M_\al}$. Thus, an orientation of $\cB_\al$ (that is, a trivialization $O_{\cB_\al}\cong\cB_\al\t\Z_2$) induces an orientation of $\M_\al$. We will show that orientations of $\cB_\al$ are induced by orientations of analytic orientation functors on $X$, as in \S\ref{fm94}. Hence by the results of \S\ref{fm9} and \S\ref{fm10}, we can sometimes show that $\cB_\al$ and thus $\M_\al$ are orientable, and that a flag structure on $X$ induces an orientation on $\cB_\al$ and hence on~$\M_\al$.

\subsection{Infinite-dimensional moduli spaces of submanifolds}
\label{fm141}

The next definition shows how the analytic orientation functors of \S\ref{fm94} lead to principal $\Z_2$-bundles on infinite-dimensional moduli spaces of submanifolds. The restrictions to codimension 4 submanifolds, and to $n\equiv 1,7,8\mod 8$, are to work with these particular orientation functors.

\begin{dfn}
\label{fm14def1}
Let $(X,g_X)$ be a compact, oriented, spin Riemannian manifold of dimension $n\ge 4$, and let $\rho:H\ra\SO(4)\subset\O(4)$ be a Lie group morphism. Using the notation of Definition \ref{fm2def6}, for each $\al$ in $\La_{n-4}^H(X)$ write $\cB_\al^H$ for the infinite-dimensional moduli space of pairs $(M,\ga_M)$ where $M\subset X$ is a compact, oriented $(n-4)$-submanifold and $\ga_M$ is an isotopy class of normal $H$-structures for $M\subset X$ with $[M,\ga_M]=\al$. By Hamilton \cite[Ex.~4.4.7]{Hami}, $\cB_\al^H$ has the structure of a Fr\'echet manifold, though we care about $\cB_\al^H$ primarily as a topological space.

Suppose that $n\equiv 1,7,8\mod 8$. Define principal $\Z_2$-bundles $O_\al^+,O_\al^-,O_\al^0\ra\cB_\al^H$ to have fibres at $M\in\cB_\al^H$
\e
O_\al^\pm\vert_M=\begin{cases} \Or(\Ker F_M^\pm), & n\equiv 1\mod 8, \\
\Or(\Ker F_M^\pm)\ot_{\Z_2}\Or(\Coker F_M^\pm), & n\equiv 7,8\mod 8, \end{cases}
\label{fm14eq1}
\e
and $O_\al^0\vert_M=O_\al^+\vert_M\ot_{\Z_2}O_\al^-\vert_M$, where $F_M^\pm$ are the Fueter operators of \eq{fm9eq4}, and $\Or(V)$ is the $\Z_2$-torsor of orientations on $V$. Here $F_M^\pm$ is linear over $\K_n$ in Table \ref{fm9tab1}, and we restrict to $n\equiv 1,7,8\mod 8$ as otherwise $F_M^\pm$ is linear over $\C$ or $\H$ and $\Or(\cdots)$ are canonically trivial. As in Definition \ref{fm9def3}, by the constructions of determinant line bundles when $n\equiv 7,8\mod 8$, and Pfaffian line bundles when $n\equiv 1\mod 8$, these are the fibres of principal $\Z_2$-bundles.
\end{dfn}

\begin{prop}
\label{fm14prop1}
In Definition\/ {\rm\ref{fm14def1},} an orientation for $X$ in the sense of Definition\/ {\rm\ref{fm9def2}} of the orientation functors $\sO_{n,4}^{\bs\Spin,H,+},$ $\sO_{n,4}^{\bs\Spin,H,-},$ $\sO_{n,4}^{\bs\Spin,H,0}$ in Definition\/ {\rm\ref{fm9def6},} or their $\Z_2$-reductions $\sO_{n,4,\Z_2}^{\bs\Spin,H,+},$ $\sO_{n,4,\Z_2}^{\bs\Spin,H,-},$ $\sO_{n,4,\Z_2}^{\bs\Spin,H,0}$ when $n\equiv 7\mod 8,$ induces trivializations of the principal\/ $\Z_2$-bundles $O_\al^+\ra\cB_\al^H,$ $O_\al^-\ra\cB_\al^H,$ $O_\al^0\ra\cB_\al^H,$ respectively, for all\/~$\al\in\La_{n-4}^H(X)$. 
\end{prop}

\begin{proof}
By Definition \ref{fm9def2}, an orientation of $\sO_{n,4}^{\bs\Spin,H,+}$ for $X$ is a natural isomorphism $\eta_X:F_{A\qs\Z_2}^{0\qs\Z_2}\ci\sO_{n,4}^{\bs\Spin,H,+}\ci\Pi_X^{\bs\Spin}\Ra\boo$. But unwinding the definitions shows that $F_{A\qs\Z_2}^{0\qs\Z_2}\ci\sO_{n,4}^{\bs\Spin,H,+}\ci\Pi_X^{\bs\Spin}$ maps $M\mapsto O_\al^+\vert_M$ for $\al=[M]$. Thus $\eta_X(M)$ gives an isomorphism~$O_\al^+\vert_M\ra\Z_2$.

We claim these isomorphisms depend continuously on $M\in\cB_\al^H$, and so define trivializations of $O_\al^+\ra\cB_\al^H$ for all $\al\in\La_{n-4}^H(X)$. To see this, note that a Fr\'echet-smooth map $\ga:[0,1]\ra\cB_\al^H$ gives a smooth 1-parameter family of $(n-4)$-submanifolds $M_t\subset X$ for $t\in[0,1]$, so that $N=\coprod_{t\in[0,1]}M_t\t\{t\}$ is an $(n-3)$-submanifold of $X\t[0,1]$, which defines a morphism $[N]:M_0\ra M_1$ in $\Bord_X^{n-4}(MH)$. The compatibility of $\eta_X$ with morphisms in $\Bord_X^{n-4}(MH)$ then implies that the isomorphisms $O_\al^+\vert_{M_t}\ra\Z_2$ are continuous in $t\in[0,1]$. The cases $\sO_{n,4}^{\bs\Spin,H,-},$ $\sO_{n,4}^{\bs\Spin,H,0}$ are the same. The proposition follows.
\end{proof}

\begin{rem}
\label{fm14rem1}
Although classes $\al\in\La_{n-4}^H(X)$ correspond to isomorphism classes in $\Bord_X^{n-4}(MH)$, the spaces $\cB_\al^H$ have many connected components, as for $M_0,M_1\subset X$ to be L-equivalent is weaker than for them to be connected by a smooth family of diffeomorphic submanifolds $M_t\subset X$ for $t\in[0,1]$. So there are many more choices of orientations for $O_\al^+\ra\cB_\al^H$ than those coming from an orientation of $\sO_{n,4}^{\bs\Spin,H,+}$ for~$X$.
\end{rem}

Combining Theorem \ref{fm11thm5}, Proposition \ref{fm14prop1}, and the material on flag structures in \S\ref{fm10}, yields:

\begin{thm}
\label{fm14thm1}
{\bf(a)} In Definition\/ {\rm\ref{fm14def1},} suppose\/ $(X,g_X)$ is a compact spin Riemannian $7$-manifold. Then $O_\al^0\ra\cB_\al^H$ is trivializable for any $\rho:H\ra\SO(4)$ and all\/ $\al\in\La_3^H(X),$ and\/ $O_\al^-\ra\cB_\al^H$ is trivializable for any $\rho:H\ra\SO(4)$ which factors via $\U(2)$ or $\Spin(4),$ and all\/ $\al\in\La_3^H(X)$. In each case a choice of flag structure $F$ on $X$ determines trivializations of\/ $O_\al^0,O_\al^-\ra\cB_\al^H$. 
\smallskip

\noindent{\bf(b)} In Definition\/ {\rm\ref{fm14def1},} suppose\/ $(X,g_X)$ is a compact spin Riemannian $8$-mani\-fold satisfying Theorem\/ {\rm\ref{fm10thm2}$(*)$}. Then $O_\al^0\ra\cB_\al^H$ and\/ $O_\al^-\ra\cB_\al^H$ are trivializable for any $\rho:H\ra\SO(4)$ which factors via\/ $\U(2)\hookra\SO(4)$ or\/ $\Spin(4)\twoheadrightarrow\SO(4),$ and all\/ $\al\in\La_4^H(X)$. In each case a choice of flag structure $F$ on $X$ determines trivializations of\/ $O_\al^0,O_\al^-\ra\cB_\al^H$. 
\end{thm}

\subsection{\texorpdfstring{Associative 3-folds in $G_2$-manifolds}{Associative 3-folds in G₂-manifolds}}
\label{fm142}

See the first author \cite[\S 10]{Joyc1} for background on $G_2$ and associative 3-folds.

\begin{dfn}
\label{fm14def2}
Let $(X,\vp,g)$ be a compact $G_2$-manifold, as in Definition \ref{fm12def3}, and $M\subset X$ be an oriented 3-submanifold. We call $M$ an {\it associative\/ $3$-fold\/} if $M$ is calibrated with respect to $\vp$, that is, if $\vp\vert_M=\vol_M$ is the volume form of $M$, defined using the orientation of $M$ and the Riemannian metric $g\vert_M$. (We do not need $\d\vp=0$, though this is part of the usual definition of calibration.) 

Let $\rho:H\ra\SO(4)\subset\O(4)$ be a Lie group morphism. Using the notation of Definition \ref{fm2def6}, for each $\al$ in $\La_3^H(X)$ write $\M_\al^{\ass,H}$ for the moduli space of pairs $(M,\ga_M)$ where $M\subset X$ is a compact associative 3-fold and $\ga_M$ is an isotopy class of normal $H$-structures for $M\subset X$ with $[M,\ga_M]=\al$, such that the orientation induced by $\ga_M$ on the fibres of $\nu_M$ agrees with that induced by the orientations on $X$ and $M$. Then $\M_\al^{\ass,H}$ is a subset of $\cB_\al^H$ in Definition~\ref{fm14def1}. 

As in McLean \cite[\S 5]{McLe}, the deformation theory of $\M_\al^{\ass,H}$ is controlled by an elliptic operator, so one can show that $\M_\al^{\ass,H}$ is a derived manifold of virtual dimension 0 in the sense of \cite{Joyc2,Joyc3,Joyc4,Joyc6}. If $\vp$ is generic then $\M_\al^{\ass,H}$ is an ordinary 0-manifold. Examples of compact associative 3-folds in compact 7-manifolds with holonomy $G_2$ are given in~\cite[\S 12]{Joyc1}. 
\end{dfn}

The next theorem was stated as a conjecture for $H=\SO(4)$, and partially proved, in the first author \cite[\S 3.2]{Joyc5}. The proof given there was complete when $\M_\al^{\ass,\SO(4)}$ is unobstructed. We give a new proof.

\begin{thm}
\label{fm14thm2}
Let\/ $(X,\vp,g)$ be a compact\/ $G_2$-manifold and\/ $\rho:H\ra\SO(4)\subset\O(4)$ be a Lie group morphism. Then a choice of flag structure $F$ on $X$ in the sense of\/ {\rm\S\ref{fm101}} determines orientations on the moduli spaces of associative $3$-folds $\M_\al^{\ass,H}$ for all\/ $\al\in \La_3^H(X)$.
\end{thm}

\begin{proof}
By McLean \cite[Th.~5.2]{McLe}, the deformation complex of\/ $M\subset X$ can be identified with the Fueter operator from \eq{fm9eq4} for\/ $n=7,$
\begin{equation*}
F_M^-=(\sD_M)^{\Si_\nu^-}:\Ga^\iy(E_0\ot_\H\Si_\nu^-)\longra\Ga^\iy(E_1\ot_\H\Si_\nu^-).
\end{equation*}
Hence the orientation bundle of $\M_\al^{\ass,H}$ is canonically isomorphic to the restriction to $\M_\al^{\ass,H}\subset\cB_\al^H$ of the principal $\Z_2$-bundle $O_\al^-\ra\cB_\al^H$ in Definition~\ref{fm14def1}.

Now Theorem \ref{fm14thm1}(a) does {\it not\/} tell us that $O_\al^-\ra\cB_\al^H$ is trivializable. In fact, Corollary \ref{fm11cor2} says $\sO_{7,4,\Z_2}^{\bs\Spin,\SO(4),-}$ can be non-orientable for some spin 7-manifolds $X$, so $O_\al^-\ra\cB_\al^H$ could conceivably be nontrivial for such $X$. So one might expect that associative moduli spaces could be non-orientable.

To get round this, note that as $M$ is associative there are canonical isomorphisms $E_0\cong E_1\cong\Si_\nu^+$. Thus we can identify $E_i\ot_\H\Si_\nu^+\cong\La^0T^*M\op\La^1T^*M.$ The Fueter operator $F_M^+$ can be deformed into the Hodge--de Rham operator
\begin{equation*}
\ast(\d+\d^*):\Ga^\iy(\La^0T^*M\op\La^2T^*M)\longra\Ga^\iy(\La^0T^*M\op\La^2T^*M).
\end{equation*}
Since the orientation $\Z_2$-torsor of $\ast(\d+\d^*)$ has a canonical trivialization, it follows that the restriction of $O_\al^+\ra\cB_\al^H$ to $\M_\al^{\ass,H}\subset\cB_\al^H$ is canonically trivial. (But see Remark \ref{fm14rem2}(c) for a warning about this.)

Combining the two facts above, the orientation bundle of\/ $\M_\al^{\ass,H}$ is canonically isomorphic to the restriction to $\M_\al^{\ass,H}\subset\cB_\al^H$ of the principal\/ $\Z_2$-bundle $O_\al^0\ra\cB_\al^H$. Thus, trivializations of $O_\al^0\ra\cB_\al^H$ restrict to orientations of $\M_\al^{\ass,H}$. The theorem now follows from Theorem~\ref{fm14thm1}(a).
\end{proof}

\begin{rem}
\label{fm14rem2}
{\bf(a)} Orienting moduli spaces of associative 3-folds is important in the Donaldson--Segal programme \cite{DoSe}, which aims to define enumerative invariants of $(X,\vp,g)$ by counting $G_2$-instantons and associative 3-folds on $X$, with signs.
\smallskip

\noindent{\bf(b)} One could ask whether we can use our methods to prove existence of {\it Floer gradings\/} in $\Z$ or $\Z_k$ for $k>2$ on moduli spaces of associative 3-folds, that one could use in a conjectural Floer theory defined using associative 3-folds in $X$ and asymptotically cylindrical Cayley 4-folds in $X\t\R$. However, Corollary \ref{fm11cor2} shows that there must be obstructions to doing this.
\smallskip

\noindent{\bf(c)} Here is a rather confusing point. In dimension 7, at least for torsion-free $G_2$-structures, the elliptic operators $F_M^\pm$ and $\ast(\d+\d^*)$ are {\it self-adjoint}, so that $\Ker F_M^\pm\cong\Coker F_M^\pm$, and $\Or(\Ker F_M^\pm)\cong\Or(\Coker F_M^\pm)$, and in \eq{fm14eq1} we have a canonical isomorphism $\Or(\Ker F_M^\pm)\ot_{\Z_2}\Or(\Coker F_M^\pm)\cong\Z_2$. Na\"\i vely, this might lead us to believe that the $\Z_2$-bundles $O_\al^+,O_\al^-,O_\al^0\ra\cB_\al^H$ in Definition \ref{fm14def1} are canonically trivial.

In fact this is {\it false}. Although we have canonical isomorphisms $O_\al^*\vert_M\cong\Z_2$ at each point $M\in\cB_\al^H$, these isomorphisms {\it do not depend continuously on\/} $M\in\cB_\al^H$, and can jump discontinuously when $\Ker F_M^\pm$ jump in dimension, so they do {\it not\/} give trivializations of the principal $\Z_2$-bundles~$O_\al^*\ra\cB_\al^H$.

The proof of Theorem \ref{fm14thm2} is an exception to this. On $\M_\al^{\ass,H}\subset\cB_\al^H$ we can replace $F_M^+$ by $\ast(\d+\d^*)$. As the kernel of $\ast(\d+\d^*)$ is isomorphic to $H^0(M,\R)\op H^2(M,\R)$ it cannot jump in dimension as $M$ varies smoothly, so the induced canonical isomorphisms $O_\al^+\vert_M\cong\Z_2$ depend continuously on $M$ in $\M_\al^{\ass,H}$, and $O_\al^+\vert_{\M_\al^{\ass,H}}$ is canonically trivial. We do not expect $O_\al^+$ to be trivial on $\cB_\al^H$ in general.
\end{rem}

\subsection{\texorpdfstring{Cayley 4-folds in $\Spin(7)$-manifolds}{Cayley 4-folds in Spin(7)-manifolds}}
\label{fm143}

See the first author \cite[\S 10]{Joyc1} for background on $\Spin(7)$ and Cayley 4-folds.

\begin{dfn}
\label{fm14def3}
Let $(X,\Om,g)$ be a compact $\Spin(7)$-manifold, as in Definition \ref{fm12def4}, and let $M\subset X$ be an oriented 4-submanifold. We call $M$ a {\it Cayley\/ $4$-fold\/} if $M$ is calibrated with respect to $\Om$, that is, if $\Om\vert_M=\vol_M$ is the volume form of $M$, defined using the orientation of $M$ and the Riemannian metric $g\vert_M$. (We do not need $\d\Om=0$ here.) 

Let $\rho:H\ra\SO(4)\subset\O(4)$ be a Lie group morphism. Using the notation of Definition \ref{fm2def6}, for each $\al$ in $\La_4^H(X)$ write $\M_\al^{\Cay,H}$ for the moduli space of pairs $(M,\ga_M)$ where $M\subset X$ is a compact Cayley 4-fold and $\ga_M$ is an isotopy class of normal $H$-structures for $M\subset X$ with $[M,\ga_M]=\al$, such that the orientation induced by $\ga_M$ on the fibres of $\nu_M$ agrees with that induced by the orientations on $X$ and $M$. Then $\M_\al^{\Cay,H}$ is a subset of $\cB_\al^H$ in Definition~\ref{fm14def1}. 

By McLean \cite[Th.~6.3]{McLe}, the deformation complex of $M\subset X$ can be identified with the elliptic Fueter operator from \eq{fm9eq4} for $n=8,$
\begin{equation*}
F_M^-=(\sD_M^+)^{\Si_\nu^-}:\Ga^\iy(\Si_M^+\ot_\H\Si_\nu^-)\longra\Ga^\iy(\Si_M^-\ot_\H\Si_\nu^-).
\end{equation*}
The index of\/ $F_M^-$ is given by the Atiyah--Singer index formula as
\e
\ind F_M^-=\int_M\wh{A}(TM)\ch(\Si^-_\nu)=\frac{\chi(M)+\sign(M)}{2}-[M]\bu[M].
\label{fm14eq2}
\e
Using this one can show that $\M_\al^{\Cay,H}$ is a derived manifold of virtual dimension \eq{fm14eq2} in the sense of \cite{Joyc2,Joyc3,Joyc4,Joyc6}. If $\Om$ is generic then $\M_\al^{\Cay,H}$ is an ordinary manifold of dimension \eq{fm14eq2}. Examples of compact Cayley 4-folds in compact 8-manifolds with holonomy $\Spin(7)$ are given in \cite[\S 14]{Joyc1}. From Definition \ref{fm14def1}, we see that the orientation bundle $O_\al^{\Cay,H}$ of\/ $\M_\al^{\Cay,H}$ is the restriction to $\M_\al^{\Cay,H}\subset\cB_\al^H$ of $O_\al^-\ra\cB_\al^H$. Hence Theorem \ref{fm14thm1}(b) implies:
\end{dfn}

\begin{thm}
\label{fm14thm3}
Let\/ $(X,\Om,g)$ be a compact\/ $\Spin(7)$-manifold satisfying Theorem\/ {\rm\ref{fm10thm2}$(*)$,} and suppose $\rho:H\ra\SO(4)$ factors via $\U(2)\hookra\SO(4)$ or\/ $\Spin(4)\ra\SO(4)$. Then the moduli spaces $\M_\al^{\Cay,H}$ in Definition\/ {\rm\ref{fm14def3}} are orientable for all\/ $\al\in\La_4^H(X)$. A choice of flag structure $F$ on $X$ in the sense of\/ {\rm\S\ref{fm102}} determines orientations on $\M_\al^{\Cay,H}$ for all\/~$\al\in\La_4^H(X)$.
\end{thm}

Theorem \ref{fm14thm3} is one of our main results. We know of no other results in the literature on orientability of moduli spaces of Cayley 4-folds.

\begin{rem}
\label{fm14rem3}
{\bf(a)} Since $X$ is spin, if $M\subset X$ is a Cayley 4-fold then spin structures on the normal bundle $\nu_M$ of $M\subset X$ are equivalent to spin structures on $M$ by 2-out-of-3. Thus taking $H=\Spin(4)$ in Theorem \ref{fm14thm3} implies that if $(X,\Om,g)$ is a compact $\Spin(7)$-manifold satisfying Theorem \ref{fm10thm2}$(*)$ then moduli spaces of compact, {\it spin\/} Cayley 4-folds $M\subset X$ are orientable, and a flag structure $F$ on $X$ determines orientations on them.
\smallskip

\noindent{\bf(b)} Although studying moduli spaces $\M_\al^{\Cay,H}$ of Cayley 4-folds $M\subset X$ equipp\-ed with a normal $H$-structure up to isotopy may seem unnatural, these do occur in gauge theory problems. In a similar way to the $G_2$-instanton case discussed in Donaldson and Segal \cite{DoSe}, it is expected that a sequence $(\nabla_n)_{n=1}^\iy$ of $\Spin(7)$-instantons on a principal $\SU(2)$-bundle $P\ra X$ over a $\Spin(7)$-manifold $(X,\Om,g)$ may `bubble' (develop removable singularity) along a Cayley 4-fold~$M\subset X$. 

If the limit connection $\nabla_\iy$ is trivial on $X\sm M$, one can show that $M$ has a normal $\SU(2)$-structure natural up to isotopy. If instead $P$ is an $\SU(m)$-bundle for $m\ge 3$ then $M$ has a normal $\U(2)$-structure natural up to isotopy. So the cases $H=\SU(2),\U(2)$ in Theorem \ref{fm14thm3} are relevant to orienting compactifications of moduli spaces of $\Spin(7)$-instantons with structure group~$\SU(m)$.
\end{rem}

\section{Proof of Theorem \ref{fm3thm1}}
\label{fm15}

For each of $T=M\SU(2)$, $M\U(2)$, $M\Spin(4)$, $M\SO(4)$, and $K(\Z,4)$ we will proceed with the following steps, which will prove Theorem \ref{fm3thm1}(a),(d). Later in \S\ref{fm158}--\S\ref{fm159} we will prove Theorem~\ref{fm3thm1}(b),(c).
\begin{itemize}
\setlength{\itemsep}{0pt}
\setlength{\parsep}{0pt}
\item[(i)] Compute $\ti\Om_n^{\bs\Spin}(T)$ using the Atiyah--Hirzebruch spectral sequence \eq{fm2eq8}, as explained in \S\ref{fm22}. To do this, we first need to determine $\ti H_p(T,\Z)$, $\ti H_p(T,\Z_2)$ and the action of the Steenrod square $\Sq^2$. By Proposition \ref{fm2prop1} this will provide the differentials of the spectral sequence on the $E^2$-page. Possible higher differentials will be computed using the naturality of the spectral sequence, leading to the $E^\iy$-page.
\item[(ii)] Verify the isomorphisms \eq{fm3eq9}--\eq{fm3eq21}, thus proving Theorem~\ref{fm3thm1}(d).
\item[(iii)] Solve the extension problems \eq{fm2eq9} and verify the generators listed in Table \ref{fm3tab1}, partly using step (ii). This proves Theorem~\ref{fm3thm1}(a).
\end{itemize}

\subsection{\texorpdfstring{Computation of $\ti\Om_n^{\bs{\mathrm{Spin}}}(K(\Z,4))$}{Computation of Ωₙˢᵖⁱⁿ(K(ℤ,4))}}
\label{fm151}

\subsubsection{Description of the (co)homology}
\label{fm1511}

We describe the (co)homology of $K(\Z,4)$ and choose notation for the generators.

\begin{nota}
\label{fm15nota1}
A bar accent `$\bar{*}$' indicates the specialization to $\Z_2$-(co)hom\-ol\-ogy of a symbol $*$ in $\Z$-(co)homology. In cohomology $e_4^3,e_4e_6',\ldots$ mean cup products of $e_4,e_6',\ldots$; cohomology variables $e_i$ originate from $H^*(K(\Z,4),\Z)$, and $\bar e'_i$ are variables in $H^*(K(\Z,4),\Z_2)$ not coming from $H^*(K(\Z,4),\Z)$. Where there are nontrivial pairings between homology and cohomology over $\Z$ or $\Z_2$, we write $\ep_i,\ep_i'$ for the homology variables dual to $\ep_i,\ep_i'$, so for example $\an{e_4,\ep_4}=1$. Homology products such as $\bar\ep_4\bar\ep_6'$ mean the homology class dual to $\bar e_4\bar e_6'$, in the given bases we have written down.
\end{nota}

From Serre \cite[Th.~3]{Serr} we find that in the range $n\le 10$, the $\Z_2$-cohomology $H^*(K(\Z,4),\Z_2)$ has generators $\bar e_4,\bar e_6',\bar e_7,\bar e_{10}'$, where
\e
\bar e_6'=\Sq^2(\bar e_4), \qquad \bar e_7=\Sq^3(\bar e_4), \qquad \bar e_{10}'=\Sq^4(\bar e_6')=\Sq^4\Sq^2(\bar e_4).
\label{fm15eq1}
\e
Since $\Sq^2\circ\Sq^2=\Sq^3\circ\Sq^1$ by \eq{fm2eq15} and $\Sq^1(\bar e_4)=0$ as $H^5(K(\Z,4),\Z_2)=0$, we see that $\Sq^2(\bar e_6')=0$. Also $\Sq^2(\bar e_4^2)=0$ by the Cartan formula.

The homology groups $H_{n+i}(K(\Z,n),\Z)$ are computed by Breen--Mikhailov--Touz\'e \cite[App.~B]{BMT} for $n\le 11$ and $i\le 10$ (alternatively, one could use for this the Bockstein spectral sequence described in \S\ref{fm1571} below). This leads to the following table.
\begin{equation}
\begin{aligned}
&\!\begin{tabular}{c|ccccccccccc}
$n$ & 
$\!\!0,1,2,3,5,7,9\!\!\!\!\!\!\!$ &  $4$  & $6$  & $8$ & $10$ \\
\hline
\parbox[top][4ex][c]{2cm}{$\!\!\!\ti H_n(K(\Z,4),\Z)\!\!\!$} & 
$0$ &  $\!\!\Z\an{\ep_4}\!\!$ & $\!\!\Z_2\an{\ep_6'}\!\!$ & $\!\!\Z\an{\ep_4^2}\!\op\!\Z_3\an{\ep_8}\!\!$ & $\!\!\!\Z_2\an{\ep_{10}',\ep_4\ep_6'}\!\!\!$
\end{tabular}\!\!
\\
&\begin{tabular}{c|cccccccccccc}
$n$ & 
$0,1,2,3,5,6,10$ & $4$ & $7$ &  $8$ & $9$   \\
\hline
\parbox[top][4ex][c]{2cm}{$\!\!\!\!\ti H^n(K(\Z,4),\Z)\!\!\!$} & 
$0$ &  $\Z\an{e_4}$ & $\Z_2\an{e_7}$  & $\Z\an{e_4^2}$ & $\Z_3\an{e_9}$ 
\end{tabular}
\\
&\!\!\begin{tabular}{c|ccccccccccc}
$n$ & 
$\!\!0,1,2,3,5,9\!\!\!\!\!\!$ & $4$ & $6$ & $7$ & $8$ & $10$ \\
\hline
\parbox[top][4ex][c]{2.1cm}{$\!\!\!\ti H_n(K(\Z,4),\Z_2)\!\!\!\!\!$} & 
$0$ & $\!\!\Z_2\an{\bar\ep_4}\!$ & $\!\!\Z_2\an{\bar\ep_6'}\!\!\!$ & $\!\!\Z_2\an{\bar\ep_7}\!\!$ & $\!\!\Z_2\an{\bar\ep_4^2}\!\!\!$ & $\!\!\!\Z_2\an{\bar\ep_{10}',\bar\ep_4\bar\ep_6'}\!\!\!$ \end{tabular}
\\
&\begin{tabular}{c|ccccccccccc}
$n$ & 
$\!\!0,1,2,3,5,9\!\!\!\!\!\!\!\!$ & $4$ & $6$ & $7$ & $8$ & $10$ \\
\hline
\parbox[top][4ex][c]{2.1cm}{$\!\!\!\!\ti H^n(K(\Z,4),\Z_2)\!\!\!\!\!\!$} & 
$0$ &  $\!\!\Z_2\an{\bar e_4}\!\!\!$ & $\!\!\Z_2\an{\bar e_6'}\!\!\!$ & $\!\!\Z_2\an{\bar e_7}\!\!\!$ & $\!\!\Z_2\an{\bar e_4^2}\!\!\!$ & $\!\!\Z_2\an{\bar e_{10}', \bar e_4\bar e_6'}\!\!\!$  \end{tabular}
\end{aligned}
\label{fm15eq2}
\end{equation}

\subsubsection{Computation of the spectral sequence}
\label{fm1512}

Consider the Atiyah--Hirzebruch spectral sequence
\begin{equation*}
 \ti H_p(K(\Z,4),\Om^{\bs\Spin}_q(*))\Longra\ti\Om^{\bs\Spin}_{p+q}(K(\Z,4))
\end{equation*}
whose $E^2$-page for $p+q\le 10$ is shown in Figure~\ref{fm15fig1}.

\begin{figure}[htb]
\centering
\begin{tikzpicture}
\matrix (m) [matrix of math nodes,
nodes in empty cells,nodes={minimum width=9ex,
minimum height=5ex,outer sep=-5pt},
column sep=.4ex,row sep=1ex]{
4 & \Z\an{\al_4\ep_4} & \Z_2\an{\al_4\ep_6'}  \\
2 &  \Z_2\an{\al_1^2\bar\ep_4} & \Z_2\an{\al_1^2\bar\ep_6'} & \Z_2\an{\al_1^2\bar\ep_7} &  \Z_2\an{\al_1^2\bar\ep_4^2} \\
1 &  \Z_2\an{\al_1\bar\ep_4}  & \Z_2\an{\al_1\bar\ep_6'} & \Z_2\an{\al_1\bar\ep_7} & \Z_2\an{\al_1\bar\ep_4^2} \\
0 & \Z\an{\ep_4} & \Z_2\an{\ep_6'}  & & \Z\an{\ep_4^2}\!\op\!\Z_3\an{\ep_8} & \Z_2\an{\ep_{10}',\ep_4\ep_6'}  \\
\quad\strut & 4 & 6 & 7 & 8  & 10 \strut \\};
\draw[thick] (m-1-1.north east) node[above]{$q$} -- (m-5-1.east);
\draw[thick] (m-5-1.north) -- (m-5-6.north east) node[right]{$p$};

\draw[-stealth] (m-4-3) --  node[below]{\scriptsize $d^2_{6,0}$}  (m-3-2);
\draw[-stealth] (m-3-3) --  node[left, xshift=-0.5ex, yshift=-0.2ex]{\scriptsize $d^2_{6,1}$}  (m-2-2);

\end{tikzpicture}
\caption{$E^2$-page of $\ti H_p(K(\Z,4),\Om^{\bs\Spin}_q(*))\Ra\ti\Om^{\bs\Spin}_{p+q}(K(\Z,4))$, $p+q\le 10$}
\label{fm15fig1}
\end{figure}

Note that Stong \cite{Ston2} has already computed $\ti\Om_n^{\bs\Spin}(K(\Z,4))$ and obtained the following result.
\e
\begin{tabular}{c|cccccccccc}
$n$ & 
$\!\!0,1,2,3,5,6,7\!\!\!$ & $4$ & $8$ & $9$ \\
\hline
\parbox[top][4ex][c]{2.4cm}{$\!\!\ti\Om_n^{\bs\Spin}(K(\Z,4))$} & 
$0$ & $\Z$ & $\Z^2$ & $\Z_2$ 
\end{tabular}
\label{fm15eq3}
\e
Nevertheless we will compute the spectral sequence for $\ti\Om_n^{\bs\Spin}(K(\Z,4))$, as this will yield some facts needed later in the proof. Moreover, we wish to determine explicit generators for $\ti\Om_n^{\bs\Spin}(K(\Z,4))$.

As $\Sq^2(\bar e_4)=\bar e_6'$, $\Sq^2(\bar e_6')=\Sq^2(\bar e_4^2)=0$, Proposition \ref{fm2prop1} implies that the only non-trivial differentials in Figure \ref{fm15fig1} are $d_{6,0}^2$ and $d_{6,1}^2$. This leads to the $E^3$-page of the spectral sequence shown in Figure \ref{fm15fig2}. 

\begin{figure}[htb]
\centering
\begin{tikzpicture}
\matrix (m) [matrix of math nodes,
nodes in empty cells,nodes={minimum width=9ex,
minimum height=5ex,outer sep=-5pt},
column sep=.4ex,row sep=1ex]{
4 & \Z\an{\al_4\ep_4} & \Z_2\an{\al_4\ep_6'}  \\
2 &   & \Z_2\an{\al_1^2\bar\ep_6'} & \Z_2\an{\al_1^2\bar\ep_7} &  \Z_2\an{\al_1^2\bar\ep_4^2} \\
1 &   &  & \Z_2\an{\al_1\bar\ep_7} & \Z_2\an{\al_1\bar\ep_4^2} \\
0 & \Z\an{\ep_4} &  & & \Z\an{\ep_4^2}\!\op\!\Z_3\an{\ep_8} & \Z_2\an{\ep_{10}',\ep_4\ep_6'}  \\
\quad\strut & 4 & 6 & 7 & 8  & 10 \strut \\};
\draw[thick] (m-1-1.north east) node[above]{$q$} -- (m-5-1.east);
\draw[thick] (m-5-1.north) -- (m-5-6.north east) node[right]{$p$};
\draw[-stealth,bend right=15] (m-4-6) to  node[right,xshift=5.5ex,yshift=-1.5ex]{\scriptsize $d^3_{10,0}$}  (m-2-4);
\end{tikzpicture}
\caption{$E^3$-page of $\ti H_p(K(\Z,4),\Om^{\bs\Spin}_q(*))\Ra\ti\Om^{\bs\Spin}_{p+q}(K(\Z,4))$, $p+q\le 10$}
\label{fm15fig2}
\end{figure}

The only possible non-zero higher differential in Figure \ref{fm15fig2} is $d_{10,0}^3$. Now $\ti\Om_9^{\bs\Spin}(K(\Z,4))\cong\Z_2$ from \eq{fm15eq3} implies that $d_{10,0}^3$ is surjective. Hence Figure \ref{fm15fig2} leads to the $E^\iy$-page shown in Figure \ref{fm15fig3}. For later reference we record the following.

\begin{lem}
\label{fm15lem1}
For the Atiyah--Hirzebruch spectral sequence of\/ $K(\Z,4)$ we have
\e
\label{fm15eq4}
d^3_{10,0}(\ep_{10}')=\al_1^2\bar\ep_7,\qquad
d^3_{10,0}(\ep_4\ep_6')=\al_1^2\bar\ep_7.
\e
\end{lem}

\begin{proof}
We use the Atiyah--Hirzebruch spectral sequence of $K(\Z,3)$ discussed in \S\ref{fm171} below. The morphism from Figure \ref{fm17fig1} to Figure \ref{fm15fig2} yields a commutative diagram
\begin{equation*}
 \begin{tikzcd}
  E^3_{9,0}(K(\Z,3))=\Z_2\an{\de_9'}\dar{d^3_{9,0}}\rar & E^3_{10,0}(K(\Z,4))=\Z_2\an{\ep_{10}',\ep_4\ep_6'}\dar{d^3_{10,0}}\\
  E^3_{6,2}(K(\Z,3))=\Z_2\an{\al_1^2\bar\de_3^2}\rar & E^3_{7,2}(K(\Z,4))=\Z_2\an{\al_1^2\bar\ep_7}.
 \end{tikzcd}
\end{equation*}
The horizontal maps are induced by the suspension in ordinary homology composed with $\chi_*$. We have seen above that $\de_9'\mapsto\bar\ep_{10}'$, $\al_1^2\bar\de_3^2\mapsto\al_1^2\bar\ep_7$ and that $d^3_{9,0}:\Z_2\an{\de_9'}\ra\Z_2\an{\al_1^2\bar\de_3^2}$ in Figure \ref{fm17fig1} is an isomorphism, hence $d^3_{10,0}(\bar\ep_{10}')=\al_1^2\bar\ep_7$. This proves the first equality in \eq{fm15eq4}.

To prove the second equality, consider the space $K(\Z,2)\w K(\Z,2)$. Since $K(\Z,2)\simeq\CP^\iy$, $\ti H^*(K(\Z,2)\w K(\Z,2),\Z)$ is the subring of the polynomial ring $\Z[f,g]$ consisting of all polynomials $p(f,g)$ satisfying $p(f,0)=p(0,g)=0$. Let $\mu:K(\Z,2)\w K(\Z,2)\ra K(\Z,4)$ be a classifying map with $\mu^*(e_4)=fg$. Consider the Atiyah--Hirzebruch spectral sequence $\ti H_p(K(\Z,2)\w K(\Z,2),\Om^{\bs\Spin}_q(*))\Ra\Om^{\bs\Spin}_{p+q}(K(\Z,2)\w K(\Z,2))$ and recall Proposition \ref{fm2prop1} to compute $d^2_{p,0}, d^2_{p,1}$. We have $\Sq^2(\bar f^m \bar g^n)=m \bar f^{m+1}\bar g^n+n \bar f^m \bar g^{n+1}$ in $\Z_2$-cohomology. Write $\varphi^i\ga^j$ for the homology classes dual to the cohomology classes $f^ig^j$. The spectral sequence has $E^2_{10,0}=\Z\an{\varphi^4\ga^1,\varphi^3\ga^2,\varphi^2\ga^3,\varphi^1\ga^4}$ and $E^2_{7,2}=0$. Hence $E^3_{7,2}=0$ and from the expression for the Steenrod square we get $E^3_{10,0}=\Z\an{\varphi^4\ga^1+\varphi^3\ga^2, \varphi^2\ga^3+\varphi^1\ga^4}$. The map $\mu$ induces a morphism of spectral sequences, so in particular a morphism from the $E^3$-page of the spectral sequence for $K(\Z,2)\w K(\Z,2)$ to Figure \ref{fm15fig2}. This yields a commutative diagram
\begin{equation*}
 \begin{tikzcd}
  \begin{array}{c}
  E^3_{10,0}(K(\Z,2)\w K(\Z,2))\\
  =\Z\an{\varphi^4\ga^1+\varphi^3\ga^2, \varphi^2\ga^3+\varphi^1\ga^4}\dar{d^3_{10,0}}
  \end{array}
  \rar & E^3_{10,0}(K(\Z,4))=\Z_2\an{\ep_{10}',\ep_4\ep_6'}\dar{d^3_{10,0}}\\
  E^3_{7,2}(K(\Z,2)\w K(\Z,2))=0\rar & E^3_{7,2}(K(\Z,4))=\Z_2\an{\al_1^2\bar\ep_7}.
 \end{tikzcd}
\end{equation*}
The horizontal maps are induced by $\mu_*:\ti H_n(K(\Z,2)\w K(\Z,2),\Om_q^{\bs\Spin}(*))\ra \ti H_n(K(\Z,4),\Om_q^{\bs\Spin}(*))$. Now $\mu_*(\varphi^4\ga^1)=\ep_{10}'$, $\mu_*(\varphi^3\ga^2)=\ep_4\ep_6'$, $\mu_*(\varphi^2\ga^3)=\ep_4\ep_6'$, $\mu_*(\varphi^1\ga^4)=\ep_{10}'$, so the commutative diagram implies $d^3_{10,0}(\ep_{10}'+\ep_4\ep_6')=0$. The second equation in \eq{fm15eq4} then follows from the first equation.
\end{proof}

\begin{figure}[htb]
\centering
\begin{tikzpicture}
\matrix (m) [matrix of math nodes,
nodes in empty cells,nodes={minimum width=9ex,
minimum height=5ex,outer sep=-5pt},
column sep=.4ex,row sep=1ex]{
4 & \Z\an{\al_4\ep_4}  \\
2 &   & \Z_2\an{\al_1^2\bar\ep_6'} &  \\
1 &   &  & \Z_2\an{\al_1\bar\ep_7} & \Z_2\an{\al_1\bar\ep_4^2} \\
0 & \Z\an{\ep_4} &  & & \Z\an{\ep_4^2}\!\op\!\Z_3\an{\ep_8} &    \\
\quad\strut & 4 & 6 & 7 & 8  \strut \\};
\draw[thick] (m-1-1.north east) node[above]{$q$} -- (m-5-1.east);
\draw[thick] (m-5-1.north) -- (m-5-6.north east) node[right]{$p$};
\end{tikzpicture}
\caption{$E^\iy$-page of $\ti H_p(K(\Z,4),\Om^{\bs\Spin}_q(*))\Ra\ti\Om^{\bs\Spin}_{p+q}(K(\Z,4))$, $p+q\le 9$}
\label{fm15fig3}
\end{figure}

\subsubsection{Determining the filtration and generators}
\label{fm1513}

From Figure \ref{fm15fig3} we see that the filtration of $\ti\Om_n^{\bs\Spin}(K(\Z,4))$ is trivial for $n\le 7$. For $n=8$ we claim that the filtration is
\begin{equation}
\begin{tikzcd}[column sep=small]
0\arrow[r,symbol=\subset,"\Z\an{\al_4\ep_4}"' yshift=-1.5ex] & F_{4,8}\!\cong\!\Z\arrow[r,symbol=\subset,"\Z_2\an{\al_1^2\bar\ep_6'}"' yshift=-1.5ex] & F_{6,8}\!\cong\!\Z\arrow[r,symbol=\subset,"\Z_2\an{\al_1\bar\ep_7}"' yshift=-1.5ex] & F_{7,8}\!\cong\!\Z\arrow[r,symbol=\subset,"\Z\an{\ep_4^2}\op\Z_3\an{\ep_8}"' yshift=-1.5ex] & F_{8,8}\!=\!\ti\Om_8^{\bs\Spin}(K(\Z,4))\!\cong\! \Z^2
\end{tikzcd}\!\!\!
\label{fm15eq5}
\end{equation}
Here the graded pieces are as in Figure \ref{fm15fig3}. By \eq{fm15eq3} we have $\ti\Om^{\bs\Spin}_8(K(\Z,4))\ab\cong\Z^2$. Therefore all the extensions in \eq{fm15eq5} must be nontrivial, as claimed (for instance, we cannot have $F_{6,8}\cong\Z\op\Z_2$, as then $\ti\Om^{\bs\Spin}_n(K(\Z,4))$ would have a $\Z_2$ subgroup, contradicting $\ti\Om^{\bs\Spin}_8(K(\Z,4))\cong\Z^2$).

We verify that the elements given in Table \ref{fm3tab1} are indeed generators and prove \eq{fm3eq18} and \eq{fm3eq19}.

Let $n=4$. From Proposition \ref{fm2prop2} we have an isomorphism $\ti\Om^{\bs\Spin}_4(K(\Z,4))\ra H_4(K(\Z,4),\Z)\cong\Z$, $[X,\al]\mapsto\int_X\al$, which proves \eq{fm3eq18}. The generator $\de$ is represented by $[\cS^3\t\cS^1_{\rm b},\al]$, where $\al\in H^3(\cS^3\t\cS^1_{\rm b})$ is Poincar\'e dual to a point. As $\an{\al,[\cS^3\t\cS^1_{\rm b}]}=1$, $\de$ maps to $1$ under the isomorphism \eq{fm2eq11} and hence $\de$ generates $\ti\Om^{\bs\Spin}_4(K(\Z,4))$.

Let $n=8$. We can use \eq{fm15eq5} to successively choose generators $\ze'_2,\ze'_3$ of $F_{4,8}=\Z\an{\ze'_3}$ and $\ti\Om^{\bs\Spin}_8(K(\Z,4))=F_{8,8}=\Z\an{\ze'_2,\ze'_3}$ with $12\ze'_3=\al_4\ep_4$ and $\ze'_2\mod\an{\ze'_3}=\ep_4^2\mod\an{\ze'_3}$. Recall that $\al_4\ep_4$ is represented by $K3\t\cS^4$ and the cohomology class $\Pd[K3\t\{*\}]$. Observe that \eq{fm3eq19} is a well-defined morphism. Table \ref{fm15tab1} shows the values of the integrals $\int_X\al\cup\al$ and $\int_X\al\cup p_1(TX)$ for $[K3\t\cS^4,\Pd[K3\t\{*\}]]$ and for $\ze_2,\ze_3$. From this we see that equation \eq{fm3eq19} maps $[K3\t\cS^4,\Pd[K3\t\{*\}]],\ze_2,\ze_3$ to $(0,12),\ab(1,0),\ab(0,1)$ respectively. 

\begin{table}[htb]
\centerline{\begin{tabular}{|l|l|l|l|}
\hline
 & $[K3\t\cS^4,\Pd[K3\t\{*\}]]$ & $\ze_2$ &  $\ze_3$ \\
\hline
$\ts\int_{X_{\vphantom{(}}}^{\vphantom{(}}\al\cup\al$  &  0 & 1 & 0 \\
\hline
$\ts\int_{X_{\vphantom{(}}}^{\vphantom{(}}\al\cup p_1(TX)$  &  48 & $-2$ & 4 \\
\hline
\end{tabular}}
\caption{Invariants of $[K3\t\cS^4,\Pd[K3\t\{*\}]],\ze_2,\ze_3$}
\label{fm15tab1}
\end{table}

As $\ti\Om^{\bs\Spin}_8(K(\Z,4))\cong\Z^2$ this forces $[K3\t\cS^4,\Pd[K3\t\{*\}]]=12\ze_3$, so $\ze_3'=\ze_3$. Also, as $\int_X\al\cup\al=1$ for $\ze_2$ we see that $\ze'_2\mod\an{\ze_3}=\ep_4^2\mod\an{\ze'_3}$, so  $\ze_2$ is a possible choice for $\ze_2'$. This proves $\ti\Om^{\bs\Spin}_8(K(\Z,4))=\Z\an{\ze_2,\ze_3}$, and as \eq{fm3eq19} maps $\ze_2,\ze_3$ to $(1,0),(0,1)$ it also follows that \eq{fm3eq19} is an isomorphism.

Let $n=9$. By Figure \ref{fm15fig3} we have a commutative diagram of extensions
\begin{equation}
\begin{tikzcd}[column sep=small]
	0\rar & \begin{array}{l}F_{7,8}\\=\Z\an{\ze_3}\end{array}\rar\dar{\al_1} & \begin{array}{l}F_{8,8}= \\ \Z\an{\ze_2,\ze_3}\end{array}\rar\dar{\al_1} & \begin{array}{l} E^\iy_{8,0}= \\ \Z\an{\ep_4^2}\op\Z_3\an{\ep_8}\dar{\al_1}\end{array}\rar & 0\\
	0\rar & 0\rar & F_{8,9}\!=\!\ti\Om^{\bs\Spin}_9(K(\Z,4))\rar & E^\iy_{8,1}\!=\!\Z_2\an{\al_1\ep_4^2}\rar & 0.
\end{tikzcd}\!\!\!\!\!\!
\label{fm15eq6}
\end{equation}
This proves $\al_1\ze_3=0$, claimed in \eq{fm3eq8}, and that $\al_1\ze_2$ generates $\ti\Om^{\bs\Spin}_9(K(\Z,4))$. This completes the proof of Table \ref{fm3tab1} for $K(\Z,4)$.

\subsection{\texorpdfstring{Computation of $\ti\Om_n^{\bs{\mathrm{Spin}}}(M\SO(4))$}{Computation of Ωₙˢᵖⁱⁿ(MSO(4))}}
\label{fm152}

\subsubsection{Description of the (co)homology}
\label{fm1521}

Brown \cite[Th.~1.5]{Brow1} shows that $H^*(B\SO(4),\Z)=\Z[p_1,e,W_3]/\an{2W_3=0}$, which leads the following tables of (co)homology groups.
\begin{equation*}
\renewcommand{\arraystretch}{1.5}
\begin{tabular}{c|cccccc}
$p$ & 
$\!\!0,1\!\!\!\!\!\!\!\!\!\!\!$ & $2$ & $3$ & $4$ & $5$ & $6$\\
\hline
\parbox[top][4ex][c]{2.1cm}{$\!\!\!\ti H_p(B\SO(4),\Z)\!\!$} & 
$0$ &  $\!\!\!\Z_2\an{\om_2}\!$ & 0 & $\!\!\!\Z\an{\pi_1,\varep}$ & $\!\!\!\Z_2\an{\om_5}$ & $\!\!\Z_2\an{\om_6,\om_6'}$
\\
\parbox[top][4ex][c]{2.1cm}{$\!\!\!\ti H^p(B\SO(4),\Z)\!\!$} & 
$0$ &  $0$ & $\!\!\!\Z_2\an{W_3}\!$ & $\!\!\Z\an{p_1,e}$ & $0$ & $\!\!\Z_2\an{W_3^2}$
\\
\parbox[top][4ex][c]{2.2cm}{$\!\!\!\ti H_p(B\SO(4),\Z_2)\!\!$} & 
$0$ &  $\!\!\!\Z_2\an{\bar\om_2}\!$ & $\!\!\!\Z_2\an{\bar\om_3}\!$ & $\!\!\!\Z_2\an{\bar\om_2^2,\bar\om_4}\!\!$ & $\!\!\!\Z_2\an{\bar\om_2\bar\om_3}\!\!$ & $\!\!\!\Z_2\an{\bar\om_3^2,\bar\om_2^3,\bar\om_2\bar\om_4}\!\!\!\!$
\\
\parbox[top][4ex][c]{2.2cm}{$\!\!\!\ti H^p(B\SO(4),\Z_2)\!\!\!$} & 
$0$ &  $\!\!\!\Z_2\an{\bar w_2}\!$ & $\!\!\!\Z_2\an{\bar w_3}\!$ & $\!\!\Z_2\an{\bar w_2^2,\bar w_4}\!\!$ & $\!\!\!\Z_2\an{\bar w_2\bar w_3}\!\!$ & $\!\!\!\Z_2\an{\bar w_3^2,\bar w_2^3,\bar w_2\bar w_4}\!\!\!\!$
\end{tabular}
\end{equation*}
Here we use the following notation. The Stiefel--Whitney classes are written $\bar w_k,$ so the polynomials $P(\bar w_2, \bar w_2,\ldots)$ form a basis of $H^*(B\SO(4),\Z_2)$ and we write the dual basis of $H_*(B\SO(4),\Z_2)$ as $P(\bar\om_2,\bar\om_3,\ldots).$ The Pontrjagin and Euler classes form a basis $\{p_1, e\}$ of $\ti H^4(B\SO(4),\Z)$ and we write $\{\pi_1,\varep\}$ for the dual basis of $\ti H_4(B\SO(4),\Z).$ The remaining classes are defined using the Bockstein homomorphism by
\begin{equation*}
W_3=\be(\bar w_2), \;\>
\om_2=\be(\bar\om_3),\;\>
\om_5=\be(\bar\om_3^2),\;\>
\om_6=\be(\bar\om_2^2\bar\om_3),\;\>
\om_6'=\be(\bar\om_3\bar\om_4).
\end{equation*}
The mod $2$ reductions of the integral (co)homology classes are
\begin{align}
p_1&\mapsto\bar w_2^2, &
e&\mapsto\bar w_4, &
W_3&\mapsto\bar w_3, &
\om_2&\mapsto\bar\om_2, &
\pi_1&\mapsto\bar\om_2^2,\nonumber\\
\varep &\mapsto\bar\om_4,&
\om_5& \mapsto\bar\om_2\bar\om_3, &
\om_6&\mapsto\bar\om_2^3, &
\om_6'&\mapsto\bar\om_2\bar\om_4.
\label{fm15eq7}
\end{align}

The (co)homologies of the Thom space $M\SO(4)$ are the same as for $B\SO(4)$ except for a degree shift under the (co)homological Thom isomorphisms. Recall from Milnor--Stasheff \cite[p.~130]{MiSt} that
\begin{equation}
\label{fm15eq8}
\Sq^k(\bar t)=\begin{cases}
0, & \text{if $k=1,$}\\
\bar w_k^T, & \text{if $k\ge 2.$}\end{cases}
\end{equation}

\subsubsection{Computation of the spectral sequence}
\label{fm1522}

Consider the Atiyah--Hirzebruch spectral sequence
\begin{equation*}
 \ti H_p(M\SO(4),\Om^{\bs\Spin}_q(*))\Longra\ti\Om^{\bs\Spin}_{p+q}(M\SO(4)),
\end{equation*}
whose $E^2$-page for $p+q\le 10$ is shown in Figure~\ref{fm15fig4}.

\begin{figure}[htb]
\!\!\!\!\!\begin{tikzpicture}
\matrix (m) [matrix of math nodes,
nodes in empty cells,nodes={minimum width=9ex,
minimum height=5ex,outer sep=-5pt},
column sep=.0ex,row sep=1ex]{
4 & \Z\an{\al_4 \tau}   \\
2 & \!\Z_2\an{\al_1^2 \bar\tau} & \!\Z_2\an{\al_1^2 \bar\om_2^T} & \!\Z_2\an{\al_1^2 \bar\om_3^T} & \!{\substack{\ts \Z_2\langle\al_1^2 (\bar\om_2^2)^T,\\ \ts \quad\al_1^2 \bar\om_4^T\rangle}} \\
1 & \!\Z_2\an{\al_1 \bar\tau} & \!\Z_2\an{\al_1 \bar\om_2^T} & \!\Z_2\an{\al_1 \bar\om_3^T} & {\substack{\ts \!\Z_2\langle\al_1 (\bar\om_2^2)^T,\\ \ts\quad \al_1 \bar\om_4^T\rangle}} & \!\!\Z_2\an{\al_1 (\bar\om_2\bar\om_3)^T} \\
0 & \!\Z\an{\tau} & \Z_2\an{ \om_2^T} &  & \Z\an{ \pi_1^T, \varep^T} & \!\!\Z_2\an{ \om_5^T}\!\! & \!\!\!\!\!\Z_2\an{ \om_6^T, \om_6^{\prime T}}\!\! \\
\quad\strut & 4 & 6 & 7 & 8  & 9 & 10 \strut \\};
\draw[thick] (m-1-1.north east) node[above]{$q$} -- (m-5-1.east);
\draw[thick] (m-5-1.north) -- (m-5-7.north east) node[right]{$p$};

\draw[-stealth] (m-4-3) --  node[left, xshift=-0.5ex, yshift=-0.2ex]{\scriptsize $d^2_{6,0}$}  (m-3-2);
\draw[-stealth] (m-3-3) --  node[left, xshift=-0.5ex, yshift=-0.2ex]{\scriptsize $d^2_{6,1}$}  (m-2-2);
\draw[-stealth] (m-4-7) --  node[right,xshift=3ex]{\scriptsize $d^2_{10,0}$}  (m-3-5);

\end{tikzpicture}\!\!\!\!\!
\caption{$E^2$-page of $\ti H_p(M\SO(4),\Om^{\bs\Spin}_q(*))\Rightarrow\ti\Om^{\bs\Spin}_{p+q}(M\SO(4))$, $p+q\le 10$}
\label{fm15fig4}
\end{figure}

From Proposition \ref{fm2prop1} and \eq{fm15eq8} we find that the only non-trivial differentials $d_{p,q}^2$ with $p+q\le 10$ are $d_{6,0}^2$, $d_{6,1}^2$, $d_{10,0}^2$, and that for these we have
\begin{equation*}
d^2_{6,0}( \om_2^T)=\al_1 \bar\tau, \;\> d^2_{6,1}(\al_1 \bar\om_2^T)=\al_1^2 \bar\tau, \;\>
d^2_{10,0}(\om_6^T)=\al_1 (\bar\om_2^2)^T,\;\>
d^2_{10,0}(\om^{\prime T}_6)=0.
\end{equation*}
This leads to the $E^3$-page of the spectral sequence shown in Figure \ref{fm15fig5}. Observe that the only possible non-zero higher differentials are $d^3_{9,0}$ and $d^3_{10,0}$.

\begin{figure}[htb]
\centering
\begin{tikzpicture}
\matrix (m) [matrix of math nodes,
nodes in empty cells,nodes={minimum width=10ex,
minimum height=5ex,outer sep=-5pt},
column sep=.0ex,row sep=1ex]{
4 & \Z\an{\al_4 \tau}   \\
2 & & \!\Z_2\an{\al_1^2 \bar\om_2^T} & \!\Z_2\an{\al_1^2 \bar\om_3^T}\\
1 & &  & \!\Z_2\an{\al_1 \bar\om_3^T} & \!\Z_2\an{\al_1 \bar\om_4^T}\\
0 & \!\Z\an{\tau} &  &  & \Z\an{\pi_1^T,\varep^T} & \!\!\Z_2\an{\om_5^T}\!\! & \Z_2\an{\om_6^{\prime T}}\\
\quad\strut & 4 & 6 & 7 & 8  & 9 & 10 \strut \\};
\draw[thick] (m-1-1.north east) node[above]{$q$} -- (m-5-1.east);
\draw[thick] (m-5-1.north) -- (m-5-7.north east) node[right]{$p$};
\draw[-stealth,bend right=10] (m-4-7.north) to node[above]{\scriptsize $d^3_{10,0}$} (m-2-4.east);
\draw[-stealth] (m-4-6.north) to node[above]{\scriptsize $d^3_{9,0}$} (m-2-3.east);\end{tikzpicture}
\caption{$E^3$-page of $\ti H_p(M\SO(4),\Om^{\bs\Spin}_q(*))\Rightarrow\ti\Om^{\bs\Spin}_{p+q}(M\SO(4))$ for $p+q\le 10$. This is also the $E^\iy$-page for $p+q\le 9$.}
\label{fm15fig5}
\end{figure}

We will show that $d^3_{9,0}=d^3_{10,0}=0$, so the spectral sequence converges at the $E^3$-page, and Figure \ref{fm15fig5} is also the $E^\iy$-page of the spectral sequence in the region $p+q\le 9$. Consider the classifying map $\phi:M\SO(4)\ra K(\Z,4)$ of the Thom class in $H^4(M\SO(4),\Z)$. As $\phi$ induces a morphism of spectral sequences, it induces a morphism from Figure \ref{fm15fig5} to Figure \ref{fm15fig2}. In particular, this gives commutative diagrams
\begin{equation*}
 \begin{tikzcd}
  E^3_{9,0}(M\SO(4))=\Z_2\an{\om_5^T}\rar{\phi_*}\dar{d^3_{9,0}(M\SO(4))} & E^3_{9,0}(K(\Z,4))=0\dar{d^3_{9,0}(K(\Z,4))}\\
  E^3_{6,2}(M\SO(4))=\Z_2\an{\al_1^2\bar\om_2^T}\arrow[r,"\phi_*","\cong"'] & E^3_{6,2}(K(\Z,4))=\Z_2\an{\al_1^2\bar\ep_6'}
 \end{tikzcd}
\end{equation*}
and
\begin{equation*}
 \begin{tikzcd}
  E^3_{10,0}(M\SO(4))=\Z_2\an{\om_6'^T}\rar{\phi_*}\dar{d^3_{10,0}(M\SO(4))} & E^3_{10,0}(K(\Z,4))=\Z_2\an{\ep_{10}',\ep_4\ep_6'}\dar{d^3_{10,0}(K(\Z,4))}\\
  E^3_{7,2}(M\SO(4))=\Z_2\an{\al_1^2\bar\om_3^T}\arrow[r,"\phi_*","\cong"'] & E^3_{7,2}(K(\Z,4))=\Z_2\an{\al_1^2\bar\ep_7}.
 \end{tikzcd}
\end{equation*}
Since $\phi_*(\bar\om_2^T)=\bar\ep_6'$, the bottom horizontal map in the first diagram is an isomorphism, which implies $d^3_{9,0}(M\SO(4))=0$. Since $\phi_*(\bar\om_3^T)=\bar\ep_7$, the bottom horizontal map in the second diagram is an isomorphism. Moreover, $\phi_*(\om_6'^T)=\ep_{10}'+\ep_4\ep_6'$ and $d_{10,0}^3(K(\Z,4))(\varep_{10}')=d_{10,0}^3(K(\Z,4))(\varep_4\varep_6')=\al_1^2\bar\varep_7$ by \eq{fm15eq4}, hence $d^3_{10,0}(M\SO(4))=0$ by the commutative diagram.

\subsubsection{Determining the filtration and generators}
\label{fm1523}

From Figure \ref{fm15fig5} we see that $\ti\Om_n^{\bs\Spin}(M\SO(4))=0$ for $n=0,1,2,3,\ab 5,\ab 6,\ab 7$.

Let $n=4$. By Proposition \ref{fm2prop2} there is an isomorphism $\ti\Om^{\bs\Spin}_4(M\SO(4))\ra H_4(M\SO(4),\Z)\cong\Z$, $[X,M]\mapsto\int_M 1$ and $\de\mapsto1$, hence $\ti\Om^{\bs\Spin}_4(M\SO(4))=\Z\an{\de}$.

Let $n=8$. Observe that \eq{fm3eq17} is a well-defined morphism $\ti\Om_8^{\bs\Spin}(M\SO(4))\ab\ra\frac{1}{4}\Z\op\Z^2$ (we will shortly show that it actually maps to $\Z^3\subset\frac{1}{4}\Z\op\Z^2$). For $\frac{\ze_1}{4},\ze_2,\ze_3$, calculation shows that the invariants are as in Table \ref{fm15tab2}, so $\frac{\ze_1}{4},\ze_2,\ze_3$ are mapped to $(1,0,0),(0,1,0),(0,0,1)$ under \eq{fm3eq17}. According to Figure \ref{fm15fig5} the group $\ti\Om^{\bs\Spin}_8(M\SO(4))$ has a filtration
\begin{equation}
\begin{tikzcd}
0\arrow[r,symbol=\subset,"\Z\an{\al_4\tau}"' yshift=-1.5ex] & F_{4,8}\arrow[r,symbol=\subset,"\Z_2\an{\al_1^2\bar\om_2^T}"' yshift=-1.5ex] & F_{6,8}\arrow[r,symbol=\subset,"\Z_2\an{\al_1\bar\om_3^T}"' yshift=-1.5ex] & F_{7,8}\arrow[r,symbol=\subset,"\Z\an{\pi_1^T,\ep^T}"' yshift=-1.5ex] & F_{8,8}=\ti\Om_8^{\bs\Spin}(M\SO(4)).
\end{tikzcd}
\label{fm15eq9}
\end{equation}
By \eq{fm2eq11} the projection $F_{8,8}\ra F_{8,8}/F_{7,8}=E^\iy_{8,0}=\Z\an{\pi_1^T,\varep^T}\cong\Z^2$ is given by $[X,M]\mapsto\bigl(\ts\int_M p_1(\nu_M),\ts\int_M e(\nu_M)\bigr)$, so from Table \ref{fm15tab2} we see that $\ze_2,\ze_3$ are lifts to $F_{8,8}$ of a basis of $F_{8,8}/F_{7,8}$. Moreover, $F_{4,8}\cong\Z\an{\al_4\tau}\subset\ti\Om^{\bs\Spin}_8(M\SO(4))$ is generated by
\e
 \ze_1=[K3\t\cS^3\t\cS^1_{\rm b},K3\t\{N\}\t\{1\}],
\label{fm15eq10}
\e
since $\al_4=[K3]$ and $[\cS^3\t\cS^1_{\rm b},\{N\}\t\{1\}]$ maps to $\tau\in H_4(M\SO(4),\Z)$. Note that \eq{fm15eq10} is mapped to $(4,0,0)$ by \eq{fm3eq17}, so its invariants are those of $4\frac{\ze_1}{4}$ in Table \ref{fm15tab2}.

\begin{table}[htb]
\centerline{\begin{tabular}{|l|l|l|l|}
\hline
 & $\frac{\ze_1}{4}$ & $\ze_2$ &  $\ze_3$ \\
\hline
$\int_{{M_j}_{\vphantom{(}}}^{\vphantom{(}}p_1(TM_j)$  &  $-12$ & 0 & 3 \\
\hline
$\int_{{M_j}_{\vphantom{(}}}^{\vphantom{(}}e(\nu)$  &  0 & 1 & 0 \\
\hline
$\int_{{M_j}_{\vphantom{(}}}^{\vphantom{(}}p_1(\nu)$ &  0 & $-2$ & 1 \\
\hline
\end{tabular}}
\caption{Invariants of $\frac{\ze_1}{4},\ze_2,\ze_3$}
\label{fm15tab2}
\end{table}

From \eq{fm15eq9} we see there are three possibilities:
\begin{itemize}
\setlength{\itemsep}{0pt}
\setlength{\parsep}{0pt}
\item[(i)] $F_{6,8}\cong F_{7,8}\cong\Z$, the maps $F_{4,8}\hookrightarrow F_{6,8}$ and $F_{6,8}\hookrightarrow F_{7,8}$ are $2\cdot {}:\Z\ra\Z$, and $\ze_1$ in \eq{fm15eq10} is divisible by 4 in $\ti\Om^{\bs\Spin}_8(M\SO(4))$.
\item[(ii)] $F_{7,8}\cong\Z\op\Z_2$, and $\ze_1$ in \eq{fm15eq10} is divisible by 2 but not by 4.
\item[(iii)] $F_{7,8}\cong\Z\op\Z_2^2$ or $\Z\op\Z_4$, and $\ze_1$ is indivisible in $\ti\Om^{\bs\Spin}_8(M\SO(4))$.
\end{itemize}

From this and the fact that $\ze_1,\ze_2,\ze_3$ map to $(4,0,0),(0,1,0),(0,0,1)$ under \eq{fm3eq17}, we see that the image of \eq{fm3eq17} in $\frac{1}{4}\Z\op\Z^2$ must be $\Z^3$ in case (i), $2\Z\op\Z^2$ in case (ii), and $4\Z\op\Z^2$ in case (iii). But the latter two cases contradict that \eq{fm3eq17} maps $\frac{\ze_1}{4}\mapsto(1,0,0)$. Thus case (i) must hold. This forces \eq{fm3eq17} to be an isomorphism and also $\ti\Om^{\bs\Spin}_8(M\SO(4))=\Z\an{\frac{\ze_1}{4},\ze_2,\ze_3}$. Let $\phi:M\SO(4)\ra K(\Z,4)$ be the classifying map of the Thom class in $H^4(M\SO(4),\Z)$. Then $\phi_*(\ze_1)=[K3\t\cS^3\t\cS^1_{\rm b},\al]$ where $\al$ is Poincar\'e dual to $K3\t\{N\}\t\{1\}$. Hence $\al^2=0$ and $\int_{K3\t\{N\}\t\{1\}}p_1(\nu)=0$, so \eq{fm3eq19} proves $\phi_*(\ze_1)=0$, which implies that~$\phi_*\bigl(\frac{\ze_1}{4}\bigr)=0$.

Let $n=9$. From Figure \ref{fm15fig5} we know that $\ti\Om_9^{\bs\Spin}(M\SO(4))$ has a 3-step filtration with quotients $\Z_2,\Z_2,\Z_2$. There is a commutative diagram of extensions
\begin{equation*}
\begin{tikzcd}
\begin{array}{l}
 F_{6,8}\\=\Z\an{2\frac{\ze_1}{4}}\end{array}\dar\rar[symbol=\subset,"E_{7,1}^\iy=\Z_2\an{\al_1\bar\om_3^T}" {xshift=-2ex, yshift=2ex}] & \begin{array}{l}F_{7,8}\\=\Z\an{\frac{\ze_1}{4}}\end{array}\dar{\al_1}\rar[symbol=\subset,"E_{8,0}^\iy=\Z\an{\pi_1^T,\varep^T}" {xshift=-1ex,yshift=2ex}] & \begin{array}{l}F_{8,8}\\=\Z\an{\frac{\ze_1}{4},\ze_2,\ze_3}\end{array}\!\!\!\!=\ti\Om_8^{\bs\Spin}(M\SO(4))\dar[xshift=-14ex]{\al_1}\\
 F_{6,9}=0\rar[symbol=\subset,"E_{7,2}^\iy \\=\Z_2\an{\al_1^2\bar\om_2^T}"' yshift=-1ex]  & F_{7,9}\cong\Z_2\rar[symbol=\subset,"E_{8,1}^\iy=\Z_2\an{\al_1\bar\om_4^T}"' yshift=-1ex] & F_{8,9}
\underset{E^\iy_{9,0}=\Z_2\an{\om_5^T}}{\subset} \ti\Om_9^{\bs\Spin}(M\SO(4)).
\end{tikzcd}
\end{equation*}
Observe that $\al_1$ maps $E_{7,1}^\iy\ra E_{7,2}^\iy$ isomorphically, so $\al_1\frac{\ze_1}{4}$ is a generator of $F_{7,9}$. Multiplication by $\al_1$ also induces a morphism of extensions
\begin{equation*}
\begin{tikzcd}[ampersand replacement=\&]
F_{7,8}=\Z\ban{\frac{\ze_1}{4}}\rar\dar{\al_1} \& F_{8,8}=\Z\ban{\frac{\ze_1}{4},\ze_2,\ze_3}\dar{\al_1}\arrow[r,"{\bigl(\begin{smallmatrix}0&-2&1\\0&1&0\end{smallmatrix}\bigr)}" yshift=1ex] \& E^\iy_{8,0}=\Z\an{\pi_1^T,\varep^T}\dar{(0,1)}\\
F_{7,9}=\Z_2\ban{\al_1\frac{\ze_1}{4}}\rar \& F_{8,9}\rar \& E^\iy_{8,1}=\Z_2\an{\al_1\bar\om_4^T},
\end{tikzcd}
\end{equation*}
where the matrix is a consequence of the second and third row of Table \ref{fm15tab2} and where the right vertical map is $(0,1)$ because $\al_1\pi_1^T=\al_1(\bar\om_2^2)^T$ and $\al_1\varep^T=\al_1\bar\om_4^T$ by \eq{fm15eq7}. By the commutative diagram $\al_1\ze_2$ maps to a non-zero element in $E^\iy_{8,1}$, hence $\al_1\ze_2$ is a generator of $F_{8,9}/F_{7,9}$. Moreover, we have $F_{8,9}\cong\Z_2^2$ rather than $\Z_4$, as this would imply $2\al_1\frac{\ze_1}{4}=\al_1\ze_2$, contradicting that $\al_1$ has order $2$, while $\al_1\ze_2\neq 0$ in $\ti\Om_9^{\bs\Spin}(K(\Z,4))$.

The final step of the filtration of $\ti\Om_9^{\bs\Spin}(M\SO(4))$ is obtained by adjoining an element which maps to the generator of $H_9(M\SO(4),\Z)\cong\Z_2\an{\om_5^T}$. As $\om_5^T\mod{2}=(\bar\om_2\bar\om_3)^T$, any $[X,M]\in \ti\Om_9^{\bs\Spin}(M\SO(4))$ with $\int_M w_2(\nu_M)w_3(\nu_M)\neq 0$ generates $F_{9,9}/F_{8,9}$. Observe that $\eta$ has this property. We have $F_{9,9}\cong\Z_2^3$, in other words, we claim the extension problem $0\ra F_{8,9}\ra F_{9,9}\ra\Z_2\ra 0$ is trivial: the other possibilities are $\Z_2\op \Z_4$ and $\Z_4\op \Z_2$ in which case $\eta=2\al_1\frac{\ze_1}{4}$ or $\eta=2\al_1\ze_2$ and by applying $\phi_*$ we would get $\phi_*(\eta)=2\cdot\phi_*(\al_1\frac{\ze_1}{4})=0$ or $\phi_*(\eta)=2\phi_*(\al_1\ze_2)=0$, which contradicts $\phi_*(\eta)\neq 0$, which we show next. This completes the proof of Table \ref{fm3tab1} for $M\SO(4)$.

For this, observe that $\eta$ has a natural lift to $\hat\eta\in\ti\Om_8^{\bs\Spin}(\Om M\SO(4))$; elements of $\ti\Om_n^{\bs\Spin}(\Om M\SO(4))$ are bordism classes $[X,M]$ where $X$ is a compact spin $n$-manifold and $M\subset X\t\cS^1$ is a compact oriented $(n-4)$-submanifold with $M\cap(X\t\{-1\})=\es$. There is a commutative diagram
\begin{equation}
\begin{tikzcd}
	\ti\Om_8^{\bs\Spin}(\Om M\SO(4))\rar\dar{(\Om\phi)_*} & \ti\Om_9^{\bs\Spin}(M\SO(4))\dar{\phi_*}\\
	\begin{array}{l}\ti\Om_8^{\bs\Spin}(\Om K(\Z,4))\\ \cong\ti\Om_8^{\bs\Spin}(K(\Z,3))\end{array}\rar & \ti\Om_9^{\bs\Spin}(K(\Z,4)).
\end{tikzcd}
\label{fm15eq11}
\end{equation}

Observe that there is well-defined morphism $\ti\Om_8^{\bs\Spin}(K(\Z,3))\ra\Z_2$, $[X,\al]\mapsto\int_X \al\cup\Sq^2(\al)$. The image $(\Om\phi)_*(\hat\eta)$ in $\ti\Om_8^{\bs\Spin}(K(\Z,3))$ is given by $X=(\SU(3)\t\cS^3)/\SO(3)$ and the cohomology class $\al$ dual to the fundamental class of the Wu manifold $M$, embedded as $(\SU(3)\t\{N\})/\SO(3)$. Since $\int_{(\SU(3)\t\cS^3)/\SO(3)}\al\cup\Sq^2(\al)=\int_M w_2(\nu_M)\cup w_3(\nu_M)\neq 0$, it follows that $(\Om\phi)_*(\hat\eta)\neq0$. Moreover, the bottom horizontal map in \eq{fm15eq11} is an isomorphism by Lemma \ref{fm15lem2} in the following section, so $\phi_*(\eta)\neq 0$ and hence
\e
 \phi_*(\eta)=\al_1\ze_2
\label{fm15eq12}
\e
in $\ti\Om_9^{\bs\Spin}(K(\Z,4))$.

\subsection{Proofs of two technical lemmas}
\label{fm153}

For a topological group $G$, recall the map $\chi:\Si G\ra BG$ from \eq{fm3eq54}. Write $\ti H^*(X)\ra \ti H^{*+1}(\Si X)$, $a\mapsto a^\si$ for the suspension isomorphism.

\begin{lem}
\label{fm15lem2}
For\/ $n=8$ and\/ $G=K(\Z,3)$ the induced map\/ $\ti\Om_8^{\bs\Spin}(K(\Z,3))\cong\ti\Om_9^{\bs\Spin}(\Si K(\Z,3))\xrightarrow{\chi_*}\ti\Om_9^{\bs\Spin}(K(\Z,4))$ is an isomorphism.
\end{lem}

\begin{proof}
Let $C$ be the mapping cone of $\chi:\Si K(\Z,3)\ra K(\Z,4)$. To prove the lemma, we will use the induced long exact sequence
\begin{equation}
\begin{tikzcd}[column sep=3.4ex]
 \cdots\rar & \ti\Om^{\bs\Spin}_n(\Si K(\Z,3))\rar{\chi_*} & \ti\Om^{\bs\Spin}_n(K(\Z,4))\rar{\jmath_*} & \ti\Om^{\bs\Spin}_n(C)\rar{\partial} & \cdots
\end{tikzcd}\!\!
\label{fm15eq13}
\end{equation}
and determine $\ti\Om^{\bs\Spin}_n(C)$ using the Atiyah--Hirzebruch spectral sequence. We compute the $\Z_2$-(co)homology groups of $C$ using the long exact sequence
\begin{equation*}
\begin{tikzcd}[column sep=4ex]
 {}& \begin{array}[t]{l}\ti H^n(\Si K(\Z,3),\Z_2)\\ \mathrlap{\overset{\si}{\cong}H^{n-1}(K(\Z,3),\Z_2)}\end{array}\arrow[l,"\chi^*"'] & \ti H^n(K(\Z,4),\Z_2)\arrow[l,"\jmath^*"'] & \ti H^n(C,\Z_2)\arrow[l,"\de"'] & {}\lar
\end{tikzcd}
\end{equation*}
Recall $\ti H^n(K(\Z,4),\Z_2)$ from \eq{fm15eq2}. The groups $\ti H^n(K(\Z,3),\Z_2)$ are shown in \eq{fm17eq2} below. Using $\chi^*(\bar e_4)=(\bar d_3)^\si$ and the fact that $\chi^*$ commutes with stable cohomology operations, the long exact sequence implies $\ti H^8(C,\Z_2)=\Z_2\an{\bar c}$ for the unique class $\bar c$ satisfying $\jmath^*(\bar c)=\bar e_4^2$ and $\ti H^{10}(C,\Z_2)=\Z_2\an{\de(\bar d_3\bar d_5')^\si,\bar b}$ for a class $\bar b$ satisfying $\jmath^*(\bar b)=\bar e_4\bar e_6'$. Moreover $\ti H^n(C,\Z_2)=0$ for all  other values $8\neq n\le 9$. To compute $\Sq^2(\bar c)$, we will use the following lemma, which is applicable since $\chi$ can be identified with the canonical map $\Si\Om K(\Z,4)\ra K(\Z,4)$. In the notation of the lemma, we have $\bar d=\bar e_4$ and $\bar e=\bar d_3$, hence $\Sq^2(\bar c)$ is the image of $\bar d_3\cup\Sq^2(\bar d_3)=\bar d_3\cup \bar d_5'$ under the suspension and codifferential, which is the generator of $\ti H^{10}(C,\Z_2)$. Therefore, $\Sq^2(\bar c)\neq 0$.

By considering also the long exact sequence in $\Z$-homology one finds the homology groups of $C$ as in the following table.
\begin{equation*}
\begin{tabular}{c|ccccc}
$n$ & $0,1,2,3,4,5,6,7,9$ & $8$ & $10$\\
\hline
\parbox[top][4ex][c]{2cm}{$\ti H_n(C,\Z_2)$} & $0$ & $\Z_2$ & $\Z_2^2$ \\
\parbox[top][4ex][c]{2cm}{$\ti H_n(C,\Z)$} & $0$ & $\Z$ & $\Z_2^2$ \\
\end{tabular}
\end{equation*}
The $E^2$-page of the Atiyah--Hirzebruch spectral sequence $\ti H_p(C,\ti\Om^{\bs\Spin}_q(*))\Ra\ti\Om^{\bs\Spin}_{p+q}(C)$ is shown in Figure \ref{fm15fig6}. The only possible non-zero differential in the region $p+q\le 10$ is $d^2_{10,0}$, which by Proposition \ref{fm2prop1} is dual to the Steenrod square $\Sq^2:\ti H^8(C,\Z_2)\ra \ti H^{10}(C,\Z_2)$, hence non-zero. This leads to the $E^\iy$-page shown in Figure~\ref{fm15fig7}.

\begin{figure}[htb]
\centering
\begin{tikzpicture}
\matrix (m) [matrix of math nodes,
nodes in empty cells,nodes={minimum width=11.8ex,
minimum height=5ex,outer sep=-5pt},
column sep=.5ex,row sep=1ex]{
1 & \Z_2 & \\
0 & \Z & \Z_2^2 \\
\quad\strut & 8 \strut & 10\strut \\};
\draw[thick] (m-1-1.north east) node[above]{$q$} -- (m-3-1.east);
\draw[thick] (m-3-1.north) -- (m-3-3.north east) node[right]{$p$};
\draw[-stealth] (m-2-3) -- node[above right]{\scriptsize $d^2_{10,0}$}  (m-1-2);
\end{tikzpicture}
\caption{$E^2$-page of $\ti H_p(C,\ti\Om^{\bs\Spin}_q(*))\Ra\ti\Om^{\bs\Spin}_{p+q}(C)$, $p+q\le 10$}
\label{fm15fig6}
\end{figure}

\begin{figure}[htb]
\centering
\begin{tikzpicture}
\matrix (m) [matrix of math nodes,
nodes in empty cells,nodes={minimum width=11.8ex,
minimum height=5ex,outer sep=-5pt},
column sep=.5ex,row sep=1ex]{
1 & 0 \\
0 & \Z \\
\quad\strut & 8 \strut \\};
\draw[thick] (m-1-1.north east) node[above]{$q$} -- (m-3-1.east);
\draw[thick] (m-3-1.north) -- (m-3-2.north east) node[right]{$p$};
\end{tikzpicture}

\caption{$E^\iy$-page of $\tilde H_p(C;\tilde\Om^{\bs\Spin}_q(*))\Ra\ti\Om^{\bs\Spin}_{p+q}(C)$, $p+q\le 9$}
\label{fm15fig7}
\end{figure}

All extensions are trivial in Figure \ref{fm15fig7}, so $\ti\Om^{\bs\Spin}_n(C)=0$ for all $8\neq n\le 9$. By \eq{fm15eq13}, the map $\chi_*:\ti\Om_9^{\bs\Spin}(\Si K(\Z,3))\ra\ti\Om_9^{\bs\Spin}(K(\Z,4))$ is then surjective.
\end{proof}

\begin{lem}
\label{fm15lem3}
Let\/ $\chi:\Si\Om K\ra K$ be the map adjoint to\/ $\id_{\Om K}:\Om K\ra\Om K$ and form the mapping cone cofibre sequence $\Si\Om K\overset{\chi}{\longra}K\overset{\jmath}{\longra}C$. Let\/ $\bar d\in \ti H^n(K,\Z_2)$, and suppose there is a unique class\/ $\bar c\in \ti H^{2n}(C,\Z_2)$ satisfying\/ $\jmath^*(\bar c)=\bar d\cup \bar d,$ whose Steenrod square we wish to compute.

Let\/ $\bar e\in \ti H^{n-1}(\Om K,\Z_2)$ be the unique class with\/ $\bar e^\si=\chi^*(\bar d)$. Then\/ $\Sq^2(\bar c)\in\ti H^{2n+2}(C,\Z_2)$ equals the image of\/ $\Sq^2(\bar e)\cup \bar e\in H^{2n}(\Om K,\Z_2)$ under the suspension isomorphism and the co-differential of the mapping cone sequence,
\begin{equation*}
 \ti H^{2n}(\Om K,\Z_2)\overset{\si}{\longra}\ti H^{2n+1}(\Si\Om K,\Z_2)\overset{\de}{\longra}\ti H^{2n+2}(C,\Z_2).
\end{equation*}
\end{lem}

\begin{proof}
View the one-point union $K\vee K$ as the subcomplex $(K\t\{*\})\cup(\{*\}\t K)$ of $K\t K$. As the inclusion of a subcomplex is a cofibration, the mapping cone of the inclusion is just the smash product $K\w K$. There is a homotopy-commutative diagram whose rows are cofibre sequences
\begin{equation*}
 \begin{tikzcd}
   \Si\Om K\rar{\chi}\dar{\De'} & K\rar{\jmath}\dar{\De} & C\dar{\De''}\\
   K \vee K\rar{\io} & K\t K\rar{\pi} & K\w K,
 \end{tikzcd}
\end{equation*}
defined as follows. The map $\De'$ is adjoint to the map $\Om K\ra\Om(K\vee K)$ that sends a loop $\ga:[0,1]\ra K,$ $s\mapsto \ga(s)$ to the concatenation of the loop $(\ga(s),*)$ with the loop $(*,\ga(s))$ in $K\vee K$. There is a homotopy $\{h_t\}_{t\in[0,1]}:\De\circ\chi\simeq\io\circ\De'$, namely
\begin{equation*}
 h_t(\ga): s\mapsto
 \begin{cases}
 \bigl(\ga\bigl(\frac{2s}{1+t}\bigr),*\bigr) & \text{if $s\in[0,\frac{1-t}{2}]$,}\\
 \bigl(\ga\bigl(\frac{2s}{1+t}\bigr),\ga\bigl(\frac{2s-1+t}{1+t}\bigr)\bigr)& \text{if $s\in[\frac{1-t}{2},\frac{1+t}{2}]$,}\\
 \bigl(*,\ga\bigl(\frac{2s-1+t}{1+t}\bigr)\bigr) & \text{if $s\in[\frac{1+t}{2},1]$.}\\
 \end{cases}
\end{equation*}
The map $\pi\circ\De:K\ra K\w K$ and the null-homotopy $\{\pi\circ h_t\}_{t\in[0,1]}:\pi\circ\De\circ\chi\simeq \pi\circ\io\circ\De'=*$ together induce a map $\De'':C\ra K\w K$ on the mapping cone (the induced map on the mapping cone depends on the choice of null-homotopy).

Let $\tau: K\w K\ra K\w K$ be the map that exchanges the two factors of $K$. Then $\tau\circ\pi\circ h_t$ defines another null-homotopy of $\pi\circ\De\circ\chi=\tau\circ\pi\circ\De\circ\chi$ whose induced map on the mapping cone is simply $\tau\circ\De'':C\ra K\w K$.

Observe that $\jmath^*(\De'')^*(d \t d)=\De^*(d\t\bar d)=d^2$, in terms of the cross product on (relative) cohomology, so $c=(\De'')^*(d \t d)$ by the characterization of $t$. Using the naturality and the Cartan formula for the Steenrod operations, we compute
\begin{align*}
 \Sq^2(c)=(\De'')^*\Sq^2(d \t d)&=(\De'')^*\bigl(\Sq^2(d)\t d + d\t\Sq^2(d)\bigr)\\
 &=(\De'')^*\bigl(\Sq^2(d)\t d\bigr) + (\De'')^*\tau^*\bigl(\Sq^2(d)\t d\bigr).
\end{align*}
As $\jmath^*(\De'')^*=\pi^*\De^*$ and $\jmath^*(\De'')^*\tau^*=\pi^*\De^*\tau^*=\pi^*\De^*$ we have $\jmath^*\Sq^2(c)=0$ in $\Z_2$-cohomology, so from the long exact sequence of the mapping cone we know that $\Sq^2(c)$ is in the image of $\de$. We can write $\Sq^2(c)$ in an explicit way as an image under $\de$ as follows. Let $\{k_t\}_{t\in[0,1]}$ be the concatenation of the homotopy $\{\pi\circ h_t\}_{t\in[0,1]}$ with the inverse homotopy $\{\tau\circ\pi\circ h_t\}_{t\in[0,1]}$. Then $\{k_t\}_{t\in[0,1]}$ is a homotopy from the constant map $\Si\Om K\ra K\w K$ to itself, so can be viewed as a map $k:\Si^2\Om K\ra K\w K$. As a special case of a general formula for mapping cones that describes the dependence in cohomology of the induced map on the mapping cone on the choice of homotopy, we have
\begin{equation*}
 (\De'')^*\bigl(\Sq^2(d)\t d\bigr) + (\De'')^*\tau^*\bigl(\Sq^2(d)\t d\bigr)=\de k^*(\Sq^2(d)\t d).
\end{equation*}
It is not hard to see that $k$ is homotopic to the map
\begin{multline*}
 \Si^2\Om K=\cS^1\w\cS^1\w\Om K\xrightarrow{\id_{\cS^1}\w\id_{\cS^1}\w\De}\cS^1\w\cS^1\w\Om K\w\Om K\cong \cS^1\w\Om K\w\cS^1\w\Om K\\
 \xrightarrow{\chi\w\chi} K\w K,
\end{multline*}
so $k^*(\Sq^2(d)\t d)=(\Sq^2(e)\cup e)^\si$ and therefore $\Sq^2(c)=\de(\Sq^2(e)\cup e)^\si$.
\end{proof}

\subsection{\texorpdfstring{Computation of $\ti\Om_n^{\bs{\mathrm{Spin}}}(M\Spin(4))$}{Computation of Ωₙˢᵖⁱⁿ(MSpin(4))}}
\label{fm154}

\subsubsection{Description of the (co)homology}
\label{fm1541}

As $\Spin(4)=\Sp(1)\t\Sp(1)$, we have $B\Spin(4)=B\Sp(1)\t B\Sp(1)$ and hence $H^*(B\Spin(4),\Z)=\Z[c_2^+,c_2^-]$ is a polynomial ring in two variables $c_2^\pm$, the second Chern classes of the spinor quaternionic line bundles $\Si^\pm\ra B\Spin(4)$. Dually, $H_4(B\Spin(4),\Z)=\Z\an{\ga_2^+,\ga_2^-}$ is generated by homology classes $\{\ga_2^+,\ga_2^-\}$ dual to $\{c_2^+, c_2^-\}$. As the $\Z$-(co)homology of $B\Spin(4)$ has no torsion, the $\Z_2$-(co)homology is obtained by reduction modulo $2$. The (co)homologies of the Thom space $M\Spin(4)$ are the same as for $B\Spin(4)$ except for a degree shift under the (co)homological Thom isomorphism.

\subsubsection{Computation of the spectral sequence}
\label{fm1542}

Consider the Atiyah--Hirzebruch spectral sequence
\begin{equation*}
 \ti H_p(M\Spin(4),\Om^{\bs\Spin}_q(*))\Longra\ti\Om^{\bs\Spin}_{p+q}(M\Spin(4)),
\end{equation*}
whose $E^2$-page for $p+q\le 10$ is shown in Figure \ref{fm15fig8}.

\begin{figure}[htb]
\centering
\begin{tikzpicture}
  \matrix (m) [matrix of math nodes,
    nodes in empty cells,nodes={minimum width=11.8ex,
    minimum height=5ex,outer sep=-5pt},
    column sep=1ex,row sep=1ex]{
    4 & \Z\an{\al_4\tau}  &\\
	2 & \Z_2\an{\al_1^2\bar\tau} & \Z_2\an{\al_1^2(\bar\ga_2^+)^T,\al_1^2(\bar\ga_2^-)^T} &\\
	1 & \Z_2\an{\al_1\bar\tau}  & \Z_2\an{\al_1(\bar\ga_2^+)^T,\al_1(\bar\ga_2^-)^T} & \\
	0 & \Z\an{\tau}  & \Z\an{(\ga_2^+)^T,(\ga_2^-)^T} \\
    \quad\strut  & 4 &  8  \strut \\};
 \draw[thick] (m-1-1.north east) node[above]{$q$} -- (m-5-1.east);
 \draw[thick] (m-5-1.north) -- (m-5-3.north east) node[right]{$p$};
\end{tikzpicture}
\caption{$E^2$-page of $\ti H_p(M\Spin(4),\Om_q^{\bs\Spin}(*))\!\Ra\ab\ti\Om^{\bs\Spin}_{p+q}(M\Spin(4))$, $p\!+\!q\!\le\! 10$, which is also the $E^\iy$-page for $p+q\le 9$.}
\label{fm15fig8}
\end{figure}

All differentials vanish in this region, so this is also the $E^\iy$-page for $p+q\le 9$. Moreover, all extension problems are trivial.

\subsubsection{Generators}
\label{fm1543}

Let $n=4$. By Proposition \ref{fm2prop2} we have an isomorphism $\ti\Om^{\bs\Spin}_4(M\Spin(4))\ra H_4(M\Spin(4),\Z)\cong\Z$, $[X,M]\mapsto\int_M 1$ and it is easy to see that $\de\mapsto1$, hence $\de$ generates $\ti\Om^{\bs\Spin}_4(M\Spin(4))$.

Let $n=5,6$. From Figure \ref{fm15fig8} we see that the generators in dimensions $5,6$ are obtained from the generator of $\ti\Om^{\bs\Spin}_4(M\Spin(4))$ by multiplication by~$\al_1, \al_1^2$.

Let $n=8$. Observe that \eq{fm3eq15} is a well-defined morphism which, by Rokhlin's theorem \cite[p.89]{LaMi}, takes values in $\Z^3$. It is easy to check that \eq{fm3eq15} maps $\ze_1\mapsto(-1,0,0)$, $\ze_2\mapsto (0,1,0)$, and $\ze_2'\mapsto(0,0,-1)$, where we note that $\ze_2'$ is obtained from $\ze_2$ by exchanging the two normal spinor bundles {\it and} reversing the orientation on $M$. As $\ti\Om_8^{\bs\Spin}(M\Spin(4))\cong\Z^3$ and \eq{fm3eq15} maps to a basis, it follows that \eq{fm3eq15} an isomorphism and that $\ze_1,\ze_2,\ze_2'$ generate $\ti\Om_8^{\bs\Spin}(M\Spin(4))$.

Let $n=9$. From Figure \ref{fm15fig8} we see that the generator in dimension $9$ is obtained from the generator in dimension $8$ by multiplication with $\al_1$. This completes the proof of Table \ref{fm3tab1} for $M\Spin(4)$.

\subsection{\texorpdfstring{Computation of $\ti\Om_n^{\bs{\mathrm{Spin}}}(M\U(2))$}{Computation of Ωₙˢᵖⁱⁿ(MU(2))}}
\label{fm155}

\subsubsection{Description of the (co)homology}
\label{fm1551}

Recall that $H^*(B\U(2),\Z)=\Z[c_1,c_2]$ is a polynomial ring on the Chern classes. Write $\ga_1\in H_2(B\U(2),\Z)$ and $\ga_2\in H_4(B\U(2),\Z)$ for the homology classes dual to $c_1$ and $c_2$. This gives the following table of (co)homology groups.
\begin{equation*}
\begin{tabular}{c|cccccc}
$p$ & 
$\!\!0,1,3,5$ & $2$ & $4$ & $6$\\
\hline
\parbox[top][4ex][l]{2.4cm}{$\ti H_p(B\U(2),\Z)$} & 
$0$ &  $\Z\an{\ga_1}$ & $\Z\an{\ga_1^2,\ga_2}$ & $\Z\an{\ga_1^3,\ga_1\ga_2}$ \\
\parbox[top][4ex][l]{2.4cm}{$\ti H^p(B\U(2),\Z)$} & 
$0$ &  $\Z\an{c_1}$ & $\Z\an{c_1^2,c_2}$ & $\Z\an{c_1^3,c_1c_2}$ \\
\parbox[top][4ex][l]{2.4cm}{$\ti H_p(B\U(2),\Z_2)$} & 
$0$ &  $\Z_2\an{\bar\ga_1}$ & $\Z_2\an{\bar\ga_1^2,\bar\ga_2}$ & $\Z_2\an{\bar\ga_1^3,\bar\ga_1\bar\ga_2}$ \\
\parbox[top][4ex][l]{2.4cm}{$\ti H^p(B\U(2),\Z_2)$} & 
$0$ &  $\Z_2\an{\bar c_1}$ & $\Z_2\an{\bar c_1^2,\bar c_2}$ & $\Z_2\an{\bar c_1^3,\bar c_1\bar c_2}$
\end{tabular}
\end{equation*}
The (co)homologies of the Thom space $M\U(2)$ are the same as for $B\U(2)$ except for a degree shift under the (co)homological Thom isomorphism.

The inclusion $\U(2)\hookrightarrow\SO(4)$ induces a map $\ka:B\U(2)\ra B\SO(4)$, which acts on $\Z$- and $\Z_2$-cohomology by
\begin{align*}
\ka^*&:p_1\longmapsto c_1^2-2c_2,&
\ka^*&:e\longmapsto c_2,&
\ka^*&:W_3\longmapsto 0,\\
\ka^*&:\bar w_2\longmapsto \bar c_1,&
\ka^*&:\bar w_3\longmapsto 0,&
\ka^*&:\bar w_4\longmapsto \bar c_2.
\end{align*}
From this we deduce the action of $\ka$ on $\Z$- and $\Z_2$-homology to be
\begin{align*}
\ka_*&:\ga_1\longmapsto\om_2,&
\ka_*&:\ga_1^2\longmapsto \pi_1,&
\ka_*&:\ga_2\longmapsto -2\pi_1+\varep,&
\ka_*&:\ga_1^3\longmapsto \om_6, \\
\ka_*&:\ga_1\ga_2\longmapsto \om_6', &
\ka_*&:\bar\ga_1\longmapsto \bar\om_2,&
\ka_*&:\bar\ga_1^2\longmapsto \bar\om_2^2,&
\ka_*&:\bar\ga_2\longmapsto\bar\om_4, \\
\ka_*&:\bar\ga_1^3\longmapsto \bar\om_2^3,&
\ka_*&:\bar\ga_1\bar\ga_2\longmapsto\bar\om_2\bar\om_4.
\end{align*}
The inclusion $\U(2)\hookrightarrow\SO(4)$ induces a map $\mu:M\U(2)\ra M\SO(4)$. As $\ka,\mu$ commute with the Thom isomorphism, we see that $\mu$ acts on homology by
\begin{align*}
\mu_*&:\tau\longmapsto\tau,&
\mu_*&:\ga_1^T\longmapsto\om_1^T,&
\mu_*&:(\ga_1^2)^T\longmapsto \pi^T_1,\\
\mu_*&:\ga^T_2\longmapsto -2\pi^T_1+\varep^T,&
\mu_*&:(\ga_1^3)^T\longmapsto \om^T_6, &
\mu_*&:(\ga_1\ga_2)^T\longmapsto \om_6^{\prime T}, \\
\mu_*&:\bar\tau\longmapsto\bar\tau,&
\mu_*&:\bar\ga_1^T\longmapsto \bar\om_2^T,&
\mu_*&:(\bar\ga_1^2)^T\longmapsto(\bar\om_2^2)^T,\\
\mu_*&:\bar\ga^T_2\longmapsto\bar\om_4^T.
\end{align*}

\subsubsection{Computation of the spectral sequence}
\label{fm1552}

Consider the Atiyah--Hirzebruch spectral sequence
\begin{equation*}
 \ti H_p(M\U(2),\Om^{\bs\Spin}_q(*))\Longra\ti\Om^{\bs\Spin}_{p+q}(M\U(2)),
\end{equation*}
whose $E^2$-page for $p+q\le 10$ is shown in Figure \ref{fm15fig9}. Since $\mu:M\U(2)\ra M\SO(4)$ induces a morphism of the spectral sequences, we can compare Figures \ref{fm15fig4} and \ref{fm15fig9} to see that
\begin{equation*}
d^2_{6,0}(\ga_1^T)=\al_1\bar\tau,\;\>
d^2_{8,0}=0,\;\>
d^2_{10,0}\bigl((\ga_1^3)^T\bigr)=\al_1\bar\ga_1^2,\;\>
d^2_{10,0}\bigl((\ga_1\ga_2)^T\bigr)=0.
\end{equation*}
The differentials $d^2_{6,1}=d^2_{6,0}\circ\rho_2$, $d^2_{8,1}=d^2_{8,0}\circ\rho_2$ are obtained by composing with the reduction modulo two. This leads to the $E^3$-page of the spectral sequence for $p+q\le 9$, as in Figure \ref{fm15fig10}.

\begin{figure}[htb]
\centering
\begin{tikzpicture}
  \matrix (m) [matrix of math nodes,
    nodes in empty cells,nodes={minimum width=11.8ex,
    minimum height=5ex,outer sep=-5pt},
    column sep=.8ex,row sep=1ex]{
    4 & \Z\an{\al_4\tau} & \Z\an{\al_4\ga^T_1} &\\
	2 & \Z_2\an{\al_1^2\bar\tau} & \Z_2\an{\al_1^2\bar\ga^T_1} & \Z_2\an{\al_1^2(\bar\ga^2_1)^T,\al_1^2\bar\ga_2^T} &\\
	1 & \Z_2\an{\al_1\bar\tau} & \Z_2\an{\al_1\bar\ga^T_1} & \Z_2\an{\al_1(\bar\ga^2_1)^T,\al_1\bar\ga_2^T} & \\
	0 & \Z\an{\tau} & \Z\an{\ga^T_1} & \Z\an{(\ga^2_1)^T,\ga_2^T} & \Z\an{(\ga^3_1)^T,(\ga_1\ga_2)^T} \\
    \quad\strut  & 4 & 6 & 8 & 10 \strut \\};
  \draw[-stealth] (m-4-5) -- node[right,yshift=0.5ex]{\scriptsize $d^2_{10,0}$} (m-3-4);
  \draw[-stealth] (m-4-4) -- node[right,yshift=0.25ex]{\scriptsize $d^2_{8,0}$} (m-3-3);
  \draw[-stealth] (m-3-4) -- node[right,yshift=0.25ex]{\scriptsize $d^2_{8,1}$} (m-2-3);
  \draw[-stealth] (m-4-3) -- node[right,yshift=0.1ex]{\scriptsize $d^2_{6,0}$} (m-3-2);
  \draw[-stealth] (m-3-3) -- node[right,yshift=0.1ex]{\scriptsize $d^2_{6,1}$} (m-2-2);
 \draw[thick] (m-1-1.north east) node[above]{$q$} -- (m-5-1.east);
 \draw[thick] (m-5-1.north) -- (m-5-5.north east) node[right]{$p$};
\end{tikzpicture}
\caption{$E^2$-page of $\ti H_p(M\U(2),\Om^{\bs\Spin}(*)_q)\Rightarrow\ti\Om^{\bs\Spin}_{p+q}(M\U(2))$, $p+q\le 10$}
\label{fm15fig9}
\end{figure}

\begin{figure}[htb]
\centering
\begin{tikzpicture}
\matrix (m) [matrix of math nodes, nodes in empty cells,nodes={minimum width=12ex, minimum height=5ex,outer sep=-5pt}, column sep=1ex,row sep=1ex]{
4 & \Z\an{\al_4\tau}  &\\
2 & 0 & \Z_2\an{\al_1^2\bar\ga_1^T} & \\
1 & 0 & 0 & \Z_2\an{\al_1\bar\ga_2^T} \\
0 & \Z\an{\tau} & \Z\an{2\ga_1^T} & \Z\an{(\ga^2_1)^T,\ga_2^T} \\
\quad\strut & 4 & 6 & 8 \strut \\};
\draw[thick] (m-1-1.north east) node[above]{$q$} -- (m-5-1.east);
\draw[thick] (m-5-1.north) -- (m-5-5.north east) node[right]{$p$};
\end{tikzpicture}
\caption{$E^3$-page of $\ti H_p(M\U(2),\Om^{\bs\Spin}(*)_q)\Rightarrow\ti\Om^{\bs\Spin}_{p+q}(M\U(2))$, $p+q\le 9$. This is also the $E^\iy$-page in this region.}
\label{fm15fig10}
\end{figure}

There are no higher differentials (for reasons of degree), and the spectral sequence collapses at the $E^3$-page. Hence Figure \ref{fm15fig10} is also the $E^\iy$-page for $p+q\le 9$. The only non-trivial extension problem is in dimension $n=8$, where $\ti\Om_8^{\bs\Spin}(M\U(2))$ could be either $\Z^3$ or $\Z^3\op\Z_2$. Consider the classifying map $\phi:M\U(2)\ra K(\Z,4)$ of the Thom class. By definition, $\phi^*(e_4)=t$, so $\phi^*(\bar e_6')=\phi^*(\Sq^2(\bar e_4))=\Sq^2(\bar t)=w_2\cup \bar t=\bar c_1\cup \bar t$, by \eq{fm15eq8}. Hence the induced map $\phi_*$ from Figure \ref{fm15fig10} to Figure \ref{fm15fig3} maps $\al_4\tau\mapsto\al_4\ep_4$ and $\al_1^2\bar\ga_1^T\mapsto\al_1^2\bar\ep_6'$. Let $\ti F_{p,q}$ be the filtration of $\ti\Om_n^{\bs\Spin}(M\U(2))$ whose associated graded modules are given by the $E^\iy$-page shown in Figure \ref{fm15fig10}. Then $\phi$ induces a map of extensions
\begin{equation*}
 \begin{tikzcd}
 	0\rar & \ti F_{4,8}\cong\Z\an{\al_4\tau}\rar\arrow[d,"\phi_*","\cong"'] & \ti F_{6,8}\dar{\phi_*}\rar & \Z_2\an{\al_1^2\bar\ga_1^T}\rar\arrow[d,"\phi_*","\cong"'] & 0\\
 	0\rar & F_{4,8}\cong\Z\an{\al_4\ep_4}\rar & F_{6,8}\rar & \Z_2\an{\al_1^2\bar\ep_6'}\rar & 0,\!
 \end{tikzcd}
\end{equation*}
where the far left and far right map are isomorphisms. Hence the map in the middle is also an isomorphism by the $5$-lemma, so $\ti F_{6,8}\cong F_{6,8}\cong\Z$ by \eq{fm15eq5}. The group $\ti\Om_8^{\bs\Spin}(M\U(2))=\ti F_{8,8}$ is thus an extension of $E_{8,0}^\iy\cong\Z^2$ by $\ti F_{6,8}\cong\Z$, which must be $\Z^3$.

\subsubsection{Generators}
\label{fm1553}

Let $n=4$. By Proposition \ref{fm2prop2} we have an isomorphism $\ti\Om^{\bs\Spin}_4(M\U(2))\ra H_4(M\U(2),\Z)\cong\Z$, $[X,M]\mapsto\int_M 1$ and it is easy to see that $\de\mapsto1$, hence $\de$ generates $\ti\Om^{\bs\Spin}_4(M\U(2))$.

Let $n=6$. By Proposition \ref{fm2prop2} we have an injective map $\ti\Om^{\bs\Spin}_6(M\U(2))\ra H_6(M\U(2),\Z)\cong\Z,$ $[X,M]\mapsto\int_M c_1(\nu_M)$ with cokernel $\Z_2$. We have $\varep\mapsto 2$, so this is the generator of $\ti\Om^{\bs\Spin}_6(M\U(2))$.

Let $n=8$. Observe that \eq{fm3eq13} is a well-defined morphism that maps $\ze_1\mapsto(1,0,0)$, $\ze_2\mapsto(0,1,0)$, $\ze_2'\mapsto(0,0,1)$. As $\ti\Om^{\bs\Spin}_8(M\U(2))\cong\Z^3$ and \eq{fm3eq13} maps to a basis, the map \eq{fm3eq13} must be an isomorphism and take values in $\Z^3$. Moreover, it follows that $\ze_1,\ze_2,\ze_2'$ generate $\ti\Om^{\bs\Spin}_8(M\U(2))$.

Let $n=9$. From Figure \ref{fm15fig10} we see that the generator in dimension $9$ is obtained by multiplying the generator $\ze_2$ (detected by $[X,M]\mapsto\ts\int_M c_2(\nu_M)$, corresponding to $\ga_2^T\in E^\iy_{8,0}$) by $\al_1$. Hence $\ti\Om^{\bs\Spin}_8(M\U(2))=\Z_2\an{\al_1\ze_2}$. This complete the proof of Table \ref{fm3tab1} for $M\U(2)$.

\subsection{\texorpdfstring{Computation of $\ti\Om_n^{\bs{\mathrm{Spin}}}(M\SU(2))$}{Computation of Ωₙˢᵖⁱⁿ(MSU(2))}}
\label{fm156}

\subsubsection{Description of the (co)homology}
\label{fm1561}

Recall that $H^*(B\SU(2),\Z)=\Z[c_2]$ is a polynomial ring on the second Chern class. Let $\ga_2\in H_4(B\SU(2),\Z)$ be dual to $c_2$. The (co)homologies of the Thom space $M\SU(2)$ are the same as for $B\SU(2)$ except for a degree shift under the (co)homological Thom isomorphism.

\subsubsection{Computation of the spectral sequence}
\label{fm1562}

Consider the Atiyah--Hirzebruch spectral sequence
\begin{equation*}
 \ti H_p(M\SU(2),\Om^{\bs\Spin}_q(*))\Longra\ti\Om^{\bs\Spin}_{p+q}(M\SU(2)),
\end{equation*}
whose $E^2$-page for $p+q\le 10$ is shown in Figure~\ref{fm15fig11}.

\begin{figure}[htb]
\centering
\begin{tikzpicture}
  \matrix (m) [matrix of math nodes,
    nodes in empty cells,nodes={minimum width=11.8ex,
    minimum height=5ex,outer sep=-5pt},
    column sep=1ex,row sep=1ex]{
    4 & \Z\an{\al_4\tau}  &\\
	2 & \Z_2\an{\al_1^2\bar\tau} & \Z_2\an{\al_1^2\bar\ga_2^T} &\\
	1 & \Z_2\an{\al_1\bar\tau}  & \Z_2\an{\al_1\bar\ga_2^T} & \\
	0 & \Z\an{\tau}  & \Z\an{\ga_2^T} \\
    \quad\strut  & 4 &  8  \strut \\};
 \draw[thick] (m-1-1.north east) node[above]{$q$} -- (m-5-1.east);
 \draw[thick] (m-5-1.north) -- (m-5-3.north east) node[right]{$p$};
\end{tikzpicture}
\caption{$E^2$-page of $\ti H_p(M\SU(2),\Om_q^{\bs\Spin}(*))\!\Ra\!\ti\Om^{\bs\Spin}_{p+q}(M\SU(2))$, $p\!+\!q\!\le\! 10$.}
\label{fm15fig11}
\end{figure}

All differentials vanish in this region, so the $E^2$-page is also the $E^\iy$-page for $p+q\le 9$. All extensions are trivial in this region.

\subsubsection{Generators}
\label{fm1563}

Let $n=4$. By Proposition \ref{fm2prop2} we have an isomorphism $\ti\Om^{\bs\Spin}_4(M\SU(2))\ra H_4(M\SU(2),\Z)\cong\Z$, $[X,M]\mapsto\int_M 1$ which maps $\de\mapsto1$, hence $\de$ generates $\ti\Om^{\bs\Spin}_4(M\SU(2))$.

Let $n=8$. Observe that \eq{fm3eq10} is a well-defined morphism which, by Rokhlin's theorem \cite[p.~89]{LaMi}, takes values in $\Z^2$. It is easy to check that $\ze_1\mapsto(1,0)$, $\ze_2\mapsto(0,1)$. Since $\ti\Om^{\bs\Spin}_8(M\SU(2))\cong\Z^2$ and \eq{fm3eq10} maps to a basis, the map \eq{fm3eq10} must be an isomorphism. Moreover, it follows that $\ze_1,\ze_2$ generate $\ti\Om^{\bs\Spin}_8(M\SU(2))$.

Let $n=9$. From Figure \ref{fm15fig11} we see that the generator in dimension $9$ is obtained from the generator $\ze_2$ (which is detected by $\int_M c_2(\nu_M)$, corresponding to $E^\iy_{8,0}=\Z\an{\ga_2^T}$) by multiplication with $\al_1$. This proves Table \ref{fm3tab1} for $M\SU(2)$.

\subsection{\texorpdfstring{Computation of $\ti\Om_n^{\bs{\mathrm{Spin}}}(K(\Z_2,4))$}{Computation of Ωₙˢᵖⁱⁿ(K(ℤ₂,4))}}
\label{fm157}

\subsubsection{Description of the (co)homology}
\label{fm1571}

The integer cohomology groups of $K(\Z_2,4)$ are not commonly found in the literature. A general algorithmic framework for computing the cohomology of Eilenberg--MacLane spaces was developed by Cartan \cite{Car}. Moreover, May \cite[Th.~10.4]{May} computes the Bockstein spectral sequence of $K(\Z_{p^t},n)$, which implies the integer cohomology groups, but the result is somewhat complicated. For convenience, we present here an elementary calculation in degrees $\le 10$.

For this we begin with the $\Z_2$-cohomology of $K(\Z_2,4)$, which was described by Serre \cite{Serr}. Let $\bar f_4'\in \ti H^4(K(\Z_2,4),\Z_2)$ be the primary class as in \S\ref{fm23}. The $\Z_2$-cohomology is then a polynomial algebra on generators
\begin{align*}
 \bar f_4',\enskip
 \bar f_5=\Sq^1(\bar f_4'),\enskip
 \bar f_6'=\Sq^2(\bar f_4'),\enskip
 \bar f_7=\Sq^3(\bar f_4'),\; \bar f_7'=\Sq^2\Sq^1(\bar f_4'),\\
 \bar f_8=\Sq^3\Sq^1(\bar f_4'),\enskip
 \bar f_9'=\Sq^4\Sq^1(\bar f_4'),\enskip
 \bar f_{10}'=\Sq^4\Sq^2(\bar f_4'),\enskip
 \ldots.
\end{align*}
This determines $\ti H^n(K(\Z_2,4),\Z_2)$ for all $n\le 10$. Dually, $\ti H_n(K(\Z_2,4),\Z_2)$ are generated as groups by homology classes  $\bar \varphi_4', \bar \varphi_5,\ldots, \bar \varphi_8,(\bar \varphi_4')^2$, again using Notation \ref{fm15nota1}. These groups are recorded in \eq{fm15eq17} and \eq{fm15eq19}.

\begin{prop}
\label{fm15prop4}
For each\/ $n\in\N,$ the reduced cohomology\/ $\ti H^*(K(\Z_2,n),\Z)$ consists entirely of elements of order\/ $2^k$ with\/~$k\ge 1$.
\end{prop}

\begin{proof}
We can equivalently prove that the localization $\ti H^*(K(\Z_2,n),\Z)_{(2)}$ at the prime $2$ vanishes. This is proved by induction.

For the base case, recall that $K(\Z_2,1)\simeq\RP^\iy$ and that $\ti H^*(\RP^\iy,\Z)$ is generated as a ring by an element of order two. In particular, $\ti H^*(K(\Z_2,1),\Z)_{(2)}$ vanishes. For the inductive step, consider the Serre spectral sequences of the path space fibration
\begin{equation*}
 K(\Z_2,n)\simeq \Om K(\Z_2,n+1) \longra PK(\Z_2,n+1) \longra K(\Z_2,n+1).
\end{equation*}
Since localization is an exact functor, it induces a spectral sequence
\begin{equation*}
 \ti H^p(K(\Z_2,n+1);H^q(K(\Z_2,n),\Z)_{(2)})\Longrightarrow \ti H^{p+q}(PK(\Z_2,n+1),\Z)_{(2)}.
\end{equation*}

The path space $PK(\Z_2,n+1)$ is contractible, $H^*(K(\Z_2,n),\Z)_{(2)}$ vanishes in positive degree by the inductive hypothesis, and $H^0(K(\Z_2,n),\Z)_{(2)}\cong\Z_{(2)}$. Therefore the $E_2$-page of the spectral sequence reduces to a single row
\begin{equation*}
 E^{p,0}_2=\ti H^p(K(\Z_2,n+1),\Z_{(2)})\cong \ti H^p(K(\Z_2,n+1),\Z)_{(2)}
\end{equation*}
and, in particular, degenerates at the $E_2$-page. Hence $\ti H^p(K(\Z_2,n+1),\Z)_{(2)}\cong \ti H^p(PK(\Z_2,n+1),\Z)_{(2)}=0$ for all $p$, thus completing the inductive step.
\end{proof}

The Bockstein spectral sequence gives a method for computing the integral cohomology of a topological space $X$ from the cohomology with $\Z_p$-coefficients and the rational cohomology. For each prime $p$, the short exact sequence of coefficients
\begin{equation*}
\begin{tikzcd}[row sep=0]
	0\rar & \Z \rar{\mu_p} & \Z\rar{\rho_p} & \Z_p\rar & 0
\end{tikzcd}
\end{equation*}
induces an exact couple (see McCleary~\cite[p.37]{McCl})
\begin{equation*}
\xymatrix@C=60pt@R=15pt{
*+[r]{H^*(X,\Z)} \ar[rr]_{\mu_p} && *+[l]{H^*(X,\Z)} \ar[dl]^{\rho_p} \\
 & E_1=H^*(X,\Z_p) \ar[ul]^{\be_p}_(0.4){[1]} }
\end{equation*}
where $\be_p$ increases the degree by $1$. The spectral sequence is singly graded, the first page being $E_1^q=H^q(X,\Z_p)$, $d_1=\rho_p\circ\be_p$. The higher pages are obtained iteratively by forming derived couples. After $r$ iterations, the exact couple takes the form
\begin{equation*}
\xymatrix@C=60pt@R=15pt{
*+[r]{p^rH^*(X,\Z)} \ar[rr]_{\mu_p} && *+[r]{p^rH^*(X,\Z)} \ar[dl]^{\rho_p\left(\frac{1}{p^r}\text{\textvisiblespace}\right)} \\
 & E_{r+1}, \ar[ul]^{\be_p}_(0.4){[1]} }
\end{equation*}
where $E_{r+1}$ consists of elements $x\in H^*(X,\Z_p)$ such that $\be_p(x)$ is divisible by $p^r$, modulo classes of the form $x=\rho_p\left(\be_p(y)/p^s\right)$ for $0\le s\le r-1$ and $y\in H^{*-1}(X,\Z_p)$ with $\be_p(y)$ divisible by $p^s$. The differential is $d_{r+1}=\rho_p(\be_p/p^r)$.

Moreover, there is one summand $\Z_{p^r}$ in $H^*(X,\Z)$ for each summand $\Z_p$ in the image of the $r$-th differential $d_r:E_r\ra E_r$; if $y$ generates the $\Z_p$ summand and $d_r(x)=y$, then $\be_p(x)/p^{r-1}$ generates the summand $\Z_{p^r}$ in $H^*(X,\Z)$. A class $x\in H^*(X,\Z_p)$ that survives to the $E_\iy$-page corresponds to an infinite cyclic summand $\Z\an{\be_p(x)}$ in $H^*(X,\Z)$. Moreover, $\rho_p(\be_p(x))=d_1(x)$. In this way one reconstructs the entire integral cohomology of $X$.

With this background in place, we now calculate $H^*(K(\Z_2,4),\Z)$. By Proposition \ref{fm15prop4} it suffices to consider the prime $p=2$. By \eq{fm2eq16} the first differential is the Steenrod operation $\Sq^1=\rho_2\circ\be_2$. The Adem relation \eq{fm2eq15} allows the computation of the action of $\Sq^1$ on all classes \eq{fm15eq17}:
\begin{align*}
\bar f_4'&\mapsto \bar f_5, & \bar f_5&\mapsto 0, & \bar f_6'&\mapsto \bar f_7,\\
\bar f_7&\mapsto 0, &\bar f_7'&\mapsto \bar f_8, & \bar f_8&\mapsto 0, \\ 
\bar f_4'^2&\mapsto 0, & \bar f_9'&\mapsto \bar f_5^2, &\bar f_4'\bar f_5&\mapsto \bar f_5^2,\\
\bar f_{10}'&\mapsto \bar f_{11}, &\bar f_4'\bar f_6'&\mapsto\bar f_5\bar f_6'+\bar f_4'\bar f_7, &\bar f_5^2&\mapsto 0.
\end{align*}

Hence $E_2^8=\Z_2\an{\bar f_4'^2}$, $E_2^9=\Z_2\an{\bar f_9'+\bar f_4'\bar f_5}$, and $E_2^q=0$ for $8,9 \neq q\le 10$. Moreover, the image of the first differential gives rise to $\Z_2$-summands in $H^*(K(\Z_2,4),\Z)$ on generators
\begin{align*}
 f_5&=\Sq^1_\Z(\bar f_4'),
 &f_7&=\Sq^3_\Z(\bar f_4'),
 &f_8&=\Sq^3_\Z\Sq^1(\bar f_4'),
 &&f_5^2,
\end{align*}
where we recall $\Sq_\Z$ from \eq{fm2eq17}. We claim that
\e
\label{fm15eq14}
 d_2^8(\bar f_4'^2)=\bar f_9'+\bar f_4'\bar f_5,
\e
which implies that $H^9(K(\Z_2,4),\Z)$ is isomorphic to $\Z_4$, and that the Bockstein spectral sequence degenerates on the $E_2$-page. Recall the Pontrjagin square from \S\ref{fm23}.

\begin{lem}
\label{fm15lem4}
{\bf(a)} $\be_2(\bar f_4'^2)=2\be_4\cP(\bar f_4')$. \quad {\bf(b)} $\rho_2(\be_4\cP(\bar f_4'))\neq 0$.
\smallskip

\noindent{\bf(c)} $d_2^8(\bar f_4'^2)\neq 0,$ hence \eq{fm15eq14} holds.
\end{lem}

\begin{proof}
(a) Consider the commutative diagram
\begin{equation*}
\begin{tikzcd}[column sep=large]
	0\rar & \Z\rar{\mu_4}\dar{\mu_2} & \Z\rar{\rho_4}\dar{\id} & \Z_4\dar{\bar\rho_2}\rar & 0\\
	0\rar & \Z\rar{\mu_2} & \Z\rar{\rho_2} & \Z_2\rar & 0\mathrlap{.}
\end{tikzcd}
\end{equation*}
By the naturality of the Bockstein homomorphism with respect to short exact coefficient sequences, we have $2\be_4=\be_2\circ\rho_2$, which implies
\begin{equation*}
 2\be_4\cP(\bar f_4')=\be_2\rho_2\cP(\bar f_4')=\be_2(\bar f_4'^2).
\end{equation*}
 
\noindent
(b) The proof is by contradiction, so suppose that $\rho_2\be_4\cP(\bar f_4')=0$. Then the cohomology operation
\e
\label{fm15eq15}
 \rho_2\circ\be_4\circ\cP: H^4(X,\Z_2)\longra H^9(X,\Z_2)
\e
vanishes for every topological space $X$. The commutative diagram
\begin{equation*}
\begin{tikzcd}[column sep=large]
	0\rar & \Z\rar{\mu_2}\dar{\id} & \Z\rar{\rho_2}\dar{\mu_2} & \Z_2\dar{\ti\mu_2}\rar & 0\\
	0\rar & \Z\rar{\mu_4} & \Z\rar{\rho_4} & \Z_4\rar & 0
\end{tikzcd}
\end{equation*}
and the naturality of the Bockstein homomorphism imply $\be_4\circ\ti\mu_2=\be_2$. By applying $\rho_2\circ\be_4$ to Definition \ref{fm2def4}(b), we find that
\begin{align*}
\rho_2\circ\be_4\circ\cP(\bar x+\bar y)&=\rho_2\circ\be_4\circ\cP(\bar x)+\rho_2\circ\be_4\circ\cP(\bar y)+\rho_2\circ\be_4\circ\ti\mu_2(\bar x\cup \bar y)\\
 &=\rho_2\circ\be_4\circ\ti\mu_2(\bar x\cup \bar y)=\rho_2\circ\be_2(\bar x\cup \bar y)=\Sq^1(\bar x\cup \bar y).
\end{align*}
Since we have assumed that the operation \eq{fm15eq15} vanishes, the left hand side of the above equation vanishes which implies $\Sq^1(\bar x\cup \bar y)=0$ for every topological space $X$ and $\bar x,\bar y\in H^4(X,\Z_2)$. This, however, is false: for example, take $X=K(\Z_2,4)\t K(\Z_2,4)$ and $\bar x=\bar f_4'\t 1$, $\bar y=1\t\bar f_4'$. Then
\begin{equation*}
 \Sq^1(\bar x\cup \bar y)=\Sq^1(\bar f_4'\t\bar f_4')=\bar f_5\t \bar f_4'+\bar f_4'\t\bar f_5\neq 0 \in H^9(X,\Z_2).
\end{equation*}

\noindent
(c) By (a), $\be_4\cP(\bar f_4')=\be_2(\bar f_4'^2)/2$, so $d_2^8(\bar f_4'^2)=\rho_2(\be_4\cP(\bar f_4'))$ and this class is non-zero by (b).
\end{proof}

We thus find that the $\Z_4$-summand in $H^9(K(\Z_2,4),\Z)$ is generated by
\begin{equation*}
 f_9=\be_4\cP(\bar f_4'),
\end{equation*}
and this class satisfies $2 f_9=\be_2(\bar f_4'^2)$ and $\rho_2(f_9)=\bar f_9'+\bar f_4'\bar f_5$. This completes the calculation of $H^q(K(\Z_2,4),\Z)$ for $q\le 10$, recorded in \eq{fm15eq16}.

Dually, we obtain the $\Z$-homology of $K(\Z_2,4)$ from the homological Bockstein spectral sequence. This leads to $\Z_2$-summands in $H_*(K(\Z_2,4),\Z)$ generated by $\varphi_4', \varphi_6', \varphi_7', \varphi_9', \varphi_{10}', \varphi_4'\varphi_6'$ and a class $\ti\varphi_8\in H_8(K(\Z_2,4),\Z)$ that generates a $\Z_4$-summand such that $\rho_2(\ti\varphi_8)=(\bar\varphi_4')^2$. This is recorded in~\eq{fm15eq18}.
\ea
&\begin{tabular}{P{2.5cm}|ccccccc}
 $n$ & $5$ & $7$ & $8$ & $9$ & $10$ & 0--3,4,6 \\
 \hline
 $\tilde H^n(K(\Z_2,4),\Z)$\!\!\! & $\!\!\Z_2\an{f_5}\!\!$ & $\!\!\Z_2\an{f_7}\!\!$ & $\!\!\Z_2\an{f_8}\!\!$ & $\!\!\Z_4\an{f_9}\!\!$ & $\!\!\Z_2\an{f_5^2}\!\!$ & $0$ \\
\end{tabular}\!\!\!
\label{fm15eq16}
\allowdisplaybreaks\\
&\begin{tabular}{P{2.5cm}|ccccccccc}
 $n$ & $4$ & $5$ & $6$ & $7$ & $8$\\
  \hline
 $\tilde H^n(K(\Z_2,4),\Z_2)$\!\!\! & $\!\!\Z_2\an{\bar f_4'}\!\!$ & $\!\!\Z_2\an{\bar f_5}\!\!$ & $\!\!\Z_2\an{\bar f_6'}\!\!$ & $\!\!\Z_2\an{\bar f_7,\bar f_7'}\!\!$ &  $\!\!\Z_2\an{\bar f_8,(\bar f_4')^2}\!\!$
\end{tabular}
\nonumber\\
&\begin{tabular}{P{2.5cm}|ccc}
 $\cdots$ & $9$ & $10$ & 0--3 \\
  \hline
  & $\!\!\Z_2\an{\bar f_9', \bar f_4'\bar f_5}\!\!$ & $\!\!\Z_2\an{\bar f_{10}',\bar f_4'\bar f_6',\bar f_5^2}\!\!$ & $0$
\end{tabular}
\label{fm15eq17}
\allowdisplaybreaks\\
&\begin{tabular}{P{2.5cm}|cccccc}
 $n$ & $4$ & $6$ & $7$ & $8$ & $9$ & $10$\\
  \hline
 $\tilde H_n(K(\Z_2,4),\Z)$\!\!\! & $\!\!\Z_2\an{\varphi_4'}\!\!$ & $\!\!\Z_2\an{\varphi_6'}\!\!$ & $\!\!\Z_2\an{\varphi_7'}\!\!$ & $\!\!\Z_4\an{\ti\varphi_8}\!\!$ & $\!\!\Z_2\an{\varphi_9'}\!\!$ & $\!\!\Z_2\an{\varphi_{10}',\varphi_4'\varphi_6'}\!\!$  \\
\end{tabular}
\nonumber\\
&\begin{tabular}{P{2.5cm}|c}
 $\cdots$ & 0--3,5 \\
  \hline
  & $0$ \\
\end{tabular}
\label{fm15eq18}
\allowdisplaybreaks\\
&\begin{tabular}{P{2.5cm}|ccccc}
 $n$ & $4$ & $5$ & $6$ & $7$ & $8$\\
  \hline
 $\tilde H_n(K(\Z_2,4),\Z_2)$\!\!\!\! & $\!\!\Z_2\an{\bar \varphi_4'}\!\!$ & $\!\!\Z_2\an{\bar \varphi_5}\!\!$ & $\!\!\Z_2\an{\bar \varphi_6'}\!\!$ & $\!\!\Z_2\an{\bar \varphi_7,\bar \varphi_7'}\!\!$ &  $\!\!\Z_2\an{\bar \varphi_8,(\bar \varphi_4')^2}\!\!$
\end{tabular}
\nonumber\\
&\begin{tabular}{P{2.5cm}|ccc}
$\cdots$ & $9$ & $10$ & 0--3 \\
\hline
& $\!\!\Z_2\an{\bar \varphi_9', \bar \varphi_4'\bar \varphi_5}\!\!$ & $\!\!\Z_2\an{\bar \varphi_{10}',\bar \varphi_4'\bar \varphi_6',\bar \varphi_5^2}\!\!$ & $0$
\end{tabular}
\label{fm15eq19}
\ea

\subsubsection{Computation of the spectral sequence}
\label{fm1572}

Consider the Atiyah--Hirzebruch spectral sequence
\begin{equation*}
 \ti H_p(K(\Z_2,4),\Om^{\bs\Spin}_q(*))\Longra\ti\Om^{\bs\Spin}_{p+q}(K(\Z_2,4)).
\end{equation*}
The $E^2$-page of the spectral sequence is given by \eq{fm15eq18} and \eq{fm15eq19}. By Proposition \ref{fm2prop1}, the differentials on the $E^2$-page can be deduced from
\begin{align*}
 \Sq^2(\bar f_4')&=\bar f_6',
 &\Sq^2(\bar f_5)&=\bar f_7',
&\Sq^2(\bar f_6')&=\bar f_8,\\
 \Sq^2(\bar f_7)&=\bar f_9'
&\Sq^2(\bar f_7')&=0,
\end{align*}
which results in the $E^3$-page of the spectral sequence shown in Figure \ref{fm15fig12} along with all possibly non-zero higher differentials.

\begin{figure}[htb]
\centering
\begin{tikzpicture}
  \matrix (m) [matrix of math nodes,
    nodes in empty cells,nodes={minimum width=10ex,
    minimum height=5ex,outer sep=-5pt},
    column sep=1ex,row sep=1ex]{
	4 & \Z_2\an{\al_4\varphi_4'}\\
	2 & && \Z_2\an{\al_1^2\bar\varphi_7'}\\
	1 & &&& \Z_2\an{\al_1\bar\varphi_4'^2}\\
	0 & \Z_2\an{\varphi_4'} && & \Z_4\an{\ti\varphi_8} & \Z_2\an{\varphi_{10}',\varphi_4'\varphi_6'}  \\
    \quad\strut  & 4 && 7 &  8 & 10  \strut \\};
 \draw[thick] (m-1-1.north east) node[above]{$q$} -- (m-5-1.east);
 \draw[thick] (m-5-1.north) -- (m-5-6.north east) node[right]{$p$};
 \draw[-stealth,bend right=15] (m-4-6) to node[above]{\scriptsize $d^3_{10,0}$} (m-2-4);
 \draw[-stealth] (m-2-4) -- node[above]{\scriptsize $d^3_{7,2}$} (m-1-2);
 \draw[-stealth,bend left=10] (m-3-5) to node[below]{\scriptsize $d^4_{8,1}$} (m-1-2);

\end{tikzpicture}
\caption{$E^3$-page of $\ti H_p(K(\Z_2,4),\Om^{\bs\Spin}_q(*))\Ra\ti\Om^{\bs\Spin}_{p+q}(K(\Z_2,4))$, $p\!+\!q\!\le\! 10$. We will show that $d_{7,2}^3\neq 0$ and $d_{10,0}^3=0$.}
\label{fm15fig12}
\end{figure}

We claim that $d_{10,0}^3=0$, $d_{8,1}^4=0$, and $d_{7,2}^3\neq 0$. For this we will compare with the Atiyah--Hirzebruch spectral sequence of $K(\Z,4)$ and use naturality of the differentials and of the filtrations. Consider the map $\psi:K(\Z,4)\ra K(\Z_2,4)$, unique up to homotopy, with $\psi^*(\bar f_4')=\bar e_4$. Then $\psi$ maps the $\Z$-homology by
\e
\begin{aligned}
 \ep_4&\longmapsto\varphi_4',
&\ep_6'&\longmapsto\varphi_6',
&\ep_4^2&\longmapsto\pm\ti\varphi_8,\\
\ep_8&\longmapsto 0,
&\ep_{10}'&\longmapsto\varphi_{10}',
&\ep_4\ep_6'&\longmapsto\varphi_4'\varphi_6',
\end{aligned}
\label{fm15eq20}
\e
and the $\Z_2$-homology by
\e
\begin{aligned}
 \bar\ep_4&\longmapsto\bar\varphi_4',
&\bar\ep_6'&\longmapsto\bar\varphi_6',
&\bar\ep_7&\longmapsto\bar\varphi_7,\\
 \bar\ep_4^2&\longmapsto\bar\varphi_4'^2,
&\bar\ep_{10}'&\longmapsto\bar\varphi_{10}'
&\bar\ep_4\bar\ep_6'&\longmapsto\bar\varphi_4'\bar\varphi_6'.
\end{aligned}
\label{fm15eq21}
\e

Now \eq{fm15eq4} maps both classes to $\al_1^2\bar\ep_7$, $\psi$ maps $\bar\ep_7\mapsto\bar\varphi_7$, and $\al_1^2\bar\varphi_7$ vanishes in the quotient group $E^3_{7,2}$ as in Figure \ref{fm15fig12}. Since $\psi$ maps $\ep_4\ep_6'\mapsto \varphi_4'\varphi_6'$ and $\ep_{10}'\mapsto \varphi_{10}'$ we find $d_{10,0}^3=0$ in Figure \ref{fm15fig12}. Moreover, we have $d_{8,1}^4=0$ for the spectral sequence of $K(\Z,4)$ and since $\psi$ maps $\bar\ep_4^2$ to $\bar\varphi_4'^2$ this implies $d_{8,1}^4=0$.

Finally, we claim $d_{7,2}^3\neq 0$. For this we compare the filtrations of the spectral sequences for $K(\Z,4)$, see \eq{fm15eq5}, and $K(\Z_2,4)$. Assume by contradiction that $d_{7,2}^3=0$, so $E^\iy_{4,4}\cong\Z_2$ for the spectral sequence of $K(\Z_2,4)$. Then $\psi$ would induces a map of filtrations
\begin{equation*}
\begin{tikzcd}[column sep=small]
0\arrow[r,symbol=\subset,"\Z\an{\al_4\ep_4}" yshift=1.5ex] & F_{4,8}\cong\Z\dar{\psi_*}\arrow[r,symbol=\subset,"\Z_2\an{\al_1^2\bar\ep_6'}" yshift=1.5ex] & F_{6,8}\cong\Z\arrow[r,symbol=\subset,"\Z_2\an{\al_1\bar\ep_7}" yshift=1.5ex] & F_{7,8}\cong\Z\dar{\psi_*}\arrow[r,symbol=\subset,"\Z\an{\ep_4^2}\op\Z_3\an{\ep_8}" yshift=1.5ex] & F_{8,8}=\ti\Om_8^{\bs\Spin}(K(\Z,4))\cong \Z^2\dar{\psi_*}\\
0\arrow[r,symbol=\subset,"\Z_2\an{\al_4\varphi_4'}"' yshift=-1.5ex] & F_{4,8}'\cong\Z_2\arrow[rr,"="] & & F_{7,8}'\cong\Z_2\arrow[r,symbol=\subset,"\Z_4\an{\ti\varphi_8}"' yshift=-1.5ex] & F_{8,8}'=\ti\Om_8^{\bs\Spin}(K(\Z_2,4)).
\end{tikzcd}
\end{equation*}
As $\psi$ maps $\ep_4\mapsto\varphi_4'$, the far left vertical map is non-zero. On the other hand, the generator of $F_{4,8}$ becomes divisible by $4$ in $F_{7,8}$ while $F_{4,8}'=F_{7,8}'$, so $\psi$ maps the generator to zero, a contradiction. Hence $d_{7,2}^3\neq 0$. We thus obtain the $E^\iy$-page shown in Figure \ref{fm15fig13} in the range $p+q\le 9$. All extension problems are trivial.

\begin{figure}[htb]
\centering
\begin{tikzpicture}
  \matrix (m) [matrix of math nodes,
    nodes in empty cells,nodes={minimum width=11.8ex,
    minimum height=5ex,outer sep=-5pt},
    column sep=1ex,row sep=1ex]{
	1 & & \Z_2\an{\al_1\bar\varphi_4'^2}\\
	0 & \Z_2\an{\varphi_4'} & \Z_4\an{\ti\varphi_8}\\
    \quad\strut  & 4  &  8 \strut \\};
 \draw[thick] (m-1-1.north east) node[above]{$q$} -- (m-3-1.east);
 \draw[thick] (m-3-1.north) -- (m-3-3.north east) node[right]{$p$};

\end{tikzpicture}
\caption{$E^\iy$-page of $\ti H_p(K(\Z_2,4),\Om^{\bs\Spin}_q(*))\Ra\ti\Om^{\bs\Spin}_{p+q}(K(\Z_2,4))$, $p\!+\!q\!\le\! 9$.}
\label{fm15fig13}
\end{figure}

Moreover, by comparing Figures \ref{fm15fig3} and \ref{fm15fig13} and using \eq{fm15eq20}, \eq{fm15eq21} we see that the map $\psi_*:\ti\Om_n^{\bs\Spin}(K(\Z,4))\ra\ti\Om_n^{\bs\Spin}(K(\Z_2,4))$ is surjective for all $n\le 9$. In particular, $\de$ and $\al_1\ze_2$ generate $\ti\Om_n^{\bs\Spin}(K(\Z_2,4))$ for $n=3,9$. We know that $\ti\Om_8^{\bs\Spin}(K(\Z_2,4))\cong\Z_4$, where the isomorphism is given by the Pontrjagin square, $[X,\bar\al]\mapsto \int_X\cP(\bar\al)$. As in Definition \ref{fm2def4}(c), if $\bar\al$ can be lifted to integral cohomology class, then $\int_X\cP(\bar\al)=\int_X\al\cup\al\bmod{4}$. Now \eq{fm3eq19} maps $\ze_2\mapsto(1,0)$ and $\ze_3\mapsto(0,1)$, hence $\int_X\cP(\bar\al)=1$ for $\ze_2$ and $\int_X\cP(\bar\al)=0$ for $\ze_3$. This proves that $\ti\Om_8^{\bs\Spin}(K(\Z_2,4))=\Z_4\an{\ze_2}$ and proves Table \ref{fm3tab1} for $K(\Z_2,4)$.

\subsection{Proof of Theorem \ref{fm3thm1}(b)}
\label{fm158}

The generators in Table \ref{fm3tab1} have already been verified in \S\ref{fm151}--\S\ref{fm157}, so it only remains to prove \eq{fm3eq7}.

Using \eq{fm3eq13}, one finds that
$\ti\Om_8^{\bs\Spin}(M\{1\})\ra\ti\Om_8^{\bs\Spin}(M\U(2))$ maps $\ze_1\mapsto 2\frac{\ze_1}{2}$.  Similarly, \eq{fm3eq17} shows that $\ti\Om_8^{\bs\Spin}(M\U(2))\ra\ti\Om_8^{\bs\Spin}(M\SO(4))$ maps $\frac{\ze_1}{2}\mapsto 2\frac{\ze_1}{4}$.

We prove $\frac{\ze_1}{4}\mapsto 0$ in $\ti\Om_8^{\bs\Spin}(K(\Z,4))$ using the isomorphism \eq{fm3eq19}, whose components we can rewrite using the fact that $\al$ is Poincar\'e dual to $M$ as
\begin{align*}
 \int_X\al\cup\al&=\int_M e(\nu_M),\\
 \int_X\al\cup\left[\frac{p_1(TX)}{4}+\frac{\al}{2}\right]&=\int_M \frac{p_1(TX)|_M}{4}+\frac12\int_M e(\nu_M)\\
 &=\int_M \frac{p_1(TM)+p_1(\nu_M)}{4}+\frac12\int_M e(\nu_M).
\end{align*}
Now \eq{fm3eq17} maps $\frac{\ze_1}{4}\mapsto(1,0,0)$, which implies $\int_M e(\nu_M)=0$, $\int_M p_1(\nu_M)=0$, so the isomorphism \eq{fm3eq19} maps $\frac{\ze_1}{4}$ to zero, as claimed.

We have already shown $\eta\mapsto\al_1\ze_2$ in $\ti\Om_9^{\bs\Spin}(K(\Z,4))$ in \eq{fm15eq12}.

We prove $\ze_2'\mapsto\frac{\ze_1}{4}+\ze_2+4\ze_3$ in $\ti\Om_8^{\bs\Spin}(M\SO(4))$. From $\int_M c_2(\Si_\nu^+)=0$ and $\int_M c_2(\Si_\nu^-)=-1$ we find $\int_M e(\nu_M)=\int_M c_2(\Si_\nu^+)-c_2(\Si_\nu^-)=1$ and $\int_M p_1(\nu_M)=-2\int_M c_2(\Si_\nu^+)-2\int_M c_2(\Si_\nu^-)=2$, so \eq{fm3eq17} maps $\ze_2'\mapsto(1,1,4)$.

We prove $\ze_3\!\mapsto\! 0$ in $\ti\Om_8^{\bs\Spin}(K(\Z_2,4))$ using the isomorphism $[X,\bar\al]\!\mapsto\!\int_X\cP(\bar\al)$ from \eq{fm3eq21}. By Definition \ref{fm2def4}(c), if $\bar\al$ can be lifted to integral cohomology class, then $\int_X\cP(\bar\al)=\int_X \al\cup\al\bmod{4}$. Now \eq{fm3eq19} maps $\ze_3\mapsto(0,1)$, hence $\int_X\cP(\bar a)=0$ for $\ze_3$.

\subsection{Proof of Theorem \ref{fm3thm1}(c)}
\label{fm159}

We prove $\al_1\ze_1=0$ in $\ti\Om_9^{\bs\Spin}(M\{1\})$. As $M\{1\}\cong\cS^4$, the suspension isomorphism gives $\ti\Om_9^{\bs\Spin}(M\{1\})\cong\ti\Om_9^{\bs\Spin}(\cS^4)\cong\Om_5^{\bs\Spin}(*)=0$ by Table \ref{fm2tab2}, hence $\al_1\ze_1=0$ in $\ti\Om_9^{\bs\Spin}(M\{1\})$.

We prove $\al_1\frac{\ze_1}{2}=0$ in $\ti\Om_9^{\bs\Spin}(M\U(2))$. The upper horizontal map in the commutative diagram
\begin{equation*}
\begin{tikzcd}
\ti\Om_8^{\bs\Spin}(M\U(2))\rar\dar{\al_1} & \ti\Om_8^{\bs\Spin}(M\SO(4))\dar{\al_1}\\
\ti\Om_9^{\bs\Spin}(M\U(2))\rar & \ti\Om_9^{\bs\Spin}(M\SO(4))
\end{tikzcd}
\end{equation*}
maps $\frac{\ze_1}{2}\mapsto2\frac{\ze_1}{4}$. Therefore, the image of $\al_1\frac{\ze_1}{2}$ in $\ti\Om_9^{\bs\Spin}(M\SO(4))$ is $\al_12\frac{\ze_1}{4}=(2\al_1)\frac{\ze_1}{4}=0$, since $\al_1$ has order two.  As the bottom horizontal map is injective by Table \ref{fm3tab1}, this proves $\al_1\frac{\ze_1}{2}=0$ in $\ti\Om_9^{\bs\Spin}(M\U(2))$.

We prove $\al_1\ze_3=0$ in $\ti\Om_9^{\bs\Spin}(M\U(2))$. According to \eq{fm3eq7} the upper horizontal map in the commutative diagram
\begin{equation*}
\begin{tikzcd}
\ti\Om_8^{\bs\Spin}(M\U(2))\rar\dar{\al_1} & \ti\Om_8^{\bs\Spin}(K(\Z_2,4))\dar{\al_1}\\
\ti\Om_9^{\bs\Spin}(M\U(2))\rar & \ti\Om_9^{\bs\Spin}(K(\Z_2,4))
\end{tikzcd}
\end{equation*}
maps $\ze_3\mapsto 0$, so $\al_1\ze_3=0$ in $\ti\Om_9^{\bs\Spin}(K(\Z_2,4))$. Since the lower horizontal map is an isomorphism by Table \ref{fm3tab1}, we conclude that $\al_1\ze_3=0$ in~$\ti\Om_9^{\bs\Spin}(M\U(2))$.

\section{Proof of Theorem \ref{fm3thm3}}
\label{fm16}

\subsection{Construction of maps}
\label{fm161}

The construction of the maps appearing in Theorem \ref{fm3thm3} will use the K-theory groups $K_\F(X;A)$ over the fields $\F=\R$, $\C$, or $\H$, which we therefore briefly review. These were originally constructed by Bott \cite{Bott}. As in \cite[\S I.9]{LaMi}, for a well-behaved subspace $A\subset X$ of a topological space, a class in $K_\F(X;A)$ can be represented by a chain complexes of $\F$-vector bundles
\begin{equation}
 \xi:
 \begin{tikzcd}
 	V^0\rar{\d} & V^1\rar{\d} & \cdots\rar{\d} & V^n,
 \end{tikzcd}
 \qquad
 \d^2=0,
\label{fm16eq1}
\end{equation}
over $X$ such that the restriction $\xi|_A$ is exact. The precise conditions under which two chain complexes represent the same element $[\xi]$ in $K_\F(X;A)$ are a little intricate, see \cite[\S I.9]{LaMi}. We only note the following special cases:
\begin{itemize}
\item If \eq{fm16eq1} is exact over the entire space $X$, then $[\xi]=0$.
\item If all differentials vanish, then $[\xi]=\sum_{i=0}^n (-1)^i[V^i]$ in $K_\F(X;A)$, where $[V^i]$ denotes the $K$-theory class defined by the complex of length zero.
\end{itemize}

Depending on the field, set
\begin{equation*}
 G_\F=
 \begin{cases}
 	\O=\colim\O(n) & \text{if $\F=\R$},\\
 	\U=\colim\U(n) & \text{if $\F=\C$},\\
 	\Sp=\colim\Sp(n) & \text{if $\F=\H$}.
 \end{cases}
\end{equation*}
These are topological groups with classifying spaces $BG_\F$ and they represent the $K_\F$-theory groups in terms of sets of homotopy classes of pointed maps:
\begin{equation*}
 K_\F(X;A)\overset{\cong}{\longra}[X/A,BG_\F\t\Z]^\circ,\qquad [\xi]\longmapsto [(f_\xi,\operatorname{vdim}_\xi)]
\end{equation*}
In particular, every complex $\xi$ as in \eq{fm16eq1} has a {\it classifying map}
\begin{equation*}
 f_\xi:X/A\longra BG_\F,
\end{equation*}
which is well-defined up to homotopy.

We will apply this setup to Thom spaces. Recall from \S\ref{fm25} that a representation $\rho:H\ra\SO(k)$ defines a Thom space $MH=D_\rho/S_\rho$, where $S_\rho\subset D_\rho$ are the unit sphere and unit disk subbundles of the tautological rank $k$ vector bundle $\pi:E_\rho\ra BH$; we write $\pi$ also for the projection $D_\rho\ra BH$ of the unit disk bundle. As above, a chain complex of $\F$-vector bundles $\xi^H$ over $D_\rho$ such that the restriction $\xi|_{S_\rho}$ is exact has a classifying map
\e
 f_\xi^H: MH\longra BG_\F.
\label{fm16eq2}
\e

The maps in Theorem \ref{fm3thm3}(a)--(c) will be classifying maps of certain chain complexes $\xi^H$ of vector bundles over $MH$. In \eq{fm16eq3}, \eq{fm16eq5}, \eq{fm16eq7} below, we will write the pullback of a vector bundle as a set of pairs in which the first entry corresponds to the base and the second entry to the fibre.

\smallskip
\noindent
{\bf(a)}
Since $S_\rho\cong E\Sp(1)$ is contractible, we have $M\Sp(1)=D_\rho/S_\rho\simeq D_\rho\simeq B\Sp(1)$, which proves Theorem \ref{fm3thm3}(a). As this proof does not generalize to parts (b), (c) and does not show formula \eq{fm3eq24}, we give another proof.

If $\rho:\Sp(1)\ra\SO(4)$ is the inclusion, then $\pi:E_\rho\ra B\Sp(1)$ is a quaternionic line bundle. Let $\ul\H$ be the trivial quaternionic line bundle. Define a complex $\xi^{\Sp(1)}$ of vector bundles over $D_\rho$ by
\begin{equation}
\begin{tikzcd}[row sep=0]
 {[-1]} & {[0]}\\
 \pi^*(\ul\H)\rar & \pi^*(E_\rho),\\
 (x,q)\arrow[r,mapsto] & (x,qx),
\end{tikzcd}
\label{fm16eq3}
\end{equation}
using the scalar multiplication of $\H$ on $E_\rho$. As indicated, $\pi^*(\ul\H)$ is placed in degree $-1$ and $\pi^*(E_\rho)$ in degree $0$. Since \eq{fm16eq3} is an isomorphism over points with $x\neq 0$, the restriction $\xi^{\Sp(1)}|_{S_\rho}$ is exact. We obtain a classifying map
\e
 f_\xi^{\Sp(1)}:M\Sp(1)\longra B\Sp(1)\subset B\Sp.
 \label{fm16eq4}
\e
As indicated, the map has image in $B\Sp(1)=B\SU(2)$ since $f_\xi^{\Sp(1)}$ is just the classifying map of the quaternionic {\it line} bundle $\pi^*(E_\rho)\ra D_\rho$ with the framing \eq{fm16eq3}. We prove Theorem \ref{fm3thm3}(a) in \S\ref{fm1621} below.
\smallskip

\noindent{\bf(b)} If $\rho:\U(2)\ra\SO(4)$ is the inclusion, then $\pi:E_\rho\ra B\U(2)$ is a complex vector bundle of rank $2$. Define $\xi^{\U(2)}$ to be the chain complex
\begin{equation}
\begin{tikzcd}[row sep=0]
	{[-1]} & {[0]} & {[1]}\\
	\pi^*\La^0(E_\rho)\rar &
	\pi^*\La^1(E_\rho)\rar &
	\pi^*\La^2(E_\rho),\\
	(x,\la)\arrow[r,mapsto] & (x,\la x),\\
	& (x,\al)\arrow[r,mapsto] & (x,x\wedge\al).
\end{tikzcd}	
\label{fm16eq5}
\end{equation}
The degrees are as indicated, with $\pi^*\La^1(E_\rho)\cong E_\rho$ in degree zero. It is easy to see that $\xi^{\U(2)}|_{S_\rho}$ is exact, so we obtain a classifying map
\e
 f_\xi^{\U(2)}:M\U(2)\longra B\SU\subset B\U.
 \label{fm16eq6}
\e
As indicated, this map factors over $B\SU$ since $\bigl(f_\xi^{\U(2)}\bigr)^*(c_1)=0$. This fact and Theorem \ref{fm3thm3}(b) are proved in \S\ref{fm1622} below.

Moreover, if we regard a quaternionic line bundle as a vector bundle with structure group $\SU(2)\subset\U(2)$, then the complexes \eq{fm16eq3} and \eq{fm16eq5} are quasi-isomorphic and therefore have homotopic classifying maps.
\smallskip

\noindent{\bf(c)} If $\rho:\Spin(4)\ra\SO(4)$ is the double cover, then $E_\rho\ra B\Spin(4)$ is a real vector bundle with spin structure. By the usual description of spin structures for oriented rank $4$ vector bundles, this determines a pair of quaternionic line bundles $\Si^\pm_\rho\ra B\Spin(4)$ and an isomorphism $c:E_\rho\ra\Hom_\H(\Si^-_\rho,\Si^+_\rho)$, the {\it Clifford multiplication}. Define $\xi^{\Spin(4)}$ to be the chain complex
\begin{equation}
\begin{tikzcd}[row sep=0]
   {[-1]} & {[0]}\\
   \pi^*(\Si^-_\rho)\rar & \pi^*(\Si^+_\rho),\\
   (x,\psi)\arrow[r,mapsto] & (x,c_x(\psi)),
\end{tikzcd}	
\label{fm16eq7}
\end{equation}
of $\H$-vector bundles over $D_\rho$. Since $c_x:\Si^-_\rho\ra\Si^+_\rho$ is an isomorphism when $x\neq 0$, we find that $\xi^{\Spin(4)}|_{S_\rho}$ is exact. We thus obtain a classifying map
\e
 f_\xi^{\Spin(4)}:M\Spin(4)\longra B\Sp.
 \label{fm16eq8}
\e
If we view a quaternionic line bundle as a real vector bundle with spin structure given by the spinor bundles $\Si^+_\rho=E_\rho$, $\Si^-_\rho=\ul\H$ and Clifford multiplication \eq{fm16eq3}, then the complexes \eq{fm16eq3} and \eq{fm16eq7} coincide.

We prove Theorem \ref{fm3thm3}(c) in \S\ref{fm1623} below.
\smallskip

\noindent{\bf(d)} In this case, the map is constructed differently: Bott--Samelson \cite{BoSa} show that $\pi_3(E_8)\cong\Z$ and $\pi_n(E_8)=0$ for all $3\neq n\le 15$, so $\pi_4(BE_8)\cong\pi_3(E_8)\cong\Z$ and $\pi_n(BE_8)\cong\pi_{n-1}(E_8)=0$ for all $4\neq n\le 16$. By the Hurewicz Theorem, $H^4(BE_8,\Z)\cong\Hom(\pi_4(BE_8),\Z)\cong\Z$ and the generator of this group is a degree four cohomology class, which is classified by a map (see \S\ref{fm23})
\e
 f: BE_8\longra K(\Z,4).
\label{fm16eq9}
\e

\subsection{Proofs of Theorem \ref{fm3thm3}(a)--(d)}
\label{fm162}

The proof is based on the following theorem, see Thom~\cite[Th.~II.6]{Thom}.

\begin{thm}[Whitehead]
\label{fm16thm1}
Let\/ $f\colon X\ra Y$ be a cellular map between simply-connected CW complexes. Suppose that for all coefficients\/ $\Z_p$ with\/ $p\ge 0$, the map\/ $f^*\colon H^*(Y,\Z_p)\ra H^*(X,\Z_p)$ induced by\/ $f$ is an isomorphism if\/ $*<n$ and a monomorphism if\/ $*=n.$ Then\/ $f$ is an\/ $n$-connected map.
\end{thm}


In order to apply the Whitehead theorem to a classifying map \eq{fm16eq2}, we need to compute the pullbacks $\bigl(f_\xi^H\bigr)^*(a)$ of cohomology class $a\in H^*(BG_\F,R)$, where $R=\Z_p$ is a commutative ring. We can use $a$ to define a characteristic class
\begin{equation*}
 K_\F(X;A)\longra H^*(X;A,R),\qquad \xi\longmapsto a(\xi)\coloneqq \bigl(f_\xi^H\bigr)^*(a).
\end{equation*}
Let $s:BH\ra MH$ be the composition of the zero section of $D_\rho$ with the projection onto $MH$. By the Thom isomorphism \eq{fm2eq23}, we can write
\e
\label{fm16eq10}
 (f_\xi^H\bigr)^*(a)=a(\xi)=\pi^*(b)\cup t
\e
for some class $b\in H^*(BH,R)$. By applying $s^*$ to \eq{fm16eq10}, we see that $b$ satisfies
\e
\label{fm16eq11}
 a(s^*(\xi))=b\cup e
\e
for the Euler class $e=s^*(t)$ of the vector bundle $E_\rho$. In each of the cases \eq{fm16eq3}, \eq{fm16eq5}, and \eq{fm16eq7}, the pullback $s^*(\xi)$ has zero differentials, hence $s^*(\xi)=\sum_{i=0}^n (-1)^i [s^*(V^i)]$ in $K_\F(BH)$. This will allow us to compute $a(s^*(\xi))$ using a Whitney sum formula and we can then read off the value of $b$ in \eq{fm16eq11} (there will be a unique such $b$ in each case) and hence determine \eq{fm16eq10}.

\subsubsection{Proof of Theorem \ref{fm3thm3}(a)}
\label{fm1621}

To prove that \eq{fm16eq4} is a homotopy equivalence, we will show that
\e
\label{fm16eq12}
 \bigl(f_\xi^{\Sp(1)}\bigr)^*: H^*(B\SU(2),\Z)\longra H^*(M\Sp(1),\Z)
\e
is an isomorphism in all degrees. Since the cohomology groups on either side are torsion-free, this implies isomorphisms in $\Z_p$-cohomology, and Theorem \ref{fm16thm1} then shows that $f_\xi^{\Sp(1)}$ is a homotopy equivalence. We have $H^*(B\SU(2),\Z)=\Z[c_2]$. By the Thom isomorphism, we (additively) have $H^*(M\Sp(1),\Z)=\pi^*\Z[c_2(E_\rho)]t$.
As for all complex vector bundles, the Euler class $e=s^*(t)=c_2(E_\rho)$ is the top degree Chern class. We calculate the pullback $\bigl(f_\xi^{\Sp(1)}\bigr)^*(c_2^k)$ of $a=c_2^k$ using \eq{fm16eq10} and \eq{fm16eq11}. As $s^*(\xi)=[E_\rho]-[\ul \H]$, we have
\begin{equation*}
 a(s^*(\xi))=c_2([E_\rho]-[\ul \H])^k=c_2(E_\rho)^k=c_2(E_\rho)^{k-1}\cup e,
\end{equation*}
hence
\begin{equation*}
 \bigl(f_\xi^{\Sp(1)}\bigr)^*(c_2^k)=\pi^*c_2(E_\rho)^{k-1}\cup t\qquad\text{for all $k\ge 1$.}
\end{equation*}
In particular, \eq{fm16eq12} is an isomorphism, and the above equations also prove \eq{fm3eq24}.

\subsubsection{Proof of Theorem \ref{fm3thm3}(b)}
\label{fm1622}

To prove that \eq{fm16eq6} is $10$-connected, we will show that
\e
\label{fm16eq13}
 \bigl(f_\xi^{\U(2)}\bigr)^*: H^*(B\SU,\Z)\longra H^*(M\U(2),\Z)
\e
is an isomorphism in all degrees $\le 9$ and a monomorphism in degree $10$. Since the cohomology groups on either side are torsion-free, this implies that $f_\xi^{\U(2)}$ induces also an isomorphisms in this range with $\Z_p$-coefficients, so Theorem \ref{fm16thm1} completes the proof that $f_\xi^{\U(2)}$ is $10$-connected.

Using the Thom isomorphism \eq{fm2eq23}, we can describe the cohomology groups on either side of \eq{fm16eq13} in low dimensions as follows:
\begin{center}
\begin{tabular}{c|ll}
  & $H^*(B\SU,\Z)$		& $H^*(M\U(2),\Z)$\\ \hline
4 & $c_2$				& $t$\\
6 & $c_3$				& $\pi^*[c_1(E_\rho)]t$\\
8 & $c_4$, $c_2^2$		& $\pi^*[c_1(E_\rho)^2]t$, $\pi^*[c_2(E_\rho)]t$\\
10& $c_5$, $c_2c_3$		& $\pi^*[c_3(E_\rho)]t$, $\pi^*[c_1(E_\rho)c_2(E_\rho)]t$, $\pi^*[c_1(E_\rho)^3]t$\\
\end{tabular}
\end{center}

To calculate \eq{fm16eq13}, we use \eq{fm16eq10} and \eq{fm16eq11}. As $E_\rho$ is a complex vector bundle, $e=c_2(E_\rho)$ for its Euler class. Write $c=1+c_1+c_2+\ldots$ for the total Chern class. As $s^*(\xi)=-[\La^0(E_\rho)]+[\La^1(E_\rho)]-[\La^2(E_\rho)]$ in $K_\C^0(B\U(2))$, the Whitney sum formula implies
\begin{multline*}
 c(s^*(\xi))=\frac{c(E_\rho)}{c(\La^2(E_\rho))}\\
 =\frac{1+c_1(E_\rho)+c_2(E_\rho)}{1+c_1(E_\rho)}=
  1+\Bigl[
 1-c_1(E_\rho)+c_1(E_\rho)^2-c_1(E_\rho)^3+\ldots \Bigr]e.
\end{multline*}
Hence, comparing terms in each degree, we find
\begin{align*}
 \bigl(f_\xi^{\U(2)}\bigr)^*(c_1)&=0,\\
 \bigl(f_\xi^{\U(2)}\bigr)^*(c_k)&=(-1)^k \pi^*c_1(E_\rho)^{k-2}\cup t\qquad \text{for all $k\ge 2$.}
\end{align*}
In particular, $\bigl(f_\xi^{\U(2)}\bigr)^*(c_1)=0$, which was claimed above. Combined with \eq{fm2eq25} we find from this also
\begin{align*}
\bigl(f_\xi^{\U(2)}\bigr)^*(c_2^2)&=\bigl(f_\xi^{\U(2)}\bigr)^*(c_2)^2=t^2=\pi^*[c_2(E_\rho)]t,\\
\bigl(f_\xi^{\U(2)}\bigr)^*(c_2c_3)&=\bigl(f_\xi^{\U(2)}\bigr)^*(c_2)\bigl(f_\xi^{\U(2)}\bigr)^*(c_3)=-\pi^*[c_1(E_\rho)c_2(E_\rho)]t.
\end{align*}
Referring back to the table of the cohomology groups in low dimension, these formulas show that $\bigl(f_\xi^{\U(2)}\bigr)^*$ is an isomorphism in all dimension $\le 9$ and a monomorphism in dimension $10$. Therefore, the Whitehead Theorem completes the proof, and the above formulas also prove \eq{fm3eq26}.

\subsubsection{Proof of Theorem \ref{fm3thm3}(c)}
\label{fm1623}

To prove that \eq{fm16eq8} is $12$-connected, we will show that
\e
\label{fm16eq14}
 \bigl(f_\xi^{\Spin(4)}\bigr)^*: H^*(B\Sp,\Z)\longra H^*(M\Spin(4),\Z)
\e
is an isomorphism in all degrees $\le 12$. Since the cohomology groups on either side are torsion-free, this implies that $f_\xi^{\Spin(4)}$ induces also an isomorphisms in this range with $\Z_p$-coefficients, so Theorem \ref{fm16thm1} completes the proof that $f_\xi^{\Spin(4)}$ is $12$-connected.

Recall from Bruner--Catanzaro--May \cite[Ch.~4]{BCM} that $H^*(B\Sp,\Z)$ is a polynomial ring on the {\it symplectic Pontrjagin classes},
\begin{equation*}
 H^*(B\Sp,\Z)=\Z[q_1,q_2,\cdots],\qquad q_i=(-1)^ic_{2i}.
\end{equation*}
The cohomology of $M\Spin(4)$ was determined in \S\ref{fm1541}, namely
\begin{equation*}
 \ti H^*(M\Spin(4),\Z)\cong \pi^*H^*(B\Spin(4),\Z)\cup t=\Z[\pi^*(c_2(\Si^+_\rho)),\pi^*(c_2(\Si^-_\rho))]\cup t,
\end{equation*}
where $t\in \ti H^4(M\Spin(4),\Z)$ is the Thom class and $\Si^\pm_\rho$ are the spinor bundles on $B\Spin(4)$. In low degrees, the cohomology is therefore as follows.

\begin{center}
\begin{tabular}{c|ll}
  & $H^*(B\Sp,\Z)$ & $H^*(M\Spin(4),\Z)$\\ \hline
4 & $q_1$					& $t$\\
8 & $q_1^2$, $q_2$		& $\pi^*[c_2(\Si^+_\rho)]t$, $\pi^*[c_2(\Si^-_\rho)]t$\\
12 & $q_1^3$, $q_1q_2$, $q_3$		& $\pi^*[c_2(\Si^+_\rho)^2]t$, $\pi^*[c_2(\Si^-_\rho)^2]t$, $\pi^*[c_2(\Si^+_\rho)c_2(\Si^-_\rho)]t$\\
\end{tabular}
\end{center}

To calculate \eq{fm16eq14}, we use \eq{fm16eq10} and \eq{fm16eq11}. Since $W\cong \Hom_\H(\Si^-_\rho,\Si^+_\rho)$ we have $e=c_2(\Si^+_\rho)-c_2(\Si^-_\rho)$ for the Euler class. Write $q=1+q_1+q_2+\ldots$ for the total symplectic Pontrjagin class. As $s^*(\xi)=-[\Si^-_\rho]+[\Si^+_\rho]$ in quaternionic $K$-theory, the Whitney sum formula for symplectic Pontrjagin classes, see \cite[Th.~4.1(iii)]{BCM} implies
\begin{equation*}
 q(s^*(\xi))=\frac{1-c_2(\Si^+_\rho)}{1-c_2(\Si^-_\rho)}=1-\bigl[1+c_2(\Si^-_\rho)+c_2(\Si^-_\rho)^2+\ldots\bigr]\cup e,
\end{equation*}
so by \eq{fm16eq10} we find
\begin{equation*}
 \bigl(f_\xi^{\Spin(4)}\bigr)^*(q)=q(\xi)=1-\pi^*\bigl[1+c_2(\Si^-_\rho)+c_2(\Si^-_\rho)^2+\ldots\bigr]\cup t.
\end{equation*}
By comparing degrees, we find
\begin{equation*}
 \bigl(f_\xi^{\Spin(4)}\bigr)^*(q_k)=-\pi^*\bigl[c_2(\Si^-_\rho)^{k-1}\bigr]\cup t,\qquad \text{for all $k\ge 1,$}
\end{equation*}
which, combined with \eq{fm2eq25}, gives
\begin{align*}
 \bigl(f_\xi^{\Spin(4)}\bigr)^*(q_1^2)&=t^2=\pi^*(e)\cup t=\pi^*[c_2(\Si^+_\rho)-c_2(\Si^-_\rho)]\cup t,\\
 \bigl(f_\xi^{\Spin(4)}\bigr)^*(q_1^3)&=-\pi^*[c_2(\Si^+_\rho)^2-2c_2(\Si^+_\rho)c_2(\Si^-_\rho)+c_2(\Si^-_\rho)^2]\cup t,\\
 \bigl(f_\xi^{\Spin(4)}\bigr)^*(q_1q_2)&=\pi^*[c_2(\Si^+_\rho)c_2(\Si^-_\rho)-c_2(\Si^-_\rho)^2]\cup t.
\end{align*}
These formulas show that \eq{fm16eq14} is an isomorphism in all degrees $\le 12$, and also prove \eq{fm3eq28}.

\subsubsection{Proof of Theorem \ref{fm3thm3}(d)}
\label{fm1624}

By construction, $\pi_4(f):\pi_4(BE_8)\ra\pi_4(K(\Z,4))$ is an isomorphism. As $\pi_n(K(\Z,4))=0$ for all $4\neq n$, we conclude that the map \eq{fm16eq9} is $16$-connected, which proves Theorem \ref{fm3thm3}(d) and \eq{fm3eq30}.

\section{Proof of Theorem \ref{fm3thm5}}
\label{fm17}

The proof follows the same strategy outlined at the beginning of \S\ref{fm15}. Theorem \ref{fm3thm5}(a),(d) are proved in \S\ref{fm171}--\S\ref{fm175} and Theorem \ref{fm3thm5}(b),(c) in \S\ref{fm177}.

\subsection{\texorpdfstring{Computation of $\ti\Om_n^{\bs{\mathrm{Spin}}}(K(\Z,3))$}{Computation of Ωₙˢᵖⁱⁿ(K(ℤ,3))}}
\label{fm171}

\subsubsection{Description of the (co)homology}
\label{fm1711}

We continue to use Notation \ref{fm15nota1}. By Theorem \ref{fm2thm1}, $H^*(K(\Z,3),\Z_2)$, $*\le 9$, has generators $\bar d_3,\bar d_5',\bar d_9'$ satisfying
\e
\bar d_5'=\Sq^2(\bar d_3),\qquad \bar d_9'=\Sq^4\circ\Sq^2(\bar d_3).
\label{fm17eq1}
\e
Since $\Sq^1(\bar d_3)=0$ as $H^4(K(\Z,3),\ab\Z_2)=0$ and $\Sq^2\circ\Sq^2=\Sq^3\circ\Sq^1$ by \eq{fm2eq15},
\begin{equation*}
 \Sq^2(\bar d_5')=0.
\end{equation*}
The $\Z$-homology groups $H_{n+i}(K(\Z,n),\Z)$ are computed by Breen--Mikhailov--Touz\'e \cite[App.~B]{BMT} for $n\le 11$ and $i\le 10$. This leads to the following tables.
\ea
&\renewcommand{\arraystretch}{1.5}
\begin{tabular}{c|cccccccccccc}
$n$ & 
$0,1,2,4,6\!\!\!\!$ & $3$ & $5$  & $7$ & $8$ & $9$ & $10$ \\
\hline
$\ti H_n(K(\Z,3),\Z)$ & 
$0$ & $\!\!\Z\an{\de_3}\!\!\!$ &  $\!\!\Z_2\an{\de'_5}\!\!\!$ & $\!\!\Z_3\an{\de_7}\!\!\!$ & $\!\!\Z_2\an{\de_3\de_5'}\!\!\!$ & $\!\!\Z_2\an{\de'_9}\!\!\!$ & $\!\!\Z_3\an{\de_{10}}\!\!\!$
\end{tabular}
\nonumber\\
&\renewcommand{\arraystretch}{1.5}
\begin{tabular}{c|cccccccccccc}
$n$ & 
$0,1,2,4,5,7$ & $3$ & $6$ & $8$ & $9$   \\
\hline
$\ti H^n(K(\Z,3),\Z)$ & 
$0$ & $\!\Z\an{d_3}\!$ & $\!\Z_2\an{d_3^2}\!$ & $\!\Z_3\an{d_8}\!$ & $\!\Z_2\an{d_3^3}\!$  
\end{tabular}
\label{fm17eq2}\\
&\renewcommand{\arraystretch}{1.5}
\begin{tabular}{c|cccccccccccc}
$n$ & 
$\!\!0,1,2,4,7\!\!\!\!\!\!$ & $3$ & $5$ & $6$ & $8$ & $9$  \\
\hline
$\ti H_n(K(\Z,3),\Z_2)$ & 
$0$ & $\!\!\Z_2\an{\bar\de_3}\!\!$ &  $\!\!\Z_2\an{\bar\de'_5}\!\!$ & $\!\!\Z_2\an{\bar\de_3^2}\!\!$ & $\!\!\Z_2\an{\bar\de_3\bar\de_5'}\!\!$ & $\!\!\Z_2\an{\bar\de_3^3,\bar\de'_9}\!\!\!$ 
\end{tabular}
\nonumber\\
&\renewcommand{\arraystretch}{1.5}
\begin{tabular}{c|cccccccccccc}
$n$ & 
$\!\!0,1,2,4,7\!\!\!\!\!\!\!$ & $3$ & $5$ & $6$ & $8$ & $9$  \\
\hline
$\ti H^n(K(\Z,3),\Z_2)$ & 
$0$ & $\!\!\Z_2\an{\bar d_3}\!\!$ & $\!\!\Z_2\an{\bar d'_5}\!\!$ & $\!\!\!\Z_2\an{\bar d_3^2}\!\!$ & $\!\!\!\Z_2\an{\bar d_3\bar d'_5}\!\!$ & $\!\!\Z_2\an{\bar d_3^3,\bar d_9'}\!\!\!$
\end{tabular}
\nonumber
\ea

\subsubsection{Computation of the spectral sequence}
\label{fm1712}

The groups \eq{fm17eq2} determine the $E^2$-page of the Atiyah--Hirzebruch spectral sequence
\e
\label{fm17eq3}
 \ti H_p(K(\Z,3),\Om^{\bs\Spin}_q(*))\Longra\ti\Om^{\bs\Spin}_{p+q}(K(\Z,3))
\e
for $p+q\le 9$. Proposition \ref{fm2prop1} yields the differentials on the $E^2$-page, and this leads to the $E^3$-page shown in Figure \ref{fm17fig1}. The only possible higher differentials $d_{p,q}^r$ with $p+q\le 9$, $r\ge 3$ are $d^3_{8,0}$, $d^3_{8,1}$, and $d_{9,0}^3$. We will show that $d^3_{8,0}=0$ and that $d_{9,0}^3$ is an isomorphism. 

\begin{figure}[htb]
\centering
\begin{tikzpicture}
\matrix (m) [matrix of math nodes,
nodes in empty cells,nodes={minimum width=9ex,
minimum height=5ex,outer sep=-5pt},
column sep=1ex,row sep=1ex]{
4 & \Z\an{\al_4\de_3}  & \Z_2\an{\al_4\de_5'} \\
2 &   & \Z_2\an{\al_1^2\bar\de_5'} & \Z_2\an{\al_1^2\bar\de_3^2} & &\\
1 &&& \Z_2\an{\al_1\bar\de_3^2} && \Z_2\an{\al_1\bar\de_3\bar\de_5'}\\
0 & \Z\an{\de_3} &  &  & \Z_3\an{\de_7} & \Z_2\an{\de_3\de_5'} & \Z_2\an{\de_9'} \\
\quad\strut & 3 & 5 & 6 & 7 & 8 & 9  \strut \\};
\draw[-stealth,bend right=3,pos=0.4] (m-4-6) to  node[below]{\scriptsize $d^3_{8,0}$}  (m-2-3);

\draw[-stealth,bend right=10] (m-3-6) to  node[above]{\scriptsize $d^3_{8,1}$}  (m-1-3);

\draw[-stealth,bend left=5,shorten >= 1ex] (m-4-7) to  node[right,pos=0.25,xshift=1.5ex]{\scriptsize $d^3_{9,0}$}  (m-2-4);
\draw[thick] (m-1-1.north east) node[above]{$q$} -- (m-5-1.east);
\draw[thick] (m-5-1.north) -- (m-5-7.north east) node[right]{$p$};
\end{tikzpicture}
\caption{$E^3$-page of $\ti H_p(K(\Z,3),\Om^{\bs\Spin}_q(*))\Ra\ti\Om^{\bs\Spin}_{p+q}(K(\Z,3))$, $p+q\le 9$}
\label{fm17fig1}
\end{figure}

Let
\begin{equation*}
 \chi:\Si K(\Z,3)\longra K(\Z,4)
\end{equation*}
be the classifying map of the image of the primary class $d_3$ under the suspension isomorphism $\ti H^3(K(\Z,3),\Z)\cong \ti H^4(\Si K(\Z,3),\Z)$. We can compose the morphism of spectral sequence \eq{fm2eq10} with $\chi_*$ to get a morphism
\e
\ti\Om_n^{\bs\Spin}(K(\Z,3)) \xrightarrow{\cong}\ti\Om_{n+1}^{\bs\Spin}(\Si K(\Z,3)) \xrightarrow{\chi_*} \ti\Om_{n+1}^{\bs\Spin}(K(\Z,4))
\label{fm17eq4}
\e
and a morphism of spectral sequences. In particular, we get a morphism from Figure \ref{fm17fig1} to \ref{fm15fig2}. On the $E^2$-page this is just the suspension isomorphism in ordinary homology and thus maps $\de_3\mapsto\ep_4$ on the $E^2$-page of the spectral sequence. Dually on cohomology it maps $e_4\mapsto d_3$, and thus maps $\bar e_4\mapsto\bar d_3$ on mod 2 reductions. The induced action on $\Z_2$-cohomology preserves Steenrod squares, so by \eq{fm15eq1} and \eq{fm17eq1} it maps $\bar e_6'\mapsto\bar d_5'$. Dually, on $\Z_2$-homology it maps $\bar\de_5'\mapsto\bar\ep_6'$. Also, $\bar e_7\mapsto \bar d_3^2$ and $\bar e_{10}'\mapsto \bar d_9'$ in $\Z_2$-cohomology and hence $\bar\de_3^2\mapsto\bar\ep_7$ and $\bar \de_9'\mapsto\bar\ep_{10}'$ in $\Z_2$-homology. It maps $\de_3\de_5'\mapsto 0$ as the group in position $(9,0)$ in Figure \ref{fm15fig2} is zero. 

Thus we see that the morphism from Figure \ref{fm17fig1} to Figure \ref{fm15fig2} gives an isomorphism $E^3_{5,2}\ra E^3_{6,2}$, but the zero map $E^3_{8,0}\ra E^3_{9,0}=0$. This forces $d^3_{8,0}=0$ in Figure \ref{fm17fig1}. Similarly, we have an isomorphism $E^3_{6,2}\ra E^3_{7,2}$ and an injective map $E^3_{9,0}\ra E^3_{10,0}.$ Hence $d^3_{9,0}$ in Figure \ref{fm17fig1} is the restriction of $d^3_{10,0}$ in Figure \ref{fm15fig2} and we recall from \S\ref{fm1512} that $d_{10,0}^3$ is surjective, so $d^3_{9,0}$ in Figure \ref{fm17fig1} is an isomorphism. Hence Figure \ref{fm17fig1} leads to the $E^\iy$-page shown in Figure~\ref{fm17fig2}.

\begin{figure}[htb]
\centering
\begin{tikzpicture}
\matrix (m) [matrix of math nodes,
nodes in empty cells,nodes={minimum width=9ex,
minimum height=5ex,outer sep=-5pt},
column sep=1ex,row sep=1ex]{
4 & \Z\an{\al_4\de_3}  & \\
2 &   & \Z_2\an{\al_1^2\bar\de_5'} &\\
1 &&& \Z_2\an{\al_1\bar\de_3^2}\\
0 & \Z\an{\de_3} &  &  & \Z_3\an{\de_7} & \Z_2\an{\de_3\de_5'} \\
\quad\strut & 3 & 5 & 6 & 7 & 8  \strut \\};
\draw[thick] (m-1-1.north east) node[above]{$q$} -- (m-5-1.east);
\draw[thick] (m-5-1.north) -- (m-5-7.north east) node[right]{$p$};
\end{tikzpicture}
\caption{$E^\iy$-page of $\ti H_p(K(\Z,3),\Om^{\bs\Spin}_q(*))\Ra\ti\Om^{\bs\Spin}_{p+q}(K(\Z,3))$, $p+q\le 8$}
\label{fm17fig2}
\end{figure}

\subsubsection{Determining the filtration}
\label{fm1713}

The groups $\ti\Om^{\bs\Spin}_n(K(\Z,3))$ in Table \ref{fm3tab6} and \eq{fm3eq49}, \eq{fm3eq51} follow from Figure \ref{fm17fig2} for all $n\neq 7$. For $n=7$, there is a nontrivial filtration, which we claim is
\begin{equation}
\begin{tikzcd}[column sep=small]
0\arrow[r,symbol=\subset,"\Z\an{\al_4\de_3}"' yshift=-1.5ex] & F_{3,7}'\!\cong\!\Z\arrow[r,symbol=\subset,"\Z_2\an{\al_1^2\bar\de_5'}"' yshift=-1.5ex] & F_{5,7}'\!\cong\!\Z\arrow[r,symbol=\subset,"\Z_2\an{\al_1\bar\de_3^2}"' yshift=-1.5ex] & F_{6,7}'\!\cong\!\Z\arrow[r,symbol=\subset,"\Z_3\an{\de_7}"' yshift=-1.5ex] & F_{7,7}'\!=\!\ti\Om_7^{\bs\Spin}(K(\Z,3))\!\cong\!\Z.
\end{tikzcd}\!\!\!
\label{fm17eq5}
\end{equation}
Here the question is whether the extensions by $\Z_2$ or $\Z_3$ are trivial or nontrivial, for example, $F'_{5,7}$ is an extension of $\Z$ by $\Z_2$, so could be either $\Z$ or $\Z\op\Z_2$. 

We prove \eq{fm17eq5} by relating it to the filtration \eq{fm15eq5} for $\ti\Om^{\bs\Spin}_8(K(\Z,4))$. The morphism from the $K(\Z,3)$ spectral sequence to the $K(\Z,4)$ spectral sequence induced by \eq{fm17eq4} maps the filtration \eq{fm17eq5} to \eq{fm15eq5}, mapping $F'_{k,7}$ to $F_{k+1,8}$, and on the graded pieces maps $\de_3\mapsto\ep_4$, $\bar\de_5'\mapsto\bar\ep_6',$ $\bar\de_3^2\mapsto\bar\ep_7$ and $\de_7\mapsto\ep_8$. This forces all the extensions in \eq{fm17eq5} to be nontrivial, as for example if $F'_{5,7}=\Z\op\Z_2$ then the map $F'_{5,7}\cong\Z\op\Z_2\ra F_{6,8}\cong\Z$ would map the $\Z_2$-summand to zero, contradicting $\bar\de_5'\mapsto\bar\ep_6'$. Hence $\ti\Om^{\bs\Spin}_7(K(\Z,3))\cong\Z$, completing Table \ref{fm3tab6} in the case of $K(\Z,3)$.

The classes $\de_3$, $\de_3\de_5'$ are dual to $d_3$ and $\bar d_3\bar d_5'=\bar d_3\Sq^2(\bar d_3)$, which by Proposition \ref{fm2prop2} gives the explicit isomorphisms \eq{fm3eq49} and \eq{fm3eq51}. By \eq{fm17eq5} the group $\ti\Om_7^{\bs\Spin}(K(\Z,3))\cong\Z$ has a generator $\th=[X,\al]$ such that $12\th=[K3\t \cS^3,\al']$, where $\al'=1\t s_3$ for the generator $s_3\in H^3(\cS^3,\Z)$. Observe that \eq{fm3eq50} determines a well-defined map $\ti\Om_7^{\bs\Spin}(K(\Z,3))\ra\Q$, $[X,\al]\mapsto\frac14\int_X p_1(TX)\cup\al$. It maps $[K3\t \cS^3,\al']$ to $12$ and hence $\th$ to $1$. Therefore, \eq{fm3eq50} takes integer values and is indeed an isomorphism.

Moreover, it is easy to check (see \S\ref{fm175}) that \eq{fm3eq50} maps $\vartheta_2\mapsto -1$, so in fact $\th=\vartheta_2$. Clearly, \eq{fm3eq49} maps $\rho\mapsto 1$. The formula $\Sq^2(b_2)=b_3$ in $H^*(\SU,\Z)$ implies that \eq{fm3eq51} maps $\upsilon\mapsto 1$.

\subsection{\texorpdfstring{Computation of $\ti\Om_n^{\bs{\mathrm{Spin}}}(K(\Z_2,3))$}{Computation of Ωₙˢᵖⁱⁿ(K(ℤ₂,3))}}
\label{fm172}

\subsubsection{Description of the (co)homology}
\label{fm1721}

Let $\bar c_3'\in \ti H^3(K(\Z_2,3),\Z_2)$ be the primary class as in \S\ref{fm23}. According to Serre \cite{Serr}, the $\Z_2$-cohomology of $K(\Z_2,3)$ is a polynomial algebra on
\begin{gather*}
\bar c_3',\enskip
\bar c_4=\Sq^1(\bar c_3'),\enskip
\bar c_5'=\Sq^2(\bar c_3'),\enskip
\bar c_6'=\Sq^2\Sq^1(\bar c_3'),\\
\bar c_7=\Sq^3\Sq^1(\bar c_3'),\enskip
\bar c_9'=\Sq^4\Sq^2(\bar c_3'),\enskip
\bar c_{10}'=\Sq^4\Sq^2\Sq^1(\bar c_3'),\ldots
\end{gather*}
We again follow Notation \ref{fm15nota1}, namely classes $\bar c_i$ lift to $\Z$-cohomology while classes $\bar c_i'$ do not admit lifts, see \eq{fm17eq9} below. Dually, $\ti H_n(K(\Z_2,3),\Z_2)$ are generated by homology classes $\bar\ga_3',\bar\ga_4,\bar\ga_3'^2,\ldots, \bar\ga_3'^2\bar\ga_4,\bar\ga_{10}'$.

\begin{align}
&\begin{tabular}{P{2.5cm}|cccccccccc}
 $n$ & $0,1,2$ & $3$ & $4$ & $5$ & $6$ & $7$\\
  \hline
 $\tilde H^n(K(\Z_2,3),\Z_2)$\!\!\!\!\! & $0$ & $\!\!\Z_2\an{\bar c_3'}\!\!$ & $\!\!\Z_2\an{\bar c_4}\!\!$ & $\!\!\Z_2\an{\bar c_5'}\!\!$ & $\!\!\Z_2\an{\bar c_3'^2, \bar c_6'}\!\!$ &  $\!\!\Z_2\an{\bar c_3'\bar c_4,\bar c_7}\!\!$
\end{tabular}
\nonumber
\\
&\begin{tabular}{P{0.5cm}|ccc}
 & $8$ & $9$ & $10$\\
  \hline
  $\cdots$ & $\!\!\Z_2\an{\bar c_4^2, \bar c_3'\bar c_5'}\!\!$ & $\!\!\Z_2\an{\bar c_3'^3, \bar c_4\bar c_5', \bar c_3'\bar c_6', \bar c_9'}\!\!$ & $\!\!\Z_2\an{\bar c_5'^2, \bar c_4\bar c_6', \bar c_3'\bar c_7, \bar c_3'^2\bar c_4, \bar c_{10}'}\!\!$
\end{tabular}
\label{fm17eq6}\\
&\begin{tabular}{P{2.5cm}|cccccccccc}
 $n$ & $0,1,2$ & $3$ & $4$ & $5$ & $6$ & $7$\\
  \hline
 $\tilde H_n(K(\Z_2,3),\Z_2)$\!\!\!\!\! & $0$ & $\!\!\Z_2\an{\bar \ga_3'}\!\!$ & $\!\!\Z_2\an{\bar \ga_4}\!\!$ & $\!\!\Z_2\an{\bar \ga_5'}\!\!$ & $\!\!\Z_2\an{\bar \ga_3'^2, \bar \ga_6'}\!\!$ &  $\!\!\Z_2\an{\bar \ga_3'\bar \ga_4,\bar \ga_7}\!\!$
\end{tabular}
\nonumber\\
&\begin{tabular}{P{0.5cm}|ccc}
 & $8$ & $9$ & $10$\\
  \hline
  $\cdots$ & $\!\!\Z_2\an{\bar \ga_4^2, \bar \ga_3'\bar \ga_5'}\!\!$ & $\!\!\Z_2\an{\bar \ga_3'^3, \bar \ga_4\bar \ga_5', \bar \ga_3'\bar \ga_6', \bar \ga_9'}\!\!$ & $\!\!\Z_2\an{\bar \ga_5'^2, \bar \ga_4\bar \ga_6', \bar \ga_3'\bar \ga_7, \bar \ga_3'^2\bar \ga_4, \bar \ga_{10}'}\!\!\!\!$
\end{tabular}\!\!
\label{fm17eq7}
\end{align}

To determine the (co)homology of $K(\Z_2,3)$ with $\Z$-coefficients we use the Bockstein spectral sequence as in \S\ref{fm1571}, where by Proposition \ref{fm15prop4} it suffices to consider the prime $p=2$. The first differential is the Steenrod operation $\Sq^1=\rho_2\circ\be_2$ and the Adem relation \eq{fm2eq15} allows the computation of the action of $\Sq^1$ on all classes \eq{fm17eq6}.
 In detail, the action of $\Sq^1$ is given by
\e
\label{fm17eq8}
\begin{aligned}
\bar c_3'&\mapsto \bar c_4,&  \bar c_4&\mapsto 0, & \bar c_5'&\mapsto \bar c_3'^2, &
\bar c_3'^2&\mapsto 0, \\
\bar c_6'&\mapsto \bar c_7, & \bar c_7&\mapsto 0, &\bar c_3'\bar c_4&\mapsto \bar c_4^2,
& \bar c_4^2&\mapsto 0, \\
\bar c_3'\bar c_5' &\mapsto \bar c_4\bar c_5'+\bar c_3'^3, & \bar c_3'^3&\mapsto \bar c_3'^2\bar c_4,
& \bar c_4\bar c_5'&\mapsto \bar c_3'^2\bar c_4,
& \bar c_3'\bar c_6'&\mapsto \bar c_4\bar c_6'+\bar c_3'\bar c_7, \\
\bar c_9'&\mapsto\bar c_5'^2, & \bar c_5'^2&\mapsto 0, & \bar c_4\bar c_6'&\mapsto \bar c_4 \bar c_7,
& \bar c_3'\bar c_7&\mapsto \bar c_4\bar c_7, \\
\bar c_3'^2\bar c_4&\mapsto 0, & \bar c_{10}'&\mapsto \bar c_{11}\mathrlap{\coloneqq\Sq^5(\bar c_6')\neq 0.}
\end{aligned}
\e
It follows that the Bockstein spectral sequence degenerates at the $E_1$-page, and the integer cohomology $\ti H^n(K(\Z_2,3),\Z)$ is a direct sum of $\Z_2$-summands on generators
\begin{align*}
 c_4&=\Sq^1_\Z(\bar c_3'),
 &c_6&=\Sq^3_\Z(\bar c_3'),
 &c_7&=\Sq^3_\Z(\bar c_4),
 &&c_4^2,\\
 c_9&=\Sq^1_\Z(\bar c_3' \bar c_5'),
 &&c_4 c_6, &c_{10}&=\Sq^1_\Z(\bar c_3' \bar c_6'), &\ti c_{10}&=\Sq^1_\Z(\bar c_9').
\end{align*}
We record this in \eq{fm17eq9}.
\begin{align}
&\begin{tabular}{P{2.6cm}|cccccccc}
 $n$ & $\!\!0,1,2,3,5\!\!$ & $4$ & $6$ & $7$ & $8$\\
  \hline
 $\tilde H^n(K(\Z_2,3),\Z)$\! & $0$ & $\!\!\Z_2\an{c_4}\!\!$ & $\!\!\Z_2\an{c_6}\!\!$ & $\!\!\Z_2\an{c_7}\!\!$ & $\!\!\Z_2\an{c_4^2}\!\!$
\end{tabular}
\nonumber
\\
&\begin{tabular}{P{2.5cm}|cc}
 & $9$ & $10$\\
  \hline
  $\cdots$ & $\Z_2\an{c_9}\!\!$ & $\!\!\Z_2\an{c_4c_6,c_{10},\ti c_{10}}\!\!$
\end{tabular}
\label{fm17eq9}
\end{align}

Dually, the homological mod-$2$ Bockstein spectral sequence also degenerates at the $E^1$-page, the first differential $d^1=\rho_2\circ \be_2:H_*(X,\Z_2)\ra H_{*-1}(X,\Z_2)$ being dual to \eq{fm17eq8}; each summand $\Z_2\an{d^1(\bar\ka)}$ in the image of $d^1$ determines a $\Z_2$-summand $\Z_2\an{\la}$, where $\la$ is the image of $\bar\ka$ under the homological Bockstein homomorphism $\be_2:H_*(X,\Z_2)\ra H_{*-1}(X,\Z)$, and $\rho_2(\la)=d^1(\bar\ka)$. In this way we find that $\ti H_n(K(\Z_2,3),\Z)$ is a $\Z_2$-vector space with basis
\begin{align*}
\ga_3'=\be_2(\bar\ga_4),\enskip \ga_5'=\be_2(\bar\ga_3'^2),\enskip \ga_6'=\be_2(\bar\ga_7),\enskip \ga_3'\ga_4=\be_2(\bar\ga_4^2),\\ \ga_3'\ga_5'=\be_2(\bar\ga_4\bar\ga_5')=\be_2(\bar\ga_3'^3),\enskip \ga_3'\ga_6'=\be_2(\bar\ga_4\bar\ga_6')=\be_2(\bar\ga_3'\bar\ga_7),\\ \ga_9'=\be_2(\bar\ga_5'^2),\enskip \ti\ga_9'=\be_2(\bar\ga_3'^2\bar\ga_4),\enskip \ga_{10}'=\be_2(\bar\ga_{11}),\enskip \ti\ga_{10}'=\be_2(\bar\ga_4\bar\ga_7).
\end{align*}
We record this in \eq{fm17eq10}. The mod-$2$ reductions of these homology classes are as suggested by the notation, namely $\rho_2(\ga_i'\cdots\ga_j')=\bar\ga_i'\cdots\bar\ga_j'$, as well as the formulas $\rho_2(\ti\ga_9')=\bar\ga_3'^3+\bar\ga_4\bar\ga_5'$ and $\rho_2(\ti\ga_{10}')=\bar\ga_4\bar\ga_6'+\bar\ga_3'\bar\ga_7$.
\begin{align}
&\begin{tabular}{P{2.5cm}|cccccccc}
 $n$ & $\!\!0,1,2,4\!\!$ & $3$ & $5$ & $6$ & $7$ & $8$\\
  \hline
 $\tilde H_n(K(\Z_2,3),\Z)$\!\! & $0$ & $\!\!\Z_2\an{\ga_3'}\!\!$ & $\!\!\Z_2\an{\ga_5'}\!\!$ & $\!\!\Z_2\an{\ga_6'}\!\!$ & $\!\!\Z_2\an{\ga_3'\ga_4}\!\!$ & $\!\!\Z_2\an{\ga_3'\ga_5'}\!\!$
\end{tabular}
\nonumber
\\
&\begin{tabular}{P{2.5cm}|cc}
 & $9$ & $10$\\
  \hline
  $\cdots$ &  $\!\!\Z_2\an{\ga_3'\ga_6',\ga_9',\ti\ga_9'}\!\!$ & $\!\!\Z_2\an{\ga_{10}',\ti\ga_{10}'}\!\!$
\end{tabular}
\label{fm17eq10}
\end{align}

\subsubsection{Computation of the spectral sequence}
\label{fm1722}

We calculate $\ti\Om_*^{\bs\Spin}(K(\Z_2,3))$ using the Atiyah--Hirzebruch spectral sequence
\begin{equation*}
 \ti H_p(K(\Z_2,3),\Om^{\bs\Spin}_q(*))\Longra\ti\Om^{\bs\Spin}_{p+q}(K(\Z_2,3)).
\end{equation*}
The Adem relation \eq{fm2eq15} implies that $\Sq^2$ acts on the classes \eq{fm17eq6} by
\begin{align*}
 &\bar c_3'\mapsto\bar c_5',\enskip \bar c_4\mapsto\bar c_6',\enskip \bar c_5'\mapsto\bar c_7,\enskip \bar c_3'^2\mapsto\bar c_4^2,\enskip \bar c_6'\mapsto 0,\\
 &\enskip \bar c_3'\bar c_4\mapsto\bar c_4\bar c_5'+\bar c_3'\bar c_6',\enskip \bar c_7\mapsto0,\enskip \bar c_4^2\mapsto0,\enskip \bar c_3'\bar c_5'\mapsto\bar c_3'^2\bar c_4+\bar c_5'^2+\bar c_3'\bar c_7.
\end{align*}
According to Proposition \ref{fm2prop1}, the differential $d^2_{p,1}$ in this spectral sequence is dual to the Steenrod operation $\Sq^2$, so given in the dual basis $\bar\ga_3',\bar\ga_4,\bar\ga_3'^2,\ab\ldots,\ab \bar\ga_3'^2\bar\ga_4,\ab\bar\ga_{10}'$ by the transpose matrix. Moreover, $d^2_{p,0}$ can be identified with $d^2_{p,1}\circ \rho_2$. This leads to the $E^3$-page of the spectral sequence shown in Figure~\ref{fm17fig3}.

\begin{figure}[htb]
\centering
\begin{tikzpicture}
  \matrix (m) [matrix of math nodes,
    nodes in empty cells,nodes={minimum width=7ex,
    minimum height=5ex,outer sep=-5pt},
    column sep=1ex,row sep=1ex]{
	4 & \!\!\Z_2\an{\al_4\ga_3'}\!\! & \!\!\Z_2\an{\al_4\ga_5'}\!\! \\
	2 &&& \!\!\Z_2\an{\al_1^2\bar\ga_6'}\!\! & \!\!\Z_2\an{\al_1^2\bar\ga_7}\!\!\\
	1 &&& \!\!\Z_2\an{\al_1\bar\ga_3'^2}\!\!\\
	0 & \Z_2\an{\ga_3'}\!\! &&& \!\!\Z_2\an{\ga_3'\ga_4}\!\! & \!\!\Z_2\an{\ga_3'\ga_5'}\!\! & \!\!\Z_2\an{\ga_9',\ti\ga_9'\!+\!\ga_3'\ga_6'}\!\! \\
    \quad\strut  & 3 & 5 & 6 & 7 & 8 & 9 \strut \\};
 \draw[thick] (m-1-1.north east) node[above]{$q$} -- (m-5-1.east);
 \draw[thick] (m-5-1.north) -- (m-5-7.north east) node[right]{$p$};

 \draw[-stealth,shorten >=1ex] (m-4-7) -- node[above right]{\scriptsize $d^3_{9,0}$} (m-2-4);
 \draw[-stealth,shorten >=1ex,shorten <=1ex] (m-2-4) -- node[below right]{\scriptsize $d^3_{6,2}$} (m-1-2);

\end{tikzpicture}
\caption{$E^3$-page of $\ti H_p(K(\Z_2,3),\Om^{\bs\Spin}_q(*))\Ra\ti\Om^{\bs\Spin}_{p+q}(K(\Z_2,3))$, $p\!+\!q\!\le\! 9$.}
\label{fm17fig3}
\end{figure}

As in \S\ref{fm1712} the map $\chi:\Si K(\Z_2,3)\ra K(\Z_2,4)$, uniquely defined up to homotopy by the condition that $\chi^*(f_4')$ is the suspension of $\bar c_3'$, maps the suspension of $\bar\ga_6'$ to $\bar\varphi_7'$ and the suspension of $\bar\ga_3'$ to $\bar\varphi_4'$ and hence identifies the differential $d^3_{6,2}$ in Figure \ref{fm17fig3} with $d_{7,2}^3$ in Figure \ref{fm15fig12}. It follows that $d^3_{6,2}\neq 0$, which leads to the $E^\iy$-page shown in Figure~\ref{fm17fig4}.

\begin{figure}[htb]
\centering
\begin{tikzpicture}
  \matrix (m) [matrix of math nodes,
    nodes in empty cells,nodes={minimum width=10ex,
    minimum height=5ex,outer sep=-5pt},
    column sep=1ex,row sep=1ex]{
	1 && \Z_2\an{\al_1\bar\ga_3'^2}\\
	0 & \Z_2\an{\ga_3'} && \Z_2\an{\ga_3'\ga_4} & \Z_2\an{\ga_3'\ga_5'}\\
    \quad\strut  & 3 & 6 & 7 & 8 \strut \\};
 \draw[thick] (m-1-1.north east) node[above]{$q$} -- (m-3-1.east);
 \draw[thick] (m-3-1.north) -- (m-3-5.north east) node[right]{$p$};
 
\end{tikzpicture}
\caption{$E^\iy$-page of $\ti H_p(K(\Z_2,3),\Om^{\bs\Spin}_q(*))\Ra\ti\Om^{\bs\Spin}_{p+q}(K(\Z_2,3))$, $p\!+\!q\!\le\! 8$.}
\label{fm17fig4}
\end{figure}

The extension problems are trivial in dimensions $7\neq n\le 8$, while for $n=7$ the group $\ti\Om^{\bs\Spin}_7(K(\Z_2,3))$ could be $\Z_2^2$ or $\Z_4$. This proves Table \ref{fm3tab6} for $K(\Z_2,3)$.

By Figure \ref{fm17fig4}, the groups $\ti\Om^{\bs\Spin}_n(K(\Z_2,3))$ for $n=3,8$ are isomorphic to $\Z_2\an{\ga_3'}$ and $\Z_2\an{\ga_3'\ga_5'}$. By \eq{fm17eq6}, these classes are dual to $\bar c_3'$ and $\bar c_3'\bar c_5'=\bar c_3'\Sq^2(\bar c_3')$, so Proposition \ref{fm2prop2} gives the isomorphisms \eq{fm3eq52} and \eq{fm3eq53}.

We have already seen that \eq{fm3eq49} maps $\rho\mapsto 1$ and that \eq{fm3eq51} maps $\upsilon\mapsto1$.

\subsection{\texorpdfstring{Computation of $\ti\Om_n^{\bs{\mathrm{Spin}}}(\SU(2))$}{Computation of Ωₙˢᵖⁱⁿ(SU(2))}}
\label{fm173}

As $\SU(2)\cong\cS^3$, we can apply the suspension isomorphism of the generalized homology theory to get $\ti\Om_q^{\bs\Spin}(\cS^3)\cong\Om_{q-3}^{\bs\Spin}(*)$ for all $q$. This proves Table \ref{fm3tab6} for $\SU(2)$. Moreover, it is easy to see that \eq{fm3eq41} maps $\rho\mapsto 1$ and \eq{fm3eq42} maps~$\vartheta_1\mapsto 1$.

\subsection{\texorpdfstring{Computation of $\ti\Om_n^{\bs{\mathrm{Spin}}}(\SU)$}{Computation of Ωₙˢᵖⁱⁿ(SU)}}
\label{fm174}

We will use the spectral sequence
\e
 \ti H_p(\SU,\ab\Om^{\bs\Spin}_q(*))\Longrightarrow\ti\Om^{\bs\Spin}_{p+q}(\SU).
 \label{fm17eq11}
\e
Borel \cite[Th.~8.2]{Bore} shows that the cohomology $H^*(\SU,\Z)$ is an exterior algebra on generators $b_i$ of degree $2i-1$, $i\ge 2$, which are dual to homology classes $\be_i\in H_{2i-1}(\SU,\Z)$ such that $(i-1)!\be_i$ corresponds under the Hurewicz homomorphism to the generator $\la_i:\cS^{2i-1}\ra\SU$ of $\pi_{2i-1}(\SU).$ The reduced cohomology $\ti H^n(\SU,\Z)$ for $n\le 9$ is therefore given by
\e
\begin{tabular}{c|cccccccccccc}
$n$ & 
$0,1,2,4,6$ & $3$ & $5$ & $7$ & $8$ & $9$   \\
\hline
\parbox[top][4ex][c]{2.1cm}{$\ti H^n(\SU,\Z)$} & 
$0$ & $\!\!\Z\an{b_2}\!\!$ &  $\!\!\Z\an{b_3}\!\!$ & $\!\!\Z\an{b_4}\!\!$ & $\!\!\Z\an{b_2b_3}\!\!$ & $\!\!\Z\an{b_5}\!\!$  
\end{tabular}
\label{fm17eq12}
\e
From the Universal Coefficient Theorem we obtain the homology, as follows. Here $b_i\cdot\be_i=1$ and $\be_2\be_3$ is the generator dual to $b_2b_3$ as in Notation \ref{fm15nota1} (we make no use of the product on~$H_*(\SU,\Z)$).
\e
\begin{tabular}{c|cccccccccccc}
$n$ & 
$0,1,2,4,6$ & $3$ & $5$ & $7$ & $8$ & $9$   \\
\hline
\parbox[top][4ex][c]{2.1cm}{$\ti H_n(\SU,\Z)$} & 
$0$ & $\!\!\Z\an{\be_2}\!\!$ &  $\!\!\Z\an{\be_3}\!\!$ & $\!\!\Z\an{\be_4}\!\!$ & $\!\!\Z\an{\be_2\be_3}\!\!$ & $\!\!\Z\an{\be_5}\!\!$  
\end{tabular}
\label{fm17eq13}
\e

The $E^2$-page of the spectral sequence \eq{fm17eq11} is therefore as in Figure \ref{fm17fig5}. We have shown all non-zero terms $\ti H_p(\SU,\Om^{\bs\Spin}_q(*))$ with $p+q\le 9$ and all possible non-zero differentials $d^2_{p,q}$ in this range.

\begin{figure}[htb]
\centering
\begin{tikzpicture}
\matrix (m) [matrix of math nodes,
nodes in empty cells,nodes={minimum width=11.8ex,
minimum height=5ex,outer sep=-5pt},
column sep=1ex,row sep=1ex]{
4 & \Z\an{\al_4\be_2}  & \Z\an{\al_4\be_3} &\\
2 & \Z_2\an{\al_1^2\bar\be_2}  & \Z_2\an{\al_1^2\bar\be_3} & \Z_2\an{\al_1^2\bar\be_4} &\\
1 & \Z_2\an{\al_1\bar\be_2}  & \Z_2\an{\al_1\bar \be_3} & \Z_2\an{\al_1\bar \be_4} & \Z_2\an{\al_1\bar \be_2\bar \be_3} \\
0 & \Z\an{\be_2} & \Z\an{\be_3} & \Z\an{\be_4} & \Z\an{\be_2\be_3} & \Z\an{\be_5} \\
\quad\strut & 3 & 5 & 7 & 8 & 9 \strut \\};
\draw[-stealth] (m-3-3) -- node[left]{\scriptsize $d^2_{5,1}$}  (m-2-2);
\draw[-stealth] (m-3-4) -- node[left]{\scriptsize $d^2_{7,1}$}  (m-2-3);
\draw[-stealth] (m-4-3) -- node[left]{\scriptsize $d^2_{5,0}$}  (m-3-2);
\draw[-stealth] (m-4-4) -- node[left]{\scriptsize $d^2_{7,0}$}  (m-3-3);
\draw[-stealth] (m-4-6) -- node[right,xshift=3ex]{\scriptsize $d^2_{9,0}$} (m-3-4);
\draw[thick] (m-1-1.north east) node[above]{$q$} -- (m-5-1.east);
\draw[thick] (m-5-1.north) -- (m-5-6.north east) node[right]{$p$};
\end{tikzpicture}
\caption{$E^2$-page of $\ti H_p(\SU,\Om^{\bs\Spin}_q(*))\Ra\ti\Om^{\bs\Spin}_{p+q}(\SU)$, $p+q\le 9$}
\label{fm17fig5}
\end{figure}

As before, the differentials $d^2_{p,0}, d^2_{p,1}$ are determined using Steenrod squares, see Proposition \ref{fm2prop1}. Borel \cite[Th.~8.3]{Bore} proves that
\begin{equation*}
\Sq^2:\bar b_2\longmapsto \bar b_3, \quad
\Sq^2:\bar b_3\longmapsto 0, \quad
\Sq^2:\bar b_4\longmapsto\bar  b_5,
\end{equation*}
so
\begin{align*}
d^2_{5,0}&:\be_3\longmapsto\al_1\bar\be_2, &
d^2_{5,1}&:\al_1\bar\be_3\longmapsto\al_1^2\bar\be_2, \\
d^2_{7,0}&=d^2_{7,1}=0, &
d^2_{9,0}&:\be_5\longmapsto\al_1\bar\be_4.
\end{align*}
From this we deduce the $E^3$-page of \eq{fm17eq11} shown in Figure \ref{fm17fig6}. 

\begin{figure}[htb]
\centering
\begin{tikzpicture}
\matrix (m) [matrix of math nodes,
nodes in empty cells,nodes={minimum width=11.8ex,
minimum height=5ex,outer sep=-5pt},
column sep=1ex,row sep=1ex]{
4 & \Z\an{\al_4\be_2}  &  &\\
2 &   & \Z_2\an{\al_1^2\bar\be_3} &  &\\
1 &   &  &  &  \\
0 & \Z\an{\be_2} & \Z\an{2\be_3} & \Z\an{\be_4} & \Z\an{\be_2\be_3} &  \\
\quad\strut & 3 & 5 & 7 & 8  \strut \\};
\draw[-stealth] (m-4-5.north) -- node[right]{\scriptsize $d^3_{8,0}$} (m-2-3.east);
\draw[-stealth,bend left=10] (m-4-5.west) to node[left]{\scriptsize $d^5_{8,0}$} (m-1-2.south);
\draw[thick] (m-1-1.north east) node[above]{$q$} -- (m-5-1.east);
\draw[thick] (m-5-1.north) -- (m-5-5.north east) node[right]{$p$};
\end{tikzpicture}
\caption{$E^3=E^\iy$-page of $\ti H_p(\SU,\Om^{\bs\Spin}_q(*))\Ra\ti\Om^{\bs\Spin}_{p+q}(\SU)$, $p+q\le 8$}
\label{fm17fig6}
\end{figure}

The only possible higher differentials are $d^3_{8,0}$ and $d^5_{8,0}$, which we claim are both trivial. For this we compare \eq{fm17eq11} with the spectral sequence \eq{fm17eq3}. Let $\phi:\SU\ra K(\Z,3)$ be the classifying map of the cohomology class $b_2\in H^3(\SU,\Z)$. Then $\phi_*(\be_2)=\de_3$ and $\phi_*(\be_3)=\de_5'$, and thus in Figures \ref{fm17fig6} and \ref{fm17fig1}, in positions $(3,4),(5,2),(8,0)$ it maps $\al_4\be_2\mapsto\al_4\de_3$, $\al_1^2\bar\be_3\mapsto\al_1^2\bar\de_5'$, and $\be_2\be_3\mapsto\de_3\de_5'$. But in Figure \ref{fm17fig1} we have shown that $d^3_{8,0}=0$, and $d^5_{8,0}=0$ is trivial as it maps $\Z_2\ra\Z$. Thus in Figure \ref{fm17fig6} we must have $d^3_{8,0}=d^5_{8,0}=0$, so Figure \ref{fm17fig6} is also the $E^\iy$-page of the spectral sequence.

Figure \ref{fm17fig6} determines the groups $\ti\Om^{\bs\Spin}_n(\SU)$ for $n\le 8$ in Table \ref{fm3tab6}, except for $n=7$ where we have a filtration. For $n=7$ we use that $\psi:\SU\ra K(\Z,3)$ induces a morphism of filtrations
\begin{equation*}
\begin{tikzcd}[column sep=10.9ex]
	0\arrow[r,symbol=\subset,"\Z\an{\al_4\be_2}" yshift=1.5ex] & \begin{array}{l} F_{3,7}\\ \cong\Z\end{array}\dar{\cong}\arrow[r,symbol=\subset,"\Z_2\an{\al_1^2\bar\be_3}" yshift=1.5ex] & \begin{array}{l} F_{5,7}\\ \cong\Z\end{array}\dar{\cong}\arrow[r,symbol=\subset,"\Z\an{\be_4}" yshift=1.5ex] & \begin{array}{l} F_{7,7}=\ti\Om_7^{\bs\Spin}(\SU)\\ \cong\Z^2\end{array}\dar{\psi_*}\\
	0\arrow[r,symbol=\subset,"\Z\an{\al_4\de_3}"' yshift=-1.5ex] & \begin{array}{l} F_{3,7}'\\ \cong\Z\end{array}\arrow[r,symbol=\subset,"\Z_2\an{\al_1^2\bar\de_5'}"' yshift=-1.5ex] & \begin{array}{l} F_{5,7}'\\ \cong\Z\end{array}\arrow[r,symbol=\subset,"\Z_2\an{\al_1\bar\de_3^2}\op\Z_3\an{\de_7}"' yshift=-1.5ex] & \begin{array}{l} F_{7,7}'=\ti\Om_7^{\bs\Spin}(K(\Z,3)).\\ \cong\Z\end{array}
\end{tikzcd}
\end{equation*}

Here the bottom line was computed in \eq{fm17eq5}, and in particular we have $F'_{5,7}\cong\Z$, not $F'_{5,7}\cong\Z\op\Z_2$. As $\psi_*$ induces isomorphisms $F_{3,7}=\Z\an{\al_4\be_2}\ra F'_{3,7}=\Z\an{\al_4\de_3}$ and $F_{5,7}/F_{3,7}=\Z_2\an{\al_1^2\bar\be_3}\ra F'_{5,7}/F'_{3,7}=\Z_2\an{\al_1^2\bar\de_5'}$, it follows that $F_{5,7}\cong\Z$, and hence $\ti\Om^{\bs\Spin}_7(\SU)\cong\Z^2$.

Using Proposition \ref{fm2prop2}, Figure \ref{fm17fig6} implies the explicit isomorphisms $\Psi_n(\SU)$ for $n=3,5,8$ which, by definition of $\be_2,2\be_3,\be_4,\be_2\be_3$ can be written as \eq{fm3eq43}, \eq{fm3eq44}, \eq{fm3eq46}.

We have already seen that \eq{fm3eq41} maps $\rho\mapsto 1$ in the $\SU(2)$-case. Recall that $\varsigma=[\cS^5,\phi]$ for the generator $\phi\in\pi_5(\SU)$. We have $\phi_*([\cS^5])=2\be_3$ under the Hurewicz homomorphism, so \eq{fm3eq44} maps $\varsigma\mapsto 1$. For $\upsilon$, recall that $H^*(\SU(3),\Z)=\Z[b_2,b_3]/(b_2^2,b_3^2)$. For the inclusion $\phi:\SU(3)\ra\SU$ we have $\phi^*(b_2b_3)=b_2b_3$, and this class is Kronecker dual to the fundamental class of $\SU(3)$, so \eq{fm3eq46} maps $\upsilon\mapsto 1$.

Finally, we explain how to get the explicit isomorphism \eq{fm3eq45}. By the spectral sequence, we have an isomorphism $\ti\Om_7^{\bs\Spin}(\SU)/F_{5,7}\overset{\cong}\longra\Z$, $[X,\phi]\mapsto\int_X\phi^*(b_4)$, which is the second component of \eq{fm3eq45}. Observe that $\vartheta_3$ is mapped to $1$ under this isomorphism, hence determines a splitting $\ti\Om_7^{\bs\Spin}(\SU)=F_{5,7}\op\Z\an{\vartheta_3}$ with projection $\ti\Om_7^{\bs\Spin}(\SU)\ra F_{5,7}$, $[X,\phi]\mapsto [X,\phi]-\int_X\phi^*(b_4)\cdot\vartheta_3$. According to the diagram above, $\psi:\SU\ra K(\Z,3)$ induces a morphism $F_{5,7}\ra F_{5,7}'\subset\ti\Om_7^{\bs\Spin}(K(\Z,3))$ whose image has index $6$, so \eq{fm3eq50} (divided by $-6$) is an isomorphism $F_{5,7}\ra \Z$, $[X,\phi]\mapsto -\frac{1}{24}\int_X p_1(TX)\phi^*(b_2)$, and the image of $[X,\phi]-\int_X\phi^*(b_4)\vartheta_3$ is $-\frac{1}{24}\int_X p_1(TX)\phi^*(b_2)+\frac16\int_X \phi^*(b_2)$, which gives the second component of~\eq{fm3eq45}.

For $\vartheta_1$ we have $\phi^*(b_2)=1\bt[\cS^3]$, $\phi^*(b_4)=0$ and $p_1(TX)=-48[K3]\bt 1$, so \eq{fm3eq45} maps $\vartheta_1\mapsto 2$ and hence $\frac{\vartheta_1}{2}\mapsto 1$.

Finally, for $\vartheta_3=[\CP^3\t\cS^1,\phi]$ we have $H^*(\CP^3\t\cS^1,\Z)=\Z[u,s]/(u^4,s^2)$ and one checks that $p_1(TX)=4u^2\bt 1$, $\phi^*(b_2)=u\bt s$, and $\phi^*(b_4)=u^3\bt s$, so \eq{fm3eq45} maps $\vartheta_3\mapsto (0,1)$.

\subsection{\texorpdfstring{Computation of $\ti\Om_n^{\bs{\mathrm{Spin}}}(\Sp)$}{{Computation of Ωₙˢᵖⁱⁿ(Sp)}}}
\label{fm175}

We use the spectral sequence
\begin{equation*}
 \ti H_p(\Sp,\Om_q^{\bs\Spin}(*))\Longrightarrow\ti\Om_{p+q}^{\bs\Spin}(\Sp).
\end{equation*}
Recall that $H^*(\Sp,\Z)$ is an exterior $\Z$-algebra on generators $a_i$ of degree $2i-1$. Let $\th_i\in H_{2i-1}(\Sp,\Z)$ be the homology class dual to $a_i$. The $E^2$-page of the spectral sequence is shown in Figure~\ref{fm17fig7}.

\begin{figure}[htb]
\centering
\begin{tikzpicture}
\matrix (m) [matrix of math nodes,
nodes in empty cells,nodes={minimum width=9ex,
minimum height=5ex,outer sep=-5pt},
column sep=.4ex,row sep=1ex]{
4   &   \Z\an{\al_4\th_1}           & \\
2   &   \Z_2\an{\al_1^2\bar\th_1}   & \Z_2\an{\al_1^2\bar\th_2}\\
1   &   \Z_2\an{\al_1\bar\th_1}     & \Z_2\an{\al_1\bar\th_2}\\
0   &   \Z\an{\th_1}                & \Z\an{\th_2}\\
\quad\strut & 3 & 7 \\};
\draw[thick] (m-1-1.north east) node[above]{$q$} -- (m-5-1.east);
\draw[thick] (m-5-1.north) -- (m-5-3.north east) node[right]{$p$};
\end{tikzpicture}
\caption{$E^2$-page of $\ti H_p(\Sp,\Om^{\bs\Spin}_q(*))\Rightarrow\ti\Om^{\bs\Spin}_{p+q}(\Sp)$, $p+q\le 10$. \\ This is also the $E^\iy$-page for $p+q\le 9$.}
\label{fm17fig7}
\end{figure}

All differentials $d^r_{p,q}$ with $p+q\le 10$, $r\ge 2$ vanish for degree reasons, so this is also $E^\iy$-page for $p+q\le 9$. All extension problems are trivial, which proves Table \ref{fm3tab6} for $\Sp$.

The isomorphism \eq{fm3eq47} follows from Figure \ref{fm17fig7} and Proposition \ref{fm2prop2}, and it clearly maps $\rho\mapsto 1$. Observe that \eq{fm3eq48} defines a well-defined morphism $\ti\Om^{\bs\Spin}_7(\Sp)\ra\Z^2$. It clearly maps $\vartheta_1\mapsto(1,0)$ as $\int_X\phi^*(a_2)=0$.

Recall from the proof of Lemma \ref{fm18lem2} that $H^*(X,\Z)=\Z[x,y]/(x^2,y^2)$, $\phi^*(a_1)=x$, $\phi^*(a_2)=xy$, and $p_1(TX)=-4y$. Therefore, \eq{fm3eq48} maps $\vartheta_2\mapsto (0,-1)$. As $\ti\Om_7^{\bs\Spin}(\Sp)\cong\Z^2$ and \eq{fm3eq48} maps $\vartheta_1,\vartheta_2$ to a basis of $\Z^2$, we conclude that \eq{fm3eq48} is an isomorphism.

\subsection{Proof of Proposition \ref{fm3prop6}}
\label{fm176}

By unravelling the definitions, the composition of \eq{fm3eq55} with $\hat\xi^{\bs\Spin}_{n-1}(BG)$ is \eq{fm3eq57}. The main content of Proposition \ref{fm3prop6} is how it maps the generators. We explain how comparing \eq{fm3eq37} and \eq{fm3eq48} shows $\vartheta_2\mapsto\ze_2-\ze_2'$ in detail; the other cases are similar and will be left to the reader. Since $q_1,q_2\in H^*(B\Sp,\Z)$ are the transgressions of $a_1,a_2\in H^*(\Sp,\Z)$, there is a commutative diagram
\begin{equation*}
\begin{tikzcd}[ampersand replacement=\&]
\ti\Om_7^{\bs\Spin}(\Sp)\arrow[r,"\eq{fm3eq48}"]\arrow[d,"\eq{fm3eq57}"] \& \Z^2\arrow[d,"{\begin{pmatrix}1&0\\0&1\\0&1\end{pmatrix}}" xshift=1ex]\\
\ti\Om_8^{\bs\Spin}(B\Sp)\arrow[r,"\eq{fm3eq37}"] \& \Z^3\mathrlap{.}
\end{tikzcd}
\end{equation*}
Therefore, that $\vartheta_1\mapsto(1,0)$, $\vartheta_2\mapsto(0,1)$ in \eq{fm3eq48} implies that \eq{fm3eq57} maps $\vartheta_1\mapsto\ze_1$, $\vartheta_2\mapsto\ze_2-\ze_2'$, recalling that \eq{fm3eq37} maps $\ze_2\mapsto(0,1,0)$, $\ze_2\mapsto(0,0,-1)$.

It is obvious that $\rho\mapsto\de$. By comparing \eq{fm3eq11} and \eq{fm3eq44}, one finds $\varsigma\mapsto\varep$. Similarly, comparing \eq{fm3eq32} and \eq{fm3eq42} shows $\vartheta_1\mapsto\ze_1$, comparing \eq{fm3eq35} and \eq{fm3eq45} shows $\frac{\vartheta_1}{2}\mapsto\frac{\ze_1}{2}$, and comparing \eq{fm3eq35} and \eq{fm3eq45} shows $\vartheta_3\mapsto\ze_3$

Finally, we prove that $\ti\Om_8^{\bs\Spin}(\SU)\ra\ti\Om_9^{\bs\Spin}(B\SU)$ is surjective, which shows $\upsilon\mapsto \al_1\ze_2$. This is the most difficult part of the proof as there is no explicit (cohomological) isomorphism $\ti\Om_9^{\bs\Spin}(B\SU)\cong\Z_2$. Let $C$ be the mapping cone of the map $\chi:\Si\SU\ra B\SU$ from \eq{fm3eq54}. Using the induced long exact sequence of the mapping cone,
\begin{equation*}
\begin{tikzcd}
 \cdots\rar & \ti\Om_9^{\bs\Spin}(\Si\SU)\rar{\chi_*} & \ti\Om_9^{\bs\Spin}(B\SU)\rar{\jmath_*} & \ti\Om_9^{\bs\Spin}(C)\rar & \cdots,
\end{tikzcd}
\end{equation*}
we see that it suffices to show $\ti\Om_9^{\bs\Spin}(C)=0$. To see this, we will use the Atiyah--Hirzebruch spectral sequence, so we need to determine the ordinary homology of $C$. From the analogous long exact sequence in ordinary homology
\begin{equation*}
\begin{tikzcd}
 \cdots\rar & H_n(\Si\SU)\rar{\chi_*} & H_n(B\SU)\rar{\jmath_*} & H_n(C)\rar & \cdots
\end{tikzcd}
\end{equation*}
one finds $H_n(C,\Z)=0$ for all $8\neq n<10$, $H_8(C,\Z)\cong\Z$, and $H_{10}(C,\Z)\cong\Z^2$. This leads to the $E^2$-page for the spectral sequence shown in Figure \ref{fm17fig8}. The differential $d^2_{10,0}$ is dual to $\Sq^2:H^8(C,\Z_2)\ra H^{10}(C,\Z_2)$. We claim that $d^2_{10,0}\neq 0$, which implies $\ti\Om_9^{\bs\Spin}(C)=0$, as claimed.

\begin{figure}[htb]
\centering
\begin{tikzpicture}
\matrix (m) [matrix of math nodes,
nodes in empty cells,nodes={minimum width=9ex,
minimum height=5ex,outer sep=-5pt},
column sep=.4ex,row sep=1ex]{
2   &   \Z_2   & \\
1   &   \Z_2     & \Z_2^2\\
0   &   \Z                & \Z^2\\
\quad\strut & 8 & 10 \\};
\draw[thick] (m-1-1.north east) node[above]{$q$} -- (m-4-1.east);
\draw[thick] (m-4-1.north) -- (m-4-3.north east) node[right]{$p$};
\draw[-stealth] (m-3-3) -- node[above,xshift=1ex]{\scriptsize $d^2_{10,0}$} (m-2-2);
\end{tikzpicture}
\caption{$E^2$-page of $\ti H_p(C,\Om^{\bs\Spin}_q(*))\Rightarrow\ti\Om^{\bs\Spin}_{p+q}(C)$, $p+q\le 10$.}
\label{fm17fig8}
\end{figure}

The prove $d^2_{10,0}\neq 0$, we calculate the Steenrod square using Lemma \ref{fm15lem3}, which we apply to $K=B\SU$ as follows. First, observe that $\chi:\Si\Om K\ra K$ as in the lemma can be identified with $\chi:\Si\SU\ra B\SU$ composed with the suspension of the homotopy equivalence $\Om B\SU\simeq \SU$. Denote the suspension isomorphism $\ti H^n(X,\Z_2)\ra \ti H^{n+1}(\Si X,\Z_2)$ by $a\mapsto a^\si$. Using the long exact sequence in cohomology
\begin{equation*}
\begin{tikzcd}[column sep=small]
{}\rar & H^n(C,\Z_2)\rar{\jmath^*} & H^n(B\SU,\Z_2)\rar{\chi^*} & H^n(\Si\SU,\Z_2)\rar{\de} & H^{n+1}(C,\Z_2)\rar & {}
\end{tikzcd}
\end{equation*}
we find that $H^8(C,\Z_2)\cong\Z_2$ is generated by a class $c$ with $\jmath^*(c)=\bar c_2^2$ and $H^{10}(C,\Z_2)\cong\Z^2_2$ has generators $a,b$ where $\jmath^*(a)=\bar c_2\bar c_3$ and $b=\de((\bar b_2 \bar b_3)^\si)$. Here $\chi^*(\bar c_2)=\bar b_2^\si$, so $\bar e=\bar b_2$ in the notation of Lemma \ref{fm15lem3}, thus
\begin{equation*}
 \Sq^2(c)=\de(\Sq^2(\bar b_2)\cup\bar b_2)^\si=\de(\bar b_2\bar b_3)^\si=b
\end{equation*}
is non-zero.

\subsection{Proofs of Theorem \ref{fm3thm5}(b),(c)}
\label{fm177}

We prove (b). To show that $\ti\Om^{\bs\Spin}_7(\SU)\ra\ti\Om^{\bs\Spin}_7(K(\Z,3))$ maps $\frac{\vartheta_1}{2}\mapsto-6\vartheta_2$ and $\vartheta_3\mapsto\vartheta_2$, we use the isomorphism \eq{fm3eq50}: we know that \eq{fm3eq45} maps $\frac{\vartheta_1}{2}\mapsto(1,0)$, so $\int_X\phi^*(b_4)=0$ and $\int_X\frac{p_1(TX)\phi^*(b_2)}{24}=-1$. Since $\al=\phi^*(b_2)$, this means that \eq{fm3eq50} evaluates to $-6$, proving $\frac{\vartheta_1}{2}\mapsto-6\vartheta_2$. Similarly, \eq{fm3eq45} maps $\vartheta_3\mapsto(0,1)$ so $\int_X\phi^*(b_4)=1$ and $\int_X\frac{p_1(TX)\phi^*(b_2)}{24}=\frac16$ and \eq{fm3eq50} evaluates to $1$, so $\vartheta_3\mapsto\vartheta_2$. The proof that $\vartheta_1\mapsto-12\vartheta_2$ follows similarly by comparing \eq{fm3eq42} and \eq{fm3eq50}.

To show $\al_1\ze_2\mapsto0$, consider the commutative diagram
\begin{equation*}
\begin{tikzcd}
\ti\Om_8^{\bs\Spin}(\Sp)=\Z_2\an{\al_1\vartheta_2}\rar{\eq{fm3eq57}}\dar & \ti\Om_9^{\bs\Spin}(B\Sp)=\Z_2\an{\al_1\ze_2,\al_1\ze_2'}\dar\\
\ti\Om_8^{\bs\Spin}(K(\Z,3))=\Z_2\an{\upsilon}\arrow[r,"\eq{fm3eq57}","\cong"'] & \ti\Om_9^{\bs\Spin}(K(\Z,4))=\Z_2\an{\al_1\ze_2},
\end{tikzcd}
\end{equation*}
where the groups are taken from Table \ref{fm3tab1} and Table \ref{fm3tab6}. By Proposition \ref{fm3prop6}, \eq{fm3eq57} maps $\al_1\vartheta_2\mapsto\al_1\ze_2+\al_1\ze_2'$ and $\upsilon\mapsto\al_1\ze_2$, so the bottom horizontal map in the diagram is an isomorphism. The right vertical map sends $\al_1\ze_2+\al_1\ze_2'$ to $\al_1\ze_2+\al_1(\frac{\ze_1}{4}+\ze_2+4\ze_3)=\al_1\frac{\ze_1}{4}=0$ by \eq{fm3eq7} and \eq{fm3eq8}, hence $\al_1\vartheta_2\mapsto0$ also under the left vertical map.

We prove (c). By Proposition \ref{fm3prop6} the morphism \eq{fm3eq57} maps $\vartheta_1\mapsto\ze_1$, $\frac{\vartheta_1}{2}\mapsto\frac{\ze_1}{2}$, $\vartheta_3\mapsto\ze_3$, and $\upsilon\mapsto\al_1\ze_2$. In particular, we see from Table \ref{fm3tab4} and Table \ref{fm3tab6} that $\ti\Om_8^{\bs\Spin}(\SU)\ra\ti\Om_9^{\bs\Spin}(B\SU)$ from \eq{fm3eq57} is an isomorphism. From \eq{fm3eq8} we know $\al_1\frac{\ze_1}{2}=0$ and $\al_1\ze_3=0$ in $\ti\Om_9^{\bs\Spin}(B\SU)$ which therefore implies $\al_1\frac{\vartheta_1}{2}=0$ and $\al_1\vartheta_3=0$. Finally, $\al_1\vartheta_1=0$ as $\ti\Om_8^{\bs\Spin}(\SU(2))\cong\ti\Om_5(*)$ by the suspension isomorphism and Table \ref{fm2tab2}, as $\SU(2)\cong\cS^3$.

\section{Proof of Theorem \ref{fm3thm2}}
\label{fm18}

\subsection{\texorpdfstring{Computation of $\Om_n^{\bs{\mathrm{Spin}}}(\cL K(\Z,4);K(\Z,4))$}{Computation of Ωₙˢᵖⁱⁿ(ℒK(ℤ,4);K(ℤ,4))}}
\label{fm181}

From \eq{fm2eq20} for $T=K(\Z,4)$ and the homotopy equivalence $\Om K(\Z,4)\simeq K(\Z,3)$ we obtain a spectral sequence
\e
H_p(K(\Z,4),\ti\Om^{\bs\Spin}_q(K(\Z,3)))\Longra\Om^{\bs\Spin}_{p+q}(\cL K(\Z,4);K(\Z,4)).
\label{fm18eq1}
\e

Recall the groups $\ti\Om^{\bs\Spin}_q(K(\Z,3))$ from Table \ref{fm3tab6}. This leads to the $E^2$-page of the spectral sequence \eq{fm18eq1} for $p+q\le 9$ as shown in Figure~\ref{fm18fig1}.

\begin{figure}[htb]
\centering
\begin{tikzpicture}
\matrix (m) [matrix of math nodes,
nodes in empty cells,nodes={minimum width=11.8ex,
minimum height=5ex,outer sep=-5pt},
column sep=.5ex,row sep=1ex]{
9 & ? \\
8 & \Z_2   \\
7 & \Z   \\
3 & \Z & \Z & \Z_2 \\
\quad\strut & 0 & 4 & 6 \strut \\};
\draw[-stealth] (m-4-4) -- node[above]{\scriptsize $d^6_{6,3}$}  (m-2-2);
\draw[thick] (m-1-1.north east) node[above]{$q$} -- (m-5-1.east);
 \draw[thick] (m-5-1.north) -- (m-5-4.north east) node[right]{$p$};
\end{tikzpicture}
\caption{$E^2$-page of the spectral sequence $H_p(K(\Z,4),\ti\Om^{\bs\Spin}_q(K(\Z,3)))\Ra\ab\Om^{\bs\Spin}_{p+q}(\cL K(\Z,4);K(\Z,4))$, $p+q\le 9$. This is also the $E^\iy$-page for $p+q\le 8$.}
\label{fm18fig1}
\end{figure}

The only possible higher differential in this region is $d^6_{6,3}$. By construction, $\eta$ lifts to $\Om^{\bs\Spin}_8(\cL M\SO(4);M\SO(4))$ and by \eq{fm3eq7} we have $\eta\mapsto\al_1\ze_2$ under the natural morphism $\ti\Om_n^{\bs\Spin}(M\SO(4))\ra\ti\Om_n^{\bs\Spin}(K(\Z,4))$. Hence the image of $\eta$ in $\Om^{\bs\Spin}_8(\cL K(\Z,4);K(\Z,4))$ maps under $\hat\xi^{\bs\Spin}_8(K(\Z,4))$ to $\al_1\ze_2\neq 0$, so $\Om^{\bs\Spin}_8(\cL K(\Z,4);K(\Z,4))\ne 0$ which then forces $d^6_{6,3}=0$. Therefore Figure \ref{fm18fig1} is also the $E^\iy$-page of \eq{fm18eq1} for $p+q\le 8$. All extension problems in this range are trivial, so Table \ref{fm3tab2} follows.

The class $\de$ can clearly be lifted along $\hat\xi^{\bs\Spin}_{n-1}(K(\Z,4))$ and we have just seen that $\al_1\ze_2$ can be lifted, proving Table \ref{fm3tab3} for $n\neq 8$. For $n=8$, observe the following.

\begin{lem}
\label{fm18lem1}
The image of\/ $\hat\xi^{\bs\Spin}_7(K(\Z,4))$ has index at least two.
\end{lem}

\begin{proof}
We use the isomorphism $\ti\Om_8^{\bs\Spin}(K(\Z,4))\cong\Z^2$ from \eq{fm3eq19}. Classes in $\Om^{\bs\Spin}_7(\cL K(\Z,4);K(\Z,4))$ are represented by pairs $[X,\al]$ of a compact spin $7$-manifold $X$ and a cohomology class $\al\in H^4(X\t\cS^1,\Z)$. Under the K\"unneth isomorphism, we can decompose $\al=\be\bt s + \ga\bt 1$ (here, $s\in H^1(\cS^1,\Z)$ denotes the generator) and therefore
\e
\label{fm18eq2}
 \al\cup\al=\be^2\bt s^2+2\be+2\be\ga\bt s+\ga^2\bt 1=2\be\ga\bt s+\ga^2\bt 1,
\e
which integrates to $2\int_X \be\ga$ over $X\t\cS^1$. Hence for classes in the image of $\hat\xi^{\bs\Spin}_7(K(\Z,4))$ the first component in \eq{fm3eq19} is even.
\end{proof}

On the other hand, $\ze_3$ can clearly be lifted along $\hat\xi^{\bs\Spin}_7(K(\Z,4))$ and, according to the following lemma, $2\ze_2$ can also be lifted so, conversely, the image of $\hat\xi^{\bs\Spin}_7(K(\Z,4))$ has index at most two. This proves $\Im\hat\xi^{\bs\Spin}_7(K(\Z,4))=\Z\an{2\ze_2,\ze_3}$ and completes the proof of Table \ref{fm3tab3} for $K(\Z,4)$.

\begin{lem}
\label{fm18lem2}
The class\/ $2\ze_2\in\ti\Om_8^{\bs\Spin}(M\U(2))$ can be lifted along the map
\begin{equation*}
 \hat\xi^{\bs\Spin}_7(M\U(2)):\Om_7^{\bs\Spin}(\cL M\U(2);M\U(2))\longra\ti\Om_8^{\bs\Spin}(M\U(2)).
\end{equation*}
\end{lem}

\begin{proof}
By Corollary \ref{fm3cor1} we can equivalent show that $2\ze_2\in\ti\Om_8^{\bs\Spin}(B\SU)$ lifts to $\Om_7^{\bs\Spin}(\cL B\SU;B\SU)$. We thus need to construct a compact spin $7$-manifold $X$ and a principal $\SU(n)$-bundle $P\ra X\t\cS^1_{\rm b}$ which maps to $(0,2,0)$ under \eq{fm3eq35}.

The image of $\vartheta_2\in\ti\Om_7^{\bs\Spin}(\Sp)$ in $\ti\Om_8^{\bs\Spin}(B\Sp)$ under \eq{fm3eq57} is a mapping torus principal $\Sp(2)$-bundle $P\ra X\t\cS^1$, where $X=(\Sp(2)\t\Sp(1))/(\Sp(1)\t\Sp(1))$ is a compact spin $7$-manifold. Observe that $X$ is the total space of a $3$-sphere bundle over $\cS^4=\Sp(2)/(\Sp(1)\t\Sp(1))$. This bundle has a section, so the Gysin sequence splits and yields $H^*(X,\Z)=\Z[x,y]/(x^2,y^2)$ where $x\in H^3(X,\Z)$ is Poincar\'e dual to the section and where $y\in H^4(X,\Z)$ is the pullback of the generator of $H^4(\cS^4,\Z)$. For the map $\phi:X\ra\Sp(2)$ one checks that $\phi^*(a_1)=x$, $\phi^*(a_2)=xy$, and $p_1(TX)=-4y$. Since $q_1,q_2,\ldots\in H^*(B\Sp,\Z)$ are the transgressions of $a_1,a_2,\ldots\in H^*(\Sp,\Z)$, this implies $q_1(P)=x\bt s$, $q_2(P)=xy\bt s$, where $s\in H^1(\cS^1,\Z)$ is the generator. Let $Q\ra X\t\cS^1$ be the pullback along the projection $X\t\cS^1\ra\cS^4$ of the standard $\SU(2)$-bundle over $\cS^4$, so $c_2(Q)=y\bt 1$. If we view $P$ as a principal $\SU(4)$-bundle with $c_{2n}(P)=(-1)^nq_n(P)$, then the Whitney direct sum $\SU(6)$-bundle $P\op Q\ra X\t\cS^1$ satisfies $c_2(P\op Q)=-x\bt s+y\bt 1$ and $c_4(P\op Q)=0$. This implies that \eq{fm3eq35} maps $[X\t\cS^1,P\op Q]$ to $(0,-2,0)$ so, after reversal of the orientation on $X$, we have constructed a preimage of~$2\ze_2$.
\end{proof}

\subsection{\texorpdfstring{Computation of $\Om_n^{\bs{\mathrm{Spin}}}(\cL K(\Z_2,4);K(\Z_2,4))$; proof of (c)}{Computation of Ωₙˢᵖⁱⁿ(ℒK(ℤ₂,4);K(ℤ₂,4)); proof of (c)}}
\label{fm182}

The spectral sequence \eq{fm2eq20} for $T=K(\Z_2,4)$ and the homotopy equivalence $\Om K(\Z_2,4)\simeq K(\Z_2,3)$ yield a spectral sequence with
\e
H_p(K(\Z_2,4),\ti\Om^{\bs\Spin}_q(K(\Z_2,3)))\Longra\Om^{\bs\Spin}_{p+q}(\cL K(\Z_2,4);K(\Z_2,4)).
\label{fm18eq3}
\e
Recall the groups $\ti\Om^{\bs\Spin}_q(K(\Z_2,3))$ from Table \ref{fm3tab6}. Moreover, the homology groups $H_p(K(\Z_2,4),\Z_2)$ are given in \eq{fm15eq19}. For the spectral sequence \eq{fm18eq3}, this leads to the $E^2$-page shown in Figure~\ref{fm18fig2}.

\begin{figure}[htb]
\centering
\begin{tikzpicture}
\matrix (m) [matrix of math nodes,
nodes in empty cells,nodes={minimum width=11.8ex,
minimum height=5ex,outer sep=-5pt},
column sep=.5ex,row sep=1ex]{
9 & ? \\
8 & \Z_2   \\
7 & \Z_2\text{ or }\Z_4 \\
3 & \Z_2 & \Z_2 & \Z_2 & \Z_2 \\
\quad \strut & 0 & 4 & 5 & 6 \strut \\ };
\draw[-stealth] (m-4-5) -- node[above]{\scriptsize $d^6_{6,3}$}  (m-2-2);
\draw[-stealth] (m-4-4) -- node[above]{\scriptsize $d^5_{5,3}$}  (m-3-2);
\draw[thick] (m-1-1.north east) node[above]{$q$} -- (m-5-1.east);
\draw[thick] (m-5-1.north) -- (m-5-5.north east) node[right]{$p$};
\end{tikzpicture}

\caption{The $E^2$-page of the spectral sequence $H_p(K(\Z_2,4),\ti\Om^{\bs\Spin}_q(K(\Z_2,3)))$ $\Ra$ $\Om^{\bs\Spin}_{p+q}(\cL K(\Z_2,4);K(\Z_2,4))$, $p+q\le 9$. We will show $d_{6,3}^6=0$ and $d_{5,3}^5=0$, so this is also the $E^\iy$-page for $p+q\le 8$.}
\label{fm18fig2}
\end{figure}

The obvious map $K(\Z,4)\ra K(\Z_2,4)$ induces a morphism of spectral sequences from \eq{fm18eq1} to \eq{fm18eq3}. A comparison of \eq{fm3eq49} and \eq{fm3eq52} shows that $\ti\Om^{\bs\Spin}_3(K(\Z,3))\ra\ti\Om^{\bs\Spin}_3(K(\Z_2,3))$ can be identified with the projection $\Z\ra\Z_2$. Similarly, a comparison of \eq{fm3eq51} and \eq{fm3eq53} shows that $\ti\Om^{\bs\Spin}_8(K(\Z,3))\ra\ti\Om^{\bs\Spin}_8(K(\Z_2,3))$ is an isomorphism. The map $H_6(K(\Z,4),\Z)=\Z_2\an{\varep_6'}\ra H_6(K(\Z_2,4),\Z_2)=\Z_2\an{\varphi_6'}$ is an isomorphism. The fact that $d_{6,3}^6=0$ for the spectral sequence \eq{fm18eq1} therefore implies $d_{6,3}^6=0$ in Figure~\ref{fm18fig2}.

Recall that classes $[X,\bar\al]$ in $\Om_n^{\bs\Spin}(\cL K(\Z_2,4);K(\Z_2,4))$ are represented by a compact spin $n$-manifold $X$ and cohomology class $\bar\al\in H^4(X\t \cS^1,\Z_2)$, which we can decompose as $\bar\al=\bar\be\bt \bar s+\bar\ga\bt 1$, where $\bar s\in H^1(\cS^1,\Z_2)$ denotes the generator. Consider the morphism
\e
\label{fm18eq4}
\begin{aligned}
 \Om_8^{\bs\Spin}(\cL K(\Z_2,4);K(\Z_2,4))&\longra \Z_2^2,\\
[X,\bar\al]&\longmapsto \Bigl(\ts\int_X \bar\be\cup\Sq^2(\bar\be), \ts\int_X \Sq^1(\bar\be)\cup\bar\ga\Bigr).
\end{aligned}
\e
The following two examples prove the surjectivity of \eq{fm18eq4}:
\begin{itemize}
\item The Lie group $X=\SU(3)$ is a compact spin $8$-manifold with cohomology $H^*(\SU(3),\Z_2)=\Z_2[\bar b_2,\bar b_3]/(\bar b_2^2, \bar b_3^2)$, where $\bar b_2$, $\bar b_3$ have degrees $3$ and $5$. We have $\Sq^1(\bar b_2)=0$ and $\Sq^2(\bar b_2)=\bar b_3$. Letting $\bar\be=\bar b_2$, $\bar\ga=0$, we see that \eq{fm18eq4} maps $[X,\bar\al]\mapsto (1,0)$.
\item $X=\RP^7\t \cS^1_{\rm b}$ is a compact spin $8$-manifold with cohomology $H^*(\RP^7\t \cS^1,\Z_2)=\Z_2[\bar t,\bar s]/(\bar t^8,\bar s^2)$. We have $\Sq^1(\bar t^3)=\bar t^4$, $\Sq^2(\bar t^3)=\bar t^5$. Letting $\bar\be=\bar t^3\ot 1$, $\bar \ga=\bar t^3\ot \bar s$, we see that \eq{fm18eq4} maps $[X,\bar\al]\mapsto (0,1)$.
\end{itemize}
The fact that \eq{fm18eq4} is surjective implies that
\e
\label{fm18eq5}
 d_{5,3}^5=0,\qquad \Om_8^{\bs\Spin}(\cL K(\Z_2,4);K(\Z_2,4))\cong\Z_2^2.
\e

The image of the second example under $\hat\xi^{\bs\Spin}_8(K(\Z_2,4))$ is zero: $\RP^7\t\cS^1_{\rm b}\t\cS^1_{\rm b}$ with the class $\bar\al=(\bar t^3\ot 1)\ot \bar s + (\bar t^3\ot \bar s)\ot 1=\bar t^3\ot (\bar s\ot 1+1\ot \bar s)$ is nullbordant because there is a spin diffeomorphism (the change of coordinates $(x,y)\mapsto (x+y,x)$ of the torus $\cS^1_{\rm b}\t \cS^1_{\rm b}$) that takes this example to $[\RP^7\t\cS^1_{\rm b}\t\cS^1_{\rm b},\bar t^3\t \bar s\t 1]$, which bounds the spin $8$-manifold $\RP^7\t\cS^1_{\rm b}\t D^2$ with class $\bar t^3\t \bar s\t 1$.

Since the second example, $[X,\bar \al]$ with $\RP^7\t \cS^1_{\rm b}$, maps to $(0,1)$ under the isomorphism \eq{fm18eq4} and $\hat\xi^{\bs\Spin}_8(K(\Z_2,4)):\Om_8^{\bs\Spin}(\cL K(\Z_2,4);K(\Z_2,4))\cong\Z_2^2\ra\ti\Om_9^{\bs\Spin}(K(\Z,4))\cong\Z_2$ is non-trivial, it must be the projection onto the first component of the isomorphism \eq{fm18eq4}. This proves Theorem \ref{fm3thm2}(c) for $K(\Z_2,4)$. Since $\ti\Om_9^{\bs\Spin}(K(\Z,4))\ra\ti\Om_9^{\bs\Spin}(K(\Z_2,4))$ is an isomorphism by Table \ref{fm3tab1}, the same description holds also for $\hat\xi^{\bs\Spin}_8(K(\Z,4))$.

Figure \ref{fm18fig2} and \eq{fm18eq5} imply Table \ref{fm3tab2}: The image of $\hat\xi^{\bs\Spin}_{n-1}(K(\Z_2,4))$ is at least as large as the image of $\hat\xi^{\bs\Spin}_{n-1}(K(\Z,4))$, which verifies Table \ref{fm3tab3} except for the case $n=8$, where $\Im\hat\xi^{\bf\Spin}_7(K(\Z_2,4))$ contains at least the subgroup $\Z_2\an{2\ze_2}\subset\Z_4\an{\ze_2}$ of index two. On the other hand, the image is a subgroup of index at least two. To see this, recall that the Pontrjagin square defines an isomorphism $\Om_8^{\bs\Spin}(K(\Z_2,4))\cong\Z_4$, $[Y,\bar\al]\mapsto\int_Y\cP(\bar\al)$. If $[X,\bar\al]\in \Om_7^{\bs\Spin}(\cL K(\Z_2,4);K(\Z_2,4))$, where $\bar\al=\bar\be\ot \bar s+\bar\ga\ot 1$, then by Definition \ref{fm2def4}(c) we have $\int_{X\t \cS^1}\cP(\bar\al)\bmod 2=\int_{X\t \cS^1} \bar\al\cup\bar\al$ which by \eq{fm18eq2} integrates to $0$ over $X\t \cS^1$. In other words, every element in $\Im\hat\xi^{\bs\Spin}_7(K(\Z_2,4))$ has $\int_Y\cP(\bar\al)\in\Z_4$ even. This completes the verification of Table~\ref{fm3tab3}.
 
\subsection{\texorpdfstring{Computation of $\Om_n^{\bs{\mathrm{Spin}}}(\cL M\SU(2);M\SU(2))$}{Computation of Ωₙˢᵖⁱⁿ(ℒMSU(2);MSU(2))}}
\label{fm183}

There is a spectral sequence
\e
 H_p\bigl(M\SU(2),\ti\Om_q^{\bs\Spin}(\Om M\SU(2))\bigr)\Longrightarrow\Om_{p+q}^{\bs\Spin}(\cL M\SU(2);M\SU(2)),\label{fm18eq6}
\e
so we determine $\ti\Om_q^{\bs\Spin}(\Om M\SU(2))$ first. By Theorem \ref{fm3thm3}, there is a homotopy equivalence $\Om M\SU(2)\simeq\Om B\SU(2)\simeq\SU(2)\cong\cS^3$, so $\ti\Om_q^{\bs\Spin}(\Om M\SU(2))\cong\ti\Om_q^{\bs\Spin}(\cS^3)\cong\Om_{q-3}^{\bs\Spin}(*)$ for all $q$. This yields the $E^2$-page of the spectral sequence \eq{fm18eq6} shown in Figure \ref{fm18fig3}. All differentials $d_{p,q}^r$ with $p+q\le 9$, $r\ge 2$ vanish, for degree reasons. Thus the $E^2$-page is also the $E^\iy$-page for $p+q\le 8$. All the extension problems are trivial and so we obtain Table \ref{fm3tab2} for $M\SU(2)$.

\begin{figure}[htb]
\centering
\begin{tikzpicture}
\matrix (m) [matrix of math nodes,
nodes in empty cells,nodes={minimum width=11.8ex,
minimum height=5ex,outer sep=-5pt},
column sep=.5ex,row sep=1ex]{
7 & \Z \\
5 & \Z_2 &\Z_2\an{\bar\tau}\\
4 & \Z_2 &\Z_2\an{\bar\tau}\\
3 & \Z & \Z\an{\tau}\\
\quad\strut & 0 & 4\\};
\draw[thick] (m-1-1.north east) node[above]{$q$} -- (m-5-1.east);
 \draw[thick] (m-5-1.north) -- (m-5-3.north east) node[right]{$p$};
\end{tikzpicture}
\caption{$E^2=E^\iy$-page of $H_p(M\SU(2),\ti\Om^{\bs\Spin}_q(\SU(2)))\Ra\ab\Om^{\bs\Spin}_{p+q}(\cL M\SU(2);M\SU(2))$, $p+q\le 9$}
\label{fm18fig3}
\end{figure}

Since $\de$ to $\Om^{\bs\Spin}_3(\cL M\SU(2);M\SU(2))$, the map $\hat\xi^{\bs\Spin}_{n-1}(M\SU(2))$ is clearly surjective for $n=4,5,6$. Moreover, the image of $\hat\xi^{\bs\Spin}_7(M\SU(2))$ has index two: Indeed, an element of $\Om^{\bs\Spin}_7(\cL M\SU(2);M\SU(2))$ is represented by a pair $[X,M]$ where $X$ is a compact spin $7$-manifold and $M$ is a $4$-dimensional submanifold $M\subset X\t\cS^1$ with an $\SU(2)$-structure on its normal bundle. By \eq{fm3eq24}, we have $\int_M c_2(\nu_M)=\int_X \al\cup\al$ where $\al\in H^4(M\t\cS^1,\Z)$ is Poincar\'e dual to $[M]$. The calculation \eq{fm18eq2} shows that $\int_M c_2(\nu_M)$ is always even. The isomorphism \eq{fm3eq10} now proves that $\hat\xi^{\bs\Spin}_7(M\SU(2))$ has index at least two.  Obviously, $\ze_1$ can be lifted to $\Om_7^{\bs\Spin}(\cL M\SU(2);M\SU(2))$ and, according to the following lemma, $2\ze_2$ can also be lifted. This proves Table \ref{fm3tab3} for the case $M\SU(2)$.

\begin{lem}
\label{fm18lem3}
We can lift\/ $2\ze_2\in\ti\Om_8^{\bs\Spin}(M\SU(2))$ along\/ $\hat\xi^{\bs\Spin}_7(M\SU(2))$.
\end{lem}

\begin{proof}
Lemma \ref{fm18lem2} shows that $2\ze_2$ can be lifted to $\Om_7^{\bs\Spin}(\cL M\U(2);M\U(2))$, which we improve here to $\Om_7^{\bs\Spin}(\cL M\SU(2);M\SU(2))$. Unfortunately, we do not know an explicit construction of the lift, so we give an abstract argument.

The inclusion $\SU(2)\ra\U(2)$ induces a map $\psi:M\SU(2)\ra M\U(2)$, which leads to a commutative diagram
\begin{equation*}
\begin{tikzcd}[column sep=huge]
	\Om_7^{\bs\Spin}(\cL M\SU(2);M\SU(2))\rar{\hat\xi^{\bs\Spin}_7(M\SU(2))}\dar{(\cL\psi)_*} & \Om_8^{\bs\Spin}(M\SU(2))=\Z\an{\ze_1,\ze_2}\dar{\psi_*}\\
	\Om_7^{\bs\Spin}(\cL M\U(2);M\U(2))\rar{\hat\xi^{\bs\Spin}_7(M\U(2))} & \Om_8^{\bs\Spin}(M\U(2))=\Z\an{\frac{\ze_1}{2},\ze_2,\ze_3}.
\end{tikzcd}
\end{equation*}
The map $\psi$ also induces a morphism between the spectral sequences \eq{fm18eq6} and \eq{fm18eq7} below so, in particular, a morphism between the $E^\iy$-pages of these spectral sequences, which are shown in Figure \ref{fm18fig3} and Figure \ref{fm18fig4} below. These imply that the following extension problems are mapped onto each other:
\begin{equation*}
\begin{tikzcd}
	0\rar & F_1=\Z\an{\ze_1}\rar\dar{(\Om\psi)_*} & \Om_7^{\bs\Spin}(\cL M\SU(2);M\SU(2))\rar\dar{(\cL\psi)_*} & \Z\rar\dar{\cong} & 0\\
	0\rar & F_1'=\Z\an{\frac{\ze_1}{2},\ze_3}\rar & \Om_7^{\bs\Spin}(\cL M\U(2);M\U(2))\rar & \Z\rar & 0.\!
\end{tikzcd}
\end{equation*}
Hence $\Om_7^{\bs\Spin}(\cL M\SU(2);M\SU(2))$ and $\Om_7^{\bs\Spin}(\cL M\U(2);M\U(2))$ are the extensions obtained from $F_1$ and $F_1'$, respectively, by adjoining one generator $\tau$, which can chosen in $\Om_7^{\bs\Spin}(\cL M\SU(2);M\SU(2))$. Since $2\ze_2$ is in the image of $\hat\xi^{\bs\Spin}_7(M\U(2))$ and since $\hat\xi^{\bs\Spin}_7(M\U(2))$ maps $F_1'$ to the complement of $\ze_2$, by the commutative diagram it must be that $\hat\xi^{\bs\Spin}_7(M\U(2))(\tau)\in 2\ze_2+\Z\an{\ze_1}$, hence $\hat\xi^{\bs\Spin}_7(M\SU(2))(\tau)\in 2\ze_2+\Z\an{\ze_1}$. Since $\ze_1$ can obviously be lifted along $\hat\xi^{\bs\Spin}_7(M\SU(2))$, we conclude $2\ze_2$ can also be lifted along $\hat\xi^{\bs\Spin}_7(M\SU(2))$.
\end{proof}

\subsection{\texorpdfstring{Computation of $\Om_n^{\bs{\mathrm{Spin}}}(\cL M\U(2);M\U(2))$}{Computation of Ωₙˢᵖⁱⁿ(ℒMU(2);MU(2))}}
\label{fm184}

There is a spectral sequence
\e
 H_p\bigl(M\U(2),\ti\Om_q^{\bs\Spin}(\Om M\U(2))\bigr)\Longrightarrow\Om_{p+q}^{\bs\Spin}(\cL M\U(2);M\U(2)),\label{fm18eq7}
\e
so we determine $\ti\Om_q^{\bs\Spin}(\Om M\U(2))$ first. By Theorem \ref{fm3thm3} there is a $9$-connected map $\Om M\U(2)\ra\Om B\SU\simeq\SU$, so $\ti\Om_q^{\bs\Spin}(\Om M\SU(2))\cong\ti\Om_q^{\bs\Spin}(\SU)$ for $q\le 9$.

By using Table \ref{fm3tab6} we obtain the $E_2$-page of the spectral sequence \eq{fm18eq7} for $p+q\le 9$ as shown in Figure \ref{fm18fig4}.

\begin{figure}[htb]
\centering
\begin{tikzpicture}
\matrix (m) [matrix of math nodes,
nodes in empty cells,nodes={minimum width=11.8ex,
minimum height=5ex,outer sep=-5pt},
column sep=1ex,row sep=1ex]{
8 & \Z\an{1}   \\
7 & \Z^2\an{1}   \\
5 & \Z\an{1} & \Z\an{\tau}  \\
3 & \Z\an{1} & \Z\an{\tau} & \Z\an{\ga_1^T} \\
\quad\strut & 0 & 4 & 6 \strut \\};
\draw[-stealth] (m-3-3) -- node[left]{\scriptsize $d^4_{4,5}$}  (m-1-2);
\draw[-stealth] (m-4-4) -- node[right]{\scriptsize $d^6_{6,3}$}  (m-1-2);
\draw[thick] (m-1-1.north east) node[above]{$q$} -- (m-5-1.east);
 \draw[thick] (m-5-1.north) -- (m-5-4.north east) node[right]{$p$};
\end{tikzpicture}
\caption{$E^2=E^\iy$-page of $H_p\bigl(M\U(2),\ti\Om_q^{\bs\Spin}(\Om M\U(2))\bigr)$ $\Rightarrow\Om_{p+q}^{\bs\Spin}(\cL M\U(2);M\U(2))$, $p+q\le 9$}
\label{fm18fig4}
\end{figure}

The only possible higher differentials in this region are $d^4_{4,5},d^6_{6,3}$. We claim that these vanish, so the $E^2$-page is also the $E^\iy$-page for $p+q\le 8$. All extension problems are trivial, which proves Table \ref{fm3tab2} for the case $M\U(2)$. It only remains to prove that $d^4_{4,5},d^6_{6,3}$ both vanish. It suffices to show
\begin{equation*}
\Om^{\bs\Spin}_8(\cL M\U(2);M\U(2))\ot\Q\cong\Om^{\bs\Spin}_8(\cL B\SU;B\SU)\ot\Q\cong\Q,
\end{equation*}
where the first isomorphism is by Theorem~\ref{fm3thm3}.

Recall from Dold \cite{Dold} that over the rational numbers we have a Chern--Dold character isomorphism of homology theories
\begin{equation*}
 \Om_n^{\bs\Spin}(X;A)\ot\Q\cong H_n(X;A,\Q[t]) = \bigoplus\nolimits_{i\ge 0} H_{n-4i}
 (X;A,\Q),
\end{equation*}
where $\bigoplus_{n\in\N}\Om_n^{\bs\Spin}(*)\ot\Q=\Q[t]$ is a polynomial ring with a variable $t$ of degree $4$, namely $t=\al_4\ot 1\in \Om_4^{\bs\Spin}(*)\ot\Q$. In particular, the Chern--Dold character induces a morphism of spectral sequences and we can equivalently show $H_8(\cL B\SU;B\SU,\Q[t])\cong\Q$.

\begin{lem}
\label{fm18lem4}
Let\/ $(X,\mu,e)$ be an H-space structure. Then\/ $X\t\Om X$ is weakly homotopy equivalent to the free loop space\/ $\cL X$. Hence by the K\"unneth Theorem,
\begin{equation*}
 H_*(\cL X;X,\Q)\cong H_*(X,\Q)\ot \ti H_*(\Om X,\Q).
\end{equation*}
\end{lem}

\begin{proof}
There is a commutative diagram
\begin{equation*}
 \begin{tikzcd}
  \Om X\rar{\id_{\Om X}}\dar{\io_2} & \Om X\dar\\
  X\t\Om X\dar{\pi_1}\rar{\ze} & \cL X\dar{\ev_1}\\
  X\rar{\id_X} & X
 \end{tikzcd}
\end{equation*}
where $\ze(x,\ga)$ for $x\in X$, $\ga\in\Om X$ is defined (using the $H$-space structure) to be the (possible unbased) loop $t\mapsto \mu(x,\ga(t))$. The columns of this diagram are fibrations and hence induce a pair of long exact sequences of homotopy groups
\begin{equation*}
\begin{tikzcd}
 \cdots\rar & \pi_n(\Om X)\rar\dar{\id_{\pi_n(\Om X)}} & \pi_n(X\t\Om X)\dar{\ze_*}\rar & \pi_n(X)\dar{\id_{\pi_n(X)}}\rar & \cdots\\
 \cdots\rar & \pi_n(\Om X)\rar & \pi_n(\cL X)\rar & \pi_n(X)\rar & \cdots,
\end{tikzcd}
\end{equation*}
so $\pi_n(\ze):\pi_n(X\t\Om X)\ra\pi_n(\cL X)$ is an isomorphism by the 5-lemma.
\end{proof}

By applying Lemma \ref{fm18lem4} to $X=B\SU$ and $\Om X\simeq\SU$, we find
\begin{equation*}
 H_*(\cL B\SU;B\SU,\Q[t])\cong H_*(B\SU,\Q)\ot \ti H_*(\SU,\Q)\ot\Q[t],
\end{equation*}
where $\ti H_*(\SU,\Q)$ consists of polynomials $P(b_2,b_3,\cdots)$ with zero constant term. Hence $H_8(\cL B\SU; B\SU,\Q[t])=\Q\an{b_2b_3}$ is non-zero, as claimed.

\subsection{\texorpdfstring{Computation of $\Om_n^{\bs{\mathrm{Spin}}}(\cL M\Spin(4);M\Spin(4))$}{Computation of Ωₙˢᵖⁱⁿ(ℒMSpin(4);MSpin(4))}}
\label{fm185}

There is a spectral sequence
\e
 H_p\bigl(M\Spin(4),\ti\Om_q^{\bs\Spin}(\Om M\Spin(4))\bigr)\Longrightarrow\Om_{p+q}^{\bs\Spin}(\cL M\Spin(4);M\Spin(4)),\label{fm18eq8}
\e
so we determine $\ti\Om_q^{\bs\Spin}(\Om M\Spin(4))$ first. By Theorem \ref{fm3thm3} there is an $11$-connected map $\Om M\Spin(4)\ra\Om B\Sp\simeq\Sp$, so $\ti\Om_q^{\bs\Spin}(\Om M\SU(2))\cong\ti\Om_q^{\bs\Spin}(\Sp)$ for $q\le 11$. These groups were determined in Table \ref{fm3tab6}. This leads to the $E^2$-page for $p+q\le 9$ shown in Figure~\ref{fm18fig5}.

\begin{figure}[htb]
\centering
\begin{tikzpicture}
\matrix (m) [matrix of math nodes,
nodes in empty cells,nodes={minimum width=9ex,
minimum height=5ex,outer sep=-5pt},
column sep=.4ex,row sep=1ex]{
9   &   \Z_2\an{\al_1^2\vartheta_2}    &   \\
8   &   \Z_2\an{\al_1\vartheta_2}      &   \\
7   &   \Z\an{\vartheta_1,\vartheta_2}         &   \\
5   &   \Z_2\an{\al_1^2\rho} &   \Z_2\an{\al_1^2\rho}\\
4   &   \Z_2\an{\al_1\rho}   &   \Z_2\an{\al_1\rho}\\
3   &   \Z\an{\rho}          &   \Z\an{\rho}\\
\quad\strut & 0 & 4 \\};
\draw[thick] (m-1-1.north east) node[above]{$q$} -- (m-7-1.east);
\draw[thick] (m-7-1.north) -- (m-7-3.north east) node[right]{$p$};
\end{tikzpicture}
\caption{$E^2$-page of $\ti H_p(M\Spin(4),\Om^{\bs\Spin}_q(\Om M\Spin(4)))$ $\Rightarrow$ 
$\ti\Om^{\bs\Spin}_{p+q}(\cL M\Spin(4);M\Spin(4))$, $p+q\le 9$. \\ This is also the $E^\iy$-page for $p+q\le 9$.}
\label{fm18fig5}
\end{figure}

All differentials $d^r_{p,q}$ with $p+q\le 9$, $r\ge 2$ vanish, so this is also the $E^\iy$-page for $p+q\le 8$. Clearly the extensions problems are all trivial for $p+q\le 7$. To see that the extension problem is also trivial for $p+q=8$, we argue as follows. From the filtration (the first column in Figure \ref{fm18fig5} corresponds to the inclusion of the fibre $\Om M\Spin(4)\simeq_{11}\Sp$) and the $E^\iy$-page we obtain a commutative diagram with exact rows:
\begin{equation*}
 \begin{tikzcd}[column sep=4.5ex]
     0\rar  & \ti\Om_7^{\bs\Spin}(\Sp)\rar\dar{\al_1} & \Om_7^{\bs\Spin}(\cL M\Spin(4);M\Spin(4))\rar\dar{\al_1} & \Z\an{\rho}\dar{\mod 2}\rar & 0\\
     0\rar  & \ti\Om_8^{\bs\Spin}(\Sp)\rar & \Om_8^{\bs\Spin}(\cL M\Spin(4);M\Spin(4))\rar & \Z_2\an{\al_1\rho}\rar & 0
 \end{tikzcd}
\end{equation*}
The left vertical map is surjective by Table \ref{fm3tab6} and the right vertical map is also surjective, so the four lemma implies that the middle vertical map is also surjective. Since $2\al_1=0$, every element in the image has order two, which contradicts $\Om_8^{\bs\Spin}(\cL M\Spin(4);M\Spin(4))=\Z_4$. Hence $\Om_8^{\bs\Spin}(\cL M\Spin(4);M\Spin(4))=\Z_2^2$, completing the proof of Table~\ref{fm3tab2}.

\subsection{Proof of Theorem \ref{fm3thm2}(b)}
\label{fm186}

In this section, we show that the image of the map
\begin{equation*}
 \hat\xi^{\bs\Spin}_{n-1}(T):\Om_{n-1}^{\bs\Spin}(\cL T;T)\longra\ti\Om_n^{\bs\Spin}(T)
\end{equation*}
defined in \eq{fm2eq7} is given by Table \ref{fm3tab3}, for the various spaces $T$.

\begin{lem}
\label{fm18lem5}
For\/ $n=8$ and each\/ $T=M\SU(2),$ $M\U(2),$ $M\Spin(4),$ $M\SO(4),$ $K(\Z,4),$ and\/ $K(\Z_2,4),$ the class\/ $\ze_2\in\ti\Om_8^{\bs\Spin}(T)$ is {\bf not} in the image of\/ $\hat\xi^{\bs\Spin}_7(T)$.
\end{lem}

\begin{proof}
Consider the case $T=K(\Z_2,4)$. Elements in $\Im\hat\xi^{\bs\Spin}_7(K(\Z_2,4))$ have the form $[X\t\cS^1_{\rm b},\bar\al]$ for a compact spin $7$-manifold $X$ and $\bar\al\in H^4(X\t\cS^1,\Z_2)$. If we decompose $\bar\al=\bar\be\bt[\cS^1]+\bar\ga\bt 1$ using the K\"unneth isomorphism, we find $\int_{X\t\cS^1}\bar\al\cup\bar\al=2\int_X\bar\be\cup\bar\ga=0$. In particular, the mod $2$ reduction of the Pontrjagin square $\int_{X\t\cS^1}\cP(\bar\al)\bmod2=\int_{X\t\cS^1}\bar\al\cup\bar\al$ vanishes, see Definition \ref{fm2def4}(c), so the isomorphism \eq{fm3eq21} implies that $\ze_2\notin \Im\hat\xi^{\bs\Spin}_7(K(\Z_2,4))$.

Each $T$ as in the statement of the lemma has a morphism $T\ra K(\Z_2,4)$, which determines a commutative diagram
\begin{equation*}
\begin{tikzcd}[column sep=huge]
 \Om_7^{\bs\Spin}(\cL T, T)\rar{\hat\xi^{\bs\Spin}_7(T)}\dar & \ti\Om_8^{\bs\Spin}(T)\dar\\
 \Om_7^{\bs\Spin}(\cL K(\Z_2,4);K(\Z_2,4))\rar{\hat\xi^{\bs\Spin}_7(K(\Z_2,4))} & \ti\Om_8^{\bs\Spin}(K(\Z_2,4)).
\end{tikzcd}
\end{equation*}
According to Table \ref{fm3tab1}, the right vertical morphism maps $\ze_2$ to itself, hence $\ze_2\notin\Im\hat\xi^{\bs\Spin}_7(K(\Z_2,4))$ implies $\ze_2\notin\Im\hat\xi^{\bs\Spin}_7(T)$.
\end{proof}

Next, we show that the groups listed in Table \ref{fm3tab3} are all contained in the image of $\hat\xi^{\bs\Spin}_{n-1}(T)$ using the following lemma.

\begin{lem}
\label{fm18lem6}
The following elements have natural lifts along\/ $\hat\xi^{\bs\Spin}_{n-1}$.
\begin{itemize}
\setlength{\itemsep}{0pt}
\setlength{\parsep}{0pt}
\item[{\bf(a)}] $\de$ to\/ $\Om_3^{\bs\Spin}(\cL M\{1\})$ and
$\varepsilon$ to\/ $\Om_5^{\bs\Spin}(\cL M\U(2);M\U(2));$
\item[{\bf(b)}] $\ze_1$ to\/ $\Om_7^{\bs\Spin}(\cL M\{1\};M\{1\}),$ $\frac{\ze_1}{2}$ to\/ $\Om_7^{\bs\Spin}(\cL M\U(2);M\U(2)),$ and\/ $\frac{\ze_1}{4}$ to\/ $\Om_7^{\bs\Spin}(\cL M\SO(4);M\U(4));$
\item[{\bf(c)}] $2\ze_2$ to\/ $\Om_7^{\bs\Spin}(\cL M\SU(2);M\SU(2));$
\item[{\bf(d)}] $\ze_2\pm\ze_2'$ to\/ $\Om_7^{\bs\Spin}(\cL M\Spin(4);M\Spin(4));$
\item[{\bf(e)}] $\ze_3$ to\/ $\Om_7^{\bs\Spin}(\cL M\U(2);M\U(2));$
\item[{\bf(f)}] $\al_1\ze_2$ to\/ $\Om_8^{\bs\Spin}(\cL M\U(2);M\U(2));$
\item[{\bf(g)}] $\eta$ to\/ $\Om_8^{\bs\Spin}(\cL M\SO(4);M\SO(4))$.
\end{itemize}
\end{lem}

\begin{proof}
All elements of the form $[X\t \cS^1_{\rm b},M]$ are in the image of $\hat\xi^{\bs\Spin}_{n-1}$, so the claim is obvious apart from the cases (c), (d), and (f). Part (c) was already proved in Lemma \ref{fm18lem3}. To prove (d), note that according to \eq{fm3eq58} the composition
\begin{equation*}
\begin{tikzcd}
\ti\Om_7^{\bs\Spin}(\Sp)\rar{\eq{fm3eq55}} & \begin{array}{l}\Om_7^{\bs\Spin}(\cL B\Sp;B\Sp)\cong \\ \Om_7^{\bs\Spin}(\cL M\Spin(4);M\Spin(4))\end{array}\arrow[r,"\hat\xi^{\bs\Spin}_{n-1}(M\Spin(4))",xshift=-3ex] & \Om_8^{\bs\Spin}(M\Spin(4)).
\end{tikzcd}
\end{equation*}
maps $\vartheta_2\mapsto\ze_2-\ze_2'$. By (c), we can lift $2\ze_2$ to $\Om_7^{\bs\Spin}(\cL M\SU(2);M\SU(2))$, which maps to $\Om_7^{\bs\Spin}(\cL M\Spin(4);M\Spin(4))$, and therefore $2\ze_2-(\ze_2-\ze_2')=\ze_2+\ze_2'$ can also be lifted. Part (f) follows from \eq{fm3eq58} as $\upsilon\mapsto\al_1\ze_2$.
\end{proof}

We now complete the proof of Theorem \ref{fm3thm2}(b) and verify Table \ref{fm3tab3}. In view of Table \ref{fm3tab1}, Lemma \ref{fm18lem6}(a) shows that \eq{fm2eq7} is surjective for all $n\le 7$ and all~$T$.

If $n=8$, then Lemma \ref{fm18lem5} implies the image of $\hat\xi^{\bs\Spin}_7(T)$ is a subgroup of index at least two. Moreover, by Lemma \ref{fm18lem6}(b)--(e) the subgroups listed in Table \ref{fm3tab3} are all contained in the image of $\hat\xi^{\bs\Spin}_7(T)$ and are clearly subgroups of index two, so must be the entire image. This proves Theorem \ref{fm3thm2}(b) for $n=8$ and all~$T$.

Let $n=9$. Lemma \ref{fm18lem6}(f) implies $\al_1\ze_2\in\Im\hat\xi^{\bs\Spin}_8(T)$ for $T=M\U(2),\ab M\SO(4),\ab K(\Z,4),\ab K(\Z_2,4)$, and then Table \ref{fm3tab1} shows that $\Im\hat\xi^{\bs\Spin}_8(T)$ is surjective in these cases. Moreover, Lemma \ref{fm18lem6}(g) implies $\eta\in\Im\hat\xi^{\bs\Spin}_8(M\SO(4))$ and, the image of $\hat\xi^{\bs\Spin}_{n-1}(T)$ being closed under multiplication by $\al_1$, Lemma \ref{fm18lem6}(b) implies $\al_1\frac{\ze_1}{4}\in\Im\hat\xi^{\bs\Spin}_8(M\SO(4))$. Hence $\hat\xi^{\bs\Spin}_8(M\SO(4))$ is surjective by Table~\ref{fm3tab1}.

It remains to prove Table \ref{fm3tab3} for $T=M\SU(2)$ and $T=M\Spin(4)$. Consider the commutative diagram induced by $\psi:M\SU(2)\ra M\U(2)$,
\begin{equation*}
\begin{tikzcd}[column sep=huge]
 \Om_8^{\bs\Spin}(\cL M\SU(2);M\SU(2))\rar{\hat\xi^{\bs\Spin}_8(M\SU(2))}\dar{(\cL \psi)_*} & \ti\Om_9^{\bs\Spin}(M\SU(2))\dar{\psi_*}\\
 \Om_8^{\bs\Spin}(\cL M\U(2);M\U(2))\rar{\hat\xi^{\bs\Spin}_8(M\U(2))} & \ti\Om_9^{\bs\Spin}(M\U(2)).
\end{tikzcd}
\end{equation*}
According to Theorem \ref{fm3thm2}(a) we have $\Om_8^{\bs\Spin}(\cL M\SU(2);M\SU(2))\cong\Z_2$ and $\Om_8^{\bs\Spin}(\cL M\U(2);M\U(2))\cong\Z$, so the left vertical morphism $(\cL\psi)_*$ in the diagram vanishes. By the commutativity of the diagram, $\psi_*(\Im\hat\xi^{\bs\Spin}_8(M\SU(2)))=\Im(\hat\xi^{\bs\Spin}_8(M\U(2)\circ(\cL\psi)_*)=0$. The right vertical map $\psi_*$ in the diagram is an isomorphism by Table \ref{fm3tab1}, hence $\Im\hat\xi^{\bs\Spin}_8(M\SU(2))=0$.

For $T=M\Spin(4)$, note in \S\ref{fm185} that $\al_1:\Om_8^{\bs\Spin}(\cL M\Spin(4);M\Spin(4))\ab\ra\ab\Om_9^{\bs\Spin}(\cL M\Spin(4);M\Spin(4))$ is surjective. We know $\Im\hat\xi^{\bs\Spin}_7(M\Spin(4))=\Z\an{\ze_1,\ze_2-\ze_2',\ze_2+\ze_2'}$ and have $\al_1\ze_1=0$ by \eq{fm3eq8}, hence $
\Im\hat\xi^{\bs\Spin}_8(M\Spin(4))=\al_1\Im\hat\xi^{\bs\Spin}_7(M\Spin(4))=\Z_2\an{\al_1(\ze_2+\ze_2')}$.

\section{Proofs of theorems in \S\ref{fm9}--\S\ref{fm11}}
\label{fm19}

\subsection{Proof of Theorem \ref{fm9thm1}}
\label{fm191}

For (a), let $X$ be a compact $n$-manifold with a $\bs B$-structure. By Definition \ref{fm9def2}, $\sO$ is orientable for $X$ if and only if there exists a natural isomorphism $\eta_X$ in the diagram
\begin{equation*}
\xymatrix@C=55pt@R=10pt{
\ddrrtwocell_{}\omit^{}\omit{_{\,\,\,\,\,\eta_X}} & {\Bord^{\bs B}_n(BG)} \ar[dr]^{\sO}  \\
*+[r]{\Bord_X(BG)} \ar[ur]^{\Pi_X^{\bs B}}  \ar[dr]_\boo &   & *+[l]{A\qs B} \ar[dl]^{F_{A\qs B}^{0\qs B}} \\
& {0\qs B.} & }
\end{equation*}
As $B$ is abelian, the existence of $\eta_X$ is equivalent to the commuting of the diagram for each object $P$ in $\Bord_X(BG)$:
\e
\begin{gathered}
\xymatrix@C=55pt@R=10pt{
& {\Aut_{\Bord^{\bs B}_n(BG)}(X,P)} \ar[dr]^{\sO} \\
*+[r]{\Aut_{\Bord_X(BG)}(P)} \ar[ur]^{\Pi_{X,P}^{\bs B}}  \ar[dr]_{\ul{0}} && *+[l]{B} \ar[dl]^{\id_{B}} \\
& {B.} }
\end{gathered}
\label{fm19eq1}
\e

We extend \eq{fm19eq1} to the diagram:
\e
\begin{gathered}
\xymatrix@!0@C=40pt@R=30pt{
*+[r]{\Om^{\bs B}_n(\cL BG;BG)} \ar[rrrr]^(0.6){\hat\xi^{\bs B}_n(BG)} &&&& *+[l]{\Om_{n+1}^{\bs B}(BG)}
\ar[d]_{\eq{fm4eq5}}^\cong \ar[drrr]^{\eq{fm4eq4}}_(0.3)\cong
\\
*+[r]{\Om^{\bs B}_n(\cL BG)} \ar[u]_{\Pi^{\bs B}_n(BG)} \ar[urrrr]^(0.4){\xi^{\bs B}_n(BG)}
&&&& *+[l]{\Aut_{\Bord^{\bs B}_n(BG)}(X,P)} \ar[drrr]_{\sO} &&& *+[l]{\Aut_{\Bord^{\bs B}_n(BG)}(\boo)} \ar[lll]_(0.75){\ot\id_{(X,P)}} \ar[d]^{\sO} \\
*+[r]{\Aut_{\Bord_X(BG)}(P)} \ar[u]_{\chi_P^{\bs B}} \ar[urrrr]^(0.4){\Pi_{X,P}^{\bs B}}  \ar[drrrr]_(0.4){\ul{0}} &&&&&&& *+[l]{B} \ar[dlll]^{\id_{B}} \\
&&&& *+[l]{B.} }
\end{gathered}
\label{fm19eq2}
\e
Here the top left triangle commutes by \eq{fm2eq7}, the top left parallelogram commutes by \eq{fm4eq15}, the top right triangle commutes by \eq{fm4eq6}, and the bottom right triangle commutes as $\sO$ is a monoidal functor.

We know that all of \eq{fm19eq2} commutes except possibly the bottom parallelogram. The route clockwise round the outside of \eq{fm19eq2} from $\Om^{\bs B}_n(\cL BG)$ to $B$ appears in \eq{fm9eq3} as a route from $\Om^{\bs B}_n(\cL BG)$ to $B$. If $\Xi_{n,\sO}^{\bs B,G}\equiv\ul{0}$ then the composition of this route in \eq{fm9eq3} is $\ul{0}$, since \eq{fm9eq3} commutes. Thus the outside of \eq{fm19eq2} commutes. This forces the bottom parallelogram of \eq{fm19eq2} to commute, so $\sO$ is orientable for $X$. This proves the `if' part of~(a).

For the `only if', suppose $\sO$ is orientable for every compact $n$-manifold $X$ with $\bs B$-structure. Any element $\ti\om$ of $\Om^{\bs B}_n(\cL BG;BG)$ can be lifted through $\Pi^{\bs B}_n(BG)$ to some $\om\in\Om^{\bs B}_n(\cL BG)$, and we may then choose an object $(X,\bar Q)$ in $\Bord^{\bs B}_n(\cL BG)$ whose class in $\pi_0(\Bord^{\bs B}_n(\cL BG))$ is identified with $\om$ under \eq{fm4eq11}. Then $\bar Q\ra X\t\cS^1$ is a principal $G$-bundle. Define $P\ra X$ by $P=\bar Q\vert_{X\t\{0\}}$. Define $Q\ra X\t[0,1]$ to be the pullback of $\bar Q\ra X\t\cS^1$ by $\id_X\t\pi$, where $\pi:[0,1]\ra\cS^1=\R/\Z=[0,1]/(0\!\sim\! 1)$ is the projection. Then $Q\vert_{X\t\{0\}}=Q\vert_{X\t\{1\}}=P$, so $[Q]:P\ra P$ is a morphism in $\Bord_X(BG)$, that is, $[Q]\in\Aut_{\Bord_X(BG)}(P)$. The definition of $\chi_P^{\bs B}$ in Proposition \ref{fm4prop3} implies that $\chi_P^{\bs B}([Q])=\om$.

Since $\sO$ is orientable for $X$, \eq{fm19eq2} commutes. So comparing the two routes round the outside of \eq{fm19eq2} from $[Q]\in\Aut_{\Bord_X(BG)}(P)$ and using $\chi_P^{\bs B}([Q])=\om$ shows that the route clockwise round the outside of \eq{fm19eq2} from $\om\in\Om^{\bs B}_n(\cL BG)$ to $B$ is $\ul{0}$. Hence in \eq{fm9eq3} we see that $\Xi_{n,\sO}^{\bs B,G}\ci\Pi^{\bs B}_n(BG)(\om)=\Xi_{n,\sO}^{\bs B,G}(\ti\om)=\ul{0}$. As this holds for all $\ti\om\in\Om^{\bs B}_n(\cL BG;BG)$, this proves that $\Xi_{n,\sO}^{\bs B,G}\equiv 0$, completing~(a). 

Part (b) follows by a minor modification of the argument above: if $X$ is fixed, $\sO$ is orientable for $X$ if and only if $\Xi_{n,\sO}^{\bs B,G}$ is zero on every element of $\Om^{\bs B}_n(\cL BG;BG)$ of the form $[X,Q]$, for this fixed $X$, where $Q\ra X\t\cS^1$ relates to $P$ in \eq{fm19eq1} by $P=Q\vert_{X\t\{0\}}$ for $0\in \cS^1=[0,1]/(0\sim 1)$. The analogues (i)--(iii) are proved in a very similar way.

\subsection{Proof of Theorem \ref{fm10thm1}}
\label{fm192}

\begin{dfn}
\label{fm19def1}
We will define a new bordism category $\tBord^{\bs\Spin}_7(K(\Z,4))$, and show that it is equivalent to $\Bord^{\bs\Spin}_7(K(\Z,4))$ in Definition~\ref{fm6def1}.
\begin{itemize}
\setlength{\itemsep}{0pt}
\setlength{\parsep}{0pt}
\item[(a)] {\it Objects\/} of $\tBord^{\bs\Spin}_7(K(\Z,4))$ are triples $(X,N,[N]_\fund)$ for $X$ an oriented, spin 7-manifold and $N\subset X$ a compact, oriented 3-submanifold, and $[N]_\fund\in C_3(X,\Z)$ a choice of fundamental cycle for $N$ in homology.
\item[(b)] {\it Morphisms} $[Y,C]:(X_0,N_0,[N_0]_\fund)\ra(X_1,N_1,[N_1]_\fund)$ are equivalence classes of pairs $(Y,M),$ see (c), where $Y$ is a compact, oriented, spin 8-manifold with boundary $\pd Y=-X_0\amalg X_1$ in oriented spin 7-manifolds, and $C\in C_4(Y,\Z)$ is a 4-chain on $Y$ such that $\pd C=-[N_0]_\fund+[N_1]_\fund$.
\item[(c)] In the situation of (b), two choices $(Y_0,C_0)$ and $(Y_1,C_1)$ are {\it equivalent} if there exists a pair $(Z,D),$ where $Z$ is a compact, oriented, spin 9-manifold with corners with an oriented, spin diffeomorphism
\begin{equation*}
\pd Z\cong (-X_0\t[0,1])\amalg (X_1\t[0,1]) \amalg (-Y_0\t\{0\})\amalg (Y_1\t\{1\}),
\end{equation*}
and $D\in C_5(Z,\Z)$ is a 5-chain on $Z$ with
\begin{equation*}
\pd D=-[N_0]_\fund\bt[0,1]_\fund+[N_1]_\fund\bt[0,1]_\fund-C_0+C_1.
\end{equation*}
\item[(d)] If $[Y,C]:(X_0,N_0,[N_0]_\fund)\ra(X_1,N_1,[N_1]_\fund)$ and $[\hat Y,\hat C]:(X_1,\ab N_1,\ab[N_1]_\fund)\ab\ra(X_2,N_2,[N_2]_\fund)$ are morphisms, the {\it composition} is 
\begin{equation*}
[\hat Y,\hat C]\ci[Y,C]=\bigl[\hat Y\amalg_{X_1}Y,\hat C+C\bigr].
\end{equation*}
That is, we glue $Y,\hat Y$ along their common boundary component $X_1$ to make an oriented spin 8-manifold $\hat Y\amalg_{X_1}Y$, and we add the 4-chains $\hat C,C$. Composition is associative.
\item[(e)] {\it Identities\/} are $\id_{(X,N,[N]_\fund)}=\bigl[X\t[0,1],[N]_\fund\bt[0,1]_\fund\bigr]$.
\item[(f)] The {\it monoidal structure\/} on $\tBord^{\bs\Spin}_7(K(\Z,4))$ is defined as
\begin{equation*}
(X,N,[N]_\fund)\ot (\hat X,\hat N,[\hat N]_\fund)=(X\amalg\hat X,N\amalg\hat N,[N]_\fund+[\hat N]_\fund)
\end{equation*}
on objects, and $[Y,C]\ot[\hat Y,\hat C]=[Y\amalg\hat Y,C+\hat C]$ on morphisms.
\item[(g)] The {\it unit object} in $\tBord^{\bs\Spin}_7(K(\Z,4))$ is $\boo=(\es,\es,0).$
\item[(h)] For a pair of objects, the {\it symmetry isomorphism\/} is
\begin{align*}
&\si_{X_0,X_1}=[Y,C]:(X_0,N_0,[N_0]_\fund)\ot(X_1,N_1,[N_1]_\fund)\\
&\qquad \longra (X_1,N_1,[N_1]_\fund)\ot(X_0,N_0,[N_0]_\fund)
\end{align*}
where $Y=(X_0\amalg X_1)\t[0,1]$ and $C=[N_0]_\fund\bt[0,1]_\fund+[N_1]_\fund\bt[0,1]_\fund$, taking the boundary diffeomorphisms to be the obvious identifications $(X_0\amalg X_1)\t\{0\}\cong X_0\amalg X_1$ and $(X_0\amalg X_1)\t\{1\}\cong X_1\amalg X_0$, swapping round factors at~$1\in\pd[0,1]$.
\end{itemize}

Next, we define an orientation functor $\ti\sH_7^\Z:\tBord^{\bs\Spin}_7(K(\Z,4))\ab\ra 0\qs\Z_2$. For each object $(X,N,[N]_\fund)$ in $\tBord^{\bs\Spin}_7(K(\Z,4))$ with $N\ne\es$, define
\ea
\ti\sH_7^\Z(X,N,[N]_\fund)&=\bigl\{F:\bigl\{\text{nonvanishing sections $s$ of $\nu_N\ra N$}\bigr\}\ra\{\pm 1\}
\nonumber\\
&\text{such that $F(s)=(-1)^{d(s,s')}F(s')$ for all $s,s'$}\bigr\},
\label{fm19eq3}
\ea
where $d(s,s')$ is as in \eq{fm10eq1}. Let $s_0$ be a nonvanishing section of $\nu_N\ra N$, and define a function $F_0$ as in \eq{fm19eq3} by $F_0(s)=(-1)^{d(s,s_0)}$. Then from \eq{fm10eq2} we see that $\ti\sH_7^\Z(X,N)=\{F_0,-F_0\}$, so 
$\ti\sH_7^\Z(X,N)$ is a $\Z_2$-torsor under the $\Z_2$ action of multiplying functions $F$ by $\{\pm 1\}$. When $N=\es$ we define $\ti\sH_7^\Z(X,\es)=\{\pm 1\}$.

Now let $[Y,C]:(X_0,N_0,[N_0]_\fund)\ra(X_1,N_1,[N_1]_\fund)$ be a morphism in $\tBord^{\bs\Spin}_7(K(\Z,4))$. We will define an isomorphism of $\Z_2$-torsors
\e
\ti\sH_7^\Z([Y,C]):\ti\sH_7^\Z(X_0,N_0,[N_0]_\fund)\longra \ti\sH_7^\Z(X_1,N_1,[N_1]_\fund).
\label{fm19eq4}
\e
Let $F\in \ti\sH_7^\Z(X_0,N_0,[N_0]_\fund)$. Define
\begin{equation*}
\ti\sH_7^\Z([Y,C])(F):\bigl\{\text{nonvanishing sections $s_1$ of $\nu_{N_1}\ra N_1$}\bigr\}\longra\{\pm 1\}
\end{equation*}
as follows: choose a representative $(Y,C)$ for $[Y,C]$. Choose a nonvanishing section $s_0$ of $\nu_{N_0}\ra N_0$. Choose a vector field $v$ on $Y$ which is tangent to $\pd Y=X_0\amalg X_1$ at $\pd Y$, and which on $N_0\subset X_0$ projects to the section $s_0$ of $\nu_{N_0}$, and on $N_1\subset X_1$ projects to the section $s_1$ of $\nu_{N_1}$. Define $C'\in C_4(Y,\Z)$ by $C'=\exp(\ep v)(C)$ for $\ep>0$ small. That is, we move $C$ a small distance in direction of $v$. As $s_0,s_1$ are nonvanishing, $\exp(\ep v)(N_i)$ is disjoint from $N_i$ in $X_i\subset\pd Y$ for $i=0,1$. Hence $\pd C$ and $C'$ have disjoint support, and $C$ and $\pd C'$ have disjoint support. Therefore we can define the homological intersection $C\bu C'\in\Z$ on the oriented 8-manifold $Y$, generalizing the intersection product $\bu:H_4(Y,\Z)\t H_4(Y,\Z)\ra\Z$. Now define
\e
\ti\sH_7^\Z([Y,C])(F):s_1\longmapsto F(s_0)\cdot(-1)^{C\bu C'}.
\label{fm19eq5}
\e
This also makes sense if $N_0=\es$ or $N_1=\es$, when we replace $F(s_i)$ by $F\in\{\pm 1\}$.

One can show that if we replace $s_0$ by $\hat s_0$, giving $\hat C'$ instead of $C'$, then $C\bu\hat C'=d(\hat s_0,s_0)+C\bu C'$. As $F(\hat s_0)=(-1)^{d(\hat s_0,s_0)}F(s_0)$, we see $\ti\sH_7^\Z([Y,C])(F)(s_1)$ is independent of the choice of $s_0$. Similarly, if we replace $s_1$ by $\ti s_1$, giving $\ti C'$ instead of $C'$, then $C\bu\ti C'=C\bu C'+d(s_1,\ti s_1)$, so $\ti\sH_7^\Z([Y,C])(F)(s_1)=(-1)^{d(s_1,\ti s_1)}F(\ti s_1)$, and $\ti\sH_7^\Z([Y,C])$ does map as in \eq{fm19eq4}. It is also straightforward to show using (c) above that $\ti\sH_7^\Z([Y,C])$ is independent of the representative $(Y,C)$. Hence $\ti\sH_7^\Z([Y,C])$ is well defined.

It is now easy to show that $\ti\sH_7^\Z$ is compatible with composition, and so is a functor. We define a monoidal structure on $\ti\sH_7^\Z$ by, for all $(X_0,N_0,[N_0]_\fund),\ab(X_1,N_1,[N_1]_\fund)$ in $\tBord^{\bs\Spin}_7(K(\Z,4))$, as in \eq{fmAeq2} we define isomorphisms in $0\qs\Z_2$ by
\begin{align*}
\phi_{X_0,X_1}&:\ti\sH_7^\Z(X_0,N_0,[N_0]_\fund)\ot_{\Z_2}\ti\sH_7^\Z(X_1,N_1,[N_1]_\fund)\\
&\qquad \longra \sH_7^\Z(Z_0\amalg X_1,N_0\amalg N_1,[N_0]_\fund+[N_1]_\fund),\\
\phi_{X_0,X_1}&:F_0\ot_{\Z_2}F_1\longmapsto \bigl(s_0\amalg s_1\longmapsto F_0(s_0)\cdot F_1(s_1)\bigr),
\end{align*}
and we define $\phi_{\boo}:\boo_{0\qs\Z_2}=\Z_2\ra \ti\sH_7^\Z(\es,\es,0)=\{\pm 1\}$ to be the usual isomorphism $\ul{n}\mapsto (-1)^{\ul{n}}$. The $\phi_{X_0,X_1}$ commute with symmetric structures, so $\ti\sH_7^\Z:\tBord^{\bs\Spin}_7(K(\Z,4))\ab\ra 0\qs\Z_2$ is a symmetric monoidal functor.
\end{dfn}

\begin{prop}
\label{fm19prop1}
There is an equivalence of Picard groupoids
\e
\Phi:\Bord^{\bs\Spin}_7(K(\Z,4))\,{\buildrel\simeq\over\longra}\,\tBord^{\bs\Spin}_7(K(\Z,4)).
\label{fm19eq6}
\e
and a monoidal natural isomorphism $\eta:\ti\sH_7^\Z\ci\Phi\Ra\sH_7^\Z$.
\end{prop}

\begin{proof}
By Poincar\'e duality, if $X$ is a compact oriented $n$-manifold without boundary we have canonical isomorphisms $H^4(X,\Z)\cong H_{n-4}(X,\Z)$, and if $X$ has boundary then $H^4(X,\Z)\cong H_{n-4}(X,\pd X,\Z)$. In fact, we can use complexes $(C_{n-*}(X,\Z),\pd)$ rather than $(C^*(X,\Z),\d)$ to compute cohomology. 

Since $\Bord^{\bs\Spin}_7(K(\Z,4))$ in Definition \ref{fm6def1} does not depend, up to equivalence of categories, on the cochain model $(C_*(-,\Z),\d)$ used to define cohomology of manifolds, we can define a variant $\hBord^{\bs\Spin}_7(K(\Z,4))$ of $\Bord^{\bs\Spin}_7(K(\Z,4))$ in Definition \ref{fm6def1} in which we replace 4-cochains $C\in H^4(X,\Z)$ for an $n$-manifold $X$ by $(n-4)$-chains $C'\in H_{n-4}(X,\Z)$, and this will give an equivalent Picard groupoid $\hBord^{\bs\Spin}_7(K(\Z,4))\simeq\Bord^{\bs\Spin}_7(K(\Z,4))$.

Next observe that there is an obvious forgetful functor $\tBord^{\bs\Spin}_7(K(\Z,4))\ra\hBord^{\bs\Spin}_7(K(\Z,4))$ mapping $(X,N,[N]_\fund)\!\mapsto\!(X,[N]_\fund)$ on objects and $[Y,C]\ab\mapsto[Y,C]$ on morphisms. As the definitions of morphisms in the two categories are the same, this is an equivalence of $\tBord^{\bs\Spin}_7(K(\Z,4))$ with a subcategory of $\hBord^{\bs\Spin}_7(K(\Z,4))$. An object $(X,B)$ in $\hBord^{\bs\Spin}_7(K(\Z,4))$, so that $B\in C_3(X,\Z)$ with $\pd B=0$, lies in the essential image of $\tBord^{\bs\Spin}_7(K(\Z,4))$ if $[B]=[N]\in H_3(X,\Z)$ for some compact, oriented, embedded submanifold $N\subset X$. But by Thom \cite[Th.~II.27]{Thom}, every class in $H_3(X,\Z)$ is realized by an compact oriented 3-submanifold. So $\tBord^{\bs\Spin}_7(K(\Z,4))\simeq\hBord^{\bs\Spin}_7(K(\Z,4))$. Composing with the equivalence $\hBord^{\bs\Spin}_7(K(\Z,4))\simeq\Bord^{\bs\Spin}_7(K(\Z,4))$ gives the equivalence of Picard groupoids~\eq{fm19eq6}.

To show there is a monoidal natural isomorphism $\eta:\ti\sH_7^\Z\ci\Phi\Ra\sH_7^\Z$, we apply Theorem \ref{fmAthm2}(c). By Proposition \ref{fm6prop1}, we must show three things:
\begin{itemize}
\setlength{\itemsep}{0pt}
\setlength{\parsep}{0pt}
\item[(i)] $\pi_0(\ti\sH_7^\Z\ci\Phi)=\pi_0(\sH_7^\Z)$ in morphisms $\Om_7^{\bs\Spin}(K(\Z,4))\ra 0$;
\item[(ii)] $\pi_1(\ti\sH_7^\Z\ci\Phi)=\pi_1(\sH_7^\Z)$ in morphisms $\Om_8^{\bs\Spin}(K(\Z,4))\ra\Z_2$;
\item[(iii)] The difference class $\om(\ti\sH_7^\Z\ci\Phi,\sH_7^\Z)$ in $H^2_\sym(\Om_7^{\bs\Spin}(K(\Z,4)),\Z_2)$ is zero.
\end{itemize}
Here (i) is clearly trivial, and (iii) is trivial as $\Om_7^{\bs\Spin}(K(\Z,4))=0$ by Tables \ref{fm2tab1} and \ref{fm3tab1}. For (ii), $\pi_1(\ti\sH_7^\Z)$ maps $[Y,\al]\mapsto \al\bu\al\mod 2$, where $Y$ is a compact spin 8-manifold, $\al\in H_4(Y,\Z)$, and $\bu:H_4(Y,\Z)\t H_4(Y,\Z)\ra\Z$ is the intersection product. Also $\pi_1(\Phi)$ maps $[Y,\be]\mapsto[Y,\Pd^{-1}(\be)]$, where $Y$ is a compact spin 8-manifold, $\be\in H^4(Y,\Z)$, and $\Pd^{-1}:H^4(Y,\Z)\ra H_4(Y,\Z)$ is the Poincar\'e duality isomorphism. Therefore $\pi_1(\ti\sH_7^\Z\ci\Phi)$ maps $[Y,\be]\mapsto\int_Y\be\cup\be\mod 2$. But $\pi_1(\sH_7^\Z)$ is defined on the generators of $\ti\Om_8^{\bs\Spin}(K(\Z,4))$ in \eq{fm9eq6}, and from \eq{fm3eq4}--\eq{fm3eq5} and \eq{fm9eq6} we see that $\pi_1(\sH_7^\Z)$ does map $[Y,\be]\mapsto\int_Y\be\cup\be\mod 2$. The proposition follows.
\end{proof}

Now let $X$ be a compact spin 7-manifold. Proposition \ref{fm19prop1} implies that orientations on $X$ for the orientation functors $\sH_7^\Z:\Bord^{\bs\Spin}_7(K(\Z,4))\ra 0\qs\Z_2$ and $\ti\sH_7^\Z:\tBord^{\bs\Spin}_7(K(\Z,4))\ra 0\qs\Z_2$ are equivalent. However, Definition \ref{fm19def1} is designed so that an orientation on $X$ for $\ti\sH_7^\Z$ is equivalent to a flag structure on $X$ in the sense of Definition \ref{fm10def2}. To see this, note that if $(N_0,s_0),(N_1,s_1)$ are disjoint flagged submanifolds in $X$ with $[N_0]=[N_1]$ in $H_3(X,\Z)$, and $[N_i]_\fund$ is a fundamental chain for $N_i$, then we can choose $C\in C_4(X\t[0,1],\Z)$ with $\pd C=-[N_0]_\fund\t\{0\}+[N_1]_\fund\t\{1\}$, and as in Definition \ref{fm19def1} with $Y=X\t[0,1]$ we can choose $C'$ and define the sign $(-1)^{C\bu C'}$ in \eq{fm19eq5}. But we can show that in this case $(-1)^{C\bu C'}=(-1)^{D((N_0,s_0),(N_1,s_1))}$ for $D((N_0,s_0),(N_1,s_1))$ as in Definition \ref{fm10def1}. Theorem \ref{fm10thm1} follows.

\subsection{Proof of Theorem \ref{fm11thm1}}
\label{fm193}

All of part (a) is immediate except the data in Table \ref{fm11tab1}. To prove Table \ref{fm11tab1} when $G=\SU(m)$, observe that for $[X,P]$ in $\Om_8^{\bs\Spin}(B\SU(m))$ we can compute $\pi_1(\sN_7^{\bs\Spin,\SU(m)})([X,P])$ as the index of the positive Dirac operator $\sD^+_X$ on the compact spin $8$-manifold $X$ twisted by the real vector bundle $\Ad(P)$. Let $E\ra X$ be the vector bundle associated to the principal $\SU(m)$-bundle $P\ra X$ with fibre $\C^m$. To compute the numerical index of an elliptic operator, we may pass to its complexification and then we can apply the Atiyah--Singer Index Formula. As $(\Ad(P)\ot\C)\op\C=E^*\ot E$, this gives
\begin{align*}
&\ind\bigl(\sD^+_X\ot\Ad(P)\bigr)=\ts\int_X\hat{A}(X)\ch\bigl(\su(E)\ot\C-\su(\ul\C^m)\ot\C\bigr)\\
&=\ts\int_X\left(1-\frac{p_1(TX)}{24} + \frac{7p_1(TX)^2-4p_2(X)}{5760} \right)\cdot\left(\ch(E^*)\ch(E)-m^2\right).
\end{align*}
Using $c_1(E)=0$ to simplify the expression, we have
\begin{equation*}
\ch(E)\ch(E^*)-m^2=-2m c_2(E)+\frac{(6+m)c_2(E)^2-2m c_4(E)}{6},
\end{equation*}
which we substitute into the index formula and find
\e
\ts\ind\bigl(\sD^+_X\ot\Ad(P)\bigr)=\frac{m}{12}\int_X p_1(TX)c_2(E)\!+\!\frac{m+6}{6}\int_X c_2(E)^2\!-\!\frac{m}{3}\int_X c_4(E).
\label{fm19eq7}
\e

In Theorem \ref{fm11thm1}, we have described $P\ra X$ in terms of a $4$-dimensional submanifold $M\subset X$ and a $\U(2)$-structure on its normal bundle $\nu_M$. We can rewrite the above index formula in terms of integrals over $M$, as follows. Let $\al\in H^4(X,\Z)$ be the cohomology class Poincar\'e dual to $[M]$. Then $c_2(E)=\al$ and $c_4(E)=c_1(\nu_M)^2\al$. Moreover, $p_1(TX)\vert_X=p_1(\nu_M)+p_1(TM)$ and $p_1(\nu_M)=-c_2(\nu_M\ot\C)=-c_2(\nu_M\op\overline{\nu_M})=-2c_2(\nu_M)+c_1(\nu_M)^2$. Also, $\al^2\vert_X=\al\cup e(\nu_M)=\al\cup c_2(\nu_M)$. Inserting all this into \eq{fm19eq7} and using the Hirzebruch Signature Theorem gives
\begin{align*}
\ind \bigl(\sD^+_X\ot\Ad(P)\bigr) &=
\begin{aligned}[t]
&\ts\frac{m}{12}\int_M \bigl(p_1(TM)-2c_2(\nu_M)+c_1(\nu_M)^2\bigr)\\
&\ts+\frac{m+6}{6}\int_M c_2(\nu_M)-\frac{m}{3}\int_M c_1(\nu_M)^2
\end{aligned}
 \\
 &=\ts\frac{m}{4}\sign(M)+\int_M c_2(\nu_M)-\frac{m}{4}\int_M c_1(\nu_M)^2.
\end{align*}
Recall from Theorem \ref{fm3thm1}(d) that \eq{fm3eq13} maps $\frac{\ze_1}{2}\mapsto(1,0,0),$ $\ze_2\mapsto (0,1,0),$ and $\ze_3\mapsto(0,0,1)$. By the previous formula, $\ind \bigl(\sD^+_X\ot\Ad(P)\bigr)$ is the inner product of \eq{fm3eq13} with the vector $(-2m,1,0)$. This verifies Table \ref{fm11tab1} in the case~$G=\SU(m)$.

The proof of Table \ref{fm11tab1} in the case $G=\Sp(m)$ is similar, and we explain the details next. For $[X,P]$ in $\Om_8^{\bs\Spin}(B\Sp(m))$ let $E\ra X$ be the quaternionic vector bundle associated to $P\ra X$ with fibre $\H^m$. For $\Sp(m)$, the complexified adjoint representation is the second symmetric power of the defining representation, therefore $\Ad(P)\ot\C=S^2_\C E$. Viewing $E$ as a complex vector bundle, we have
\begin{equation*}
\ch S^2_\C E=m(2m+1)-(2m+2)c_2(E)+\frac{m+7}{6}c_2(E)^2-\frac{m+4}{3}c_4(E).
\end{equation*}
The Atiyah--Singer Index Formula therefore gives
\begin{align}
\label{fm19eq8}
\ind \bigl(\sD^+_X\ot\Ad(P)\bigr)&=
\ts\int_X
\begin{aligned}[t]
&\ts\Bigl(1-\frac{p_1(TX)}{24} + \frac{7p_1(TX)^2-4p_2(X)}{5760} \Bigr)\cdot{}\\
&\ts\bigl(-(2m\!+\!2)c_2(E)\!+\!\frac{m+7}{6}c_2(E)^2\!-\!\frac{m+4}{3}c_4(E)\bigr)
\end{aligned}\\
&=\ts\frac{m+1}{12}\int_X p_1(TX)c_2(E)+\frac{m+7}{6}\int_Xc_2(E)^2-\frac{m+4}{3}\int_X c_4(E).\nonumber
\end{align}
In Theorem \ref{fm11thm1}, we have described $P\ra X$ in terms of a $4$-dimensional submanifold $M\subset X$ and a $\Spin(4)$-structure on its normal bundle $\nu_M$, with spinor bundles $\Si_{\nu_M}^\pm$. We can rewrite the above index formula in terms of integrals over $M$, as follows. We have
\begin{equation*}
c_2(E)=\al,\qquad c_2(E)^2=\bigl(c_2(\Si_{\nu_M}^+)-c_2(\Si_{\nu_M}^-)\bigr)\al,\qquad c_4(E)=-c_2(\Si_{\nu_M}^-)\al.
\end{equation*}
Moreover, since $\nu_M=\Hom_\H(\Si^-_{\nu_M},\Si^+_{\nu_M})$ we have $p_1(\nu_M)=-2c_2^+ - 2c_2^-$ and hence $p_1(TX)\vert_M=p_1(TM)+p_1(\nu_M)=p_1(TM)-2c_2^+ - 2c_2^-$. Inserting all this into the index formula \eq{fm19eq8} and applying the Hirzebruch Signature Theorem gives
\begin{equation*}
\ind \bigl(\sD^+_X\ot\Ad(P)\bigr)=\ts\frac{m+1}{4}\sign(M) + \int_M c_2(\Si_{\nu_M}^+).
\end{equation*}

Recall from Theorem \ref{fm3thm1}(d) that \eq{fm3eq16} maps $\ze_1\mapsto(1,0,0),$ $\ze_2\mapsto (0,1,0),$ and $\ze_2'\mapsto(0,0,-1)$. By the previous formula, $\ind \bigl(\sD^+_X\ot\Ad(P)\bigr)$ is in this case the inner product of \eq{fm3eq16} with the vector $(-4m-4,1,0)$. This verifies Table \ref{fm11tab1} in the case~$G=\Sp(m)$.

All of (b) is immediate except the data in Table \ref{fm11tab2}. Now as $\sN_8^{\bs\Spin,G}$ is a symmetric monoidal functor it satisfies $q'\ci \pi_0(\sN_8^{\bs\Spin,G})=\pi_1(\sN_8^{\bs\Spin,G})\ci q$. But $q:\Om_8^{\bs\Spin}(BG)\ra\Om_9^{\bs\Spin}(BG)$ acts by multiplication by $\al_1$ by Proposition \ref{fm5prop1}(a), and $q':\Z\ra\Z_2$ is reduction mod 2. As all the generators in Table \ref{fm11tab2} are of the form $\al_1\ze$, the data in Table \ref{fm11tab2} is determined by Table \ref{fm11tab1} as shown. Part (c) is immediate. This completes the proof.

\subsection{Proof of Theorem \ref{fm11thm2}}
\label{fm194}

For part (a), by Theorems \ref{fm3thm1} and \ref{fmAthm2} and \eq{fm5eq4}--\eq{fm5eq5}, all we have to do is determine $\pi_1(\sO_{7,4}^{\bs\Spin,\SO(4),*})$. In the splitting \eq{fm11eq3}, $\pi_1(\sO_{7,4}^{\bs\Spin,\SO(4),*})=0$ on $\Om_8^{\bs\Spin}(*)$, as $\Om_8^{\bs\Spin}(*)$ corresponds to elements $[Y,N]\in\Om_8^{\bs\Spin}(M\SO(4))$ with $N=\es$, so the action of \eq{fm11eq2} (as the index of an operator on $N$) is clearly zero. Thus it remains only to prove Table~\ref{fm11tab3}.

To compute the action of $\sO_{7,4}^{\bs\Spin,\SO(4),*}$ in \eq{fm11eq2} on $\frac{\ze_1}{4},\ze_2,\ze_3$ in \eq{fm3eq3}--\eq{fm3eq5}, observe that in the case in which $N$ and the fibres of $\nu_N\ra N$ are spin,
\begin{align*}
\sO_{7,4}^{\bs\Spin,\SO(4),\pm}([Y,N])=\sO_7\bigl([N,F_N^\pm]\bigr)=\ind F_N^\pm=\int_N\hat{A}(TN)\ch(\Si_\nu^\pm),
\end{align*}
where the first step holds by Definition \ref{fm9def6} with $F_N^\pm$ the Fueter operators of Definition \ref{fm9def5}, the second step by Theorem \ref{fm9thm2}, and the third by the Atiyah--Singer Index Theorem, with $\Si_\nu^\pm$ the spinor bundles of $\nu_N\ra N$. As $c_1(\Si_\nu^\pm)=0$ so that $\ch(\Si_\nu^\pm)=2-c_2(\Si_\nu^\pm)$, we see that
\begin{equation*}
\sO_{7,4}^{\bs\Spin,\SO(4),\pm}([Y,N])=\int_N\Bigl(\frac{1}{12}p_1(TN)-c_2(\Si_\nu^\pm)\Bigr).
\end{equation*}

Using $e(\nu)=c_2(\Si_\nu^+)-c_2(\Si_\nu^-)$ (the Euler class), and $p_1(\nu)=-2c_2(\Si_\nu^+)-2c_2(\Si_\nu^-)$ (the first Pontryagin class), we may rewrite these as
\ea
\sO_{7,4}^{\bs\Spin,\SO(4),+}([Y,N])&=\int_N\Bigl(\frac{1}{12}p_1(TN)-\frac{1}{2}e(\nu)-\frac{1}{4}p_1(\nu)\Bigr),
\label{fm19eq9}\\
\sO_{7,4}^{\bs\Spin,\SO(4),-}([Y,N])&=\int_N\Bigl(\frac{1}{12}p_1(TN)+\frac{1}{2}e(\nu)-\frac{1}{4}p_1(\nu)\Bigr).
\label{fm19eq10}
\ea 
Equations \eq{fm19eq9}--\eq{fm19eq10} also hold when $N$ and $\nu_N\ra N$ are not spin, so that $\Si_\nu^\pm$ are not defined and $\hat{A}(TN)$ need not be integral, but $p_1(TN),e(\nu),p_1(\nu)$ are defined and integral. Also $\sO_{7,4}^{\bs\Spin,\SO(4),0}([Y,N])=\sO_{7,4}^{\bs\Spin,\SO(4),-}([Y,N])-
\sO_{7,4}^{\bs\Spin,\SO(4),+}([Y,N])$ by definition, so 
\e
\sO_{7,4}^{\bs\Spin,\SO(4),0}([Y,N])=\int_Ne(\nu)=[N]\bu[N],
\label{fm19eq11}
\e
where $[N]\bu[N]$ is the self-intersection of $N$ in $Y$. Table \ref{fm11tab3} now follows from Table \ref{fm19tab1}, which is an easy computation, and \eq{fm19eq9}--\eq{fm19eq11}. This completes~(a).

\begin{table}[htb]
\centerline{\begin{tabular}{|l|l|l|l|}
\hline
 & $\frac{\ze_1}{4}$ & $\ze_2$ &  $\ze_3$ \\
\hline
$\int_{{M_j}_{\vphantom{(}}}^{\vphantom{(}}p_1(TM_j)$  &  $-12$ & 0 & 3 \\
\hline
$\int_{{M_j}_{\vphantom{(}}}^{\vphantom{(}}e(\nu)$  &  0 & 1 & 0 \\
\hline
$\int_{{M_j}_{\vphantom{(}}}^{\vphantom{(}}p_1(\nu)$ &  0 & $-2$ & 1 \\
\hline
\end{tabular}}
\caption{Invariants of $\frac{\ze_1}{4},\ze_2,\ze_3$ in \eq{fm3eq3}--\eq{fm3eq5}}
\label{fm19tab1}
\end{table}

For part (b), by Theorems \ref{fm3thm1} and \ref{fmAthm2} and \eq{fm5eq4}--\eq{fm5eq5}, all we have to do is determine $\pi_0(\sO),\pi_1(\sO)$ for $\sO$ as in \eq{fm11eq4}. For $\pi_0(\sO)$, equation \eq{fm11eq5} follows from Definition \ref{fm9def6}. Also $\pi_1(\sO)$ is zero on $\Om_9^{\bs\Spin}(*)$ as in (a) above. Thus it remains only to prove Table \ref{fm11tab4}. As the symmetric monoidal functors \eq{fm11eq4} commute with the linear quadratic invariants $q$ in Theorem \ref{fmAthm2}(a), and in $\Bord_{8,4}^{\bs\Spin}(M\SO(4))$ we have $q(\frac{\ze_1}{4})=\al_1\frac{\ze_1}{4}$, $q(\ze_2)=\al_1\ze_2$ by the description of $q$ in Proposition \ref{fm5prop1}(a), and in $\Z\qs\Z_2$ we have $q(m)=m\mod 2$, the $\al_1\frac{\ze_1}{4},\al_1\ze_2$ columns in Table \ref{fm11tab4} follow from \eq{fm11eq5} and the $\frac{\ze_1}{4},\ze_2$ columns in Table~\ref{fm11tab3}.

For the $\eta$ column in Table \ref{fm11tab4} observe from \eq{fm3eq6} that the normal bundle $\nu$ of $\SU(3)/\SO(3)\t\{(1,0)\}$ in $(\SU(3)\t\cS^3)/\SO(3)\t\cS^1$ is $E\op\R$, where $E$ is the normal bundle of $\SU(3)/\SO(3)$ in $(\SU(3)\t\cS^3)/\SO(3)$, and $\R$ is the normal bundle of $\{(1,0)\}$ in $\cS^1$. Changing the sign in the $\R$ summand exchanges the (locally defined) spin bundles $\Si^\pm_\nu$ in Definition \ref{fm9def5}, and so exchanges $\pi_1\bigl(\sO_{8,4}^{\bs\Spin,\SO(4),\pm}\bigr)$ acting on $\eta$. Hence 
\begin{equation*}
\pi_1\bigl(\sO_{7,4}^{\bs\Spin,\SO(4),+}\bigr)(\eta)=\pi_1\bigl(\sO_{7,4}^{\bs\Spin,\SO(4),-}\bigr)(\eta).
\end{equation*}
As $\sO_{7,4}^{\bs\Spin,\SO(4),0}=\sO_{7,4}^{\bs\Spin,\SO(4),-}-\sO_{7,4}^{\bs\Spin,\SO(4),+}$, the third row in the $\eta$ column in Table \ref{fm11tab4} follows. Part (c) follows from Theorem \ref{fm3thm1}(a),(b) and Proposition \ref{fm11prop3}, and part (d) is immediate. 

\subsection{Proof of Theorem \ref{fm11thm3}}
\label{fm195}

We will first show that:
\begin{itemize}
\setlength{\itemsep}{0pt}
\setlength{\parsep}{0pt}
\item[(i)] Suppose a Lie group $H$ has a torus subgroup $T\subseteq H$, and write $G=Z(T)$ for the centralizer of $T$. Then $\inc:G\hookra H$ is of complex type.
\item[(ii)] Let $\io:G\ra H$ be a morphism of connected Lie groups which is a covering map. Then $\io$ is of complex type.
\item[(iii)] Compositions of complex type morphisms are of complex type.
\end{itemize}

Parts (ii),(iii) are obvious. For (i), write $\g,\h$ for the Lie algebras of $G,H$. Under the adjoint representation of $T$ on $\h$ we have a splitting $h=\g\op\m$, where $\g$ is a trivial $T$-representation and $\m$ contains only nontrivial $T$-representations. Let $\U(1)\subseteq T$ be a sufficiently general $\U(1)$-subgroup. Then $\m$ contains only nontrivial $\U(1)$ representations, so we may split $\m=\bigop_{k>0}V_k\ot_\R\R^2[k]$ as $\U(1)$-representations, where $V_k$ is a real vector space and $\R^2[k]$ is the irreducible real $\U(1)$-representation with action $e^{i\th}\mapsto \bigl(\begin{smallmatrix}\cos k\th& \sin k\th \\ -\sin k\th &\cos k\th\end{smallmatrix}\bigr)$. 

We make $\m$ into a complex vector space by identifying $\R^2[k]\cong\C$ with $i\in\C$ acting by $\bigl(\begin{smallmatrix}0& 1 \\ -1 & 0\end{smallmatrix}\bigr)$. As the $G$- and $\U(1)$-actions on $\m$ commute and the complex structure on $\m$ is determined by the $\U(1)$-action, it is preserved by $G$. Hence $\inc:G\hookra H$ is of complex type.

Suppose now that $H$ is a compact, connected, simply-connected, simple Lie group corresponding to a Dynkin diagram $\Ga$, e.g. $H=E_8$. Then $H$ has a maximal torus $\U(1)^{\Ga_0}$ with $\U(1)$ factors corresponding to the set of vertices $\Ga_0$ of $\Ga$. Choose $k$ vertices $v_1,\ldots,v_k$ in $\Ga$, corresponding to a subgroup $\U(1)^k\subset H$. Then $\inc:Z(\U(1)^k)\hookra H$ is of complex type by (i).

By Lie theory, it is easy to show the Lie algebra of $Z(\U(1)^k)$ is ${\mathfrak z}(\U(1)^k)=\u(1)^{\op^k}\op\g$, where $\g$ is the semisimple Lie algebra whose Dynkin diagram $\Ga'$ is the result of deleting vertices $v_1,\ldots,v_k$ and any edges meeting them from $\Ga$. Write $G$ for the compact, connected, simply-connected, semisimple Lie group with Dynkin diagram $\Ga'$. It is then nearly true that $Z(\U(1)^k)=\U(1)^k\t G$. 

In fact $Z(\U(1)^k)$ could have finitely many connected components, and its identity component $Z(\U(1)^k)_1$ is of the form $Z(\U(1)^k)_1=(\U(1)^k\t G)/K$ for $K\subset \U(1)^k\t G$ a finite normal subgroup. But $Z(\U(1)^k)\hookra H$ of complex type implies that $Z(\U(1)^k)_1\hookra H$ is of complex type, which implies that $\U(1)^k\t G\hookra H$ is of complex type by (ii),(iii).

In \eq{fm11eq9}, the morphisms $E_7\t\U(1)\ra E_8$, $E_6\t\U(1)^2\ra E_8$, $\Spin(14)\t\U(1)\ra E_8$, $\SU(8)\t\U(1)\ra E_8$, $\Sp(3)\t\U(1)\ra F_4$, and $\Spin(7)\t\U(1)\ra F_4$, all arise this way by deleting 1 or 2 vertices from the Dynkin diagrams $E_8,F_4$.

For $G_2\ra\Spin(8)$, we have inclusions $G_2\hookra\Spin(7)\hookra\Spin(8)$, where in Lie algebras $\mathfrak{spin}(7)/\g_2$ and $\mathfrak{spin}(8)/\mathfrak{spin}(7)$ are both the irreducible 7-dimensional $G_2$-representation $\La_7$. Hence $\mathfrak{spin}(7)/\g_2\cong\La_7\op\La_7\cong\La_7\ot_\R\C$, so $G_2\hookra\Spin(8)$ is of complex type. Also $\Spin(m)\ra\SO(m)$ is by~(ii).

Next consider the three embeddings of Lie groups:
\begin{align*}
&{\rm(A)} & \U(1)&\longra\SU(m+1), 	& e^{i\th}&\longmapsto \mathop{\rm diag}\bigl(e^{i\th},\ldots,  e^{i\th},e^{-im\th}\bigr), \\
&{\rm(B)} & \U(1)&\longra\Sp(m+1), 	& e^{i\th}&\longmapsto \mathop{\rm diag}\bigl(1,\ldots,1, e^{i\th}\bigr), \\
&{\rm(C)} & \U(1)&\longra\SO(m+2), 	& e^{i\th}&\longmapsto \begin{pmatrix} 
1 & 0 & \cdots  & 0 & 0 & 0 \\
0 & 1 & 0 & \cdots   & 0 & 0\\
\vdots & 0 & \ddots & \ddots & \vdots & \vdots \\
0 & \vdots & \ddots & 1 & 0 & 0 \\
0 & 0 & \cdots & 0 & \cos\th & \sin\th \\
0 & 0 & \cdots & 0 & -\sin\th & \cos\th
\end{pmatrix}.
\end{align*}

For (A), $Z(\U(1))\cong\U(m)\subset\SU(m+1)$, where the embedding $\U(m)\hookra\SU(m+1)$ maps $A\mapsto \bigl(\begin{smallmatrix} A & 0 \\ 0 & \det A^{-1} \end{smallmatrix}\bigr)$. Hence $\U(m)\hookra\SU(m+1)$ is of complex type by (i), completing Theorem \ref{fm11thm3}(a). Also $\SU(m)\t\U(1)\ra\U(m)$ is a covering map, so $\SU(m)\t\U(1)\ra\SU(m+1)$ is of complex type by~(ii),(iii).

For (B), $Z(\U(1))\cong\Sp(m)\t\U(1)\subset\Sp(m+1)$, so $\Sp(m)\t\U(1)\hookra\Sp(m+1)$ is of complex type by (i). For (C), $Z(\U(1))\cong\SO(m)\t\SO(2)\subset\SO(m+2)$ with $\SO(2)\cong\U(1)$, so $\SO(m)\t\U(1)\ra\SO(m+2)$ is of complex type by (i). We show $\Spin(m)\t\U(1)\ra\Spin(m+2)$ is of complex type by lifting to Spin groups. We have now constructed the last four complex type morphisms in \eq{fm11eq9}. This completes the proof of Theorem~\ref{fm11thm3}.

\appendix
\section{Picard groupoids}
\label{fmA}

Categorical groups may be viewed as a categorification of the concept of a group. Similarly, Picard groupoids categorify abelian groups. These will be important tools in this monograph, so we briefly review them here. We state a classification result for Picard groupoids, originally due to Sinh \cite{Sinh}, and develop it from our point of view. This is mainly to fix the necessary terminology that is needed to prove an additional classification result for morphisms between Picard groupoids, Theorem \ref{fmAthm2}, that we could not find in the literature.

For more background on symmetric monoidal categories, we refer to Joyal--Street \cite{JoSt} and MacLane \cite[Ch.~VII.1 \& Ch.~XI]{MacL}.

\begin{dfn}
\label{fmAdef1}
A {\it monoidal category} $(\cC,\ot,\boo,\al)$ is a category $\cC$ with a {\it tensor product} functor $\ot:\cC\t\cC\ra\cC,$ a unit object $\boo\in\cC,$ a natural associativity isomorphism $\al,$ and unit isomorphisms. Usually, we will not make these explicit, which is justified by MacLane's coherence theorem. To simplify our exposition, we will usually assume that all unit isomorphisms are identities. The set $\pi_0(\cC)$ of {\it isomorphism classes} of objects of a monoidal category is a (possibly non-commutative) monoid. Moreover, the operation induced by the tensor product and the ordinary composition agree in the {\it automorphism group} $\pi_1(\cC)=\Aut_\cC(\boo),$ which implies that $\pi_1(\cC)$ is an abelian group (Eckmann--Hilton argument). We write $\pi_0(\cC)$ multiplicatively and $\pi_1(\cC)$ additively.

A {\it categorical group} is a monoidal category $(\cG,\ot,\boo,\al)$ in which all morphisms are invertible and for which the monoid $\pi_0(\cG)$ is a group.
\end{dfn}

This means that every object $x$ has a {\it dual}, an object $x^*$ for which there exist isomorphisms $\ep_x: x^*\ot x\cong{\boo}$ and $\eta_x:{\boo}\cong x\ot x^*$ (one usually requires some axioms, which play no role here). In a categorical group, all of the automorphism groups can be identified with each other via
\e
\label{fmAeq1}
 \pi_1(\cG)\longra\Aut_\cG(x),\enskip \left({\boo}\xrightarrow{\vp}{\boo}\right)\longmapsto \left(x\cong {\boo}\ot x\xrightarrow{\vp\ot x}{\boo}\ot x\cong x\right).
\e

\begin{dfn}
\label{fmAdef2}
Let $\cG$ be a categorical group. The {\it conjugation} $\la_\cG:\pi_0(\cG)\ra\Aut(\pi_1(\cG)),$ $x\mapsto\la_x,$ takes $\la_x(\vp)$ for $x\in\pi_0(\cG)$ and $\vp\in\pi_1(\cG)$ to
\begin{equation*}
 {\boo}\xrightarrow{\eta_x} x\ot x^*\cong x\ot{\boo}\ot x^*\xrightarrow{x\ot \vp\ot x^*} x\ot{\boo}\ot x^*\cong x\ot x^*\xrightarrow{\eta_x^{-1}} {\boo}.
\end{equation*}
\end{dfn}

\begin{ex}
\label{fmAex1}
Given a group $\pi_0$ and an abelian group $\pi_1,$ let $\cG=\pi_0\qs\pi_1$ denote the category of $\pi_0$-graded $\pi_1$-torsors. In other words, the objects of $\cG$ are all pairs $(x,S),$ where $x\in\pi_0$ and $S$ is a set with a free, transitive left action of the group $\pi_1.$  If $x=y,$ then $\Hom_\cG\bigl((x,S),(y,T)\bigr)$ is the set of all $\pi_1$-equivariant maps $\vp: S\ra T,$ otherwise the morphism set is defined to be empty. Define the tensor product of objects by $(x_0,S_0)\ot(x_1, S_1)=(x_0x_1,S_0\ot_{\pi_1} S_1),$ where $S_0\ot_{\pi_1} S_1=(S_0\t S_1)/\pi_1$ is the quotient by the anti-diagonal $\pi_1$-action.

As any two $\pi_1$-torsors are isomorphic and every isomorphism is multiplication by a group element, $\cG$ is a categorical group with $\pi_0(\cG)=\pi_0,$ $\pi_1(\cG)=\pi_1,$ and a trivial conjugation action of $\pi_0(\cG)$ on $\pi_1(\cG).$ In Example \ref{fmAex2} we will generalize this construction and get a possibly non-trivial conjugation action.

In case $\pi_1=0$ the construction of the category $\pi_0\qs\pi_1$ boils down to the abelian group $\pi_0$ viewed as a discrete monoidal category in the usual way.
\end{ex}

We wish to classify all categorical groups up to monoidal equivalence. Recall here that a {\it monoidal structure} on a functor $F:\cC\ra\cD$ of monoidal categories $(\cC,\ot_\cC,\boo_\cC)$ and $(\cD,\ot_\cD,\boo_\cD)$ is a collection of isomorphisms
\e
\label{fmAeq2}
F(x)\ot_\cD F(y)\xrightarrow{\phi_{x,y}} F(x\ot_\cC y),\quad {\boo_\cD}\xrightarrow{\phi_{\boo}} F({\boo_\cC}),
\e
for all objects $x, y$ of $\cC,$ compatible with the associativity and unit isomorphisms in $\cC$ and $\cD,$ see \cite[p.~25]{JoSt}. A {\it monoidal transformation} of such functors is a natural transformation $F\Rightarrow G$ that maps the isomorphisms \eq{fmAeq2} for $F$ and $G$ onto each other, see \cite[p.~25]{JoSt}. A {\it monoidal equivalence} is a pair of monoidal functors whose composites either way admit monoidal natural isomorphisms to the identity functors of $\cC$ and $\cD.$

Monoidal structures are categorically well-behaved. For example, if $F:\cC\xrightarrow{\sim}\cD$ is an equivalence of categories and $(\cD,\ot_\cD,{\boo_{\cD}})$ has a monoidal structure, then there is a monoidal structure on $\cC$ and $F$ such that $F$ becomes a monoidal equivalence. Indeed, by doctrinal adjunction every equivalence $F$ is part of an adjoint equivalence $(F,G,\ep,\eta)$ and we can define $x\ot_\cC y=G(F(x)\ot_\cD F(y)),$ $\phi_{x,y}=\eta_{F(x),F(y)},$ and ${\boo_\cC}=G({\boo_\cD}).$

\begin{dfn}
\label{fmAdef3}
Let $\pi_0$ be a group, $\pi_1$ an abelian group, and $\la:\pi_0\ra\Aut(\pi_1)$ a group action. Recall that the {\it normalized bar cochain complex} has as its $n$\textsuperscript{th} cochain group $C^n(\pi_0,\pi_1)$ the set of all maps $\be:(\pi_0)^{\t n}\ra \pi_1$ such that $\be(x_1,\ldots,x_n)=0$ whenever $x_i=1$ for some $i.$ The codifferential is defined by
\begin{equation*}
\begin{split}
 \de\be(x_1,\ldots,x_{n+1})=&\la_{x_1}\be(x_2,\ldots,x_{n+1})+\sum_{i=1}^n (-1)^i \be(x_1,\ldots,x_ix_{i+1},\ldots,x_{n+1})\\
 &\qquad+(-1)^{n+1}\be(x_1,\ldots,x_n).
\end{split}
\end{equation*}
The cohomology of this cochain complex is the {\it group cohomology} $H^*(\pi_0,\pi_1).$
\end{dfn}

\begin{ex}
\label{fmAex2}
Let $\pi_0$ be a group, $\pi_1$ an abelian group, $\la:\pi_0\ra\Aut(\pi_1)$ a group action, and $\al\in C^3(\pi_0,\pi_1)$ be a normalized $3$-cocycle satisfying $\de\al=0.$ We will define a categorical group $\cG(\pi_0,\pi_1,\la,\al).$ The object set is $\pi_0$ and the morphism set from $x$ to $y$ is $\pi_1$ if $x=y$ and empty otherwise. The composition is given by the binary operation on $\pi_1$ whenever this makes sense. The tensor product is defined using the binary operations on $\pi_0$ and $\pi_1$ as
\begin{equation*}
\Bigl(x_0\xrightarrow{\vp_0}x_0\Bigr)\ot\Bigl(x_1\xrightarrow{\vp_1}x_1\Bigr)=\Bigl(x_0x_1\xrightarrow{\vp_0+\la_{x_0}(\vp_1)}x_0x_1\Bigr).
\end{equation*}
The associativity isomorphism is $\al(x_0,x_1,x_2)\in\pi_1,$ which is viewed as an element of $\Hom_{\cG}(x_0(x_1x_2),(x_0x_1)x_2).$ All of the unit isomorphisms are identities.
\end{ex}

Notice that $\cG(\pi_0,\pi_1,\la,\al)$ is a categorical group with $\pi_0(\cG)=\pi_0,$ $\pi_1(\cG)=\pi_1,$ and conjugation action $\la.$ For categorical groups of this kind, all of the terminology for monoidal categories boils down to concepts in group cohomology:

\begin{lem}
\label{fmAlem1}
{\bf(a)} Monoidal functors\/ $F:\cG(\pi_0,\pi_1,\la,\al)\ra\cG(\pi_0',\pi_1',\la',\al')$ are in {\rm 1-1} correspondence to triples\/ $(f_0,f_1,\phi)$ of group morphisms\/ $f_0:\pi_0\ra\pi_0'$ and\/ $f_1:\pi_1\ra\pi_1'$ and where\/ $\phi\in C^2(\pi_0,\pi_1')$ is a\/ $2$-cochain such that
\begin{align*}
 f_1\ci\la&=\la'\ci(f_0\t f_1),
&(f_0\t f_0\t f_0)^*(\al')+\de\phi=(f_1)_*(\al)\in C^3(\pi_0,\pi_1').
\end{align*}

\noindent{\bf(b)} Let\/ $F$ and\/ $G$ be monoidal functors with triples\/ $(f_0,f_1,\phi)$ and\/ $(g_0,g_1,\ga).$ Then\/ $F$ and\/ $G$ are \textup(non-monoidally\textup) naturally isomorphic if and only if\/ $f_0=g_0$ and\/ $f_1=g_1.$ In this case, there exists a monoidal natural isomorphism if and only if the \begin{bfseries}difference class\end{bfseries}\/ $\om(F,G)=[\phi-\ga]\in H^2(\pi_0,\pi_1')$ vanishes. Given a monoidal functor\/ $F$ and\/ $\om\in H^2(\pi_0,\pi_1'),$ there exists a monoidal functor\/ $G$ such that\/ $f_0=g_0,$ $f_1=g_1,$ and\/ $\om(F,G)=\om.$ In other words, the set of all monoidal functors\/ $F$ up to monoidal natural isomorphism with given\/ $\pi_0(F)=f_0$ and\/ $\pi_1(F)=f_1$ is a\/ $H^2(\pi_0,\pi_1')$-torsor.
\smallskip

\noindent{\bf(c)} If non-empty, the set of all monoidal natural isomorphisms\/ $F\Rightarrow G$ is a torsor over the group\/ $H^1(\pi_0,\pi_1').$
\end{lem}

\begin{proof}
Everything follows by spelling out the definitions in the special case of Example \ref{fmAex2}.
\end{proof}

In particular, (a) implies that the identity functor can be made into a monoidal equivalence between $\cG(\pi_0,\pi_1,\la,\al)$ and $\cG(\pi_0,\pi_1,\la,\al')$ if and only if $\al$ and $\al'$ represent the same group cohomology class in $H^3(\pi_0,\pi_1).$

\begin{dfn}
\label{fmAdef4}
A category is {\it skeletal}	if $x\cong y$ for objects implies $x=y.$
\end{dfn}

In particular, the set of objects in a skeletal category is $\pi_0(\cG).$ By choosing an object in every isomorphism class, one obtains the following well-known result.

\begin{lem}
\label{fmAlem2}
Every category\/ $\cG$ has an equivalent skeletal full subcategory\/ $\sk\cG.$
\end{lem}

On the subcategory $\sk\cG$ we can construct a monoidal structure such that the inclusion functor becomes a monoidal equivalence. We conclude:

\begin{lem}
\label{fmAlem3}
Every skeletal categorical group\/ $\cG$ is monoidally isomorphic to the category\/ $\cG(\pi_0,\pi_1,\la,\al)$ where\/ $\pi_0=\pi_0(\cG),$ $\pi_1=\pi_1(\cG),$ $\la$ is the conjugation action, and $\al\in C^3(\pi_0,\pi_1)$ satisfies $\de\al=0$ and corresponds to the collection of associativity isomorphisms in $\cG,$ viewed as belonging to\/ $\pi_1$ by \eq{fmAeq1}.
\end{lem}

The group cohomology class $[\al]\in H^3(\pi_0,\pi_1)$ is called the {\it associativity invariant} of $\cG$ (it is independent of the choice of monoidal structure on the skeletal subcategory).

We see that $(\pi_0,\pi_1,\la,[\al])$ is a complete invariant of categorical groups in the sense that every quadruple arises as the invariants a categorical group. Conversely, the quadruple uniquely describes the categorical group up to monoidal equivalence. This kind of argument also shows that every result that holds in the special case of Example \ref{fmAex2} and is invariant under monoidal equivalence remains true for general categorical groups. This leads to the following result.

\begin{thm}
\label{fmAthm1}
{\bf(a)} For every quadruple\/ $(\pi_0,\pi_1,\la,[\al]),$ where\/ $\pi_0$ is group,\/ $\pi_1$ is an abelian group,\/ $\la:\pi_0\ra\Aut(\pi_1)$ is a morphism, and\/ $[\al]\in H^3(\pi_0,\pi_1),$ there exists a categorical group\/ $\cG$ with these invariants.
\smallskip

\noindent{\bf(b)} Consider categorical groups\/ $\cG$ and\/ $\cG'$ with invariants\/ $(\pi_0,\pi_1,\la,[\al])$ and $(\pi_0',\pi_1',\la',[\al']).$ Let\/ $f_0:\pi_0\ra\pi_0'$ and\/ $f_1:\pi_1\ra\pi_1'$ be group morphisms such that\/ $f_1\ci\la=\la'\ci(f_0\t f_1).$ Then there exists a monoidal functor\/ $F:\cG\ra\cG'$ with\/ $\pi_0(F)=f_0$ and\/ $\pi_1(F)=f_1$ if and only if\/ $(f_0)^*([\al'])=(f_1)_*([\al])$ in\/~$H^3(\pi_0,\pi_1').$

\smallskip

\noindent{\bf(c)} Let\/ $F, G:\cG\ra\cG'$ be monoidal functors of categorical groups with\/ $\pi_0(F)=\pi_0(G)$ and\/ $\pi_1(F)=\pi_1(G).$ There is a monoidal transformation\/ $F\Rightarrow G$ if and only if the difference class\/ $\om(F,G)\in H^2(\pi_0(\cG),\pi_1(\cG'))$ vanishes. The set of all monoidal functors\/ $F:\cG\ra\cG'$ with given\/ $\pi_0(F)=f_0$ and\/ $\pi_1(F)=f_1$ modulo monoidal natural isomorphism is a torsor over\/ $H^2(\pi_0(\cG),\pi_1(\cG')).$
\smallskip

\noindent{\bf(d)} Let\/ $F, G:\cG\ra\cG'$ be monoidal functors of categorical groups. Define an equivalence relation $\sim$ between monoidal transformations\/ $\th:F\Rightarrow G$ such that $\th_1\sim\th_2$ if there exists $\ep\in\Aut_{\cG'}(\boo)=\pi_1(\cG')$ such that the following commutes in $\cG'$ for all objects {\rm$x\in\cG$:}
\e
\begin{gathered}
\xymatrix@C=80pt@R=15pt{*+[r]{\boo\ot\cG(x)} \ar[r]_(0.6){\text{unit iso.}} \ar[d]^{\ep\ot\th_1} & \cG(x) \ar[r]_(0.4){\text{unit iso.}} & *+[l]{\cG(x)\ot\boo} \ar[d]_{\th_2\ot\ep} \\
*+[r]{\boo\ot\cG'(x)} \ar[r]^(0.6){\text{unit iso.}}  & \cG'(x) \ar[r]^(0.4){\text{unit iso.}} & *+[l]{\cG'(x)\ot\boo.\!} }
\end{gathered}
\label{fmAeq3}
\e
If non-empty, the set of all monoidal transformations\/ $\th:F\Rightarrow G$ modulo equivalence $\sim$ is a torsor over\/ $H^1(\pi_0(\cG),\pi_1(\cG')).$ 
\end{thm}

The next goal is to classify Picard groupoids and to classify the monoidal functors and transformations between them.

\begin{dfn}
\label{fmAdef5}
A {\it symmetry} $\si$ on a monoidal category $(\cC,\ot,\boo,\al)$ is a natural isomorphism $\si_{x,y}: x\ot y\ra y\ot x$ such that $\si_{y,x}\ci\si_{x,y}=1_{x\ot y}$ and such that the unit and the hexagon coherence diagrams of \cite[XI, \S 1, eq.~(7)]{MacL} commute. A {\it Picard groupoid} is a categorical group $\cG$ equipped with a symmetry $\si.$
\end{dfn}

In particular, $\pi_0(\cG)$ is then a commutative monoid and the conjugation action becomes trivial. From now on, we write the abelian groups $\pi_0, \pi_1$ additively.

\begin{ex}
\label{fmAex3}
Additive categories are symmetric monoidal categories.
\end{ex}

Recall that for a {\it symmetric monoidal functor} of symmetric monoidal categories the isomorphisms \eq{fmAeq2} are required to commute with the symmetry, see \cite[XI, \S 2]{MacL}. A symmetric monoidal functor between Picard groupoids is also called a {\it morphism of Picard groupoids}. There are no further conditions for a monoidal transformation between symmetric monoidal functors. Symmetries are also categorically well-behaved and can be transported along monoidal equivalences using an adjoint equivalence as above.

\begin{ex}
\label{fmAex4}
This example is very instructive; there are two categorical groups with $\pi_0=\pi_1=\Z_2.$ Both have objects $\ov 0,\ov 1$ and tensor product $\ov x+\ov y=\ov{x+y},$ where `$\ov{\phantom{a}}$' denotes the remainder modulo two. Both of the categorical groups have $\Aut(\ov 0)=\Aut(\ov 1)=\{1,\eta\},$ where $\eta^2=1.$ One of the categorical groups has a trivial associativity isomorphism $\al(\ov x,\ov y,\ov z)=1.$ The other, nonequivalent categorical group has $\al(\ov 1,\ov 1,\ov 1)=\eta$ and $\al(\ov x,\ov y,\ov z)=1$ for all $(\ov x,\ov y,\ov z)\neq(\ov 1,\ov 1,\ov 1),$ representing the non-trivial class in $H^3(\Z_2,\Z_2)\cong\Z_2.$

Only the first category admits a symmetry. Hence, up to (non-symmetric) monoidal equivalence, there is a unique Picard groupoid with $\pi_0=\pi_1=\Z_2.$ However, there are two inequivalent symmetries on the unique two-object two-morphism categorical group: the trivial symmetry $\si(\ov 1,\ov 1)=1$ and the non-trivial symmetry $\si(\ov 1,\ov 1)=\eta.$ This boils down to equipping the categorical group of $\Z_2$-graded $\Z_2$-torsors with the ordinary symmetry or the super symmetry.	
\end{ex}

Later we will see in general that the associativity invariant $[\al]$ vanishes for Picard groupoids. Instead, there is a quadratic invariant coming from the symmetry. Before defining this invariant, we briefly recall quadratic maps.

\begin{dfn}
\label{fmAdef6}
Let $\pi_0$ and $\pi_1$ be abelian groups.
\begin{itemize}
\item
A map $q:\pi_0\ra\pi_1$ is {\it quadratic} if $b_q(x,y)=q(x+y)-q(x)-q(y)$ defines a bilinear map. This implies $q(\la x)=\la^2q(x)$ for $\la\in\Z.$ Let $\Quad(\pi_0,\pi_1)$ be the abelian group of all quadratic maps.
\item
A bilinear map $\al:\pi_0\t\pi_0\ra\pi_1$ is {\it alternating} if $\al(x,x)=0$ for all $x\in\pi_0$ and {\it skew-symmetric} if $\al(x,y)+\al(y,x)=0$ for all $x,y\in\pi_0.$ (This is also a good definition when $\al$ is not bilinear.) Let $\Alt(\pi_0,\pi_1)$ be the abelian group of all alternating bilinear maps and let $\Skew(\pi_0,\pi_1)$ be the abelian group of all skew-symmetric bilinear maps.
\end{itemize}
\end{dfn}

By expanding $\al(x+y,x+y)=0,$ one finds $\Alt(\pi_0,\pi_1)\subset\Skew(\pi_0,\pi_1).$ If a quadratic map $q:\pi_0\ra\pi_1$ is also a linear map, then $q(2x)=4q(x)$ and $q(2x)=q(x+x)=q(x)+q(x),$ so $2q(x)=0.$ Therefore, $q$ factors through a linear map $\pi_0/2\pi_0\ra\pi_1.$ Conversely, every linear map $\pi_0/2\pi_0\ra\pi_1$ determines a linear quadratic map by composing with the canonical projection. Hence $\Hom(\pi_0/2\pi_0,\pi_1)\subset\Quad(\pi_0,\pi_1)$ is the subset of {\it linear quadratic maps}.

\begin{prop}
\label{fmAprop1}
{\bf(a)} There is a short exact sequence
\begin{equation}
\label{fmAeq4}
 \begin{tikzcd}[column sep=3ex]
  0\rar&\Alt(\pi_0,\pi_1)\rar&\Skew(\pi_0,\pi_1)\rar{\De^*} & \Hom(\pi_0/2\pi_0,\pi_1)\rar & 0,
 \end{tikzcd}
\end{equation}
where\/ $\De^*$ maps\/ $\al\in\Skew(\pi_0,\pi_1)$ to the quadratic map\/ $q(x)=\al(x,x).$
\smallskip

\noindent{\bf(b)} There is a group morphism\/ $A: Z^2(\pi_0,\pi_1)\ra\Alt(\pi_0,\pi_1)$ that maps a\/ $2$-cocycle\/ $F:\pi_0\t\pi_0\ra\pi_1$ to the alternating bilinear map\/ $\al(x,y)=F(x,y)-F(y,x).$ Moreover,\/ $A$ is surjective with kernel the symmetric\/ $2$-cocycles.
\end{prop}

\begin{proof}
(a) If $\al\in\Skew(\pi_0,\pi_1),$ then the map $b_q$ associated to $q(x)=\al(x,x)$ vanishes and is therefore bilinear. Moreover, $q(x+y)=\al(x+y,x+y)=\al(x,x)+\al(y,y)=q(x)+q(y).$ We prove surjectivity of $\De^*$ in \eq{fmAeq4}, which is the only non-trivial assertion of the exact sequence. Pick a basis $\{\ov f_k\}$ of the $\Z_2$-vector space $\pi_0/2\pi_0.$ Since a linear quadratic map $q$ satisfies $2q=0,$ we can define a bilinear map $\al:\pi_0/2\pi_0\ot\pi_0/2\pi_0\ra\pi_1$ by $\al(\ov f_k, \ov f_\ell):=\de_{k,\ell}q(\ov f_k).$ Evidently, the pullback of $\al$ to $\pi_0\ot\pi_0$ is a (diagonal) skew-symmetric bilinear form that maps to $q$ under $\De^*.$
\smallskip

\noindent
(b) We show that $\al(x,y)=F(x,y)-F(y,x)$ is bilinear: As $F$ is a 2-cocycle,
\ea
 F(x+y,z)+F(x,y)&=F(x,y+z)+F(y,z),
\label{fmAeq5}\\
 F(z,x+y)+F(x,y)&=F(z+x,y)+F(z,x),
\label{fmAeq6}\\
 F(x,z+y)+F(z,y)&=F(x+z,y)+F(x,z).
\label{fmAeq7}
\ea
Subtracting \eq{fmAeq6} from \eq{fmAeq5} and then using  \eq{fmAeq7} gives $\al(x+y,z)=F(x+y,z)-F(z,x+y)=F(x,z)-F(z,x)+F(y,z)-F(z,y)=\al(x,z)+\al(y,z).$

To show that $A$ is surjective we will assume that $\pi_0$ is a finitely generated abelian group, an assumption that can be removed by observing that every abelian group is a filtered colimit of finitely generated abelian groups. Write $\pi_0=\Z^n\op\Z_{q_1}\op\cdots\op\Z_{q_n}$ for prime powers $q_k.$ Let $e_1,\ldots, e_n\in\Z^n$ and $\ov f_k\in\Z_{q_k}$ be the corresponding cyclic generators. Being alternating, $\al$ is completely determined by its values
\begin{equation*}
 \al_{ij}=\al(e_i,e_j),
 \qquad
 \al_{ik}=\al(e_i,\ov f_k),
 \qquad
 \al_{k\ell}=\al(\ov f_k,\ov f_\ell),
\end{equation*}
for all $i<j$ and $k<\ell.$ Moreover, $q_k\al_{ik}=0$ and $q_k\al_{k\ell}=q_\ell\al_{k\ell}=0.$ This implies that we can define a bilinear map $F:\pi_0\ot\pi_0\ra\pi_1$ by
\begin{align*}
 F(e_i,e_j)&=
 \begin{cases}
 \al_{ij} & \text{if $i<j,$}\\
 0 & \text{if $i\ge j,$}
 \end{cases}
&F(\ov f_k,\ov f_\ell)&=
 \begin{cases}
 \al_{k\ell} & \text{if $k<\ell,$}\\
 0 & \text{if $k\ge\ell,$}
 \end{cases}\\
 F(e_i,\ov f_k)&=\al_{ik}\quad\text{for all $i,k$},
&F(\ov f_k,e_i)&=0\quad\text{for all $i,k$}.
\end{align*}
Being bilinear, $F$ defines a $2$-cocycle in $Z^2(\pi_0,\pi_1)$ and $\al(x,y)=F(x,y)-F(y,x)$ is easily checked on the cyclic generators.
\end{proof}

\begin{ex}
\label{fmAex5}
Continue with the categorical group $\cG(\pi_0,\pi_1,\la,\al)$ from Example \ref{fmAex2} and the trivial action $\la=1.$ By the hexagon coherence diagrams of \cite[XI, \S1, eq.~(7)]{MacL}, a symmetry on $\cG(\pi_0,\pi_1,1,\al)$ is equivalently a skew-symmetric map $\si:\pi_0\t\pi_0\ra\pi_1$ satisfying
\begin{equation*}
-\al(z,x,y)+\si(x+y,z)-\al(x,y,z)=\si(x,z)-\al(x,z,y)+\si(y,z).
\end{equation*}
Set $\al=0,$ so that a symmetry isomorphism is just a skew-symmetric bilinear form $\si\in\Skew(\pi_0,\pi_1)$ and define $\cP(\pi_0,\pi_1,\si)$ to be the Picard groupoid with symmetry $\si$ and underlying categorical group $\cG(\pi_0,\pi_1,1,0).$

For Picard groupoids of this kind, all of the terminology for symmetric monoidal categories again boils down to linear algebra. For example, a symmetric monoidal functor $\cP(\pi_0,\pi_1,\si)\ra\cP(\pi_0',\pi_1',\si')$ is equivalently a triple $(f_0,f_1,\phi),$ where $f_0:\ab\pi_0\ra\pi_0',$ $f_1:\pi_1\ra\pi_1'$ are group morphisms and $\phi\in C^2(\pi_0,\pi_1')$ is a normalized $2$-cocycle satisfying
\e
\label{fmAeq8}
\si'(f_0(x),f_0(y))-f_1(\si(x,y))=\phi(x,y)-\phi(y,x).
\e
Consider a pair of monoidal functors $(f_0,f_1,\phi)$ and $(g_0,g_1,\ga).$ A non-monoidal natural isomorphism exists precisely when $f_0=g_0$ and $f_1=g_1$ and is then given by an arbitrary map $u:\pi_0\ra\pi_1'$ of sets. For a monoidal transformation,
\begin{equation*}
u(x)-u(x+y)+u(y)=\phi(x,y)-\ga(x,y).
\end{equation*}
The obstruction for the existence of a monoidal transformation is therefore the group cohomology class $[\phi-\ga]\in H^2_\sym(\pi_0,\pi_1')\subset H^2(\pi_0,\pi_1'),$ see Theorem \ref{fmAthm2}(c) below.
\end{ex}

\begin{lem}
\label{fmAlem4}
Let\/ $\cG$ be a Picard groupoid. Then the associativity invariant\/ $[\al]\in H^3(\pi_0(\cG),\pi_1(\cG))$ vanishes. In particular, the underlying categorical group of\/ $\cG$ is monoidally equivalent to some\/ $\cG(\pi_0,\pi_1,1,0).$
\end{lem}

\begin{proof}
This is essentially a consequence of Eilenberg--MacLane \cite[Th.~26.1]{EiMa} (or see \cite[p.7]{JoOs}) and we use their terminology of `symmetric cochains' which are certain pairs $(\al,\si).$ As the conjugation action of $\cG$ vanishes, Theorem \ref{fmAthm1} implies that the underlying categorical group of $\cG$ is monoidally equivalent to some $\cG(\pi_0,\pi_1,1,\al).$ For the symmetry $\si$ of $\cG$ a calculation shows that $\si(x,x)$ is a linear quadratic form, so by \eq{fmAeq4} there exists $\si'\in\Skew(\pi_0,\pi_1)$ with $\si(x,x)=\si'(x,x).$ By bilinearity, $(0,\si')$ is a symmetric $3$-cocycle. Now \cite[Th.~26.1]{EiMa} implies that $(\al,\si)$ and $(0,\si')$ are cohomologous, since they agree on the diagonal. In particular, $\al=\pd\phi$ is a coboundary. Then the 2-cochain $\phi$ determines a monoidal structure on the identity functor that makes  $\cG(\pi_0,\pi_1,1,\al)$ monoidally equivalent to $\cG(\pi_0,\pi_1,1,0).$
\end{proof}

Transporting the symmetry from $\cG$ to $\cG(\pi_0,\pi_1,1,0),$ we conclude that every Picard groupoid is symmetric monoidally equivalent to $\cP(\pi_0,\pi_1,\si)$ for some $\si\in\Skew(\pi_0,\pi_1).$ However, there are symmetric monoidal equivalences $\cP(\pi_0,\pi_1,\si)\ra\cP(\pi_0,\pi_1,\si')$ that change $\si$ by a $2$-cocycle, so by Proposition \ref{fmAprop1} only the linear quadratic form $q(x)=\si(x,x)$ is an invariant of $\cG,$ called the {\it symmetry invariant}. The triple $(\pi_0,\pi_1,q)$ is a complete invariant of Picard groupoids. As before, every result that holds in the special case of Example \ref{fmAex5} and is invariant under symmetric monoidal equivalence remains true for general Picard groupoids. This leads to the following classification result for Picard groupoids:

\begin{thm}
\label{fmAthm2}
{\bf(a)} Let\/ $\pi_0$ and\/ $\pi_1$ be abelian groups. Up to equivalence, Picard groupoids\/ $\cG$ with\/ $\pi_0(\cG)=\pi_0$ and\/ $\pi_1(\cG)=\pi_1$ are classified by their symmetry invariant, a linear quadratic form\/ $q:\pi_0(\cG)\ra\pi_1(\cG).$ Conversely, every triple $(\pi_0,\pi_1,q)$ occurs as the invariants of some Picard groupoid.
\smallskip

\noindent{\bf(b)} Let\/ $\cG$ and\/ $\cG'$ be Picard groupoids with symmetry invariants\/ $q$ and\/ $q'.$ Let\/ $f_0:\pi_0(\cG)\ra\pi_0(\cG')$ and\/ $f_1:\pi_1(\cG)\ra\pi_1(\cG')$ be group morphisms. There exists a symmetric monoidal functor\/ $F:\cG\ra\cG'$ with\/ $\pi_0(F)=f_0$ and\/ $\pi_1(F)=f_1$ if and only if\/ $q'\ci f_0=f_1\ci q.$
\smallskip

\noindent{\bf(c)} Let\/ $F, G:\cG\ra\cG'$ be symmetric monoidal functors of Picard groupoids with\/ $\pi_0(F)=\pi_0(G)$ and\/ $\pi_1(F)=\pi_1(G).$ There is a monoidal natural isomorphism\/ $F\Rightarrow G$ if and only if the difference class\/ $\om(F,G)\in H^2(\pi_0(\cG),\pi_1(\cG'))$ vanishes. Furthermore, $\om(F,G)$ lies in the subgroup\/ $H^2_\sym(\pi_0(\cG),\pi_1(\cG')),$ which is part of a short exact sequence
\begin{equation}
\text{\begin{small}$\displaystyle
\!\!\xymatrix@C=10pt{
0 \ar[r] & H^2_\sym(\pi_0(\cG),\pi_1(\cG')) \ar[r] & H^2(\pi_0(\cG),\pi_1(\cG')) \ar[r] & \Alt(\pi_0(\cG),\pi_1(\cG')) \ar[r] & 0.}$\end{small}}
\label{fmAeq9}
\end{equation}
The set of all symmetric monoidal functors\/ $F:\cG\ra\cG'$ with given\/ $\pi_0(F)=f_0$ and\/ $\pi_1(F)=f_1$ modulo monoidal natural isomorphism is a torsor over\/ $H^2_\sym(\pi_0(\cG),\pi_1(\cG')).$
\smallskip

\noindent{\bf(d)} Let\/ $F, G:\cG\ra\cG'$ be symmetric monoidal functors of Picard groupoids. If non-empty, the set of all monoidal natural isomorphisms\/ $F\Rightarrow G$ is a torsor over\/ $H^1(\pi_0(\cG),\pi_1(\cG')).$ Here $H^1(\pi_0(\cG),\pi_1(\cG'))=\Hom(\pi_0(\cG),\pi_1(\cG'))$ as $\pi_0(\cG)$ acts trivially on~$\pi_1(\cG')$.
\end{thm}

\begin{proof}
(a) This is immediate from the preceding discussion.
\smallskip

\noindent (b) By choosing identifications $\cG\simeq \cP(\pi_0,\pi_1,\si)$ and $\cG'\simeq \cP(\pi_0',\pi_1',\si'),$ symmetric monoidal functors $\cG\ra\cG'$ correspond to $2$-cocycles $\phi\in C^2(\pi_0,\pi_1')$ satisfying \eq{fmAeq8}. Suppose that $q'\ci f_0=f_1\ci q.$ Then the left hand side of \eq{fmAeq8} vanishes for $x=y,$ hence is in the kernel of $\Skew(\pi_0,\pi_1')\ra\Hom(\pi_0/2\pi_0,\pi_1').$ By Proposition \ref{fmAprop1}(a) and (b) these are the alternating forms which can all be written as $\phi(x,y)-\phi(y,x)$ for some $2$-cocycle $\phi.$
\smallskip

\noindent (c) The exact sequence follows from Proposition \ref{fmAprop1}, and the rest follows from Theorem \ref{fmAthm1}(c), but also requiring functors to be {\it symmetric\/} monoidal.
\smallskip

\noindent (d) This follows from Theorem \ref{fmAthm1}(d) and the fact that the equivalence relation $\sim$ on monoidal transformations $\th:F\Ra G$ defined in \eq{fmAeq3} is trivial in the symmetric case.
\end{proof}

It follows that every Picard groupoid is equivalent to a category of $\pi_0$-graded $\pi_1$-torsors $\cG=\pi_0\qs\pi_1$ (see Example \ref{fmAex1}) with symmetry isomorphism determined by $\si\in\Skew(\pi_0,\pi_1)$ as
\begin{align*}
(x_0,S_0)\ot_\si(x_1,S_1)&\longra(x_1,S_1)\ot_\si(x_0,S_0),\\
s_0\ot_\si s_1&\longmapsto\si(x_0,x_1)(s_1\ot_\si s_0).
\end{align*}
These Picard groupoids are equivalent if the `diagonal' quadratic forms $\si(x,x)$ coincide. We may therefore view a symmetry isomorphism as a {\it sign convention} when commuting objects past each other. From this point of view, Theorem \ref{fmAthm2} classifies all possible sign conventions on $\pi_0\qs\pi_1$ up to equivalence. Sign conventions are very important in the construction of the Quillen determinant line bundle (see \cite{Upme1}) and they are equally important here.

\begin{rem}
\label{fmArem1}
The similarity between Theorems \ref{fmAthm1}, \ref{fmAthm2} and obstruction theory (Postnikov invariants, in particular in stable homotopy theory) is not a coincidence. The reason is that Picard groupoids are a model for stable homotopy 1-types. In this context, the invariant $q$ is called the {\it stable 1\textsuperscript{st} Postnikov invariant}.
\end{rem}

\begin{ex}
\label{fmAex6}
To calculate the obstruction $\om(F,G)$ for a pair of symmetric monoidal functors $F,G:\cG\ra(\pi_0'\qs\pi_1',\ot_{\si'}),$ where $\si'\in\Skew(\pi_0',\pi_1'),$ the symmetric monoidal equivalence $Z:\cP(\pi_0,\pi_1,\si)\ra\cG$ needs to be made explicit. Pick an object $x_i$ of $\cG$ in each isomorphism class $[x_i]\in\pi_0(\cG).$ This determines the functor $Z,$ by mapping $[x_i]$ to the object $x_i$ and an automorphism $f\in\pi_1(\cG)$ of $[x_i]$ to $f\ot x_i\in\Aut_\cG(x_i).$ Pick isomorphisms $\ze_{i,j}: x_i\ot x_j\ra x_{i+j}$ such that these satisfy the associativity axiom for a monoidal structure on $Z$ (always possible by Lemma \ref{fmAlem4}). By choosing bases of each torsor $F(x_i),$ the morphism of torsors
\begin{equation*}
\begin{tikzcd}[column sep=large]
 F(x_i)\ot F(x_j)\ar[r,"\phi_{x_i,x_j}"] & F(x_i\ot x_j)\ar[r,"F(\ze_{i,j})"] & F(x_{i+j})
\end{tikzcd}
\end{equation*}
determines a $2$-cocycle $\phi_{i,j}\in C^2(\pi_0,\pi_1').$ Here, $\phi_{x,y}$ denotes the monoidal structure on $F.$ Another choice of bases changes $\phi_{i,j}$ by a coboundary. In the same way, we obtain from $G$ a $2$-cocycle $\psi_{i,j}\in C^2(\pi_0,\pi_1').$ Then $[\phi_{i,j}-\psi_{i,j}]$ is the obstruction class for the existence of a symmetric monoidal isomorphism $F\cong G.$
\end{ex}

\begin{ex}
\label{fmAex7}
In the situation of Theorem \ref{fmAthm2}(c), write $\pi_i=\pi_i(\cG)$, $\pi_i'=\pi_i(\cG')$, and suppose that $\pi_0$ is a finitely generated abelian group and $\pi_1'=\Z_2$. Then for $n,a_{p,k}\ge 0$ with only finitely many $a_{p,k}$ nonzero we may write 
\begin{equation*}
\pi_0\cong\Z^n\t\prod_{\text{$p$ prime, $k\ge 1$}}(\Z_{p^k})^{a_{p,k}}.
\end{equation*}
Since $H^2_\sym(\Ga_1\t\Ga_2,\Z_2)\cong H^2_\sym(\Ga_1,\Z_2)\t H^2_\sym(\Ga_2,\Z_2)$ it follows that
\begin{equation*}
H^2_\sym(\pi_0,\Z_2)\cong H^2_\sym(\Z,\Z_2)^n\t\prod_{\text{$p$ prime, $k\ge 1$}}(H^2_\sym(\Z_{p^k},\Z_2))^{a_{p,k}}.
\end{equation*}
Calculation shows that
\begin{equation*}
H^2_\sym(\Z,\Z_2)=0,\quad H^2_\sym(\Z_{p^k},\Z_2)\cong\begin{cases} \Z_2, & p=2, \\ 0, & p>2.\end{cases}
\end{equation*}
Thus, each factor of $\Z_{2^k}$ in $\pi_0$ contributes a factor $\Z_2$ to $H^2_\sym(\pi_0,\Z_2)$. 

Write $G=\{\ga\in\pi_0:2G=0\}$ for the 2-torsion subgroup of $\pi_0$. Then $G\cong\prod_{k\ge 1}\bigl(\Z_2\an{\ul{2^{k-1}}}\bigr)^{a_{2,k}}$, that is, each factor of $\Z_{2^k}$ in $\pi_0$ contributes a factor $\Z_2$ to $G$. We can make the relationship between $H^2_\sym(\pi_0,\Z_2)$ and $G$ explicit as follows: there is an isomorphism $\Psi_{\pi_0}:H^2_\sym(\pi_0,\Z_2)\ra\Hom(G,\Z_2)$, such that
\e
\Psi_{\pi_0}([C]):\ga\longmapsto C(\ga,\ga)+C(0,0)
\label{fmAeq10}
\e
whenever $[C]\in H^2_\sym(\pi_0,\Z_2)$ is represented by a cocycle $C\in C^2(\pi_0,\Z_2)$ which is {\it symmetric}, i.e.\ $C(\al,\be)=C(\be,\al)$ for $\al,\be\in \pi_0$, and $\ga\in G\subset\pi_0$. 

To see that \eq{fmAeq10} is well defined, note that if $C'=C+\d D$ for $D:\pi_0\ra\Z_2$ then $C'(\ga,\ga)=C(\ga,\ga)+D(2\ga)=C(\ga,\ga)+D(0)$ as $2\ga=0$ and $C'(0,0)=C(0,0)+D(0)$, and the condition $\d C=0$ implies that $\Psi([C])$ is a group morphism. The map $\Ga\mapsto\Psi_\Ga$ is compatible with products $\Ga=\Ga_1\t\Ga_2$, and we can check that $\Psi_\Ga$ is an isomorphism for $\Ga=\Z_{2^k}$, $k\ge 1$, so $\Psi_\Ga$ is an isomorphism for any finitely generated abelian group $\Ga$.

Now let $F,G$ be as in Theorem \ref{fmAthm2}(c). As $\pi_i(F)=\pi_i(G)$ there exists a (not necessarily monoidal) natural isomorphism $\ze:F\Ra G$. Define a symmetric cocycle $C_\ze\in C^2(\pi_0,\Z_2)$ by, for all objects $M_1,M_2\in\cF$
\begin{align*}
&C_\ze([M_1],[M_2])\\
&= F_\ot(M_1,M_2)\ci(\ze(M_1)\ot\ze(M_2))^{-1}\ci G_\ot(M_1,M_2)^{-1}\ci\ze(M_1\ot M_2)\\
&\quad \in\Hom_\cG(F(M_1\ot M_2),F(M_1\ot M_2))\cong\Z_2,
\end{align*}
where $F_\ot(M_1,M_2):F(M_1)\ot F(M_2)\ra F(M_1\ot M_2)$ is the isomorphism in \eq{fmAeq2}. Then $\om(F,G)=[C_\ze]\in H^2_\sym(\pi_0,\Z_2)$. So Theorem \ref{fmAthm2}(c) says that a monoidal isomorphism $\eta:F\Ra G$ exists if and only if $[C_\ze]=0$, or equivalently, if $\Psi([C_\ze])=0$, that is, if $C_\ze(\ga,\ga)+C_\ze(0,0)=\ul{0}\in\Z_2$ for all~$\ga\in G\subseteq\pi_0$.
\end{ex}

\medskip

\noindent{\small\sc Dominic Joyce, The Mathematical Institute, Radcliffe
Observatory Quarter, Woodstock Road, Oxford, OX2 6GG, U.K. 

\noindent E-mail:  {\tt joyce@maths.ox.ac.uk.}

\smallskip

\noindent Markus Upmeier, Department of Mathematics, University of Aberdeen, Fraser Noble Building, Elphinstone Rd, Aberdeen, AB24 3UE, U.K.

\noindent E-mail: {\tt markus.upmeier@abdn.ac.uk.}}


\begin{thebibliography}{103}
\addcontentsline{toc}{section}{References}

\bibitem{ABP} D.W. Anderson, E.H. Brown, Jr., and F.P. Peterson, {\it The structure of the Spin cobordism ring}, Ann. Math. 86 (1967), 271--298.

\bibitem{Atiy} M.F. Atiyah, {\it Bordism and cobordism}, Proc. Camb. Phil. Soc. 57 (1961), 200--208.

\bibitem{AtSi} M.F. Atiyah and I.M. Singer, {\it The Index of Elliptic Operators: IV}, Ann. Math. 92 (1970), 119--138.

\bibitem{BeFa} K. Behrend and B. Fantechi, {\it The intrinsic normal cone}, Invent. Math. 128 (1997), 45--88. \href{https://arxiv.org/abs/alg-geom/9601010}{alg-geom/9601010}.

\bibitem{Blan} A. Blanc, {\it Topological K-theory of complex noncommutative spaces}, Compositio Math. 152 (2015), 489--555. \href{http://arxiv.org/abs/1211.7360}{arXiv:1211.7360}.

\bibitem{Bojk1} A. Bojko, {\it Orientations for DT invariants on quasi-projective Calabi--Yau\/ $4$-folds},  Adv. Math. 388 (2021), no. 107859. \href{https://arxiv.org/abs/2008.08441}{arXiv:2008.08441}.

\bibitem{Bojk2} A. Bojko, {\it Hilbert schemes of points on Calabi--Yau\/ $4$-folds via wall-crossing}, Adv. Math. 448 (2024), no. 109715. \href{https://arxiv.org/abs/2102.01056}{arXiv:2102.01056}.

\bibitem{Bore} A. Borel, {\it Sur l'homologie et la cohomologie des groupes de Lie compacts connexes}, Amer. J. Math. 76 (1954), 273--342.

\bibitem{BoJo} D. Borisov and D. Joyce, {\it Virtual fundamental classes for moduli spaces of sheaves on Calabi--Yau four-folds}, Geom. Topol. 21 (2017), 3231--3311. \href{http://arxiv.org/abs/1504.00690}{arXiv:1504.00690}.

\bibitem{Bott} R. Bott, {\it Quelques remarques sur les th\'eor\`emes de p\'eriodicit\'e}, Bull. Soc. Math. France 87 (1959), 293--310.

\bibitem{BoSa} R. Bott and H. Samelson, {\it Applications of the theory of Morse to symmetric spaces}, Amer. J. Math. 80 (1958), 964--1029.

\bibitem{BMT} L. Breen, R. Mikhailov and A. Touz\'e, {\it Derived functors of the divided power functors}, Geom. Topol. 20 (2016), 257--352. \href{http://arxiv.org/abs/1312.5676}{arXiv:1312.5676}.

\bibitem{BrTh} W. Browder, E. Thomas, {\it Axioms for the generalized Pontryagin cohomology operations} Q. J. Math. 13 (1962), 55--60.

\bibitem{Brow1} E.H. Brown, Jr., {\it The cohomology of\/ {$B{\rm SO}_{n}$} and\/ {$B{\rm O}_{n}$} with integer coefficients}, Proc. Amer. Math. Soc. 85 (1982), 283--288.

\bibitem{BCM} R.R. Bruner, M. Catanzaro and J.P. May, {\it Characteristic classes}, lecture notes, 2012. Available at \hfil\break \url{https://www.math.uchicago.edu/~may/CHAR/charclasses.pdf}.

\bibitem{Bu1} C. Bu, {\it Enumerative invariants in self-dual categories. I. Motivic invariants}, \href{https://arxiv.org/abs/2302.00038}{arXiv:2302.00038}, 2023.

\bibitem{Bu2} C. Bu, {\it Enumerative invariants in self-dual categories. II. Homological invariants}, \href{https://arxiv.org/abs/2309.00056}{arXiv:2309.00056}, 2023.

\bibitem{Cao1} Y. Cao, {\it Counting conics on sextic $4$-folds}, Math. Res. Lett. 26 (2019), 343--1357. \href{https://arxiv.org/abs/1805.04696}{arXiv:1805.04696}.

\bibitem{Cao2} Y. Cao, {\it Genus zero Gopakumar--Vafa type invariants for Calabi--Yau 4--folds II: Fano $3$-folds}, Comm. Contemp. Math. 22 (2020), no. 1950060. \href{https://arxiv.org/abs/1801.06130}{arXiv:1801.06130}.

\bibitem{CGJ} Y. Cao, J. Gross and D. Joyce, {\it Orientability of moduli spaces of\/ $\Spin(7)$-instantons and coherent sheaves on Calabi--Yau $4$-folds}, Adv. Math. 368 (2020). \href{http://arxiv.org/abs/1811.09658}{arXiv:1811.09658}.

\bibitem{CaKo1} Y. Cao and M. Kool, {\it Zero-dimensional Donaldson--Thomas invariants of Calabi--Yau\/ $4$-folds}, Adv. Math. 338 (2018), 601--648. \href{http://arxiv.org/abs/1712.07347}{arXiv:1712.07347}.

\bibitem{CaKo2}  Y. Cao and M. Kool, {\it Curve counting and DT/PT correspondence for Calabi--Yau\/ $4$-folds}, Adv. Math. 375 (2020), 107371. \href{http://arxiv.org/abs/1903.12171}{arXiv:1903.12171}.

\bibitem{CKM} Y. Cao, M. Kool and S. Monavari, {\it Stable pair invariants of local Calabi--Yau\/ $4$-folds}, Int. Math. Res. Not. 6 (2022), 4753--4798. \href{http://arxiv.org/abs/2004.09355}{arXiv:2004.09355}.

\bibitem{CaLe} Y. Cao and N.C. Leung, {\it Donaldson--Thomas theory for Calabi--Yau $4$-folds}, \href{http://arxiv.org/abs/1407.7659}{arXiv:1407.7659}, 2014.

\bibitem{CMT1} Y. Cao, D. Maulik, and Y. Toda, {\it Genus zero Gopakumar--Vafa type invariants for Calabi--Yau\/ $4$-folds}, Adv. Math. 338 (2018), 41--92. \hfil\break \href{https://arxiv.org/abs/1801.02513}{arXiv:1801.02513}.

\bibitem{CMT2} Y. Cao, D. Maulik, and Y. Toda, {\it Stable pairs and Gopakumar--Vafa type invariants for Calabi--Yau\/ $4$-folds}, J. Eur. Math. Soc. 24 (2022), 527--581. \href{https://arxiv.org/abs/1902.00003}{arXiv:1902.00003}.

\bibitem{COT1} Y. Cao, G. Oberdieck, and Y. Toda, {\it Stable pairs and Gopakumar--Vafa type invariants on holomorphic symplectic $4$-folds}, Adv. Math. 408 (2022), no. 108605. \href{https://arxiv.org/abs/2201.11540}{arXiv:2201.11540}.

\bibitem{COT2} Y. Cao, G. Oberdieck, and Y. Toda, {\it Gopakumar--Vafa type invariants of holomorphic symplectic $4$-folds}, Comm. Math. Phys. 405 (2024), no. 26. \href{https://arxiv.org/abs/2201.10878}{arXiv:2201.10878}.

\bibitem{CaQu} Y. Cao and F. Qu, {\it Tautological Hilbert scheme invariants of Calabi-Yau $4$-folds via virtual pullback}, \href{https://arxiv.org/abs/2012.04415}{arXiv:2012.04415}, 2020. 

\bibitem{CaTo1} Y. Cao and Y. Toda, {\it Curve counting via stable objects in derived categories of Calabi--Yau\/ $4$-folds}, Adv. Math. 406 (2022), no. 108531. \hfil\break \href{https://arxiv.org/abs/1909.04897}{arXiv:1909.04897}.

\bibitem{CaTo2} Y. Cao and Y. Toda, {\it Tautological stable pair invariants of Calabi--Yau\/ $4$-folds}, Adv. Math. 396 (2022), no. 108176. \href{https://arxiv.org/abs/2009.03553}{arXiv:2009.03553}.

\bibitem{CaTo3} Y. Cao and Y. Toda, {\it  Gopakumar--Vafa type invariants on Calabi--Yau\/ $4$-folds via descendent insertions}, Comm. Math. Phys. 383 (2021), 281--310. \href{https://arxiv.org/abs/2003.00787}{arXiv:2003.00787}.

\bibitem{Car} H. Cartan, {\it Alg\'ebres d’Eilenberg--MacLane et homotopie}, Expos\'es 2 \`a 16, S\'eminaire Henri Cartan, Ecole Normale Sup\'erieure, Paris, 1956.

\bibitem{Conn} P.E. Conner, {\it Differentiable Periodic Maps}, second edition, Lect. Notes Math. 738, Springer, 1979.

\bibitem{Dold} A. Dold, {\it Relations between ordinary and extraordinary homology}, Matematika 9 (1965), 8--14.

\bibitem{Dona} S.K. Donaldson, {\it Floer Homology Groups in Yang--Mills Theory}, Cambridge University Press, 2009.

\bibitem{DoKr} S.K. Donaldson and P.B. Kronheimer, {\it The Geometry of Four-Manifolds}, Oxford University Press, 1990.

\bibitem{DoSe} S.K. Donaldson and E. Segal, {\it Gauge Theory in Higher Dimensions, II}, Surveys in Diff. Geom. 16 (2011), 1--41. \href{http://arxiv.org/abs/0902.3239}{arXiv:0902.3239}.

\bibitem{DoTh} S.K. Donaldson and R.P. Thomas, {\it Gauge Theory in Higher Dimensions}, Chapter 3 in S.A. Huggett et al., editors, {\it The Geometric Universe}, Oxford University Press, 1998.

\bibitem{EiMa} S. Eilenberg and S. MacLane, {\it On the groups $H(\Pi,n)$. II. Methods of computation}, Ann. Math. 60 (1954), 49--139.

\bibitem{FGZ} A. Fernandez Herrero, T.L. G\'omez, and A. Zamora, {\it The moduli stack of principal $\rho$-sheaves and Gieseker--Harder--Narasimhan filtrations}, Math. Z. 307 (2024), no. 51. \href{https://arxiv.org/abs/2107.03918}{arXiv:2107.03918}.

\bibitem{Free} D.S. Freed, {\it On determinant line bundles}, pages 189--238 in {\it Mathematical Aspects of String Theory}, Adv. Ser. Math. Phys. 1, World Scientific, 1987.

\bibitem{GeMa} S.I. Gelfand and Y.I. Manin, {\it  Methods of Homological Algebra}, second edition, Springer, 2003.

\bibitem{Gilk} P.B. Gilkey, {\it The geometry of spherical space form groups}, World Scientific, 1989.

\bibitem{Gome} T.L. G\'omez, {\it  Algebraic stacks}, Proc. Indian Acad. Sci. Math. Sci. 111 (2001), 1--31. \href{http://arxiv.org/abs/math/9911199}{math.AG/9911199}.

\bibitem{GoSo} T.L. G\'omez and I. Sols, {\it Moduli space of principal sheaves over projective varieties}, Ann. of Math. 161 (2005), 1037--1092. \href{http://arxiv.org/abs/math/0206277}{math.AG/0206277}.

\bibitem{GJT} J. Gross, D. Joyce and Y. Tanaka, {\it Universal structures in $\C$-linear enumerative invariant theories}, SIGMA 18 (2022), 068. \href{https://arxiv.org/abs/2005.05637}{arXiv:2005.05637}.

\bibitem{Hami} R.S. Hamilton, {\it The inverse function theorem of Nash and Moser}, Bull. Amer. Math. Soc. 7 (1982), 65--222.

\bibitem{Hart} R. Hartshorne, {\it  Algebraic Geometry}, Graduate Texts in Math. 52, Spr\-in\-ger, 1977.

\bibitem{Hirs} M.W. Hirsch, {\it Differential topology}, Graduate Texts in Math. 33, Springer, 1976.

\bibitem{Huyb} D. Huybrechts, {\it  Fourier--Mukai transforms in Algebraic Geometry}, Oxford Mathematical Monographs, Clarendon Press, Oxford, 2006.

\bibitem{HuLe} D. Huybrechts and M. Lehn, {\it  The geometry of moduli spaces of sheaves}, Aspects of Math. E31, Vieweg, 1997.

\bibitem{JoOs} N. Johnson and A. Osorno, {\it Modelling stable one-types}, Theory Appl. Categ. 26 (2012), 520--537. \href{https://arxiv.org/abs/1201.2686}{arXiv:1201.2686}.

\bibitem{JoSt} A. Joyal and R. Street, {\it Braided Tensor Categories}, Adv. Math. 102 (1993), 20--78.

\bibitem{Joyc1} D. Joyce, {\it Compact manifolds with special holonomy}, Oxford University Press, 2000.

\bibitem{Joyc2} D. Joyce, {\it An introduction to d-manifolds and
derived differential geometry}, pages 230--281 in L. Brambila-Paz et al., editors, {\it Moduli spaces}, L.M.S. Lecture Notes 411, Cambridge University Press, 2014. \href{https://arxiv.org/abs/1206.4207}{arXiv:1206.4207}.

\bibitem{Joyc3} D. Joyce, {\it Kuranishi spaces as a\/ $2$-category}, pages 253--298 in  J. Morgan, editor, {\it Virtual Fundamental Cycles in Symplectic Topology}, Mathematical Surveys and Monographs 237, A.M.S., 2019. \href{https://arxiv.org/abs/1510.07444}{arXiv:1510.07444}.

\bibitem{Joyc4} D. Joyce, {\it D-manifolds and d-orbifolds: a
theory of derived differential geometry}, to be published by OUP. Preliminary version (2012) available at \url{https://people.maths.ox.ac.uk/~joyce/dmanifolds.html}.

\bibitem{Joyc5} D. Joyce, {\it Conjectures on counting associative $3$-folds in $G_2$-manifolds}, pages 97--160 in V. Mu\~noz et al., editors, {\it Modern Geometry: A Celebration of the Work of Simon Donaldson}, Proc. Symp. Pure Math. 99, A.M.S., 2018. \href{https://arxiv.org/abs/1610.09836}{arXiv:1610.09836}.

\bibitem{Joyc6} D. Joyce, {\it Kuranishi spaces and Symplectic Geometry}, multiple volume book in progress, 2017--2027. Preliminary versions of volumes I, II available at \url{https://people.maths.ox.ac.uk/~joyce/Kuranishi.html}.

\bibitem{Joyc7} D. Joyce, {\it Enumerative invariants and wall-crossing formulae in abelian categories}, \href{https://arxiv.org/abs/2111.04694}{arXiv:2111.04694}, 2021.

\bibitem{JTU} D. Joyce, Y. Tanaka, and M. Upmeier, {\it  On orientations for gauge-theoretic moduli spaces}, Adv. Math. 362 (2020), no. 106957. \href{https://arxiv.org/abs/1811.01096}{arXiv:1811.01096}.

\bibitem{JoUp} D. Joyce and M. Upmeier, {\it Canonical orientations for moduli spaces of\/ $G_2$-instantons with gauge group $\SU(m)$ or $\U(m)$}, J. Diff. Geom. 124 (2023), 199--229. \href{http://arxiv.org/abs/1811.02405}{arXiv:1811.02405}.

\bibitem{KiPa} Y.-H. Kiem and H. Park, {\it Localizing virtual cycles for Donaldson--Thomas invariants of Calabi--Yau\/ $4$-folds}, J. Alg. Geom. 32 (2023), 585--639. \href{http://arxiv.org/abs/2012.13167}{arXiv:2012.13167}.

\bibitem{Lash} R. Lashof, {\it Poincar\'e duality and cobordism}, Trans. A.M.S. 109 (1963), 257--277.

\bibitem{LaMi} H.B. Lawson and M.-L. Michelsohn, {\it Spin geometry}, Princeton Math. Series 38, PUP, Princeton, NJ, 1989.

\bibitem{Lewi} C. Lewis, {\it $\Spin(7)$ instantons}, D.Phil.\ thesis, Oxford University, 1998.

\bibitem{MacL} S. MacLane, {\it Categories for the working mathematician}, second edition, Springer, 1971.

\bibitem{Maun} C.R.F. Maunder, {\it The spectral sequence of an extraordinary cohomology theory}, Proc. Cambridge Philos. Soc. 59 (1963), 567--574.

\bibitem{May} J.P. May, {\it A general algebraic approach to Steenrod operations}, pages 153--231 in {\it The Steenrod Algebra and Its Applications: A Conference to Celebrate N.E. Steenrod's Sixtieth Birthday}, Springer, Berlin--Heidelberg, 1970.

\bibitem{McCl} J. McCleary, {\it A user's guide to spectral sequences}, Camb. Stud. Adv. Math. 58, CUP, Cambridge, 2001.

\bibitem{McLe} R.C. McLean, {\it Deformations of calibrated submanifolds}, Comm. Anal. Geom. 6 (1998), 705--747.

\bibitem{MNS} G. Menet, J. Nordstr\"om and H.N. S\'a Earp, {\it Construction of\/ $G_2$-instantons via twisted connected sums}, Math. Res. Lett. 28 (2021), 471--509. \hfil\break \href{http://arxiv.org/abs/1510.03836}{arXiv:1510.03836}.

\bibitem{Metz} D.S. Metzler, {\it Topological and smooth stacks}, \href{http://arxiv.org/abs/math/0306176}{math.DG/0306176}, 2003.

\bibitem{Miln} J.W. Milnor, {\it Remarks concerning spin manifolds}, pages 5--62 in {\it Differential and Combinatorial Topology (A Symposium in Honor of Marston Morse)}, PUP, Princeton, NJ, 1965.

\bibitem{MiSt} J.W. Milnor and J.D. Stasheff, {\it Characteristic classes}, PUP, Princeton, NJ, 1974.

\bibitem{Nooh1} B. Noohi, {\it Foundations of topological stacks, I}, \href{http://arxiv.org/abs/math/0503247}{math.AG/0503247}, 2005.

\bibitem{Nooh2} B. Noohi, {\it Homotopy types of topological stacks}, \href{http://arxiv.org/abs/0808.3799}{arXiv:0808.3799}, 2008.

\bibitem{Novi} S.P. Novikov, {\it Homotopy properties of Thom complexes}, pages 211--250 in S.P. Novikov and I.A. Taimanov, editors, {\it Topological library. Part 1: cobordisms and their applications}, World Scientific, Hackensack, NJ, 2007. Translated from Russian original in Mat. Sb. 57 (1962), 407--442. 

\bibitem{OhTh} J. Oh and R.P. Thomas, {\it Counting sheaves on Calabi--Yau\/ $4$-folds. I}, Duke Math. J. 172 (2023), 1333--1409. \href{https://arxiv.org/abs/2009.05542}{arXiv:2009.05542}.

\bibitem{PTVV} T. Pantev, B. To\"en, M. Vaqui\'e and G. Vezzosi, {\it  Shifted symplectic structures}, Publ. Math. I.H.E.S. 117 (2013), 271--328. \href{http://arxiv.org/abs/1111.3209}{arXiv:1111.3209}.

\bibitem{Park} H. Park, {\it Virtual pullbacks in Donaldson--Thomas theory of Calabi--Yau\/ $4$-folds}, \href{http://arxiv.org/abs/2110.03631}{arXiv:2110.03631}, 2021.

\bibitem{SaEa} H.N. S\'a Earp, {\it $G_2$-instantons over asymptotically cylindrical manifolds}, Geom. Topol. 19 (2015), 61--111. \href{http://arxiv.org/abs/1101.0880}{arXiv:1101.0880}.

\bibitem{SaWa} H.N. S\'a Earp and T. Walpuski, {\it $G_2$-instantons over twisted connected sums}, Geom. Topol. 19 (2015), 1263--1285. \href{http://arxiv.org/abs/1310.7933}{arXiv:1310.7933}.

\bibitem{Serr} J.-P. Serre, {\it Cohomology mod 2 of Eilenberg--MacLane complexes}, pages 153--190 in S.P. Novikov and I.A. Taimanov, editors, {\it Topological Library, Part 3}, World Scientific, NJ, 2012.

\bibitem{Simp} C. Simpson, {\it  The topological realization of a simplicial presheaf}, \hfil\break \href{http://arxiv.org/abs/q-alg/9609004}{q-alg/9609004}, 1996.

\bibitem{Sinh} H.X. Sinh, {\it Gr-cat\'egories}, Th\'ese de Doctorat, Universit\'e Paris VII, 1975.

\bibitem{Ston1} R.E. Stong, {\it Notes on cobordism theory}, Mathematical notes, Princeton University Press, 1968.

\bibitem{Ston2} R.E. Stong, {\it Appendix: calculation of $\Om\Spin(K(\Z,4))$}, pages 430--437 in {\it Workshop on unified string theories (Santa Barbara, Calif., 1985)}, World Sci. Publishing, Singapore, 1986.

\bibitem{Tana} Y. Tanaka, {\it A construction of\/ $\Spin(7)$-instantons}, Ann. Global Anal. Geom. 42 (2012), 495--521. \href{http://arxiv.org/abs/1201.3150}{arXiv:1201.3150}.

\bibitem{Thom} R. Thom, {\it Quelques propri\'et\'es globales des vari\'et\'es diff\'erentiables}, Comment. Math. Helv. 28 (1954), 17--86. English translation: pages 131--209 in S.P. Novikov and I.A. Taimanov, editors, {\it Topological library. Part 1: cobordisms and their applications}, World Scientific, 2007.

\bibitem{Toen1} B. To\"en, {\it  Higher and derived stacks: a global overview}, pages 435--487 in {\it  Algebraic Geometry --- Seattle 2005}, Proc. Symp. Pure Math. 80, Part 1, A.M.S., 2009. \href{http://arxiv.org/abs/math/0604504}{math.AG/0604504}.

\bibitem{Toen2} B. To\"en, {\it  Derived Algebraic Geometry}, EMS Surveys in Math. Sci. 1 (2014), 153--240. \href{http://arxiv.org/abs/1401.1044}{arXiv:1401.1044}.

\bibitem{ToVa} B. To\"en and M. Vaqui\'e, {\it Moduli of objects in dg-categories}, Ann. Sci. \'Ec. Norm. Sup. 40 (2007), 387--444. \href{https://arxiv.org/abs/math/0503269}{math.AG/0503269}.

\bibitem{ToVe1} B. To\"en and G. Vezzosi, {\it  From HAG to DAG:
derived moduli stacks}, pages 173--216 in {\it  Axiomatic, enriched and motivic homotopy theory}, NATO Sci. Ser. II Math. Phys. Chem., 131, Kluwer, 2004. \href{http://arxiv.org/abs/math/0210407}{math.AG/0210407}.

\bibitem{ToVe2} B. To\"en and G. Vezzosi, {\it  Homotopical Algebraic Geometry II: Geometric Stacks and Applications}, Mem. A.M.S. 193 (2008), no. 902. \hfil\break \href{http://arxiv.org/abs/math/0404373}{math.AG/0404373}.

\bibitem{Upme1} M. Upmeier, {\it A categorified excision principle for elliptic symbol families}, Q. J. Math. 72 (2021), 1099--1132. \href{https://arxiv.org/abs/1901.10818}{arXiv:1901.10818}.

\bibitem{Upme2} M. Upmeier, {\it Bordism invariance of orientations and real APS index theory}, Adv. Math. 461 (2025), no. 110048. \href{https://arxiv.org/abs/2312.06818}{arXiv:2312.06818}.

\bibitem{Walp1} T. Walpuski, {\it Gauge theory on $G_2$-manifolds}, PhD Thesis, Imperial College London, 2013.

\bibitem{Walp2} T. Walpuski, {\it $G_2$-instantons on generalized Kummer constructions}, Geom. Topol. 17 (2013), 2345--2388. \href{http://arxiv.org/abs/1109.6609}{arXiv:1109.6609}.

\bibitem{Walp3} T. Walpuski, {\it $G_2$-instantons over twisted connected sums: an example}, Math. Res. Lett. 23 (2016), 529--544. \href{http://arxiv.org/abs/1505.01080}{arXiv:1505.01080}.

\bibitem{Walp4} T. Walpuski, {\it $G_2$-instantons, associative submanifolds and Fueter sections}, Comm. Anal. Geom. 25 (2017), 847--893. \href{http://arxiv.org/abs/1205.5350}{arXiv:1205.5350}.

\bibitem{Walp5} T. Walpuski, {\it $\Spin(7)$-instantons, Cayley submanifolds and Fueter sections}, Comm. Math. Phys. 352 (2017), 1--36. \href{http://arxiv.org/abs/1409.6705}{arXiv:1409.6705}.

\end{thebibliography}
\end{document}